\documentclass[12pt, a4paper]{article}

\usepackage{amssymb,amsmath,amsthm}
\usepackage[all]{xy}

\usepackage{amssymb,amsmath,amsthm}
\usepackage[mathscr]{eucal}%
\usepackage[all]{xy}
\usepackage{lmodern}
\usepackage[T1]{fontenc}
\usepackage[utf8]{inputenc}
\usepackage{tikz-cd}
\usepackage[shortlabels]{enumitem}
\usepackage{relsize}
\usepackage{geometry}
\usepackage{mathtools}
\usepackage{adjustbox}
\usepackage{lscape}
\usepackage{diagbox}
\usepackage{slashbox}
\usepackage{bm}
\usepackage{caption}
\usepackage{float}

\usepackage{stmaryrd}
\usepackage{mathrsfs}  
\usepackage{multirow}
\usepackage{hyperref}
\usetikzlibrary{positioning}
\usepackage{tikz-3dplot}
\usepackage{xifthen}
\usetikzlibrary{quotes,arrows.meta}
\tdplotsetmaincoords{60}{125}
\tdplotsetrotatedcoords{0}{0}{0}
\usepackage{xcolor}
\definecolor{mygreen}{rgb}{0.16,.55,0.0}
\usepackage[nottoc]{tocbibind}

\theoremstyle{plain}
\newtheorem{prop}{Proposition}[subsubsection]
\newtheorem{conj}[prop]{Conjecture}
\newtheorem{lem}[prop]{Lemma}
\newtheorem{thm}[prop]{Theorem}
\newtheorem{cor}[prop]{Corollary}
\theoremstyle{definition}
\newtheorem{definit}[prop]{Definition}
\newtheorem{ex}[prop]{Example}
\newtheorem{rem}[prop]{Remark}

\theoremstyle{plain} 
\newtheorem{prop1}{Proposition}[subsection]
\newtheorem{conj1}[prop1]{Conjecture}
\newtheorem{lem1}[prop1]{Lemma}
\newtheorem{thm1}[prop1]{Theorem}
\newtheorem{cor1}[prop1]{Corollary}
\theoremstyle{definition}
\newtheorem{definit1}[prop1]{Definition}
\newtheorem{ex1}[prop1]{Example}
\newtheorem{rem1}[prop1]{Remark}

\theoremstyle{plain}

\theoremstyle{definition}

\theoremstyle{plain}

\theoremstyle{definition}

\def\ilim#1{\displaystyle \lim_{\stackrel{\longrightarrow}{#1}}}

\def\plim#1{\displaystyle \lim_{\stackrel{\longleftarrow}{#1}}}
\def\varddots{\mathinner{\raise7pt\vbox{\kern3pt\hbox{.}}\mkern1mu\smash{\raise4pt\hbox{.}}\mkern1mu\smash{\raise1pt\hbox{.}}}}

\DeclareMathAlphabet{\mathpzc}{OT1}{pzc}{m}{it}

\DeclarePairedDelimiter{\scalar}{\langle}{\rangle}
\DeclarePairedDelimiter{\set}{\{}{\}}

\DeclareMathOperator{\Aut}{Aut}
\DeclareMathOperator{\End}{End}
\DeclareMathOperator{\Hom}{Hom}
\DeclareMathOperator{\cont}{cont}
\DeclareMathOperator{\Ind}{Ind}
\DeclareMathOperator{\cInd}{c-Ind}
\DeclareMathOperator{\Res}{Res}

\DeclareMathOperator{\GL}{GL}

\DeclareMathOperator{\Gal}{Gal}

\DeclareMathOperator{\tr}{tr}
\DeclareMathOperator{\rk}{\mathrm{rk}}

\DeclareMathOperator{\Id}{Id}

\DeclareMathOperator{\soc}{soc}

\DeclareMathOperator{\Ext}{Ext}

\DeclareMathOperator{\gr}{gr}
\DeclareMathOperator{\s}{Sym}
\DeclareMathOperator{\ad}{ad}

\newcommand{\p}{\mathfrak{p}}
\newcommand{\q}{\mathfrak{q}}

\newcommand{\m}{\mathfrak{m}}

\newcommand{\NN}{\mathbb{N}}
\newcommand{\Z}{{\mathbb Z}}

\newcommand{\cC}{\mathcal{C}}

\newcommand{\QK}[1]{I_{1}}

\newcommand{\Qpf}{\mathbb{Q}_{p^f}}
\newcommand{\Qp}{\mathbb{Q}_{p}}

\newcommand{\Zp}{\mathbb{Z}_{p}}
\newcommand{\F}{\mathbb{F}}
\newcommand{\Fp}{\mathbb{F}_{p}}
\newcommand{\Fpf}{\mathbb{F}_{p^f}}
\newcommand{\Fq}{\mathbb{F}_{q}}

\newcommand{\Qpbar}{\overline{\mathbb{Q}}_p}

\newcommand{\Fpbar}{\overline{\mathbb{F}}_p}

\newcommand{\TT}{\mathcal{T}}

\newcommand{\Q}{\mathbb{Q}}

\newcommand{\GG}{\mathcal{G}}

\newcommand{\R}{\mathbb{R}}

\newcommand{\onto}{\twoheadrightarrow}
\newcommand{\into}{\hookrightarrow}
\newcommand{\congto}{\xrightarrow{\,\sim\,}}
\newcommand{\xonto}[2][]{%
  \xrightarrow[#1]{#2}\mathrel{\mkern-14mu}\rightarrow}
\newcommand{\eps}{\varepsilon}

\newcommand{\EE}{\mathrm{E}}
\newcommand{\cZ}{\mathcal{Z}}
\newcommand{\cO}{\mathcal{O}}
\newcommand{\et}{\acute{\mathrm{e}}\mathrm{t}}

\newcommand{\LL}{L^{\!\otimes}}
\newcommand{\LLbar}{\overline L^{\otimes}}
\newcommand{\Lbar}{\overline L}
\DeclareMathOperator{\ind}{ind}

\newcommand{\Av}{{\mathbb A}_{F^+}^{\infty,v}}

\newcommand{\Ap}{{\mathbb A}_{F^+}^{\infty,p}}
\newcommand{\A}{{\mathbb A}_{F^+}^{\infty}}
\newcommand{\GAv}{H(\Av)}

\newcommand{\GAp}{H(\Ap)}
\newcommand{\GA}{H(\A)}

\newcommand{\fgk}{\Phi \Gamma_{\!\!{\mathbb F}}^{\rm \acute et}}

\newcommand{\fghatk}{\widehat{\Phi \Gamma}_{\!\!{\mathbb F}}^{\rm \acute et}}

\newcommand{\repk}{{\rm Rep}_{{\mathbb F}}}
\newcommand{\irepk}{{\rm IndRep}_{{\mathbb F}}}

\newcommand{\cJ}{{\mathcal{J}}}

\newcommand{\pE}{\varpi_E}

\newcommand{\rbar}{\overline{r}}
\newcommand{\rhobar}{\overline{\rho}}

\newcommand{\gp}{{\Gal}(\Qpbar/\Qp)}

\newcommand{\gK}{{\Gal}(\Qpbar/K)}
\newcommand{\gKQ}{{\Gal}(K/\Qp)}

\newcommand{\gF}{{\Gal}(\overline F/F)}

\DeclareMathOperator{\Frob}{Frob}

\newcommand{\oE}{{\mathcal O}_E}
\newcommand{\oK}{{\mathcal O}_K}
\newcommand{\oFF}{{\mathcal O}_{\!F^+}}
\newcommand{\oFFv}{{\mathcal O}_{\!F_v^+}}

\newcommand{\oF}{{\mathcal O}_{\!F}}

\newcommand{\sep}{\mathrm{sep}}

\newcommand{\Ker}{\mathrm{Ker}}%
\newcommand{\brho}{\overline{\rho}}%
\newcommand{\smatr}[4]{\bigl(\begin{smallmatrix} {#1}& {#2}\\ {#3}&{#4}\end{smallmatrix}\bigl)}%
\newcommand{\smat}[1]{\left( \begin{smallmatrix} #1 \end{smallmatrix} \right)}

\newcommand{\defeq}{\stackrel{\textrm{\tiny{\upshape{def}}}}{=}}

\newcommand{\ovl}[1]{\overline{#1}}

\newcommand{\un}[1]{\underline{#1}}
\renewcommand{\bf}[1]{\mathbf{#1}}
\newcommand{\tld}[1]{\tilde{#1}}
\newcommand{\wtld}[1]{\widetilde{#1}}

\DeclareMathOperator{\JH}{\mathrm{JH}}

\newcommand{\phz}{\varphi}
\newcommand{\ra}{\rightarrow}

\newcommand{\lra}{\longrightarrow}

\newcommand{\ppar}[1]{(\mkern-3mu(#1)\mkern-3mu)}
\newcommand{\bbra}[1]{\llbracket #1\rrbracket}

\newcommand{\fm}{\mathfrak{m}}

\DeclareMathOperator{\Spec}{Spec}

\newcommand{\xto}[1][]{\xrightarrow{#1}}
\newcommand{\simto}{\xto[\sim]} %

\renewcommand{\subset}{\subseteq}
\renewcommand{\simeq}{\cong} 
\theoremstyle{plain}

\title{Conjectures and results on modular representations of $\GL_n(K)$ for a $p$-adic field $K$}

\author{Christophe Breuil\footnote{CNRS, B\^atiment 307, Facult\'e d'Orsay, Universit\'e Paris-Saclay, 91405 Orsay Cedex, France}\\
\and
Florian Herzig\footnote{Dept.\ of Math., Univ.\ of Toronto, 40 St.\ George St., BA6290, Toronto, ON M5S 2E4, Canada}\\
\and
Yongquan Hu\footnote{Morningside Center of Math., No.\ 55, Zhongguancun East Road, Beijing, 100190, China}\\
\and
Stefano Morra\footnote{Lab.\ d'Analyse, G\'eom\'etrie, Alg\`ebre, 99 Av.\ Jean Baptiste Cl\'ement, 93430 Villetaneuse, France }\\
\and
Benjamin Schraen\footnote{B\^atiment 307, Facult\'e d'Orsay, Universit\'e Paris-Saclay, 91405 Orsay Cedex, France}}

\date{ }

\begin{document} 

\maketitle

\setcounter{tocdepth}{3}

\begin{abstract}
Let $p$ be a prime number and $K$ a finite extension of $\Qp$. We state conjectures on the smooth representations of $\GL_n(K)$ that occur in spaces of mod $p$ automorphic forms (for compact unitary groups). In particular, when $K$ is unramified, we conjecture that they are of finite length and predict their internal structure (extensions, form of subquotients) from the structure of a certain algebraic representation of $\GL_n$. When $n=2$ and $K$ is unramified, we prove several cases of our conjectures, including new finite length results.
\end{abstract}

\tableofcontents

\newpage

\section{Introduction}\label{intro}

\subsection{Preamble}\label{preamble}

Let $p$ be a prime number and $K$ a local field of residue characteristic $p$. In the early nineties, Barthel and Livn\'e had the fancy idea to start classifying irreducible (admissible) smooth representations of $\GL_2(K)$ over an algebraically closed field of characteristic $p$ (\cite{BL1}, \cite{BL2}). They found four nonempty distinct classes of such representations: $1$-dimensional ones, irreducible principal series, special series, and those which are not an irreducible constituent of a principal series that they called supersingular. In 2001, one of us classified supersingular representations of $\GL_2(\Qp)$ with a central character (\cite{breuilI}) and showed that they are in ``natural'' bijection with $2$-dimensional irreducible representations of $\gp$ in characteristic $p$. This was one of the starting points of the mod $p$ and $p$-adic Langlands programmes for $\GL_2(\Qp)$, which was developed essentially during the decade 2000-2010 (see for instance \cite{breuilII}, \cite{Brcompl}, \cite{emerton-ordII}, \cite{kisin-asterisque}, \cite{Colmez}, \cite{berger10a}, \cite{paskunasIHES}, \cite{emerton-local-global}, \cite{CDP}, \cite{CEGGPS2}, \dots).

There are two main novel features of the mod $p$ local Langlands correspondence for $\GL_2(\Qp)$ (compared to previous Langlands correspondences). The first one is that it involves {\it reducible} representations of $\GL_2(\Qp)$. More precisely, the representation of $\GL_2(\Qp)$ is irreducible (resp.\ semisimple, resp.\ indecomposable) if and only if its corresponding $2$-dimensional representation of $\gp$ is, and, in the reducible case, is given (at least generically) by an extension between two specific principal series. The second one, found by Colmez in \cite{Colmez}, is that the correspondence can be made {\it functorial} by an exact functor from finite length representations of $\GL_2(\Qp)$ to \'etale $(\varphi,\Gamma)$-modules, i.e.\ to finite length representations of $\gp$ by Fontaine's equivalence. Thanks to this exact functor, one can extend the correspondence first to extensions of representations, and then to deformations on both sides. 

When $K$ is not $\Qp$, trouble comes from supersingular representations. Contrary to the case $K=\Qp$, they can be more numerous than $2$-dimensional irreducible representations of ${\rm Gal}(\overline K/K)$ (\cite{BP}) and they cannot be described as quotients of a compact induction by a finite number of equations (\cite[Cor.5.5]{yongquan-jussieu}, \cite[Thm.0.1]{Benj}, \cite[Thm.1.1]{Wu}), justifying {\it a posteriori} the terminology ``very strange'' that was used to describe them in the introduction of \cite{BL2}. As a consequence, no classification of supersingular representations of $\GL_2(K)$ is known so far, which has hitherto made impossible to find a definition of a hypothetical local mod $p$ correspondence for $\GL_2(K)$ by purely local (either representation theoretic or geometric) means.

Fortunately, the global theory comes to the rescue. If a local correspondence exists, there is a place where it should be realized: the mod $p$ cohomology of Shimura varieties. Let us assume now that $K$ is a finite unramified extension of $\Qp$ with residue field $\Fpf$ and let $K_1\defeq 1+pM_2(\cO_K)\subseteq \GL_2(\cO_K)$. Following the pioneering work of \cite{BDJ} on Serre weight conjectures, a series of articles (\cite{BP}, \cite{EGS}, \cite{HuWang}, \cite{LMS}, \cite{DanWild}) led to a complete description of the $K_1$-invariants of the $\GL_2(K)$-representations carried by Hecke isotypic subspaces in such mod $p$ cohomology groups. Although these invariants are only a tiny piece of the representations of $\GL_2(K)$, combined with weight cycling this turned out to give a strong hint on the form of these representations, as well as being a useful technical result. Indeed, very recently, building on this description and on results of \cite{BHHMS1}, Hu and Wang could prove that, at least when $K$ is quadratic unramified and the representation of $\gK$ is a nonsplit extension between two (sufficiently generic) characters, these $\GL_2(K)$-representations are indecomposable of length $3$ (in particular are of finite length), with similar principal series as in the case $K=\Qp$ in socle and cosocle, and a supersingular representation ``in the middle'' (\cite[Thm.1.7]{HuWang2}).

These recent results maintain the hope of a local Langlands correspondence for $\GL_2(K)$. They also prompted us to make public some conjectures we had in mind for many years on the form of the $\GL_n(K)$-representations carried by Hecke isotypic subspaces, and on a functorial link to representations of $\gp$ via $(\varphi,\Gamma)$-modules. We state such conjectures in the present work (Conjecture \ref{theconjbar}, Conjecture \ref{conj:generale}, Conjecture \ref{theconj}) and we prove some special cases in the case $n=2$ and $K$ unramified, including some new finite length results (Theorem \ref{specialcase1}, Theorem \ref{cor:pi-irredglob}, Corollary \ref{specialcase2}). Moreover, when $n=2$ and $K$ is unramified, we also define (and use in the proofs!) an abelian category $\mathcal C$ of smooth admissible representations of $\GL_2(K)$ in characteristic $p$ (containing the representations coming from the global theory) together with an exact functor from $\mathcal C$ to a new category of multivariable $(\varphi,\Gamma)$-modules.

\subsection{Conjectures}

Let us first describe our conjectures with some details. As usual, we mostly work in the setting of compact unitary groups (except in \S\ref{Cgroup}), so that we do not (yet) mix delicate representation theoretic issues with difficult geometric problems (ultimately, we think that the representations of $\GL_n(K)$ should not change from one global setting to another). We fix $F$ a CM-field, i.e.\ a totally imaginary quadratic extension of a totally real number field $F^+$, and we assume {\it for simplicity in this introduction} that $p$ is inert in $F^+$. We also assume (not for simplicity) that the unique $p$-adic place $v$ of $F^+$ splits in $F$. We fix a continuous absolutely irreducible representation
\[\rbar:\gF\longrightarrow \GL_n(\F),\]
where $\F$ is a (sufficiently large) extension of $\Fp$ and we assume that $\rbar$ is automorphic for a unitary group $H$ over $F^+$ that is compact at all infinite places and becomes $\GL_n$ over $F$. Equivalently there exists a compact open subgroup $U^v\subseteq \GAv$ such that
\[S(U^{v},\F)[\m]\defeq \{f:H(F^+)\backslash \GA/U^v\rightarrow \F\ {\rm locally\ constant}\}[\m]\ne 0,\]
where $[\m]$ means the Hecke-isotypic subspace associated to $\rbar$ (one has to choose a finite set of bad places $\Sigma$ in the definition of $\m$, but we forget this issue here, see \S\ref{wlgc} below). 

Let $\tilde v\vert v$ in $F$, $K\defeq F_{\tilde v}$ the corresponding completion and $\rbar_{\tilde{v}}$ the restriction of $\rbar$ to a decomposition subgroup at ${\tilde{v}}$. Then $S(U^{v},\F)[\m]$ is an admissible smooth representation of $\GL_n(K)$ over $\F$ by the usual right translation action on functions. Our main conjecture gives the form of this $\GL_n(K)$-representation (assuming it is of finite length) as well as a functorial link to $\rbar_{\tilde{v}}$. But to state it we need a few preliminaries on certain algebraic representations of $\GL_n$ over $\F$.

Let us first assume {\it for simplicity} that $K=\Qp$. We let $\rm Std$ be the standard $n$-dimensional algebraic representation of $\GL_n$ over $\F$ and define the following algebraic representation of $\GL_n$ over $\F$:
\[\LLbar\defeq \bigotimes_{i=1}^{n-1}{\bigwedge}^{\!\!i}_{\F}{\rm Std}.\]
We fix $P\subseteq \GL_n$ a parabolic subgroup containing the Borel $B$ of upper-triangular matrices, and let $M_P$ be its Levi subgroup containing the torus $T$ of diagonal matrices. We fix $\widetilde P\subseteq P$ a Zariski closed algebraic subgroup containing $M_P$ and we consider the algebraic representation $\LLbar\vert_{{\widetilde P}}$ of $\widetilde P$ over $\F$.

\begin{definit1}[Definition \ref{goodsubqt}]
A subquotient of $\LLbar\vert_{{\widetilde P}}$ is a {\it good} subquotient if its restriction to the center $Z_{M_{P}}$ of $M_P$ is a (direct) sum of isotypic components of $\LLbar\vert_{Z_{M_{P}}}$.
\end{definit1}

Note that an isotypic component of $\LLbar\vert_{Z_{M_{P}}}$ carries an action of $M_{P}$ (Lemma \ref{action}). Hence, viewing an isotypic component of $\LLbar\vert_{Z_{M_{P}}}$ as a representation of $\widetilde P$ via the surjection $\widetilde P\twoheadrightarrow M_P$, one can see $\LLbar\vert_{{\widetilde P}}$ as a successive extension of such isotypic components (Lemma \ref{filtr}). On the $\GL_n(\Qp)$-side, the isotypic components of $\LLbar\vert_{Z_{M_{P}}}$ will play the role of irreducible constituents. Note that the isotypic components of $\LLbar\vert_{Z_{M_{P}}}$ are by definition all distinct.

To an isotypic component $C$ of $\LLbar\vert_{Z_{M_{P}}}$, we associate a parabolic subgroup $P(C)$ of $\GL_n$ containing $B$ as follows. Let $\lambda\in X(T)=\Hom_{\rm Gr}(T,{\mathbb G}_{\rm m})$ be any weight such that $C$ is the isotypic component of $\lambda\vert_{Z_{M_{P}}}$ and define (see (\ref{l'}))
\[\lambda'\defeq \frac{1}{\vert W(P)\vert}\sum_{w'\in W(P)}w'(\lambda)\ \in \ X(T)\otimes_{\Z}\Q,\]
where $W(P)$ is the Weyl group of $M_P$. Let $\theta$ be the highest weight of $\LLbar\vert_T$ and $w$ in the Weyl group of $\GL_n$ such that $w(\lambda')$ is dominant with respect to $B$. Then one can check that (see Proposition \ref{parabolicprop})
\[\theta-w(\lambda')= \sum_{\alpha\in S}n_\alpha \alpha,\]
where $S$ is the set of simple roots of $\GL_n$ (with respect to $B$) and the $n_\alpha$ are in $\Q_{\geq 0}$. Then $P(C)$ is by definition the parabolic subgroup of $\GL_n$ corresponding to the subset $\{\alpha\in S : n_\alpha \ne 0\}$ of $S$. We denote by $P(C)^-$ its opposite parabolic subgroup.

We now go back to the above global setting. Assuming a weak genericity condition on $\rbar_{\tilde{v}}$, one can replace $\rbar_{\tilde{v}}$ by a suitable conjugate so that the image of $\rbar_{\tilde{v}}$ is contained in the $\F$-points of a Zariski closed algebraic subgroup $\widetilde P_{\rbar_{\tilde{v}}}$ of a parabolic $P_{\rbar_{\tilde{v}}}$ as above which is ``as small as possible'' (see Definition \ref{gooddef} and Theorem \ref{choicegood}). The following conjecture is part of Conjecture \ref{theconj} (see Definition \ref{compatible2} and Definition \ref{compatible1}).

\begin{conj1}\label{conjIntro}
Assume that $\rbar_{\tilde{v}}$ has distinct irreducible constituents and that the ratio of any two $1$-dimensional constituents is not in $\{\omega,\omega^{-1}\}$, where $\omega$ is the mod $p$ cyclotomic character. Then we have a $\GL_n(\Qp)$-equivariant isomorphism for some integer $d\geq 1$:
\[S(U^{v},\F)[\m]\cong \big(\Pi_{\tilde v}\otimes(\omega^{n-1}\circ{\det})\big)^{\oplus d},\]
where $\Pi_{\tilde v}$ is an admissible smooth representation of $\GL_n(\Qp)$ over $\F$ of finite length with distinct irreducible constituents such that there exists a bijection $\Phi$ between the {\upshape(}finite{\upshape)} set of subquotients of $\Pi_{\tilde v}$ and the {\upshape(}finite{\upshape)} set of good subquotients of $\LLbar\vert_{{\widetilde P}_{\rbar_{\tilde{v}}}}$ satisfying the following properties:
\begin{enumerate}
\item$\Phi$ respects inclusions, and thus extends to a bijection between the sets of all subquotients on both sides;
\item$\Phi^{-1}$ sends an isotypic component $C$ of $\LLbar\vert_{Z_{M_{P_{\rbar_{\tilde{v}}}}}}$ to an irreducible constituent of $\Pi_{\tilde v}$ of the form $\Ind_{P(C)^-(\Qp)}^{\GL_n(\Qp)}\pi(C)$, where $\pi(C)$ is a supersingular representation of $M_{P(C)}(\Qp)$ over $\F$.
\end{enumerate}
\end{conj1}

When $K$ is not necessarily $\Qp$, the conjecture is completely analogous, defining $\LLbar$ by
\[\LLbar\defeq  \bigotimes_{\gKQ}\Big(\bigotimes_{i=1}^{n-1}{\bigwedge}^{\!\!i}_{\F}{\rm Std}\Big),\]
replacing $\widetilde P$ by $\widetilde P^{\gKQ}\defeq \underbrace{\widetilde P\times\cdots\times\widetilde P}_{\gKQ}$ and taking isotypic components of $\LLbar\vert_{Z_{M_{P}}}$ for the {\it diagonal embedding} $Z_{M_P}\hookrightarrow Z_{M_{P}}^{\gKQ}$ in the definition of good subquotients of $\LLbar\vert_{{\widetilde P}^{\gKQ}}$.

\begin{ex1}\label{exIntro}
(i) If $\rbar_{\tilde{v}}$ is irreducible, then $\widetilde P_{\rbar_{\tilde{v}}}=\GL_n=M_{P_{\rbar_{\tilde{v}}}}$ and there is only one isotypic component $C$ in $\LLbar\vert_{Z_{\GL_n}}$. It is such that $P(C)=\GL_n$: the representation $\Pi_{\tilde v}$ in Conjecture \ref{conjIntro} is irreducible and supersingular.\\
(ii) If $\rbar_{\tilde{v}}$ is semisimple, then $\widetilde P_{\rbar_{\tilde{v}}}=M_{P_{\rbar_{\tilde{v}}}}$, and since the direct sum decomposition of $\LLbar\vert_{Z_{M_{P_{\rbar_{\tilde{v}}}}}}$ into isotypic components for the (diagonal) $Z_{M_{P_{\rbar_{\tilde{v}}}}}$-action is a direct sum decomposition as a $\widetilde P_{\rbar_{\tilde{v}}}=M_{P_{\rbar_{\tilde{v}}}}$-representation, we see that the representation $\Pi_{\tilde v}$ in Conjecture \ref{conjIntro} is also semisimple.\\
(iii) If $K=\Qp$ and $n=2$, we have $\LLbar={\rm Std}$. When $\rbar_{\tilde{v}}$ is irreducible, by (i) the representation $\Pi_{\tilde v}$ of $\GL_2(\Qp)$ in Conjecture \ref{conjIntro} is supersingular. When $\rbar_{\tilde{v}}$ is reducible split, then $\widetilde P_{\rbar_{\tilde{v}}}=T=M_{P_{\rbar_{\tilde{v}}}}$, and $\LLbar\vert_{T}=\F \lambda_1 \oplus \F \lambda_2$, where $\lambda_i:{\rm diag}(x_1,x_2)\mapsto x_i$, $i\in \{1,2\}$. There are two isotypic components $C=\F \lambda_1$ or $C=\F\lambda_2$, both with $P(C)=B$: the representation $\Pi_{\tilde v}$ in Conjecture \ref{conjIntro} is a direct sum of two irreducible principal series. Finally, when $\rbar_{\tilde{v}}$ is reducible nonsplit, then $\widetilde P_{\rbar_{\tilde{v}}}=B$, $\LLbar\vert_{B}$ is a nonsplit extension of $\F \lambda_2$ by $\F \lambda_1$ and $\Pi_{\tilde v}$ is a nonsplit extension between two irreducible principal series. Note that Conjecture \ref{conjIntro} is known in that case (\cite{CS1}, \cite{CS2} for $\rbar_{\tilde{v}}$ irreducible, \cite[Cor.7.40]{breuil-ding} for arbitrary $\rbar_{\tilde{v}}$, all generalizing methods of \cite{emerton-local-global}).\\
(iv) For $K$ arbitrary (unramified) and $n=2$, see Example \ref{exemples} and Example $1$ of \S\ref{exemples5}. 
\end{ex1}

Conjecture \ref{conjIntro} only gives part of the picture. For instance there should be reducible subquotients of $\Pi_{\tilde v}$ which are also parabolic inductions $\Ind_{P(C)^-(\Qp)}^{\GL_n(\Qp)}\pi(C)$ with $\pi(C)$ of the form $\pi(C)\cong \pi_1(C)\otimes\cdots\otimes\pi_d(C)$, where the (reducible) $\pi_i(C)$ have themselves the same form as $\Pi_{\tilde v}$ but for the smaller $\GL_{n_i}(K)$ appearing in the Levi $M_{P(C)}(K)$ (which gives a ``fractal'' flavour to the whole picture!). In fact, it is possible that, in the end, this ``fractal'' picture will automatically follow from property (ii) in Conjecture \ref{conjIntro} (i.e.\ from the statement for {\it irreducible} subquotients only), as one can already see in many of the examples of \S\ref{exemples5} using the work of Hauseux (\cite{Ha2}, \cite{Ha3}), see Remark \ref{comments}(iv). Also some parabolic (possibly reducible) inductions as above should be deduced from others by a permutation on the factors $\pi_i(C)$. Tracking down all these internal symmetries (with the various twists by characters that occur) and all the implications between them is not really difficult but a bit tedious, as the reader will see from the technical lemmas in \S\ref{compatible1sec} (see e.g.\ Proposition \ref{cons2}). The interested reader should maybe first have a look at the various examples in \S\ref{exemples5} before going into the full combinatorics.
 
Finally, the full picture has to take into account the Galois action. There is a simple way to extend Colmez's functor from representations of $\GL_2(\Qp)$ to representations of $\GL_n(K)$ that we recall now (see \cite{breuil-foncteur} or \S\ref{covariant}). Let $\xi:{\mathbb G}_{\rm m}\rightarrow T$ be the cocharacter $x\mapsto {\rm diag}(x^{n-1},x^{n-2},\dots,1)$ and $N_1\defeq  {\rm Ker}(N_0\buildrel \ell\over\longrightarrow {\mathcal O}_{K} \buildrel {\rm{trace}}\over \longrightarrow {\mathbb Z}_p)$, where $N_0$ is the unipotent radical of $B(\cO_K)$ and the map $\ell$ is the sum of the entries on the first diagonal (following the notation of \cite{schneider-vigneras}). Let $\pi$ be a smooth representation of $\GL_n(K)$ over $\F$ and endow the algebraic dual $(\pi^{N_1})^\vee$ of $\pi^{N_1}$ with the residual $\F\bbra{N_0/N_1}\simeq \F\bbra{\Zp}\simeq \F\bbra{X}$-module structure (where $X\defeq [1]-1$), an action of $\Zp^\times$ and an endomorphism $\psi$ which commutes with the $\Zp^\times$-action by
\[\left\{\begin{array}{lll}
(xf)(v)&\defeq &f(\xi(x^{-1})v),\ x\in \Zp^\times,\ f\in (\pi^{N_1})^\vee,\ v\in \pi^{N_1}\\
\psi(f)(v)&\defeq &f\big(\sum _{N_1/\xi(p)N_1\xi(p)^{-1}}n_1\xi(p)v\big),\ f\in (\pi^{N_1})^\vee,\ v\in \pi^{N_1}.
\end{array}\right.\]
Then one defines a covariant left exact functor $V$ from the category of smooth representations of $\GL_n(K)$ over $\F$ to the category of (filtered) direct limits of continuous finite-dimensional representations of $\gp$ over $\F$ by
\begin{equation}\label{functorIntro}
V(\pi)\defeq  \big(\ilim{D} {\bf V}^\vee(D)\big)\otimes \delta,
\end{equation}
where the inductive limit is taken over the continuous morphisms of $\F\bbra{X}$-modules $h:(\pi^{N_1})^\vee\rightarrow D$, where $D$ is an \'etale $(\varphi,\Gamma)$-module of finite rank over $\F\ppar{X}$ and $h$ intertwines the actions of $\Zp^\times$ (recall $\Gamma\simeq \Zp^\times$), commutes with $\psi$ and is surjective when tensored by $\F\ppar{X}$. (Here ${\bf V}^\vee$ is Fontaine's contravariant functor associating a representation of $\gp$ to $D$ and recall that any \'etale $(\varphi,\Gamma)$-module is endowed with an endomorphism $\psi$ which is left inverse to the Frobenius $\varphi$.) In (\ref{functorIntro}), $\delta$ is a certain power of $\omega$ which is here for normalization issues (see Example \ref{exdelta}, see also the end of \S\ref{Cgroup}). In general, one doesn't know when $V(\pi)$ is nonzero or if it is finite-dimensional.

Using (\ref{functorIntro}), one can strengthen Conjecture \ref{conjIntro} (when $K=\Qp$) so that it takes into account the action of $\gp$ as follows.

\begin{conj1}[see Definition \ref{compatible1} and Conjecture \ref{theconj}]\label{conjIntrobis}
There is a bijection $\Phi$ as in Conjecture \ref{conjIntro} that moreover commutes with the action of $\gp$ in the following sense: for each subquotient $\Pi'_{\tilde v}$ of $\Pi_{\tilde v}$ one has $V(\Pi'_{\tilde v})=\Phi(\Pi'_{\tilde v})\circ \rbar_{\tilde{v}}$. {\upshape(}Recall that $\Phi(\Pi'_{\tilde v})$ is an algebraic representation of ${\widetilde P}_{\rbar_{\tilde{v}}}$ over $\F$ and that $\rbar_{\tilde{v}}$ takes values in ${\widetilde P}_{\rbar_{\tilde{v}}}(\F)$.{\upshape)}
\end{conj1}

If $K$ is not necessarily $\Qp$, then by definition $\Phi(\Pi'_{\tilde v})$ is an algebraic representation of $\widetilde P_{\rbar_{\tilde{v}}}^{\gKQ}$ and there is a completely analogous conjecture replacing $\Phi(\Pi'_{\tilde v})\circ \rbar_{\tilde{v}}$ by $\Phi(\Pi'_{\tilde v})\circ (\rbar_{\tilde{v}}^{\sigma})_{\sigma\in {\gKQ}}$, which is again a representation of $\gp$.

In particular the functor $V$, when applied to $\Pi_{\tilde v}$ and its subquotients $\Pi'_{\tilde v}$, should behave like an exact functor. Note that Conjecture \ref{conjIntrobis} is known when $K=\Qp$ and $n=2$ by the same references as in Example \ref{exIntro}(iii). In the special case $\Pi'_{\tilde v}=\Pi_{\tilde v}$, Conjecture \ref{conjIntrobis} implies in particular

\begin{conj1}[Conjecture \ref{theconjbar}]\label{imageVIntro}
The functor $V$ induces an isomorphism
\begin{equation*}
V\big(S(U^{v},\F)[\m]\otimes (\omega^{-(n-1)}\circ{\det})\big)\cong \Big(\ind_K^{\otimes\Qp}\!\big(\bigotimes_{i=1}^{n-1}{\bigwedge}^{\!\!i}_{\F}\rbar_{\tilde{v}}\big)\Big)^{\oplus d},
\end{equation*}
where $\ind_K^{\otimes\Qp}$ is the {\it tensor induction} from $\gK$ to $\gp$.
\end{conj1}

The statement in Conjecture \ref{imageVIntro} makes sense even if $K$ is ramified, and we conjecture it for an arbitrary finite extension $K$ of $\Qp$ and an arbitrary representation $\rbar_{\tilde{v}}$ (see Conjecture \ref{theconjbar}). In fact, using $C$-parameters (\cite{BG}), it can even be formulated in a more intrinsic way and in a more general global setting, see Conjecture \ref{conj:generale}.

\begin{rem1}
Assuming $K=\Qp$, the first appearance of the $\gp$-represen\-tation on the right-hand side of the isomorphism in Conjecture \ref{imageVIntro} is in \cite{BH}, where its ``ordinary part'' was related to the ``ordinary part'' of $S(U^{v},\F)[\m]$ (see Theorem \ref{enns} for an improvement). Note that the algebraic representation $\LLbar$ of $\GL_n$ is {\it not} irreducible for $n>2$. One could have thought about using the irreducible algebraic representation of $\GL_n$ of highest weight $\theta$ instead of the reducible $\LLbar$ to make predictions (at least for $p$ big enough the latter strictly contains the former as a direct factor). However, we chose the representation $\LLbar$. One reason is that it can also be seen as a representation of ${\GL_n\times\cdots\times\GL_n}$ ($n-1$ times) in an obvious way -- in which case a better notation is $\overline L^{\boxtimes}\defeq \boxtimes_{i=1}^{n-1}{\bigwedge}^{i}_{\F}{\rm Std}$ -- and one can hope to state a stronger variant of Conjecture \ref{conjIntrobis} replacing $\overline L^{\otimes}$ by $\overline L^{\boxtimes}$ and $\Phi(\Pi'_{\tilde v})\circ \rbar_{\tilde{v}}$ by $\Phi(\Pi'_{\tilde v})\circ (\rbar_{\tilde{v}},\rbar_{\tilde{v}},\dots,\rbar_{\tilde{v}})$ (see \cite{Za}, \cite{Zabrprod} where such a possibility is mentioned). However one has to be careful with defining a ``multivariable'' functor $V$ in that context (there is a tentative definition in \cite{Za} when $K=\Qp$ generalizing (\ref{functorIntro}), but see Remark \ref{Wu} when $n=2$ and $K\ne \Qp$).
\end{rem1}

If a representation $\Pi_{\tilde v}$ as in Conjecture \ref{conjIntrobis} exists, we do hope that it will realize a mod $p$ local Langlands correspondence for $\GL_n(K)$.

\subsection{Results}\label{resultsIntro}

Let us now describe our main results when $n=2$ and $K=\Qpf$ is unramified. For a finite place $\tilde w$ of $F$ we denote by $R_{\rbar_{\tilde{w}}}^\square$ the (unrestricted) framed deformation ring of $\rbar_{\tilde{w}}\defeq \rbar\vert_{{\rm Gal}(\overline F_{\tilde{w}}/F_{\tilde w})}$ over $W(\F)$. We let $I_K\subseteq \gK$ be the inertia subgroup and $\omega_{f'}$ for $f'\in \{f,2f\}$ be Serre's fundamental character of level $f'$. We make the following extra assumptions on $F$, $H$, $\rbar$ and $U^v=\prod_{w\ne v}U_w$ (recall we assumed $p$ inert in $F^+$ for simplicity):
\begin{enumerate}
\item$F/F^+$ is unramified at all finite places of $F^+$;
\item$H$ is quasi-split at all finite places of $F^+$;
\item$\rbar\vert_{{\Gal}(\overline F/F(\sqrt[p]{1}))}$ is adequate (\cite[Def.2.20]{Th2});
\item$\rbar_{\tilde{w}}$ is unramified if ${\tilde{w}}\vert_{F^+}$ is inert in $F$;
\item $R_{\rbar_{\tilde{w}}}^\square$ is formally smooth over $W(\F)$ if $\rbar_{\tilde{w}}$ is ramified and ${\tilde{w}}\vert_{F^+}\ne v$;
\item$\rbar_{\tilde{v}}\vert_{I_K}$ is, up to twist, of one of the following forms:
\[\left\{\begin{array}{lll}
\rbar_{\tilde{v}}\vert_{I_K}&\cong &\begin{pmatrix}\omega_f^{(r_0+1)+\cdots+p^{f-1}(r_{f-1}+1)}&0\\0&1\end{pmatrix},\\
\rbar_{\tilde{v}}\vert_{I_K}&\cong &\begin{pmatrix}\omega_{2f}^{(r_0+1)+\cdots+p^{f-1}(r_{f-1}+1)}&0\\0&\omega_{2f}^{p^f(\mathrm{same})}\end{pmatrix},
\end{array}\right.\]
where the $r_i$ satisfy the following bounds:
\begin{equation}\label{strongIntro}
\left\{\!\!\begin{array}{llllll}
\max\{12,2f\!-\!1\} \!\!&\!\!\le \!\!&\!\! r_j \!\!&\!\!\le \!&\!\! p-\max\{15,2f\!+\!2\} \!& \!\!\text{if $j > 0$ or $\rbar_{\tilde{v}}$ is reducible,}\\
\max\{13,2f\} \!\!&\!\!\le \!\!&\!\! r_0 \!\!&\!\!\le \!&\!\! p-\max\{14,2f\!+\!1\} \!& \!\!\text{if $\rbar_{\tilde{v}}$ is irreducible;}
\end{array}\right.
\end{equation}
\item$U_{w}$ is maximal hyperspecial in $H(F_{w}^+)$ if $w$ is inert in $F$.
\end{enumerate}
(We also need to fix a place $v_1$ which splits in $F$, where nothing ramifies and $U_{v_1}$ is contained in the Iwahori subgroup at $v_1$, we forget that here along with the set $\Sigma$ of bad places and the definition of the ideal $\m$.)

\begin{thm1}[Theorem \ref{specialcase1}]\label{thmIntro1}
Assume $n=2$, $K/\Qp$ unramified, and the above conditions (i)--(vii). Then Conjecture \ref{imageVIntro} holds.
\end{thm1}

We sketch the proof of Theorem \ref{thmIntro1}. We denote by $I_1$ the pro-$p$ Iwahori subgroup in $\GL_2(\cO_K)$ and set
\[\rhobar\defeq \rbar_{\tilde{v}}(1)\ \ \ \ \ \Pi\defeq S(U^{v},\F)[\m].\]
Note that the central character of $\Pi$ is ${\det(\rhobar)}\omega^{-1}$ (Lemma \ref{central}). There are two main steps in the proof which involve quite different arguments:
\begin{enumerate}
\item one proves a $\gp$-equivariant injection $(\ind_K^{\otimes\Qp}\!\rhobar)^{\oplus d}\hookrightarrow V(\Pi)$;
\item one proves $\dim_{\F}V(\Pi)\leq 2^fd$ ($=\dim_{\F}(\ind_K^{\otimes\Qp}\!\rhobar)^{\oplus d}$).
\end{enumerate}

We first sketch the proof of (i). Arguing as in the proof of \cite[Prop.8.2.6]{BHHMS1}, there is an integer $d\geq 1$ and a $\GL_2({\mathcal O}_K)K^\times$-equivariant isomorphism $\Pi^{K_1}\cong D_0(\rhobar)^{\oplus d}$, where $D_0(\rhobar)$ is defined as in \cite[\S13]{BP} (see Corollary \ref{D0r}). Taking into account the action of $\smatr{0}{1}{p}{0}$ on $\Pi^{I_1}\subseteq \Pi^{K_1}$, one can promote this isomorphism to an isomorphism of {\it diagrams}:

\begin{thm1}[{\cite[Thm.1.3]{DoLe}} when $d=1$, Theorem \ref{nonminimal} when $d>1$]\label{thmIntro2}
There is a diagram $D(\rhobar)=(D_1(\rhobar)\hookrightarrow D_0(\rhobar))$ only depending on $\rhobar$ such that one has an isomorphism of diagrams:
\[D(\rhobar)^{\oplus d}\cong (\Pi^{I_1}\hookrightarrow \Pi^{K_1}).\]
\end{thm1}

Theorem \ref{thmIntro2} can actually be made stronger, i.e.~one can show that certain constants $\nu_i\in\F^\times$ associated to the weight cycling on $D_1(\rhobar)\cong D_0(\rhobar)^{I_1}$ as in \cite[\S 6]{breuil-IL} (up to suitable normalization) are as predicted in \cite[Thm.6.4]{breuil-IL}. When $d=1$, Theorem \ref{thmIntro2} (and its strengthening) is entirely due to Dotto and Le (\cite[Thm.1.3]{DoLe}). When $d>1$, we check from their proof that the action of $\smatr{0}{1}{p}{0}$ on $\Pi^{I_1}\cong (D_0(\rhobar)^{I_1})^{\oplus d}$ ``respects'' each copy of $D_0(\rhobar)^{I_1}$. Note that Theorem \ref{thmIntro2} holds under much weaker bounds on the $r_i$ than the bounds (\ref{strongIntro}), see \S\ref{globalp}. 
 
Then item (i) above follows from the following purely local result.

\begin{thm1}[Theorem \ref{thm:main-tensor-ind}]\label{thmIntro3}
Let $\pi$ be an {\upshape(}admissible{\upshape)} smooth representation of $\GL_2(K)$ over $\F$ such that one has an isomorphism of diagrams $D(\rhobar)^{\oplus d}\cong (\pi^{I_1}\hookrightarrow \pi^{K_1})$. Then one has a $\gp$-equivariant injection $(\ind_K^{\otimes\Qp}\!\rhobar)^{\oplus d}\hookrightarrow V(\pi)$.
\end{thm1}

The proof of Theorem \ref{thmIntro3} is a long and technical computation of $(\varphi,\Gamma)$-modules that is given in \S\ref{tensorinduction}. It uses the previous computations in \cite{breuil-IL} and the bounds (\ref{strongIntro}) (though one can slightly weaken them, see (\ref{eq:6})).

We now sketch the (longer) proof of (ii). We let $Z_1$ be the center of $I_1$ (or of $K_1$) and $\m_{I_1/Z_1}$ the maximal ideal of the Iwasawa algebra $\F\bbra{I_1/Z_1}$. The main idea is to focus on the structure of the (algebraic) dual $\pi^\vee$ as an $\F\bbra{I_1/Z_1}$-module and to use the results of \cite{BHHMS1}. Recall that the graded ring $\gr(\F\bbra{I_1/Z_1})$ for the $\m_{I_1/Z_1}$-adic filtration (we use the normalization of \cite[\S I.2.3]{LiOy}) is not commutative, but contains a regular sequence of central elements $(h_0,\dots,h_{f-1})$ such that $R\defeq \gr(\F\bbra{I_1/Z_1})/(h_0,\dots,h_{f-1})$ is a commutative polynomial algebra in $2f$ variables $\F[y_i,z_i, 0\leq i\leq f-1]$ (see \cite[\S5.3]{BHHMS1} and (\ref{elementsyi}), (\ref{gri1})). We let $J\defeq (y_iz_i, h_i, 0\leq i\leq f-1)$ (an ideal of $\gr(\F\bbra{I_1/Z_1})$) and define
\begin{equation}\label{rbarIntro}
\overline{R}\defeq \gr(\F\bbra{I_1/Z_1})/J\simeq \F[y_i,z_i, 0\leq i\leq f-1]/(y_iz_i, 0\leq i\leq f-1).
\end{equation}
Then $\mathfrak{p}_0\defeq (z_i, 0\leq i\leq f-1)$ is one of the $2^f$ minimal prime ideals of $\overline R$. If $N$ is any finite type $\gr(\F\bbra{I_1/Z_1})$-module killed by a power of $J$, one can define its multiplicity $m_{\mathfrak{p}_0}(N)\in \Z_{\geq 0}$ at ${\mathfrak{p}_0}$, see (\ref{def:multiatp}).

For \ $\pi$ \ a \ smooth \ representation \ of \ $\GL_2(K)$ \ over \ $\F$ \ with \ a \ central \ character, we endow $\pi^\vee$ with the $\m_{I_1/Z_1}$-adic filtration and we let $\gr(\pi^\vee)$ be the associated graded $\gr(\F\bbra{I_1/Z_1})$-module.

\begin{thm1}[Theorem \ref{thm:upperbound}]\label{thmIntro4}
Let $\pi$ be an {\upshape(}admissible{\upshape)} smooth representation of $\GL_2(K)$ over $\F$ satisfying the following two properties:
\begin{enumerate}
\item there is a $\GL_2(\oK)K^\times$-equivariant isomorphism $D_0(\rhobar)^{\oplus d}\simeq \pi^{K_1}$;
\item for any character $\chi:I\rightarrow \F^\times$ appearing in $\pi[\m_{I_1/Z_1}]$ there is an equality of multiplicities
\[[\pi[\m_{I_1/Z_1}^3]:\chi]=[\pi[\m_{I_1/Z_1}]:\chi].\]
\end{enumerate}
Then $\gr(\pi^\vee)$ is killed by $J$ and one has $m_{\mathfrak{p}_0}(\gr(\pi^\vee))\leq 2^fd$.
\end{thm1}

By the proof of \cite[Cor.5.3.5]{BHHMS1}, property (ii) in Theorem \ref{thmIntro4} implies that $\gr(\pi^\vee)$ is killed by $J$. By an explicit computation (using both properties (i) and (ii)), one proves in Theorem \ref{thm:cycle-pi} that there is a surjection of $\overline R$-modules
\[(\oplus_{\lambda\in \mathscr P}R/{\mathfrak a}(\lambda))^{\oplus d}\twoheadrightarrow \gr(\pi^\vee),\]
where $\mathscr P$ is a combinatorial finite set associated to $\rhobar$ (in bijection with the set of $\chi$ appearing in $\pi[\m_{I_1/Z_1}]$, see \S\ref{combi}) and the ${\mathfrak a}(\lambda)$ are explicit ideals of $R$ containing the image of $J$ (see Definition \ref{def:a(lambda)}). Then Theorem \ref{thmIntro4} follows from the equality $m_{\mathfrak{p}_0}(\oplus_{\lambda\in \mathscr P}R/{\mathfrak a}(\lambda))=2^f$ which is an easy computation.

Arguing as in \cite{BHHMS1}, the representation $\Pi$ satisfies all assumptions of Theorem \ref{thmIntro4}, see Corollary \ref{D0r} and Theorem \ref{mi13}. Hence the upper bound in item (ii) below Theorem \ref{thmIntro1} follows from Theorem \ref{thmIntro4} combined with the next result:

\begin{thm1}[Corollary \ref{cor:upperbound}]\label{thmIntro5}
Let $\pi$ be an admissible smooth representation of $\GL_2(K)$ over $\F$ with a central character such that $\gr(\pi^\vee)$ is killed by some power of $J$. Then one has $\dim_{\F}V(\pi)\leq m_{\mathfrak{p}_0}(\gr(\pi^\vee))$.
\end{thm1}

We prove Theorem \ref{thmIntro5} by first associating to $\pi$ an ``\'etale $(\varphi,\cO_K^\times)$-module over $A$'' (Definition \ref{phiok}). This is the ``multivariable $(\varphi,\Gamma)$-module'' mentioned at the end of \S\ref{preamble}. Though one could probably give a more direct proof without explicitly introducing them, these \'etale $(\varphi,\cO_K^\times)$-modules are important for our finite length results below and are likely to play a role later, so we describe them now.

We start with the ring $A$. Let $\F\bbra{N_0}\simeq \F\bbra{\cO_K}$ be the Iwasawa algebra of the unipotent radical $N_0$ of $B(\cO_K)$. Then $\F\bbra{N_0}\simeq \F\bbra{Y_0,\dots,Y_{f-1}}$, where the $Y_i$ are eigenvectors for the action of the finite torus on $\F\bbra{N_0}$ (see (\ref{elementsyi})). Let $S$ be the {\it multiplicative system} in $\F\bbra{N_0}$ generated by the $Y_i$. The filtration on $\F\bbra{N_0}$ by powers of its maximal ideal $\m_{N_0}$ naturally extends to a filtration on the localized ring $\F\bbra{N_0}_S$ and we define $A$ to be the completion of $\F\bbra{N_0}_S$ (where $\F\bbra{N_0}_S$ denotes the localization of $\F\bbra{N_0}$ at $S$) for this filtration (\cite[\S I.3.4]{LiOy}). The ring $A$ is {\it not} local, but it is a regular noetherian domain (Corollary \ref{lem:grA}) and a complete filtered ring in the sense of \cite[\S I.3.3]{LiOy} with associated graded ring $\gr(A)\simeq \gr(\F\bbra{N_0}_S)$ (see Remark \ref{rem:on_filtered_Amod}(iii) for a concrete description of $A$). Most importantly, the natural action of $\cO_K^\times$ on $\F\bbra{N_0}\simeq \F\bbra{\cO_K}$ by multiplication on $\cO_K$ extends by continuity to $A$ (Lemma \ref{actionoK}) and any ideal of $A$ preserved by $\cO_K^\times$ is either $0$ or $A$ (Corollary \ref{cor:idealsofA}). 

Let $\pi$ be an admissible smooth representation of $\GL_2(K)$ over $\F$ with a central character and recall that $\pi^\vee$ is endowed with the $\m_{I_1/Z_1}$-adic filtration (which, in general, {\it strictly} contains the $\m_{N_0}$-adic filtration). We endow $(\pi^\vee)_S\defeq \F\bbra{N_0}_S\otimes_{\F\bbra{N_0}}\pi^\vee$ with the tensor product filtration and define $D_A(\pi)$ as the completion of $(\pi^\vee)_S$. Then $D_A(\pi)$ is a complete filtered $A$-module such that $\gr(D_A(\pi))\simeq \gr((\pi^\vee)_S)$ (Lemma \ref{lem:grlocal}). The action of $\cO_K^\times$ on $\pi^\vee$ extends by continuity to $D_A(\pi)$, as well as the map
\[\psi:\pi^\vee\longrightarrow \pi^\vee,\ \ \ f\longmapsto \psi(f)\defeq \Big(v\in \pi\mapsto f(\xi(p)v)=f\big(\smatr{p}{0} {0} {1} v\big)\Big)\]
(Lemma \ref{lem:psicontinue}). The latter can be linearized into an $A$-linear morphism
\[\beta:D_A(\pi)\longrightarrow A\otimes_{\phi,A}D_A(\pi),\]
where $\phi$ is a Frobenius endomorphism on the characteristic $p$ ring $A$ (see (\ref{mapbeta}) for the definition of $\beta$, and \S \ref{ringA} for the definition of $\phi$ on $A$).

We let $\mathcal C$ be the abelian category of admissible smooth representations $\pi$ with a central character such that $\gr((\pi^\vee)_S)$ is a finite type $\gr(\F\bbra{N_0}_S)$-module. It follows from (\ref{rbarIntro}) that
\[\big(\gr(\F\bbra{I_1/Z_1})/J\big)[(y_0\cdots y_{f-1})^{-1}] \simeq  \F[y_0, \dots, y_{f-1}][(y_0\cdots y_{f-1})^{-1}] \simeq \gr(\F\bbra{N_0}_S)\]
which easily implies that, if $\gr(\pi^\vee)$ is killed by a power of $J$, then $\pi$ is in $\mathcal C$ (Proposition \ref{prop:finiteness}). In particular the representation $\Pi$ is in $\mathcal C$. Note that {\it any} finite length admissible smooth representation $\pi$ of $\GL_2(\Qp)$ over $\F$ with a central character is such that $\gr(\pi^\vee)$ is killed by a power of $J$ (Corollary \ref{qp}), hence is in $\mathcal C$.

For $\pi$ in $\mathcal C$, by general results of \cite{Lyubeznik}, there exists a largest quotient $D_A(\pi)^{\et}$ of $D_A(\pi)$ such that the map $\beta$ induces an isomorphism $\beta^{\et}:D_A(\pi)^{\et}\buildrel \sim\over\rightarrow A\otimes_{\phi,A}D_A(\pi)^{\et}$ (see the beginning of \S\ref{multivariablepsi}). We let $\varphi:D_A(\pi)^{\et}\rightarrow D_A(\pi)^{\et}$ such that 
$\Id\otimes\varphi=(\beta^{\et})^{-1}$. Then $D_A(\pi)^{\et}$ equipped with $\varphi$ and the induced action of $\cO_K^\times$ is our {\it \'etale $(\varphi,\cO_K^\times)$-module over $A$} associated to $\pi$ in $\mathcal C$.

\begin{thm1}[Proposition \ \ \ref{prop:exactnessDA}, \ \ Corollary \ \ \ref{cor:psiiso}, \ \ Theorem \ \ \ref{thm:exactnessDA} \ \ and \ \ Corollary \ \ \ref{cor:upperbound}]\label{thmIntro6}\

  \begin{enumerate}
  \item If $\pi$ is in $\mathcal C$, then $D_A(\pi)$ and $D_A(\pi)^{\et}$ are finite projective $A$-modules and $\rk_A(D_A(\pi)^{\et})\!\leq m_{\mathfrak{p}_0}(\gr(\pi^\vee))$.
  \item The {\upshape(}contravariant{\upshape)} functors $\pi\rightarrow D_A(\pi)$ and $\pi\rightarrow D_A(\pi)^{\et}$ are exact on the
    abelian category $\mathcal C$.
  \end{enumerate}
\end{thm1}

One key ingredient in the proof of Theorem \ref{thmIntro6} (cf.\ the proof of Proposition \ref{prop:OK_modules}) is that if the annihilator of an $A$-module endowed with an $A$-semilinear $\cO_K^\times$-action is nonzero, then this annihilator is $A$ (since there are no proper nonzero ideals of $A$ which are preserved by $\cO_K^\times$, see above) and hence the $A$-module must be $0$.

For a smooth representation $\pi$ of $\GL_2(K)$ over $\F$ such that $\dim_{\F}V(\pi)<+\infty$, we denote by $D_\xi^\vee(\pi)$ the unique \'etale $(\varphi,\Gamma)$-module over $\F\ppar{X}$ such that $V(\pi)= {\bf V}^\vee(D_\xi^\vee(\pi))\otimes \delta$ (see (\ref{functorIntro})). We denote by $\tr : A\rightarrow\F\ppar{X}$ the ring morphism induced by the trace $\tr : \F\bbra{N_0}\rightarrow\F\bbra{\Zp}\simeq\F\bbra{X}$.

\begin{thm1}[Theorem \ref{thm:functors_comparison}]\label{thmIntro7}
If $\pi$ is in $\mathcal C$, then we have an isomorphism of \'etale $(\varphi,\Gamma)$-modules over $\F\ppar{X}$:
\[D_A(\pi)^{\et}\otimes_A\F\ppar{X}\buildrel\sim\over\longrightarrow D^\vee_{\xi}(\pi).\]
In particular, $\dim_{\F}V(\pi)=\rk_A(D_A(\pi)^{\et})<+\infty$ and the functor $\pi\longmapsto V(\pi)$ in {\upshape(\ref{functorIntro})} is exact on the category $\mathcal C$.
\end{thm1}

The proof essentially follows by a careful unravelling of all the definitions and constructions involved. The last statement follows from the first and from Theorem \ref{thmIntro6}. 

Theorem \ref{thmIntro7} and Theorem \ref{thmIntro6}(i) imply in particular the bound on $V(\pi)$ in Theorem \ref{thmIntro5}, which finally proves Theorem \ref{thmIntro1}.

We see that the multivariable $(\varphi,\cO_K^\times)$-module $D_A(\pi)^{\et}$ plays an important role in the proof of Theorem \ref{thmIntro5}. One natural question therefore is to understand more the internal structure of $D_A(\Pi)^{\et}$ (at least conjecturally): does $D_A(\Pi)^{\et}$ only depend on $\rhobar$? Does it determine $\rhobar$? We plan to come back to these questions, as well as generalizations in higher dimension, in future work.

The modules $D_A(\Pi)^{\et}$ and $D_\xi^\vee(\Pi)$ are also crucial tools in the proof of our finite length results on the representation $\Pi$ which provide evidence to Conjecture \ref{conjIntro} and Conjecture \ref{conjIntrobis} and that we describe now.

\begin{thm1}[Theorem \ref{thm:gen-socleglob}]\label{thmIntro8}
Assume moreover $d=1$, i.e.\ $\Pi^{K_1}\cong D_0(\rhobar)$ {\upshape(}the so-called {\it minimal} case{\upshape)}. Then the $\GL_2(K)$-representation $\Pi$ is generated by its $\GL_2(\cO_K)$-socle, in particular is of finite type.
\end{thm1}

Note that the last finiteness assertion in Theorem \ref{thmIntro8} (with $\Pi^{K_1}$ instead of the
$\GL_2(\cO_K)$-socle) was known for $\rhobar$ {\it non}-semisimple (and sufficiently generic) by \cite[Thm.1.6]{HuWang2}, but the proof there doesn't extend to the semisimple case.

We sketch the proof of Theorem \ref{thmIntro8}. Let $\Pi'\subseteq \Pi$ be a nonzero subrepresentation and $\Pi''\defeq \Pi/\Pi'$. As $\gr(\Pi^\vee)$ and hence its quotient $\gr(\Pi'^\vee)$ are killed by $J$, the representations $\Pi$, $\Pi'$, $\Pi''$ are all in $\mathcal C$, thus Theorem \ref{thmIntro6}(i) and Theorem \ref{thmIntro7} imply $\dim_{\F}V(\Pi')\leq m_{\mathfrak{p}_0}(\gr(\Pi'^\vee))$ and $\dim_{\F}V(\Pi'')\leq m_{\mathfrak{p}_0}(\gr(\Pi''^\vee))$. Since $V(\Pi'')\simeq V(\Pi)/V(\Pi')$ by the last statement in Theorem \ref{thmIntro7}, and since $m_{\mathfrak{p}_0}$ is an additive function by Lemma \ref{lem:Z-additive} (and Definition \ref{Zdef}), we deduce $\dim_{\F}V(\Pi')= m_{\mathfrak{p}_0}(\gr(\Pi'^\vee))$ and $\dim_{\F}V(\Pi'')= m_{\mathfrak{p}_0}(\gr(\Pi''^\vee))$ as we have seen that $\dim_{\F}V(\Pi)= m_{\mathfrak{p}_0}(\gr(\Pi^\vee))\ (=2^f)$. On the other hand, by computations analogous to the ones used in the proofs of Theorem \ref{thmIntro3} and Theorem \ref{thmIntro4}, we also have inequalities
\[m_{\mathfrak{p}_0}(\gr(\Pi'^\vee))\leq \mathrm{lg}(\soc_{\GL_2(\cO_K)}(\Pi'))\leq \dim_{\F}V(\Pi')\]
and thus we deduce
\begin{equation}\label{nonzeroIntro}
m_{\mathfrak{p}_0}(\gr(\Pi'^\vee))=\mathrm{lg}(\soc_{\GL_2(\cO_K)}(\Pi'))=\dim_{\F}V(\Pi')\ne 0.
\end{equation}
Now take $\Pi'$ to be the nonzero subrepresentation generated over $\GL_2(K)$ by the $\GL_2(\cO_K)$-socle of $\Pi$. We wish to prove $\Pi''=0$. As
\[\mathrm{lg}(\soc_{\GL_2(\cO_K)}(\Pi'))=\mathrm{lg}(\soc_{\GL_2(\cO_K)}(\Pi))=2^f=\dim_{\F}V(\Pi)\]
we already have by (\ref{nonzeroIntro}) and the exactness of $V$ that
\begin{equation}\label{zeroIntro}
m_{\mathfrak{p}_0}(\gr(\Pi''^\vee))=\dim_{\F}V(\Pi'')=0.
\end{equation}
To deduce $\Pi''=0$ from (\ref{zeroIntro}), we need the following key new ingredient: $\Pi$ is {\it essentially self-dual of grade (or codimension) $2f$}, i.e.\ ${\rm Ext}^{j}_{\F\bbra{I_1/Z_1}}\big(\Pi^{\vee},\F\bbra{I_1/Z_1}\big)=0$ if $j<2f$ and there is a $\GL_2(K)$-equivariant isomorphism
\begin{equation}\label{dualIntro}
{\rm Ext}^{2f}_{\F\bbra{I_1/Z_1}}\big(\Pi^{\vee},\F\bbra{I_1/Z_1}\big)\cong \Pi^{\vee}\otimes (\det(\rhobar)\omega^{-1}),
\end{equation}
where ${\rm Ext}^{2f}_{\F\bbra{I_1/Z_1}}(\Pi^{\vee},\F\bbra{I_1/Z_1})$ is endowed with the action of $\GL_2(K)$ defined by Kohlhaase in \cite[Prop.3.2]{Ko}. This follows by the same argument as in \cite[Thm.8.2]{HuWang2} (using Remark \ref{dim=f}). We then define $\widetilde \Pi$ as the admissible smooth representation of $\GL_2(K)$ over $\F$ such that
\begin{equation*}
\widetilde \Pi^{\vee}\otimes(\det(\rhobar)\omega^{-1})\cong \mathrm{Im}\Big({\rm Ext}^{2f}_{\F\bbra{I_1/Z_1}}\big(\Pi^{\vee},\F\bbra{I_1/Z_1}\big)\ra {\rm Ext}^{2f}_{\F\bbra{I_1/Z_1}}\big(\Pi''^{\vee},\F\bbra{I_1/Z_1}\big)\Big),
\end{equation*}
and by (\ref{dualIntro}) $\widetilde \Pi$ is a subrepresentation of $\Pi$. By (\ref{dualIntro}) and general results on ${\rm Ext}_\Lambda^{j}(-,\Lambda)$ for Auslander regular rings $\Lambda$, $\Pi''^{\vee}\subseteq \Pi^{\vee}$ is also of grade $2f$ if it is nonzero, and hence ${\rm Ext}^{2f}_{\F\bbra{I_1/Z_1}}(\Pi''^{\vee}\!,\F\bbra{I_1/Z_1})$ is nonzero if and only if $\Pi''\ne 0$. From the short exact sequence
\begin{equation}\label{exactIntro}
0\!\rightarrow \!\widetilde \Pi^{\vee}\!\otimes\!(\det(\rhobar)\omega^{-1})\!\rightarrow \!{\rm Ext}^{2f}_{\F\bbra{I_1/Z_1}}\big(\Pi''^{\vee},\F\bbra{I_1/Z_1}\big)\!\rightarrow \!{\rm Ext}^{2f+1}_{\F\bbra{I_1/Z_1}}\big(\Pi'^{\vee},\F\bbra{I_1/Z_1}\big)
\end{equation}
and the fact that the last ${\rm Ext}^{2f+1}$ has grade $\geq 2f+1$, we finally obtain:
\begin{equation}\label{zeroequivIntro}
\widetilde \Pi \textrm{\ is\ nonzero\ if\ and\ only\ if\ }\Pi''\textrm{\ is\ nonzero.}
\end{equation}
We now use the following general theorem.

\begin{thm1}[Theorem \ref{prop:cycle-Ext}]\label{thmIntro9}
Let $\pi$ be an admissible smooth representation of $\GL_2(K)$ over $\F$ with a central character such that $\gr(\pi^\vee)$ is killed by a power of $J$. Then\ \ \ the\ \ \ $\gr(\F\bbra{I_1/Z_1})$-module\ \ \ \ {\upshape(}for\ \ \ the\ \ \ $\m_{I_1/Z_1}$-adic\ \ \ filtration\ \ \ on\ \ \ ${\rm Ext}^{2f}_{\F\bbra{I_1/Z_1}}(\pi^{\vee},\F\bbra{I_1/Z_1})${\upshape)}:
\[\gr\Big({\rm Ext}^{2f}_{\F\bbra{I_1/Z_1}}\big(\pi^{\vee},\F\bbra{I_1/Z_1}\big)\Big)\]
is also finitely generated and annihilated by a power of $J$, and we have
\[m_{\mathfrak{p}_0}(\gr(\pi^\vee))=m_{\mathfrak{p}_0}\bigg(\gr\Big({\rm Ext}^{2f}_{\F\bbra{I_1/Z_1}}\big(\pi^{\vee},\F\bbra{I_1/Z_1}\big)\Big)\bigg).\]
\end{thm1}

From the injection in (\ref{exactIntro}) and from Theorem \ref{thmIntro9} applied to $\pi=\Pi''$ we have $m_{\mathfrak{p}_0}(\gr(\widetilde\Pi^\vee))\leq m_{\mathfrak{p}_0}(\gr(\Pi''^\vee))$, hence we obtain
\[m_{\mathfrak{p}_0}(\gr(\widetilde\Pi^\vee))=m_{\mathfrak{p}_0}(\gr(\Pi''^\vee))\buildrel{(\ref{zeroIntro})}\over =0.\]
This implies $\widetilde \Pi=0$ by (\ref{nonzeroIntro}) (applied to the subrepresentation $\Pi'=\widetilde\Pi$) and thus $\Pi''=0$ by (\ref{zeroequivIntro}), finishing the proof of Theorem \ref{thmIntro8}.

The following corollary immediately follows from Theorem \ref{thmIntro8} and from \cite[Thm.19.10(i)]{BP}.

\begin{cor1}[Theorem \ref{thm:gen-socleglob}]\label{thmIntro10}
Assume moreover $d=1$ and $\rhobar$ irreducible. Then the $\GL_2(K)$-representation $\Pi$ is irreducible and is a supersingular representation.
\end{cor1}

When $\rhobar$ is reducible (split), we can prove the following result.

\begin{thm1}[Theorem \ref{cor:pi-irredglob}]\label{thmIntro11}
Assume moreover $d=1$ and $\rhobar$ reducible, i.e.\ $\rhobar=\begin{pmatrix}\chi_{1} &0\\0 &\chi_2\end{pmatrix}$. Then one has
\begin{equation*}
\Pi=\Ind_{B(K)}^{\GL_2(K)}(\chi_1\otimes \chi_2\omega^{-1})\oplus \Pi' \oplus \Ind_{B(K)}^{\GL_2(K)}(\chi_2\otimes \chi_1\omega^{-1}),
\end{equation*}
where $\Pi'$ is generated by its $\GL_2(\cO_{K})$-socle and $\Pi'^{\vee}$ is essentially self-dual of grade $2f$, i.e.\ satisfies {\upshape(\ref{dualIntro})}. Moreover, when $f=2$, $\Pi'$ is irreducible and supersingular {\upshape(}and hence $\Pi$ is semisimple{\upshape)}. 
\end{thm1}

The fact that the two principal series in Theorem \ref{thmIntro11} occur as subobjects of $\Pi$ was already known (and is not difficult). To prove that they also occur as quotients (and that the obvious composition is the identity), we again crucially use the essential self-duality (\ref{dualIntro}). The rest of the statement follows from Theorem \ref{thmIntro8} and \cite[Thm.19.10(ii)]{BP}.

The following last corollary sums up the above results.

\begin{cor1}[Theorem \ref{specialcase2}]\label{thmIntro12}
Assume (i) to (vii) as at the beginning of \S\ref{resultsIntro} and assume $d=1$ as in Theorem \ref{thmIntro8}. Then Conjecture \ref{conjIntrobis} holds for $n=2$ and $\rhobar$ irreducible, or for $n=2$, $K$ quadratic and $\rhobar$ semisimple.
\end{cor1}

Note finally that when $f=2$, $\rhobar$ is {\it non}-semisimple (sufficiently generic) and $d=1$, Conjecture \ref{conjIntro} at least is known and follows from \cite[Thm.1.7]{HuWang2}.

\subsection{Notation}

We finish this introduction with some very general notation (many more will be defined in the text).

Throughout the text, we fix $\Qpbar$ an algebraic closure of $\Qp$ and $K$ an arbitrary finite extension of $\Qp$ in $\Qpbar$ with residue field $\Fq$, $q=p^f$ ($f\in \Z_{\geq 1}$). The field $K$ is unramified from \S\ref{goodcomponent} on. We also fix a finite extension $E$ of $\Qp$, with ring of integers $\oE$, uniformizer $\pE$ and residue field $\F$, and we assume that $\F$ contains $\Fq$. The finite field $\F$ is the main coefficient field in this work. We denote by $\varepsilon$ the $p$-adic cyclotomic character of $\gp$ and by $\omega$ its reduction mod $p$. 
We normalize Hodge--Tate weights so that $\varepsilon$ has Hodge--Tate weight $1$ at each embedding $K\into E$.
We normalize local class field theory so that uniformizers correspond to geometric Frobeniuses. 

If $H$ is any split connected reductive algebraic group, we denote by $Z_H$ the center of $H$ and by $T_H$ a split maximal torus. If $P_H$ is a parabolic subgroup of $H$ containing $T_H$, we denote by $M_{P_H}$ its Levi subgroup containing $T_H$, $N_{P_H}$ its unipotent radical and $P_H^-$ its opposite parabolic subgroup with respect to $T_H$ (so $P_H\cap P_H^-=M_H$).

We let $n\geq 2$ be an integer and denote by $G$ the algebraic group $\GL_n$ over $\Z$. The integer $n$ is arbitrary in \S\ref{conjectures} and is $2$ in \S\ref{gl2}.

Irreducible for a representation always means absolutely irreducible.

Finally, though we mainly work with the group $\GL_n$, several proofs in \S\ref{conjectures} can be extended more or less {\it verbatim} to a split connected reductive algebraic group over $\Z$ with connected center, and \S\ref{Cgroup} deals with possibly nonsplit reductive groups.

\bigskip

\textbf{Acknowledgements}: B.\;S. would like to thank Gabriel Dospinescu and Vytas Pa\v{s}k\={u}nas for stimulating
discussions around $C$-groups. Y.\;H.\ is partially supported by National Key R$\&$D Program of China 2020YFA0712600; CAS Project for Young Scientists in Basic Research, Grant No.\ YSBR-033; National Natural
Science Foundation of China Grants 12288201 and 11971028; National
Center for Mathematics and Interdisciplinary Sciences and Hua Loo-Keng
Key Laboratory of Mathematics, Chinese Academy of Sciences. F.\;H.\ is partially supported by an
NSERC grant. C.\;B.,
S.\;M. and B.\;S. are members of the A.N.R.\ project CLap-CLap
ANR-18-CE40-0026. 
Some of these results were presented at the conference ``Recent developments around $p$-adic modular forms'' (ICTS, December 2020) and the ``Spring School towards a mod $p$ Langlands correspondence'' (Essen, April 2021).
We very much thank the organizers of these events.

\newpage

\section{Local-global compatibility conjectures}\label{conjectures}

We state local-global compatibility conjectures (Conjecture \ref{theconjbar}, Conjecture \ref{conj:generale} and Conjecture \ref{theconj}) which ``functorially'' relate Hecke-isotypic components with their action of $\GL_n(K)$ in spaces of mod $p$ automorphic forms to representations of $\gp$. Conjecture \ref{theconj} assumes $K$ is unramified but is much stronger and more precise than Conjecture \ref{theconjbar} and Conjecture \ref{conj:generale} as it predicts the number, position and form of the irreducible constituents of these Hecke-isotypic components, as well as their contribution on the Galois side.

Throughout this section, we let $T\subseteq G=\GL_n$ the diagonal torus over $\Z$ and $X(T)$ the $\Z$-module $\Hom_{\rm Gr}(T,{\mathbb G}_{\rm m})$. As usual, we identify $X(T)$ with $\oplus_{i=1}^{n}\Z e_i$ via $e_i\mapsto \big({\rm diag}(x_1,\ldots,x_n)\mapsto x_i\big)$ and define $\langle \ ,\ \rangle:X(T)\times X(T)\rightarrow \Z$, $\langle e_i,e_j\rangle\defeq \delta_{i,j}$, which we extend by $\Q$-bilinearity to $X(T)\otimes_{\Z}\Q$. This provides an isomorphism of $\Z$-modules $X(T)\buildrel\sim\over\rightarrow \Hom_{\Z}(X(T),\Z)\cong \Hom_{\rm Gr}({\mathbb G}_{\rm m},T)$ given by
\begin{equation}\label{isot}
e_i\longmapsto e_i^*\defeq \big(x\mapsto {\rm diag}(\underbrace{1,\ldots,1}_{i-1},x,1,\ldots,1)\big),\ \ i\in \{1,\ldots,n\}.
\end{equation}
We denote by $R=\{e_i-e_j : 1\leq i\ne j\leq n\}\subset X(T)$ the roots of $(G,T)$, by $B\subseteq G$ the Borel subgroup (over $\Z$) of upper-triangular matrices and by $N$ the unipotent radical of $B$, so that the positive roots are $R^+=\{e_i-e_{j} : 1\leq i<j\leq n\}\subset R$ and the simple roots are $S=\{e_i-e_{i+1} : 1\leq i\leq n-1\}\subset R^+$. An element of $X(T)\otimes_{\Z}\Q$ is dominant if $\langle \lambda,e_i-e_{i+1}\rangle \geq 0$ for all $i\in \{1,\ldots,n-1\}$. If $\lambda,\mu\in X(T)\otimes_{\Z}\Q$, we write $\lambda\leq \mu$ if $\mu-\lambda \in \sum_{i=1}^{n-1} \Q_{\geq 0}(e_i-e_{i+1})$. If $\lambda=\sum_{i=1}^{n-1} n_i(e_i-e_{i+1})$ for some $n_i\in \Q$, its support is by definition the set of simple roots $e_i-e_{i+1}$ such that $n_i\ne 0$. Finally, we denote by $W\cong {\mathcal S}_n$ the Weyl group of $(G,T)$, which acts on the left on $X(T)$ by $w(\lambda)(t)\defeq \lambda(w^{-1}tw)$ for $\lambda\in X(T)$ and $t\in T$.

If $P$ is a standard parabolic subgroup of ${G}$ (that is, containing $B$), we denote by $S(P)\subseteq S$ the subset of simple roots of $M_{P}$, $R(P)^+\subseteq R^+$ the positive roots of $M_P$ (generated by $S(P)$) and $W(P)\subseteq W$ its Weyl group.

\subsection{Weak local-global compatibility conjecture}\label{global}

We state our first local-global compatibility conjecture (see Conjecture \ref{theconjbar} and its generalization Conjecture \ref{conj:generale}) which relate Hecke-isotypic components with their action of $\GL_n(K)$ to representations of $\gp$ without taking care of their irreducible constituents.

\subsubsection{The functors \texorpdfstring{$D_{\xi_H}^\vee$}{D\_\{xi\_H\}\^{}v} and \texorpdfstring{$V_H$}{V\_H}}\label{covariant}

We review the simple generalization of Colmez's functor defined in \cite{breuil-foncteur}. 

Throughout this section, we fix a connected reductive algebraic group $H$ which is split over $K$ with a connected center, $B_H\subseteq H$ a Borel subgroup and $T_H\subseteq B_H$ a split maximal torus in $B_H$. We let $\big(X(T_H),R_H,X^{\vee}(T_H),R_H^{\vee}\big)$ be the associated root datum, $R_H^+\subset X(T_H)$ the (positive) roots of $B_H$, $S_H\subseteq R_H^+$ the simple roots and $S_H^\vee$ the associated simple coroots.

We need to recall some notation of \cite{breuil-foncteur} (to which we refer the reader for any further details). For $\alpha\in R_H^+$, we let $N_\alpha\subseteq N_H$ be the associated (commutative) root subgroup, where $N_H\defeq N_{B_H}$ is the unipotent radical of $B_H$. For $\alpha\in S_H$, we fix an isomorphism $\iota_\alpha:N_\alpha \buildrel\sim\over \rightarrow {\mathbb G}_{\rm a}$ of algebraic groups over $K$ such that
\begin{equation}\label{alpha}
\iota_\alpha(tn_\alpha t^{-1})=\alpha(t)\iota_\alpha(n_\alpha)\ \ \forall\ t\in T_H,\ \ \forall\ n_\alpha\in N_\alpha.
\end{equation}
We fix an open compact subgroup $N_0\subset N_H(K)$ such that $\prod_{\alpha\in R_H^+}N_\alpha \buildrel\sim\over\rightarrow N_H$ induces a bijection $\prod_{\alpha\in R_H^+}N_{\alpha}(K)\cap N_0\buildrel \sim \over\rightarrow N_0$ for any order on the $\alpha\in R_H^+$ and such that $\iota_\alpha$ induces isomorphisms for $\alpha\in S_H$:
\begin{equation*}
N_\alpha(K)\cap N_0\buildrel\sim\over \longrightarrow \oK\subset K={\mathbb G}_{\rm a}(K).
\end{equation*}
We denote by $\ell$ the composite $N_H \twoheadrightarrow \prod_{\alpha\in S_H}N_\alpha \buildrel {\sum_{\alpha\in S_H}\iota_\alpha}\over\longrightarrow {\mathbb G}_{\rm a}$ (a morphism of algebraic groups over $K$). The morphism $\ell$ thus induces a group morphism still denoted $\ell: N_0\rightarrow \oK$ and we define
\begin{equation}\label{n1}
N_1\defeq  \Ker\big(N_0\buildrel \ell \over \rightarrow \oK \buildrel {\rm Tr}_{K/\Qp} \over \longrightarrow \Qp\big)
\end{equation}
which is a normal open compact subgroup of $N_0$. We fix an isomorphism of $\Zp$-modules $\psi:{\rm Tr}_{K/\Qp}(\oK)\buildrel \sim\over\rightarrow \Zp$. When $N_H\ne 0$, i.e.\ when $H\ne T_H$, this fixes an isomorphism
\begin{equation}\label{n0n1}
N_0/N_1\buildrel {\rm Tr}_{K/\Qp}\circ \ell \over {\buildrel \sim\over\longrightarrow}{\rm Tr}_{K/\Qp}(\oK) \buildrel \psi \over {\buildrel \sim\over\longrightarrow}\Zp.
\end{equation}
We fix fundamental coweights $(\lambda_{\alpha^\vee})_{\alpha\in S_H}$ (which exist since $H$ has a connected center) and set
\begin{equation}\label{coweights}
\xi_H \defeq \sum_{\alpha^\vee\in S_H^\vee}\lambda_{\alpha^\vee}\in \Hom_{\rm Gr}({\mathbb G}_{\rm m},T_H)=X^\vee(T_H).
\end{equation}
Note that $\xi_H(x)N_1\xi_H(x^{-1})\subseteq N_1$ for any $x\in \Zp\backslash \{0\}$. Let $\F\bbra{X}[F]$ be the noncommutative polynomial ring in $F$ over the ring of formal power series $\F\bbra{X}$ such that $FS(X)=S(X^p)F$.

For $\pi$ a smooth representation of $B_H(K)$ over $\F$, we endow the invariant subspace $\pi^{N_1}\subseteq \pi$ with a structure of an $\F\bbra{X}[F]$-module as follows:
\begin{enumerate}
\item $\F\bbra{X}\cong \F\bbra{\Zp}$ acts via $\F\bbra{N_0/N_1}\buildrel (\ref{n0n1}) \over \simeq \F\bbra{\Zp}$ (here $X\defeq [1]-1$);
\item
\label{def:F}
$F$ acts via the ``Hecke'' action $F(v)\defeq \sum_{n_1\in N_1/\xi_H(p)N_1\xi_H(p^{-1})}n_1\xi_H(p)v \ \in \pi^{N_1}$ for $v\in \pi^{N_1}$.
\end{enumerate}
Note that $\pi^{N_1}$ is a torsion $\F\bbra{X}$-module (but not a torsion $\F[F]$-module in general). We also endow $\pi^{N_1}$ with an action of $\Zp^\times$ by making $x\in \Zp^\times$ act by $\xi_H(x)$. This action commutes with $F$ and satisfies $\xi_H(x)\circ (1+X) = (1+X)^x\circ \xi_H(x)$. 

As in \cite{breuil-foncteur}, we denote by $\fgk$ the category of finite-dimensional \'etale $(\varphi,\Gamma)$-modules over $\F\bbra{X}[X^{-1}]=\F\ppar{X}$ and by $\fghatk$ the corresponding category of (pseudocompact) pro-objects, see \cite[\S2]{breuil-foncteur} for more details. Both $\fgk$ and $\fghatk$ are abelian categories. Let $M\subseteq \pi^{N_1}$ be a finite type $\F\bbra{X}[F]$-submodule which is $\Zp^\times$-stable and assume that $M$ is admissible as an $\F\bbra{X}$-module, that is, $M[X]\defeq \{m\in M : Xm=0\}$ is finite-dimensional over $\F$. Let $M^\vee\defeq \Hom_{\F}(M,\F)$ (algebraic $\F$-linear dual) which is also an $\F\bbra{X}$-module (but not a torsion $\F\bbra{X}$-module in general). Then by a key result of Colmez $M^\vee[X^{-1}]$ can be endowed with the structure of an object of $\fgk$ (\cite{Colmez}, see also \cite[Lemma 2.6]{breuil-foncteur}). More precisely $X$ acts on $f\in M^\vee$ by $(Xf)(m)\defeq f(Xm)$ ($m\in M$), $x\in \Zp^\times$ acts by $(xf)(m)\defeq f(x^{-1}m)$, and the operator $\varphi$ is defined as follows. Take the $\F$-linear dual of $\Id\otimes F:\F\bbra{X}\otimes_{\phz,\F\bbra{X}}M\longrightarrow M$, compose with\footnote{The formula for this isomorphism given in the proof of \cite[Lemma 2.6]{breuil-foncteur} is actually wrong, the present formula is the correct one. Note that it is also the same as $f\mapsto  \sum_{i=0}^{p-1}\frac{1}{(1+X)^i}\otimes f((1+X)^i\otimes \cdot)$.}
\begin{eqnarray}
\nonumber (\F\bbra{X}\otimes_{\phz,\F\bbra{X}}M)^\vee&\buildrel\sim\over\longrightarrow &\F\bbra{X}\otimes_{\phz,\F\bbra{X}}M^\vee\\
f&\longmapsto & \sum_{i=0}^{p-1}(1+X)^i\otimes f\big(\frac{1}{(1+X)^i}\otimes \cdot\big)\label{correctformula}
\end{eqnarray}
and invert $X$: the resulting morphism $M^\vee[X^{-1}]\rightarrow \F\bbra{X}\otimes_{\phz,\F\bbra{X}}M^\vee[X^{-1}]$ turns out to be an $\F\ppar{X}$-linear isomorphism whose inverse is by definition $\Id\otimes\phz$.

When $H\ne T_H$ we then define
\begin{equation}\label{DxiH}
D_{\xi_H}^\vee(\pi)\defeq  \plim{M} M^\vee[X^{-1}],
\end{equation}
where the projective limit is taken over the finite type $\F\bbra{X}[F]$-submodules $M$ of $\pi^{N_1}$ (for the preorder defined by inclusion) which are admissible as $\F\bbra{X}$-modules and invariant under the action of $\Zp^\times$. When $H=T_H$, one has to replace $M^\vee[X^{-1}]$ by $\F\ppar{X}\otimes_{\F}M^\vee$, we refer the reader to \cite[\S3]{breuil-foncteur}. The functor $D_{\xi_H}^\vee$ is right exact contravariant from the category of smooth representations of $B_H(K)$ over $\F$ to the category $\fghatk$ and, up to isomorphism, only depends on the choice of the cocharacter $\xi_H$. Moreover, if $D_{\xi_H}^\vee(\pi)$
 turns out to be in $\fgk$ (and not just $\fghatk$), then $D_{\xi_H}^\vee(\pi)$ is exactly the maximal \'etale $(\varphi,\Gamma)$-module which occurs as a quotient of $(\pi^{N_1})^\vee[X^{-1}]$, see \cite[Rem.5.6(iii)]{breuil-foncteur}.

\begin{rem}\label{dim1}
If $H={\mathbb G}_{\rm m}=T_H$, then by definition $\xi_H=1$. It follows, for $\dim_{\F}\pi=1$, that $D_{\xi_H}^\vee(\pi)$ is always the trivial (rank one) $(\varphi,\Gamma)$-module (even if $\pi$ is a nontrivial character).
\end{rem}

Let us now assume that the dual group $\widehat H$ of $H$ also has a connected center, and let us fix $\theta_H\in X(T_H)$ such that $\theta_H\circ \alpha^\vee=\Id_{{\mathbb G}_{\rm m}}$ for all $\alpha\in S_H$ (\cite[Prop.2.1.1]{BH}, such an element is called a {\it twisting element}). In \S\ref{Cgroup} below, it is possible to avoid this assumption using $C$-parameters, but since our main aim is $G=\GL_n$ in the rest of the paper, there is no harm in making this assumption.

Consider the smooth character
\[K^\times \longrightarrow \F^\times,\ x\longmapsto \omega\big(\theta_H(\xi_H(x))\big)\]
and denote by $\delta_H$ the restriction of this character to $\Qp^\times\subseteq K^\times$. Seeing $\omega\circ \theta_H\circ \xi_H$ as a character of $\gK$ via local class field theory for $K$ (as normalized in \S\ref{intro}), and remembering that the restriction from $K^\times$ to $\Qp^\times$ corresponds via local class field theory to the composition with the transfer $\gp^{\rm ab}\rightarrow \gK^{\rm ab}$, we see that
\[\delta_H\cong \ind_K^{\otimes\Qp}\!(\omega\circ \theta_H\circ \xi_H),\]
where $\ind_K^{\otimes\Qp}$ is the tensor induction from $\gK$ to $\gp$ (see the end of \S\ref{somprel} below).

Denote \ by \ $\repk$ \ the \ abelian \ category \ of \ continuous \ linear \ representations \ of $\gp$ on finite-dimensional $\F$-vector spaces (equipped with the discrete topology) and $\irepk$ the corresponding category of ind-objects, i.e.\ the category of filtered direct limits of objects of $\repk$. Recall that there is a covariant equivalence of categories ${\bf V}:\fgk \buildrel\sim\over\rightarrow \repk$ (see \cite[Thm.A.3.4.3]{Fo} where this functor is denoted ${\bf V}_{\mathcal E}$) compatible with tensor products and duals on both sides. We denote by ${\bf V}^\vee$ the dual of ${\bf V}$ (i.e.\ the dual Galois representation). When $H\ne T_H$, we then define the covariant functor $V_H$ from the category of smooth representations of $B_H(K)$ over $\F$ to the category $\irepk$ by
\begin{equation}\label{VG}
V_H(\pi)\defeq  \ilim{M} \big({\bf V}^\vee(M^\vee[X^{-1}])\big)\otimes \delta_H,
\end{equation}
where the inductive limit is taken over the finite type $\F\bbra{X}[F]$-submodules of $\pi^{N_1}$ which are admissible as $\F\bbra{X}$-modules and preserved by $\Zp^\times$. Likewise, when $H=T_H$, with $\F\ppar{X}\otimes_{\F}M^\vee$ instead of $M^\vee[X^{-1}]$ (note that $\delta_H$ is then $1$).

\begin{lem}
The functor $V_H$ is left exact.
\end{lem}
\begin{proof}
We give the proof for $H\ne T_H$, leaving the case $H=T_H$ to the reader. Let $0\rightarrow \pi'\rightarrow \pi\buildrel s \over \rightarrow \pi''\rightarrow 0$ be an exact sequence of smooth $B_H(K)$-representa\-tions over $\F$, which gives a short exact sequence $0\rightarrow {\pi'}^{N_1}\rightarrow \pi^{N_1}\buildrel s \over\rightarrow {\pi''}^{N_1}$. If $M$ is a finite type $\F\bbra{X}[F]$-submodule of $\pi^{N_1}$ which is admissible as an $\F\bbra{X}$-module and stable under the action of $\Zp^\times$, then so are $M\cap {\pi'}^{N_1}$ and $s(M)$ (see e.g.\ \cite[Lemma 2.1(i)]{breuil-foncteur}). The functor $M\rightarrow {\bf V}^\vee(M^\vee[X^{-1}])$ being covariant exact (since both $M\mapsto M^\vee[X^{-1}]$ and ${\bf V}^\vee$ are contravariant exact), each such $M\subseteq \pi^{N_1}$ gives rise to a short exact sequence in $\repk$:
\[0\rightarrow {\bf V}^\vee\big((M\cap {\pi'}^{N_1})^\vee[X^{-1}]\big)\rightarrow {\bf V}^\vee\big(M^\vee[X^{-1}]\big)\rightarrow {\bf V}^\vee\big(s(M)^\vee[X^{-1}]\big)\rightarrow 0.\]
Twisting by $\delta_H$ and taking the inductive limit over such $M$, we obtain a short exact sequence $0\rightarrow V_H(\pi')\rightarrow V_H(\pi)\rightarrow \ilim{M}{\bf V}^\vee\big(s(M)^\vee[X^{-1}]\big)\otimes \delta_H\rightarrow 0$ in $\irepk$. But we have an injection
\[\ilim{M}{\bf V}^\vee\big(s(M)^\vee[X^{-1}]\big)\otimes \delta_H\hookrightarrow V_H(\pi'')\]
in $\irepk$ since all transitions maps in the inductive limits are injective, therefore we end up with an exact sequence $0\rightarrow V_H(\pi')\rightarrow V_H(\pi)\rightarrow V_H(\pi'')$.
\end{proof}

\begin{ex}\label{exdelta}
For $H=G\times_{\Z}K={\GL_n}_{\slash K}$ (so $H\cong \widehat H$), we take in the sequel (writing just $G$ as a subscript instead of $G\times_{\Z}K$)
\[\xi_G(x)\defeq {\rm diag}(x^{n-1},\ldots,x,1)\ \ {\rm and}\ \ \theta_G\big({\rm diag}(x_1,\ldots,x_n)\big)= x_1^{n-1}x_2^{n-2}\cdots x_{n-1},\]
so that $\delta_G=\ind_K^{\otimes\Qp}\!(\omega^{(n-1)^2+(n-2)^2+\cdots +4+1})$. 
(In fact, since the tensor induction of a character is given by composition with the transfer map \cite{Coll}, by local class field theory we see that 
$\delta_G = \omega^{[K:\Q_p]((n-1)^2+(n-2)^2+\cdots +4+1)}$.)
\end{ex}

\begin{rem}\label{trivial}
(i) The covariant functor $V_H$ depends on the choices of $\xi_H$ and $\delta_H$ (though we don't include it in the notation). The reader may wonder why we need to assume the existence of $\theta_H$ and normalize $V_H$ using the strange twist $\delta_H$ above. This comes from the local-global compatibility: it turns out that this normalization is essentially what is going on in spaces of mod $p$ automorphic forms (see \cite[\S4]{BH}, \cite[Cor.9.8]{breuil-foncteur}, Example \ref{enlightening} and \S\S\ref{wlgc}, \ref{slgc} below). This normalization is also natural if one uses $C$-parameters, see \S\ref{Cgroup}.\\
(ii) For $H$ as in Example \ref{exdelta}, $\pi$ a smooth representation of $B(K)$ over $\F$ and $\chi:K^\times \rightarrow \F^\times$ a smooth character, one checks that $V_G(\pi\otimes (\chi\circ {\det}))\cong V_G(\pi)\otimes \delta$, where $\delta$ is the continuous character of $\gp$ associated via local class field theory to $x\mapsto \chi\big({\det}(\xi_G(x))\big)$ for $x\in \Qp^\times$. An explicit computation gives $\delta=(\chi\vert_{\Qp^\times})^{\frac{n(n-1)}{2}}\cong \ind_K^{\otimes\Qp}\!(\chi^{\frac{n(n-1)}{2}})$.

\end{rem}

When restricted to the abelian category of finite length admissible smooth representations of $H(K)$ over $\F$ with all irreducible constituents isomorphic to irreducible constituents of principal series, it is proven in \cite[\S9]{breuil-foncteur} that the functors $D_{\xi_H}^\vee$ and $V_H$ are exact. It seems reasonable to us, and also consistent with the conjectural formalism developed in the sequel (see e.g.\ Remark \ref{listrem}(iii)), to hope that there exists a suitable abelian category of admissible smooth representations of $H(K)$ over $\F$ containing the previous abelian category and the representations ``coming from the global theory'' on which the functors $D_{\xi_H}^\vee$ and $V_H$ are still exact. See for instance the category $\mathcal C$ in \S\ref{multivariablepsi} when $H={\GL_2}_{\slash K}$ and $K$ is unramified.

We now recall the behaviour of the functor $V_H$ with respect to parabolic induction.

We assume for simplicity $H=G\times_{\Z}K={\GL_n}_{\slash K}$ and let $\xi_G$, $\theta_G$ as in Example \ref{exdelta}. We let $P$ be a standard parabolic subgroup of $G\times_{\Z}K$ and write $M_P=\prod_{i=1}^dM_i$ with $M_i\cong {\GL_{n_i}}_{\slash K}$. We define $V_{M_P}$ as in (\ref{VG}) using $\xi_{M_P}\defeq \xi_G$ and $\theta_{M_P}\defeq \theta_G$ (to define $D_{\xi_{M_P}}^\vee$ and $\delta_{M_P}$). We write $\xi_{M_P}=\oplus_{i=1}^d\xi_{M_P,i}$ in $X^\vee(T)=\oplus_{i=1}^dX^\vee(T_i)$ and $\theta_{M_P}=\oplus_{i=1}^d\theta_{M_P,i}$ in $X(T)=\oplus_{i=1}^dX(T_i)$, where $T_i$ is the diagonal torus in $M_i$, and let $V_{M_P,i}\defeq V_{\GL_{n_i}}$ but defined with $\xi_{M_P,i}$ and $\theta_{M_P,i}$. Finally we define $V_{M_i}\defeq V_{\GL_{n_i}}$ with $\xi_{M_i}$ and $\theta_{M_i}$ as in Example \ref{exdelta} replacing $n$ by $n_i$, and we recall that $\xi_{M_i}$, $\theta_{M_i}$ and $\delta_{M_i}$ are trivial characters if $n_i=1$.

If $\pi_P$ is a smooth representation of $M_P(K)$ over $\F$, that we see as a representation of $P^-(K)$ via $P^-(K)\twoheadrightarrow M_P(K)$, we define the usual smooth parabolic induction
\[\Ind_{P^-(K)}^{G(K)}\pi_P\defeq \{f:G(K)\rightarrow \pi_P \ {\rm loc.\ const.},\ f(px)=p(f(x)),\ p\in P^-(K),\ x\in \pi_P\},\]
with $G(K)$ acting (smoothly) on the left by $(gf)(g')\defeq f(g'g)$. 

\begin{lem}\label{parabtensor}
Let $\pi_P$ be a smooth representation of $M_P(K)$ over $\F$ of the form $\pi_P=\pi_1\otimes \cdots \otimes \pi_d$, where the $\pi_i$ are smooth representations of $M_i(K)$ over $\F$. Assume that the $\pi_i$ have central characters $Z(\pi_i):K^\times \rightarrow \F^\times$ and that $V_{M_P}(\pi_P)\cong \bigotimes_{i=1}^dV_{M_P,i}(\pi_i)$. Then we have an isomorphism in $\irepk$ {\upshape(}using implicitly local class field theory for $\gp${\upshape)}:
\[V_G\Big(\Ind_{P^-(K)}^{G(K)}\pi_P\Big)\otimes \delta_G^{-1}\cong \bigotimes_{i=1}^d \Big(V_{M_i}(\pi_i)\otimes \big(Z(\pi_i)^{n-\sum_{j=1}^in_j}\big)\vert_{\Qp^\times}\delta_{M_i}^{-1}\Big).\]
\end{lem}
\begin{proof}
By \cite[Thm.6.1]{breuil-foncteur} we have $V_G\big(\Ind_{P^-(K)}^{G(K)}\pi_P\big)\cong V_{M_P}(\pi_P)$ so that from the assumption (all isomorphisms are in $\irepk$):
\begin{equation}\label{isotens}
V_G\big(\Ind_{P^-(K)}^{G(K)}\pi_P\big)\cong \bigotimes_{i=1}^dV_{M_P,i}(\pi_i).
\end{equation}
An easy computation yields in $M_i(K)$ for $x\in K^\times$:
\[\xi_{M_P,i}(x)={\rm diag}(\underbrace{x^{n-\sum_{j=1}^in_j},\ldots,x^{n-\sum_{j=1}^in_j}}_{n_i})\xi_{M_i}(x)\]
which implies by \cite[Rem.4.3]{breuil-foncteur} that
\begin{equation}\label{compF}
V_{M_P,i}(\pi_i)\otimes \delta_{M_P,i}^{-1}\cong V_{M_i}(\pi_i) \otimes \big(Z(\pi_i)^{n-\sum_{j=1}^in_j}\big)\vert_{\Qp^\times}\delta_{M_i}^{-1},
\end{equation}
where $\delta_{M_P,i}\defeq \ind_K^{\otimes\Qp}\!(\omega\circ\theta_{M_P,i}\circ\xi_{M_P,i})$ (and recall $V_{M_i}(\pi_i)=1$ if $n_i=\dim_{\F}\pi_i=1$, see Remark \ref{dim1}). Since $\delta_G=\prod_{i=1}^d\delta_{M_P,i}$, twisting (\ref{isotens}) by $\delta_G^{-1}$ gives the result by (\ref{compF}).
\end{proof}

\begin{ex}\label{enlightening}
An enlightening and important example is the case of principal series $\Ind_{B^-(K)}^{G(K)}(\chi_1\otimes \cdots \otimes \chi_n)$, where the $\chi_i:K^\times \rightarrow \F^\times$ are smooth characters. The assumptions of Lemma \ref{parabtensor} are then trivially satisfied and thus we have
\[V_G\big(\Ind_{B^-(K)}^{G(K)}(\chi_1\otimes \cdots \otimes \chi_n)\big)\otimes \delta_G^{-1}\cong (\chi_1^{n-1}\chi_2^{n-2}\cdots \chi_{n-1})\vert_{\Qp^\times}.\]
In particular we deduce (using Example \ref{exdelta} for $\delta_G$) that
\begin{eqnarray*}
V_G\big(\Ind_{B^-(K)}^{G(K)}(\chi_1\omega^{-(n-1)}\otimes \chi_2\omega^{-(n-2)}\otimes \cdots \otimes \chi_n)\big)&\cong &(\chi_1^{n-1}\chi_2^{n-2}\cdots \chi_{n-1})\vert_{\Qp^\times}\\
&\cong &\ind_K^{\otimes\Qp}\!(\chi_1^{n-1}\chi_2^{n-2}\cdots \chi_{n-1}),
\end{eqnarray*}
where $\chi_1^{n-1}\chi_2^{n-2}\cdots \chi_{n-1}$ on the last line is seen as a character of $\gK$ via local class field theory for $K$.
\end{ex}

\begin{rem}\label{tensorprod}
Using \cite[Prop.5.5]{breuil-foncteur} the assumptions of Lemma \ref{parabtensor} are satisfied when all finite type $\F\bbra{X}[F]$-submodules of $\pi_i^{N_1}$ for $i\in \{1,\dots,d\}$ are automatically admissible as $\F\bbra{X}$-modules. This happens for instance if the $\pi_i$ are principal series or (when $K=\Qp$) are finite length representations of $\GL_2(\Qp)$ with a central character, but is not known otherwise. Contrary to what is stated in \cite[Rem.5.6(ii)]{breuil-foncteur}, we currently do not have a proof of an isomorphism $V_{M_P}(\pi_P)\cong \bigotimes_{i=1}^dV_{M_P,i}(\pi_i)$ for any smooth representations $\pi_i$, though we expect that it will indeed be satisfied for representations ``coming from'' the global theory. Note that, in \cite[Prop.3.2]{Za}, Z\'abr\'adi does prove a compatibility of his functor with the tensor product which looks close to the isomorphism above. However, {\it loc.cit.}\ deals with an {\it external} tensor product, whereas we have an {\it internal} tensor product. In particular he has {\it two} operators $F$, one for each factor in the external tensor product (whereas we consider the resulting diagonal operator), and his argument doesn't extend.
\end{rem}

\subsubsection{Global setting}\label{somprel}

We recall our global setting (see e.g.\ \cite[\S7.1]{EGH} or \cite[\S6]{Th} or \cite[\S4.1]{BH} or many other references) and define the $\gp$-representation $\LLbar(\rhobar)$ for $\rhobar:\gK\longrightarrow G(\F)$.

We let $F^+$ be a totally real finite extension of $\Q$ with ring of integers $\oFF$, $F/F^+$ a totally imaginary quadratic extension with ring of integers $\oF$ (do not confuse $F$ with the operator $F$ of \S\ref{covariant}!) and $c$ the nontrivial element of $\Gal(F/F^+)$. If $v$ (resp.\ $\tilde{v}$) is a finite place of $F^+$ (resp.\ $F$), we let $F_v^+$ (resp.\ $F_{\tilde{v}}$) be the completion of $F^+$ (resp.\ $F$) at $v$ (resp.\ $\tilde{v}$) and $\oFFv$ (resp.\ ${\mathcal O}_{\!F_{\tilde{v}}}$) the ring of integers of $F_v^+$ (resp.\ $F_{\tilde{v}})$. If $v$ splits in $F$ and $\tilde{v},\tilde{v}^c$ are the two places of $F$ above $v$, we have $\oFFv={\mathcal O}_{\!F_{\tilde{v}}}\buildrel c \over \simeq {\mathcal O}_{\!F_{\tilde{v}^c}}$, where the last isomorphism is induced by $c$. We let ${\mathbb A}_{F^+}^{\infty}$ (resp.\ ${\mathbb A}_{F^+}^{\infty,v}$) denote the finite ad\`eles of $F^+$ (resp.\ the finite ad\`eles of $F^+$ outside $v$). Finally we always assume that all places of $F^+$ above $p$ split in $F$.

We let $n\in \Z_{>1}$, $N$ a positive integer prime to $p$ and $H$ a connected reductive algebraic group over $\oFF[1/N]$ satisfying the following conditions:
\begin{enumerate}
\item there is an isomorphism $\iota:H\times_{\oFF[1/N]} \oF[1/N]\buildrel\sim\over\longrightarrow G\times_{\Z}\oF[1/N]$;
\item $H\times_{\oFF[1/N]} F^+$ is an outer form of $G\times_{\Z}F^+={\GL_n}_{/F^+}$;
\item $H\times_{\oFF[1/N]} F^+$ is isomorphic to ${\rm U}_n(\R)$ at all infinite places of $F^+$.
\end{enumerate}
One can prove that such groups exist (cf.\ e.g.\ \cite[\S7.1.1]{EGH}). Condition (i) implies that if $v$ is any finite place of $F^+$ that splits in $F$ and if $\tilde{v}\vert v$ in $F$ the isomorphism $\iota$ induces $\iota_{\tilde{v}}:H(F_v^+)\buildrel\sim\over\rightarrow {\GL_n}(F_{\tilde{v}})=G(F_{\tilde{v}})$ which restricts to an isomorphism still denoted by $\iota_{\tilde{v}}:H(\oFFv)\buildrel\sim\over\rightarrow {\GL_n}({\mathcal O}_{\!F_{\tilde{v}}})$ if $v$ doesn't divide $N$. Condition (ii) implies that $c\circ \iota_{\tilde{v}}:H(F_v^+)\buildrel\sim\over\rightarrow {\GL_n}(F_{{\tilde{v}}^c})$ (resp.\ $c\circ \iota_{\tilde{v}}:H(\oFFv)\buildrel\sim\over\rightarrow {\GL_n}({\mathcal O}_{\!F_{{\tilde{v}}^c}})$ if $v$ doesn't divide $N$) is conjugate in ${\GL_n}(F_{{\tilde{v}}^c})$ (resp.\ in ${\GL_n}({\mathcal O}_{\!F_{{\tilde{v}}^c}})$) to $\tau^{-1}\circ \iota_{{\tilde{v}}^c}$, where $\tau$ is the transpose in ${\GL_n}(F_{{\tilde{v}}^c})$ (resp.\ in ${\GL_n}({\mathcal O}_{\!F_{{\tilde{v}}^c}})$).

If $U$ is any compact open subgroup of $\GA$ then
\[S(U,\F)\defeq \{f:H(F^+)\backslash \GA/U\rightarrow \F\}\]
is a finite-dimensional $\F$-vector space since $H(F^+)\backslash \GA/U$ is a finite set. Fix $v\vert p$ in $F^+$ and a compact open subgroup $U^{v}$ of $\GAv$, we define
\begin{equation*}
S(U^{v},\F)\defeq \ilim{U_{v}}{S(U^{v}U_{v},\F)},
\end{equation*}
where $U_{v}$ runs among compact open subgroups of $H(\oFFv)$. We endow $S(U^{v},\F)$ with a linear left action of $H(F^+_{v})$ by $(h_vf)(h)\defeq f(hh_v)$ ($h_v\in H(F^+_{v})$, $h\in \GA$). Thus, for $\tilde{v}$ dividing $v$ in $F$, the isomorphism $\iota_{\tilde{v}}$ gives an admissible smooth action of $G(F_v^+)=\GL_n(F_{\tilde{v}})$ on $S(U^{v},\F)$. By what is above, the action of $G(F_v^+)$ induced by $\iota_{\tilde{v}}$ is the inverse transpose of the one induced by $\iota_{{\tilde{v}}^c}$.

If $U$ is a compact open subgroup of $\GA$, following \cite[\S7.1.2]{EGH} we say that $U$ is {\it unramified} at a finite place $v$ of $F^+$ which splits in $F$ and doesn't divide $N$ if we have $U=U^{v}\times H(\oFFv)$, where $U^{v}$ is a compact open subgroup of $\GAv$. Note that a compact open subgroup of $\GA$ is unramified at all but a finite number of finite places of $F^+$ which split in $F$. If $U$ is a compact open subgroup of $\GA$ and $\Sigma$ a finite set of finite places of $F^+$ containing the set of places of $F^+$ that split in $F$ and divide $pN$ and the set of places of $F^+$ that split in $F$ at which $U$ is {\it not} unramified, we denote by $\TT^\Sigma\defeq \oE[T^{(j)}_{\tilde{w}}]$ the commutative polynomial $\oE$-algebra generated by formal variables $T^{(j)}_{\tilde{w}}$ for $j\in \{1,\dots,n\}$ and ${\tilde{w}}$ a place of $F$ lying above a finite place $w$ of $F^+$ that splits in $F$ and {\it doesn't} belong to $\Sigma$. The algebra $\TT^\Sigma$ acts on $S(U,\F)$ by making $T^{(j)}_{\tilde{w}}$ act by the double coset
\[\iota_{\tilde{w}}^{-1}\left[{\GL_n}({\mathcal O}_{\!F_{\tilde{w}}})\begin{pmatrix}{\bf 1}_{n-j}&\\& \varpi_{\tilde{w}}{\bf 1}_{j}\end{pmatrix}{\GL_n}({\mathcal O}_{\!F_{\tilde{w}}})\right],\]
where $\varpi_{\tilde{w}}$ is a uniformizer in ${\mathcal O}_{\!F_{\tilde{w}}}$. Explicitly, if we write
\[{\GL_n}({\mathcal O}_{\!F_{\tilde{w}}})\smat{{\bf 1}_{n-j}&\\& \varpi_{\tilde{w}}{\bf 1}_{j}}{\GL_n}({\mathcal O}_{\!F_{\tilde{w}}})=\coprod_{i}g_i\smat{{\bf 1}_{n-j}&\\& \varpi_{\tilde{w}}{\bf 1}_{j}}{\GL_n}({\mathcal O}_{\!F_{\tilde{w}}}),\]
we have for $f\in S(U,\F)$ and $g\in \GA$:
\[(T^{(j)}_{\tilde{w}}f)(g)\defeq \sum_if\bigg(g\iota_{\tilde{w}}^{-1}\Big(g_i\smat{{\bf 1}_{n-j}&\\& \varpi_{\tilde{w}}{\bf 1}_{j}}\Big)\bigg).\]
One checks that $T^{(j)}_{{\tilde{w}}^c} = (T_{\tilde{w}}^{(n)})^{-1}T^{(n-j)}_{\tilde{w}}$ on $S(U,\F)$. We let $\TT^\Sigma(U,\F)$ be the image of $\TT^\Sigma$ in $\End_{\oE}(S(U,\F))$ (if $U'\subseteq U$, we thus have $S(U,\F)\subseteq S(U',\F)$ and $\TT^\Sigma(U',\F)\twoheadrightarrow \TT^\Sigma(U,\F)$). If $S$ is any $\TT^\Sigma$-module and $I$ any ideal of $\TT^\Sigma$, we set in the sequel $S[I]\defeq \{x\in S : Ix=0\}$.

We now fix $v\vert p$ and a compact open subgroup $U^{v}$ of $\GAv$. If $\Sigma$ is a finite set of finite places of $F^+$ containing the set of places of $F^+$ that split in $F$ and divide $pN$ and the set of places of $F^+$ prime to $p$ that split in $F$ and at which $U^{v}U_v$ (for any $U_v$) is not unramified, the algebra $\TT^\Sigma$ acts on $S(U^{v}U_{v},\F)$ (via its quotient $\TT^\Sigma(U^{v}U_{v},\F)$) for any $U_{v}$ and thus also on $S(U^{v},\F)$. This action commutes with that of $H(F^+_{v})$. If $\m^\Sigma$ is a maximal ideal of $\TT^{\Sigma}$ with residue field $\F$, we can define the localized subspaces $S(U^{v}U_{v},\F)_{\m^\Sigma}$ and their inductive limit
\[\ilim{U_{v}}{S(U^{v}U_{v},\F)_{\m^\Sigma}}=S(U^{v},\F)_{\m^\Sigma},\]
which inherits an induced (admissible smooth) action of $H(F^+_{v})$ together with a commuting action of $\plim{U_v}\TT^\Sigma(U^vU_v,\F)_{\m^\Sigma}$. We have
\[S(U^{v}U_{v},\F)[\m^{\Sigma}]\!\subseteq S(U^{v}U_{v},\F)_{\m^\Sigma}\subseteq S(U^{v}U_{v},\F)\]
and thus inclusions of admissible smooth $H(F^+_{v})$-representations over $\F$:
\[S(U^{v},\F)[\m^{\Sigma}]\subseteq S(U^{v},\F)_{\m^\Sigma}\subseteq S(U^{v},\F).\]
Moreover, as representations of $H(F^+_{v})$, $S(U^{v}\!,\F)_{\m^\Sigma}$ is a direct summand of $S(U^{v}\!,\F)$ (= the maximal vector subspace on which the elements of $\m^{\Sigma}$ act nilpotently). 

We now go back to the notation of \S\ref{covariant}. For $\lambda\in X(T)$ a dominant weight with respect to $B$, we consider the following algebraic representation of $G\times_{\Z}\F$ over $\F$:
\begin{equation}\label{algbar}
\Lbar(\lambda)\defeq \big({\rm ind}_{B^-}^G\lambda\big)_{/\Z}\otimes_{\Z}\F=\big({\rm ind}_{{B^-}\times_{\Z}\F}^{{ G}\times_{\Z}\F}\lambda\big)_{/\F},
\end{equation}
where ind means the algebraic induction functor of \cite[\S I.3.3]{Ja} and the last equality follows from \cite[II.8.8(1)]{Ja}. For $\alpha=e_i-e_{i+1}\in S$, we set
\begin{equation}
\lambda_{\alpha}\defeq e_1+\cdots +e_i\in X(T),\label{eq:lambda-alpha}
\end{equation}
so that the $\lambda_{\alpha}$ for $\alpha\in S$ are fundamental weights of $G$ (see e.g.\ \cite[\S2.1]{BH}). Let $\rhobar:\gK\longrightarrow G(\F)$ be a continuous homomorphism, viewing $\Lbar(\lambda_{\alpha})$ as a continuous homomorphism
\[G(\F)\longrightarrow {\rm Aut}\big(\Lbar(\lambda_{\alpha})(\F)\big)\]
(where $\Lbar(\lambda_{\alpha})(\F)$ is the underlying $\F$-vector space of the algebraic representation $\Lbar(\lambda_{\alpha})$), we define the Galois representations for $\alpha\in S$:
\[\Lbar(\lambda_{\alpha})(\rhobar):\gK\buildrel\rhobar\over\longrightarrow G(\F)\buildrel\Lbar(\lambda_{\alpha}) \over\longrightarrow {\rm Aut}\big(\Lbar(\lambda_{\alpha})(\F)\big).\]
Recall that $\Lbar(\lambda_{\alpha})(\rhobar)={\bigwedge}^{i}_{\F}\rhobar$ if $\alpha=e_i-e_{i+1}$ (\cite[Ex.2.1.3]{BH}). We let
\[\bigotimes_{\alpha\in S}\big(\Lbar(\lambda_{\alpha})(\rhobar)\big)\cong \bigotimes_{i=1}^{n-1}{\bigwedge}^{\!\!i}_{\F}\rhobar\]
be the tensor product of the representations $\Lbar(\lambda_{\alpha})(\rhobar)$ (over $\F$) and define the following finite-dimensional conti\-nuous representation of $\gp$ over $\F$:
\begin{equation}\label{tensorind}
\LLbar(\rhobar)\defeq \ind_K^{\otimes\Qp}\!\Big(\bigotimes_{\alpha\in S}\big(\Lbar(\lambda_{\alpha})( \rhobar)\big)\Big),
\end{equation}
where $\ind_K^{\otimes\Qp}$ means the {\it tensor induction from $\gK$ to $\gp$} (\cite{Coll}, \cite[\S13]{Curtis-Reiner1}, see also the end of the proof of Lemma \ref{galqp}). Note that there are $\gp$-equivariant isomorphisms
\begin{equation}\label{dual}
\LLbar(\rhobar^\vee)\cong \LLbar(\rhobar)^\vee\cong \LLbar(\rhobar)\otimes \ind_K^{\otimes\Qp}\!\!\big({\det(\rhobar)}^{-(n-1)}\big)
\end{equation}
(recall $\ind_K^{\otimes\Qp}\!\!\big({\det(\rhobar)}^{-(n-1)}\big)$ is still one dimensional).

\begin{ex}\label{indn=2}
For $n=2$, we thus just have $\LLbar(\rhobar)=\ind_K^{\otimes\Qp}\!(\rhobar)$.
\end{ex}

\subsubsection{Weak local-global compatibility conjecture}\label{wlgc}

We state our weak local-global compatibility conjecture (Conjecture \ref{theconjbar}).

Let $\rbar:\gF\rightarrow {\GL}_n(\F)$ be a continuous representation and $\rbar^\vee$ its dual. We assume:
\begin{enumerate}
\item $\rbar^c\cong \rbar^\vee\otimes\omega^{1-n}$ (where $\rbar^c(g)\defeq \rbar(cgc)$ for $g\in \gF$);
\item $\rbar$ is an absolutely irreducible representation of $\gF$.
\end{enumerate}
Fix $v\vert p$ in $F^+$, $V^v\subseteq U^{v}\subseteq \GAv$ compact open subgroups and $\Sigma$ a finite set of finite places of $F^+$ containing
\begin{enumerate}[(a)]
\item the set of places of $F^+$ that split in $F$ and divide $pN$;
\item the set of places of $F^+$ that split in $F$ at which $V^{v}$ is not unramified;
\item the set of places of $F^+$ that split in $F$ at which $\rbar$ is ramified.
\end{enumerate}
We associate to $\rbar$ and $\Sigma$ the maximal ideal $\m^\Sigma$ in $\TT^\Sigma$ with residue field $\F$ generated by $\pE$ and all elements
\[\Big((-1)^j{\rm Norm}({\tilde{w}})^{j(j-1)/2}T_{\tilde{w}}^{(j)}-a^{(j)}_{\tilde{w}}\Big)_{j,{\tilde{w}}},\]
where $j\in \{1,\dots,n\}$, ${\tilde{w}}$ is a place of $F$ lying above a finite place $w$ of $F^+$ that splits in $F$ and doesn't belong to $\Sigma$, $X^n+\overline a_{\tilde{w}}^{(1)}X^{n-1}+\cdots + \overline a_{\tilde{w}}^{(n-1)}X+\overline a_{\tilde{w}}^{(n)}$ is the characteristic polynomial of $\rbar(\Frob_{\tilde{w}})$ (an element of $\F[X]$, $\Frob_{\tilde{w}}$ is a geometric Frobenius at ${\tilde{w}}$) and where $a^{(j)}_{\tilde{w}}$ is any element in $\oE$ lifting $\overline a^{(j)}_{\tilde{w}}$. Note that $S(V^{v},\F)[\m^{\Sigma}]\ne 0$ in fact implies assumption (i) above on $\rbar$ (though strictly speaking we need (i) to define $\m^{\Sigma}$ in $\TT^\Sigma$). Note also that if $U$ is any subgroup of $\GA$ containing $V^v$ as a normal subgroup, then $U$ naturally acts on $S(V^{v},\F)$ and $S(V^{v},\F)[\m^{\Sigma}]$.

For ${\tilde{v}}\vert v$ in $F$, we denote by $V_{G,{\tilde{v}}}$ the functor defined in (\ref{VG}) applied to smooth representations of $H(F_v^+)$ over $\F$, where we identify $H(F_v^+)$ with $\GL_n(F_{\tilde{v}})=G(F_{\tilde{v}})$ via $\iota_{\tilde{v}}$. For any finite place ${\tilde{w}}$ of $F$, we denote by $\rbar_{\tilde{w}}$ the restriction of $\rbar$ to a decomposition subgroup at ${\tilde{w}}$. 

\begin{conj}\label{theconjbar}
Let $\rbar:\gF\rightarrow {\GL}_n(\F)$ be a continuous representation that satisfies conditions (i) and (ii) above and fix a place $v$ of $F^+$ which divides $p$. Assume that there exist compact open subgroups $V^v\subseteq U^{v}\subseteq \GAv$ with $V^v$ normal in $U^v$, a finite-dimensional representation $\sigma^v$ of $U^v/V^v$ over $\F$ and a finite set $\Sigma$ of finite places of $F^+$ as above such that $\Hom_{U^v}(\sigma^v,S(V^{v},\F)[\m^{\Sigma}])\ne 0$. Let ${\tilde{v}}\vert v$ in $F$. Then there is an integer $d\in \Z_{>0}$ depending only on $v$, $U^{v}$, $V^v$, $\sigma^v$ and $\rbar$ such that there is an isomorphism of representations of $\gp$ on $\F$:
\begin{equation}\label{beta1serre}
V_{G,{\tilde{v}}}\big(\Hom_{U^v}(\sigma^v,S(V^{v},\F)[\m^{\Sigma}])\otimes (\omega^{-(n-1)}\circ{\det})\big)\cong \LLbar(\rbar_{\tilde{v}})^{\oplus d}.
\end{equation}
\end{conj}

\begin{rem}\label{theconjbarrem}
(i) In the special case $\sigma^v=1$, Conjecture \ref{theconjbar} boils down to $V_{G,{\tilde{v}}}(S(U^{v},\F)[\m^{\Sigma}]\otimes (\omega^{-(n-1)}\circ{\det}))\cong \LLbar(\rbar_{\tilde{v}})^{\oplus d}$.\\
(ii) Conjecture \ref{theconjbar} implies that the $G(F_{\tilde{v}})$-repre\-sentation $\Hom_{U^v}(\sigma^v\!,S(V^{v},\F)[\m^{\Sigma}])$ determines the $\gp$-representation $\LLbar(\rbar_{\tilde{v}})$. Note that this doesn't imply in general that $\Hom_{U^v}(\sigma^v,S(V^{v},\F)[\m^{\Sigma}])$ determines the $\Gal(\overline{F_{\tilde{v}}}/F_{\tilde{v}})$-represen\-tation $\rbar_{\tilde{v}}$ itself (though this is also expected, see \cite{PQ} and the references therein).\\
(iii) See \S\S\ref{tensorinduction},~\ref{gl2results} below for nontrivial evidence on Conjecture \ref{theconjbar} when $K$ is unramified and $n=2$.
\end{rem}

We now check that, at least when $p$ is odd, $F/F^+$ is unramified at finite places and $H\times_{\oFF[1/N]} F^+$ is quasi-split at finite places, Conjecture \ref{theconjbar} holds for ${\tilde{v}}$ if and only if it holds for ${\tilde{v}}^c$ (these extra assumptions come from the use of \cite[\S6]{Th} in the next lemma).

\begin{lem}\label{central}
Assume $p>2$, $F/F^+$ unramified at finite places and $H\times_{\oFF[1/N]} F^+$ quasi-split at finite places of $F^+$. Let ${\tilde{v}}\vert v$ in $F$. Then the action of the center $(F_v^+)^\times\subseteq \GL_n(F_v^+)$ on $S(V^{v},\F)[\m^{\Sigma}]$ via $\iota_{\tilde{v}}$ is given by ${\det}(\rbar_{\tilde{v}})\omega^{\frac{n(n-1)}{2}}$ {\upshape(}via local class field theory for $F_v^+${\upshape)}.
\end{lem}
\begin{proof}
We can assume $S(V^{v},\F)[\m^{\Sigma}]\ne 0$. The map $S(V^vU_v,\oE)\longrightarrow S(V^vU_v,\F)$ being surjective for $U_v$ small enough (see e.g.\ \cite[Lemma 4.4.1]{BH} or \cite[\S7.1.2]{EGH}), we have a surjection of smooth $H(F_v^+)$-representations:
\begin{equation}\label{surj}
S(V^v,\oE)_{\m^{\Sigma}}\twoheadrightarrow S(V^v,\F)_{\m^{\Sigma}}
\end{equation}
(where $S(V^vU_v,\oE)$, $S(V^v,\oE)_{\m^{\Sigma}}$ are defined as $S(V^vU_v,\F)$, $S(V^v,\F)_{\m^{\Sigma}}$ replacing $\F$ by $\oE$). By classical local-global compatibility applied to $\big(\ilim{U}S(U,\oE)\big)\otimes_{\oE}E$, see e.g.\ \cite[Thm.7.2.1]{EGH}, we easily deduce with (\ref{surj}) that {\it if} $(F_v^+)^\times$ acts via $\iota_{\tilde{v}}$ on the whole $S(V^v,\F)[\m^{\Sigma}]$ (inside $S(V^v,\F)_{\m^{\Sigma}}$) by a single character, then this character must be ${\det}(\rbar_{\tilde{v}})\omega^{\frac{n(n-1)}{2}}$.\\
Let us prove that $(F_v^+)^\times$ indeed acts by a character. The functor associating to any local artinian $\oE$-algebra $A$ with residue field $\F$ the set of isomorphism classes of deformations $r_A$ of $\rbar$ to $A$ such that $r_A^{c}\simeq r_A^\vee\otimes\varepsilon^{1-n}$ is pro-representable by a local complete noetherian $\oE$-algebra $R_{\rbar,\Sigma}$ of residue field $\F$. When $p>2$, $F/F^+$ is unramified at finite places and $H\times_{\oFF[1/N]} F^+$ is quasi-split at finite places of $F^+$, it follows from \cite[Prop.6.7]{Th} that there is a natural such deformation with values in $\TT^\Sigma(V^vU_v,\oE)_{\m^{\Sigma}}$ for any $U_v$ (where $\TT^\Sigma(V^vU_v,\oE)_{\m^{\Sigma}}$ is defined as $\TT^\Sigma(V^vU_v,\F)_{\m^{\Sigma}}$ in \S\ref{somprel} replacing $\F$ by $\oE$), and hence by universality a continuous morphism of local $\oE$-algebras:
\begin{equation}\label{rmap}
R_{\rbar,\Sigma}\longrightarrow \TT^\Sigma(V^vU_v,\oE)_{\m^{\Sigma}}.
\end{equation}
Likewise, the functor associating to any $A$ as above the set of isomorphism classes of $\Gal(\overline{F_{\tilde{v}}}/F_{\tilde{v}})^{\rm ab}$-deformations \ of \ ${\det}(\rbar_{\tilde{v}})$ \ over \ $A$ \ is \ pro-representable \ by \ the Iwasawa \ algebra $\oE\bbra{\Gal(\overline{F_{\tilde{v}}}/F_{\tilde{v}})^{\rm ab}}$, and considering ${\det}_A(r_A\vert_{\Gal(\overline{F_{\tilde{v}}}/F_{\tilde{v}})})$ for $A=R_{\rbar,\Sigma}$ provides by the universal property again a continuous morphism of local $\oE$-algebras:
\begin{equation}\label{zmap}
\oE\bbra{\Gal(\overline{F_{\tilde{v}}}/F_{\tilde{v}})^{\rm ab}}\longrightarrow R_{\rbar,\Sigma}.
\end{equation}
Since $\TT^\Sigma(V^vU_v,\oE)_{\m^{\Sigma}}$ acts by a character on $S(V^vU_v,\F)[{\m^{\Sigma}}]$ for any $U_v$, so is the case of $R_{\rbar,\Sigma}$ on $S(V^v,\F)[{\m^{\Sigma}}]$ by (\ref{rmap}). Using (\ref{surj}), we see that it is enough to prove that the induced morphism
\[\oE\bbra{\Gal(\overline{F_{\tilde{v}}}/F_{\tilde{v}})^{\rm ab}}\buildrel (\ref{zmap})\over \longrightarrow R_{\rbar,\Sigma}\buildrel (\ref{rmap})\over \longrightarrow \plim{U_v}\TT^\Sigma(V^vU_v,\oE)_{\m^{\Sigma}}\]
gives an action of $\Gal(\overline{F_{\tilde{v}}}/F_{\tilde{v}})^{\rm ab}$ on $S(V^v,\oE)_{\m^{\Sigma}}$ which, when restricted to $F_{\tilde{v}}^\times \hookrightarrow \Gal(\overline{F_{\tilde{v}}}/F_{\tilde{v}})^{\rm ab}$ (via the local reciprocity map), coincides with the action of $F_{\tilde{v}}^\times$ on $S(V^v,\oE)_{\m^{\Sigma}}$ as center of $H(F_v^+)\buildrel \iota_{\tilde{v}} \over \simeq G(F_{\tilde{v}})$. We can work in $S(V^v,\oE)_{\m^{\Sigma}}\otimes_{\oE}E$, in which case this follows from local-global compatibility (as in \cite[Thm.7.2.1]{EGH}) and from the fact that, by construction of the map (\ref{rmap}) (see \cite[\S6]{Th}) and by (\ref{zmap}), $\Gal(\overline{F_{\tilde{v}}}/F_{\tilde{v}})^{\rm ab}$ acts on $\pi^{V^v}\subseteq S(V^v,\oE)_{\m^{\Sigma}}\otimes_{\oE}E$ by multiplication by the character ${\det}(r_\pi)\vert_{\Gal(\overline{F_{\tilde{v}}}/F_{\tilde{v}})}$, \ where \ $\pi$ \ is \ an \ irreducible \ $\GA$-subrepresentation of $\big(\ilim{U}S(U,\oE)\big)\otimes_{\oE}E$ such that $\pi^{V^v}$ occurs in $S(V^v,\oE)_{\m^{\Sigma}}\otimes_{\oE}E$ and where $r_\pi$ is its associated (irreducible) $p$-adic representation of $\gF$ (\cite[Thm.7.2.1]{EGH} again).
\end{proof}

Let $\pi$ be a smooth representation of $G(K)=\GL_n(K)$ over $\F$ with central character $Z(\pi)$ and denote by $\pi^\star$ the smooth representation of $G(K)$ with the same underlying vector space as $\pi$ but where $g\in G(K)$ acts by $\tau(g)^{-1}$. 

\begin{lem}\label{fgz}
There is a $\gp$-equivariant isomorphism
\begin{equation*}
V_G(\pi^\star)\cong V_G(\pi)\otimes Z(\pi)^{-(n-1)}\vert_{\Qp^\times},
\end{equation*}
where $Z(\pi)\vert_{\Qp^\times}$ is seen as a character of $\gp$ via local class field theory.
\end{lem}
\begin{proof}
We use the notation of \S\ref{covariant}. Let $w_0\in W$ be the element of maximal length, the isomorphism $\pi^{N_1}\buildrel\sim\over\rightarrow \pi^{w_0N_1w_0}$, $v\mapsto w_0v$ shows that one can compute $V_G(\pi)$ using $w_0N_1w_0$ instead of $N_1$ and conjugating everything by $w_0$ (e.g.\ $x\in \Zp^\times$ acts by $w_0\xi_G(x)w_0$, etc.). Now, it is easy to check that the $\F$-linear isomorphism $(\pi^\star)^{N_1}\buildrel\sim\over\rightarrow \pi^{w_0N_1w_0}$, $v\mapsto w_0v$ is compatible with the $\F\bbra{X}[F]$-module structure on both sides but where we twist the $\F\bbra{X}[F]$-action as follows on the right-hand side: $X$ acts by $(1+X)^{-1}-1$ and $F$ acts by $p^{-(n-1)}F$, $p^{-(n-1)}$ being here in the center of $G(K)$. Likewise, it is compatible with the action of $\Zp^\times$ but where $x\in \Zp^\times$ acts by $x^{-(n-1)}\xi_G(x)$ on the right-hand side (with $x^{-(n-1)}$ in the center of $G(K)$). All this easily implies the lemma.
\end{proof}

\begin{lem}\label{fwc}
Assume $p>2$, $F/F^+$ unramified at finite places and $H\times_{\oFF[1/N]} F^+$ quasi-split at finite places of $F^+$. We have a $\gp$-equivariant isomorphism
\begin{multline*}
V_{G,{\tilde{v}}^c}\big(\Hom_{U^v}(\sigma^v,S(V^{v},\F)[\m^{\Sigma}])\big)\\
\cong V_{G,{\tilde{v}}}\big(\Hom_{U^v}(\sigma^v,S(V^{v},\F)[\m^{\Sigma}])\big)\otimes \ind_{F_{\tilde{v}}}^{\otimes\Qp}\big({\det(\rbar_{\tilde{v}})}^{-(n-1)}\omega^{{\frac{-n(n-1)^2}{2}}}\big).
\end{multline*}
\end{lem}
\begin{proof}
This follows from Lemma \ref{fgz} applied to $\pi=\Hom_{U^v}(\sigma^v,S(V^{v},\F)[\m^{\Sigma}])$ together with Lemma \ref{central}, recalling that $Z(\pi)\vert_{\Qp^\times}$, seen as a character of $\gp$ via local class field theory, is $\ind_{F_{\tilde{v}}}^{\otimes\Qp}\!(Z(\pi))$ (where $Z(\pi)$ is here seen as a character of $\Gal(\overline{F_{\tilde{v}}}/F_{\tilde{v}})$).
\end{proof}

\begin{prop}
Assume \ $p>2$, \ $F/F^+$ \ unramified \ at \ finite \ places \ and \ $H\times_{\oFF[1/N]} F^+$ quasi-split at finite places of $F^+$. Conjecture \ref{theconjbar} holds for ${\tilde{v}}$ if and only if it holds for ${\tilde{v}}^c$.
\end{prop}
\begin{proof}
This follows from Lemma \ref{fwc} together with $\rbar_{{\tilde{v}}^c}\cong \rbar_{\tilde{v}}^\vee\otimes\omega^{1-n}$, (\ref{dual}) and an easy computation.
\end{proof}

\subsubsection{A reformulation using \texorpdfstring{$C$}{C}-groups}\label{Cgroup}

We show that one can give a more general and more natural formulation of Conjecture \ref{theconjbar} (in the special case of Remark \ref{theconjbarrem}(i)) using $C$-parameters (Conjecture \ref{conj:generale}).

We start by some reminders about $L$-groups and $C$-groups.

Let $k$ be a field and $k^{\sep}$ a separable closure of $k$. We note $\Gamma_k\defeq \Gal(k^{\sep}/k)$. Let $H$ be a connected reductive group defined over $k$, let $\widehat{H}$ be its dual group, ${}^LH$ its $L$-group and ${}^CH$ its $C$-group. We refer to \cite[\S2]{BorelCorvallis}, \cite[\S\S2,5]{BG}, \cite[\S9]{GHS} and \cite[\S1.1]{ZhuSatake} for more details concerning these $L$-groups and $C$-groups. Note that these two groups can be defined over $\Z$. Their construction depends on the choice of a pinning $(B_{H},T_{H},\set{x_{\alpha}}_{\alpha\in S_{H}})$ of $H_{k^{\sep}}$. The dual group $\widehat{H}$ has a natural pinned structure $(B_{\widehat{H}},T_{\widehat{H}},\set{x_{\widehat{\alpha}}}_{\alpha\in S_H})$ with $B_{\widehat{H}}$ a Borel subgroup of $\widehat{H}$, $T_{\widehat{H}}\subset B_{\widehat{H}}$ a maximal split torus and $\set{x_{\widehat{\alpha}}}_{\alpha\in S_H}$ a pinning of $(B_{\widehat{H}},T_{\widehat{H}})$ (see \cite[\S\S5,6]{Conrad} for the fact that everything can be defined over $\Z$) on which the group $\Gamma_k$ is acting. Let $1\rightarrow\mathbb{G}_m\rightarrow\widetilde{H}\rightarrow H\rightarrow1$ be the central $\mathbb{G}_m$-extension of $H$ (over $k$) whose existence is proved in \cite[Prop.5.3.1(a)]{BG}. The inverse images $T_{\widetilde{H}}$ and $B_{\widetilde{H}}$ of $T_{H}$ and $B_{H}$ in $\widetilde{H}_{k^{\sep}}$ are respectively a maximal torus and a
Borel subgroup of $\widetilde{H}_{k^{\sep}}$. Moreover, since the above extension is central, there is a unique pinning $\set{\widetilde{x}_\alpha}_{\alpha\in S_H}$ of $(B_{\widetilde{H}},T_{\widetilde{H}})$ inducing $\set{x_\alpha}_{\alpha\in S_H}$ on $(B,T)$ via the map $\widetilde{H}_{k^{\sep}}\rightarrow H_{k^{\sep}}$. This gives rise to a pinned dual data
$(\widehat{\widetilde{H}},B_{\widehat{\widetilde{B}}},T_{\widehat{\widetilde{H}}},\set{\widetilde{x}_{\widehat{\alpha}}}_{\alpha\in S_H})$ with an action of $\Gamma_k$ (trivial on some open subgroup) and a $\Gamma_k$-equivariant injection $(\widehat{H},B_{\widehat{H}},T_{\widehat{H}})\hookrightarrow(\widehat{\widetilde{H}},B_{\widehat{\widetilde{H}}},T_{\widehat{\widetilde{H}}})$ such that $\set{x_{\widehat{\alpha}}}_{\alpha\in S_H}$ induces $\set{\widetilde{x}_{\widehat{\alpha}}}_{\alpha\in S_H}$.

The \emph{$L$-groups and $C$-groups} are then defined as the group schemes
\begin{equation}\label{LCparameter}
{}^LH\defeq\widehat{H}\rtimes\Gamma_k \ \quad {}^CH\defeq\widehat{\widetilde{H}}\rtimes\Gamma_k.
\end{equation}
We have the following simple description of $\widehat{\widetilde{H}}$ given in
\cite[\S1.1]{ZhuSatake}. Let $\widehat{H}^{\ad}$ and $T_{\widehat{H}}^{\ad}$ be the quotients of $\widehat{H}$ and
$T_{\widehat{H}}$ by the center of $\widehat{H}$ and let $\delta_{\ad}$ be the cocharacter of
$T_{\widehat{H}}^{\ad}\subset\widehat{H}^{\ad}$ defined as the half sum of positive roots of $\widehat{H}$ with respect to $(B_{\widehat{H}},T_{\widehat{H}})$. The group $\widehat{H}^{\ad}$ acts on $\widehat{H}$ by the adjoint action and, after precomposition with $\delta_{\ad}$, this defines an action, in the category of $\Z$-group schemes, of $\mathbb{G}_m$ on $\widehat{H}$. There is an isomorphism of $\Z$-group schemes $\widehat{\widetilde{H}}\simeq\widehat{H}\rtimes\mathbb{G}_m$ identifying $B_{\widehat{\widetilde{H}}}$ with $B_{\widehat{H}}\rtimes\mathbb{G}_m$ and $T_{\widehat{\widetilde{H}}}$ with $T_{\widehat{H}}\rtimes\mathbb{G}_m=T_{\widehat{H}}\times\mathbb{G}_m$. We
note that, since $\delta_{\ad}$ is fixed by the Galois action, this isomorphism is Galois equivariant. Using this isomorphism, we identify $X(T_{\widehat{\widetilde{H}}})$ with $X(T_{\widehat{H}})\times\Z\simeq X^\vee(T_H)\times\Z$. This shows that we have an exact sequence of $\Z$-group schemes:
\[ 1\longrightarrow{}^LH\longrightarrow{}^CH\xrightarrow{\ d\ }\mathbb{G}_m\longrightarrow1.\]

Let $A$ be a topological $\Zp$-algebra and assume from now on that $k$ is either a number field or a finite extension of $\Qp$, so that we have an $A$-valued $p$-adic cyclotomic character. We recall that a morphism $\rho : \Gamma_k\rightarrow{}^LH(A)$ is called \emph{admissible} if its composition with the second projection ${}^LH(A)\rightarrow\Gamma_k$ is the identity (see \cite[\S3]{BorelCorvallis}).

\begin{definit}
An \emph{$L$-parameter} (resp.~\emph{$C$-parameter}) of $H$ over $A$ is an admissible
continuous morphism $\rho : \Gamma_k\longrightarrow{}^LH(A)$ (resp.~$\rho : \Gamma_k\longrightarrow{}^CH(A)$ such that $d\circ\rho$ is the $p$-adic cyclotomic character). When $A$ is moreover an algebraically closed field, we say that two $L$-parameters (resp.~$C$-parameters) of $H$ over $A$ are \emph{equivalent} if they are conjugate by an element of $\widehat{H}(A)$ (resp.~$\widehat{\widetilde{H}}(A)$). 
\end{definit} 

\begin{rem}\label{rem:conjugate_in_Cgroup}
Assume $A$ is an algebraically closed field. Each element of $\widehat{\widetilde{H}}(A)$ is the product of an element of $\widehat{H}(A)$ and an element of the center of $\widehat{\widetilde{H}}(A)$. This can be deduced from \cite[Prop.5.3.3]{BG} or \cite[(1.2)]{ZhuSatake}. This implies that two $C$-parameters of $H$ over $A$ are equivalent if and only if they are conjugate by an element of $\widehat{H}(A)$.
\end{rem}

For simplicity, we assume from now on that $A$ is moreover an algebraically closed field. We also assume (not for simplicity) that $H$ has a connected center. We generalize now the representation $\LLbar(\rhobar)\otimes_{\F}\Fpbar$ (see (\ref{tensorind}) for $\LLbar(\rhobar)$).

Let $(\lambda_{\alpha^\vee})_{\alpha\in S_H}$ be a family of fundamental coweights of $H$ such that
\begin{equation}\label{xih}
\xi_H\defeq \sum_{\alpha\in S_H}\lambda_{\alpha^\vee}\in X(T_{\widehat H})\cong X^\vee(T_H)
\end{equation}
is fixed under the action of $\Gamma_k$ (compare with (\ref{coweights}) and note that the cocharacters $\lambda_{\alpha^\vee}$ exist since $H$ has a connected center but each of them doesn't have to be fixed by $\Gamma_k$). Let $(r_{\lambda_{\alpha^\vee}},V_{\lambda_{\alpha^\vee}})$ be the irreducible algebraic representation of $\widehat{H}$ of highest weight $\lambda_{\alpha^\vee}$ over $A$ and let $(r_{\xi_H}^{\otimes},V_{\xi_H}^\otimes)$ be the irreducible algebraic representation of $\widehat{H}^{S_H}$ over $A$ of highest weight $(\lambda_{\alpha^\vee})_{\alpha\in S_H}=$ the character of $T_{\widehat{H}}^{S_H}$ defined by $(x_\alpha)_{\alpha\in S_H}\mapsto\sum_{\alpha}\lambda_{\alpha^\vee}(x_\alpha)$. Note that we have an isomorphism of algebraic representations of $\widehat{H}^{S_H}$:
\begin{equation}\label{highestweight}
(r_{\xi_H}^\otimes,V_{\xi_H}^\otimes)\cong \bigotimes_{\alpha\in S_H} (r_{\lambda_{\alpha^\vee}},V_{\lambda_{\alpha^\vee}}).
\end{equation}
Let $\gamma\in\Gamma_k$ and $\chi_{\alpha,\gamma}$ be the character of $\widehat{H}$ corresponding to the cocharacter $\gamma(\lambda_{\alpha^\vee})-\lambda_{\gamma\alpha^\vee}\in X^\vee(Z_H)\subset X^\vee(T_H)$. Comparing the highest weights, for $\gamma\in\Gamma_k$ there is an isomorphism of algebraic irreducible representations of $\widehat{H}^{S_H}$:
\[ \big(r_{\xi_H}^\otimes(\gamma^{-1}\cdot),V_{\xi_H}^\otimes\big)\simeq\left(\otimes_{\alpha\in
      S_H}(r_{\lambda_{\alpha^\vee}}\otimes\chi_{\gamma^{-1}\alpha,\gamma})\circ
    c_\gamma,V_{\xi_H}^\otimes\right),\]
where $c_\gamma$ is the automorphism of $\widehat{H}^{S_H}$ defined by $(x_\alpha)_{\alpha\in S_H}\mapsto (x_{\gamma^{-1}\alpha})_{\alpha\in S_H}$. Therefore there exists an $A$-linear automorphism $M_\gamma$ of $V^\otimes_{\xi_H}$, well defined up to a nonzero scalar, such that, for $(x_\alpha)_{\alpha\in S_H}\in\widehat{H}(A)^{S_H}$:
\begin{equation}\label{isomgamma}
M_\gamma \left(r_{\xi_H}^\otimes((\gamma^{-1}x_{\alpha})_{\alpha\in S_H})
  \right)M_\gamma^{-1}=\left(\otimes_{\alpha\in
      S_H}r_{\lambda_{\alpha^\vee}}(x_{\gamma^{-1}\alpha})\right)\prod_{\alpha\in
    S_H}\chi_{\alpha,\gamma}(x_\alpha).
\end{equation}
Moreover the subspaces of highest weight of these two representations over $V_{\xi_H}^\otimes$ being the same, we can choose $M_\gamma$ such that it induces the identity on this line. With this choice, the map $\gamma\mapsto M_\gamma$ is a representation of $\Gamma_k$ over $V^\otimes_{\xi_H}$. Since $\xi_H\in X^\vee(T_H)^{\Gamma_k}$, we have $\prod_{\alpha\in S_H}\chi_{\alpha,\gamma}=1$ for all $\gamma\in\Gamma_k$ so that, for $x\in \widehat{H}(A)$, we have from (\ref{isomgamma}) and (\ref{highestweight}) (replacing $\gamma^{-1}x_{\alpha}$ by $x$ for all $\alpha\in S_H$):
\[ M_\gamma \left(\otimes_{\alpha\in
      S_H}r_{\lambda_{\alpha^\vee}}(x)\right)
  M_\gamma^{-1}=\left(\otimes_{\alpha\in
      S_H}r_{\lambda_{\alpha^\vee}}(\gamma x)\right).\]
All this proves that there is an algebraic representation $(L^\otimes_{\xi_H},V^\otimes_{\xi_H})$ of ${}^LH$ on
$V^\otimes_{\xi_H}$ defined by
\[ L^{\otimes}_{\xi_H}(x,\gamma)\defeq \left(\otimes_{\alpha\in S_H}r_{\lambda_{\alpha^\vee}}(x)\right) M_\gamma\] for $x\in\widehat{H}(A)$ and $\gamma\in\Gamma_k$. The isomorphism class of this representation does not depend on the choice of the $\lambda_{\alpha^\vee}$ such that $\xi_H=\sum\lambda_{\alpha^\vee}$. Namely any other choice will twist each $r_{\lambda_{\alpha^\vee}}$ by a character whose product over all $\alpha$ is trivial.

If $\rho$ is an $L$-parameter of $H$ over $A$ we define the $\Gamma_k$-representation $L^{\otimes}_{\xi_H}(\rho)$ as the composition $L^\otimes_{\xi_H}\circ\rho$. Moreover if two $L$-parameters $\rho_1$ and $\rho_2$ are equivalent, the representations $L^\otimes_{\xi_H}(\rho_1)$ and $L^\otimes_{\xi_H}(\rho_2)$ are clearly isomorphic. If $\rho$ is a $C$-parameter of $H$ over $A$, $\rho$ is in particular an $L$-parameter of $\widetilde{H}$ over $A$ by (\ref{LCparameter}), and we define the $\Gamma_k$-representation $L^{\otimes,C}_{\xi_H}(\rho)\defeq L^\otimes_{\xi_{\widetilde{H}}}(\rho)$, where
\begin{equation}\label{xitilde}
\xi_{\widetilde{H}}\defeq (\xi_H,0)\in X(T_{\widehat{\widetilde{H}}})\simeq X(T_{\widehat{H}})\times\Z.
\end{equation}

We now compare $L^{\otimes}_{\xi_H}(\rho)$, $L^{\otimes,C}_{\xi_H}(\rho)$ between $k$ and finite extensions $k'$ of $k$.

We fix $k'\subset k^{\sep}$ a finite extension of $k$, $H'$ a connected reductive group over $k'$ and we let $H\defeq \Res_{k'/k}(H')$. We let $\Sigma_{k'}$ be the set of embeddings $k'\hookrightarrow k^{\sep}$ inducing the identity on $k$ and $\tau_0\in\Sigma_{k'}$ the inclusion $k'\subset k^{\sep}$. For $\tau\in\Sigma_{k'}$ we choose $g_\tau\in\Gamma_k$ such that $\tau=g_{\tau}\circ\tau_0$, and we have $\Gamma_k=\coprod_{\tau\in\Sigma_{k'}}g_\tau\Gamma_{k'}$. The dual group $\widehat H$ of $H$ is isomorphic to $\ind_{\Gamma_{k'}}^{\Gamma_k}\widehat{H'}$, i.e.\ the group scheme of functions $f : \Gamma_k\rightarrow \widehat{H'}$ such that $f(gh)=h^{-1}f(g)$ for all $g\in\Gamma_k$ and $h\in\Gamma_{k'}$ (see \cite[\S5.1(4)]{BorelCorvallis}). More explicitly, the map $f\mapsto (f(g_\tau))_{\tau\in\Sigma_{k'}}$ induces an isomorphism $\ind_{\Gamma_{k'}}^{\Gamma_k}\widehat{H'}\simeq\widehat{H'}^{\Sigma_{k'}}$ and the action of $\Gamma_k$ on $\widehat{H'}^{\Sigma_{k'}}$ is given by
\[g\cdot(x_\tau)_{\tau\in\Sigma_{k'}}=\big((g_\tau^{-1}gg_{g^{-1}\circ\tau})x_{g^{-1}\circ\tau}\big)_{\tau\in\Sigma_{k'}}.\]
The map $(x_\tau)_{\tau\in\Sigma_{k'}}\mapsto x_{\tau_0}$ is a $\Gamma_{k'}$-equivariant map $\widehat{H}\rightarrow\widehat{H'}$. It extends to a morphism of group schemes $\widehat{H}\rtimes\Gamma_{k'}\rightarrow{}^LH'$ (resp.~$\widehat{\widetilde{H}}\rtimes\Gamma_{k'}\rightarrow{}^CH'$) inducing the identity on the $\Gamma_{k'}$ factor (resp.~the $\mathbb{G}_m$ and $\Gamma_{k'}$ factors). If $\rho$ is an $L$-parameter (resp.~a $C$-parameter) of $H$ over $A$, we can define an $L$-parameter (resp.~a $C$-parameter) $\rho'$ of $H'$ by restriction of $\rho$ to $\Gamma_{k'}$ and composition with the above morphism.

\begin{lem}\label{lem:shapiro}
The map $\rho\mapsto\rho'$ induces a bijection between equivalence classes of $L$-parameters {\upshape(}resp.~of $C$-parameters{\upshape)} of $H$ over $A$ and equivalence classes of $L$-parameters {\upshape(}resp.~$C$-parameters{\upshape)} of $H'$ over $A$.
\end{lem}
\begin{proof}
A map $\rho$ from $\Gamma_k$ to ${}^LH(A)$ of the form $(c_\rho,\Id)$ is admissible if and only if $c_\rho$ is a $1$-cocycle of $\Gamma_k$ in $\widehat{H}(A)$ and is continuous if and only if $c_\rho$ is continuous. Moreover two admissible $\rho$ are equivalent if and only if they are conjugate by an element of $\widehat{H}(A)$. Therefore the map associating to $\rho$ the class $[c_\rho]$ of $c_\rho$ induces a bijection between the set of equivalence classes of $L$-parameters and the set of classes $[c]\in H^1_{\cont}(\Gamma_k,\widehat{H}(A))$. The fact that the above map $\rho\mapsto\rho'$ induces an isomorphism $H^1_{\cont}(\Gamma_k,\widehat{H}(A))\buildrel\sim\over\rightarrow H^1_{\cont}(\Gamma_{k'},\widehat{H'}(A))$ is a consequence of a nonabelian version of Shapiro's Lemma (see for example \cite[Prop.8]{Stix} noting that everything can be made continuous there or \cite[Lemma 9.4.1]{GHS} in a more restricted context).

Therefore the map associated to a $C$-parameter $\rho$ the class $[c_\rho]$ of $c_\rho$ induces a bijection between the set of equivalence classes of $C$-parameters and the set of classes $c\in H^1_{\cont}(\Gamma_k,\widehat{\widetilde{H}}(A))$ such that $d(c)\in H^1_{\cont}(\Gamma_k,A^\times)\simeq\Hom_{\mathrm{gp}}^{\cont}(\Gamma_k,A^\times)$ coincides with the $p$-adic cyclotomic character. Let $\widetilde{H}_1\defeq\Res_{k'/k}\widetilde{H'}$, so that $\widetilde{H}$ can be identified to a quotient of $\widetilde{H}_1$. It follows from Remark \ref{rem:conjugate_in_Cgroup} that $H^1(\Gamma_k,\widehat{\widetilde{H_1}}(A))$ is the set of classes of $1$-cocycles of $\Gamma_k$ with values in $\widehat{\widetilde{H_1}}(A)$ up to $\widehat{H}(A)$-conjugation. It follows again from Remark \ref{rem:conjugate_in_Cgroup} that the set of equivalence classes of $C$-parameters of $H$ over $A$ is in bijection with the subset of $H^1_{\cont}(\Gamma_k,\widehat{\widetilde{H}_1}(A))$ of classes whose image in $H^1_{\cont}(\Gamma_k,(A^\times)^{[k':k]})\simeq\Hom_{\mathrm{gp}}^{\cont}(\Gamma_k,(A^\times)^{[k':k]})$ is the image of the $p$-adic cyclotomic character via the diagonal embedding $A^\times\hookrightarrow(A^\times)^{[k':k]}$. The conclusion follows from the commutativity of the following diagram
\[\begin{tikzcd}
      H^1_{\cont}\big(\Gamma_k,\widehat{\widetilde{H}}(A)\big)\ar[r]
      \ar[d,"\wr"] &
      H^1_{\cont}\big(\Gamma_k,(\widehat{\Res_{k'/k}\mathbb{G}_m})(A)\big)
      \ar[d,"\wr"] \\
      H^1_{\cont}\big(\Gamma_{k'},\widehat{\widetilde{H'}}(A)\big)\ar[r]
      & H^1_{\cont}\big(\Gamma_{k'},\widehat{\mathbb{G}_m}(A)\big)
    \end{tikzcd}\]
and from the fact that the classes corresponding to the cyclotomic characters correspond under the right vertical arrow.
\end{proof}

\begin{lem}\label{lem:tens_induction}
Let $\rho$ be an $L$-parameter, resp.~a $C$-parameter, of $H$ over $A$ and $\rho'$ the $L$-parameter, resp.~$C$-parameter, of $H'$ over $A$ corresponding to $\rho$ by Lemma \ref{lem:shapiro}. Let $\xi_{H'}\in X(T_{\widehat{H'}})$ be as in {\upshape(\ref{xih}) (}with $H'$ instead of $H${\upshape)} and let $\xi_{H}\in X(T_{\widehat{H}})\simeq X(T_{\widehat{H'}})^{\Sigma_{k'}}$ be the character $(\xi_{H'})_{\tau\in\Sigma_{k'}}$ {\upshape(}which is fixed by $\Gamma_k${\upshape)}. Then we have an isomorphism of representations of $\Gamma_k$ over $A$:
\[ L^\otimes_{\xi_H}(\rho)\simeq\ind_{k'}^{\otimes k}\big(L^\otimes_{\xi_{H'}}(\rho')\big)\ \ \ \ 
    \text{resp. } L^{\otimes,C}_{\xi_H}(\rho)\simeq\ind_{k'}^{\otimes
      k}\big(L^{\otimes,C}_{\xi_{H'}}(\rho')\big).\]
\end{lem}
\begin{proof}
Let $\rho'$ be an $L$-parameter of $H'$ over $A$. If $g\in\Gamma_k$ and $\tau\in\Sigma_{k'}$, let $gg_\tau=g_{g\circ\tau}h(g,\tau)$ with $h(g,\tau)\in\Gamma_{k'}$. For $g\in\Gamma_k$, we can check that the
above automorphism $M_g$ of $V^\otimes_{\xi_H}=(V_{\xi_{H'}}^\otimes)^{\otimes [k':k]}$ is defined by $M_g(\otimes_{\tau\in\Sigma_{k'}} v_{\tau})=\otimes_{\tau\in\Sigma_{k'}}(M_{h(g,g^{-1}\circ\tau)}v_{g^{-1}\circ\tau})$. Moreover, setting for $g\in\Gamma_k$:
\[\rho(g)\defeq \big((\rho'(h(g,g^{-1}\circ\tau))_{\tau\in\Sigma_{k'}},g\big)\in\widehat{H'}(A)^{\Sigma_{k'}}\rtimes\Gamma_k\]
it is easy to check that $\rho$ is an admissible morphism and that the equivalence class of $\rho$ corresponds to $\rho'$ via Lemma \ref{lem:shapiro}. The result follows from an explicit computation together with the definition of the tensor induction (\cite[\S13]{Curtis-Reiner1}, see also the end of the proof of Lemma \ref{galqp} below). The case of $C$-parameters can be deduced from the case of $L$-parameters as in the proof of Lemma \ref{lem:shapiro}.
\end{proof}

We will later need to ``untwist'' a $C$-parameter into an $L$-parameter. This can be done when the group $H$ has a twisting element (as we assumed in \S\ref{covariant}), i.e.\ a character $\theta_H\in X(T_H)^{\Gamma_k}\simeq X^\vee(T_{\widehat{H}})^{\Gamma_k}$ such that $\scalar{\theta_H,\alpha^\vee}=1$ for all $\alpha\in S_{H}$. By \cite[(1.3)]{ZhuSatake}, there exists a Galois equivariant isomorphism $\widehat{\widetilde{H}}\simeq\widehat{H}\times\mathbb{G}_m$ given
explicitly by
\[t_{\theta_H} : 
  \begin{array}{ccc}
    \widehat{H}\rtimes\mathbb{G}_m & \simeq &
                                              \widehat{H}\times\mathbb{G}_m \\
    (h,t)&\mapsto&(h\theta_H(t),t).
  \end{array}\]
This induces an isomorphism of group schemes ${}^CH\simeq{}^LH\times\mathbb{G}_m$. The choice of $\theta_H$ gives a bijection between the equivalence classes of $C$-parameters and of $L$-parameters of $H$ over $A$ given by $\rho^C\mapsto\rho$, so that $t_{\theta_H}\circ\rho^C\simeq(\rho,\varepsilon)$, where $\varepsilon$ is (the image in $A^\times$) of the $p$-adic cyclotomic character.

Let $\xi_H\in X^\vee(T_H)^{\Gamma_k}\simeq X(T_{\widehat{H}})^{\Gamma_k}$ be a dominant character of $\widehat{H}$ fixed by $\Gamma_k$ as above. The algebraic representation $r_{\xi_{\widetilde H}}\circ t_{\theta_H}^{-1}$ of $\widehat{H}\times\mathbb{G}_m$ (see (\ref{xitilde}) for $\xi_{\widetilde H}$) is the representation of highest weight $(\xi_H,-\scalar{\xi_H,\theta_H})$ and similarly $L^\otimes_{\xi_{\widetilde H}}\circ t_{\theta_H}^{-1}=L^\otimes_{\xi_H}\otimes x^{-\scalar{\xi_H,\theta_H}}$ (where we note $x^h$ the character $x\mapsto x^h$ of $\mathbb{G}_m$). This proves that we have
\begin{equation}\label{eq:twisting}
  L^{\otimes,C}_{\xi_H}(\rho^C)\simeq
  L^{\otimes}_{\xi_H}(\rho)\otimes\varepsilon^{-\scalar{\xi_H,\theta_H}}.
\end{equation}

On order to state the reformulation/generalization Conjecture \ref{theconjbar} (more precisely of its variant in Remark \ref{theconjbarrem}(i) and extending scalars from $\F$ to $\Fpbar$), we broaden the global setting of \S\ref{somprel} following \cite{DPS}.

We now let $H$ be a connected reductive group defined over $\Q$. We fix some compact open subgroup $U^p\subset H(\mathbb{A}_{\Q}^{\infty,p})$ satisfying the hypotheses of \cite[\S9.2]{DPS} ($U^p$ there is denoted $K_f^p$). For $i\geq0$ an integer, let $\widetilde{H}^i(\F_p)$ be the completed cohomology of the tower of locally symmetric spaces associated to $H$ of tame level $U^p$ defined in \cite{emerton-inv} (see \cite[\S9.2]{DPS}). Let $\Sigma$ be a set of finite places of $\Q$ containing $p$ and the places of $\Q$ where $H$ is not unramified or $U^p$ is not hyperspecial. Let $\mathbb{T}^{\Sigma}$ be the abstract Hecke algebra defined as the tensor product of the spherical $\Z[p^{-1}]$-Hecke algebras $\mathcal{H}_\ell$ of $H(\Q_\ell)$ with respect to $U^p_\ell$. We recall that a maximal open ideal $\m\subset\mathbb{T}^\Sigma$ is \emph{weakly non-Eisenstein} \cite[Def.9.13]{DPS} if the equivalent assumptions of \cite[Lemma 9.10]{DPS} are satisfied. In this case there is a unique $q_0\geq0$ such that $\widetilde{H}^{q_0}(\F_p)_{\mathfrak{m}}\neq0$. Then the $H(\mathbb{Q}_p)$-representation $\widetilde{H}^{q_0}(\F_p)[\mathfrak{m}]$ is smooth and admissible and the residue field of $\m$ is finite. We choose an embedding $\mathbb{T}/\mathfrak{m}\hookrightarrow\Fpbar$.

Considering \cite[Conj.9.3.1]{DPS}, the following construction is natural. Let $\rbar^C : \Gal(\overline{\Q}/\Q)\rightarrow {}^CH(\Fpbar)$ be a $C$-parameter unramified outside a finite number of primes and choose $\Sigma$ big enough to contain all the primes of ramification of $\rbar^C$. For each $\ell\notin\Sigma$, let $x_\ell : \mathcal{H}_\ell\rightarrow\Fpbar$ be the character such that the semisimplification of $\rbar^C(\Frob_\ell)$ is contained in the $\widehat{H}(\Fpbar)$-conjugacy class $CC(x_\ell)\zeta(\ell^{-1})$ of ${}^CH(\Fpbar)$ defined by the version of Satake isomorphism for $C$-groups in \cite{ZhuSatake} and $\zeta$ is the cocharacter $t\mapsto (2\delta_{\ad}(t^{-1}),t^2)$ of
$\widehat{\widetilde{H}}$ (recall $\delta_{\ad}$ is defined at the beginning of this section). We define $\mathfrak{m}^\Sigma$ as the maximal ideal of $\mathbb{T}^\Sigma$ generated by the kernels of all the $x_\ell$ with $\ell\notin\Sigma$. Note that this gives us a natural embedding $\mathbb{T}^\Sigma/\m^\Sigma\hookrightarrow\Fpbar$.

Assume that $H_{\Qp}\defeq H\times_{\Q}{\Qp}$ is isomorphic to $\Res_{K/\Qp}(H')$ for a finite extension $K$ of $\Qp$ and some split connected reductive group $H'$ over $K$ (in particular $H_{\Qp}$ is quasi-split) and that $H'$ has a connected center. Then we can fix a cocharacter $\xi_{H'}$ of $H'$ such that $\scalar{\xi_{H'},\alpha}=1$ for all $\alpha\in S_{H'}$ and define $\xi_{H_{\Qp}}\defeq \Res_{K/\Qp}(\xi_{H'})\vert_{\mathbb{G}_{m}}$ (restriction to the diagonal embedding $\mathbb{G}_m\hookrightarrow \Res_{K/\Qp}(\mathbb{G}_m)=\mathbb{G}_m^{[K:\Qp]}$), which is a cocharacter of $H_{\Qp}$ satisfying $\scalar{\xi_{H_{\Qp}},\alpha}=1$ for all $\alpha\in S_{H_{\Qp}}$. We can finally conjecture:

\begin{conj}\label{conj:generale}
Assume that the $H(\Qp)$-representation $\pi\defeq \widetilde{H}^{q_0}(\F)[\m^\Sigma]$ is nonzero. Then $D^\vee_{\xi_{H'}}(\pi)$ {\upshape(}defined similarly to {\upshape(}\ref{DxiH}{\upshape)}{\upshape)} is finite-dimensional over $\Fpbar\ppar{X}$ and there is an integer $d\in\Z_{>0}$ such that we have an isomorphism of representations of $\gp$ over $\Fpbar$:
\[{\bf V}^\vee\big(D^\vee_{\xi_{H'}}(\pi)\big)\otimes_{\mathbb{T}^\Sigma/\mathfrak{m}^\Sigma}\Fpbar\simeq
    \big(L^{\otimes,C}_{\xi_{H_{\Qp}}}\big(\rbar^C|_{\gp}\big)\big)^{\oplus d}.\]
\end{conj}

We now check that, when $H$ is the restriction of scalars of a compact unitary group as in \S\ref{somprel}, Conjecture \ref{conj:generale} is equivalent to the special case of Conjecture \ref{theconjbar} in Remark \ref{theconjbarrem}(i) where the coefficient field is $\Fpbar$ instead of $\F$. 

We go back to the notation of \S\S\ref{somprel}, \ref{wlgc} and we fix an embedding $\F\hookrightarrow\Fpbar$. For simplification we assume that there is a unique place $v$ of $F^+$ over $p$ and we fix $\tilde v$ in $F$ above $v$, so that we have an isomorphism $(\Res_{F^+/\Q}H)\times_{\Q}{\Qp}\simeq\Res_{F_{\tilde v}/\Q_p}\GL_n$. The base field $k$ at the beginning is now $F^+$, the connected reductive group $H$ over $k$ is the compact unitary group $H$ of \S\ref{somprel} (so that $\widehat H\simeq G=\GL_n$), $\xi_{H}$ is the cocharacter $\xi_G$ of Example \ref{exdelta}, the twisting element $\theta_H$ is the character $\theta_G$ of Example \ref{exdelta} and the algebraically closed field $A$ is $\Fpbar$.

Let $\rbar$ be a continuous irreducible representation $\Gal(\overline{\Q}/F)\rightarrow\GL_n(\Fpbar)$ as in \S\ref{wlgc} (composed with our embedding $\F\hookrightarrow\Fpbar$). Let $\rbar' : \Gal(\overline{\Q}/F^+)\rightarrow\mathcal{G}_n(\Fpbar)$ be the continuous morphism associated to $\rbar$ using \cite[Lemma 2.1.4]{CHT} and denote by $(\rbar')^C:\Gal(\overline{\Q}/F^+)\rightarrow {}^CH(\Fpbar)$ the $C$-parameter of $H$ over $\Fpbar$ obtained by the construction of \cite[\S8.3]{BG}. A simple computation shows that $(\rbar')^C$ (or more precisely its composition with $\widehat{\widetilde{H}}\rtimes\Gal(\overline{\Q}/F^+)\twoheadrightarrow\widehat{\widetilde{H}}\rtimes\Gal(F/F^+)$) is the composition of $(\rbar',\omega)$ with
\begin{equation}\label{eq:GntoC}
  \begin{array}{cll}
    \mathcal{G}_n\times\mathbb{G}_m&\longrightarrow&\widehat{H}\rtimes(\mathbb{G}_m\times\Gal(F/F^+)) \\
                                  (g,\mu,\gamma,\lambda) &\longmapsto& (g\theta_H'(\lambda)^{-1},\lambda,\gamma)
  \end{array}
\end{equation}
where $g\in\GL_n(\Fpbar)$, $\mu,\lambda\in\Fpbar^\times$, $\gamma\in\Gal(F/F^+)$ and $\theta_H'\in X(T)$ is the character $\theta'_H({\rm diag}(x_1,\ldots,x_n))= x_2^{-1}x_3^{-2}\cdots x_{n}^{1-n}$. Finally we define $\rbar^C$ as the $C$-parameter of $\Res_{F^+/\Q}(H)$ over $\Fpbar$ obtained from the application of Lemma \ref{lem:shapiro} to $(\rbar')^C$. We can check that the maximal ideal $\m^\Sigma$ of $\mathbb{T}^\Sigma$ defined
by $\rbar^C$ coincides with the ideal $\m^\Sigma$ defined in \S\ref{wlgc}. This can be checked using the formulas relating the Satake isomorphism for $C$-groups with the usual Satake isomorphism (\cite[\S1.4]{ZhuSatake}) and the explicit formulas \cite[(3.13)]{GrossSatake}, \cite[(3.14)]{GrossSatake}.

Note that, seeing now $\theta_H$ and $\theta'_H$ as {\it cocharacters} of $T$ (recall $\widehat{\GL_n}\simeq \GL_n$), we have $\theta_H\circ\omega=(\theta'_H\circ\omega)\omega^{n-1}$, so that we have, using \eqref{eq:GntoC}:
\[(\rbar')^C=t_{\theta_H}^{-1}\circ((\rbar\otimes\omega^{n-1}),\omega).\]
Let $\xi_v\defeq \xi_{H}\times_{F^+}F_v^+$ and $\xi_p\defeq \Res_{F_v^+/\Q_p}(\xi_v)|_{\mathbb{G}_{m}}$. Then \eqref{eq:twisting} and Lemma \ref{lem:tens_induction} imply (note that $\xi_v$ is fixed by $\Gal(\overline{\Q}_p/F_v^+)$ since $H\times_{F^+}F_v^+$ is split):
\[L^{\otimes,C}_{\xi_p}(\rbar^C|_{\gp})\simeq\ind_{F_v^+}^{\otimes\Q_p}\big(r_{\xi_{v}}^\otimes(\rbar_{\tilde{v}}\otimes\omega^{n-1})\omega^{-\scalar{\xi_H,\theta_H}}\big)=\LLbar(\rbar_{\tilde{v}})\otimes\delta_G^{-1}.\]
This shows that Conjecture \ref{conj:generale} is equivalent to the special case of Conjecture \ref{theconjbar} in Remark \ref{theconjbarrem}(i) (with $\Fpbar$ instead of $\F$).

\subsection{Good subquotients of \texorpdfstring{$\LLbar$}{L\^{}\textbackslash otimes}}\label{goodcomponent}

From now on we assume that $K$ is unramified (i.e.\ $K=\Qpf$), and we remind the reader that $G = \GL_n/\Z$. We define the algebraic representation $\LLbar$ of $\prod_{\sigma\in \gKQ}\!G$ together with ``good subquotients'' of $\LLbar$, and prove various properties of these good subquotients. This section is entirely on the ``Galois side'' (though no Galois representation appears yet). All the results, except Remark \ref{gln2}, in fact hold for any split reductive connected algebraic group $G/\Z$ with connected center.

\subsubsection{Definition and first properties}

We define good subquotients of $\LLbar$. 

If $H$ is an algebraic group over $\Z$, we now write $H$ instead of $H\times_{\Z}\F$ (in order not to burden the notation) and $H^{\gKQ}$ for the group product $\prod_{\sigma\in \gKQ}\!H$ (it is not a subgroup of $H$!).

We define the following algebraic representation of ${G}^{\gKQ}$ over $\F$:
\begin{equation}\label{lbar}
\LLbar\defeq \bigotimes_{\gKQ}\Big(\bigotimes_{\alpha\in S}\Lbar(\lambda_{\alpha})\Big)
\end{equation}
(recall that $\Lbar(\lambda_{\alpha})$ is defined in (\ref{algbar}) and \eqref{eq:lambda-alpha}). Note that $\LLbar$ is also the tensor product of all fundamental representations of the pro\-duct group ${G}^{\gKQ}$. In particular the center $Z_{ G}^{\gKQ}$ acts on $\LLbar$ by the character $\underbrace{\theta_{G}\vert_{Z_{G}}\otimes\cdots\otimes \theta_{G}\vert_{Z_{G}}}_{\gKQ}$, where $\theta_G$ is as in Example \ref{exdelta}, i.e.
\begin{equation}\label{thetag}
\theta_{G}=\sum_{\alpha\in S}\lambda_{\alpha}\in X(T).
\end{equation}

\begin{rem}\label{explicit}
(i) With the notation of \S\ref{Cgroup}, the representation $\LLbar$ is the restriction to $\widehat H$ of the representation $(L^\otimes_{\xi_H},V^\otimes_{\xi_H})$ of ${}^LH$, where $k=\F$, $H=\Res_{K/\Qp}(G)$ and $\xi_H=(\xi_G,\dots,\xi_G)\in X(T_{\widehat{H}})$ ($\xi_G$ as in Example \ref{exdelta}).\\
(ii) Since $\lambda_\alpha\in \oplus_{i=1}^n\Z_{\geq 0}e_i$ by~\eqref{eq:lambda-alpha}, all the weights of $X(T)$ appearing in each $\Lbar(\lambda_{\alpha})\vert_T$ are also in $\oplus_{i=1}^n\Z_{\geq 0}e_i$, and thus the same holds for the weights of $\LLbar\vert_T$ (where $T$ is diagonally embedded into ${G}^{\gKQ}$). This follows from the classical fact that the weights appearing in $\Lbar(\lambda)\vert_T$ for any dominant $\lambda\in X(T)$ are the points in $\oplus_{i=1}^n\Z e_i\cong X(T)$ of the convex hull in $\oplus_{i=1}^n\R e_i$ of the weights $w(\lambda)$, $w\in W$.
\end{rem}

Fix $P$ a standard parabolic subgroup of ${G}$, if $R$ is a finite-dimensional algebraic representation of $P^{\gKQ}$ over $\F$, we write $R\vert_{Z_{M_{P}}}$ for the restriction of $R$ to $Z_{M_{P}}$ acting via the diagonal embedding
\begin{equation}\label{embed}
Z_{M_{P}}\ \hookrightarrow\ Z_{M_{P}}^{\gKQ}\ \subseteq\ {G}^{\gKQ}.
\end{equation}
Since $Z_{M_P}$ is a torus, it follows from \cite[\S I.2.11]{Ja} that $R\vert_{Z_{M_{P}}}$ is the {\it direct sum} of its isotypic components. For instance, if $P=G$ and $R=\LLbar$, there is only one isotypic component as $Z_{M_{G}}=Z_{G}$ acts on $\LLbar$ via the character $f\theta_{G}\vert_{Z_{G}}$.

\begin{lem}\label{action}
Any isotypic component of $R\vert_{Z_{M_{P}}}$ carries an action of $M_{P}^{\gKQ}$ when viewed inside $R\vert_{M_{P}^{\gKQ}}$.
\end{lem}
\begin{proof}
This just comes from the fact that the action of $Z_{M_P}$ commutes with that of $M_P^{\gKQ}$.
\end{proof}

\begin{definit}\label{goodsubqt}
Let $\widetilde P\subseteq P$ be a Zariski closed algebraic subgroup containing $M_P$ and $R$ an algebraic representation of $P^{\gKQ}$ over $\F$, a subquotient (resp.\ subrepresentation, resp.\ quotient) of $R\vert_{{\widetilde P}^{\gKQ}}$ is a {\it good} subquotient (resp.\ subrepresentation, resp.\ quotient) if its restriction to $Z_{M_{P}}$ is a (direct) sum of isotypic components of $R\vert_{Z_{M_{P}}}$.
\end{definit}

\begin{rem}\label{zariski}
A Zariski closed subgroup $\widetilde P$ as in Definition \ref{goodsubqt} actually determines the standard parabolic subgroup $P$ that contains it. Indeed, assume there is another standard parabolic subgroup $P'$ such that $M_{P'}\subseteq \widetilde P\subseteq P'$. Then we have $M_{P'}\subseteq P$ which implies $P'\subseteq P$. Symmetrically, we also have $P\subseteq P'$, hence $P=P'$.
\end{rem}

Since isotypic components of $R\vert_{Z_{M_{P}}}$ tautologically occur with multiplicity $1$, we see in particular that there is only a finite number of good subquotients of $R\vert_{{\widetilde P}^{\gKQ}}$. For instance the entire $\LLbar$ is the only good subquotient of $\LLbar\vert_{{G}^{\gKQ}}$. If $\widetilde{\!\widetilde P}\subseteq \widetilde P$ is another Zariski closed algebraic subgroup as in Definition \ref{goodsubqt}, any good subquotient (resp.\ subrepresentation, resp.\ quotient) of $R\vert_{{\widetilde P}^{\gKQ}}$ is a good subquotient (resp.\ subrepresentation, resp.\ quotient) of $R\vert_{{\widetilde{\!\widetilde P}}^{\gKQ}}$ (but the converse is wrong). 

\begin{lem}\label{filtr}
There exists a filtration on $\LLbar\vert_{{\widetilde P}^{\gKQ}}$ by good subrepresentations such that the graded pieces exhaust the isotypic components of $\LLbar\vert_{Z_{M_{P}}}$ seen as representations of ${\widetilde P}^{\gKQ}$ via the surjection ${\widetilde P}^{\gKQ}\twoheadrightarrow M_{P}^{\gKQ}$ and Lemma \ref{action}.
\end{lem}
\begin{proof}
It is enough to prove the lemma for $\widetilde P=P$. We prove the following statement (which implies the lemma): let $H$ be a split connected reductive algebraic group over $\Z$ with connected center, $T_H\subseteq H$ a split maximal torus in $H$, $B_H\subseteq H$ a Borel subgroup containing $T_H$ with set of (positive) roots $R_H^+$, $V$ a finite-dimensional $H$-module over $\F$, $Q_H\subseteq H$ a parabolic subgroup containing $B_H$ with Levi decomposition $M_{Q_H}N_{Q_H}$ and center $Z_{M_{Q_H}}\subseteq T_H$, $Z'_{M_{Q_H}}$ a subtorus of $Z_{M_{Q_H}}$ and $\lambda'_{Q_H}\in X(Z'_{M_{Q_H}})\defeq \Hom_{\rm Gr}(Z'_{M_{Q_H}},{\mathbb G}_{\rm m})$. Then the $Z'_{M_{Q_H}}$-isotypic component $V_{\lambda'_{Q_H}}$ of $V$ is a quotient of two subrepresentations in $V\vert_{Q_H}$ which are both direct sums of isotypic components of $V\vert_{Z'_{M_{Q_H}}}$ (one applies this result to $H=G^{\gKQ}$, $V=\LLbar$, ${Q_H}=P^{\gKQ}$ and $Z'_{M_{Q_H}}=Z_{M_P}$). Note that as above $V=\oplus_{\lambda'_{Q_H}}V_{\lambda'_{Q_H}}$ and that $V_{\lambda'_{Q_H}}$ carries from $V\vert_{M_{Q_H}}$ an action of $M_{Q_H}$ by the same proof as for Lemma \ref{action}. Let $R({Q_H})^+\subseteq R_H^+$ be the positive roots of $M_{Q_H}$, if $\alpha\in R_H^+\backslash R({Q_H})^+$, denote by $\overline\alpha$ its image via the quotient map $X(T_H)\twoheadrightarrow X(Z'_{M_{Q_H}})$ and $N_\alpha\subseteq N_{Q_H}$ the root subgroup. If $n_\alpha\in N_\alpha$ and $\lambda'_{Q_H}\in X(Z'_{M_{Q_H}})$, then we have $n_\alpha(V_{\lambda'_{Q_H}})\subseteq \sum_{i=0}^{+\infty}V_{\lambda'_{Q_H}+i\overline\alpha}$ by \cite[\S II.1.19]{Ja} (the sum being finite inside $V$). Fix $\lambda'_{Q_H}\in X(Z'_{M_{Q_H}})$ that occurs in $V\vert_{Z'_{M_{Q_H}}}$ and let ${\mathcal W}(\lambda'_{Q_H})$ be the set of $\lambda''_{Q_H}\in X(Z'_{M_{Q_H}})$ of the form $\lambda'_{Q_H}+ \big(\sum_{\alpha\in R_H^+\backslash R({Q_H})^+}\Z_{\geq 0}\overline\alpha\big)$ that occur in $V\vert_{Z'_{M_{Q_H}}}$, we deduce that both subspaces
\[\sum_{\lambda''_{Q_H}\in {\mathcal W}(\lambda'_{Q_H})\backslash \{\lambda'_{Q_H}\}}V_{\lambda''_{Q_H}}\ \ \subsetneq \ \sum_{\lambda''_{Q_H}\in {\mathcal W}(\lambda'_{Q_H})}V_{\lambda''_{Q_H}}\]
are preserved by $N_{Q_H}$, hence by ${Q_H}$, inside $V$. Since their cokernel is exactly $V_{\lambda'_{Q_H}}$, this proves the statement.
\end{proof}

We will use the following lemma extensively.

\begin{lem}\label{Q}
If $Q$ is a {\upshape(}standard{\upshape)} parabolic subgroup of $G$ containing $P$, any isotypic component of $R\vert_{Z_{M_{ Q}}}$ is a good subquotient of $R\vert_{{ P}^{\gKQ}}$ {\upshape(}hence of $R\vert_{{\widetilde P}^{\gKQ}}${\upshape)}.
\end{lem}
\begin{proof}
By Lemma \ref{filtr} (applied in the case $\widetilde P=P$ and with $P$ there being $Q$), such an isotypic component is a good subquotient of $R\vert_{{Q}^{\gKQ}}$, and thus is a subquotient of $R\vert_{{P}^{\gKQ}}$ since $ P\subseteq Q$. It is also obviously a direct sum of isotypic components of $R\vert_{Z_{M_{P}}}$ since $Z_{M_{Q}}\subseteq Z_{M_{P}}$. This proves the lemma.
\end{proof}

\begin{rem}
Let $\widetilde P$, $P$ and $R$ as in Definition \ref{goodsubqt} and define a good subquotient of $R\vert_{{\widetilde P}}$ (for the diagonal embedding $\widetilde P\hookrightarrow\widetilde{P}^{\gKQ}$ similar to (\ref{embed})) as a subquotient of $R\vert_{{\widetilde P}}$ such that its restriction to $Z_{M_{P}}$ is a sum of isotypic components of $R\vert_{Z_{M_{P}}}$. Then, using the same kind of argument as for the proof of Lemma \ref{filtr}, one can prove that a good subquotient of $R\vert_{{\widetilde P}}$ is also a good subquotient of $R\vert_{{\widetilde P}^{\gKQ}}$, so that good subquotients of $R\vert_{{\widetilde P}}$ and of $R\vert_{{\widetilde P}^{\gKQ}}$ are actually the same.
\end{rem}

\subsubsection{The parabolic group associated to an isotypic component}\label{assocparabol}

Fix $P\subseteq G$ a standard parabolic subgroup and $C_P$ an isotypic component of $\LLbar\vert_{Z_{M_{P}}}$, we associate to $C_P$ a subset of the set of simple roots $S$ (see (\ref{PC})), as well as the standard parabolic subgroup of $G$, denoted by $P(C_P)$, corresponding to this subset.

We will use the following two lemmas, the first being well-known.

\begin{lem}\label{known}
Let $\lambda\in X(T)\otimes_{\Z}\Q$ be dominant. Then the Weyl group of the root subsystem of $R$ generated by the simple roots $\alpha\in S$ such that $s_\alpha$ fixes $\lambda$ is the subgroup of $W$ of elements fixing $\lambda$.
\end{lem}

\begin{lem}\label{sumw}
Let $\alpha \in S$. Then $\sum_{w\in W(P)}w(\alpha)\geq 0$, and we have $\sum_{w\in W(P)}w(\alpha)= 0$ if and only if $\alpha\in S(P)$. Moreover, if $\alpha \in S\backslash S(P)$, then $\alpha$ is in the support of $\sum_{w\in W(P)}w(\alpha)$.
\end{lem}
\begin{proof}
If $\alpha\in S(P)$, it is clear that $\sum_{w\in W(P)}w(\alpha)=0$ since, for each $w\in W(P)$, we also have $ws_\alpha\in W(P)$. If $\alpha\in S\backslash S(P)$, then $-\alpha$ is dominant for $M_P$, that is, $-\langle \alpha,\beta\rangle \geq 0$ for $\beta\in S(P)$. This implies that $w(-\alpha)\leq -\alpha$ for $w\in W(P)$. Summing over $W(P)$ gives $-\sum_{w\in W(P)}w(\alpha)\leq -\vert W(P)\vert \alpha$ or equivalently $\vert W(P)\vert \alpha\leq \sum_{w\in W(P)}w(\alpha)$. This proves the lemma.
\end{proof}

If $w\in W$ satisfies $w(S(P))\subseteq S$, we denote by ${}^w P$ the standard parabolic subgroup of $G$ whose associated set of simple roots is $w(S(P))$. It has Levi subgroup $M_{{}^w P}={w}M_{P} {w}^{-1}$ (so ${}^w P=(wM_Pw^{-1})N$) and Weyl group $W({}^w P)={w}W(P){w}^{-1}$ (caution: ${}^w P$ is not $wPw^{-1}$ if $w\ne 1$!). If $\lambda\in X(T)$, we define
\begin{equation}\label{l'}
\lambda'\defeq \frac{1}{\vert W(P)\vert}\sum_{w'\in W(P)}w'(\lambda)\ \in \ (X(T)\otimes_{\Z}\Q)^{W( P)}.
\end{equation}

\begin{rem}\label{lambda'}
(i) Note that $\lambda'$ only depends on $\lambda\vert_{Z_{M_P}}$ since two distinct $\lambda$ with the same restriction to $Z_{M_P}$ differ by an element in $\sum_{\alpha\in S(P)}\Z\alpha$ and since $\sum_{w'\in W(P)}w(\alpha)=0$ for $\alpha\in S(P)$ by Lemma \ref{sumw}.\\
(ii) It easily follows from the definitions and Lemma \ref{sumw} that if $w\in W$ satisfies $w(S(P))\subseteq S$ and $\lambda\in X(T)$ is any weight, then $w(\lambda')=(w(\lambda))'$, where $(w(\lambda))'$ is given by the same formula as in (\ref{l'}) applied to the parabolic ${}^wP$ and the character $w(\lambda)$.
\end{rem}

\begin{lem}\label{parabolic}
Let $P$ be a standard parabolic subgroup of $G$.
\begin{enumerate}
\item Let $\lambda\in X( T)$, there exists $w\in W$ such that $w(S(P))\subseteq S$ and $w(\lambda)\vert_{Z_{M_{{}^w P}}}$ coincides with the restriction to $Z_{M_{{}^w P}}$ of a dominant weight of $X( T)\otimes_{\Z}\Q$.
\item Let $\lambda \in X(T)$ such that $\lambda\vert_{Z_{M_{P}}}$ occurs in $\LLbar\vert_{Z_{M_{P}}}$
  and let $w$ as in (i). Then we have $f\theta_{G}-w(\lambda)=\sum_{\alpha\in S}n_{\alpha}\alpha$ for
  some $n_{\alpha}\in \Z_{\geq 0}$ {\upshape(}see {\upshape(\ref{thetag})} for $\theta_G${\upshape)} and
  the subset
  \begin{eqnarray}\label{PC}
    w(S( P))\cup \{\alpha\in S : n_{\alpha}\ne 0\}\subseteq S
    \end{eqnarray}
    only depends on $\lambda\vert_{Z_{M_{P}}}$.
  \end{enumerate}
\end{lem}
\begin{proof}
(i) We first claim that it is equivalent to find $w$ such that $w(S( P))\subseteq S$ and $w(\lambda')$ is do\-minant with $\lambda'$ as in (\ref{l'}). Assume we have such a $w$, since $w'(\lambda)\vert_{Z_{M_{P}}}=\lambda\vert_{Z_{M_{ P}}}$ for all $w'\in W( P)$, we have $\lambda'\vert_{Z_{M_{P}}}=\lambda\vert_{Z_{M_{ P}}}$ and thus $w(\lambda)\vert_{Z_{M_{{}^w P}}}=w(\lambda')\vert_{Z_{M_{{}^w P}}}$. Conversely, assume that there is $w$ such that $w(S( P))\subseteq S$ and $w(\lambda)\vert_{Z_{M_{{}^w P}}}=\mu\vert_{Z_{M_{{}^w P}}}$\ \ for\ \ some\ \ dominant\ \ $\mu$\ \ in\ \ $X( T)\otimes_{\Z}\Q$,\ \ and\ \ set $\mu'\defeq \frac{1}{\vert W( P)\vert}\sum_{w'\in W({}^w P)}w'(\mu)\in (X( T)\otimes_{\Z}\Q)^{W({}^w P)}$. Then we have $\mu'=w(\lambda')$ by Remark \ref{lambda'}(ii) and $\mu\geq \mu'$ (as $\mu\geq w'(\mu)$ for any $w'\in W$ since $\mu$ is dominant). Thus $\mu-w(\lambda')=\mu-\mu'=\sum_{\alpha\in S({}^w P)}n_{\alpha}\alpha$ for some $n_{\alpha}\in \Q_{\geq 0}$ (recall $\mu\vert_{Z_{M_{{}^wP}}}=\mu'\vert_{Z_{M_{{}^wP}}}$). This implies that
\[\langle w(\lambda'),\beta\rangle = \langle \mu,\beta\rangle - \sum_{\alpha\in S({}^w P)}n_{\alpha} \langle\alpha,\beta\rangle \geq 0\]
for any $\beta\in S\backslash S({}^w P)$ (as $\mu$ is dominant and $\langle\alpha,\beta\rangle\leq 0$ if $\alpha\ne \beta\in S$). Since $\langle w(\lambda'),\beta\rangle=\langle \mu',\beta\rangle=0$ for $\beta\in S({}^w P)$ (use again Lemma \ref{sumw}), we see that $w(\lambda')$ is dominant.\\
Now let us find such a $w$. First, pick $w'\in W$ such that $w'(\lambda')$ is dominant, by Lemma \ref{known} applied to $w'(\lambda')$ the set of elements $\beta$ in $S$ such that $s_{\beta}$ fixes $w'(\lambda')$ generate a root subsystem of $R$ with corresponding Weyl group the subgroup of $W$ of elements that fix $w'(\lambda')$. This root subsystem has two natural bases of simple roots: namely the elements $\beta$ above and the elements $w'(\gamma)\in w'(S)$ such that $s_{\gamma}$ fixes $\lambda'$ (they are usually distinct as $W$ doesn't preserve $S$). This second basis obviously contains $w'(S( P))$. Therefore, there is $w''$ in the Weyl group of this root subsystem, i.e.\ $w''\in W$ fixing $w'(\lambda')$, that maps the second basis to the first. In particular we have $w''w'(S( P))\subseteq S$ and $w''w'(\lambda')=w'(\lambda')$ dominant, thus we can take $w\defeq w''w'$.

\noindent
(ii) The positivity of the $n_{\alpha}$ follows from the fact $f\theta_{ G}$ is the highest weight of $\LLbar\vert_T$ (for the diagonal embedding of $T$ as in (\ref{embed})). Let $w_1,w_2$ as in (i) and $\lambda'$ as in (\ref{l'}). Then $w_1(\lambda')=w_2(\lambda')$ as these two weights are dominant (by the first part of the proof of (i)) and in a single $W$-orbit. Since $\lambda'$ only depends on $\lambda\vert_{Z_{M_P}}$ by Remark \ref{lambda'}(i), it is therefore enough to prove that the support of $f\theta_{G}-w(\lambda')$ is exactly the set of simple roots (\ref{PC}) for one (any) $w$ as in (i). Writing $f\theta_{ G}-w(\lambda')=(f\theta_{ G}-w(\lambda)) + (w(\lambda)-w(\lambda'))$ and recalling that $w(\lambda)-w(\lambda')$ is a sum of roots in $w(S( P))\subseteq S$ (as $w(\lambda),w(\lambda')$ have same restriction to $Z_{M_{{}^{w} P}}$ from the proof of (i)), we see that this support is contained in (\ref{PC}) and that it contains $\{\alpha\in S\backslash w(S(P)) : n_\alpha\ne 0\}$. It is thus enough to prove that this support also contains $w(S( P))$. Since $f\theta_{ G}\geq w(\lambda')$ (use $f\theta_{G}\geq ww'(\lambda)$ for any $w'\in W$ and sum over $w'\in W(P)$) and $\langle\beta,\alpha\rangle\leq 0$ if $\alpha\ne \beta\in S$, it is enough to check that $\langle f\theta_{G}-w(\lambda'),\alpha\rangle>0$ (in $\Q$) for any $\alpha\in w(S( P))$. But this follows from Lemma \ref{sumw} and $\langle f\theta_{ G}-w(\lambda'),\alpha\rangle=f\langle \theta_{ G},\alpha\rangle-\langle w(\lambda'),\alpha\rangle=f-0=f$.
\end{proof}

\begin{rem}
Note that it is not true in general that, for $\lambda$ as in Lemma \ref{parabolic}(ii), one can find $w\in W$ such that $w(S(P))\subseteq S$ and $w(\lambda)\vert_{Z_{M_{{}^w P}}}$ is the restriction to $Z_{M_{{}^w P}}$ of a dominant weight of $X(T)$ (one really needs $X(T)\otimes_{\Z}\Q$).
\end{rem}

The proof of Lemma \ref{parabolic} also gives the following equivalent proposition that we will use repeatedly in the sequel.

\begin{prop}\label{parabolicprop}
Let $P$ be a standard parabolic subgroup of $G$.
\begin{enumerate}
\item Let $\lambda\in X( T)$ and $\lambda'$ as in {\upshape(\ref{l'})}, there exists $w\in W$ such that $w(S(P))\subseteq S$ and $w(\lambda')$ is a dominant weight of $X( T)\otimes_{\Z}\Q$.
\item Let $\lambda \in X(T)$ such that $\lambda\vert_{Z_{M_{P}}}$ occurs in $\LLbar\vert_{Z_{M_{P}}}$
  and let $w$ as in (i). Then we have $f\theta_{G}-w(\lambda')=\sum_{\alpha\in S}n_{\alpha}\alpha$ for
  some $n_{\alpha}\in \Q_{\geq 0}$ and the support of $f\theta_G-w(\lambda')$ is $S(P(C_P))$.
\end{enumerate}
\end{prop}

Let $C_P$ be an isotypic component of $\LLbar\vert_{Z_{M_{P}}}$ associated to some $\lambda_P\in X(Z_{M_P})=\Hom_{\rm Gr}(Z_{M_P},{\mathbb G}_{\rm m})$. We denote by $P(C_P)$ the unique standard parabolic subgroup of $G$ whose associated set of simple roots $S(P(C_P))$ is (\ref{PC}) for one (equivalently any) weight $\lambda\in X(T)$ such that $\lambda\vert_{Z_{M_P}}=\lambda_P$. We also define
\begin{equation}\label{wCP}
W(C_P)\defeq \{w\in W{\rm \ as\ in\ Proposition\ }\ref{parabolicprop}{\rm (i)\ for\ }\lambda\in X(T) : \lambda\vert_{Z_{M_P}}=\lambda_P\}
\end{equation}
($W(C_P)$ doesn't depend on the choice of such $\lambda$ by the claim in the proof of Lemma \ref{parabolic}(i) and by Remark \ref{lambda'}(i)). We see from (\ref{PC}) that for all $w\in W(C_P)$ we have the inclusion
\begin{equation}\label{wPinP(CP)}
{}^w P\subseteq P(C_P).
\end{equation}
{\it Note that the set $W(C_P)$ is not in general a group, in particular it is distinct in general from the Weyl group $W( P(C_P))$} (see Lemma \ref{inclusion} below for the relation between the two).

\begin{rem}
The inclusion ${}^w P\subseteq P(C_P)$ for some $w\in W$ (such that $w(S(P))\subseteq S$) {\it doesn't} imply $w\in W(C_P)$ (take $P=B$). Also $P(C_P)$ doesn't necessarily contain $P$, see e.g.\ the end of Example \ref{exemples}(ii) below. The subgroup generated by all ${}^w P$ for $w\in W(C_P)$ may also be {\it strictly} contained in $P(C_P)$ (see e.g.\ Example \ref{exemples}(i) below).
\end{rem}

The parabolic subgroups $P(C_P)$ respect inclusions.

\begin{lem}\label{inclusionP}
Let $P'\subseteq P$ be two standard parabolic subgroups of $G$, $C_P$ an isotypic component of $\LLbar\vert_{Z_{M_{P}}}$ and $C_{P'}$ an isotypic component of $\LLbar\vert_{Z_{M_{P'}}}$ such that $C_{P'}\subseteq C_P\vert_{Z_{M_{P'}}}$. Then $P(C_{P'})\subseteq P(C_P)$.
\end{lem}
\begin{proof}
Let $\lambda \in X(T)$ such that $C_{P'}$ is the isotypic component of $\lambda\vert_{Z_{M_{P'}}}$. Then by assumption $C_P$ is the isotypic component of $\lambda\vert_{Z_{M_{P}}}$. Define $\lambda'_P\in (X(T)\otimes_{\Z}\Q)^{W(P)}$, $\lambda'_{P'}\in (X(T)\otimes_{\Z}\Q)^{W(P')}$ as in (\ref{l'}) for respectively $P$ and $P'$ and let $(w_P,w_{P'})\in W\times W$ such that $w_P(S(P))\subseteq S$ and $w_P(\lambda'_P)$ dominant, $w_{P'}(S(P'))\subseteq S$ and $w_{P'}(\lambda'_{P'})$ dominant ($w_P,w_{P'}$ exist by Proposition \ref{parabolicprop}(i)). Then we have
\begin{equation*}
w_P(\lambda'_P)=\frac{1}{\vert W(P)\vert}\sum_{w'\in W({}^{w_P}P)}\!\!w'w_P(\lambda),\ \ w_P(\lambda'_{P'})=\frac{1}{\vert W(P')\vert}\sum_{w'\in W({}^{w_P}P')}\!\!w'w_P(\lambda)
\end{equation*}
and also
\begin{equation}\label{subsum}
w_P(\lambda'_P)=\frac{\vert W(P')\vert}{\vert W(P)\vert}\sum_{\sigma \in W({}^{w_P}P)/W({}^{w_P}P')}\sigma w_P(\lambda'_{P'}).
\end{equation}
Since $w_{P'}(\lambda'_{P'})$ is dominant, we have $w_{P'}(\lambda'_{P'})\geq w w_{P'}(\lambda'_{P'})$ for any $w\in W$ and in particular $w_{P'}(\lambda'_{P'})\geq\sigma w_P(\lambda'_{P'})=(\sigma w_Pw_{P'}^{-1})w_{P'}(\lambda'_{P'})$. Summing up these inequalities over $\sigma\in W({}^{w_P}P)/W({}^{w_P}P')$ and multiplying by $\frac{\vert W(P')\vert}{\vert W(P)\vert}$, one gets with (\ref{subsum}):
\begin{equation}\label{ineg}
w_{P'}(\lambda'_{P'})\geq w_P(\lambda'_P).
\end{equation}
Now the result follows from
\[f\theta_G-w_P(\lambda'_P)=\big(f\theta_G-w_{P'}(\lambda'_{P'})\big)+ \big(w_{P'}(\lambda'_{P'})-w_P(\lambda'_P)\big)\]
together with Proposition \ref{parabolicprop}(ii) and (\ref{ineg}).
\end{proof}

\begin{ex}\label{exemples}
We give a few simple examples (beyond $\GL_2(\Qp)$).

\noindent
(i) Assume $n=2$ and $P=B$. Then $\LLbar\vert_{Z_{M_{B}}}=\LLbar\vert_{T}$ has $f+1$ isotypic components $C(\lambda_i)$ given by the characters $\lambda_i:{\rm diag}(x_1,x_2)\mapsto x_1^{f-i}x_2^{i}$ for $0\leq i\leq f$. For $i< f/2$, $\lambda_i$ is dominant, $W(C(\lambda_i))=\{1\}$ and $f\theta_G-\lambda_i=i(e_1-e_2)$. For $i=f/2$ (if $f$ is even), $\lambda_i=s_{e_1-e_2}(\lambda_i)$ is dominant, $W(C(\lambda_i))=\{1,s_{e_1-e_2}\}$ and $f\theta_G-w(\lambda_i)=f/2(e_1-e_2)$ for $w\in W(C(\lambda_i))$. For $i>f/2$, $s_{e_1-e_2}(\lambda_i)$ is dominant, $W(C(\lambda_i))=\{s_{e_1-e_2}\}$ and $f\theta_G-s_{e_1-e_2}(\lambda_i)=(f-i)(e_1-e_2)$. We see that ${}^w B=B\subsetneq P(C(\lambda_i))=G$ if $i\notin \{0,f\}$ and ${}^w B=P(C(\lambda_i))=B$ if $i\in \{0,f\}$.

\noindent
(ii) Assume $n=3$ and $K=\Qp$.

\noindent
If $P=B$, then $\LLbar\vert_{T}$ has $7$ isotypic components given by the $6$ characters $\lambda_w:{\rm diag}(x_1,x_2,x_3)\mapsto x_{w^{-1}(1)}^2x_{w^{-1}(2)}$ for $w\in {\mathcal S}_3$ and the character ${\rm det}:{\rm diag}(x_1,x_2,x_3)\mapsto x_1x_2x_3$. If $C_P$ corresponds to some $\lambda_w$, one gets that $W(C_P)$ is the singleton $\{w\}$ and $\theta_G-w(\lambda_w)=0$, which implies ${}^w \!B=P(C_P)=B$. If $C_P$ corresponds to $\rm det$, one gets $W(C_P)=W$ and $\theta_G-w({\rm det})=(e_1-e_2)+(e_2-e_3)$ for $w\in W$, which implies ${}^w B=B\subsetneq P(C_P)=G$.

\noindent
If $P$ is the standard parabolic subgroup of Levi ${\rm diag}(\GL_2,\GL_1)$, then $\LLbar\vert_{Z_{M_{P}}}$ has $3$ isotypic components $C_P$ given by the characters
\[\lambda_0:{\rm diag}(x_1,x_1,x_2)\mapsto x_1^3,\ \lambda_1:{\rm diag}(x_1,x_1,x_2)\mapsto x_1^2x_2,\ \lambda_2:{\rm diag}(x_1,x_1,x_2)\mapsto x_1x_2^2.\]
One has $\lambda'_0=3/2(e_1+e_2)$, $\lambda'_1=e_1+e_2+e_3$, $\lambda'_2=1/2(e_1+e_2)+2e_3$ from which one deduces for the three respective isotypic components $C_P$ (where $w\in W(C_P)$):
\[\begin{array}{llllll}
W(C_P)&=&\{1\}\ \ &\theta_G-w(\lambda'_0)&=&1/2(e_1-e_2)\\
W(C_P)&=&\{1,s_{e_1-e_2}s_{e_2-e_3}\}\ \ &\theta_G-w(\lambda'_1)&=&(e_1-e_2)+(e_2-e_3)\\
W(C_P)&=&\{s_{e_1-e_2}s_{e_2-e_3}\}\ \ &\theta_G-w(\lambda'_2)&=&1/2(e_2-e_3).
\end{array}\]
If $C_P$ corresponds to $\lambda_0$ one gets ${}^w P=P(C_P)=P$, if $C_P$ corresponds to $\lambda_1$ one gets ${}^wP\subsetneq P(C_P)=G$ (${}^wP$ being $P$ if $w=\Id$ and the standard parabolic subgroup of Levi ${\rm diag}(\GL_1,\GL_2)$ if $w=s_{e_1-e_2}s_{e_2-e_3}$), and if $C_P$ corresponds to $\lambda_2$ one gets ${}^w P=P(C_P)=$ the standard parabolic subgroup of Levi ${\rm diag}(\GL_1,\GL_2)$. In this last case we see that $P(C_P)$ doesn't contain $P$.

\noindent
Finally, if $M_P={\rm diag}(\GL_1,\GL_2)$, the situation is symmetric.
\end{ex}

\begin{lem}\label{inclusion}
We have $W(C_P)\subseteq W(P(C_P))w$ for any fixed element $w\in W(C_P)$. Equivalently $w' w^{-1}\in W(P(C_P))$ for any $w,w'\in W(C_P)$.
\end{lem}
\begin{proof}
Let $\lambda_P\in X(Z_{M_P})$ corresponding to $C_P$, $w_{C_P}\in W(C_P)$, $\lambda\in X( T)$ such that $\lambda\vert_{Z_{M_{P}}}=\lambda_P$ and define $\lambda'$ as in (\ref{l'}). Recall that an element $w\in W$ is in $W(C_P)$ if and only if $w(S(P))\subseteq S$ and $w(\lambda')$ is dominant (see Proposition \ref{parabolicprop}(i)), and that we have $w(\lambda')=w_{C_P}(\lambda')$ for all $w\in W(C_P)$ (see the beginning of the proof of Lemma \ref{parabolic}(ii)). We rewrite this $ww_{C_P}^{-1}(w_{C_P}(\lambda'))=w_{C_P}(\lambda')$ $\forall \ w\in W(C_P)$. By the definition of $P(C_P)$ and Proposition \ref{parabolicprop}(ii), we know that $S(P(C_P))$ is the set of simple roots in the support of $f\theta_{ G}-w_{C_P}(\lambda')$. Since $w_{C_P}(\lambda')$ is dominant, by Lemma \ref{known} the subgroup of $W$ fixing $w_{C_P}(\lambda')$ is generated by the simple reflections $s_{\beta}$ fixing $w_{C_P}(\lambda')$, or equivalently such that $\langle w_{C_P}(\lambda'),\beta\rangle=0$. Since $\langle f\theta_{ G}-w_{C_P}(\lambda'),\beta\rangle=f-0=f$, we see that these simple roots $\beta$ are all in the support of $f\theta_{ G}-w_{C_P}(\lambda')$. Therefore $W(P(C_P))$ contains the subgroup of $W$ fixing $w_{C_P}(\lambda')$. Since $ww_{C_P}^{-1}$ fixes $w_{C_P}(\lambda')$, it follows that ${w}w_{C_P}^{-1}\in W(P(C_P))$.
\end{proof}

\begin{rem}
The inclusion in Lemma \ref{inclusion} is not an equality in general (take $P= G$).
\end{rem}

\subsubsection{The structure of isotypic components of \texorpdfstring{$\LLbar$}{L\^{}\textbackslash otimes}}\label{pcppetit}

We let $P$ be a standard parabolic subgroup of $G$, we prove an important structure theorem on the isotypic components of $\LLbar\vert_{Z_{M_P}}$ (Theorem \ref{decompstrong}) as well as several useful technical results.

Recall that $W(C_P)$ is defined in (\ref{wCP}) and $P(C_P)$ is defined just before.

\begin{lem}\label{single}
If $P(C_P)={}^w P$ for some $w\in W(C_P)$ then $W(C_P)$ has just one element.
\end{lem}
\begin{proof}
Let \ $w_{C_P}\in W(C_P)$ \ such \ that \ $P(C_P)={}^{w_{C_P}} P$ \ and \ let \ $w'_{C_P}\in W(C_P)$. Since $P(C_P)={}^{w_{C_P}} P$ we get $S(P(C_P))=w_{C_P}(S(P))$ and $W(P(C_P))=w_{C_P}W( P)w_{C_P}^{-1}$. By Lemma \ref{inclusion} applied to the element $w_{C_P}$, we deduce $W(C_P)\subseteq w_{C_P}W( P)$ and thus $w_{C_P}^{-1}w'_{C_P}\in W(P)$. But since $S(P(C_P))$ contains $w(S(P))$ for all $w\in W(C_P)$ by definition of $W(C_P)$ and (\ref{PC}), we have $w'_{C_P}(S(P))\subseteq S(P(C_P))=w_{C_P}(S(P))$ which implies $w'_{C_P}(S(P))=w_{C_P}(S(P))$ since the cardinalities are the same on both sides, that is, $w_{C_P}^{-1}w'_{C_P}(S(P))=S(P)$. Since $w_{C_P}^{-1}w'_{C_P}\in W(P)$, this forces $w'_{C_P}=w_{C_P}$.
\end{proof}

\begin{rem}\label{ctrex}
(i) The converse to Lemma \ref{single} is wrong in general (e.g.\ consider the $C(\lambda_i)$ with $i\notin\{0,f/2,f\}$ in Example \ref{exemples}(i)).\\
(ii) For a general isotypic component $C_P$, it is not true that one can find $w\in W(C_P)$ such that $w^{-1}M_{P(C_P)}w$ is the Levi subgroup of a standard parabolic subgroup of $G$.
\end{rem}

\begin{prop}\label{minimal}
The isotypic components $C_P$ such that $P(C_P)={}^w P$ for some {\upshape(}necessarily unique{\upshape)} $w\in W(C_P)$ are those isotypic components which are associated to $fw^{-1}(\theta_G)\vert_{Z_{M_P}}$ for the $w\in W$ such that $w(S(P))\subseteq S$.
\end{prop}
\begin{proof}
Let $w\in W$ such that $w(S(P))\subseteq S$ and $\lambda\defeq fw^{-1}(\theta_G)\in X(T)$. Since $w(\lambda)=f\theta_G$ is dominant and $f\theta_G-w(\lambda)=0$, the set (\ref{PC}) is $w(S(P))$. This implies $P(C_P)={}^wP$.\\
Conversely, let $C_P$ as in the statement, $\lambda \in X(T)$ such that $C_P$ is the isotypic component associated to the character $\lambda\vert_{Z_{M_P}}$ of $Z_{M_P}$ and define $\lambda'$ as in (\ref{l'}). Since $S(P(C_P))=w(S(P))$ by assumption, from Proposition \ref{parabolicprop}(ii) we obtain
\[ fw^{-1}(\theta_G)-\lambda'=\sum_{\alpha\in S(P)}n_\alpha\alpha\]
(for some $n_\alpha\in \Q_{>0}$), which implies $fw^{-1}(\theta_G)\vert_{Z_{M_P}}=\lambda'\vert_{Z_{M_P}}$. Since $\lambda\vert_{Z_{M_P}}=\lambda'\vert_{Z_{M_P}}$ (see the beginning of the proof of Lemma \ref{parabolic}(i)), we deduce that $C_P$ is the isotypic component associated to the character $fw^{-1}(\theta_G)\vert_{Z_{M_P}}$.
\end{proof}

Note that if $C_P$ is associated to $fw^{-1}(\theta_G)\vert_{Z_{M_P}}$ (with $w(S(P))\subseteq S$), we have $W(C_P)=\{w\}$ by Lemma \ref{single}.

\begin{ex}
Coming back to Example \ref{exemples}, the isotypic components as in Proposition \ref{minimal} are the isotypic components $C(\lambda_0),\ C(\lambda_f)$ when $n=2$, $P=B$, the isotypic components associated to the six $\lambda_w$ when $n=3$, $K=\Qp$, $P=B$, and the isotypic components associated to $\lambda_0$, $\lambda_2$ when $n=3$, $K=\Qp$, $M_P=\GL_2\times \GL_1$.
\end{ex}

We set for $\alpha =e_j-e_{j+1}\in S(P)$:
\begin{equation}\label{alphaP}
\lambda_{\alpha,P}\defeq \!\!\!\sum_{e_i-e_{j+1}\in R(P)^+}\!\!\!e_i\in X(T).
\end{equation}
One easily checks that the $\lambda_{\alpha,P}$ for $\alpha\in S(P)$ are fundamental weights for the reductive group $M_P$ and that $\langle \lambda_{\alpha,P},\beta\rangle\leq 0$ for $\beta\in S\backslash S(P)$. For any $\lambda\in X(T)$, we define ${\overline L}_{P}(\lambda)$ as in (\ref{algbar}) but with $(M_P, M_P\cap B^-)$ instead of $(G,B^-)$. When $S(P)=\emptyset$, we define ${\overline L}^{\otimes}_P$ to be the trivial representation of $T^{\gKQ}$ over $\F$ and, when $S(P)\ne \emptyset$, we define similarly to (\ref{lbar}) the algebraic representation of $M_P^{\gKQ}$ over $\F$:
\begin{equation}\label{lbarP}
{\overline L}^{\otimes}_P\defeq \bigotimes_{\gKQ}\Big(\bigotimes_{\alpha\in S(P)}\Lbar_P(\lambda_{\alpha,P})\Big).
\end{equation}
We also define
\begin{equation}\label{thetaP}
\theta_{P}\defeq \!\!\sum_{\alpha\in S(P)}\!\!\lambda_{\alpha,P}\in X(T)\ \ \ {\rm and}\ \ \ \theta^{P}\defeq \theta_{G}-\theta_{P}\in X(T).
\end{equation}
Since for $\alpha\in S(P)$ we have $\langle\theta^P,\alpha\rangle = \langle\theta_G,\alpha\rangle-\langle\theta_P,\alpha\rangle=1-1=0$, we see that $\theta^{P}$ extends to an element of $\Hom_{\rm Gr}(M_{P},{\mathbb G}_{\rm m})$. Likewise we have for $\alpha\in S(P)$ and $w\in W$ such that $w(S(P))\subseteq S$:
\[\langle w^{-1}(\theta^{{}^{w}P}),\alpha\rangle = \langle \theta^{{}^{w}P},w(\alpha)\rangle=0\]
so that $w^{-1}(\theta^{{}^{w}P})$ also extends to $\Hom_{\rm Gr}(M_{P},{\mathbb G}_{\rm m})$. Note that, since $\langle\theta_P,\beta\rangle \leq 0$ for $\beta\in S\backslash S(P)$, we get $\langle\theta^P,\beta\rangle = \langle\theta_G,\beta\rangle-\langle\theta_P,\beta\rangle\geq 1$, thus $\theta^P$ is a dominant weight.

\begin{ex}
If $G=\GL_6$ and $M_P=\GL_2\times \GL_3\times \GL_1$, one gets
\[\begin{array}{lll}
\theta_P:{\rm diag}(x_1,\ldots,x_6)&\longmapsto &(x_1)(x_3^2x_4)\\
\theta^P:{\rm diag}(x_1,\ldots,x_6)&\longmapsto&(x_1x_2)^4(x_3x_4x_5).
\end{array}\]
\end{ex}

\begin{lem}\label{utile}
If $w\in W(P)$, we have $w(\theta^P)=\theta^P$.
\end{lem}
\begin{proof}
The character $\theta^P$ extends to $M_P$ and factors through $M_P/M_P^{\rm der}$. But conjugation by $W(P)$ is trivial on $M_P/M_P^{\rm der}$.
\end{proof}

\begin{lem}\label{isohigh}
Let $\lambda\in X(T)$ be a dominant weight and denote by ${\overline L}(\lambda)_\mu\subseteq {\overline L}(\lambda)$ for $\mu\in X(T)$ the isotypic component of ${\overline L}(\lambda)\vert_{T}$ associated to $\mu$ {\upshape(}i.e.\ the weight space of ${\overline L}(\lambda)$ for $\mu$, see {\upshape\cite[\S I.2.11]{Ja})}. Then
\[\bigoplus_{\mu\in \lambda-\sum_{\alpha\in S(P)}\Z_{\geq 0}\alpha}\!\!\!\!\!\!\!{\overline L}(\lambda)_\mu\ \ \subseteq \ \ {\overline L}(\lambda)\]
is an $M_P$-subrepresentation of ${\overline L}(\lambda)\vert_{M_P}$ which is isomorphic to ${\overline L}_P(\lambda)$.
\end{lem}
\begin{proof}
Since $\oplus_{\mu\in \lambda-\sum_{\alpha\in S(P)}\Z_{\geq 0}\alpha}{\overline L}(\lambda)_\mu$ is the isotypic component of ${\overline L}(\lambda)\vert_{Z_{M_P}}$ associated to $\lambda\vert_{Z_{M_P}}$ (as $\lambda\vert_{Z_{M_P}}\cong \mu\vert_{Z_{M_P}}\Longleftrightarrow \lambda-\mu\in \sum_{\alpha\in S(P)}\Z\alpha$), it is endowed with an action of $M_P$ by the same proof as for Lemma \ref{action}. By \cite[II.2.2(1)]{Ja}, \cite[I.6.11(2)]{Ja} and the transitivity of induction (\cite[I.3.5(2)]{Ja}), we have an injection of algebraic representations of $M_P$ over $\F$:
\begin{equation}\label{injecP}
H^0(N_P,{\overline L}(\lambda))\hookrightarrow {\overline L}_P(\lambda)
\end{equation}
(recall $N_P$ is the unipotent radical of $P$) and by \cite[II.2.11(1)]{Ja} we have an isomorphism of algebraic representations of $M_P$ over $\F$:
\begin{equation*}
\bigoplus_{\mu\in \lambda-\sum_{\alpha\in S(P)}\Z_{\geq 0}\alpha}\!\!\!\!\!\!\!{\overline L}(\lambda)_\mu\ \ \buildrel\sim\over\longrightarrow\ \ H^0(N_P,{\overline L}(\lambda)).
\end{equation*}
It is therefore enough to prove that (\ref{injecP}) is an isomorphism, or equivalently that
\begin{equation*}
\dim_{\F}\bigg(\bigoplus_{\mu\in \lambda-\sum_{\alpha\in S(P)}\Z_{\geq 0}\alpha}\!\!\!\!\!\!\!{\overline L}(\lambda)_\mu\bigg)=\dim_{\F}{\overline L}_P(\lambda).
\end{equation*}
Let $L(\lambda)\defeq \big({\rm ind}_{B^-}^G\lambda\big)_{/\Z}\otimes_{\Z}E$, $L_P(\lambda)\defeq \big({\rm ind}_{M_P\cap B^-}^{M_P}\lambda\big)_{/\Z}\otimes_{\Z}E$ and $L(\lambda)_{\mu}\subseteq L(\lambda)$ the weight space associated to $\mu$, we have $\dim_{\F}{\overline L}(\lambda)_\mu=\dim_E{L}(\lambda)_\mu$, and thus
\[\dim_{\F}\bigg(\bigoplus_{\mu\in \lambda-\sum_{\alpha\in S(P)}\Z_{\geq 0}\alpha}\!\!\!\!\!\!\!{\overline L}(\lambda)_\mu\bigg)=\dim_{E}\bigg(\bigoplus_{\mu\in \lambda-\sum_{\alpha\in S(P)}\Z_{\geq 0}\alpha}\!\!\!\!\!\!\!{L}(\lambda)_\mu\bigg).\]
Likewise, we have $\dim_{\F}{\overline L}_P(\lambda)=\dim_{E}{L}_P(\lambda)$. It is therefore enough to have
\[\dim_{E}\bigg(\bigoplus_{\mu\in \lambda-\sum_{\alpha\in S(P)}\Z_{\geq 0}\alpha}\!\!\!\!\!\!\!{L}(\lambda)_\mu\bigg)=\dim_{E}{L}_P(\lambda).\]
But now, we are over a field of characteristic $0$, where it is well known that $L(\lambda)$ and $L_P(\lambda)$ as defined above are {\it simple modules with highest weight $\lambda$}. Then the result follows from \cite[Prop.II.2.11]{Ja}.
\end{proof}

The following lemma is a special case of Lemma \ref{isohigh}.

\begin{lem}\label{supps(p)}
Let $\lambda\in X(T)$ be a dominant weight such that $\langle \lambda,\alpha\rangle =0$ for all $\alpha\in S(P)$ {\upshape(}equivalently $\lambda$ extends to an element in $\Hom_{\rm Gr}(M_{P},{\mathbb G}_{\rm m})${\upshape)}. Then any $\mu\in X(T)$ distinct from $\lambda$ with ${\overline L}(\lambda)_\mu\ne 0$ is such that $\lambda-\mu$ contains at least one root of $S\backslash S(P)$ in its support.
\end{lem}
\begin{proof}
Since $\lambda\in \Hom_{\rm Gr}(M_{P},{\mathbb G}_{\rm m})$, we have ${\overline L}_P(\lambda)\cong \lambda$ by (\ref{algbar}) applied with $M_P$ instead of $G$. By Lemma \ref{isohigh}, we deduce $\bigoplus_{\mu\in \lambda-\sum_{\alpha\in S(P)}\Z_{\geq 0}\alpha}{\overline L}(\lambda)_\mu \cong \lambda$ inside ${\overline L}(\lambda)$. This clearly implies the lemma.
\end{proof}

If $R$ is any algebraic representation of $M_P$ or of $M_P^{\gKQ}$ and $w\in W$ such that $w(S(P))\subseteq S$, we define an algebraic representation of $M_{{}^w P}=wM_{P}w^{-1}$ or of $M_{{}^w P}^{\gKQ}=wM_{P}^{\gKQ}w^{-1}$ ($w$ acting diagonally via $W\hookrightarrow W^{\gKQ}$) by ($g\in M_{{}^w P}$ or $M_{{}^w P}^{\gKQ}$):
\begin{equation}\label{wact}
w(R)(g)\defeq R(w^{-1}gw).
\end{equation}
If $\alpha\in S(P)$, one then easily checks that $w(\lambda_{\alpha,P})=\lambda_{w(\alpha),{}^wP}$ and $w(\Lbar_P(\lambda_{\alpha,P}))= \Lbar_{{}^wP}(\lambda_{w(\alpha),{}^wP})$, from which one gets
\begin{equation}\label{wlp}
w({\overline L}^{\otimes}_P)={\overline L}^{\otimes}_{{}^wP}.
\end{equation}

\begin{thm}\label{decompstrong}
Let $C_P$ be an isotypic component of $\LLbar\vert_{Z_{M_P}}$, associated to $\lambda\vert_{Z_{M_P}}$ for $\lambda\in X(T)$. For any $w\in W(C_P)$, there is an isomorphism of algebraic representations of $M_{P}^{\gKQ}$ over $\F$:
\begin{equation}\label{decompfort}
C_P\cong w^{-1}\big(C_{P(C_P),{}^wP}\big)\otimes \big(\underbrace{w^{-1}(\theta^{P(C_P)})\otimes \cdots \otimes w^{-1}(\theta^{P(C_P)})}_{\gKQ}\big),
\end{equation}
where \ $C_{P(C_P),{}^wP}$ \ is \ the \ isotypic \ component \ of \ ${\overline L}^{\otimes}_{P(C_P)}\vert_{Z_{M_{{}^wP}}}$ \ associated \ to $(w(\lambda)-f\theta^{P(C_P)})\vert_{Z_{M_{{}^wP}}}$ {\upshape(}thus an $M_{{}^wP}^{\gKQ}$-representation, recall ${}^wP\subseteq P(C_P)${\upshape)} and $w^{-1}(C_{P(C_P),{}^wP})$ is defined in {\upshape(\ref{wact})}. 
\end{thm}
\begin{proof}
Step 1: Assuming the result holds if $w=\Id$, we prove it for any $w$. For $\mu\in X(T)$ we have $\mu\vert_{Z_{M_{P}}}=\lambda\vert_{Z_{M_{P}}}$ if and only if $w(\mu)\vert_{Z_{M_{{}^w P}}}=w(\lambda)\vert_{Z_{M_{{}^w P}}}$, therefore the image $w(C_P)$ of $C_P$ for the diagonal action of $w\in W$ on $\LLbar$ is the isotypic component of $\LLbar\vert_{Z_{M_{{}^w P}}}$ associated to $w(\lambda)\vert_{Z_{M_{{}^w P}}}$. Note that, as an algebraic $M_{{}^w P}^{\gKQ}$-subrepresentation of $\LLbar\vert_{M_{{}^w P}^{\gKQ}}$, $w(C_P)$ is indeed isomorphic to $g\mapsto C_P(w^{-1}gw)$ if $g\in M_{{}^w P}^{\gKQ}$, so the notation is consistent with (\ref{wact}). By Remark \ref{lambda'}(ii) we have $w(\lambda')=(w(\lambda))'$ in $(X(T)\otimes_{\Z}\Q)^{W({}^wP)}$. Recall that $w(\lambda')$, and hence $(w(\lambda))'$, are dominant since $w\in W(C_P)$ (see Proposition \ref{parabolicprop}(i)). Therefore $\Id \in W(w(C_P))$ and by the case $w=\Id$ for the parabolic subgroup ${}^w P$ and the isotypic component $w(C_P)$, we have $w(C_P)\cong C_{P(w(C_P)),{}^wP}\otimes \big({\theta^{P(w(C_P))}\otimes \cdots \otimes \theta^{P(w(C_P))}}\big)$. Moreover $S(P(w(C_P)))$, which is the support of $f\theta_G-(w(\lambda))'$ by Proposition \ref{parabolicprop}(ii) (applied to $w=\Id$), is the same as $S(P_{C_P})$, which is the support of $f\theta_G-w(\lambda')$ by {\it loc.cit.}\ (applied to $w$), i.e.\ we have $P(w(C_P))=P(C_P)$. We thus deduce $w(C_P)\cong C_{P(C_P),{}^wP}\otimes \big({\theta^{P(C_P)}\otimes \cdots \otimes \theta^{P(C_P)}}\big)$ which gives (\ref{decompfort}) by applying $w^{-1}$.

\noindent
Step 2: From now on we assume $w=\Id$ (so in particular $P\subseteq P(C_P)$). Writing
\[\LLbar=\bigg(\bigotimes_{\gKQ}\Big(\bigotimes_{\alpha\in S(P(C_P))}\Lbar(\lambda_{\alpha})\Big)\bigg)\bigotimes \bigg(\bigotimes_{\gKQ}\Big(\bigotimes_{\alpha\in S\backslash S(P(C_P))}\Lbar(\lambda_{\alpha})\Big)\bigg),\]
we prove that any $(\mu_1,\mu_2)\in X(T)\times X(T)$ such that
\begin{enumerate}
\item $\mu_{1}$ occurs in $\big(\bigotimes_{\gKQ}\big(\bigotimes_{\alpha\in S(P(C_P))}\Lbar(\lambda_{\alpha})\big)\big)\vert_T$ (for the diagonal action of $T$);
\item $\mu_2$ occurs in $\big(\bigotimes_{\gKQ}\big(\bigotimes_{\alpha\in S\backslash S(P(C_P))}\Lbar(\lambda_{\alpha})\big)\big)\vert_T$ ({idem});
\item $\mu_{1}\vert_{Z_{M_P}}+\mu_2\vert_{Z_{M_P}}=\lambda\vert_{Z_{M_P}}$
\end{enumerate}
must be such that $\mu_2=f\sum_{\alpha\in S\backslash S(P(C_P))}\lambda_\alpha$ (note that $\mu_2\leq f\sum_{\alpha\in S\backslash S(P(C_P))}\lambda_\alpha$ and $\mu_1\leq f\sum_{\alpha\in S(P(C_P))}\lambda_\alpha$). Let $\lambda'$, $\mu_1'$, $\mu_2'$ as in (\ref{l'}) for $P(C_P)$ and the respective characters $\lambda$, $\mu_1$, $\mu_2$, we have $\lambda'=\mu_1'+\mu_2'$ from (iii) and Remark \ref{lambda'}(i), and thus
\begin{equation}\label{summation}
f\theta_G-\lambda'=f\Big(\!\sum_{\alpha \in S(P(C_P))}\!\!\!\!\!\lambda_\alpha\Big) - \mu_1' + f\Big(\!\sum_{\alpha \in S\backslash S(P(C_P))}\!\!\!\!\!\lambda_\alpha\Big) - \mu_2'.
\end{equation}
Assume $\mu_2$ is not $f\sum_{\alpha\in S\backslash S(P(C_P))}\lambda_\alpha$. Then writing $\mu_2=\sum_{j,\alpha}\mu_{2,j,\alpha}$ where $(j,\alpha)\in \gKQ\times S\backslash S(P(C_P))$ and $\mu_{2,j,\alpha}$ occurs in $L(\lambda_\alpha)$ and applying Lemma \ref{supps(p)} with $P=P(C_P)$, $\lambda=\lambda_{\alpha}$ and $\mu=\mu_{2,j,\alpha}$ for $\alpha\in S\backslash S(P(C_P))$ (the assumptions in Lemma \ref{supps(p)} are satisfied since the $\lambda_{\alpha}$, $\alpha\in S$ are fundamental weights), we get that $f\sum_{\alpha\in S\backslash S(P(C_P))}\lambda_\alpha-\mu_2$ has at least one root of $S\backslash S(P(C_P))$ in its support. Averaging over $w\in W(P(C_P))$ as in (\ref{l'}) and using $w(\lambda_\alpha)=\lambda_\alpha$ for $w\in W(P(C_P))$ and $\alpha \in S\backslash S(P(C_P))$ (same proof as for Lemma \ref{utile}), we get applying Lemma \ref{sumw} to $P=P(C_P)$ that $f\sum_{\alpha\in S\backslash S(P(C_P))}\lambda_\alpha-\mu'_2$ has still at least one root of $S\backslash S(P(C_P))$ in its support (and that $\mu'_2\leq f\sum_{\alpha\in S\backslash S(P(C_P))}\lambda_\alpha$). Since $\mu'_1\leq f\sum_{\alpha\in S(P(C_P))}\lambda_\alpha$ by the proof of Step 3 below, this root doesn't vanish in the sum (\ref{summation}). But by Proposition \ref{parabolicprop}(ii), $S(P(C_P))$ {\it is} the support of (\ref{summation}), which is a contradiction. Therefore, we must have $\mu_2=f\sum_{\alpha\in S\backslash S(P(C_P))}\lambda_\alpha$ and thus from (iii) that
\begin{equation}\label{mu1}
C_P\cong C'_{P(C_P),P}\ \otimes \bigotimes_{\gKQ}\bigg(\sum_{\alpha\in S\backslash S(P(C_P))}\!\!\!\!\!\lambda_{\alpha}\bigg),\end{equation}
where $C'_{P(C_P),P}$ is the isotypic component of $\big(\bigotimes_{\gKQ}\big(\bigotimes_{\alpha\in S(P(C_P))}\Lbar(\lambda_{\alpha})\big)\big)\vert_{Z_{M_P}}$ associated to $\big(\lambda- f\sum_{\alpha \in S\backslash S(P(C_P))}\lambda_\alpha\big)\vert_{Z_{M_P}}$ $(=(\lambda-\mu_2)\vert_{Z_{M_P}}=\mu_1\vert_{Z_{M_P}})$.

\noindent
Step 3: We prove that
\[f\Big(\!\sum_{\alpha \in S(P(C_P))}\!\!\!\!\lambda_\alpha\Big) - \mu_1\ \ \in \!\!\sum_{\alpha\in S(P(C_P))}\!\!\!\!\Z_{\geq 0}\alpha\]
(i.e.\ no root of $S\backslash S(P(C_P))$ is in the support). Since $\lambda_\alpha$ is dominant, we have $\lambda_\alpha\geq \lambda_\alpha'$, where $\lambda_{\alpha}'$ is defined as in (\ref{l'}) for $P=P(C_P)$ and the character $\lambda_\alpha$. This implies (with obvious notation)
\begin{equation}\label{mu'1}
f\big(\!\sum_{\alpha \in S(P(C_P))}\lambda_\alpha\big)-\mu'_1\geq f\big(\!\sum_{\alpha \in S(P(C_P))}\lambda'_\alpha\big)-\mu'_1=\Big(f\big(\!\sum_{\alpha \in S(P(C_P))}\lambda_\alpha\big)-\mu_1\Big)'\geq 0,
\end{equation}
where the last inequality follows from Lemma \ref{sumw} (applied with $P=P(C_P)$). If $f\big(\!\sum_{\alpha \in S(P(C_P))}\lambda_\alpha\big)\!-\!\mu_1$ has roots of $S\backslash S(P(C_P))$ in its support, then by Lemma \ref{sumw} again so is the case of $\big(f\big(\!\sum_{\alpha \in S(P(C_P))}\lambda_\alpha\big)-\mu_1\big)'$, and thus of $f\big(\!\sum_{\alpha \in S(P(C_P))}\lambda_\alpha\big)-\mu'_1$ by (\ref{mu'1}). As in Step 2, this is again a contradiction by (\ref{summation}) and the definition of $P(C_P)$.

\noindent
Step 4: We prove the statement for $w=\Id$. By Lemma \ref{isohigh} applied with $P=P(C_P)$ and the various $\Lbar(\lambda_{\alpha})$ for $\alpha\in S(P(C_P))$, we deduce from Step 3 that $\mu_1$ is a weight of $\bigotimes_{\gKQ}\!\big(\!\bigotimes_{\alpha\in S(P(C_P))}\Lbar_{P(C_P)}(\lambda_{\alpha})\big)$ inside $\bigotimes_{\gKQ}\!\big(\!\bigotimes_{\alpha\in S(P(C_P))}\Lbar(\lambda_{\alpha})\big)$ (see just after (\ref{alphaP})). Let $\alpha\in S(P(C_P))$, for each $\beta\in S(P(C_P))$ we have $\langle\lambda_\alpha,\beta\rangle = \langle\lambda_{\alpha,P(C_P)},\beta\rangle$ \ (a \ straightforward \ check \ from \ (\ref{alphaP})), \ thus \ $\lambda_\alpha-\lambda_{\alpha,P(C_P)}$ \ extends \ to $\Hom_{\rm Gr}(M_{P(C_P)},{\mathbb G}_{\rm m})$ which implies $\Lbar_{P(C_P)}(\lambda_{\alpha})\cong \Lbar_{P(C_P)}(\lambda_{\alpha,P(C_P)})\otimes (\lambda_\alpha-\lambda_{\alpha,P(C_P)})$. Thus $\mu_1-f\sum_{\alpha\in S(P(C_P))}(\lambda_\alpha-\lambda_{\alpha,P(C_P)})$ is a weight of
\[\bigotimes_{\gKQ}\Big(\bigotimes_{\alpha\in S(P(C_P))}\Lbar_{P(C_P)}(\lambda_{\alpha,P(C_P)})\Big)=\LLbar_{P(C_P)},\]
or in other terms:
\[C'_{P(C_P),P}\cong C_{P(C_P),P}\ \otimes \bigotimes_{\gKQ}\bigg(\sum_{\alpha\in S(P(C_P))}\!\!(\lambda_\alpha-\lambda_{\alpha,P(C_P)})\bigg),\]
where \ \ $C_{P(C_P),P}$ \ is \ \ the \ isotypic \ \ component \ of ${\overline L}^{\otimes}_{P(C_P)}\vert_{Z_{M_{P}}}$ \ associated \ to \ $\big(\lambda-f\sum_{\alpha \in S\backslash S(P(C_P))}\lambda_\alpha - f\sum_{\alpha\in S(P(C_P))}(\lambda_\alpha-\lambda_{\alpha,P(C_P)})\big)\vert_{Z_{M_P}}$. But by (\ref{thetaP}):
\[\sum_{\alpha \in S\backslash S(P(C_P))}\!\!\!\!\!\lambda_\alpha \ \ + \!\!\!\ \sum_{\alpha\in S(P(C_P))}\!\!\!\!(\lambda_\alpha-\lambda_{\alpha,P(C_P)}) =\ \ \theta_G-\!\!\!\!\!\!\!\sum_{\alpha\in S(P(C_P))}\!\!\!\!\lambda_{\alpha,P(C_P)}\ \ =\ \ \theta^{P(C_P)},\]
so together with (\ref{mu1}) we are done.
\end{proof}

\begin{rem}
The character $w^{-1}(\theta^{P(C_P)})$ of $M_P$ doesn't depend on $w\in W(C_P)$, as follows from Lemma \ref{inclusion} and Lemma \ref{utile} (the latter applied with $P$ there being $P(C_P)$). In particular, by (\ref{decompfort}) we see that the representation $w^{-1}\big(C_{P(C_P),{}^wP}\big)$ of $M_{P}^{\gKQ}$ is also independent of $w\in W(C_P)$.
\end{rem}

When $C_P$ is as in Proposition \ref{minimal}, its underlying $M_P^{\gKQ}$-representation looks like $\LLbar$ but for the reductive group $M_P$ instead of $G$. 

\begin{cor}\label{decomp}
Let $C_P$ be an isotypic component of $\LLbar\vert_{Z_{M_{P}}}$ such that $P(C_P)={}^w P$ for some {\upshape(}unique{\upshape)} $w\in W$ such that $w(S(P))\subseteq S$. Then there is an isomorphism
\[C_P\cong {\overline L}^{\otimes}_P\otimes \big(\underbrace{w^{-1}(\theta^{{}^wP})\otimes \cdots \otimes w^{-1}(\theta^{{}^wP})}_{\gKQ}\big)\]
of algebraic representations of $M_P^{\gKQ}$ over $\F$.
\end{cor}
\begin{proof}
If $P(C_P)={}^w P$, then ${\overline L}^{\otimes}_{P(C_P)}\vert_{Z_{M_{{}^wP}}}={\overline L}^{\otimes}_{{}^wP}\vert_{Z_{M_{{}^wP}}}$ has only one isotypic component, corresponding to $f\theta_{{}^wP}\vert_{Z_{M_{{}^wP}}}$. So the corollary follows from Theorem \ref{decompstrong} together with (\ref{wlp}). Note that, by Proposition \ref{minimal}, $C_P$ corresponds to $\lambda=fw^{-1}(\theta_G)$, which is consistent with Theorem \ref{decompstrong} since
\begin{multline*}
(w(\lambda)-f\theta^{P(C_P)})\vert_{Z_{M_{{}^wP}}}=\big(w(fw^{-1}(\theta_G))-f\theta^{{}^wP}\big)\vert_{Z_{M_{{}^wP}}}=f(\theta_G - \theta^{{}^wP})\vert_{Z_{M_{{}^wP}}}\\
=f\theta_{{}^wP}\vert_{Z_{M_{{}^wP}}}.
\end{multline*}
\end{proof}

\begin{rem}\label{gln2}
In this remark, we use that we are working with $G=\GL_n$. We write $M_{P(C_P)}={\rm diag}(M_1,\ldots,M_d)$ for some $d>0$ with $M_i\cong \GL_{n_i}$, and correspondingly $T={\rm diag}(T_1,\ldots,T_d)$, where $T_i$ is the diagonal torus of $\GL_{n_i}$, so that we have $X(T)=\oplus_{i=1}^d X(T_i)$ and $S(P(C_P))=\amalg_{i=1}^d S(M_i)$, where $X(T_i)\defeq \Hom_{\rm Gr}(T_i,{\mathbb G}_{\rm m})$ and $S(M_i)\defeq S(P(C_P))\cap X(T_i)$ is the set of simple roots of $M_i$ (for the Borel subgroup of upper-triangular matrices). Note that $S(M_i)=\emptyset$ if $M_i\cong \GL_1$. For $i\in \{1,\ldots,d\}$ such that $n_i>1$, one easily checks that $\lambda_{\alpha,P(C_P)}\in X(T_i)\subseteq X(T)$ if $\alpha\in S(M_i)$ and that the $\lambda_{\alpha,P(C_P)}\in X(T_i)$ for $\alpha\in S(M_i)$ are fundamental weights for the reductive group $M_i$. For $i\in \{1,\ldots,d\}$ and $\lambda_i\in X(T_i)$, we define $\Lbar_{M_i}(\lambda_i)$ as in (\ref{algbar}) but for the reductive group $M_i$ instead of $G$. When $n_i=1$, we define $\LLbar_i$ to be the trivial representation of $M_i^{\gKQ}\cong {\mathbb G}_{\rm m}^{\gKQ}$, and when $n_i>1$, we define as in (\ref{lbar}) the algebraic representation of $M_i^{\gKQ}$ over $\F$ (seeing $\lambda_{\alpha,P(C_P)}$ in $X(T_i)$):
\begin{equation}\label{lbari}
\LLbar_i\defeq \bigotimes_{\gKQ}\Big(\bigotimes_{\alpha\in S(M_i)}\Lbar_{M_i}(\lambda_{\alpha,P(C_P)})\Big).
\end{equation}
We then clearly have ${\overline L}^{\otimes}_{P(C_P)}\!\cong \bigotimes_{i=1}^d \!\LLbar_i$. Likewise, we have $\theta^{P(C_P)}=\otimes_{i=1}^d(\theta^{P(C_P)})_i$, where $(\theta^{P(C_P)})_i\in X(T_i)$ extends to $\Hom_{\rm Gr}(M_i,{\mathbb G}_{\rm m})$ and where we denote by $\mu_i$ the image in $X(T_i)$ of a character $\mu\in X(T)$.

For any $w\in W(C_P)$, we define $({}^wP)_i$ as the standard parabolic subgroup of $M_i$ which is the image of ${}^wP$ under
\[{}^wP\hookrightarrow P(C_P)\twoheadrightarrow M_{P(C_P)}\twoheadrightarrow M_i\]
(in particular its Levi $M_{({}^wP)_i}$ is the image of $M_{{}^wP}$ under $M_{{}^wP}\hookrightarrow M_{P(C_P)}\twoheadrightarrow M_i$). Applying $w$ to (\ref{decompfort}), it is not difficult to deduce from Theorem \ref{decompstrong} an isomorphism of algebraic representations of $M_{{}^wP}^{\gKQ}\cong \prod_{i=1}^dM_{({}^wP)_i}^{\gKQ}$ over $\F$:
\begin{equation}\label{facteursbis}
w(C_P)\cong \bigotimes_{i=1}^d \Big(C_{w,i}\otimes \big(\underbrace{(\theta^{P(C_P)})_i\otimes \cdots \otimes (\theta^{P(C_P)})_i}_{\gKQ}\big)\Big),
\end{equation}
where \ $C_{w,i}$ \ is \ the \ isotypic \ component \ of \ $\LLbar_i\vert_{Z_{M_{({}^wP)_i}}}$ \ associated \ to $(w(\lambda)-f\theta^{P(C_P)})_i\vert_{Z_{M_{({}^wP)_i}}}$ (thus an $M_{({}^wP)_i}^{\gKQ}$-representation, note that $C_{w,i}$ is trivial if $n_i=1$). If $w'$ is another element in $W(C_P)$, writing $w'=w_{P(C_P)}w$ with $w_{P(C_P)}\in W(P(C_P))$ (Lemma \ref{inclusion}), we have $M_{{}^{w'}\!P}=w_{P(C_P)}M_{{}^wP}w_{P(C_P)}^{-1}$, and thus $w'(C_p)\cong w_{P(C_P)}(w(C_P))$ and $C_{P(C_P),{}^{w'}\!P}\cong w_{P(C_P)}\big(C_{P(C_P),{}^{w}P}\big)$ (as the twist by $\theta^{P(C_P)}\otimes \cdots \otimes \theta^{P(C_P)}$ doesn't involve the choice of $w$). Since $w_{P(C_P)}M_iw_{P(C_P)}^{-1}=M_i$ for all $i$, we get $M_{({}^{w'}\!P)_i}= w_{P(C_P)}M_{({}^wP)_i}w_{P(C_P)}^{-1}$ (inside $M_i$) and deduce for $i\in \{1,\ldots,d\}$ an isomorphism of algebraic representations of $M_{({}^{w'}\!P)_i}^{\gKQ}$ over $\F$ (with notation similar to (\ref{wact})):
\begin{equation}\label{isowi}
C_{w',i}\cong w_{P(C_P)}(C_{w,i}).
\end{equation}
We will avoid applying $w^{-1}$ to $C_{w,i}$ since $w^{-1}M_{P(C_P)}w$ is not in general the Levi subgroup of a standard parabolic subgroup of $G$ (see Remark \ref{ctrex}(ii)), although it indeed contains $M_P$.
\end{rem}

\subsubsection{From one isotypic component to another}

We let $P$ be a standard parabolic subgroup of $G$. We show that, if $C_P$ is an isotypic component of $\LLbar\vert_{Z_{M_{P}}}$, then one can associate to $C_P$ in a natural way another isotypic component $w\cdot C_P$ of $\LLbar\vert_{Z_{M_{P}}}$ for any $w\in W$ such that $w\big(S(P(C_P))\big)\subseteq S$ (see Proposition \ref{w(c)}). Note that, on the contrary to $w(C_P)$, $w\cdot C_P$ is an isotypic component of $\LLbar\vert_{Z_{M_{P}}}$ for the {\it same} standard parabolic subgroup $P$ as $C_P$.

\begin{lem}\label{dom}
Let $\mu\in X(T)$ be a dominant weight. Then $\mu$ occurs in $\LLbar\vert_T$ {\upshape(}for the diagonal embedding of $T$ analogous to {\upshape(\ref{embed}))} if and only if $\mu\leq f\theta_G$ in $X(T)$.
\end{lem}
\begin{proof}
Since this statement only concerns weights, we can work in characteristic $0$, i.e.\ with $\LL\defeq \bigotimes_{\gKQ}\big(\bigotimes_{\alpha\in S}L(\lambda_{\alpha})\big)$, where $L(\lambda_\alpha)\defeq \big({\rm ind}_{B^-}^G\lambda_\alpha\big)_{/\Z}\otimes_{\Z}E$ (see (\ref{algbar})). Arguing as in the proof of \cite[Lemma 2.2.3]{BH}, it is equivalent to prove that $\mu$ is a weight of the algebraic representation $L(f\theta_G)$ of $G$. The result then follows from the inequalities $w(\mu)\leq \mu\leq f\theta_G$ for all $w\in W$ (the left ones hold since $\mu$ is dominant and the right ones since $f\theta_G$ is the highest weight) combined with \cite[Prop.21.3]{Hum}.
\end{proof}

\begin{prop}\label{w(c)}
Let $\lambda_P\in X(Z_{M_P})$ be a character of $Z_{M_P}$ which occurs in $\LLbar\vert_{Z_{M_P}}$ {\upshape(}for the diagonal embedding, as usual{\upshape)} with associated isotypic component $C_P$ of $\LLbar\vert_{Z_{M_{P}}}$, and let $w\in W$ such that $w\big(S(P(C_P))\big)\subseteq S$.
\begin{enumerate}
\item For $w_{C_P}\in W(C_P)$ the character of $Z_{M_{P}}$:
  \begin{eqnarray}\label{penible}
    \lambda_P - \big(fw_{C_P}^{-1}(\theta_{ G})+f(ww_{C_P})^{-1}(\theta_{G})\big)\vert_{Z_{M_{P}}}
    \end{eqnarray}
    doesn't depend on $w_{C_P}$.
  \item The character {\upshape(\ref{penible})} corresponds to an isotypic component $w\cdot C_P$ of $\LLbar\vert_{Z_{M_{P}}}$, i.e.\ occurs in $\LLbar\vert_{Z_{M_{P}}}$.
  \item We have $P(w\cdot C_P)={}^{w} P(C_P)$.
  \end{enumerate}
\end{prop}
\begin{proof}
(i) For any $\alpha\in S(P(C_P))$ we have (since $w(\alpha)$ is still in $S$)
\begin{eqnarray}\label{0}
\langle w^{-1}(\theta_{G})-\theta_{G},\alpha\rangle=\langle \theta_{G},w(\alpha)\rangle-\langle \theta_{G},\alpha\rangle=1-1=0
\end{eqnarray}
which implies $s_\alpha(w^{-1}(\theta_{ G})-\theta_{ G})=w^{-1}(\theta_{ G})-\theta_{ G}$, and thus for all $w'\in W( P(C_P))$:
\begin{equation}\label{invariant}
w'(w^{-1}(\theta_{ G})-\theta_{ G})=w^{-1}(\theta_{ G})-\theta_{ G}. 
\end{equation}
Let $w'_{C_P}\in W(C_P)$, by Lemma \ref{inclusion} we have $w'_{C_P}w_{C_P}^{-1}\in W( P(C_P))$ and thus by (\ref{invariant}):
\[(w'_{C_P}w_{C_P}^{-1})(w^{-1}(\theta_{G})-\theta_{G})=w^{-1}(\theta_{G})-\theta_{ G}.\]
Applying ${w'_{C_P}}^{\!\!\!\!\!-1}$ we get in particular
\[\big(w_{C_P}^{-1}(w^{-1}(\theta_{ G})-\theta_{ G})\big)\vert_{Z_{M_{ P}}}=\big({w'_{C_P}}^{\!\!\!\!\!-1}(w^{-1}(\theta_{ G})-\theta_{ G})\big)\vert_{Z_{M_{P}}}\]
from which (i) follows.\\
(ii) Let $\lambda\in X(T)$ such that $\lambda\vert_{Z_{M_{P}}}=\lambda_P$. Applying $ww_{C_P}$ to (\ref{penible}), it is sufficient to prove that $f\theta_G-w\big(f\theta_G-w_{C_P}(\lambda)\big)$ occurs in $\LLbar\vert_T$ (since $\LLbar\vert_T$ is acted on by the diagonal action of $W\hookrightarrow W^{\gKQ}$). Recall from Lemma \ref{parabolic}(ii) (and the definition of $P(C_P)$) that
\begin{equation}\label{positive}
f\theta_G-w_{C_P}(\lambda)\ \ \in \sum_{\alpha\in S(P(C_P))}\!\!\!\Z_{\geq 0} \alpha.
\end{equation}
For $\beta=w(\alpha)\in w(S(P(C_P)))$ and {\it any} $w'\in W$, we have
\begin{eqnarray}\label{dominant2}
\langle f\theta_G-w(f\theta_G-w'(\lambda)),\beta\rangle&=&\langle ww'(\lambda),\beta\rangle+f\langle \theta_{ G}-w(\theta_{ G}),\beta\rangle\\
\nonumber &=&\langle ww'(\lambda),\beta\rangle+f\langle w^{-1}(\theta_{ G})-\theta_{ G},\alpha\rangle\\
\nonumber &=&\langle ww'(\lambda),\beta\rangle,
\end{eqnarray}
where the last equality follows from (\ref{0}). This can be rewritten as
\begin{eqnarray}\label{dominant2'}
s_\beta\big(f\theta_G-w(f\theta_G-w'(\lambda))\big)\!\!&=&\!\! f\theta_G-w(f\theta_G-w'(\lambda))-\langle ww'(\lambda),\beta\rangle\beta\\
\nonumber \!\!&=&\!\! f\theta_G-w(f\theta_G-s_\alpha w'(\lambda)).
\end{eqnarray}
Iterating (\ref{dominant2'}), we see that for any $w_{P(C_P)}\in W(P(C_P))$, we have for $w'\in W$ that
\begin{equation}\label{dominant2''}
ww_{P(C_P)}w^{-1}\big(f\theta_G-w(f\theta_G-w'(\lambda))\big)= f\theta_G-w(f\theta_G-w_{P(C_P)}w'(\lambda)).
\end{equation}
Choose $w_{P(C_P)}\in W(P(C_P))$ such that $w_{P(C_P)}(w_{C_P}(\lambda))$ is dominant for the root subsystem generated by $S(P(C_P))$, equivalently
\begin{equation}\label{dominant3}
\langle ww_{P(C_P)}w_{C_P}(\lambda),\beta\rangle\geq 0\ \ \ \forall\ \beta\in w(S(P(C_P))).
\end{equation}
As $\lambda$ occurs in $\LLbar\vert_T$, we get that $w_{P(C_P)}(w_{C_P}(\lambda))\in w_{C_P}(\lambda) + \sum_{\alpha\in S(P(C_P))}\Z \alpha$ occurs in $\LLbar\vert_T$ ($\LLbar$ is stable under $W$), and thus $w_{P(C_P)}(w_{C_P}(\lambda))\leq f\theta_G$. Since on the other hand by (\ref{positive}):
\[f\theta_G-w_{P(C_P)}(w_{C_P}(\lambda))=(f\theta_G-w_{C_P}(\lambda))\ \ +\!\!\sum_{\alpha\in S(P(C_P))}\!\!\!\Z \alpha \ \ \ \in \sum_{\alpha\in S(P(C_P))}\!\!\!\Z \alpha,\]
we see that we must have
\begin{equation}\label{positivebis}
f\theta_G-w_{P(C_P)}w_{C_P}(\lambda)\ \ \ \in \sum_{\alpha\in S(P(C_P))}\!\!\!\Z_{\geq 0} \alpha.
\end{equation}
Since $w(S(P(C_P)))\subseteq S$, we deduce $\langle w(f\theta_{ G}-w_{P(C_P)}w_{C_P}(\lambda)),\beta\rangle\leq 0$ for $\beta\in S\backslash w(S( P(C_P)))$. In particular we have for such $\beta$:
\begin{eqnarray}\label{dominant1}
\langle f\theta_G-w(f\theta_G-w_{P(C_P)}w_{C_P}(\lambda)),\beta\rangle\!&=&\!f-\langle w(f\theta_{G}-w_{P(C_P)}w_{C_P}(\lambda)),\beta\rangle\\
\nonumber \!&\geq &\! f.
\end{eqnarray}
Combining (\ref{dominant2}) for $w'=w_{P(C_P)}w_{C_P}$ with (\ref{dominant3}) and (\ref{dominant1}), we obtain that $f\theta_G-w(f\theta_G-w_{P(C_P)}w_{C_P}(\lambda))$ is a dominant weight. Applying $w$ to (\ref{positivebis}), we also get since $w(S(P(C_P)))\subseteq S$:
\[f\theta_G-w(f\theta_G-w_{P(C_P)}w_{C_P}(\lambda)) \leq f\theta_G.\]
Lemma \ref{dom} then implies that $f\theta_G-w(f\theta_G-w_{P(C_P)}w_{C_P}(\lambda))$ occurs in $\LLbar\vert_T$. By (\ref{dominant2''}) applied with $w'=w_{C_P}$, we finally deduce that $f\theta_G-w(f\theta_G-w_{C_P}(\lambda))$ also occurs in $\LLbar\vert_T$.\\
(iii) By definition $S(P(w\cdot C_P))\subseteq S$ is the union of $w'(S(P))$ and of the support of
\begin{equation}\label{support}
f\theta_G - w'\big(\lambda - fw_{C_P}^{-1}(\theta_{ G})+f(ww_{C_P})^{-1}(\theta_{G})\big)
\end{equation}
for any $w'\in W$ such that $w'(S(P))\subseteq S$ and $w'\big(\lambda - fw_{C_P}^{-1}(\theta_{ G})+f(ww_{C_P})^{-1}(\theta_{G})\big)$ is the restriction to $Z_{M_{{}^{w'}\!\!P}}$ of a dominant weight of $X( T)\otimes_{\Z}\Q$. Consider the case $w'\defeq ww_{C_P}$, since $w_{C_P}(S(P))\subseteq S(P(C_P))$ and $w(S(P(C_P)))\subseteq S$, we get $w'(S(P))\subseteq S$. Let us check that
\[w'\big(\lambda - fw_{C_P}^{-1}(\theta_{ G})+f(ww_{C_P})^{-1}(\theta_{G})\big)=ww_{C_P}(\lambda)-fw(\theta_{ G})+f\theta_{ G}\]
is the restriction to $Z_{M_{{}^{w'}\!\!P}}$ of a dominant weight of $X( T)\otimes_{\Z}\Q$. Let $\lambda'$ as in (\ref{l'}), since $\lambda\vert_{Z_{M_{ P}}}=\lambda'\vert_{Z_{M_{ P}}}$, we have $w'(\lambda)\vert_{Z_{M_{{}^{w'}\!\!P}}}=w'(\lambda')\vert_{Z_{M_{{}^{w'}\!\!P}}}$ and it is enough to prove that $ww_{C_P}(\lambda')-fw(\theta_{ G})+f\theta_{ G}$ is dominant. As in (\ref{dominant2}) we have if $\alpha\in w(S(P(C_P)))$:
\begin{eqnarray*}
\langle ww_{C_P}(\lambda')-fw(\theta_{ G})+f\theta_{ G},\alpha\rangle&=&\langle ww_{C_P}(\lambda'),\alpha\rangle+f\langle \theta_{ G}-w(\theta_{ G}),\alpha\rangle\\
&=&\langle w_{C_P}(\lambda'),w^{-1}(\alpha)\rangle \geq 0
\end{eqnarray*}
since $w_{C_P}(\lambda')$ is dominant in $X(T)\otimes_{\Z}\Q$ by Proposition \ref{parabolicprop}(i), and as in (\ref{dominant1}) we have if $\alpha\in S\backslash w(S( P(C_P)))$:
\begin{equation*}
\langle ww_{C_P}(\lambda')-fw(\theta_{ G})+f\theta_{ G},\alpha\rangle = f-\langle w(f\theta_{ G}-w_{C_P}(\lambda')),\alpha\rangle\geq f
\end{equation*}
since $w\big(f\theta_G-w_{C_P}(\lambda')\big)\in \sum_{\beta\in S(P(C_P))}\Q_{\geq 0} w(\beta)$ from Proposition \ref{parabolicprop}(ii). Now all that remains is to compute (\ref{support}) for $w'=ww_{C_P}$, which gives $w(f\theta_G-w_{C_P}(\lambda))$, the support of which is $w({\rm support}(f\theta_G-w_{C_P}(\lambda)))$. Therefore we obtain
\[S(P(w\cdot C_P))=w\Big(w_{C_P}(S(P)) \cup {\rm support}\big(f\theta_G-w_{C_P}(\lambda)\big)\Big)=w\big(S(P(C_P))\big)\]
which finishes the proof.
\end{proof}

\begin{rem}\label{minimalbis}
If $C_P$ is one of the isotypic components of Proposition \ref{minimal}, say associated to $fw_{C_P}^{-1}(\theta_G)\vert_{Z_{M_P}}$ for some $w_{C_P}\in W$ such that $w_{C_P}(S(P))\subseteq S$, and if $w\in W$ is such that $w(S(P(C_P)))\subseteq S$, i.e.\ $ww_{C_P}(S(P))\subseteq S$, we see from (\ref{penible}) that $w\cdot C_P$ is the isotypic component associated to $f(ww_{C_P})^{-1}(\theta_G)\vert_{Z_{M_P}}$.
\end{rem}

\begin{ex}\label{exemplesbis}
Let us consider Example \ref{exemples}(ii) (Example \ref{exemples}(i) only provides components $C_P$ which are either as in Remark \ref{minimalbis} or such that $P(C_P)=G$). If $P=B$ and $C_P$ is associated to $\lambda_{\Id}=\theta_G$, then $w\cdot C_P$ for $w\in {\mathcal S}_3$ gives the isotypic component associated to $\lambda_w$ (and there is no $w\cdot C_P\ne C_P$ if $C_P$ corresponds to det since $P(C_P)$ is the whole $G$). If $M_P=\GL_2\times \GL_1$, consider $C_P$ associated to $\lambda_0$ and $w\in {\mathcal S}_3$ the unique permutation $e_1\mapsto e_2$, $e_2\mapsto e_3$, $e_3\mapsto e_1$ (so that $w(S(P(C_P)))=w(e_1-e_2)\subseteq S$). Then $w\cdot C_P$ is the isotypic component associated to $\lambda_2$ (here again, there is no $w\cdot C_P\ne C_P$ for $C_P$ corresponding to $\lambda_1$).
\end{ex}

\subsection{Good conjugates of \texorpdfstring{$\rhobar$}{rhobar}}\label{goodconjugate}

Following and generalizing the mod $p$ variant of \cite[\S 3.2]{BH}, we define and study {\it good conjugates} of a continuous $\rhobar:\gK\rightarrow G(\F)$ under a mild assumption on $\rhobar$ (see Definition \ref{gooddef}) and still assuming $K$ unramified. Though some of the results might hold for more general split reductive groups, we use here in the proofs that we work with $\GL_n$.

\subsubsection{Some preliminaries}

We start with a few group-theoretic preliminaries. 

We fix a standard parabolic subgroup $P$ of $G$. Recall that a subset $C\subseteq R^+$ is closed if $\alpha\in C$, $\beta\in C$ with $\alpha+\beta\in R^+$ implies $\alpha+\beta\in C$. For instance $R(P)^+\subseteq R^+$ is closed.

\begin{definit}\label{closedP}
A subset $X\subseteq R^+$ is a {\it closed subset relative to $P$} if it satisfies the following three conditions:
\begin{enumerate}
\item it contains $R(P)^+$;
\item $X\backslash R(P)^+$ is a closed subset of $R^+$;
\item for any $w\in W(P)$, $w(X\backslash R(P)^+)=X\backslash R(P)^+$.
\end{enumerate}
\end{definit}

Note that a closed subset relative to $B$ is the same thing as a closed subset and that $R^+$ is the only closed subset relative to $G$.

\begin{lem}\label{gln}
Let $X\subseteq R^+$ be a closed subset relative to $P$. Then $X$ is a closed subset of $R^+$.
\end{lem}
\begin{proof}
Since we already know that both $R(P)^+$ and $X\backslash R(P)^+$ are closed, it remains to show that if $\alpha\in R(P)^+$ and $\beta\in X\backslash R(P)^+$ are such that $\alpha+\beta\in R^+$, then $\alpha+\beta\in X$. We work with $\GL_n$, and it is then easy to check that $\alpha + \beta=s_\alpha(\beta)$. Since $s_\alpha\in W(P)$, we have $\alpha+\beta\in X\backslash R(P)^+\subseteq X$ by Definition \ref{closedP}(iii).
\end{proof}

\begin{rem}
Note that Lemma \ref{gln} doesn't hold for an arbitrary split connected reductive algebraic group (for instance it doesn't work for ${\rm GSp}_4$). An alternative definition would be to consider closed subsets $Y$ of $R^+\backslash R(P)^+$ such that $Y\cup R(P)$ is also closed.
\end{rem}

If $X\subseteq R^+$ is any closed subset, we let $N_X\subseteq N$ be the Zariski closed algebraic subgroup generated by the root subgroups $N_\alpha$ for $\alpha\in X$ (see \cite[\S II.1.7]{Ja}). Thanks to Lemma \ref{gln}, we can thus consider $N_X$ for any $X\subseteq R^+$ closed relative to $P$.

\begin{lem}\label{closedX}\

  \begin{enumerate}
  \item Let $X$ be a closed subset of $R^+$ relative to $P$. Then $M_PN_X$ is a Zariski closed algebraic subgroup of $P$ containing $M_P$.
  \item Let $\widetilde P\subseteq P$ be a Zariski closed algebraic subgroup containing $M_P$. Then
    there exists a unique closed subset $X$ relative to $P$ such that $\widetilde P=M_P N_X$.
  \end{enumerate}
\end{lem}
\begin{proof}
(i) Since $M_P N_X=M_P N_{X\backslash R(P)^+}$, it is enough to prove that $M_P$ normalizes $N_{X\backslash R(P)^+}$. Let $\alpha\in R(P)^+$, $\beta\in {X\backslash R(P)^+}$ and let $n_\alpha\in N_\alpha$, $n_\beta\in N_\beta$. Then
\begin{eqnarray}\label{conjn}
n_\alpha n_{\beta}n_\alpha^{-1}=\big(\prod_{i,j>0}n_{i\alpha+j\beta}\big)n_\beta,
\end{eqnarray}
where the product is over all integers $i,j>0$ such that $i\alpha+j\beta\in R^+$ (see \cite[\S II.1.2]{Ja}). Since $X\subseteq R^+$ is closed, all these $i\alpha+j\beta$ are in $X$, and since $\beta\notin R(P)^+$, they are all in $X\backslash R(P)^+$. Therefore $n_\alpha n_{\beta}n_\alpha^{-1}\in N_{X\backslash R(P)^+}$. Let $w\in W(P)$, $\beta\in X\backslash R(P)^+$ and $n_\beta\in N_\beta$. Then $w(\beta)\in X\backslash R(P)^+$ implies $w n_\beta w^{-1}\in N_{X\backslash R(P)^+}$. The Bruhat decomposition for the reductive group $M_P$ then shows that $M_P$ normalizes $N_{X\backslash R(P)^+}$.\\
(ii) Let $\widetilde P\subseteq P$ be a closed algebraic subgroup containing $M_P$. Then $\widetilde P=M_P(\widetilde P\cap B)=M_P(\widetilde P\cap N)$ (since $T\subseteq M_P\subseteq \widetilde P$). By \cite[Lemma 3.4.1]{BH} applied to $\widetilde P\cap B\subseteq B$, we deduce $\widetilde P\cap N=N_X$ for a (unique) closed subset $X\subseteq R^+$. Since $M_P\cap N\subseteq \widetilde P\cap N$, the set $X$ contains $R(P)^+$. Since $\widetilde P\cap N_P=N_{X\backslash R(P)^+}$, the set $X\backslash R(P)^+$ is closed, and moreover $\widetilde P=M_PN_{X\backslash R(P)^+}$. Since $M_P$ normalizes $N_P$ and $\widetilde P$, it normalizes $\widetilde P\cap N_P=N_{X\backslash R(P)^+}$, from which Definition \ref{closedP}(iii) easily follows.
\end{proof}

\begin{rem}\label{wrP}
(i) The sets $R(P)^+$ and $R^+$ are closed with respect to $P$ (they correspond respectively to $\widetilde P=M_P$ and $\widetilde P=P$ in Lemma \ref{closedX}). In particular, if $X$ is closed with respect to $P$, from $w(R^+\backslash R(P)^+)=R^+\backslash R(P)^+$ and $w(X\backslash R(P)^+)=X\backslash R(P)^+$, we also get $w(R^+\backslash X)=R^+\backslash X$ for all $w\in W(P)$.\\
(ii) If $X\subseteq R^+$ is a closed subset relative to $P$, it follows from the proof of Lemma \ref{closedX}(i) that $M_P$ normalizes $N_{X\backslash R(P)^+}$.
\end{rem}

\begin{lem}\label{orderw}
Let $X\subseteq R^+$ be a closed subset relative to $P$. Then there are roots $\alpha_1,\ldots,\alpha_m\in R^+\backslash X$ such that we have a partition
\[R^+=X\amalg \{w(\alpha_1) : w\in W(P)\}\amalg \cdots \amalg \{w(\alpha_m) : w\in W(P)\}\]
and such that, for all $i$, $\alpha_i$ is not in the smallest closed subset relative to $P$ containing $X$ and the $\alpha_{j}$ for $1\leq j\leq i-1$.
\end{lem}
\begin{proof}
Since $w(R^+\backslash X)=R^+\backslash X$ for all $w\in W(P)$ (Remark \ref{wrP}(i)), we have a partition $R^+=X\amalg \{w(\alpha_1) : w\in W(P)\}\amalg \cdots \amalg \{w(\alpha_m) : w\in W(P)\}$ for some $\alpha_1,\ldots,\alpha_m\in R^+\backslash X$.
Denote by $h(\cdot)$ the height of a positive root (see e.g.\ \cite[Rem.2.5.3]{BH}). Replacing each $\alpha_i$ by a suitable $w(\alpha_i)$ for $w\in W(P)$, we can assume $h(\alpha_i)$ maximal among the $h(w(\alpha_i))$, $w\in W(P)$. Permuting the $\alpha_i$ if necessary, we can assume $h(\alpha_1)\geq h(\alpha_2)\geq\cdots \geq h(\alpha_m)$. It is enough to prove that each set $X\amalg \{w(\alpha_1) : w\in W(P)\}\amalg \cdots \amalg \{w(\alpha_i) : w\in W(P)\}$ for $1\leq i\leq m$ is closed relative to $P$, or equivalently that $X_i\defeq (X\backslash R(P)^+)\amalg \{w(\alpha_1) : w\in W(P)\}\amalg \cdots \amalg \{w(\alpha_i) : w\in W(P)\}$ satisfies conditions (ii) and (iii) in Definition \ref{closedP} for $1\leq i\leq m$. Since (iii) is clear, let us prove (ii), i.e.\ that each of the $X_i$ is closed in $R^+$.\\
This is obvious if $i=m$ since $R^+\backslash R(P)^+$ is closed, so we can assume $i<m$. If $X_i$ is not closed for some $i<m$, then its complement in $R^+$ contains an element $x$ which is the sum of at least two roots of $X_i$, at least one being in $\{w'(\alpha_j) : w'\in W(P), 1\leq j\leq i\}$ (since $R^+\backslash R(P)^+$ is closed). Such an element $x$ is in $R(P)^+\amalg \{w(\alpha_j) : w\in W(P), i+1\leq j\leq m\}$ and, since $w'(X_i)=X_i$ for $w'\in W(P)$, it also satisfies $w'(x)\in R^+$ for any $w'\in W(P)$. In particular $x$ can't be in $R(P)^+$, and is thus of the form $x=w(\alpha_k)$ for some $k\in \{i+1,\ldots,m\}$ and some $w\in W(P)$. Thus $w(\alpha_k)$ is the sum of at least two roots of $X_i$, one at least being in $\{w'(\alpha_j) : w'\in W(P), 1\leq j\leq i\}$. Applying a convenient $w'\in W(P)$ and using again $w'(X_i)=X_i$, we can modify $w$ if necessary and assume that $\alpha_j$ for some $j\in \{1,\ldots,i\}$ appears in the sum of $w(\alpha_k)$. This implies in particular $h(w(\alpha_{k}))> h(\alpha_j)$ for some $j\leq i$ (see the argument in the proof of \cite[Lemma 3.2.1]{BH}), which is impossible since by assumption $h(w(\alpha_{k}))\leq h(\alpha_{k})\leq h(\alpha_{j})$. Hence $X_i$ is closed for all $i$.
\end{proof}

\begin{lem}\label{w(X)}
Let $X\subseteq R^+$ be a closed subset relative to $P$, $\widetilde P\defeq M_PN_X$ and let $w\in W$ such that $w(S(P))\subseteq S$. Then the following assertions are equivalent:
\begin{enumerate}
\item $w\widetilde Pw^{-1}$ is contained in ${}^{w}P$;
\item $w(X\backslash R(P)^+)\subseteq R^+$.
\end{enumerate}
\end{lem}
\begin{proof}
We have
\[w\widetilde Pw^{-1}=(wM_Pw^{-1})(wN_{X\backslash R(P)^+}w^{-1})=(wM_Pw^{-1})N_{w(X\backslash R(P)^+)}.\]
As ${}^wP\!=\!(wM_Pw^{-1})N$, we deduce $w\widetilde Pw^{-1}\subseteq {}^{w}P$ if and only if $w(X\backslash R(P)^+)\subseteq R^+$.
\end{proof}

\subsubsection{Good conjugates of a generic \texorpdfstring{$\rhobar$}{rhobar}}\label{goodconjugatesub}

We define good conjugates of a $\gK$-representation $\rhobar$ under a mild genericity assumption and show how two good conjugates are related (Theorem \ref{choicegood}). The intuitive idea is that conjugating a good conjugate of $\rhobar$ can only increase the image in $G(\F)$.

We fix a continuous homomorphism
\begin{eqnarray}\label{morprho}
\rhobar:\gK\longrightarrow P_{\rhobar}(\F)\subseteq G(\F),
\end{eqnarray}
where $P_{\rhobar}\subseteq G$ is a standard parabolic subgroup. We consider
\[\rhobar^{P_{\rhobar}-\rm ss}:\gK\buildrel \rhobar\over \longrightarrow P_{\rhobar}(\F) \twoheadrightarrow M_{P_{\rhobar}}(\F),\]
and assume that the image of $\rhobar^{P_{\rhobar}-\rm ss}$ is {\it not} contained in the $\F$-points of a proper (not necessarily standard) parabolic subgroup of $M_{P_{\rhobar}}$. This implies in particular that $P_{\rhobar}$ is uniquely determined by the homomorphism $\rhobar$. Finally we let $\rhobar^{\rm ss}$ be the homomorphism $\gK\rightarrow G(\F)$ obtained by composing $\rhobar^{P_{\rhobar}-\rm ss}$ with the inclusion $M_{P_{\rhobar}}(\F)\subseteq G(\F)$ (so $\rhobar^{\rm ss}$ is the usual semisimplification of $\rhobar$). We let $X_{\rhobar}$ be the smallest closed subset of $R^+$ relative to $P_{\rhobar}$ such that $\widetilde P_{\rhobar}(\F)\defeq M_{P_{\rhobar}}(\F) N_{X_{\rhobar}}(\F)$ contains all the $\rhobar(g)$, $g\in \gK$. By Lemma \ref{closedX}, $\widetilde P_{\rhobar}$ is the smallest closed algebraic subgroup of $P_{\rhobar}$ containing $M_{P_{\rhobar}}$ such that $\rhobar$ takes values in $\widetilde P_{\rhobar}(\F)$, i.e.\ $\rhobar:\gK\rightarrow \widetilde P_{\rhobar}(\F)\hookrightarrow P_{\rhobar}(\F)\hookrightarrow G(\F)$. Note that $X_{\rhobar^{\rm ss}}=R(P)^+$ and $\widetilde P_{\rhobar^{P_{\rhobar}-\rm ss}}=M_{P_{\rhobar}}$.

\begin{lem}\label{smallest}
Assume that the irreducible constituents of $\rhobar^{\rm ss}$ of dimension $1$ {\upshape(}i.e.\ the characters of $\gK$ occurring in $\rhobar^{\rm ss}${\upshape)} are all distinct. Let $\alpha\in R^+\backslash X_{\rhobar}$ and $n_\alpha\in N_\alpha(\F)\backslash \{1\}$. Then $X_{n_\alpha\rhobar n_\alpha^{-1}}$ is the smallest closed subset relative to $P_{\rhobar}$ containing $X_{\rhobar}$ and $\alpha$.
\end{lem}
\begin{proof}
The proof of this lemma is quite technical, but is no more than simple computations in $\GL_n$. We denote by $X_{\rhobar,\alpha}\subseteq R^+$ the smallest closed subset relative to $P_{\rhobar}$ containing $X_{\rhobar}$ and $\alpha$ and by $\widetilde X_{\rhobar}\subseteq X_{\rhobar}$ the subset of roots which are {\it not} the sum of at least two roots of $X_{\rhobar,\alpha}$. For $g\in \gK$ we can write
\begin{eqnarray}\label{nconj0}
\rhobar(g)=\rhobar^{P_{\rhobar}-\rm ss}(g)\prod_{\beta\in X_{\rhobar}\backslash R(P_{\rhobar})^+}n_\beta(g),
\end{eqnarray}
where $\rhobar^{P_{\rhobar}-\rm ss}(g)\in M_{P_{\rhobar}}(\F)$ and $n_\beta(g)\in N_\beta(\F)$.
Using (\ref{conjn}), we see that
\begin{eqnarray}\label{nconj1}
n_\alpha \Big(\prod_{\beta\in X_{\rhobar}\backslash R(P_{\rhobar})^+}\!\!\!\!\!n_\beta(g)\Big)n_\alpha^{-1}\in \prod_{\gamma}N_\gamma(\F),
\end{eqnarray}
where $\gamma$ runs among the roots in $R^+$ of the form $\Z_{\geq 0}\alpha + \Z_{> 0}\beta_1 + \cdots + \Z_{> 0}\beta_s$ for $s\geq 1$ and $\beta_i\in X_{\rhobar}\backslash R(P_{\rhobar})^+$. This clearly implies $X_{n_\alpha\rhobar n_\alpha^{-1}}\subseteq X_{\rhobar,\alpha}$. To prove the reverse inclusion, it is enough to prove $\widetilde X_{\rhobar}\subseteq X_{n_\alpha\rhobar n_\alpha^{-1}}$ and $w(\alpha)\in X_{n_\alpha\rhobar n_\alpha^{-1}}$ for some $w\in W(P_{\rhobar})$ (as then $\alpha\in X_{n_\alpha\rhobar n_\alpha^{-1}}$ by Remark \ref{wrP}(i)).\\
An easy explicit matrix computation in $\GL_n$ (that we leave to the reader) gives that $n_\alpha\rhobar^{P_{\rhobar}-\rm ss}(g) n_\alpha^{-1}$ is of the form in $\GL_n(\F)$:
\begin{eqnarray}\label{nconj2}
n_\alpha\rhobar^{P_{\rhobar}-\rm ss}(g) n_\alpha^{-1}\in \rhobar^{P_{\rhobar}-\rm ss}(g)\!\!\prod_{\beta\in \{w(\alpha) : w\in W(P_{\rhobar})\}}m_{\beta}(g)
\end{eqnarray}
with $m_\beta(g)\in N_\beta(\F)$ (note that, as $w\in W(P_{\rhobar})$, $w(\alpha)$ is of the form $\alpha + n_1 \alpha_1 + \cdots + n_t \alpha_t$ for some $t\geq 0$, $\alpha_i\in S(P_{\rhobar})$, $n_i\in \Z$). It then follows from (\ref{nconj1}) and (\ref{nconj2}) that, for $\beta\in \widetilde X_{\rhobar}\backslash (\widetilde X_{\rhobar}\cap R(P_{\rhobar})^+)$, the entry $n_\beta(g)$ in (\ref{nconj0}) is not affected by the conjugation by $n_\alpha$. In particular, we have $\widetilde X_{\rhobar}\subseteq X_{n_\alpha\rhobar n_\alpha^{-1}}$.\\
We now prove that $w(\alpha)\in X_{n_\alpha\rhobar n_\alpha^{-1}}$ for some $w\in W(P_{\rhobar})$. We first claim that none of the roots $\gamma$ in (\ref{nconj1}) are in $\{w(\alpha) : w\in W(P_{\rhobar})\}$. Indeed, assume $w(\alpha)=m \alpha + m_1\beta_1 + \cdots + m_s\beta_s$ for some $s\geq 0$, $m\geq 0$, $\beta_i\in X_{\rhobar}\backslash R(P_{\rhobar})^+$, $m_i>0$. If $m=0$, then we get $w(\alpha)=m_1\beta_1 + \cdots + m_s\beta_s\in X_{\rhobar}\backslash R(P_{\rhobar})^+$ since $X_{\rhobar}\backslash R(P_{\rhobar})^+$ is closed in $R^+$, which implies $\alpha\in X_{\rhobar}\backslash R(P_{\rhobar})^+$ by Definition \ref{closedP}(iii), a contradiction. If $m>0$, then we get $(m-1)\alpha + m_1\beta_1 + \cdots + m_s\beta_s = n_1 \alpha_1 + \cdots + n_t \alpha_t$ (writing $w(\alpha)$ as in the above form), which implies in particular all $\beta_i\in R(P_{\rhobar})^+$, a contradiction. We deduce from this that for all $g\in \gK$:
\[n_\alpha\rhobar(g) n_\alpha^{-1}\in n_\alpha\rhobar^{P_{\rhobar}-\rm ss}(g) n_\alpha^{-1}\prod_{\gamma}N_\gamma(\F)\]
with $\gamma$ in $R^+\backslash \big(R(P_{\rhobar})^+\amalg \{w(\alpha) : w\in W(P_{\rhobar})\}\big)$.\\
We can see $\rhobar^{P_{\rhobar}-\rm ss}(g)$ as a block matrix ${\rm diag}(\rhobar_1(g),\ldots,\rhobar_d(g))$, where $\rhobar_i:\gK\rightarrow \GL_{n_i}(\F)$ is irreducible. Assume that $\{w(\alpha) : w\in W(P_{\rhobar})\}\supsetneq \{\alpha\}$. Then using that, for fixed $i$, the $\rhobar_i(g)$ for $g\in \gK$ do not take all values in the $\F$-points of a strict (not necessarily standard) parabolic subgroup of $\GL_{n_i}$, one can check that at least one $m_\beta(g)$ in (\ref{nconj2}) is nontrivial for some $g\in \gK$. If $\{w(\alpha) : w\in W(P_{\rhobar})\}=\{\alpha\}$, then there are integers $1\leq i<j\leq d$ such that $n_i=n_j=1$ and the non-diagonal entry in $m_\alpha(g)$ is $(\rhobar_i(g)-\rhobar_j(g))x_\alpha$, where $x_\alpha\in \F^\times$ is the non-diagonal entry in $n_\alpha$. By assumption, there is at least one $g\in \gK$ such that $\rhobar_i(g)\ne \rhobar_j(g)$, which implies $m_\alpha(g)\ne 1$ for that $g$.\\
Hence we finally deduce that
\[n_\alpha\rhobar(g) n_\alpha^{-1}\in \rhobar^{P_{\rhobar}-\rm ss}(g)\bigg(\prod_{\beta\in \{w(\alpha) : w\in W(P_{\rhobar})\}}\!\!\!\!\!m_{\beta}(g)\bigg)\prod_{\gamma}N_\gamma(\F)\]
with $\gamma$ in $R^+\backslash \big(R(P_{\rhobar})^+\amalg \{w(\alpha) : w\in W(P_{\rhobar})\}\big)$ and at least one $m_\beta(g)$ being nontrivial for some $g\in \gK$ and some $\beta\in \{w(\alpha) : w\in W(P_{\rhobar})\}$. This implies that this $\beta$ is in $X_{n_\alpha\rhobar n_\alpha^{-1}}$ and finishes the proof.
\end{proof}

\begin{prop}\label{minP}
Let $\rhobar:\gK\rightarrow P_{\rhobar}(\F)$ and $X_{\rhobar}$ as below {\upshape(\ref{morprho})}, and assume that the irreducible constituents of $\rhobar^{\rm ss}$ of dimension $1$ are all distinct. Then there is $h_0\in P_{\rhobar}(\F)$ {\upshape(}non unique in general{\upshape)} such that $X_{h_0\rhobar h_0^{-1}}\subseteq X_{h\rhobar h^{-1}}$ for all $h\in P_{\rhobar}(\F)$.
\end{prop}
\begin{proof}
The proof is modelled on that of \cite[Prop.3.2.3]{BH}. Since $M_{P_{\rhobar}}$ normalizes $N_{X_{\rhobar}\backslash R(P_{\rhobar})^+}$ (Remark \ref{wrP}(ii)), it is enough to prove the same statement with $h_0,h\in N_{P_{\rhobar}}(\F)$. Using that $\rhobar^{P_{\rhobar}-\rm ss}(g)^{-1}h\rhobar^{P_{\rhobar}-\rm ss}(g)\in N_{X_{\rhobar}\backslash R(P_{\rhobar})^+}(\F)$ for $h\in N_{X_{\rhobar}\backslash R(P_{\rhobar})^+}(\F)\subseteq N_{P_{\rhobar}}(\F)$ by Remark \ref{wrP}(ii) again, and that $N_{X_{\rhobar}\backslash R(P_{\rhobar})^+}(\F)$ is a group, we deduce $X_{h\rhobar h^{-1}}\subseteq X_{\rhobar}$ for all $h\in N_{X_{\rhobar}\backslash R(P_{\rhobar})^+}(\F)$. Replacing $\rhobar$ by a suitable conjugate $h_0\rhobar h_0^{-1}$ with $h_0\in N_{X_{\rhobar}\backslash R(P_{\rhobar})^+}(\F)$, we can assume $X_{h\rhobar h^{-1}}=X_{\rhobar}$ for all $h\in N_{X_{\rhobar}\backslash R(P_{\rhobar})^+}(\F)$. It is enough to prove $X_{\rhobar} \subseteq X_{h\rhobar h^{-1}}$ for all $h\in N_{P_{\rhobar}}(\F)$. Choosing roots $\alpha_1,\ldots,\alpha_m\in R^+\backslash X_{\rhobar}$ as in Lemma \ref{orderw} (for $P=P_{\rhobar}$ and $X=X_{\rhobar}$), we can write any $h\in N_{P_{\rhobar}}(\F)$ as $h=h_mh_{m-1}\cdots h_1h_{\rhobar}$, where $h_i\in \prod_{\beta\in \{w(\alpha_i) : w\in W(P_{\rhobar})\}}N_{\beta}(\F)$ and $h_{\rhobar}\in N_{X_{\rhobar}\backslash R(P_{\rhobar})^+}(\F)$. We have $X_{h_{\rhobar}\rhobar h_{\rhobar}^{-1}}=X_{\rhobar}$ and a straightforward induction applying successively Lemma \ref{smallest} to $X_{h_{\rhobar}\rhobar h_{\rhobar}^{-1}}$ and $\alpha=\alpha_1$, $X_{h_1h_{\rhobar}\rhobar (h_1h_{\rhobar})^{-1}}$ and $\alpha=\alpha_2$, etc. (which we can do thanks to Lemma \ref{orderw}) gives that $X_{h\rhobar h^{-1}}$ is the smallest closed subset of $R^+$ relative to $P_{\rhobar}$ containing $X_{\rhobar}$ and the $\alpha_i$, $i=1,\ldots,m$. In particular $X_{\rhobar}\subseteq X_{h\rhobar h^{-1}}$ for all $h\in N_{P_{\rhobar}}(\F)$.
\end{proof}

\begin{definit}\label{gooddef}
Let $\rhobar:\gK\longrightarrow G(\F)$ be a continuous homomorphism such that the irreducible constituents of $\rhobar^{\rm ss}$ of dimension $1$ are all distinct. A {\it good conjugate} of $\rhobar$ is a conjugate $\rhobar'$ of $\rhobar$ in $G(\F)$ which satisfies the two conditions:
\begin{enumerate}
\item it is of the form $\rhobar':\gK\rightarrow P_{\rhobar'}(\F)\subseteq G(\F)$ for a standard parabolic subgroup $P_{\rhobar'}$ of $G$ such that the image of ${\rhobar'}^{P_{\rhobar'}-\rm ss}:\gK\buildrel \rhobar'\over \rightarrow P_{\rhobar'}(\F) \twoheadrightarrow M_{P_{\rhobar'}}(\F)$ is not contained in the $\F$-points of a proper parabolic subgroup of $M_{P_{\rhobar'}}$;
\item $X_{\rhobar'}\subseteq X_{h\rhobar' h^{-1}}$ for all $h\in P_{\rhobar'}(\F)$.
\end{enumerate}
\end{definit}

From Proposition \ref{minP}, we easily deduce that good conjugates always exist. If $\rhobar$ is irreducible, then any conjugate of $\rhobar$ in $G(\F)$ is a good conjugate.

For $\rhobar:\gK\longrightarrow \widetilde P_{\rhobar}(\F)\subseteq P_{\rhobar}(\F)$ as in (\ref{morprho}), set
\begin{equation}
\begin{array}{lll}\label{wrhobar}
W_{\rhobar}&\defeq &\{w\in W :  w(S(P_{\rhobar}))\subseteq S\ {\rm and}\ w(X_{\rhobar}\backslash R(P_{\rhobar})^+)\subseteq R^+\}\\
&=&\{w\in W : w(S(P_{\rhobar}))\subseteq S\ {\rm and}\ w\widetilde P_{\rhobar}w^{-1}\subseteq {}^{w}P_{\rhobar}\},
\end{array}
\end{equation}
where the second equality follows from Lemma \ref{w(X)}. Using the definition of $X_{\rhobar}$ we see that, for any $w\in W_{\rhobar}$, we have $X_{w\rhobar w^{-1}}=w(X_{\rhobar})$, where
\[w\rhobar w^{-1}:\gK\longrightarrow w\widetilde P_{\rhobar}(\F)w^{-1}=\widetilde P_{w\rhobar w^{-1}}(\F)\subseteq ({}^{w}P_{\rhobar})(\F).\]
(and note that the set $X_{w\rhobar w^{-1}}$ is relative to ${}^{w}P_{\rhobar}$, while the set $X_{\rhobar}$ is relative to $P_{\rhobar}$).

\begin{lem}\label{goodw}
Let $\rhobar : \gK\rightarrow G(\F)$ as in Definition \ref{gooddef} and $\rhobar':\gK \rightarrow \widetilde P_{\rhobar'}(\F)\subseteq P_{\rhobar'}(\F)$ a good conjugate of $\rhobar$ {\upshape(}where $\widetilde P_{\rhobar'}\defeq M_{P_{\rhobar'}}N_{X_{\rhobar'}}=M_{P_{\rhobar'}}N_{X_{\rhobar'}\backslash R(P_{\rhobar'})^+}${\upshape)}. Then any $h\rhobar' h^{-1}$ for $h\in \widetilde P_{\rhobar'}(\F)$ and any $w\rhobar' w^{-1}$ for $w\in W_{\rhobar'}$ is a good conjugate of $\rhobar$. Moreover we have $X_{h\rhobar'h^{-1}}=X_{\rhobar'}$ and $X_{w\rhobar'w^{-1}}=w(X_{\rhobar'})$.
\end{lem}
\begin{proof}
Again, the proof is formally the same as that of \cite[Lemma 3.2.5]{BH}. The statement is obvious for $h\in \widetilde P_{\rhobar'}(\F)$ (as $h N_{X\backslash R(P)^+} h^{-1}=N_{X\backslash R(P)^+}$ for any $X$ closed subset relative $P$ and any $h\in N_{X\backslash R(P)^+}$) and the very last equality follows from the discussion just above. Following the argument in the proof of Proposition \ref{minP}, it is enough to check
\[X_{h(w\rhobar' w^{-1}){h}^{-1}}=X_{w\rhobar'w^{-1}}\]
for all $h\in N_{X_{w\rhobar' w^{-1}}\backslash R(P_{w\rhobar' w^{-1}})^+}(\F)=N_{w(X_{\rhobar'}\backslash R(P_{\rhobar'})^+)}(\F)$. We have
\[h(w\rhobar' w^{-1}){h}^{-1}=w(w^{-1}hw)\rhobar'(w^{-1}h^{-1}w)w^{-1}.\]
Since $w^{-1}hw\in N_{X_{\rhobar'}\backslash R(P_{\rhobar'})^+}(\F)$, we have $X_{(w^{-1}hw)\rhobar'(w^{-1}h^{-1}w)}\subseteq X_{\rhobar'}$ and since $\rhobar'$ is a good conjugate, we have $X_{\rhobar'}\subseteq X_{(w^{-1}hw)\rhobar'(w^{-1}h^{-1}w)}$, hence $X_{\rhobar'}=X_{(w^{-1}hw)\rhobar'(w^{-1}h^{-1}w)}$. Applying the discussion just before this lemma to $(w^{-1}hw)\rhobar'(w^{-1}h^{-1}w)$ and then to $\rhobar'$, we thus get $X_{h(w\rhobar' w^{-1}){h}^{-1}}=w(X_{(w^{-1}hw)\rhobar'(w^{-1}h^{-1}w)})=w(X_{\rhobar'})=X_{w\rhobar' w^{-1}}$.
\end{proof}

We now state and prove the main result of this section (see \cite[Prop.3.2.6]{BH}).

\begin{thm}\label{choicegood}
Let $\rhobar : \gK\rightarrow G(\F)$ be a continuous homomorphism such that the irreducible constituents of $\rhobar^{\rm ss}$ of dimension $1$ are all distinct. Let $\rhobar'$ and $\rhobar''$ be two good conjugates of $\rhobar$. Then there exist $h\in \widetilde P_{\rhobar'}(\F)$ and $w\in W_{\rhobar'}$ such that $\rhobar''=w(h\rhobar'h^{-1})w^{-1}$. In particular we have $X_{\rhobar''}=w(X_{\rhobar'})$.
\end{thm}
\begin{proof}
By assumption there is $x\in G(\F)$ such that $\rhobar''(g)=x\rhobar'(g)x^{-1}$ for all $g\in \gK$. We can write $x=h''wh'$ with $h'\in P_{\rhobar'}(\F)$, $h''\in P_{\rhobar''}(\F)$ and $w\in W$ such that $w(R(P_{\rhobar'})^+)\subseteq R^+$.

\noindent
Step 1: We prove that $w(S(P_{\rhobar'}))=S(P_{\rhobar''})$. We have $wh'\rhobar'(g){h'}^{-1}w^{-1}\in P_{\rhobar''}(\F)$ for all $g\in \gK$, which implies $h'\rhobar'(g){h'}^{-1}\in (w^{-1}P_{\rhobar''}w\cap P_{\rhobar'})(\F)\subseteq P_{\rhobar'}(\F)$ for all $g\in \gK$. In particular, using for instance \cite[Prop.2.1(iii)]{DM}, the image of $h'\rhobar'{h'}^{-1}$ in $M_{P_{\rhobar'}}(\F)$ is contained in the $\F$-points of the parabolic subgroup $w^{-1}P_{\rhobar''}w\cap M_{P_{\rhobar'}}$ of $M_{P_{\rhobar'}}$. But since $(h'\rhobar'{h'}^{-1})^{P_{\rhobar'}-\rm ss}$ is conjugate to ${\rhobar'}^{P_{\rhobar'}-\rm ss}$ (recall $h'\in P_{\rhobar'}(\F)$), the image of $h'\rhobar'{h'}^{-1}$ in $M_{P_{\rhobar'}}(\F)$ is not contained in the $\F$-points of a proper parabolic subgroup of $M_{P_{\rhobar'}}$. Thus we must have $w^{-1}P_{\rhobar''}w\cap M_{P_{\rhobar'}}=M_{P_{\rhobar'}}$ which implies $M_{P_{\rhobar'}}\subseteq w^{-1}M_{P_{\rhobar''}}w$. The same argument starting with $w^{-1}{h''}^{-1}\rhobar''(g){h''}w\in P_{\rhobar'}(\F)$ yields $M_{P_{\rhobar''}}\subseteq wM_{P_{\rhobar'}}w^{-1}$, i.e.\ we have $M_{P_{\rhobar'}}=w^{-1}M_{P_{\rhobar''}}w$. Since by assumption $w(R(P_{\rhobar'})^+)\subseteq R^+$, this forces $w(S(P_{\rhobar'}))=S(P_{\rhobar''})$ (and thus $w(R(P_{\rhobar'})^+)=R(P_{\rhobar''})^+$).

\noindent
Step 2: We choose roots $\alpha'_1,\ldots,\alpha'_{m'}\in R^+\backslash X_{\rhobar'}$ as in Lemma \ref{orderw} (for $P=P_{\rhobar'}$ and $X=X_{\rhobar'}$) and we write $h'=h'_{m'}h'_{m'-1}\cdots h'_1h'_{\rhobar}$, where $h'_i\in \!\prod_{\beta\in \{w'(\alpha'_i) : w'\in W(P_{\rhobar'})\}}\!N_{\beta}(\F)$ and $h'_{\rhobar'}\in \widetilde P_{\rhobar'}(\F)$. By Lemma \ref{goodw}, we can replace $\rhobar'$ by $h'_{\rhobar'}\rhobar'{{h'}^{-1}_{\rhobar'}}$ and thus assume $h'_{\rhobar'}=1$. By Lemma \ref{smallest} and an induction as in the proof of Proposition \ref{minP}, $X_{h'\rhobar' {h'}^{-1}}$ is the smallest closed subset relative to $P_{\rhobar'}$ containing $X_{\rhobar'}$ and those $\alpha'_i$ such that $h'_i\ne 1$. Since $w(h'\rhobar'{h'}^{-1})w^{-1}$ takes values in $P_{\rhobar''}(\F)$ and $w(R(P_{\rhobar'}))=R(P_{\rhobar''})$ (by Step 1), we must also have $w(X_{h'\rhobar' {h'}^{-1}}\backslash R(P_{\rhobar'})^+) \subseteq R^+\backslash R(P_{\rhobar''})^+$. This implies $ww'(\alpha'_i)\in R^+$ if $w'\in W(P_{\rhobar'})$ and $h'_i\ne 1$, and $w(X_{\rhobar'}\backslash R(P_{\rhobar'})^+) \subseteq R^+$. In particular $w\in W_{\rhobar'}$ together with Step 1.

\noindent
Step 3: We prove that $X_{\rhobar''}=w(X_{\rhobar'})$. Setting
\[h_i\defeq wh'_iw^{-1}\ \ \in \!\!\!\prod_{\beta\in \{ww'(\alpha'_i) : w'\in W(P_{\rhobar'})\}}\!\!\!\!N_{\beta}(\F)\ \ \subseteq \ \ P_{\rhobar''}(\F)\] (we proved $ww'(\alpha'_i)\in R^+$ in Step 2), we have
\begin{eqnarray}\label{newconj}
\rhobar''= h''(h_{m'}\cdots h_1)(w\rhobar'w^{-1})(h_1^{-1}\cdots h_{m'}^{-1}){h''}^{-1},
\end{eqnarray}
where $h''h_{m'}\cdots h_1\in P_{\rhobar''}(\F)$ and where $\rhobar''$ and $w\rhobar'w^{-1}$ are good conjugates of $\rhobar$ (the latter by Lemma \ref{goodw}). Applying Definition \ref{gooddef} to both $\rhobar''$ and $w\rhobar'w^{-1}$, we get $X_{\rhobar''}=X_{w\rhobar'w^{-1}}=w(X_{\rhobar'})$ (and thus $w^{-1}\widetilde P_{\rhobar''}w=\widetilde P_{\rhobar'}$).

\noindent
Step 4 : We complete the proof. We choose again roots $\alpha''_1,\ldots,\alpha''_{m''}\in R^+\backslash X_{w\rhobar'w^{-1}}$ as in Lemma \ref{orderw} for $P=P_{w\rhobar'w^{-1}}=P_{\rhobar''}$ (this latter equality from Remark \ref{zariski}) and $X=X_{w\rhobar'w^{-1}}=X_{\rhobar''}$ and we write
\[h''(h_{m'}\cdots h_1)=h''_{m''}h''_{m''-1}\cdots h''_1h''_{X_{\rhobar''}},\]
where $h''_i\in \prod_{\beta\in \{w''(\alpha''_i) : w''\in W(P_{\rhobar''})\}}N_{\beta}(\F)$ and $h''_{X_{\rhobar''}}\in \widetilde P_{w\rhobar'w^{-1}}(\F)=\widetilde P_{\rhobar''}(\F)$. From (\ref{newconj}) and Lemma \ref{smallest}, we see that we must have $h''_i=1$ for all $i\in \{1,\ldots,m''\}$ otherwise $X_{\rhobar''}$ would be strictly bigger that $X_{w\rhobar'w^{-1}}$. Thus we deduce
\[\rhobar''=h''_{X_{\rhobar''}}w\rhobar'w^{-1}{h''}^{-1}_{\!\!\!\!X_{\rhobar''}}=w(w^{-1}h''_{X_{\rhobar''}}w)\rhobar'(w^{-1}{h''}^{-1}_{\!\!\!\!X_{\rhobar''}}w)w^{-1}.\]
Setting $h\defeq w^{-1}h''_{X_{\rhobar''}}w\in w^{-1}\widetilde P_{\rhobar''}(\F)w=\widetilde P_{\rhobar'}(\F)$, this finishes the proof.
\end{proof}

\subsection{The definition of compatibility}\label{formalism}

Given a sufficiently generic $n$-dimensional representation of $\gK$ over $\F$ (where $K=\Qpf$ is still unramified) and a good conjugate $\rhobar$ of this representation as in Definition \ref{gooddef}, we define what it means for a smooth representation of $G(K)$ over $\F$ to be {\it compatible with $\widetilde P_{\rhobar}$} (Definition \ref{compatible1}, see the beginning of \S\ref{goodconjugatesub} for $\widetilde P_{\rhobar}$) and to be {\it compatible with ${\rhobar}$} (Definition \ref{compatible2}).

\subsubsection{Compatibility with \texorpdfstring{$\widetilde P$}{\textbackslash tilde P}}\label{compatible1sec}

We first define what it means for a smooth representation of $G(K)$ over $\F$ to be compatible with a Zariski closed subgroup $\widetilde P$ of a standard parabolic subgroup $P$ as in Definition \ref{goodsubqt}. We keep the notation of \S\S\ref{goodcomponent},~\ref{goodconjugate}. 

We fix a Zariski closed algebraic subgroup $\widetilde P$ of a standard parabolic subgroup $P$ of $G$ as in Definition \ref{goodsubqt} (by Remark \ref{zariski}, $P$ is in fact determined by $\widetilde P$). We let $X$ be the unique closed subset of $R^+$ relative to $P$ such that $\widetilde P=M_PN_X$ (Lemma \ref{closedX}) and define
\[W_{\widetilde P}\defeq \{w\in W : w(S(P))\subseteq S,\ w(X\backslash R(P)^+)\subseteq R^+\}.\]
Note that $W_{\widetilde P}$ is analogous to $W_{\rhobar}$ in (\ref{wrhobar}) with $\widetilde P_{\rhobar}$ replaced by $\widetilde P$.

Let $Q$ be a parabolic subgroup containing ${}^{w_{\widetilde P}}\!P$ for some $w_{\widetilde P}\in W_{\widetilde P}$, $w_Q$ an element of $W$ such that $w_Q(S(Q))\subseteq S$ and $Q'$ a parabolic subgroup containing ${}^{w_Q} Q$ (note that both $Q$ and $Q'$ are standard). So we have inclusions of standard parabolic subgroups ${}^{w_Qw_{\widetilde P}}\!P\subseteq {}^{w_Q} Q\subseteq Q'$ and likewise for the Levi subgroups
\[M_{{}^{w_Qw_{\widetilde P}}\!P}=w_Qw_{\widetilde P}M_P(w_Qw_{\widetilde P})^{-1}\subseteq M_{{}^{w_Q} Q}=w_QM_Qw_Q^{-1}\subseteq M_{Q'}.\]
Using that we work with $\GL_n$, we write
\[M_{Q'}={\rm diag}(M_1,\ldots ,M_d)\]
with $M_i\cong \GL_{n_i}$ and we define the standard parabolic subgroup $({}^{w_Q}Q)_i$ of $M_i$ as
\[({}^{w_Q}Q)_i\defeq \mathrm{Im}\big({}^{w_Q}Q\hookrightarrow Q'\twoheadrightarrow M_{Q'}\twoheadrightarrow M_i\big).\]
We define a standard parabolic subgroup $({}^{w_Qw_{\widetilde P}}\!P)_{Q}$ of $M_{{}^{w_Q}Q}$, resp.\ a standard parabo\-lic subgroup $({}^{w_Qw_{\widetilde P}}\!P)_{Q,i}$ of $M_{({}^{w_Q}Q)_i}$, as the image of ${}^{w_Qw_{\widetilde P}}\!P$ via ${}^{w_Qw_{\widetilde P}}\!P\subseteq{}^{w_Q}Q\twoheadrightarrow M_{{}^{w_Q}Q}$, resp.\ via ${}^{w_Qw_{\widetilde P}}\!P\subseteq{}^{w_Q}Q\twoheadrightarrow M_{{}^{w_Q}Q}\twoheadrightarrow M_{({}^{w_Q}Q)_i}$. Equivalently,
\begin{eqnarray*}
({}^{w_Qw_{\widetilde P}}\!P)_{Q}&=&w_Q({}^{w_{\widetilde P}}\!P\cap M_Q)w_Q^{-1}\subseteq w_QM_Qw_Q^{-1}=M_{{}^{w_Q}Q}\\
({}^{w_Qw_{\widetilde P}}\!P)_{Q,i}&=&\mathrm{Im}\big(w_Q({}^{w_{\widetilde P}}\!P\cap M_Q)w_Q^{-1}\subseteq M_{{}^{w_Q}Q}\twoheadrightarrow M_{({}^{w_Q}Q)_i}\big).
\end{eqnarray*}
Note that
\[M_{({}^{w_Qw_{\widetilde P}}\!P)_{Q}}=w_QM_{({}^{w_{\widetilde P}}\!P)\cap M_Q}w_Q^{-1}=w_Qw_{\widetilde P}M_P(w_Qw_{\widetilde P})^{-1}.\]
We finally define a Zariski closed algebraic subgroup $({}^{w_Qw_{\widetilde P}}\!\widetilde P)_{Q}$ of $({}^{w_Qw_{\widetilde P}}\!P)_{Q}$ containing $M_{({}^{w_Qw_{\widetilde P}}\!P)_{Q}}$, resp.\ a Zariski closed algebraic subgroup $({}^{w_Qw_{\widetilde P}}\!\widetilde P)_{Q,i}$ of $({}^{w_Qw_{\widetilde P}}\!P)_{Q,i}$ containing $M_{({}^{w_Qw_{\widetilde P}}\!P)_{Q,i}}$, as
\begin{eqnarray*}
({}^{w_Qw_{\widetilde P}}\!\widetilde P)_{Q}\!&\!\defeq \!&\!w_Q\big((w_{\widetilde P}\widetilde Pw_{\widetilde P}^{-1})\cap M_Q\big)w_Q^{-1}\subseteq w_Q({}^{w_{\widetilde P}}\!P\cap M_Q)w_Q^{-1}=({}^{w_Qw_{\widetilde P}}\!P)_{Q}\\
({}^{w_Qw_{\widetilde P}}\!\widetilde P)_{Q,i}\!&\!\defeq \!&\!\mathrm{Im}\Big(w_Q\big((w_{\widetilde P}\widetilde Pw_{\widetilde P}^{-1})\cap M_Q\big)w_Q^{-1}\subseteq M_{{}^{w_Q}Q}\twoheadrightarrow M_{({}^{w_Q}Q)_i}\Big).
\end{eqnarray*}
We also define the continuous group homomorphism
\begin{eqnarray*}
\omega^{-1}\circ\theta^{Q'}&:&{Q'}^-(K)\longrightarrow M_{Q'}(K)\buildrel\theta^{Q'}\over\longrightarrow K^\times \buildrel \omega^{-1}\over \longrightarrow \Fp^\times\hookrightarrow \F^\times,
\end{eqnarray*}
where $\theta^{Q'}$ is defined in (\ref{thetaP}) (applied with $P=Q'$).

We need a quite formal and easy lemma.
 
\begin{lem}\label{cons1}
Let $\Pi$ be a smooth representation of a $p$-adic analytic group over $\F$ which has finite length and distinct absolutely irreducible constituents. Let $H$ be a split connected reductive algebraic group over $\Z$, $P_H\subseteq H$ a parabolic subgroup with Levi $M_{P_H}$, $\widetilde P_H\subseteq P_H$ a Zariski closed algebraic subgroup containing $M_{P_H}$ and $R$ a {\upshape(}finite-dimensional{\upshape)} algebraic representation of $P_H^{\gKQ}$ over $\F$. Assume that there exist
\begin{enumerate}[(a)]
\item a filtration on $R$ by good subrepresentations for the $P_H^{\gKQ}$-action {\upshape(}see Definition \ref{goodsubqt}{\upshape)} such that the graded pieces exhaust the isotypic components of $R\vert_{Z_{M_{P_H}}}$;
\item a bijection $\Phi$ of partially ordered finite sets between the set of subre\-presentations of $\Pi$ and the set of good subrepresentations of $R\vert_{{\widetilde P_H}^{\gKQ}}$ {\upshape(}both being ordered by inclusion{\upshape)}.
\end{enumerate}
Then the following hold:
\begin{enumerate}
\item The bijection $\Phi$ uniquely extends to bijections between subquotients of $\Pi$ and good subquotients of $R\vert_{{\widetilde P_H}^{\gKQ}}$, and between irreducible constituents of $\Pi$ and isotypic components of $R\vert_{Z_{M_{P_H}}}$.
\item If $\Pi'$ is a subquotient of $\Pi$, then $\Phi$ induces a bijection of partially ordered finite
  sets between the set of subre\-presentations of $\Pi'$ and the set of good subrepresentations of
  $\Phi(\Pi')\vert_{{\widetilde P_H}^{\gKQ}}$.
\end{enumerate}
\end{lem}
\begin{proof}
Formal and left to the reader.
\end{proof}

\begin{rem}\label{applied}
(i) Let $\Pi$ and $\Phi$ as in Lemma \ref{cons1}, $\Pi'$ a subquotient of $\Pi$ and $\Pi''\subseteq \Pi'$ a subrepresentation. Then the bijection $\Phi$ also induces a short exact sequence $0\rightarrow \Phi(\Pi'')\rightarrow \Phi(\Pi')\rightarrow \Phi(\Pi'/\Pi'')\rightarrow 0$ of algebraic representation of $\widetilde P_H^{\gKQ}$ over $\F$.\\
(ii) By Lemma \ref{filtr} applied with $P$ there being the parabolic ${}^{w_{\widetilde P}}\!P$ above, we see that Lemma \ref{cons1} can be applied with $H=G$, $P_H={}^{w_{\widetilde P}}\!P$, $\widetilde P_H=w_{\widetilde P}\widetilde Pw_{\widetilde P}^{-1}$ and $R=\LLbar$. Using moreover Lemma \ref{Q}, one easily sees that Lemma \ref{cons1} can also be applied with $H=M_Q$, $P_H={}^{w_{\widetilde P}}\!P\cap M_Q$, $\widetilde P_H=(w_{\widetilde P}\widetilde Pw_{\widetilde P}^{-1})\cap M_Q$ and $R$ any isotypic component $C_Q$ of $\LLbar\vert_{Z_{M_Q}}$ (recall from (the proof of) Lemma \ref{filtr} applied with $P$ there being $Q$ that the action of $Q^{\gKQ}$ on the subquotient $C_Q$ of $\LLbar\vert_{Q^{\gKQ}}$ factors through $Q^{\gKQ}\twoheadrightarrow M_Q^{\gKQ}$).\\
(iii) Let $Q$ as above, $C_Q$ an isotypic component of $\LLbar\vert_{Z_{M_Q}}$, $Q'\defeq P(C_Q)$ (see \S\ref{assocparabol}) and $w_Q\in W(C_Q)$ (see (\ref{wCP}) and note that ${}^{w_Q}Q\subseteq Q'$ by (\ref{wPinP(CP)})). Lemma \ref{cons1} can also be applied with $H=M_{({}^{w_Q}Q)_i}$, $P_H=({}^{w_Qw_{\widetilde P}}\!P)_{Q,i}$, $\widetilde P_H=({}^{w_Qw_{\widetilde P}}\!\widetilde P)_{Q,i}$ and $R=C_{w_Q,i}$, where $C_{w_Q,i}$ is the algebraic representation of $M_{({}^{w_Q}Q)_i}^{\gKQ}$ defined in Remark \ref{gln2} with $P$ there being $Q$ (it is an isotypic component of $\LLbar_i\vert_{Z_{M_{({}^{w_Q}Q)_i}}}$). To prove that assumption (a) of Lemma \ref{cons1} is satisfied in that case, note that $C_{w_Q,i}$ is a good subquotient of $\LLbar_i\vert_{({}^{w_Q}Q)_{i}^{\gKQ}}$, and thus {\it a fortiori} a good subquotient of $\LLbar_i\vert_{({}^{w_Qw_{\widetilde P}}\!P)_{Q',i}^{\gKQ}}$ (Lemma \ref{Q}), where $({}^{w_Qw_{\widetilde P}}\!P)_{Q',i}\subseteq ({}^{w_Q}Q)_{i}\subseteq M_i$ is the standard parabolic subgroup of $M_i$ with the same Levi as $({}^{w_Qw_{\widetilde P}}\!P)_{Q,i}$. We have
\[({}^{w_Qw_{\widetilde P}}\!\widetilde P)_{Q,i}\subseteq ({}^{w_Qw_{\widetilde P}}\!P)_{Q,i}\subseteq ({}^{w_Qw_{\widetilde P}}\!P)_{Q',i}\subseteq M_{i}\]
and $({}^{w_Qw_{\widetilde P}}\!\widetilde P)_{Q,i}$ is a closed algebraic subgroup of $({}^{w_Qw_{\widetilde P}}\!P)_{Q',i}$ containing $M_{({}^{w_Qw_{\widetilde P}}\!P)_{Q',i}}=M_{({}^{w_Qw_{\widetilde P}}\!P)_{Q,i}}$. One then applies Lemma \ref{filtr} with $\LLbar_i$ and with
\[({}^{w_Qw_{\widetilde P}}\!\widetilde P)_{Q,i}\subseteq ({}^{w_Qw_{\widetilde P}}\!P)_{Q',i}\subseteq M_i\]
instead of $\widetilde P\subseteq P\subseteq G$, which implies that there is a filtration on $C_{w_Q,i}\vert_{({}^{w_Qw_{\widetilde P}}\!\widetilde P)_{Q,i}^{\gKQ}}$ (or on $C_{w_Q,i}\vert_{({}^{w_Qw_{\widetilde P}}\!P)_{Q',i}^{\gKQ}}$, and thus on $C_{w_Q,i}\vert_{({}^{w_Qw_{\widetilde P}}\!P)_{Q,i}^{\gKQ}}$) by good subrepresentations \ such \ that \ the \ graded \ pieces \ exhaust \ the \ isotypic \ components \ of $C_{w_Q,i}\vert_{Z_{M_{({}^{w_Qw_{\widetilde P}}\!P)_{Q,i}}}}=C_{w_Q,i}\vert_{Z_{M_{({}^{w_Qw_{\widetilde P}}\!P)_{Q',i}}}}$.
\end{rem}

\begin{lem}\label{cons1bis}
Let $\widetilde P\subseteq P$, $w_{\widetilde P}\in W_{\widetilde P}$ and $Q$ containing ${}^{w_{\widetilde P}}\!P$ as above. Let $C_Q$ be an isotypic component of $\LLbar\vert_{Z_{M_Q}}$ and $Q'\defeq P(C_Q)$.
\begin{enumerate}
\item For any $w_Q\in W(C_Q)$, there is a canonical bijection of partially ordered finite sets
  between the set of good subrepresentations of
  \[C_Q\vert_{(w_{\widetilde P}\widetilde Pw_{\widetilde P}^{-1})^{\gKQ}}=C_Q\vert_{((w_{\widetilde
      P}\widetilde Pw_{\widetilde P}^{-1})\cap M_Q)^{\gKQ}}\]
  {\upshape(}where the equality follows from Remark \ref{applied}(ii){\upshape)} and the set of good subrepresentations of $w_Q(C_Q)\vert_{({}^{w_Qw_{\widetilde P}}\!\widetilde P)_{Q}^{\gKQ}}$.
\item For any $w_Q,w'_Q\in W(C_Q)$ and $i\in \{1,\ldots,d\}$, there is a canonical bijection of
  partially \ ordered \ finite \ sets \ between \ the \ set \ of \ good \ subrepresentations of
  $C_{w_Q,i}\vert_{({}^{w_Qw_{\widetilde P}}\!\widetilde P)_{Q,i}^{\gKQ}}$ and the set of good
  subrepresentations of $C_{w_Q',i}\vert_{({}^{w'_Qw_{\widetilde P}}\!\widetilde P)_{Q,i}^{\gKQ}}$.
\end{enumerate}
\end{lem}
\begin{proof}
(i) follows from the definition of $w_Q(C_Q)$ in (\ref{wact}) and the fact that $({}^{w_Qw_{\widetilde P}}\!\widetilde P)_{Q}=w_Q\big((w_{\widetilde P}\widetilde Pw_{\widetilde P}^{-1})\cap M_Q\big)w_Q^{-1}$.\\
(ii) We have $w_Q'=w_{Q'}w_Q$ with $w_{Q'}\in W(P(C_Q))=W(Q')$ by Lemma \ref{inclusion} (applied with $P$ there being $Q$). In particular $w_{Q'}(w_Q(S(Q)))\subseteq S$ which implies $({}^{w'_Qw_{\widetilde P}}\!\widetilde P)_{Q,i}=w_{Q'}({}^{w_Qw_{\widetilde P}}\!\widetilde P)_{Q,i}w_{Q'}^{-1}$ inside $M_{({}^{w'_Q}\!Q)_{i}}=w_{Q'}M_{({}^{w_Q}Q)_{i}}w_{Q'}^{-1}$ (viewing $w_{Q'}$ as an element in $W(M_i)$ by abuse of notation). By (\ref{isowi}) (applied with $P$ there being $Q$) we have $C_{w_Q',i}=w_{Q'}(C_{w_Q,i})$, where the conjugation by $w_{Q'}^{-1}$ intertwines the actions of $({}^{w'_Qw_{\widetilde P}}\!\widetilde P)_{Q,i}$ and of $({}^{w_Qw_{\widetilde P}}\!\widetilde P)_{Q,i}$. The result follows.
\end{proof}

\begin{rem}\label{subquotient}
The bijections in Lemma \ref{cons1bis} all extend to bijections between good subquotients or isotypic components on both sides, as for Lemma \ref{cons1}.
\end{rem}

Let $\Pi$, $H$, $P_H$, $\widetilde P_H$, $R$ and $\Phi$ be as in Lemma \ref{cons1}. For any $w_{H}\in W_H$ (the Weyl group of $H$) such that $w_{H}\widetilde P_{H}w_{H}^{-1}$ is contained in a standard parabolic subgroup of $H$, we can define another bijection $w_{H}(\Phi)$ between the set of subquotients of $\Pi$ and the set of good subquotients of $R\vert_{(w_{H}\widetilde P_H w_{H}^{-1})^{\gKQ}}$ as follows: $w_{H}(\Phi)(\Pi')$ is the algebraic representation $w_{H}\big(\Phi(\Pi')\big)$ of $(w_{H}\widetilde P_{H}w_{H}^{-1})^{\gKQ}$, where $w_{H}\big( \Phi(\Pi')\big)(g)\defeq \Phi(\Pi')(w_{H}^{-1}gw_{H})$ if $g\in (w_{H}\widetilde P_{H}w_{H}^{-1})^{\gKQ}$, see (\ref{wact}).

Here is now the first crucial definition. 

\begin{definit}\label{compatible1}
An admissible smooth representation $\Pi$ of $G(K)$ over $\F$ which has finite length and distinct absolutely irreducible constituents is {\it compatible with $\widetilde P$} if there exists a bijection $\Phi$ of partially ordered finite sets between the set of subre\-presentations of $\Pi$ and the set of good subrepresentations of $\LLbar\vert_{{\widetilde P}^{\gKQ}}$ (both being ordered by inclusion) which satisfies the following conditions (once extended to all subquotients as in Lemma \ref{cons1}):
\begin{enumerate}
\item (\textbf{form of subquotients}) for any $w_{\widetilde P}\in W_{\widetilde P}$, any parabolic subgroup $Q$ containing ${}^{w_{\widetilde P}}\!P$ and any isotypic component $C_Q$ of $\LLbar\vert_{Z_{M_{ Q}}}$, writing $M_{P(C_Q)}= M_1\times \cdots \times M_d$ with $M_i\cong \GL_{n_i}$ we have
\begin{equation}\label{forme}
w_{\widetilde P}(\Phi)^{-1}(C_Q)\cong \Ind_{P(C_Q)^-(K)}^{G(K)}\big(\pi(C_Q)\otimes (\omega^{-1}\circ\theta^{P(C_Q)})\big),
\end{equation}
where $P(C_Q)$ is defined in \S\ref{assocparabol}, $\theta^{P(C_Q)}$ is defined in (\ref{thetaP}) and where $\pi(C_Q)$ is a $M_{P(C_Q)}$-representation of the form $\pi(C_Q)\cong \pi_1(C_Q)\otimes \cdots \otimes \pi_d(C_Q)$ for some (finite length) admissible smooth representations $\pi_i(C_Q)$ of $M_i(K)$ over $\F$;\\
\item (\textbf{compatibility between subquotients}) for any $w_{\widetilde P}\in W_{\widetilde P}$, any parabolic subgroup $Q$ containing ${}^{w_{\widetilde P}}\!P$, any isotypic component $C_Q$ of $\LLbar\vert_{Z_{M_{ Q}}}$ and any $w\in W$ such that $w\big(S(P(C_Q))\big)\subseteq S$, \ let \ $w(\pi(C_Q))$ \ be \ the \ representation \ of \ $M_{{}^{w}P(C_Q)}(K)={w}M_{P(C_Q)}(K){w}^{-1}$ defined by
\[w(\pi(C_Q))(g) \defeq  \pi(C_Q)(w^{-1}gw)\]
for $\pi(C_Q)$ as in (\ref{forme}) and $g\in M_{{}^{w}P(C_Q)}(K)$. Then we have
\[\pi\big(w\cdot C_Q\big) \cong w\big(\pi(C_Q)\big),\]
where $w\cdot C_Q$ is the isotypic component of $\LLbar\vert_{Z_{M_{ Q}}}$ in Proposition \ref{w(c)}(ii) (applied with $P$ there being $Q$) and where $\pi(w\cdot C_Q)$ is as in (\ref{forme}) for the isotypic component $w\cdot C_Q$ instead of $C_Q$ (note that $P(w\cdot C_Q)={}^wP(C_Q)$ by Proposition \ref{w(c)}(iii));\\
\item (\textbf{product structure}) for any $w_{\widetilde P}\in W_{\widetilde P}$, any parabolic subgroup $Q$ containing ${}^{w_{\widetilde P}}\!P$, any isotypic component $C_Q$ of $\LLbar\vert_{Z_{M_{ Q}}}$, and one, or equivalently any by Lemma \ \ref{cons1bis}(ii), \ element \ $w_Q\in W(C_Q)$, \ writing \ $M_{P(C_Q)}= {\rm diag}(M_1,\ldots ,M_d)$ with $M_i\cong \GL_{n_i}$, the restriction of $w_{\widetilde P}(\Phi)$ to the set of subquotients of $w_{\widetilde P}(\Phi)^{-1}(C_Q)$ comes from $d$ bijections $w_{\widetilde P}(\Phi)_{w_Q,i}$ of partially ordered sets between the set of $M_i(K)$-subrepresentations of $\pi_i(C_Q)$ (where $\pi_i(C_Q)$ is as in (i)) and the set of good subrepresentations of $C_{w_Q,i}\vert_{({}^{w_Qw_{\widetilde P}}\!\widetilde P)_{Q,i}^{\gKQ}}$ (where $C_{w_Q,i}$ is the isotypic component of $\LLbar_i\vert_{Z_{M_{({}^{w_Q}\!Q)_i}}}$ with its $M_{({}^{w_Q}\!Q)_i}^{\gKQ}$-action in (\ref{facteursbis}) applied with $P$ there being $Q$) in the following sense: for any subquotient $\Pi'$ of $\Phi^{-1}(C_Q)$ of the form
\[\Pi'\cong \Ind_{P(C_Q)^-(K)}^{G(K)}\big((\pi'_1\otimes \cdots\otimes\pi'_d)\otimes (\omega^{-1}\circ\theta^{P(C_Q)})\big)\]
with $\pi'_i$ a subquotient of $\pi_i(C_Q)$, the good subquotient $w_{\widetilde P}(\Phi)(\Pi')$ of
\[C_Q\vert_{(w_{\widetilde P}\widetilde Pw_{\widetilde P}^{-1})^{\gKQ}}=C_Q\vert_{((w_{\widetilde P}\widetilde Pw_{\widetilde P}^{-1})\cap M_Q)^{\gKQ}}\]
corresponds via Lemma \ref{cons1bis}(i) and Remark \ref{subquotient} to the following algebraic representation of $({}^{w_Qw_{\widetilde P}}\!\widetilde P)_{Q}^{\gKQ}=\prod_{i=1}^d({}^{w_Qw_{\widetilde P}}\!\widetilde P)_{Q,i}^{\gKQ}$:
\[\bigotimes_{i=1}^d \Big(w_{\widetilde P}(\Phi)_{w_Q,i}(\pi'_i)\otimes \!\big(\underbrace{(\theta^{P(C_Q)})_i\otimes \cdots\otimes (\theta^{P(C_Q)})_i}_{\gKQ}\big)\Big);\]
\item (\textbf{supersingular}) for any isotypic component $C_P$ of $\LLbar\vert_{Z_{M_{P}}}$, the (absolutely irreducible) $M_{P(C_P)}(K)$-repre\-sentation $\pi(C_P)$ of (\ref{forme}) is supersingular (cf.\ \cite[Def.4.7,\ Def.9.12,\ Cor.9.13]{He2}).
\end{enumerate}
\end{definit}

If $(\Pi,\Phi)$ is as in Definition \ref{compatible1}, then we have in particular $\Phi(\Pi)=\LLbar$ and $w_{\widetilde P}(\Phi)_{w_Q,i}(\pi_i(C_Q))=C_{w_Q,i}$. If $\widetilde P=G$, then $\Pi$ is compatible with $\widetilde P$ if and only if $\Pi$ is absolutely irreducible supersingular. Also it is clear from Definition \ref{compatible1} that, for a fixed $w_{\widetilde P}\in W_{\widetilde P}$, $\Pi$ is compatible with $\widetilde P$ if and only if $\Pi$ is compatible with $w_{\widetilde P}{\widetilde P}w_{\widetilde P}^{-1}$ (replace $\Phi$ by $w_{\widetilde P}(\Phi)$).\\

\begin{rem}\label{comments}
(i) In Definition \ref{compatible1}, we have used Lemma \ref{cons1} everywhere (see Remark \ref{applied}(ii)(iii)). In Definition \ref{compatible1}(iii), we have used Remark \ref{subquotient}. Also, Definition \ref{compatible1} is somewhat redundant since a parabolic subgroup $Q$ can contain ${}^{w_{\widetilde P}}\!P$ for several $w_{\widetilde P}\in W_{\widetilde P}$, but we found it too tedious to make it ``non-redundant''.\\
(ii) The representations $\pi(C_Q)$ and $\pi_i(C_Q)$ in Definition \ref{compatible1}(i) are uniquely defined since there are no nontrivial intertwinings between parabolic inductions (by \cite{emerton-ordI}).\\
(iii) When $Q={}^{w_{\widetilde P}}\!P$, $\pi(C_{{}^{w_{\widetilde P}}\!P})$ in (\ref{forme}) is absolutely irreducible, and is thus automatically of the form $\pi(C_{{}^{w_{\widetilde P}}\!P})\cong \pi_1(C_{{}^{w_{\widetilde P}}\!P})\otimes \cdots \otimes \pi_d(C_{{}^{w_{\widetilde P}}\!P})$. It is then not difficult to deduce from this, together with Lemma \ref{inclusionP} and \cite{emerton-ordI} (and the properties of $\Phi$), that each $\pi_i(C_Q)$ as in (\ref{forme}) has distinct (absolutely) irreducible constituents and that each irreducible constituent of (\ref{forme}) is of the form $\Ind_{P(C_Q)^-(K)}^{G(K)}\big((\pi'_1\otimes \cdots\otimes\pi'_d)\otimes (\omega^{-1}\circ\theta^{P(C_Q)})\big)$, where $\pi'_i$ is an irreducible constituent of $\pi_i(C_Q)$. This also justifies the terminology ``comes from $d$ bijections $w_{\widetilde P}(\Phi)_{w_Q,i}$'' in Definition \ref{compatible1}(iii).\\
(iv) It is in fact possible that Definition \ref{compatible1}(i) for parabolic subgroups $Q$ {\it strictly} containing some ${}^{w_{\widetilde P}}\!P$ and Definition \ref{compatible1}(iii) both {\it automatically follow} from the other conditions in Definition \ref{compatible1}. See for instance how the results of \cite{Ha2} are used in Example $2$, Example $4$, Example $5$ and Example $6$ of \S\ref{exemples5} below to show that several conditions of Definition \ref{compatible1} are automatic in special cases.\\
(v) In Definition \ref{compatible1}(iii), we have to use some element $w_Q$ of $W(C_Q)$ and ``pass through $w_Q(C_Q)$'' because of Remark \ref{ctrex}(ii) (see also the end of Remark \ref{gln2}). Nothing in here and what follows depends on the choice of such a $w_Q$.\\
(vi) For a given $\Pi$ compatible with $\widetilde P$, a bijection $\Phi$ as in Definition \ref{compatible1} is not unique in general (consider the case $\widetilde P=M_P$).\\
(vii) In Definition \ref{compatible1}, it is necessary in general to consider {\it all} elements $w_{\widetilde P}\in W_{\widetilde P}$, note just $w_{\widetilde P}=1$, otherwise one misses some condition, see for instance (\ref{wptildenon1}) below (note that this is also quite natural in view of Theorem \ref{choicegood}).
\end{rem}

\begin{ex}\label{exemplester}
Let us consider the case $n=3$, $K=\Qp$ and $\widetilde P=P$ with $M_P=\GL_2\times \GL_1$ in the last part of Example \ref{exemples}(ii) (see also Example \ref{exemplesbis}). We denote by $P'$ the standard parabolic subgroup of Levi $\GL_1\times \GL_2$. Then $\Pi$ is compatible with $\widetilde P$ if and only $\Pi$ has $3$ irreducible constituents and the following form (a line means a {\it nonsplit} extension of length $2$ as a subquotient and the constituent on the left-hand side is the {\it socle}):
\[\begin{xy}(0,0)*+{\Ind_{P^-(\Qp)}^{\GL_3(\Qp)}\big(\pi\cdot (\omega^{-1}\circ{\det})\otimes \chi\big)}; (41,0)*+{\rm SS} **\dir{-} ; (74,0)*+{\Ind_{{P'}^-(\Qp)}^{\GL_3(\Qp)}\big(\chi\omega^{-2}\otimes \pi\big)}**\dir{-} \end{xy}\]
where $\chi:\Qp^\times \rightarrow \F^\times$ is a smooth character, $\pi$ is a supersingular representation of $\GL_2(\Qp)$ and SS is a supersingular representation of $\GL_3(\Qp)$. The case $\widetilde P=M_P$ is analogous but with a semisimple $\Pi$ (instead of nonsplit extensions). See also \S\ref{exemples5} below for more examples.
\end{ex}

The following proposition shows that a representation $\Pi$ as in Definition \ref{compatible1} has internal symmetries.

\begin{prop}\label{cons2}
Assume $\Pi$ is compatible with $\widetilde P$ and let $\Phi$ be a bijection as in Definition \ref{compatible1}. Let $w_{\widetilde P}\in W_{\widetilde P}$, $Q$ a parabolic subgroup containing ${}^{w_{\widetilde P}}\!P$ and $C_Q$ an isotypic component of $\LLbar\vert_{Z_{M_{Q}}}$ such that $P(C_Q)={}^{w_Q}Q$ for some {\upshape(}unique{\upshape)} $w_Q\in W$ with $w_Q(S(Q))\subseteq S$. Then $\pi_i(C_Q)$ is compatible with $({}^{w_Qw_{\widetilde P}}\!\widetilde P)_{Q,i}$ for $i\in \{1,\ldots,d\}$, where $\pi_i(C_Q)$ is as in Definition \ref{compatible1}(i).
\end{prop}
\begin{proof}
The proof is long but essentially formal. Replacing $\widetilde P$ by ${w_{\widetilde P}}\widetilde Pw_{\widetilde P}^{-1}$ and $\Phi$ by $w_{\widetilde P}(\Phi)$ (see the discussion following Definition \ref{compatible1}), we can assume $w_{\widetilde P}=\Id$. We write for simplicity $w$ instead of $w_Q$. Recall \ from \ Proposition \ \ref{minimal} \ that \ $C_Q$ \ is \ the \ isotypic \ component of $fw^{-1}(\theta_G)\vert_{Z_{M_Q}}$ in $\LLbar\vert_{Z_{M_Q}}$. More precisely, by (\ref{wlp}), Corollary \ref{decomp} and Remark \ref{gln2} (especially (\ref{facteursbis})), we have an isomorphism of algebraic representations of $M_{{}^wQ}^{\gKQ}\cong\prod_{i=1}^dM_i^{\gKQ}\cong \prod_{i=1}^d\GL_{n_i}^{\gKQ}$:
\begin{equation}\label{rappel}
w(C_Q)\cong \LLbar_{{}^wQ}\otimes \big(\theta^{{}^wQ}\otimes \cdots \otimes \theta^{{}^wQ}\big)\cong \bigotimes_{i=1}^d\Big(\LLbar_i\otimes \big((\theta^{{}^wQ})_i\otimes \cdots \otimes (\theta^{{}^wQ})_i\big)\Big).
\end{equation}
Thus the map $\Phi_{w,i}$ in Definition \ref{compatible1}(iii) (recall $w_{\widetilde P}=\Id$ and $w=w_Q$) is a bijection of partially ordered sets between the set of $M_i(K)$-subrepresentations of $\pi_i(C_Q)$ and the set of good subrepresentations of $C_{w,i}\vert_{({}^w\widetilde P)_{Q,i}^{\gKQ}}=\LLbar_i\vert_{({}^w\widetilde P)_{Q,i}^{\gKQ}}$ (recall that $({}^wP)_{Q,i}$ is here a standard parabolic subgroup of $M_i$ and $({}^w\widetilde P)_{Q,i}$ a Zariski closed subgroup of $({}^wP)_{Q,i}$ containing $M_{({}^wP)_{Q,i}}$). We have to check that $\Phi_{w,i}$ satisfies conditions (i) to (iv) in Definition \ref{compatible1} (with $M_i$ instead of $G$ and $({}^w\widetilde P)_{Q,i}$ instead of $\widetilde P$). We will only check condition (i) below, leaving the others, which are again essentially formal, to the (motivated) reader.

We can assume $i=1$. Let $P_1\defeq ({}^wP)_{Q,1}$, $\widetilde P_1\defeq ({}^w\widetilde P)_{Q,1}$ (so $M_{P_1}\subseteq \widetilde P_1\subseteq P_1\subseteq M_1\subseteq M_{{}^wQ}$) and recall that $T_1$ is the torus of diagonal matrices in $M_1$. Let $w_{\widetilde P_1}\in W_{\widetilde P_1}\subseteq W(M_1)$, $Q_1$ a parabolic subgroup of $M_1$ containing ${}^{w_{\widetilde P_1}}\!P_1$ and $C_{Q_1}$ an isotypic component of $\LLbar_1\vert_{Z_{M_{Q_1}}}$, we have to prove that $w_{\widetilde P_1}(\Phi_{w,1})^{-1}(C_{Q_1})$ is of the form (\ref{forme}).

\noindent
Step 1: Let $\widetilde w_1\defeq w_{\widetilde P_1}\times \Id \times \cdots \times \Id\in W(M_1)\times \cdots \times W(M_d)=W({}^wQ)\subseteq W$ and set $w_{\widetilde P}\defeq w^{-1}\widetilde w_1w\in W(Q)$. Then $w_{\widetilde P}\in W_{\widetilde P}$ and $Q$ contains ${}^{w_{\widetilde P}}\!P$. Indeed, since $w_{\widetilde P_1}\in W_{\widetilde P_1}$ and the simple roots of $P_1$ are contained in $w(S(Q))\subseteq S$, we see that $w_{\widetilde P}=w^{-1}\widetilde w_1w$ sends the simple (resp.\ positive) roots of $\widetilde P\cap M_Q$ to simple (resp.\ positive) roots of $M_Q$ and the roots of $\widetilde P\cap N_Q$ to positive roots (using that $W(Q)$ normalizes $N_Q$). Moreover, one easily checks that ${}^{w_{\widetilde P_1}}\!P_1=({}^{ww_{\widetilde P}}\!P)_{Q,1}=({}^{w}({}^{w_{\widetilde P}}\!P))_{Q,1}$. Replacing $P$ by ${}^{w_{\widetilde P}}\!P$ and $\Phi$ by $w_{\widetilde P}(\Phi)$, we can thus assume $w_{\widetilde P_1}=\Id$.

\noindent
Step 2: Let $\lambda_1\in X(T_1)$ be a weight of $\LLbar_1\vert_{T_1}$ such that $C_{Q_1}$ is the isotypic component of $\lambda_1\vert_{Z_{M_{Q_1}}}$ and recall that $\lambda_1\vert_{Z_{M_1}}=f\theta_{M_1}\vert_{Z_{M_1}}=f\theta_{{}^wQ}\vert_{Z_{M_1}}$, where $\theta_{M_i}$ for $i\in \{1,\dots,d\}$ is defined as in (\ref{thetag}) replacing $G=\GL_n$ by $M_i=\GL_{n_i}$. Let $\lambda_{{}^wQ}\in X(T)$ be the unique character such that $\lambda_{{}^wQ}\vert_{T_1}=\lambda_1$ and $\lambda_{{}^wQ}\vert_{T_i}=f\theta_{M_i}=f\theta_{{}^wQ}\vert_{T_i}$ if $i\in \{2,\ldots,d\}$ (here, we use the convention in Remark \ref{gln2} and recall that $\theta_{M_i}$ is trivial if $M_i=\GL_1$). Then $\lambda_{{}^wQ}$ is a weight of $\bigotimes_{i=1}^d\LLbar_i\vert_{T_i}$. We set
\[\lambda\defeq  \lambda_{{}^wQ}+f\theta^{{}^wQ}\in X(T)\]
which is a weight of $\LLbar\vert_T$ (use (\ref{rappel})). We have
\begin{multline}\label{restr1}
\lambda\vert_{Z_{M_1}}=\lambda_1\vert_{Z_{M_1}}+f\theta^{{}^wQ}\vert_{Z_{M_1}}=f\theta_{M_1}\vert_{Z_{M_1}}+f\theta^{{}^wQ}\vert_{Z_{M_1}}\\
=f(\theta_{{}^wQ}+ \theta^{{}^wQ})\vert_{Z_{M_1}}=f\theta_G\vert_{Z_{M_1}}
\end{multline}
and if $i\geq 2$:
\begin{equation}\label{restrti}
\lambda\vert_{T_i}=f\theta_{M_i}+ f\theta^{{}^wQ}\vert_{T_i}= f(\theta_{{}^wQ}+ \theta^{{}^wQ})\vert_{T_i}=f\theta_G\vert_{T_i}.
\end{equation}
In particular $\lambda\vert_{Z_{M_{{}^wQ}}}=f\theta_G\vert_{Z_{M_{{}^wQ}}}$ and thus
\begin{equation}\label{deflambda}
w^{-1}(\lambda)\vert_{Z_{M_{Q}}}=fw^{-1}(\theta_G)\vert_{Z_{M_{Q}}}.
\end{equation}
Let $Q_{(1)}\subseteq {Q}$ be the standard parabolic subgroup of $G$ such that ${}^wQ_{(1)}\subseteq {}^wQ$ has Levi $M_{Q_1}\times M_2\times \cdots \times M_d$. As $P_1\subseteq Q_1$ by Step 1, we note that ${}^wQ_{(1)}$ contains ${}^wP$ and hence $Q_{(1)}$ contains $P$, $W({}^wQ_{(1)})=W(Q_1)\times W(M_2)\times \cdots \times W(M_d)$ and $w(S(Q_{(1)}))=S(Q_1)\amalg S(M_2)\amalg \cdots \amalg S(M_d)$. Let $C_{Q_{(1)}}$ be the isotypic component of $\LLbar\vert_{Z_{M_{Q_{(1)}}}}$ associated to $w^{-1}(\lambda)\vert_{Z_{M_{Q_{(1)}}}}$. From (\ref{deflambda}) we get $C_{Q_{(1)}}\subseteq C_Q$ (inside $\LLbar\vert_{Z_{M_{Q_{(1)}}}}$) and from (\ref{restr1}), (\ref{restrti}) an isomorphism of algebraic representations of $M_{Q_1}^{\gKQ}\otimes \prod_{i=2}^dM_i^{\gKQ}$:
\begin{equation}\label{decompc}
w(C_{Q_{(1)}})\cong \Big(C_{Q_1}\otimes \big((\theta^{{}^wQ})_1\otimes \cdots \otimes (\theta^{{}^wQ})_1\big)\Big)\otimes \bigotimes_{i=2}^d\Big(\LLbar_i\otimes \big((\theta^{{}^wQ})_i\otimes \cdots \otimes (\theta^{{}^wQ})_i\big)\Big).
\end{equation}

\noindent
Step 3: Define $\lambda'$, $\lambda'_{{}^wQ}$ and $\theta_G'$ by the formula (\ref{l'}) for $P={}^wQ_{(1)}$ and the respective characters $\lambda$, $\lambda_{{}^wQ}$ and $\theta_G$. Set $\lambda'_1\defeq \frac{1}{\vert W(Q_1)\vert}\sum_{w_1'\in W(Q_1)}w_1'(\lambda_1)\in (X(T_1)\otimes_{\Z}\Q)^{W(Q_1)}$. From (the proof of) Lemma \ref{utile}, we easily get $\lambda'=\lambda_{{}^wQ}'+f\theta^{{}^wQ}$ with $\lambda_{{}^wQ}'\vert_{T_1}=\lambda'_1$. Let $w_1\in W(M_1)$ such that $w_1(S(Q_1))\subseteq S(M_1)$ and $w_1(\lambda'_1)$ is dominant ($w_1$ exists by Proposition \ref{parabolicprop}(i)). We prove that $w_1(\lambda')=w_1(\lambda'_{{}^wQ})+f\theta^{{}^wQ}$ is also dominant (we consider here $w_1$ as an element of $W({}^wQ)$ in the obvious way and use that $W({}^wQ)$ acts trivially on $\theta^{{}^wQ}$). From (\ref{restrti}) we easily get $\lambda'\vert_{T_i}=f\theta'_G\vert_{T_i}$ if $i\geq 2$. But $\theta'_G$ is dominant since $\theta_G$ is (see the proof of Lemma \ref{parabolic}(i)), thus $\langle w_1(\lambda'),\alpha\rangle=\langle \lambda',\alpha\rangle= \langle f\theta'_G,\alpha\rangle\geq 0$ if $\alpha\in \{e_j-e_{j+1} : n_1+1\leq j\leq n-1\}$. Since $w_1(\lambda'_{{}^wQ})\vert_{T_1}=w_1(\lambda'_1)$ is dominant by assumption and $\langle f\theta^{{}^wQ},\alpha\rangle=0$ if $\alpha\in \{e_j-e_{j+1} : 1\leq j\leq n_1-1\}$ (see after (\ref{thetaP})), we are left to check that $\langle w_1(\lambda'),e_{n_1}-e_{n_1+1}\rangle\geq 0$. But an explicit computation gives
\begin{eqnarray*}
\langle w_1(\lambda'),e_{n_1}-e_{n_1+1}\rangle&=& \langle w_1(\lambda'_{{}^wQ}),e_{n_1}-e_{n_1+1}\rangle + \langle f\theta^{{}^wQ},e_{n_1}-e_{n_1+1}\rangle\\
&=&\langle w_1(\lambda'_{{}^wQ}),e_{n_1}\rangle-\langle w_1(\lambda'_{{}^wQ}),e_{n_1+1}\rangle+fn_2\\
&=&\langle w_1(\lambda'_{{}^wQ}),e_{n_1}\rangle -f\frac{n_2-1}{2}+ fn_2\\
&\geq & f\frac{n_2+1}{2},
\end{eqnarray*}
where the last inequality follows $\langle w_1(\lambda'_{{}^wQ}),e_{n_1}\rangle\geq 0$ by Remark \ref{explicit}(ii) applied to $\LLbar_1\vert_{T_1}$ (instead of $\LLbar\vert_T$) together with formula (\ref{l'}).

\noindent
Step 4: By definition, $S(P(C_{Q_1}))$ is the support of $f\theta_{M_1}-w_1(\lambda'_1)$ (see Proposition \ref{parabolicprop}(ii)). By Remark \ref{lambda'}(ii) we have $w^{-1}(\lambda')=(w^{-1}(\lambda))'$ in $(X(T)\otimes_{\Z}\Q)^{W(Q_{(1)})}$, where the latter is given by (\ref{l'}) applied to the parabolic $Q_{(1)}$ and the character $w^{-1}(\lambda)$. Since $w_1w(S(Q_{(1)}))\subseteq S$ and $w_1w((w^{-1}(\lambda))')=w_1w(w^{-1}(\lambda'))=w_1(\lambda')$ is dominant (Step $4$), $S(P(C_{Q_{(1)}}))$ is by definition the support of
\begin{multline}\label{support1}
f\theta_G-w_1(\lambda')=f\theta_G-\big(w_1(\lambda'_{{}^wQ})+f\theta^{{}^wQ}\big)=f\theta_{{}^wQ}-w_1(\lambda'_{{}^wQ})\\
=(f\theta_{M_1}-w_1(\lambda'_1))+\sum_{i=2}^d(f\theta_{M_i}-f\theta'_{M_i}),
\end{multline}
where $\theta'_{M_i}$ is defined by (\ref{l'}) applied to $P=M_i=G$ and the character $\theta_{M_i}$ of $T_i$. In fact, $\theta'_{M_i}$ is the character ${\det}^{\frac{n_i-1}{2}}$ of $T_i$, from which we easily see that the support of (\ref{support1}) is exactly $S(P(C_{Q_1}))\amalg S(M_2)\amalg \cdots \amalg S(M_d)$. This implies
\begin{equation}\label{levip}
M_{P(C_{Q_{(1)}})}={\rm diag}(M_{P(C_{Q_1})},M_2,\ldots , M_d).
\end{equation}

\noindent
Step 5: We now finally prove that $\Phi_{w,1}$ satisfies condition (i) in Definition \ref{compatible1}. Write $M_{P(C_{Q_1})}=M_{1,1}\times \cdots \times M_{1,d_1}$ (for some $d_1\geq 1$), by condition (i) in Definition \ref{compatible1} for the map $\Phi$ we have using (\ref{levip}):
\begin{equation}\label{phi1}
\Phi^{-1}(C_{Q_{(1)}})\cong \Ind_{P(C_{Q_{(1)}})^-(K)}^{G(K)}\Big(\big(\pi_1(C_{Q_{(1)}})\otimes \cdots \otimes \pi_d(C_{Q_{(1)}})\big)\otimes (\omega^{-1}\circ\theta^{P(C_{Q_{(1)}})})\Big),
\end{equation}
where $\pi_1(C_{Q_{(1)}})=\pi_{1,1}(C_{Q_{(1)}})\otimes \cdots \otimes \pi_{1,d_1}(C_{Q_{(1)}})$ (with obvious notation). Let
\begin{equation}\label{pi'1}
\pi'_1\defeq \Ind_{P(C_{Q_1})^-(K)}^{M_{1}(K)}\big(\pi_1(C_{Q_{(1)}})\otimes (\omega^{-1}\circ\theta^{P(C_{Q_{1}})})\big),
\end{equation}
it is enough to prove that $\pi'_1$ is a subquotient of $\pi_1(C_Q)$ and that
\[\Phi_{w,1}(\pi'_1)=C_{Q_1}\vert_{{\widetilde P_1}^{\gKQ}}\ (=C_{Q_1}\vert_{({\widetilde P_1}\cap M_{Q_1})^{\gKQ}}).\]
Note first that
\begin{equation}\label{theta1}
\theta^{P(C_{Q_{(1)}})}=\theta^{P(C_Q)} + \theta^{P(C_{Q_{1}})},
\end{equation}
where we view $\theta^{P(C_{Q_{1}})}$ as a character of $T$ (not just $T_1$) by sending the coordinates in $T_i$ to $1$ for $i\geq 2$ (this is straightforward to check from (\ref{thetaP})). From (\ref{phi1}), (\ref{pi'1}) and (\ref{theta1}), we get
\begin{equation*}
\Phi^{-1}(C_{Q_{(1)}})\cong \Ind_{P(C_{Q})^-(K)}^{G(K)}\Big(\big(\pi'_1\otimes \pi_2(C_{Q_{(1)}})\otimes \cdots \otimes \pi_d(C_{Q_{(1)}})\big)\otimes (\omega^{-1}\circ\theta^{P(C_{Q})})\Big).
\end{equation*}
Since $C_{Q_{(1)}}$ is a subquotient of $C_Q$ (both being good subquotients of $\LLbar\vert_{\widetilde P^{\gKQ}}$), $\Phi^{-1}(C_{Q_{(1)}})$ is a subquotient of $\Phi^{-1}(C_{Q})$. This implies in particular (using the ordinary functor of \cite{emerton-ordI} together with Remark \ref{comments}(iii)) that $\pi'_1$ (resp.\ $\pi_i(C_{Q_{(1)}})$ for $i\geq 2$) is a subquotient of $\pi_1(C_Q)$ (resp.\ of $\pi_i(C_{Q})$ for $i\geq 2$). By condition (iii) for $\Phi$ (in Definition \ref{compatible1}) applied to $\Pi'=\Phi^{-1}(C_{Q_{(1)}})$ (together with $P(C_Q)={}^wQ$), we also get an isomorphism of algebraic representations of $\prod_{i=1}^d({}^w\widetilde P)_{Q,i}^{\gKQ}$ over $\F$:
\begin{multline}\label{decompw}
w(C_{Q_{(1)}})=\Big(\Phi_{w,1}(\pi'_1)\otimes \!\big((\theta^{{}^wQ})_1\otimes \cdots\otimes (\theta^{{}^wQ})_1\big)\Big)\otimes \\
\bigotimes_{i=2}^d \Big(\Phi_{w,i}(\pi_i(C_{Q_{(1)}}))\otimes \!\big((\theta^{{}^wQ})_i\otimes \cdots\otimes (\theta^{{}^wQ})_i\big)\Big),
\end{multline}
where $\Phi_{w,1}(\pi'_1)$ and $\Phi_{w,i}(\pi_i(C_{Q_{(1)}}))$ ($i\geq 2$) are good subquotients of $\LLbar_i\vert_{({}^w\widetilde P)_{Q,i}^{\gKQ}}$. Since we have good subquotients of $\LLbar_i\vert_{({}^w\widetilde P)_{Q,i}^{\gKQ}}$ in each factor of (\ref{decompc}) and (\ref{decompw}), (\ref{decompc}) and (\ref{decompw}) imply $\Phi_{w,1}(\pi'_1)=C_{Q_1}\vert_{{\widetilde P_1}^{\gKQ}}$ and $\Phi_{w,i}(\pi_i(C_{Q_{(1)}}))=\LLbar_i\vert_{({}^w\widetilde P)_{Q,i}^{\gKQ}}$ for $i\geq 2$ (recall isotypic components of $\LLbar_i\vert_{M_{({}^w P)_{Q,i}}^{\gKQ}}$ tautology occur with multiplicity $1$, so there is no multiplicity issue). This finishes the proof of condition (i) in Definition \ref{compatible1} for $\Phi_{w,1}$.
\end{proof}

\begin{rem}
When $P(C_Q)$ is strictly bigger than ${}^{w_Q}Q$ for one, or equivalently any by Lemma \ref{single}, $w_Q\in W(C_Q)$, there is no real analogue of Proposition \ref{cons2} since $\LLbar_i$ has to be replaced by $C_{w_Q,i}$ in (\ref{facteursbis}) which is not $\LLbar_i$ in general.
\end{rem}

\subsubsection{Compatibility with \texorpdfstring{$\rhobar$}{rhobar}}\label{compatibilityrhobar}

We define what it means for a representation of $G(K)$ over $\F$ to be compatible with a good conjugate $\rhobar:\gK\rightarrow \widetilde P_{\rhobar}(\F)$ as in \S\ref{goodconjugatesub}. Essentially, an admissible smooth representation $\Pi$ is compatible with $\rhobar$ if it is compatible with $\widetilde P_{\rhobar}$ in the sense of Definition \ref{compatible1} and if the bijection $\Phi$ of {\it loc.cit.}\ satisfies some natural compatibilities with the functor $V_G$ in (\ref{VG}) (see Definition \ref{compatible2}).

We now fix a continuous homomorphism
\[\rhobar:\gK\longrightarrow G(\F)\]
and recall that $\rhobar^{\rm ss}$ denotes the semisimplification of the associated representation of $\gK$ (see \S\ref{goodconjugatesub}). We assume that $\rhobar$ is {\it generic} in the following sense: 
\begin{enumerate}[(a)]
\item $\rhobar^{\rm ss}$ has distinct irreducible constituents;
\item the ratio of any two irreducible constituents of $\rhobar^{\rm ss}$ of dimension $1$ is not in $\{\omega,\omega^{-1}\}$.
\end{enumerate}
By Proposition \ref{minP}, conjugating $\rhobar$ by an element of $G(\F)$ if necessary, we can assume that $\rhobar$ is a good conjugate in the sense of Definition \ref{gooddef}, that is we have
\[\rhobar:\gK\longrightarrow \widetilde P_{\rhobar}(\F)\subseteq P_{\rhobar}(\F)\subseteq G(\F),\]
where $P_{\rhobar}$ is a standard parabolic subgroup of $G$ such that $\rhobar^{\rm ss}$ is given by the composition $\gK\buildrel \rhobar\over \longrightarrow P_{\rhobar}(\F)\twoheadrightarrow M_{P_{\rhobar}}(\F)$ (see (\ref{morprho})), $\widetilde P_{\rhobar}\subseteq P_{\rhobar}$ is the smallest closed algebraic subgroup of $P_{\rhobar}$ containing $M_{P_{\rhobar}}$ and the $\rhobar(g)$ for $g\in \gK$ (in its $\F$-points), and where, for any $h\in P_{\rhobar}(\F)$, if we define $\widetilde P_{h\rhobar h^{-1}}\subseteq P_{\rhobar}$ as for $\rhobar$, then we have $\widetilde P_{\rhobar}\subseteq \widetilde P_{h\rhobar h^{-1}}$. Good conjugates are not unique, see Theorem \ref{choicegood}, but we fix such a good conjugate $\rhobar$ (and the associated pair $(\widetilde P_{\rhobar}, P_{\rhobar})$) for the moment. 

For any $\widetilde w\in W_{\rhobar}=W_{\widetilde P_{\rhobar}}$ (see (\ref{wrhobar})) and any parabolic subgroup $Q$ containing ${}^{\widetilde w}\!P_{\rhobar}$, we define the $Q$-semisimplification $\rhobar^{Q-\rm ss}$ of $\rhobar$ as the continuous homomorphism
\[\rhobar^{Q-\rm ss}:\gK\buildrel{\widetilde w}\rhobar{\widetilde w}^{-1}\over\longrightarrow {}^{\widetilde w}\!P_{\rhobar}(\F)\hookrightarrow Q(\F)\twoheadrightarrow M_{Q}(\F)\]
(strictly speaking, it also depends on $\widetilde w$). More generally, for any $w\in W$ such that $w(S(Q))\subseteq S$, we define the continuous homomorphisms
\begin{equation*}
w(\rhobar^{Q-{\rm ss}}):\gK\buildrel{w}\rhobar^{Q-\rm ss}{w}^{-1}\over\longrightarrow {w}M_{Q}(\F){w}^{-1}=M_{{}^wQ}(\F)
\end{equation*}
and note that $w(\rhobar^{Q-\rm ss})$ actually takes values in
\[({}^{w\widetilde w}\!\widetilde P_{\rhobar})_{Q}(\F)\subseteq ({}^{w\widetilde w}\!P_{\rhobar})_{Q}(\F)\subseteq M_{{}^wQ}(\F)\]
(recall from the beginning of \S\ref{compatible1sec} that $({}^{w\widetilde w}\!P_{\rhobar})_{Q}=w ({}^{\widetilde w}\!P_{\rhobar}\cap M_Q)w^{-1}$ and $({}^{w\widetilde w}\!\widetilde P_{\rhobar})_{Q}=w \big((\widetilde w\widetilde P_{\rhobar}\widetilde w^{-1})\cap M_Q\big)w^{-1}$). 

Let $\widetilde w\in W_{\rhobar}$, $Q$ a parabolic subgroup containing ${}^{\widetilde w}\!P_{\rhobar}$, $w\in W$ such that $w(S(Q))\subseteq S$ and $Q'$ a parabolic subgroup containing ${}^w Q$. We write $M_{Q'}={\rm diag}(M_1,\ldots ,M_d)$ with $M_i\cong \GL_{n_i}$ and we set for $i\in \{1,\ldots,d\}$:
\begin{equation*}
w(\rhobar^{Q-{\rm ss}})_i:\gK\buildrel w(\rhobar^{Q-{\rm ss}}) \over\longrightarrow M_{{}^wQ}(\F)\hookrightarrow M_{Q'}(\F)\twoheadrightarrow M_{i}(\F).
\end{equation*}
We also have (recall from \S\ref{compatible1sec} that $({}^wQ)_i$ is a standard parabolic subgroup of $M_i$):
\begin{equation}\label{facteuri}
w(\rhobar^{Q-{\rm ss}})_i:\gK\buildrel w(\rhobar^{Q-{\rm ss}}) \over\longrightarrow M_{{}^wQ}(\F)\twoheadrightarrow M_{({}^wQ)_i}(\F)\hookrightarrow M_{i}(\F).
\end{equation}
Composing $w(\rhobar^{Q-{\rm ss}})_i$ with $M_i(\F)\twoheadrightarrow (M_i/M^{\rm der}_i)(\F)\cong \F^\times$, we obtain by class field theory for $K$ a continuous group homomorphism
\begin{equation}\label{deti}
{\det}(w(\rhobar^{Q-\rm ss})_i):K^\times \longrightarrow \F^\times.
\end{equation}

\begin{lem}\label{det}
Let $\rhobar$, $Q$ as above, $C_Q$ an isotypic component of $\LLbar\vert_{Z_{M_Q}}$ and $Q'\defeq P(C_Q)$. Then the characters {\upshape(\ref{deti})} for $i\in \{1,\ldots,d\}$ and $w\in W(C_Q)$ {\upshape(}see {\upshape(\ref{wCP}))} don't depend on the choice of $w\in W(C_Q)$. Moreover, we have $\prod_{i=1}^d{\det}(w(\rhobar^{Q-\rm ss})_i)=\det(\rhobar)$.
\end{lem}
\begin{proof}
This follows from Lemma \ref{inclusion} (applied to $P=Q$) together with the fact that conjugation by $W(P(C_Q))$ (seen in $M_{P(C_Q)}(\F)$) is trivial on $M_{P(C_Q)}/M^{\rm der}_{P(C_Q)}$, and thus on each $M_i/M^{\rm der}_i$. The last assertion is obvious.
\end{proof}

As previously, $w(\rhobar^{Q-{\rm ss}})_i$ in (\ref{facteuri}) takes values in
\[({}^{w\widetilde w}\widetilde P_{\rhobar})_{Q,i}(\F)\subseteq ({}^{w\widetilde w}P_{\rhobar})_{Q,i}(\F)\subseteq M_{({}^wQ)_i}(\F)\subseteq M_i(\F)\cong \GL_{n_i}(\F)\]
(recall from the beginning of \S\ref{compatible1sec} that $({}^{w\widetilde w}\!P_{\rhobar})_{Q,i}$ is a standard parabolic subgroup of $M_{({}^wQ)_i}$ and that $({}^{w\widetilde w}\!\widetilde P_{\rhobar})_{Q,i}$ is a Zariski closed algebraic subgroup of $({}^{w\widetilde w}\!P_{\rhobar})_{Q,i}$ containing $M_{({}^{w\widetilde w}\!P_{\rhobar})_{Q,i}}$).

\begin{prop}\label{consgood}
Let $\rhobar$, $Q$ as above, $w\in W$ such that $w(S(Q))\subseteq S$ and $Q'\defeq {}^wQ$. Then $w(\rhobar^{Q-{\rm ss}})_i: \gK\longrightarrow M_i(\F)$ is a good conjugate with values in $({}^{w\widetilde w}\!\widetilde P_{\rhobar})_{Q,i}(\F)$ for $i\in \{1,\ldots,d\}$.
\end{prop}
\begin{proof}
Note that ${\widetilde w}\rhobar \widetilde w^{-1}$ is a good conjugate (with values in $\widetilde w\widetilde P_{\rhobar}(\F)\widetilde w^{-1}\subseteq {}^{\widetilde w}\!P_{\rhobar}(\F)$) by Lemma \ref{goodw}. Since $w(\rhobar^{Q-{\rm ss}})$ is obtained from $\rhobar^{Q-{\rm ss}}$ by permuting the blocs $M_i\cong \GL_{n_i}$ of $M_Q$, it is equivalent to prove the statement for $w=\Id$. Assume that $\rhobar_i\defeq (\rhobar^{Q-{\rm ss}})_i: \gK\rightarrow M_i(\F)$ is {\it not} a good conjugate. Then it follows from Proposition \ref{minP} that there is $h_i\in ({}^{\widetilde w}\!P_{\rhobar})_{Q,i}(\F)$ such that $h_i \rhobar_i h_i^{-1}$ is a good conjugate, and thus $X_{h_i \rhobar_i h_i^{-1}}\subsetneq X_{\rhobar_i}$ (with the notation of \S\ref{goodconjugatesub}). Let $\alpha_i$ be a positive root of $\GL_{n_i}$ in $X_{\rhobar_i}\backslash X_{h_i\rhobar_ih_i^{-1}}$ and note that, if $\alpha_i$ is a sum of roots in $R^+$ (viewing $\alpha_i$ in $R^+$), then all of these roots are positive roots of $\GL_{n_i}$. Set $h_j\defeq \Id_{\GL_{n_j}}\in \GL_{n_j}(\F)$ if $j\ne i$ and define $h=(h_1,\ldots,h_d)\in {\rm diag}(M_1,\ldots ,M_d)=M_Q(\F)\subseteq Q(\F)$. If we had $\alpha_i\in X_{h\widetilde w\rhobar \widetilde w^{-1}h^{-1}}$, then from what we just said necessarily we would have $\alpha_i\in X_{(h\rhobar^{Q-{\rm ss}}h^{-1})_i}=X_{h_i\rhobar_ih_i^{-1}}$ which is impossible. Therefore $\alpha_i\notin X_{h\widetilde w\rhobar \widetilde w^{-1}h^{-1}}$. But since $\alpha_i\in X_{\rhobar_i}\subseteq X_{\widetilde w\rhobar\widetilde w^{-1}}$ (viewing the positive roots of $\GL_{n_i}$ as a subset of $R^+$) we deduce $X_{h\widetilde w\rhobar\widetilde w^{-1} h^{-1}}\subsetneq X_{\widetilde w\rhobar\widetilde w^{-1}}$ which is impossible as $\widetilde w\rhobar\widetilde w^{-1}$ is a good conjugate.
\end{proof}

For $\sigma\in \gKQ= {\rm Gal}(\Qpf/\Qp)$ consider
\[\rhobar^{\sigma}:\gK\rightarrow \widetilde P_{\rhobar}(\F)\subseteq P_{\rhobar}(\F)\subseteq G(\F),\]
where $\rhobar^{\sigma}(g)\defeq \rhobar(\sigma g \sigma^{-1})$. Here $g\in \gK$ and $ \sigma$ is any lift of $\sigma$ in $\gp$. Since $\gK$ is normal in $\gp$, $\rhobar^{\sigma}(g)$ is well defined up to conjugation (by elements in $\widetilde P_{\rhobar}(\F)$). If $C$ is a good subquotient of $\LLbar\vert_{{\widetilde P_{\rhobar}}^{\gKQ}}$ (Definition \ref{goodsubqt}), we can view in particular $C$ as a continuous homomorphism
\begin{equation}
\label{eq:Ccont}
\underbrace{\widetilde P_{\rhobar}(\F)\times \cdots \times \widetilde P_{\rhobar}(\F)}_{\gKQ}\longrightarrow {\rm Aut}\big(C(\F)\big)
\end{equation}
(denoting by $C(\F)$ the underlying $\F$-vector space of the algebraic representation $C$) and define a $\gK$-representation $C(\rhobar)$ as
\[\gK \ \buildrel {\prod \rhobar^{\sigma}}\over \longrightarrow \ \widetilde P_{\rhobar}(\F)\times \cdots \times \widetilde P_{\rhobar}(\F)\ \buildrel C\over\longrightarrow\ {\rm Aut}\big(C(\F)\big),\]
where, in the first arrow, we choose any order on the elements $\sigma$ of $\gKQ$.

\begin{lem}\label{galqp}
The $\gK$-representation $C(\rhobar)$ is well-defined up to isomorphism and canonically extends to a $\gp$-representation.
\end{lem}
\begin{proof}
The algebraic representation $C$ of $\widetilde P_{\rhobar}^{\gKQ}$ over $\F$ doesn't depend up to isomorphism on the order of the copies of $\widetilde P_{\rhobar}$, i.e.\ any permutation of the $\widetilde P_{\rhobar}$'s yields an algebraic representation which is conjugate by an element of ${\rm Aut}(C(\F))$. Indeed, this clearly holds when $C$ is an isotypic component of $\LLbar\vert_{Z_{M_{P_{\rhobar}}}}$ as $Z_{M_{P_{\rhobar}}}$ embeds diagonally into $\widetilde P_{\rhobar}^{\gKQ}$. Thus, for a general good subquotient $C$, any permutation of the $\widetilde P_{\rhobar}$'s gives a representation $C'$ which contains the same isotypic components of $\LLbar\vert_{Z_{M_{P_{\rhobar}}}}$ as those of $C$. Assume now that $C$ is a good subrepresentation of $\LLbar\vert_{{\widetilde P_{\rhobar}}^{\gKQ}}$. Then $C'$ must be isomorphic to $C$ since isotypic components of $\LLbar\vert_{Z_{M_{P_{\rhobar}}}}$ tautologically occur with multiplicity $1$. In general, one writes $C$ as the quotient of two good subrepresentations of $\LLbar\vert_{{\widetilde P_{\rhobar}}^{\gKQ}}$. All this implies that $C(\rhobar)$ is well-defined.\\
We now prove that it extends to $\gp$. First, if $C=\LLbar\vert_{{\widetilde P_{\rhobar}}^{\gKQ}}$, then $C(\rhobar)$ is the tensor induction (\ref{tensorind}) and thus canonically extends to $\gp$. Let us recall explicitly how it extends. Fix $\sigma_1,\ldots,\sigma_f$ some representatives in $\gp$ of the elements of $\gKQ={\rm Gal}(\Q_{p^f}/\Qp)$ and define permutations $w_1,\ldots,w_f$ on $\{1,\ldots,f\}$ by $\sigma_i\sigma_j^{-1}=\sigma_{w_i(j)}^{-1}h_{i,j}$, where $h_{i,j}\in \gKQ$. The underlying $\F$-vector space $\LLbar(\F)$ of $\LLbar$ is
\[\bigotimes_{i=1}^f\Big(\big(\bigotimes_{\alpha\in S}\Lbar(\lambda_{\alpha})\big)(\F)\Big),\]
where $\big(\bigotimes_{\alpha\in S}\Lbar(\lambda_{\alpha})\big)(\F)$ is the underlying vector space of $\bigotimes_{\alpha\in S}\Lbar(\lambda_{\alpha})$, and the action of $\sigma_i$ then sends $v_1\otimes v_2\otimes \cdots \otimes v_f\in \LLbar(\F)$ to $u_1\otimes u_2\otimes \cdots \otimes u_f$, where: \begin{equation}\label{tensoract}
u_{w_i(j)}\defeq \Big(\Big(\bigotimes_{\alpha\in S}\Lbar(\lambda_{\alpha})\Big)(\rhobar(h_{i,j}))\Big)(v_j).
\end{equation}
This yields an action of $\gp$ which doesn't depend on any choice (up to isomorphism). It is enough to prove that this action of $\gp$ preserves the subspaces $C(\F)\subseteq \LLbar(\F)$, where $C$ is any good subrepresentation of $\LLbar\vert_{{\widetilde P_{\rhobar}}^{\gKQ}}$. But this is clear from (\ref{tensoract}) since $C(\F)$ is preserved by the action of $\gK$ {\it and} by any permutation of the $v_i$ (as we have seen at the beginning).
\end{proof} 

\begin{rem}
One could also use $L$-groups as in \S\ref{Cgroup} in order to have more intrinsic definitions (see Remark \ref{explicit}(i)). However the above pedestrian approach will be sufficient for our purpose. 
\end{rem}

The following lemma is in the same spirit as Lemma \ref{det}.

\begin{lem}\label{d}
Let $\rhobar$, $Q$ as above, $C_Q$ an isotypic component of $\LLbar\vert_{Z_{M_Q}}$ and $Q'\defeq P(C_Q)$. For $w\in W(C_Q)$ and $i\in \{1,\ldots,d\}$, let
\begin{itemize}
\item[$\bullet$]$C_{w,i}$ be the isotypic component of $\LLbar_i\vert_{Z_{M_{({}^wQ)_i}}}$ defined in {\upshape(\ref{facteursbis}) (}applied with $P$ there being $Q${\upshape)};
\item[$\bullet$]$w(\rhobar^{Q-{\rm ss}})_i$ the representation of $\gK$ with values in $M_{({}^wQ)_i}(\F)$ defined in {\upshape(\ref{facteuri}) (}applied to $Q'=P(C_Q)${\upshape)};
\item[$\bullet$]$C_{w,i}\big(w(\rhobar^{Q-{\rm ss}})_i\big)$ the representation of $\gp$ defined in Lemma \ref{galqp} {\upshape(}applied to $\rhobar=w(\rhobar^{Q-{\rm ss}})_i$, $\LLbar_i$ and $C=C_{w,i}${\upshape)}.
\end{itemize}
Then the $\gp$-representation $C_{w,i}\big(w(\rhobar^{Q-{\rm ss}})_i\big)$ doesn't depend on $w\in W(C_Q)$.
\end{lem}
\begin{proof}
Let $w'$ be another element in $W(C_Q)$. Then $w'=w_{P(C_Q)}w$ with $w_{P(C_Q)}\in W(P(C_Q))$ by Lemma \ref{inclusion} (with $P$ there being $Q$). Since $w_{P(C_Q)}$ respects $M_i$, we have
\[w'(\rhobar^{Q-{\rm ss}})_i=w_{P(C_Q)}w(\rhobar^{Q-{\rm ss}})_iw_{P(C_Q)}^{-1}.\]
The result then follows from (\ref{isowi}) (applied with $P=Q$).
\end{proof}

\begin{rem}\label{dd}
Lemma \ref{d} still holds replacing $C_{w,i}$ by any good subquotient of $C_{w,i}\vert_{({}^{w\widetilde w}\!\widetilde P_{\rhobar})_{Q,i}^{\gKQ}}$ and using the proof of Lemma \ref{cons1bis}(ii) and Remark \ref{subquotient} to compare with the corresponding good subquotient of $C_{w',i}\vert_{({}^{w'\widetilde w}\!\widetilde P_{\rhobar})_{Q,i}^{\gKQ}}$. The proof is the same as for Lemma \ref{d} using that $w(\rhobar^{Q-{\rm ss}})_i$ takes values in $({}^{w\widetilde w}\!\widetilde P_{\rhobar})_{Q,i}(\F)$.
\end{rem}

We now state the second crucial definition. We use the functor $V_H$ defined in \S\ref{covariant} in the case $H=\GL_m$, $m\geq 1$ (with the convention of Example \ref{exdelta}). If a smooth representation $\pi$ of $H(K)$ has a central character, we denote it by $Z(\pi)$ (so writing $Z(\pi)$ in the sequel implicitly means that $\pi$ has a central character). We also define
\begin{equation}\label{detmi}
\omega^{-1}\circ\theta_{M_i}:Z_{M_i}(K)=K^\times\buildrel\theta_{M_i}\vert_{Z_{M_i}}\over\longrightarrow K^\times \buildrel \omega^{-1}\over \longrightarrow \Fp^\times\hookrightarrow \F^\times
\end{equation}
($\theta_{M_i}$ as in (\ref{thetag}) replacing $G$ by $M_i$).

\begin{definit}\label{compatible2}
An admissible smooth representation $\Pi$ of $G(K)$ over $\F$ which has finite length and distinct absolutely irreducible constituents is {\it compatible with $\rhobar$} if there exists a bijection $\Phi$ as in Definition \ref{compatible1} for $\widetilde P=\widetilde P_{\rhobar}$ (in particular $\Pi$ is compatible with $\widetilde P_{\rhobar}$) which satisfies the following extra conditions:
\begin{enumerate}
\item for any subquotient $\Pi'$ of $\Pi$, we have an isomorphism of $\gp$-representa\-tions over $\F$:
\begin{equation}\label{DD}
V_G(\Pi')\cong \Phi(\Pi')(\rhobar),
\end{equation}
where $\Phi(\Pi')(\rhobar)$ is the associated representation of $\gp$ defined in Lem\-ma \ref{galqp};\\
\item for any $\widetilde w\in W_{\rhobar}$, any parabolic subgroup $Q$ containing ${}^{\widetilde w}\!P_{\rhobar}$ and any isotypic component $C_Q$ of $\LLbar\vert_{Z_{M_{Q}}}$, writing $M_{P(C_Q)}= {\rm diag}(M_1,\ldots, M_d)$ with $M_i\cong \GL_{n_i}$ we have for one, or equivalently any, element $w\in W(C_Q)$ and for any subquotient $\pi'_i$ of $\pi_i(C_Q)$:
\begin{eqnarray}\label{coeur}
Z\big(\pi'_i\big) &\cong & \det(w(\rhobar^{Q-{\rm ss}})_i)\cdot \omega^{-1}\circ\theta_{M_i}\\
\nonumber V_{M_i}\big(\pi'_i\big) &\cong &\widetilde w(\Phi)_{w,i}(\pi'_i)\big(w(\rhobar^{Q-{\rm ss}})_i\big),
\end{eqnarray}
where
\begin{itemize}
\item $\pi_i(C_Q)$ is the admissible smooth representation of $M_i(K)$ over $\F$ in Definition \ref{compatible1}(i);
\item $\det(w(\rhobar^{Q-{\rm ss}})_i)$ (resp.\ $\omega^{-1}\circ\theta_{M_i}$) is the character of $K^\times$ defined in (\ref{deti}) (resp.\ in (\ref{detmi}));
\item $\widetilde w(\Phi)_{w,i}(\pi'_i)$ is the good subquotient of $C_{w,i}\vert_{({}^{w\widetilde w}\!\widetilde P_{\rhobar})_{Q,i}}$ defined in Definition \ref{compatible1}(iii);
\item $w(\rhobar^{Q-{\rm ss}})_i$ is the representation of $\gK$ with values in $({}^{w\widetilde w}\!\widetilde P_{\rhobar})_{Q,i}(\F)\!\subseteq M_{({}^wQ)_i}(\F)$ defined in (\ref{facteuri}) (applied to $Q'=P(C_Q)$);
\item $\widetilde w(\Phi)_{w,i}(\pi'_i)\big(w(\rhobar^{Q-{\rm ss}})_i\big)$ \ is \ the \ representation \ of \ $\gp$ \ defined in Lemma \ref{galqp} (applied to $\rhobar=w(\rhobar^{Q-{\rm ss}})_i$, $\LLbar_i$ and $C=\widetilde w(\Phi)_{w,i}(\pi'_i)$).
\end{itemize}
\end{enumerate}
\end{definit}

If \ $\Pi$ \ is \ compatible \ with \ $\rhobar$, \ then \ we \ have \ in \ particular \ $V_G(\Pi)\cong \LLbar(\rhobar)$ and $V_{M_i}\big(\pi_i(C_Q)\big) \cong C_{w,i}\big(w(\rhobar^{Q-{\rm ss}})_i\big)$ for $Q,w,i$ as in Definition \ref{compatible2}(ii) (recall that $V_{M_i}(\pi_i(C_Q))$ is always the trivial representation of $\gp$ when $n_i=1$). If $\rhobar$ is (absolutely) irreducible, then $\widetilde P_{\rhobar}=P_{\rhobar}=G$, $W_{\rhobar}=\{\Id\}$ and $\Pi$ is compatible with ${\rhobar}$ if and only if $\Pi$ is absolutely irreducible supersingular, $Z(\Pi)\cong \det(\rhobar)\cdot \omega^{-1}\circ(\theta_{G}\vert_{Z_G})$ and $V_G(\Pi)\cong \LLbar(\rhobar)$.

\begin{rem}\label{listrem}
(i) The isomorphisms in (\ref{coeur}) are consistent with Lemma \ref{det}, Lemma \ref{d} and Remark \ref{dd} since their left-hand sides don't depend on $w\in W(C_Q)$.\\
(ii) Let $\Pi$ be compatible with $\rhobar$. From (\ref{forme}) applied with $w_{\widetilde P}=1$ and $Q=P$, (\ref{coeur}) applied with $\widetilde w=1$ and $Q=P_{\rhobar}$, the last assertion in Lemma \ref{det}, and from
\[\theta_G\vert_{Z_G}=\theta^{P(C_{Q})}\vert_{Z_{G}}\theta_{P(C_{Q})}\vert_{Z_{G}}=\theta^{P(C_{Q})}\vert_{Z_{G}}\Big(\prod_{i=1}^d\theta_{M_i}\vert_{Z_{M_i}}\Big)\] 
(which follows from (\ref{thetaP})), we deduce that each irreducible constituent $\Pi'$ of $\Pi$ is such that $Z(\Pi')=\det(\rhobar)\cdot \omega^{-1}\circ(\theta_{G}\vert_{Z_G})$. Since these irreducible constituents are all distinct by assumption, we obtain that $\Pi$ has a central character $Z(\Pi)=\det(\rhobar)\cdot \omega^{-1}\circ(\theta_{G}\vert_{Z_G})=\det(\rhobar)\cdot \omega^{\frac{-n(n-1)}{2}}$.\\
(iii) Let $\Pi$ be compatible with $\rhobar$, $\Pi'$ a subquotient of $\Pi$ and $\Pi''\subseteq \Pi'$ a subrepresentation. Then from Remark \ref{applied}(i) we have an exact sequence of $\gp$-representations:
\[0\longrightarrow \Phi(\Pi'')(\rhobar)\longrightarrow \Phi(\Pi')(\rhobar)\longrightarrow \Phi(\Pi'/\Pi'')(\rhobar)\longrightarrow 0.\]
Thus (\ref{DD}) implies that the sequence $0\rightarrow V_G(\Pi'')\rightarrow V_G(\Pi')\rightarrow V_G(\Pi'/\Pi'')\rightarrow 0$ is exact. In other terms, when applied to $\Pi$ and its subquotients $V_G$ behaves like an exact functor.\\
(iv) Let $\chi: K^\times\rightarrow \F^\times$ be a smooth character. Then it easily follows from Remark \ref{trivial}(ii) that $\Pi$ is compatible with $\rhobar$ if and only if $\Pi\otimes (\chi\circ{\det})$ is compatible with $\rhobar\otimes \chi$.\\
(v) For a given $\Pi$ compatible with ${\rhobar}$, a bijection $\Phi$ as in Definition \ref{compatible2} is still not unique in general. For instance consider the case $n=4$, $K=\Qp$, $\widetilde P_{\rhobar}=M_{P_{\rhobar}}={\rm diag}(\GL_2,\GL_2)$ and $\rhobar = \rhobar_1 \oplus \rhobar_2$ with $\rhobar_i:\gp\rightarrow \GL_2(\F)$ absolutely irreducible distinct for $i=1,2$ but such that $\wedge^2_{\F}\rhobar_1\cong \wedge^2_{\F}\rhobar_2$.
\end{rem}

Definition \ref{compatible2} doesn't depend on the choice of a good conjugate.

\begin{prop}\label{any}
If $\rhobar':\gK\rightarrow \widetilde P_{\rhobar'}(\F)\subseteq P_{\rhobar'}(\F)$ is another good conjugate of $\rhobar$, then $\Pi$ is compatible with ${\rhobar}$ if and only if $\Pi$ is compatible with ${\rhobar'}$.
\end{prop}
\begin{proof}
From Theorem \ref{choicegood}, we have $\rhobar'=wh\rhobar h^{-1}w^{-1}$ for some $h\in \widetilde P_{\rhobar}(\F)$ and some $w\in W_{\rhobar}$. By symmetry, it is enough to prove that $\Pi$ compatible with ${\rhobar}$ implies $\Pi$ compatible with $\rhobar'$. We have first that $\Pi$ is compatible with $h\rhobar h^{-1}$. Indeed, $\widetilde P_{h\rhobar h^{-1}}=\widetilde P_{\rhobar}$ and the conditions in Definition \ref{compatible2} for $h\rhobar h^{-1}$ follow from the conditions for $\rhobar$ since $w(\rhobar^{Q-{\rm ss}})_i$ and $w((h\rhobar h^{-1})^{Q-{\rm ss}})_i$ are conjugate in $({}^{w\widetilde w}\!\widetilde P_{\rhobar})_{Q,i}(\F)$ (with $\widetilde w, w$ here as in Definition \ref{compatible2}). Thus we can assume $h=\Id$. But then, it is clear from Definition \ref{compatible2} that $\Pi$ is compatible with $\rhobar'=w\rhobar w^{-1}$.
\end{proof}

Just as some statements in Definition \ref{compatible1} should follow from others (see Remark \ref{comments}(iv)), we expect the isomorphisms (\ref{DD}) to follow in many cases from the isomorphisms (\ref{coeur}):

\begin{prop}\label{prop:cmpt2}
Assume $\Pi$ is compatible with $\rhobar$ and let $\Phi$ be a bijection as in Definition \ref{compatible2}. Let $\widetilde w\in W_{\rhobar}$, $Q$ a parabolic subgroup containing ${}^{\widetilde w}\!P_{\rhobar}$, $C_Q$ an isotypic component of $\LLbar\vert_{Z_{M_{Q}}}$ and $\Pi'$ a subquotient of $\widetilde w(\Phi)^{-1}(C_Q)$ of the form
\[\Pi'\cong \Ind_{P(C_Q)^-(K)}^{G(K)}\big((\pi'_1\otimes \cdots\otimes\pi'_d)\otimes (\omega^{-1}\circ\theta^{P(C_Q)})\big),\]
where $\pi'_i$ is a subquotient of the representation $\pi_i(C_Q)$ of $M_i(K)$ over $\F$ defined in Definition \ref{compatible1}(i) {\upshape(}so that $\widetilde w(\Phi)(\Pi')$ is a good subquotient of $C_Q\vert_{{}^{\widetilde w}\!\widetilde P_{\rhobar}^{\gKQ}}=C_Q\vert_{(({\widetilde w}\widetilde P_{\rhobar}{\widetilde w}^{-1})\cap M_Q)^{\gKQ}}${\upshape)}. Assume that $V_{M_{P(C_Q)}}(\pi'_1\otimes \cdots\otimes\pi'_d)\cong \bigotimes_{i=1}^dV_{M_{P(C_Q)},i}(\pi'_i)$ {\upshape(}with the notation used in Lemma \ref{parabtensor}{\upshape)}. Then the isomorphism {\upshape(\ref{DD})} for $\Pi'$ follows from the isomorphisms {\upshape(\ref{coeur})}.
\end{prop}
\begin{proof}
For $i\in \{1,\ldots,d\}$, we have (easy computation):
\begin{equation}\label{thetai}
(\theta^{P(C_Q)})_i={\det}^{n-\sum_{j=1}^in_j}.
\end{equation}
Let $\pi_i''\defeq \pi'_i\otimes (\omega^{-1}\circ{\det})^{n-\sum_{j=1}^in_j}$, we have by Lemma \ref{parabtensor}, (\ref{thetai}) and Remark \ref{trivial}(ii):
\begin{eqnarray*}
V_G(\Pi')&=&\displaystyle{V_G\Big(\Ind_{P(C_Q)^-(K)}^{G(K)}(\otimes_{i=1}^d\pi'_i)\otimes (\omega^{-1}\circ\theta^{P(C_Q)})\Big)}\\
&\cong &\bigg( \bigotimes_{i=1}^d \Big(V_{M_i}(\pi''_i)\otimes \big(Z(\pi''_i)^{n-\sum_{j=1}^in_j}\big)\vert_{\Qp^\times}\delta_{M_i}^{-1}\Big)\bigg)\otimes \delta_G\\
&\cong & \bigg(\bigotimes_{i=1}^d \Big(V_{M_i}(\pi'_i)\otimes \Big(\big(Z(\pi'_i)\cdot \omega\circ \theta_{M_i}\big)^{n-\sum_{j=1}^in_j}\Big)\vert_{\Qp^\times}\Big)\bigg)\otimes \delta,
\end{eqnarray*}
where $\delta\defeq \big(\delta_G\prod_{i=1}^d\delta_{M_i}^{-1}\big)\ind_K^{\otimes\Qp}\!(\omega^{-\sum_{i=1}^dc_i})$ with (by an explicit computation):
\begin{eqnarray}\label{ci}
\nonumber c_i&=&n_i(n_i-1)\Big(n-\sum_{j=1}^in_j\Big)+ n_i\Big(n-\sum_{j=1}^in_j\Big)^2\\
\nonumber &=&n_i\Big(n-\sum_{j=1}^in_j\Big)\Big(n_i-1+n-\sum_{j=1}^in_j\Big)\\
&=&n_i\Big(n-\sum_{j=1}^in_j\Big)\Big(n-1-\sum_{j=1}^{i-1}n_j\Big).
\end{eqnarray}
Now, assuming (\ref{coeur}) we have for one, or equivalently any, $w$ of $W(C_Q)$:
\begin{eqnarray*}
\Phi(\Pi')(\rhobar)&\cong & \widetilde w(\Phi)(\Pi')(\rhobar)\\
&= & \displaystyle{\widetilde w(\Phi)(\Pi')(\rhobar^{Q-\rm ss})}\\
&\cong &\displaystyle{\bigotimes_{i=1}^d} \bigg(\widetilde w(\Phi)_{w,i}(\pi'_i)\big(w(\rhobar^{Q-{\rm ss}})_i\big)\!\otimes \\
&&\ \ \ \ \ \ \ \ \ \ \ \ \ \ \ \ \ \ \ \ \Big(\!\big((\theta^{P(C_Q)})_i\otimes \cdots\otimes (\theta^{P(C_Q)})_i\big)\circ \big(\oplus_\sigma (w(\rhobar^{Q-{\rm ss}})_i)^{\sigma}\big)\Big)\bigg)\\
&\cong &\displaystyle{\bigotimes_{i=1}^d \Big(V_{M_i}(\pi'_i)\otimes \Big(\big(\det(w(\rhobar^{Q-{\rm ss}})_i)\big)^{n-\sum_{j=1}^in_j}\Big)\vert_{\Qp^\times}\Big)}\\
&\cong &\displaystyle{\bigotimes_{i=1}^d \Big(V_{M_i}(\pi'_i)\otimes \Big(\big(Z(\pi'_i)\cdot \omega\circ \theta_{M_i}\big)^{n-\sum_{j=1}^in_j}\Big)\vert_{\Qp^\times}\Big)},
\end{eqnarray*}
where the first isomorphism follows from $\rhobar\cong \widetilde w\rhobar\widetilde w^{-1}$, the second equality is obvious ($\widetilde w(\Phi)(\Pi')$ being a representation of $M_Q^{\gKQ}$ as it is a subquotient of $C_Q$), the second isomorphism follows from Definition \ref{compatible1}(iii), and the last two isomorphisms from (\ref{coeur}), (\ref{thetai}) and local class field theory for $\Qp$. So we have to prove $(\delta_G\prod_{i=1}^d\delta_{M_i}^{-1})\ind_K^{\otimes\Qp}\!(\omega^{-\sum_{i=1}^dc_i})=1$, which amounts to checking the following explicit identity (using (\ref{ci}) and Example \ref{exdelta}):
\[\sum_{j=1}^{n-1}j^2 = \sum_{i=1}^d\sum_{j=1}^{n_i-1}j^2 + \sum_{i=1}^d \Big(n_i\big(n-\sum_{j=1}^in_j\big)\big(n-1-\sum_{j=1}^{i-1}n_j\big)\Big).\]
This follows easily by induction on $d$ using the case $d=2$ and the identity
\[(n-m)^2+(n-m+1)^2+\cdots +(n-1)^2 = 1 + 2^2 + \cdots + (m-1)^2 + m(n-m)(n-1)\]
for any integers $n\geq m\geq 1$.
\end{proof}

The following proposition is analogous to Proposition \ref{cons2}.

\begin{prop}\label{consbar}
Assume $\Pi$ is compatible with $\rhobar$ and let $\Phi$ be a bijection as in Definition \ref{compatible2}. Let $\widetilde w\in W_{\rhobar}$, $Q$ a parabolic subgroup containing ${}^{\widetilde w}\!P_{\rhobar}$ and $C_Q$ an isotypic component of $\LLbar\vert_{Z_{M_{Q}}}$ such that $P(C_Q)={}^wQ$ for some {\upshape(}unique{\upshape)} $w\in W$ with $w(S(Q))\subseteq S$. Then $\pi_i(C_Q)$ is compatible with $w(\rhobar^{Q-{\rm ss}})_i$ for $i\in \{1,\ldots,d\}$, where $\pi_i(C_Q)$ is as in Definition \ref{compatible1}(i) and $w(\rhobar^{Q-{\rm ss}})_i$ as in {\upshape(\ref{facteuri})}.
\end{prop}
\begin{proof}
We use the notation in the proof of Proposition \ref{cons2}. Replacing $\rhobar$ by $\widetilde w\rhobar \widetilde w^{-1}$ and $\Phi$ by $\widetilde w(\Phi)$, we can assume $\widetilde w=\Id$. We have to prove that the map $\Phi_{w,i}$ satisfies conditions (i) and (ii) of Definition \ref{compatible2} with $M_i$ instead of $G$ and $w(\rhobar^{Q-{\rm ss}})_i$ instead of $\rhobar$. Note that this makes sense thanks to Proposition \ref{consgood}. We can assume $i=1$. Condition (i) clearly follows from the second equality in (\ref{coeur}) applied to $\pi_1'=\pi_1(C_Q)$. Arguing as in Step 1 of Lemma \ref{cons2}, we need only consider a standard parabolic subgroup $Q_1$ of $M_1$ containing $({}^wP_{\rhobar})_{Q,1}$ and $C_{Q_1}$ an isotypic component of $\LLbar_1\vert_{Z_{M_{Q_1}}}$. Let $C_{Q_{(1)}}$ be the isotypic component of $\LLbar\vert_{Z_{M_{Q_{(1)}}}}$ defined in Step 2 of the proof of Proposition \ref{cons2}. Then it is easy to check that condition (ii) for $M_1$, $w(\rhobar^{Q-{\rm ss}})_1$, $C_{Q_1}$ and an element $w_1\in W(C_{Q_1})$ follows from condition (ii) with $G$, $\rhobar$, $C_{Q_{(1)}}$ and $w_1w\in W(C_{Q_{(1)}})$ (see Step 3, Step 4 and Step 5 of the proof of Proposition \ref{cons2}).
\end{proof}

\subsubsection{Explicit examples}\label{exemples5}

We explicitly give the form of a representation $\Pi$ compatible with $\rhobar$ for various $\rhobar$.

In the examples below, as in Example \ref{exemplester}, a line means a nonsplit extension between two irreducible constituents, the constituent on the left being the subobject of the corresponding (length $2$) subquotient.

\noindent
\textbf{Example 1}\\
We start with $\GL_2(\Q_{p^f})$ and $\widetilde P_{\rhobar}=P_{\rhobar}=B$ as in Example \ref{exemples}(i), i.e.\ we have
\[\rhobar\cong \begin{pmatrix}\chi_{1} &*\\0 &\chi_2\end{pmatrix},\]
where $\chi_i$ are two smooth characters $\Q_{p^f}^\times \rightarrow \F^\times$ (via class field theory) with ratio $\ne 1,\omega^{\pm 1}$ (and where $*$ is nonsplit). Let $\Pi$ be compatible with $\rhobar$. Then $\Pi$ has $f+1$ irreducible constituents and the following form:
\[\begin{xy}(-60,0)*+{\Ind_{B^-(\Q_{p^f})}^{\GL_2(\Q_{p^f})}(\chi_1\omega^{-1}\otimes \chi_2)}="a"; (-28,0)*+{\rm SS_1}="b"; (-15,0)*+{\rm SS_2}="c"; (-2,0)*+{\cdots}="x"; (13,0)*+{{\rm SS}_{f-1}}="d"; (47,0)*+{\Ind_{B^-(\Q_{p^f})}^{\GL_2(\Q_{p^f})}(\chi_2\omega^{-1}\otimes \chi_1)}="e"; 
{\ar@{-}"a";"b"}; {\ar@{-}"b";"c"}; {\ar@{-}"c";"x"}; {\ar@{-}"x";"d"}; {\ar@{-}"d";"e"}
\end{xy}\]
where the ${\rm SS}_i$ for $i\in \{1,\ldots,f-1\}$ are distinct supersingular representations of $\GL_2(\Q_{p^f})$ over $\F$ such that $Z({\rm SS}_i)={\det}(\rhobar)\omega^{-1}$ and
\[V_G({\rm SS}_i)\cong \bigoplus_{\stackrel{I\subseteq \gKQ}{\vert I\vert=f-i}}\Big(\big(\bigotimes_{\sigma\in I}\chi_1^\sigma\big)\otimes \big(\bigotimes_{\sigma\notin I}\chi_2^\sigma\big)\Big)\]
(here $\chi_i^{\sigma}\defeq \chi_i(\sigma\cdot \sigma^{-1})$ and $V_G({\rm SS}_i)$ is immediately checked to be a representation of $\gp$). Moreover it follows from Example \ref{enlightening} that
\[V_G\Big(\Ind_{B^-(\Qp)}^{\GL_2(\Q_{p^f})}(\chi_1\omega^{-1}\otimes \chi_2)\Big)\cong \otimes_{\sigma\in \gKQ}\chi_1^\sigma\]
and likewise with $\Ind_{B^-(\Qp)}^{\GL_2(\Q_{p^f})}(\chi_2\omega^{-1}\otimes \chi_1)$. Finally the conditions in (\ref{DD}) imply that $V_G$ behaves as an exact functor on the (not necessarily irreducible) subquotients of $\Pi$ (see Remark \ref{listrem}(iii)).

\noindent
Still with $\GL_2(\Q_{p^f})$ but when $\widetilde P_{\rhobar}=T$, i.e.\ $\rhobar=\chi_1\oplus \chi_2$, then $\Pi$ (compatible with $\rhobar$) is semisimple, i.e.\ has the same form as above but with split extensions everywhere. This is consistent with the discussion at the end of \cite[\S19]{BP}. Note that, if we only require $\Pi$ to be compatible with $\widetilde P_{\rhobar}$ (Definition \ref{compatible1}), then $\Pi$ has the same form as above, but with arbitrary distinct supersingular representations of $\GL_2(\Q_{p^f})$ and arbitrary distinct irreducible principal series ${\Ind_{B^-(\Q_{p^f})}^{\GL_2(\Q_{p^f})}(\eta_1\omega^{-1}\otimes \eta_2)}$ and ${\Ind_{B^-(\Q_{p^f})}^{\GL_2(\Q_{p^f})}(\eta_2\omega^{-1}\otimes \eta_1)}$. See \cite[\S10.6]{HuWang2} and \S\ref{lcresults} for instances of representations $\Pi$ (coming from mod $p$ cohomology) satisfying (special cases of) the above properties.\\

\noindent
\textbf{Example 2}\\
We go on with $\GL_3(\Qp)$ as in Example \ref{exemples}(ii) and $\widetilde P_{\rhobar}=P_{\rhobar}=B$, i.e.\ we have
\[\rhobar\cong \begin{pmatrix}\chi_{1} & * &*\\0 &\chi_2 & *\\ 0 & 0 & \chi_3 \end{pmatrix},\]
where $\chi_i$ are three smooth characters $\Qp^\times \rightarrow \F^\times$ (via class field theory) of ratio $\ne 1,\omega^{\pm 1}$. For $\tau\in W\cong{\mathcal S}_3$, we define
\[{\rm PS}_{\chi_{\tau(1)},\chi_{\tau(2)},\chi_{\tau(3)}}\defeq \Ind_{B^-(\Qp)}^{\GL_3(\Qp)}(\chi_{\tau(1)}\omega^{-2}\otimes \chi_{\tau(2)}\omega^{-1}\otimes \chi_{\tau(3)}).\]
Let $\Pi$ be compatible with $\rhobar$. Then $\Pi$ has $7$ irreducible constituents and the following form:
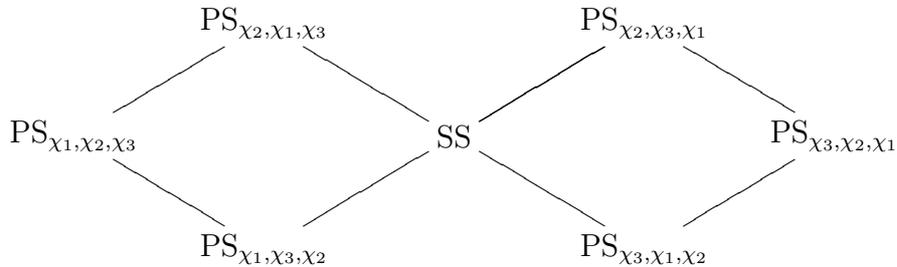
\[\begin{xy}
(-40,0)*+{{\rm PS}_{\chi_1,\chi_2,\chi_3}}="a"; (-15,15)*+{{\rm PS}_{\chi_2,\chi_1,\chi_3}}="b" ; 
(-15,-15)*+{{\rm PS}_{\chi_1,\chi_3,\chi_2}}="d"; (10,0)*+{\rm SS}="e"; (35,15)*+{{\rm PS}_{\chi_2,\chi_3,\chi_1}}="f"; 
(35,-15)*+ {{\rm PS}_{\chi_3,\chi_1,\chi_2}}="h";(60,0)*+{{\rm PS}_{\chi_3,\chi_2,\chi_1}}="i"; 
{\ar@{-}"a";"b"}; {\ar@{-}"a";"d"}; {\ar@{-}"b";"e"}; {\ar@{-}"d";"e"}; {\ar@{-}"e";"f"};
{\ar@{-}"e";"h"}; {\ar@{-}"e";"f"}; {\ar@{-}"f";"i"}; {\ar@{-}"h";"i"};
\end{xy}\]
where SS is a supersingular representation of $\GL_3(\Qp)$ over $\F$ such that $Z({\rm SS})={\det}(\rhobar)\cdot \omega^{-3}$ and $V_G({\rm SS})\cong (\chi_1\chi_2\chi_3)^{\oplus 3}={\det(\rhobar)}^{\oplus 3}$. It follows from the proof of \cite[Thm.5.2.1]{Ha1}, or from \cite[Thm.1.4(i)]{Ha2}, combined with \cite[Cor.4.3.5]{emerton-ordI}, that the nonsplit extensions between two principal series in subquotient are {\it automatically} parabolic inductions as required in condition (i) of Definition \ref{compatible1} (looking at isotypic components of $\LLbar\vert_{Z_{M_{ Q}}}$ with $M_Q\in \{\GL_2\times \GL_1,\GL_1\times \GL_2\}$, see Example \ref{exemples}(ii)). Conditions (ii) to (iv) in Definition \ref{compatible1} are then easily checked. Concerning Definition \ref{compatible2}, the subquotients involving only principal series do satisfy (\ref{DD}) and (\ref{coeur}) by \cite[Rem.9.9]{breuil-foncteur}. The reader can then easily work out the remaining conditions in (\ref{DD}) which all involve the supersingular representation SS, and also work out the shape of a $\Pi$ which is compatible with $\widetilde P_{\rhobar}=B$ only (but not necessarily with $\rhobar$).\\

\noindent
\textbf{Example 3}\\
We stay with $\GL_3(\Qp)$ but where $\widetilde P_{\rhobar}=P_{\rhobar}=P$ with $M_P={\rm diag}(\GL_2,\GL_1)$, i.e.\ we have
\[\rhobar\cong \begin{pmatrix}\rhobar_{1} & *\\0 &\chi_2 \end{pmatrix},\]
where $\rhobar_1:\gp\rightarrow \GL_2(\F)$ is any absolutely irreducible representation and $\chi_2$ is any smooth character $\Qp^\times \rightarrow \F^\times$ (via class field theory). Note that such a $\rhobar$ is always generic (see the beginning of \S\ref{compatibilityrhobar}). Then $\Pi$ is compatible with ${\rhobar}$ if and only $\Pi$ has the same form as in Example \ref{exemplester}:
\[\begin{xy}(0,0)*+{\Ind_{P^-(\Qp)}^{\GL_3(\Qp)}\big(\pi_1\cdot (\omega^{-1}\circ{\det})\otimes \chi_2\big)}; (41,0)*+{\rm SS}**\dir{-} ; (74,0)*+{\Ind_{{P'}^-(\Qp)}^{\GL_3(\Qp)}\big(\chi_2\omega^{-2}\otimes \pi_1\big)}**\dir{-} \end{xy}\]
and where moreover
\begin{itemize}
\item $\pi_1$ is the supersingular representation of $\GL_2(\Qp)$ over $\F$ corresponding to $\rhobar_{1}$ by the mod $p$ local Langlands correspondence for $\GL_2(\Qp)$, i.e.\ we have $Z(\pi_1)={\det}(\rhobar_{1}) \omega^{-1}$ (via class field theory) and $V_{\GL_2}(\pi_1)\cong \rhobar_{1}$;\\
\item $Z({\rm SS})={\det}(\rhobar)\omega^{-3}$;\\
\item $V_G(\Pi)\cong \rhobar\otimes_{\F} \wedge^2_{\F}\rhobar$;\\
\item $V_G\Big(\begin{xy}(0,0)*+{\Ind_{P^-(\Qp)}^{\GL_3(\Qp)}\big(\pi_1\cdot (\omega^{-1}\circ{\det})\otimes \chi_2\big)}; (38,0)*+{\rm SS}**\dir{-}\end{xy}\Big)\cong \Ker(\rhobar\otimes_{\F} \wedge^2_{\F}\rhobar\twoheadrightarrow \chi_2^2\otimes \rhobar_1)$.
\end{itemize}
The properties of $V_G$ in \S\ref{covariant} (in particular Lemma \ref{parabtensor} which can be applied here thanks to Remark \ref{tensorprod}) then automatically give the remaining conditions in (\ref{DD}):
\begin{eqnarray*}
V_G\Big(\Ind_{P^-(\Qp)}^{\GL_3(\Qp)}\big(\pi_1\cdot (\omega^{-1}\circ{\det})\otimes \chi_2\big)\Big)&\cong &\rhobar_1\otimes_{\F} \wedge^2_{\F}\rhobar_1\cong \rhobar_1\otimes \det(\rhobar_1)\\
V_G\Big(\begin{xy}(0,0)*+{\rm SS}; (31,0)*+ {\Ind_{{P'}^-(\Qp)}^{\GL_3(\Qp)}\big(\chi_2\omega^{-2}\otimes \pi_1\big)}**\dir{-}\end{xy}\Big)&\cong & (\rhobar\otimes_{\F} \wedge^2_{\F}\rhobar)/(\rhobar_1\otimes_{\F} \wedge^2_{\F}\rhobar_1)\\
V_G\Big(\Ind_{{P'}^-(\Qp)}^{\GL_3(\Qp)}\big(\chi_2\omega^{-2}\otimes \pi_1\big)\Big)&\cong & \rhobar_1\otimes \chi_2^2\\
V_G({\rm SS})&\cong & (\rhobar_1^{\otimes 2}\otimes \chi_2)\oplus \det(\rhobar_1)\chi_2.
\end{eqnarray*}
The case $\widetilde P_{\rhobar}=M_P$, i.e.\ $\rhobar\cong \begin{pmatrix}\rhobar_{1} & 0\\0 &\chi_2 \end{pmatrix}$, is analogous and easier since $\Pi$ is then semisimple.\\

\noindent
\textbf{Example 4}\\
We consider $\GL_4(\Qp)$ and $\widetilde P_{\rhobar}=P_{\rhobar}=P$, where $M_P={\rm diag}(\GL_2,\GL_1,\GL_1)$, that is we have a good conjugate
\[\rhobar\cong \begin{pmatrix}\rhobar_{1} & * & *\\0 &\chi_2 & *\\ 0 & 0 & \chi_3 \end{pmatrix},\]
where $\rhobar_1:\gp\rightarrow \GL_2(\F)$ is any absolutely irreducible representation and $\chi_i$ two smooth characters $\Qp^\times \rightarrow \F^\times$ (via class field theory) of ratio $\ne 1,\omega^{\pm 1}$. If $1\leq i\leq 4$ and $\sum_{j=1}^in_j=4$ with $1\leq n_j\leq 4$, we write $P_{n_1,\ldots,n_i}$ for the standard parabolic subgroup of $\GL_4$ of Levi ${\rm diag}(\GL_{n_1},\ldots,\GL_{n_i})$ (so $P_{2,1,1}=P$, $P_{1,1,1,1}=B$, etc.). As in Example $3$ above, we let $\pi_1$ be the supersingular representation of $\GL_2(\Qp)$ over $\F$ corresponding to $\rhobar_{1}$ by the mod $p$ local Langlands correspondence for $\GL_2(\Qp)$ (so $Z(\pi_1)={\det}(\rhobar_{1})\cdot \omega^{-1}$ and $V_{\GL_2}(\pi_1)\cong \rhobar_{1}$). We define the following parabolic inductions:
\begin{eqnarray*}
{\rm PI}_{\pi_1,\chi_2,\chi_3}&\defeq &\Ind_{P_{2,1,1}^-(\Qp)}^{\GL_4(\Qp)}\big(\pi_1\cdot (\omega^{-2}\circ{\det})\otimes \chi_2\omega^{-1}\otimes \chi_3\big)\\
{\rm PI}_{\pi_1,\chi_3,\chi_2}&\defeq &\Ind_{P_{2,1,1}^-(\Qp)}^{\GL_4(\Qp)}\big(\pi_1\cdot (\omega^{-2}\circ{\det})\otimes \chi_3\omega^{-1}\otimes \chi_2\big)\\
{\rm PI}_{\chi_2,\pi_1,\chi_3}&\defeq &\Ind_{P_{1,2,1}^-(\Qp)}^{\GL_4(\Qp)}\big(\chi_2\omega^{-3}\otimes \pi_1\cdot (\omega^{-1}\circ\det)\otimes \chi_3\big)\\
{\rm PI}_{\chi_3,\pi_1,\chi_2}&\defeq &\Ind_{P_{1,2,1}^-(\Qp)}^{\GL_4(\Qp)}\big(\chi_3\omega^{-3}\otimes \pi_1\cdot (\omega^{-1}\circ{\det})\otimes \chi_2\big)\\
{\rm PI}_{\chi_2,\chi_3,\pi_1}&\defeq &\Ind_{P_{1,1,2}^-(\Qp)}^{\GL_4(\Qp)}\big(\chi_2\omega^{-3}\otimes \chi_3\omega^{-2}\otimes \pi_1\big)\\
{\rm PI}_{\chi_3,\chi_2,\pi_1}&\defeq &\Ind_{P_{1,1,2}^-(\Qp)}^{\GL_4(\Qp)}\big(\chi_3\omega^{-3}\otimes\chi_2\omega^{-2}\otimes \pi_1\big)
\end{eqnarray*}
and also, for ${\rm ss}_1, {\rm ss}_2$ two (not necessarily distinct) supersingular representations of $\GL_3(\Qp)$ over $\F$: 
\begin{eqnarray*}
{\rm PI}_{{\rm ss}_1,\chi_3}&\defeq &\Ind_{P_{3,1}^-(\Qp)}^{\GL_4(\Qp)}\big({\rm ss}_1\cdot (\omega^{-1}\circ{\det})\otimes \chi_3\big)\\
{\rm PI}_{{\rm ss}_2,\chi_2}&\defeq &\Ind_{P_{3,1}^-(\Qp)}^{\GL_4(\Qp)}\big({\rm ss}_2\cdot (\omega^{-1}\circ{\det})\otimes \chi_2\big)\\
{\rm PI}_{\chi_2,{\rm ss}_2}&\defeq &\Ind_{P_{1,3}^-(\Qp)}^{\GL_4(\Qp)}\big(\chi_2\omega^{-3}\otimes {\rm ss}_2\big)\\
{\rm PI}_{\chi_3,{\rm ss}_1}&\defeq &\Ind_{P_{1,3}^-(\Qp)}^{\GL_4(\Qp)}\big(\chi_3\omega^{-3}\otimes {\rm ss}_1\big).
\end{eqnarray*}
We then let ${\rm SS}_3, {\rm SS}_4, {\rm SS}_5, {\rm SS}_6$ be $4$ distinct supersingular representations of $\GL_4(\Qp)$ over $\F$. If $\Pi$ is compatible with $\rhobar$, then it has the following form:
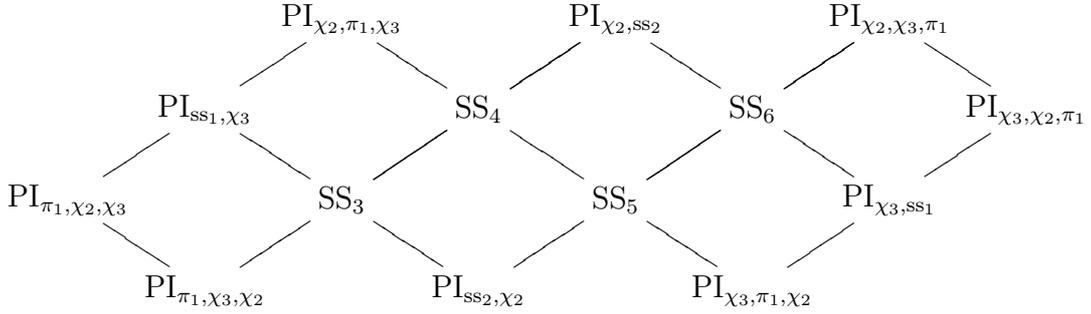
\[\begin{xy}
(-40,0)*+{{\rm PI}_{\pi_1,\chi_2,\chi_3}}="a"; (-22,12)*+{{\rm PI}_{{\rm ss}_1,\chi_3}}="b" ;(-4,24)*+{{\rm PI}_{\chi_2,\pi_1,\chi_3}}="c"; 
(-22,-12)*+{{\rm PI}_{\pi_1,\chi_3,\chi_2}}="d"; (-4,0)*+{\rm SS_3}="e"; (14,12)*+{\rm SS_4}="f"; (32,24)*+ {{\rm PI}_{\chi_2,{\rm ss_2}}}="g";
(14,-12)*+ {{\rm PI}_{{\rm ss}_2,\chi_2}}="h";(32,0)*+{\rm SS_5}="i"; (50,12)*+{\rm SS_6}="j"; (68,24)*+{{\rm PI}_{\chi_2,\chi_3,\pi_1}}="k"; 
(50,-12)*+{{\rm PI}_{\chi_3,\pi_1,\chi_2}}="l"; (68,0)*+{{\rm PI}_{\chi_3,{\rm ss}_1}}="m" ;(86,12)*+{{\rm PI}_{\chi_3,\chi_2,\pi_1}}="n"; 
{\ar@{-}"a";"b"}; {\ar@{-}"a";"d"}; {\ar@{-}"b";"c"}; {\ar@{-}"b";"e"}; {\ar@{-}"d";"e"}; {\ar@{-}"e";"f"};{\ar@{-}"c";"f"};{\ar@{-}"f";"g"};
{\ar@{-}"e";"h"}; {\ar@{-}"e";"f"}; {\ar@{-}"f";"i"}; {\ar@{-}"f";"g"}; {\ar@{-}"g";"j"}; {\ar@{-}"h";"i"};{\ar@{-}"i";"j"};{\ar@{-}"j";"k"};
{\ar@{-}"i";"l"}; {\ar@{-}"i";"j"}; {\ar@{-}"j";"k"}; {\ar@{-}"j";"m"}; {\ar@{-}"k";"n"}; {\ar@{-}"l";"m"};{\ar@{-}"m";"n"}
\end{xy}\]
where we have
\begin{equation}\label{twohalves}
\begin{array}{lll}
\begin{xy}(-40,0)*+{{\rm PI}_{\pi_1,\chi_2,\chi_3}}="a"; (-20,0)*+{{\rm PI}_{{\rm ss}_1,\chi_3}}="b" ;(0,0)*+{{\rm PI}_{\chi_2,\pi_1,\chi_3}}="c";{\ar@{-}"a";"b"}; {\ar@{-}"b";"c"}\end{xy}&\cong & \Ind_{P_{3,1}^-(\Qp)}^{\GL_4(\Qp)}(\Pi_1\cdot (\omega^{-1}\circ{\det})\otimes \chi_3)\\
\begin{xy}(-40,0)*+{{\rm PI}_{\chi_3,\pi_1,\chi_2}}="l"; (-20,0)*+{{\rm PI}_{\chi_3,{\rm ss}_1}}="m" ;(0,0)*+{{\rm PI}_{\chi_3,\chi_2,\pi_1}}="n";{\ar@{-}"l";"m"};{\ar@{-}"m";"n"}\end{xy}&\cong &\Ind_{P_{1,3}^-(\Qp)}^{\GL_4(\Qp)}(\chi_3\omega^{-3}\otimes \Pi_1)
\end{array}
\end{equation}
for \ $\Pi_1\cong \begin{xy}(0,0)*+{\Ind_{P_{2,1}^-(\Qp)}^{\GL_3(\Qp)}\big(\pi_1\cdot (\omega^{-1}\circ{\det})\otimes \chi_2\big)}; (39,0)*+{\rm ss_1}**\dir{-} ; (70,0)*+{\Ind_{{P}_{1,2}^-(\Qp)}^{\GL_3(\Qp)}\big(\chi_2\omega^{-2}\otimes \pi_1\big),}**\dir{-}
\end{xy}$ \ and \ also
\begin{multline}\label{auto}
\begin{xy}(0,0)*+{{\rm PI}_{\pi_1,\chi_2,\chi_3}}; (22,0)*+{{\rm PI}_{\pi_1,\chi_3,\chi_2}}**\dir{-}\end{xy}\!\cong \!\\
\Ind_{P_{2,2}^-(\Qp)}^{\GL_4(\Qp)}\!\Big(\!\pi_1\!\cdot \!(\omega^{-2}\!\circ\!{\det})\!\otimes\! \big(\!\!\begin{xy}(30,0)*+{\Ind_{B^-(\Qp)}^{\GL_2(\Qp)}(\chi_2\omega^{-1}\!\otimes \!\chi_3)}; (77,0)*+{\Ind_{B^-(\Qp)}^{\GL_2(\Qp)}(\chi_3\omega^{-1}\!\otimes \!\chi_2)}**\dir{-}\end{xy}\!\!\big)\!\Big)
\end{multline}
and an analogous isomorphism for \!$\begin{xy}(0,0)*+{{\rm PI}_{\chi_2,\chi_3,\pi_1}}; (21,0)*+{{\rm PI}_{\chi_3,\chi_2,\pi_1}}**\dir{-}\end{xy}$\!. It actually easily follows from \cite[Thm.1.4(i)]{Ha2} (see also \cite[Thm.B(b)(ii)]{heyer}) together with \cite[Cor.4.3.5]{emerton-ordI} \ that \ the \ isomorphism \ in \ (\ref{auto}) \ and \ the \ analogous \ isomorphism \ with \!$\begin{xy}(0,0)*+{{\rm PI}_{\chi_2,\chi_3,\pi_1}}; (21,0)*+{{\rm PI}_{\chi_3,\chi_2,\pi_1}}**\dir{-}\end{xy}$\! hold (i.e.~are not conjectural). It also follows from \cite[Thm.1.2(ii)]{Ha2} and \cite[Thm.1.2(iii)]{Ha2} that we automatically have isomorphisms
\begin{eqnarray*}
\begin{xy}(-40,0)*+{{\rm PI}_{\pi_1,\chi_2,\chi_3}}="a"; (-20,0)*+{{\rm PI}_{{\rm ss}_1,\chi_3}}="b" ;{\ar@{-}"a";"b"}\end{xy}&\cong &\Ind_{P_{3,1}^-(\Qp)}^{\GL_4(\Qp)}\Big(\!\begin{xy}(0,0)*+{\Ind_{P_{2,1}^-(\Qp)}^{\GL_3(\Qp)}\big(\pi_1\cdot (\omega^{-1}\circ{\det})\otimes \chi_2\big)}; (39,0)*+{\rm ss_1}**\dir{-} \end{xy}\!\Big)\\
\begin{xy}(-40,0)*+{{\rm PI}_{{\rm ss}_1,\chi_3}}="b" ;(-20,0)*+{{\rm PI}_{\chi_2,\pi_1,\chi_3}}="c";{\ar@{-}"b";"c"}\end{xy}&\cong &\Ind_{P_{3,1}^-(\Qp)}^{\GL_4(\Qp)}\Big(\!\begin{xy}(0,0)*+{\rm ss_1}**\dir{-} ; (31,0)*+{\Ind_{{P}_{1,2}^-(\Qp)}^{\GL_3(\Qp)}\big(\chi_2\omega^{-2}\otimes \pi_1\big)}**\dir{-}\end{xy}\!\Big)
\end{eqnarray*}
and likewise with the two ``halves'' of $\!\begin{xy}(-40,0)*+{{\rm PI}_{\chi_3,\pi_1,\chi_2}}="l"; (-20,0)*+{{\rm PI}_{\chi_3,{\rm ss}_1}}="m" ;(0,0)*+{{\rm PI}_{\chi_3,\chi_2,\pi_1}}="n";{\ar@{-}"l";"m"};{\ar@{-}"m";"n"}\end{xy}\!$. It is likely that the full isomorphisms (\ref{twohalves}) are in fact also automatic.

We must have moreover $Z({\rm ss}_1)={\det}(\rhobar_1)\chi_2\omega^{-3}$, $Z({\rm ss}_2)={\det}(\rhobar_1)\chi_3\omega^{-3}$, $Z({\rm SS}_i)=\det(\rhobar)\omega^{-6}$ for $i\in \{3,4,5,6\}$ and 
\begin{eqnarray*}
V_{\GL_3}({\rm ss}_1)&\cong &(\rhobar_1^{\otimes 2}\otimes \chi_2)\oplus {\det}(\rhobar_1)\chi_2\\ 
V_{\GL_3}({\rm ss}_2)&\cong &(\rhobar_1^{\otimes 2}\otimes \chi_3)\oplus {\det}(\rhobar_1)\chi_3\\ 
V_{\GL_4}({\rm SS}_3)&\cong &\big(\rhobar_1^{\otimes 2}\otimes {\det(\rhobar_1)}\chi_2\chi_3\big)^{\oplus 3}\oplus \big({\det(\rhobar_1)}^2\chi_2\chi_3\big)^{\oplus 2}\\
V_{\GL_4}({\rm SS}_4)&\cong &\big(\rhobar_1\otimes {\det(\rhobar_1)}\chi_2^2\chi_3\big)^{\oplus 5}\oplus \big({\rhobar_1}^{\!\!\otimes 3}\otimes \chi_2^2\chi_3\big)\\
V_{\GL_4}({\rm SS}_5)&\cong &\big(\rhobar_1\otimes {\det(\rhobar_1)}\chi_2\chi_3^2\big)^{\oplus 5}\oplus \big({\rhobar_1}^{\!\!\otimes 3}\otimes \chi_2\chi_3^2\big)\\
V_{\GL_4}({\rm SS}_6)&\cong &\big(\rhobar_1^{\otimes 2}\otimes \chi_2^2\chi_3^2\big)^{\oplus 3}\oplus \big({\det(\rhobar_1)}\chi_2^2\chi_3^2\big)^{\oplus 2}.
\end{eqnarray*}
The reader can work out all the other conditions of Definition \ref{compatible2} (applying $V_G$ to subquotients of $\Pi$). Note that by Proposition \ref{consbar} the $\GL_3(\Qp)$-representation $\Pi_1$ is compatible with the subrepresentation $\smatr{\rhobar_1}{*}{0}{\chi_2}$ of $\rhobar$ (see the last part in Example $2$).\\

\noindent
\textbf{Example 5}\\
We stay with $\GL_4(\Qp)$ but where $P_{\rhobar}=P$ with $M_P={\rm diag}(\GL_1,\GL_2,\GL_1)$ and a good conjugate of the form
\[\rhobar\cong \begin{pmatrix}\chi_2 & * & *\\0 &\rhobar_{1} & 0\\ 0 & 0 & \chi_3 \end{pmatrix},\]
where the $*$ are nonzero, $\rhobar_1:\gp\rightarrow \GL_2(\F)$ is any absolutely irreducible representation and $\chi_i$ are two smooth characters $\Qp^\times \rightarrow \F^\times$ (via class field theory) of ratio $\ne 1,\omega^{\pm 1}$. One has (see (\ref{wrhobar})) $W_{\rhobar}=\{\Id, s_{e_2-e_3}s_{e_3-e_4}\}$ = the set of permutations of the last two blocks $\GL_2$ and $\GL_1$. Using the notation and conventions of the previous case, we can check that any $\Pi$ compatible with $\rhobar$ has the following form:
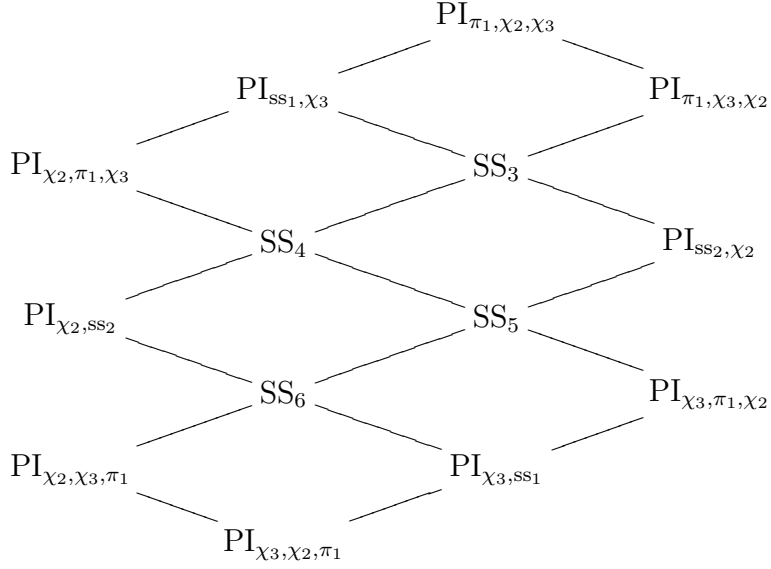
\[\begin{xy}
(-40,0)*+{{\rm PI}_{\chi_2,\pi_1,\chi_3}}="a"; (-12,10)*+{{\rm PI}_{{\rm ss}_1,\chi_3}}="b" ;(16,20)*+{{\rm PI}_{\pi_1,\chi_2,\chi_3}}="c"; 
(-40,-20)*+{{\rm PI}_{\chi_2,{\rm ss}_2}}="d"; (-12,-10)*+{\rm SS_4}="e"; (16,0)*+{\rm SS_3}="f"; (44,10)*+ {{\rm PI}_{\pi_1,\chi_3,\chi_2}}="g";
(-40,-40)*+ {{\rm PI}_{\chi_2,\chi_3,\pi_1}}="h";(-12,-30)*+{\rm SS_6}="i"; (16,-20)*+{\rm SS_5}="j"; (44,-10)*+{{\rm PI}_{{\rm ss}_2,\chi_2}}="k"; 
(-12,-50)*+{{\rm PI}_{\chi_3,\chi_2,\pi_1}}="l"; (16,-40)*+{{\rm PI}_{\chi_3,{\rm ss}_1}}="m" ;(44,-30)*+{{\rm PI}_{\chi_3,\pi_1,\chi_2}}="n"; 
{\ar@{-}"a";"b"}; {\ar@{-}"a";"e"}; {\ar@{-}"b";"c"}; {\ar@{-}"b";"f"}; {\ar@{-}"d";"e"};{\ar@{-}"d";"i"}; 
{\ar@{-}"c";"g"};{\ar@{-}"f";"g"}; {\ar@{-}"e";"j"};{\ar@{-}"f";"k"};
{\ar@{-}"e";"f"}; {\ar@{-}"h";"i"};{\ar@{-}"h";"l"};{\ar@{-}"i";"j"};{\ar@{-}"j";"k"};
{\ar@{-}"i";"m"}; {\ar@{-}"j";"n"}; {\ar@{-}"l";"m"};{\ar@{-}"m";"n"}
\end{xy}\]
(recall the socle is the first layer on the left), where condition (i) in Definition \ref{compatible1} yields, when applied to a suitable $C_Q$ with $M_Q={\rm diag}(\GL_3,\GL_1)$:
\begin{equation}\label{halves}
\begin{array}{lll}
\begin{xy}(-40,0)*+{{\rm PI}_{\chi_2,\pi_1,\chi_3}}="a"; (-20,0)*+{{\rm PI}_{{\rm ss}_1,\chi_3}}="b" ;(0,0)*+{{\rm PI}_{\pi_1,\chi_2,\chi_3}}="c";{\ar@{-}"a";"b"}; {\ar@{-}"b";"c"}\end{xy}\!&\cong & \Ind_{P_{3,1}^-(\Qp)}^{\GL_4(\Qp)}(\Pi_1\cdot (\omega^{-1}\circ{\det})\otimes \chi_3)\\
\begin{xy}(-40,0)*+{{\rm PI}_{\chi_3,\chi_2,\pi_1}}="l"; (-20,0)*+{{\rm PI}_{\chi_3,{\rm ss}_1}}="m" ;(0,0)*+{{\rm PI}_{\chi_3,\pi_1,\chi_2}}="n";{\ar@{-}"l";"m"};{\ar@{-}"m";"n"}\end{xy}\!&\cong &\Ind_{P_{1,3}^-(\Qp)}^{\GL_4(\Qp)}(\chi_3\omega^{-3}\otimes \Pi_1)
\end{array}
\end{equation}
for \ $\Pi_1\cong \begin{xy}(0,0)*+{\Ind_{{P}_{1,2}^-(\Qp)}^{\GL_3(\Qp)}\big(\chi_2\omega^{-2}\otimes \pi_1\big)}; (36,0)*+{\rm ss_1}**\dir{-} ; (80,0)*+{\Ind_{P_{2,1}^-(\Qp)}^{\GL_3(\Qp)}\big(\pi_1\cdot (\omega^{-1}\circ{\det})\otimes \chi_2\big),}**\dir{-}
\end{xy}$
and yields, when applied to a suitable $C_Q$ with $M_Q={\rm diag}(\GL_2,\GL_2)$ (that is, ${}^{s_{e_2-e_3}s_{e_3-e_4}}P_{\rhobar}\subseteq Q$, note that here $P_{\rhobar}\not\subseteq Q$, see Remark \ref{comments}(vii)):
\begin{multline}\label{wptildenon1}
\begin{xy}(0,0)*+{{\rm PI}_{\chi_2,\chi_3,\pi_1}}; (23,0)*+{{\rm PI}_{\chi_3,\chi_2,\pi_1}}**\dir{-}\end{xy}\cong \\
\Ind_{P_{2,2}^-(\Qp)}^{\GL_4(\Qp)}\!\Big(
\!\!\begin{xy}(-25,0)*+{\big(\!\Ind_{B^-(\Qp)}^{\GL_2(\Qp)}(\chi_2\omega^{-3}\otimes \chi_3\omega^{-2})}; (38,0)*+{\Ind_{B^-(\Qp)}^{\GL_2(\Qp)}(\chi_3\omega^{-3}\otimes \chi_2\omega^{-2})\big)\!\otimes \pi_1}**\dir{-}\end{xy}\!\!\Big)
\end{multline}
and an analogous isomorphism for $\begin{xy}(0,0)*+{{\rm PI}_{\pi_1,\chi_2,\chi_3}}; (23,0)*+{{\rm PI}_{\pi_1,\chi_3,\chi_2}}**\dir{-}\end{xy}$. As in Example $4$, it follows from \cite[Thm.1.4(i)]{Ha2} that (\ref{wptildenon1}) and the analogous isomorphism are automatic, and from \cite[Thm.1.2(ii)]{Ha2}, \cite[Thm.1.2(iii)]{Ha2} that isomorphisms as in (\ref{halves}) but for every ``half'' only of the extensions on the left are also automatic. 

One can again work out all the conditions of Definition \ref{compatible2} (conditions on $Z({\rm ss}_i)$, $Z({\rm SS}_i)$ and on $V_{\GL_3}({\rm ss}_i)$, $V_{\GL_4}({\rm SS}_i)$ are the same as in Example $4$).\\ \\

\noindent
\textbf{Example 6}\\
We consider $\GL_3(\Q_{p^2})$ and $\widetilde P_{\rhobar}=P_{\rhobar}=B$, i.e.
\[\rhobar\cong \begin{pmatrix}\chi_{1} & * &*\\0 &\chi_2 & *\\ 0 & 0 & \chi_3 \end{pmatrix},\]
where $\chi_i$ are three smooth characters $\Q_{p^2}^\times \rightarrow \F^\times$ (via class field theory) of ratio $\ne 1,\omega^{\pm 1}$. We let ${\rm ss}_1, {\rm ss}_2, {\rm ss}_3$ be 3 (not necessarily distinct) supersingular representations of $\GL_2(\Q_{p^2})$ over $\F$ and ${\rm SS}_i$, $i\in \{4,\ldots,10\}$ be $7$ distinct supersingular representations of $\GL_3(\Q_{p^2})$ over $\F$. We use without comment notation for $\GL_3(\Q_{p^2})$ analogous to the ones in Example $2$, Example $4$ and Example $5$ to denote principal series and parabolic inductions. If $\Pi$ is compatible with $\rhobar$, then it has the following form:
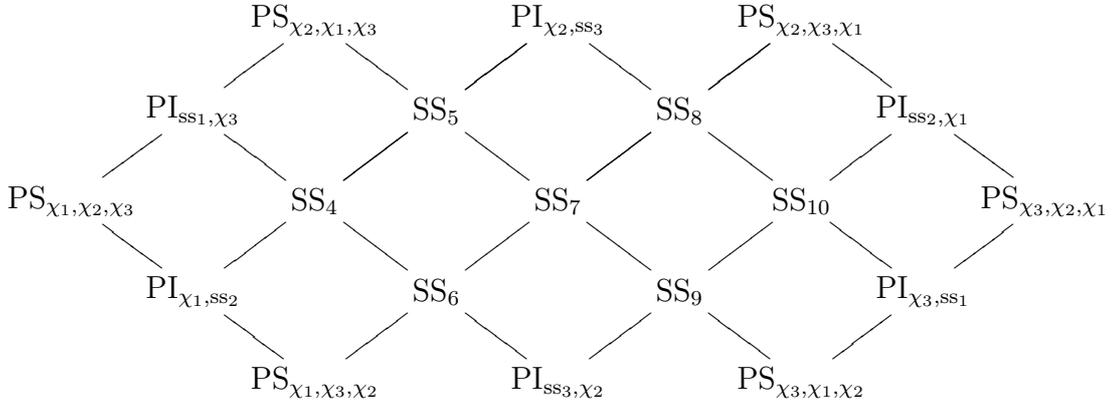
\[\begin{xy}
(-40,0)*+{{\rm PS}_{\chi_1,\chi_2,\chi_3}}="a"; (-24,12)*+{{\rm PI}_{{\rm ss_1},\chi_3}}="b" ;(-8,24)*+{{\rm PS}_{\chi_2,\chi_1,\chi_3}}="c"; 
(-24,-12)*+{{\rm PI}_{\chi_1,{\rm ss_2}}}="d"; (-8,0)*+{\rm SS_4}="e"; (8,12)*+{\rm SS_5}="f"; (24,24)*+ {{\rm PI}_{\chi_2,{\rm ss_3}}}="g";
(-8,-24)*+ {{\rm PS}_{\chi_1,\chi_3,\chi_2}}="h'";(8,-12)*+ {{\rm SS}_{6}}="h";(24,0)*+{\rm SS_7}="i"; (40,12)*+{\rm SS_8}="j"; (56,24)*+{{\rm PS}_{\chi_2,\chi_3,\chi_1}}="k"; 
(24,-24)*+ {{\rm PI}_{{\rm ss_3},\chi_2}}="l'";(40,-12)*+{{\rm SS}_{9}}="l"; (56,0)*+{{\rm SS}_{10}}="m" ;(72,12)*+{{\rm PI}_{{\rm ss_2},\chi_1}}="n"; 
(56,-24)*+{{\rm PS}_{\chi_3,\chi_1,\chi_2}}="o"; (72,-12)*+{{\rm PI}_{\chi_3,{\rm ss_1}}}="p" ;(88,0)*+{{\rm PS}_{\chi_3,\chi_2,\chi_1}}="q"; 
{\ar@{-}"a";"b"}; {\ar@{-}"a";"d"}; {\ar@{-}"b";"c"}; {\ar@{-}"b";"e"}; {\ar@{-}"d";"e"}; {\ar@{-}"e";"f"};{\ar@{-}"c";"f"};{\ar@{-}"f";"g"};
{\ar@{-}"e";"h"}; {\ar@{-}"e";"f"}; {\ar@{-}"f";"i"}; {\ar@{-}"f";"g"}; {\ar@{-}"g";"j"}; {\ar@{-}"h";"i"};{\ar@{-}"i";"j"};{\ar@{-}"j";"k"};
{\ar@{-}"i";"l"}; {\ar@{-}"i";"j"}; {\ar@{-}"j";"k"}; {\ar@{-}"j";"m"}; {\ar@{-}"k";"n"}; {\ar@{-}"l";"m"};{\ar@{-}"m";"n"};
{\ar@{-}"d";"h'"}; {\ar@{-}"h";"h'"}; {\ar@{-}"h";"l'"}; {\ar@{-}"l'";"l"}; {\ar@{-}"l";"o"}; {\ar@{-}"m";"p"};{\ar@{-}"n";"q"}; {\ar@{-}"o";"p"}; {\ar@{-}"p";"q"}; 
\end{xy}\]
where we have
\begin{equation}\label{autohalf}
\begin{array}{lll}
\begin{xy}(-40,0)*+{{\rm PS}_{\chi_1,\chi_2,\chi_3}}="a"; (-20,0)*+{{\rm PI}_{{\rm ss}_1,\chi_3}}="b" ;(0,0)*+{{\rm PS}_{\chi_2,\chi_1,\chi_3}}="c";{\ar@{-}"a";"b"}; {\ar@{-}"b";"c"}\end{xy}&\cong & \Ind_{P_{2,1}^-(\Q_{p^2})}^{\GL_3(\Q_{p^2})}\big(\Pi_1\cdot (\omega^{-1}\circ{\det})\otimes \chi_3\big)\\
\begin{xy}(-40,0)*+{{\rm PS}_{\chi_3,\chi_1,\chi_2}}="a"; (-20,0)*+{{\rm PI}_{\chi_3,{\rm ss}_1}}="b" ;(0,0)*+{{\rm PS}_{\chi_3,\chi_2,\chi_1}}="c";{\ar@{-}"a";"b"}; {\ar@{-}"b";"c"}\end{xy}&\cong & \Ind_{P_{1,2}^-(\Q_{p^2})}^{\GL_3(\Q_{p^2})}\big(\chi_3\omega^{-2}\otimes \Pi_1\big)\\
\begin{xy}(-40,0)*+{{\rm PS}_{\chi_2,\chi_3,\chi_1}}="a"; (-20,0)*+{{\rm PI}_{{\rm ss}_2,\chi_1}}="b" ;(0,0)*+{{\rm PS}_{\chi_3,\chi_2,\chi_1}}="c";{\ar@{-}"a";"b"}; {\ar@{-}"b";"c"}\end{xy}&\cong & \Ind_{P_{2,1}^-(\Q_{p^2})}^{\GL_3(\Q_{p^2})}\big(\Pi_2\cdot (\omega^{-1}\circ{\det})\otimes \chi_1\big)\\
\begin{xy}(-40,0)*+{{\rm PS}_{\chi_1,\chi_2,\chi_3}}="a"; (-20,0)*+{{\rm PI}_{\chi_1,{\rm ss}_2}}="b" ;(0,0)*+{{\rm PS}_{\chi_1,\chi_3,\chi_2}}="c";{\ar@{-}"a";"b"}; {\ar@{-}"b";"c"}\end{xy}&\cong & \Ind_{P_{1,2}^-(\Q_{p^2})}^{\GL_3(\Q_{p^2})}\big(\chi_1\omega^{-2}\otimes \Pi_2\big)
\end{array}
\end{equation}
for
\begin{eqnarray*}
\Pi_1&\cong &\begin{xy}(0,0)*+{\Ind_{B^-(\Q_{p^2})}^{\GL_2(\Q_{p^2})}\big(\chi_1\omega^{-1}\otimes \chi_2\big)}; (33,0)*+{\rm ss_1}**\dir{-} ; (66,0)*+{\Ind_{B^-(\Q_{p^2})}^{\GL_2(\Q_{p^2})}\big(\chi_2\omega^{-1}\otimes \chi_1\big)}**\dir{-}\end{xy}\\
\Pi_2&\cong &\begin{xy}(0,0)*+{\Ind_{B^-(\Q_{p^2})}^{\GL_2(\Q_{p^2})}\big(\chi_2\omega^{-1}\otimes \chi_3\big)}; (33,0)*+{\rm ss_2}**\dir{-} ; (66,0)*+{\Ind_{B^-(\Q_{p^2})}^{\GL_2(\Q_{p^2})}\big(\chi_3\omega^{-1}\otimes \chi_2\big).}**\dir{-}\end{xy}
\end{eqnarray*}
By a straightforward induction, it follows from \cite[Thm.1.3]{Ha2} combined with \cite[Cor.4.3.5]{emerton-ordI} that all isomorphisms (\ref{autohalf}) are actually true!

We must have moreover $Z({\rm ss}_1)=\chi_1\chi_2\omega^{-1}$, $Z({\rm ss}_2)=\chi_2\chi_3\omega^{-1}$, $Z({\rm ss}_3)=\chi_1\chi_3\omega^{-1}$, $Z({\rm SS}_i)=\det(\rhobar)\omega^{-3}$ for $i\in \{4,\ldots,10\}$ and, denoting by $\sigma$ the only nontrivial element of ${\rm Gal}(\Q_{p^2}/\Qp)$:
\begin{eqnarray*}
V_{\GL_2}({\rm ss}_1)&\cong &\chi_1\chi_2^\sigma\oplus \chi_1^\sigma\chi_2\\ 
V_{\GL_2}({\rm ss}_2)&\cong &\chi_2\chi_3^\sigma\oplus \chi_3^\sigma\chi_2\\ 
V_{\GL_2}({\rm ss}_3)&\cong &\chi_1\chi_3^\sigma\oplus \chi_3^\sigma\chi_1\\ 
V_{\GL_3}({\rm SS}_4)&\cong &\Big(\chi_1^2\chi_2\det(\rhobar)^{\sigma}\oplus (\chi_1^2\chi_2)^\sigma\det(\rhobar)\Big)^{\oplus 3}\oplus \Big(\chi_1^2\chi_3(\chi_2^2\chi_1)^\sigma\oplus (\chi_1^2\chi_3)^\sigma\chi_2^2\chi_1\Big)\\
V_{\GL_3}({\rm SS}_i)&\cong &{\rm analogous\ for}\ i\in \{5,6,8,9,10\}{\rm \ \ (left\ to\ reader)}\\
V_{\GL_3}({\rm SS}_7)&\cong &\big(\det(\rhobar)\det(\rhobar)^\sigma\big)^{\oplus 9}\oplus \Big(\chi_1^2\chi_2(\chi_3^2\chi_2)^\sigma\oplus (\chi_1^2\chi_2)^\sigma\chi_3^2\chi_2\Big) \oplus \\
&&\ \ \ \ \ \ \ \Big(\chi_1^2\chi_3(\chi_2^2\chi_3)^\sigma\oplus (\chi_1^2\chi_3)^\sigma\chi_2^2\chi_3\Big)\oplus \Big(\chi_2^2\chi_1(\chi_3^2\chi_1)^\sigma\oplus (\chi_2^2\chi_1)^\sigma\chi_3^2\chi_1\Big)
\end{eqnarray*}
(all obviously representations of $\gp$ over $\F$). The reader can then work out the conditions in (\ref{DD}) involving the various subquotients of $\Pi$. Finally, by Proposition \ref{consbar} the $\GL_2(\Q_{p^2})$-representation $\Pi_1$ (resp.\ $\Pi_2$) is compatible with the subrepresentation $\smatr{\chi_1}{*} {0} {\chi_2}$ (resp.\ with the quotient $\smatr{\chi_2}{*} {0} {\chi_3})$ of $\rhobar$ (see Example $1$).

\noindent
\textbf{Example 7}\\
We end up with $\GL_4(\Q_{p})$ and $\widetilde P_{\rhobar}=P_{\rhobar}=B$, i.e.
\[\rhobar\cong \begin{pmatrix}\chi_{1} & * &* & *\\0 &\chi_2 & * &*\\ 0 & 0 & \chi_3 &*\\ 0&0&0&\chi_4 \end{pmatrix},\]
where $\chi_i$ are four smooth characters $\Q_{p}^\times \rightarrow \F^\times$ of ratio $\ne 1,\omega^{\pm 1}$. The structure of a $\Pi$ compatible with $\rhobar$ is given in the next 3D diagram. Just like the previous 2D diagrams look like stacked squares, this 3D diagram looks like stacked cubes: there are 8 cubes, one being entirely ``behind''. As before, each vertex is an irreducible constituent with PS (in green) meaning principal series, SS (in red) meaning supersingular and PI${}_1$ (resp.\ PI${}_2$) (in blue) meaning parabolic induction from the standard parabolic subgroup of Levi $\GL_3\times \GL_1$ (resp.\ of Levi $\GL_1\times \GL_3$). The socle is the principal series at the very bottom and the cosocle is the principal series at the very top. Like previously, each edge is a nonsplit extension between two irreducible constituents, the dashed edges being those which are ``behind'' in the 3D picture. Near each vertex we write the value of $V_{\GL_4}$ applied to the corresponding irreducible constituent. 

The interested reader can then check all the other conditions and compatibilities in Definition \ref{compatible1} and Definition \ref{compatible2}, for instance the two left faces on the bottom correspond to the parabolic induction PI${}_1$ of Example $2$ tensored by the character $\chi_4$.

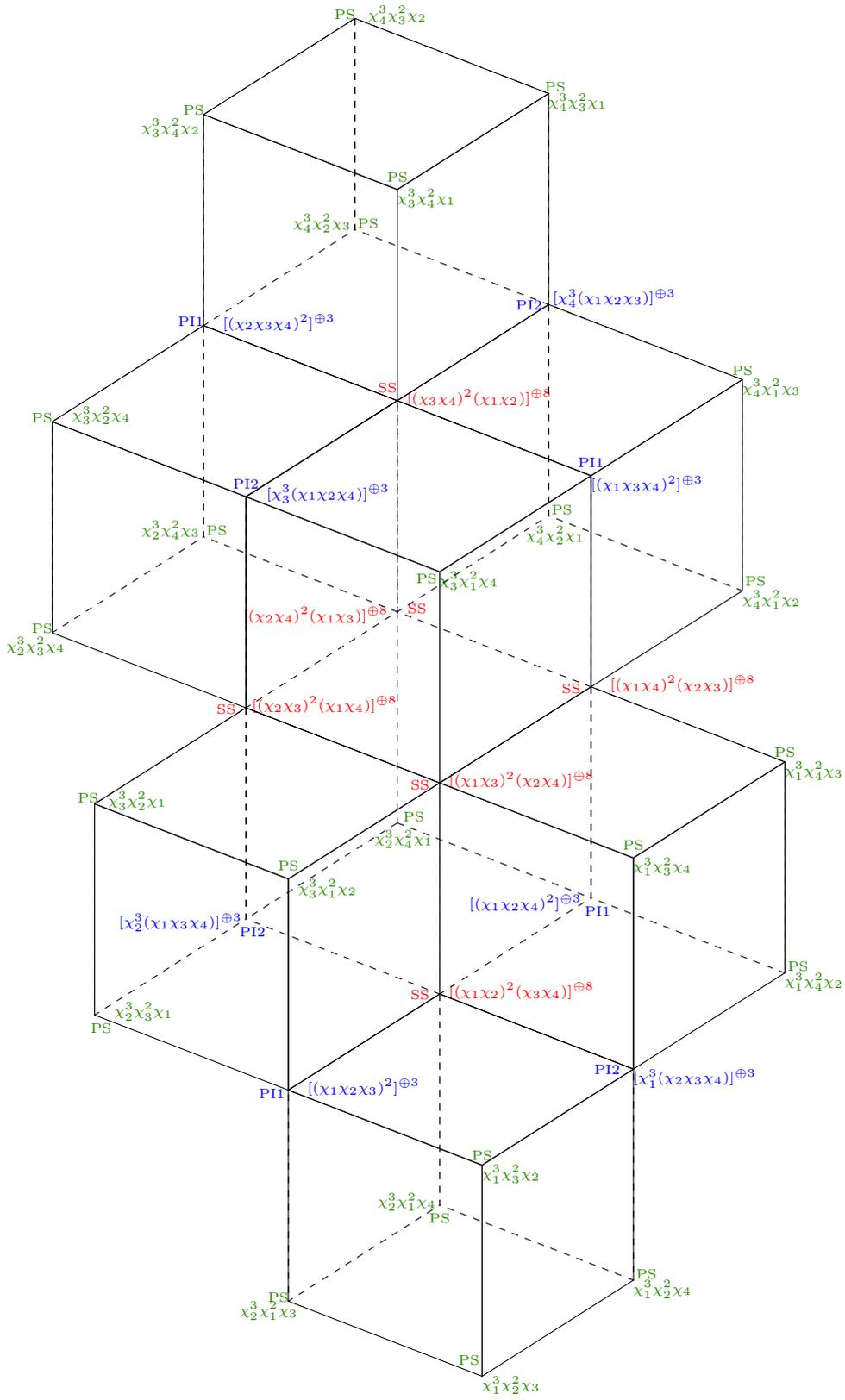
\begin{tikzpicture}[scale=3.75]

\tdplotsetrotatedcoords{0}{0}{-3}

\draw[tdplot_rotated_coords](0,0,0) node[above left=-.5pc]{\tiny{$\color{mygreen}\mathrm{PS}$}};

\draw[tdplot_rotated_coords](1,0,0) node[above left=-.5pc]{\tiny{$\color{mygreen}\mathrm{PS}$}};
\draw[tdplot_rotated_coords](0,0,-1) node[xshift=.2cm, yshift=0.1cm]{\tiny{$\color{mygreen}\mathrm{PS}$}};
\draw[tdplot_rotated_coords](0,1,0) node[xshift=.1cm, yshift=.1cm]{\tiny{$\color{mygreen}\mathrm{PS}$}};

\draw[tdplot_rotated_coords](1,0,-1) node[above left=-.5pc]{\tiny{$\color{blue}\mathrm{PI1}$}};
\draw[tdplot_rotated_coords](0,1,-1) node[xshift=-.3cm, yshift=-.0cm]{\tiny{$\color{blue}\mathrm{PI2}$}};
\draw[tdplot_rotated_coords](1,1,0) node[xshift=.0cm, yshift=.2cm]{\tiny{$\color{mygreen}\mathrm{PS}$}};

\draw[tdplot_rotated_coords](1,1,-1) node[xshift=-.15cm, yshift=.2cm]{\tiny{$\color{red}\mathrm{SS}$}};
\draw[tdplot_rotated_coords](0,2,-1) node[xshift=.0cm, yshift=.15cm]{\tiny{$\color{mygreen}\mathrm{PS}$}};
\draw[tdplot_rotated_coords](0,1,-2) node[xshift=.2cm, yshift=.1cm]{\tiny{$\color{mygreen}\mathrm{PS}$}};
\draw[tdplot_rotated_coords](2,0,-1) node[above left=-.5pc]{\tiny{$\color{mygreen}\mathrm{PS}$}};
\draw[tdplot_rotated_coords](1,0,-2) node[xshift=.2cm, yshift=.1cm]{\tiny{$\color{mygreen}\mathrm{PS}$}};

\draw[tdplot_rotated_coords](1,2,-1) node[xshift=.0cm, yshift=.2cm]{\tiny{$\color{blue}\mathrm{PI1}$}};
\draw[tdplot_rotated_coords](2,1,-1) node[xshift=.0cm, yshift=.2cm]{\tiny{$\color{blue}\mathrm{PI2}$}};
\draw[tdplot_rotated_coords](0,2,-2) node[xshift=.2cm, yshift=.15cm]{\tiny{$\color{mygreen}\mathrm{PS}$}};
\draw[tdplot_rotated_coords](2,0,-2) node[above left=-.5pc]{\tiny{$\color{mygreen}\mathrm{PS}$}};
\draw[tdplot_rotated_coords](1,1,-2) node[xshift=.3cm, yshift=.05cm]{\tiny{$\color{red}\mathrm{SS}$}};

\draw[tdplot_rotated_coords](1,2,-2) node[xshift=-.3cm, yshift=-.025cm]{\tiny{$\color{red}\mathrm{SS}$}};
\draw[tdplot_rotated_coords](2,1,-2) node[xshift=-.3cm, yshift=-.025cm]{\tiny{$\color{red}\mathrm{SS}$}};
\draw[tdplot_rotated_coords](2,2,-1) node[xshift=-.2cm, yshift=-.1cm]{\tiny{$\color{mygreen}\mathrm{PS}$}};
\draw[tdplot_rotated_coords](1,1,-3) node[xshift=.25cm, yshift=.1cm]{\tiny{$\color{mygreen}\mathrm{PS}$}};

\draw[tdplot_rotated_coords](2,2,-2) node[xshift=-.3cm, yshift=-.025cm]{\tiny{$\color{red}\mathrm{SS}$}};
\draw[tdplot_rotated_coords](1,3,-2) node[xshift=-.0cm, yshift=.15cm]{\tiny{$\color{mygreen}\mathrm{PS}$}};
\draw[tdplot_rotated_coords](3,1,-2) node[xshift=-.1cm, yshift=.1cm]{\tiny{$\color{mygreen}\mathrm{PS}$}};
\draw[tdplot_rotated_coords](1,2,-3) node[xshift=.1cm, yshift=-.2cm]{\tiny{$\color{blue}\mathrm{PI1}$}};
\draw[tdplot_rotated_coords](2,1,-3) node[xshift=.1cm, yshift=-.2cm]{\tiny{$\color{blue}\mathrm{PI2}$}};

\draw[tdplot_rotated_coords](2,2,-3) node[xshift=-.3cm, yshift=-.01cm]{\tiny{$\color{red}\mathrm{SS}$}};
\draw[tdplot_rotated_coords](1,3,-3) node[xshift=.2cm, yshift=.1cm]{\tiny{$\color{mygreen}\mathrm{PS}$}};
\draw[tdplot_rotated_coords](2,3,-2) node[xshift=.0cm, yshift=.2cm]{\tiny{$\color{mygreen}\mathrm{PS}$}};
\draw[tdplot_rotated_coords](3,1,-3) node[xshift=.1cm, yshift=-.2cm]{\tiny{$\color{mygreen}\mathrm{PS}$}};
\draw[tdplot_rotated_coords](3,2,-2) node[xshift=.0cm, yshift=.2cm]{\tiny{$\color{mygreen}\mathrm{PS}$}};

\draw[tdplot_rotated_coords](2,3,-3) node[xshift=-.4cm, yshift=-.0cm]{\tiny{$\color{blue}\mathrm{PI2}$}};
\draw[tdplot_rotated_coords](3,2,-3) node[xshift=-.25cm, yshift=-.05cm]{\tiny{$\color{blue}\mathrm{PI1}$}};
\draw[tdplot_rotated_coords](2,2,-4) node[xshift=-.0cm, yshift=-.2cm]{\tiny{$\color{mygreen}\mathrm{PS}$}};

\draw[tdplot_rotated_coords](2,3,-4) node[xshift=.2cm, yshift=.1cm]{\tiny{$\color{mygreen}\mathrm{PS}$}};
\draw[tdplot_rotated_coords](3,3,-3) node[xshift=-.0cm, yshift=.15cm]{\tiny{$\color{mygreen}\mathrm{PS}$}};
\draw[tdplot_rotated_coords](3,2,-4) node[above left=-.5pc]{\tiny{$\color{mygreen}\mathrm{PS}$}};

\draw[tdplot_rotated_coords](3,3,-4) node[xshift=-.2cm, yshift=.25cm]{\tiny{$\color{mygreen}\mathrm{PS}$}};

\draw[tdplot_rotated_coords](0,0,0) node[xshift=.65cm, yshift=.05cm]{\tiny{$\color{mygreen}\chi_4^3\chi_3^2\chi_2$}};

\draw[tdplot_rotated_coords] (1,0,0) node[xshift=-.5cm, yshift=-.2cm]{\tiny{$\color{mygreen}\chi_3^3\chi_4^2\chi_2$}};
\draw[tdplot_rotated_coords] (0,0,-1) node[xshift=-.5cm, yshift=.1cm]{\tiny{$\color{mygreen}\chi_4^3\chi_2^2\chi_3$}};
\draw[tdplot_rotated_coords] (0,1,0) node[below right=-.5pc]{\tiny{$\color{mygreen}\chi_4^3\chi_3^2\chi_1$}};

\draw[tdplot_rotated_coords] (1,0,-1) node[xshift=1.15cm, yshift=.05cm]{\tiny{$\color{blue}[(\chi_2\chi_3\chi_4)^2]^{\oplus3}$}};
\draw[tdplot_rotated_coords] (0,1,-1) node[xshift=1cm, yshift=.1cm]{\tiny{$\color{blue}[\chi_4^3(\chi_1\chi_2\chi_3)]^{\oplus3}$}};
\draw[tdplot_rotated_coords] (1,1,0) node[below right=-.5pc]{\tiny{$\color{mygreen}\chi_3^3\chi_4^2\chi_1$}};

\draw[tdplot_rotated_coords](1,1,-1) node[xshift=1.25cm, yshift=.05cm]{\tiny{$\color{red}[(\chi_3\chi_4)^2(\chi_1\chi_2)]^{\oplus 8}$}};
\draw[tdplot_rotated_coords] (0,2,-1) node[below right=-.5pc]{\tiny{$\color{mygreen}\chi_4^3\chi_1^2\chi_3$}};
\draw[tdplot_rotated_coords] (0,1,-2) node[xshift=.1cm, yshift=-.3cm]{\tiny{$\color{mygreen}\chi_4^3\chi_2^2\chi_1$}};
\draw[tdplot_rotated_coords] (2,0,-1) node[xshift=.75cm, yshift=.1cm]{\tiny{$\color{mygreen}\chi_3^3\chi_2^2\chi_4$}};
\draw[tdplot_rotated_coords] (1,0,-2) node[xshift=-.5cm, yshift=.1cm]{\tiny{$\color{mygreen}\chi_2^3\chi_4^2\chi_3$}};

\draw [tdplot_rotated_coords](1,2,-1) node[below right=-.5pc]{\tiny{$\color{blue}[(\chi_1\chi_3\chi_4)^2]^{\oplus3}$}};
\draw[tdplot_rotated_coords](2,1,-1) node[xshift=1.25cm, yshift=.05cm]{\tiny{$\color{blue}[\chi_3^3(\chi_1\chi_2\chi_4)]^{\oplus3}$}};
\draw[tdplot_rotated_coords] (0,2,-2) node[below right=-.5pc]{\tiny{$\color{mygreen}\chi_4^3\chi_1^2\chi_2$}};
\draw[tdplot_rotated_coords] (2,0,-2) node[xshift=-.25cm, yshift=-.2cm]{\tiny{$\color{mygreen}\chi_2^3\chi_3^2\chi_4$}};
\draw[tdplot_rotated_coords] (1,1,-2) node[xshift=-1.245cm, yshift=-.05cm]{\tiny{$\color{red}[(\chi_2\chi_4)^2(\chi_1\chi_3)]^{\oplus 8}$}};

\draw[tdplot_rotated_coords] (1,2,-2) node[xshift=1.40cm, yshift=.05cm]{\tiny{$\color{red}[(\chi_1\chi_4)^2(\chi_2\chi_3)]^{\oplus 8}$}};
\draw[tdplot_rotated_coords] (2,1,-2) node[xshift=1.2cm, yshift=.05cm]{\tiny{$\color{red}[(\chi_2\chi_3)^2(\chi_1\chi_4)]^{\oplus 8}$}};
\draw[tdplot_rotated_coords] (2,2,-1) node[below right=-.5pc]{\tiny{$\color{mygreen}\chi_3^3\chi_1^2\chi_4$}};
\draw[tdplot_rotated_coords] (1,1,-3) node[xshift=.1cm, yshift=-.25cm]{\tiny{$\color{mygreen}\chi_2^3\chi_4^2\chi_1$}};

\draw[tdplot_rotated_coords] (2,2,-2) node[xshift=1.25cm, yshift=.05cm]{\tiny{$\color{red}[(\chi_1\chi_3)^2(\chi_2\chi_4)]^{\oplus 8}$}};
\draw[tdplot_rotated_coords] (1,3,-2) node[below right=-.5pc]{\tiny{$\color{mygreen}\chi_1^3\chi_4^2\chi_3$}};
\draw[tdplot_rotated_coords] (3,1,-2) node[xshift=.65cm, yshift=.05cm]{\tiny{$\color{mygreen}\chi_3^3\chi_2^2\chi_1$}};
\draw[tdplot_rotated_coords] (1,2,-3) node[xshift=-1cm, yshift=-.075cm]{\tiny{$\color{blue}[(\chi_1\chi_2\chi_4)^2]^{\oplus3}$}};
\draw[tdplot_rotated_coords] (2,1,-3) node[xshift=-1cm, yshift=-.05cm]{\tiny{$\color{blue}[\chi_2^3(\chi_1\chi_3\chi_4)]^{\oplus3}$}};

\draw[tdplot_rotated_coords](2,2,-3) node[xshift=1.25cm, yshift=.05cm]{\tiny{$\color{red}[(\chi_1\chi_2)^2(\chi_3\chi_4)]^{\oplus 8}$}};
\draw[tdplot_rotated_coords] (1,3,-3) node[below right=-.5pc]{\tiny{$\color{mygreen}\chi_1^3\chi_4^2\chi_2$}};
\draw[tdplot_rotated_coords] (2,3,-2) node[below right=-.5pc]{\tiny{$\color{mygreen}\chi_1^3\chi_3^2\chi_4$}};
\draw[tdplot_rotated_coords] (3,1,-3) node[xshift=.75cm, yshift=.05cm]{\tiny{$\color{mygreen}\chi_2^3\chi_3^2\chi_1$}};
\draw[tdplot_rotated_coords] (3,2,-2) node[xshift=.6cm, yshift=-.15cm]{\tiny{$\color{mygreen}\chi_3^3\chi_1^2\chi_2$}};

\draw[tdplot_rotated_coords] (2,3,-3) node[below right=-.5pc]{\tiny{$\color{blue}[\chi_1^3(\chi_2\chi_3\chi_4)]^{\oplus3}$}};
\draw[tdplot_rotated_coords] (3,2,-3) node[xshift=1.15cm, yshift=.05cm]{\tiny{$\color{blue}[(\chi_1\chi_2\chi_3)^2]^{\oplus3}$}};
\draw[tdplot_rotated_coords] (2,2,-4) node[xshift=-.5cm, yshift=.05cm]{\tiny{$\color{mygreen}\chi_2^3\chi_1^2\chi_4$}};

\draw[tdplot_rotated_coords] (2,3,-4) node[below right=-.5pc]{\tiny{$\color{mygreen}\chi_1^3\chi_2^2\chi_4$}};
\draw[tdplot_rotated_coords] (3,3,-3) node[below right=-.5pc]{\tiny{$\color{mygreen}\chi_1^3\chi_3^2\chi_2$}};
\draw[tdplot_rotated_coords] (3,2,-4) node[xshift=-.3cm, yshift=-.15cm]{\tiny{$\color{mygreen}\chi_2^3\chi_1^2\chi_3$}};

\draw[tdplot_rotated_coords](3,3,-4) node[below right=-.5pc]{\tiny{$\color{mygreen}\chi_1^3\chi_2^2\chi_3$}};

\draw[tdplot_rotated_coords](0,0,0) -- (1,0,0)--(1,1,0)--(0,1,0)--cycle;
\draw[tdplot_rotated_coords](1,0,0)--(1,0,-1)--(1,1,-1)--(1,1,0)--cycle;
\draw[tdplot_rotated_coords](0,1,0)--(0,1,-1)--(1,1,-1)--(1,1,0)--cycle;
\draw[dashed,tdplot_rotated_coords]  (0,0,0)--(0,0,-1)--(0,1,-1)--(0,1,0)--cycle;
\draw[dashed,tdplot_rotated_coords]  (0,0,0)--(0,0,-1)--(1,0,-1)--(1,0,0)--cycle;

\draw[tdplot_rotated_coords](0,1,-1) -- (1,1,-1)--(1,2,0-1)--(0,2,0-1)--cycle;

\draw[tdplot_rotated_coords](0,2,0-1)--(0,2,-1-1)--(1,2,-1-1)--(1,2,0-1)--cycle;
\draw[dashed,tdplot_rotated_coords]  (0,1,0-1)--(0,1,-1-1)--(0,2,-1-1)--(0,2,0-1)--cycle;
\draw[dashed,tdplot_rotated_coords]  (0,1,0-1)--(0,1,-1-1)--(1,1,-1-1)--(1,1,0-1)--cycle;

\draw[tdplot_rotated_coords](1,0,-1) -- (1+1,0,0-1)--(1+1,1,0-1)--(0+1,1,0-1)--cycle;
\draw[tdplot_rotated_coords](1+1,0,0-1)--(1+1,0,-1-1)--(1+1,1,-1-1)--(1+1,1,0-1)--cycle;
\draw[tdplot_rotated_coords](0+1,1,0-1)--(1+1,1,0-1);
\draw[dashed,tdplot_rotated_coords]  (0+1,0,0-1)--(0+1,0,-1-1)--(0+1,1,-1-1)--(0+1,1,0-1)--cycle;
\draw[dashed,tdplot_rotated_coords]  (0+1,0,0-1)--(0+1,0,-1-1)--(1+1,0,-1-1)--(1+1,0,0-1)--cycle;

\draw[tdplot_rotated_coords](1,0+1,0-1) -- (1+1,0+1,0-1)--(1+1,1+1,0-1)--(0+1,1+1,0-1)--cycle;
\draw[tdplot_rotated_coords](1+1,0+1,0-1)--(1+1,0+1,-1-1)--(1+1,1+1,-1-1)--(1+1,1+1,0-1)--cycle;
\draw[tdplot_rotated_coords](0+1,1+1,0-1)--(0+1,1+1,-1-1)--(1+1,1+1,-1-1)--(1+1,1+1,0-1)--cycle;
\draw[dashed,tdplot_rotated_coords]  (0+1,0+1,0-1)--(0+1,0+1,-1-1)--(0+1,1+1,-1-1)--(0+1,1+1,0-1)--cycle;
\draw[dashed,tdplot_rotated_coords]  (0+1,0+1,0-1)--(0+1,0+1,-1-1)--(1+1,0+1,-1-1)--(1+1,0+1,0-1)--cycle;

\draw[tdplot_rotated_coords](0+1,0+2,0-2) -- (1+1,0+2,0-2)--(1+1,1+2,0-2)--(0+1,1+2,0-2)--cycle;
\draw[tdplot_rotated_coords](1+1,0+2,0-2)--(1+1,0+2,-1-2)--(1+1,1+2,-1-2)--(1+1,1+2,0-2)--cycle;
\draw[tdplot_rotated_coords](0+1,1+2,0-2)--(0+1,1+2,-1-2)--(1+1,1+2,-1-2)--(1+1,1+2,0-2)--cycle;
\draw[dashed,tdplot_rotated_coords]  (0+1,0+2,0-2)--(0+1,0+2,-1-2)--(0+1,1+2,-1-2)--(0+1,1+2,0-2)--cycle;
\draw[dashed,tdplot_rotated_coords]  (0+1,0+2,0-2)--(0+1,0+2,-1-2)--(1+1,0+2,-1-2)--(1+1,0+2,0-2)--cycle;

\draw[tdplot_rotated_coords](0+2,0+1,0-2) -- (1+2,0+1,0-2)--(1+2,1+1,0-2)--(0+2,1+1,0-2)--cycle;
\draw[tdplot_rotated_coords](1+2,0+1,0-2)--(1+2,0+1,-1-2)--(1+2,1+1,-1-2)--(1+2,1+1,0-2)--cycle;
\draw[tdplot_rotated_coords](0+2,1+1,0-2)--(0+2,1+1,-1-2)--(1+2,1+1,-1-2)--(1+2,1+1,0-2)--cycle;
\draw[dashed,tdplot_rotated_coords]  (0+2,0+1,0-2)--(0+2,0+1,-1-2)--(0+2,1+1,-1-2)--(0+2,1+1,0-2)--cycle;
\draw[dashed,tdplot_rotated_coords]  (0+2,0+1,0-2)--(0+2,0+1,-1-2)--(1+2,0+1,-1-2)--(1+2,0+1,0-2)--cycle;

\draw[tdplot_rotated_coords](0+2,0+2,0-3) -- (1+2,0+2,0-3)--(1+2,1+2,0-3)--(0+2,1+2,0-3)--cycle;
\draw[tdplot_rotated_coords](1+2,0+2,0-3)--(1+2,0+2,-1-3)--(1+2,1+2,-1-3)--(1+2,1+2,0-3)--cycle;
\draw[tdplot_rotated_coords](0+2,1+2,0-3)--(0+2,1+2,-1-3)--(1+2,1+2,-1-3)--(1+2,1+2,0-3)--cycle;
\draw[dashed,tdplot_rotated_coords]  (0+2,0+2,0-3)--(0+2,0+2,-1-3)--(0+2,1+2,-1-3)--(0+2,1+2,0-3)--cycle;
\draw[dashed,tdplot_rotated_coords]  (0+2,0+2,0-3)--(0+2,0+2,-1-3)--(1+2,0+2,-1-3)--(1+2,0+2,0-3)--cycle;

\draw[dashed,tdplot_rotated_coords](0+1,0+1,0-2) -- (1,0+1,-1-2);
\draw[dashed,tdplot_rotated_coords](1,0+1,-1-2) -- (1,0+1+1,-1-2);
\draw[dashed,tdplot_rotated_coords](1,0+1,-1-2) -- (1+1,0+1,-1-2);

\end{tikzpicture}

\subsection{Strong local-global compatibility conjecture}\label{slgc}

Back to the setting of \S\ref{global} but assuming that $F_v^+$ is unramified and that $\rbar_{\tilde{v}}$ (for ${\tilde{v}}\vert v$) is generic as at the beginning of \S\ref{compatibilityrhobar}, we conjecture that the $G(F_{\tilde{v}})$-representation $\Hom_{U^v}(\sigma^v,S(V^{v},\F)[\m^{\Sigma}])$ is a direct sum of copies of a $G(F_{\tilde{v}})$-representation which is (up to twist) compatible with any good conjugate of $\rbar_{\tilde{v}}$ (Definition \ref{compatible2}).

We consider exactly the same global setting as in \S\ref{somprel}. We fix $v\vert p$ in $F^+$ such that $F_v^+$ is an unramified extension of $\Qp$ and consider a continuous representation $\rbar:\gF\rightarrow {\GL}_n(\F)$ such that
\begin{enumerate}
\item $\rbar^c\cong \rbar^\vee\otimes\omega^{1-n}$ (recall $\rbar^c(g)=\rbar(cgc)$ for $g\in \gF$);
\item $\rbar$ is an absolutely irreducible representation of $\gF$;
\item $\rbar_{\tilde{v}}$ for ${\tilde{v}}\vert v$ has distinct irreducible constituents and the ratio of any two irreducible constituents of dimension $1$ is not in $\{\omega,\omega^{-1}\}$
\end{enumerate}
(note that condition (iii) doesn't depend on the place ${\tilde{v}}$ of $F$ dividing $v$ since $\rbar_{{\tilde{v}}^c}\cong \rbar_{\tilde{v}}^\vee\otimes\omega^{1-n}$).

The following is the main conjecture of this paper.

\begin{conj1}\label{theconj}
Let $\rbar:\gF\rightarrow {\GL}_n(\F)$ be a continuous homomorphism that satisfies conditions \emph{(}i\emph{)} to \emph{(}iii\emph{)} above and fix a place $v$ of $F^+$ which divides $p$ such that $F_v^+$ is unramified. Assume that there exist compact open subgroups $V^v\subseteq U^{v}\subseteq \GAv$ with $V^v$ normal in $U^v$, a finite-dimensional representation $\sigma^v$ of $U^v/V^v$ over $\F$ and a finite set $\Sigma$ of finite places of $F^+$ as in \S\ref{wlgc} such that $\Hom_{U^v}(\sigma^v,S(V^{v},\F)[\m^{\Sigma}])\ne 0$, where $\m^{\Sigma}$ is the maximal ideal of $\TT^\Sigma$ associated to $\rbar$. Let ${\tilde{v}}\vert v$ in $F$ and see $\Hom_{U^v}(\sigma^v,S(V^{v},\F)[\m^{\Sigma}])$ as a representation of $H(F_v^+)\cong \GL_n(F_{\tilde{v}})=G(F_{\tilde{v}})$ via $\iota_{\tilde{v}}$ {\upshape(}cf.\ \S\ref{somprel}{\upshape)}. Then there is an integer $d\in \Z_{>0}$ depending only on $v$, $U^v$, $V^v$, $\sigma^v$ and $\rbar$ and an admissible smooth representation $\Pi_{\tilde{v}}$ of $G(F_{\tilde{v}})$ over $\F$ {\upshape(}depending {\it a priori} on ${\tilde{v}}$, $U^{v}$, $V^v$, $\sigma^v$ and $\rbar${\upshape)} such that
\[\Hom_{U^v}(\sigma^v,S(V^{v},\F)[\m^{\Sigma}])\cong \big(\Pi_{\tilde{v}}\otimes (\omega^{n-1}\circ{\det})\big)^{\oplus d},\]
where $\Pi_{\tilde{v}}$ is compatible with one {\upshape(}equivalently any by Proposition \ref{any}{\upshape)} good conjugate of $\rbar_{\tilde{v}}$ in the sense of Definition \ref{gooddef}.
\end{conj1}

\begin{rem1}\label{bandebis}
(i) Conjecture \ref{theconj} implies in particular that the $G(F_{\tilde{v}})$-represen\-tation $\Hom_{U^v}(\sigma^v,S(V^{v},\F)[\m^{\Sigma}])$ is of finite length with all constituents of multiplicity $d$ (under assumptions (i) to (iii) on $\rbar$), which is already far from being known in general. See however \S\ref{gl2results} below for nontrivial evidence in the case of $\GL_2$. It also implies that $\Hom_{U^v}(\sigma^v,S(V^{v},\F)[\m^{\Sigma}])$ has a central character, but this is known (at least under some extra assumptions), see Lemma \ref{central}.\\
(ii) When $F_v^+$ is unramified and $\rbar_{\tilde{v}}$ is as in (iii) above, Conjecture \ref{theconj} of course implies (and is in fact much stronger than) Conjecture \ref{theconjbar}.\\
(iii) Assuming that $p$ is unramified in $F^+$ and that $\rbar_{\tilde{w}}$ is generic as in (iii) above for all $w\vert p$, an even stronger conjecture would be as follows.
\begin{conj1}\label{theconjbis}
For $U^p\subseteq \GAp$ such that $S(U^p,\F)[\m^{\Sigma}]\ne 0$ {\upshape(}where $\Sigma$ contains the set of places of $F^+$ that split in $F$ and divide $pN$, or at which $U^p$ is not unramified, or at which $\rbar$ ramifies, and where $S(U^p,\F)[\m^{\Sigma}]$ is defined as in \S\ref{somprel} replacing $U^v$ by $U^p${\upshape)} and for any $\tilde w\vert w$ in $F$ with $w\vert p$, there is an integer $d\in \Z_{>0}$ depending only on $p$, $U^p$ and $\rbar$ and admissible smooth representations $\Pi_{\tilde{w}}$ of $G(F_{\tilde{w}})$ over $\F$, where $\Pi_{\tilde{w}}$ is compatible with one {\upshape(}equivalently any{\upshape)} good conjugate of $\rbar_{\tilde{w}}$ such that
\[S(U^p,\F)[\m^{\Sigma}]\cong \Big(\bigotimes_{w\vert p}\big(\Pi_{\tilde{w}}\otimes (\omega^{n-1}\circ{\det})\big)\Big)^{\oplus d}.\]
\end{conj1}
\end{rem1}

As in \S\ref{wlgc}, we prove that Conjecture \ref{theconj} holds for ${\tilde{v}}$ if and only if it holds for ${\tilde{v}}^c$ (we do not need here extra assumptions). We start with two formal lemmas. We use the previous notation and denote by $w_0\in W$ the unique element with maximal length.

\begin{lem1}\label{goodconjdual}
Let $\rhobar:\gK\rightarrow \widetilde P_{\rhobar}(\F)\subseteq P_{\rhobar}(\F)\subseteq G(\F)$ be a good conjugate as in \S\ref{goodconjugatesub}. Then the continuous homomorphism $\gK\rightarrow G(\F)= \GL_n(\F)$ defined by
\begin{eqnarray}\label{dualtau}
g\longmapsto w_0\tau\big(\rhobar(g)\big)^{-1}w_0
\end{eqnarray}
is a good conjugate of the dual of the representation associated to $\rhobar$.
\end{lem1}
\begin{proof}
Denote by ${}^{w_0}\!P_{\rhobar}$ the standard parabolic subgroup of $G$ with set of simple roots $-w_0(S(P_{\rhobar}))\subseteq S$. Using that $W({}^{w_0}\!P_{\rhobar})=w_0W(P_{\rhobar})w_0$, one checks that $-w_0(X_{\rhobar})\subseteq R^+$ is a closed subset relative to ${}^{w_0}\!P_{\rhobar}$ (Definition \ref{closedP}) and thus corresponds to a Zariski-closed algebraic subgroup ${}^{w_0}\!\widetilde P_{\rhobar}\defeq w_0M_{P_{\rhobar}}w_0N_{-w_0(X_{\rhobar})}$ of ${}^{w_0}\!P_{\rhobar}$ (Lemma \ref{closedX}). Denote by $w_0\tau(\rhobar)^{-1}w_0$ the homomorphism (\ref{dualtau}), its associated representation is the dual of the representation associated to $\rhobar$. Moreover one has $\widetilde P_{w_0\tau(\rhobar)^{-1}w_0}={}^{w_0}\!\widetilde P_{\rhobar}$ and $X_{hw_0\tau(\rhobar)^{-1}w_0h^{-1}}=-w_0(X_{w_0\tau(h)^{-1}w_0\rhobar w_0\tau(h)w_0})$ for any $h\in {}^{w_0}\!P_{\rhobar}(\F)$ (note that $w_0\tau(h)^{-1}w_0\in P_{\rhobar}(\F)$). The result follows from Definition \ref{gooddef}.
\end{proof}

As in \S\ref{wlgc}, if $\pi$ is a smooth representation of $G(K)$ over $\F$ we denote by $\pi^\star$ the smooth representation of $G(K)$ with the same underlying vector space as $\pi$ but where $g\in G(K)= \GL_n(K)$ acts by $\tau(g)^{-1}$. 

\begin{lem1}\label{fgz2}
Let $\rhobar:\gK\rightarrow G(\F)$ be a continuous homomorphism such that $\rhobar^{\rm ss}$ has distinct irreducible constituents and the ratio of any two irreducible constituents of dimension $1$ is not in $\{\omega,\omega^{-1}\}$. Let $\Pi$ be a smooth representation of $G(K)$ over $\F$. Then $\Pi$ is compatible with one {\upshape(}equivalently any by Proposition \ref{any}{\upshape)} good conjugate of $\rhobar$ if and only if $\Pi^\star$ is compatible with one {\upshape(}ibid.{\upshape)} good conjugate of $\rhobar^\vee\otimes \omega^{n-1}$ {\upshape(}denoting by $\rhobar^\vee$ the dual of the representation associated to $\rhobar${\upshape)}.
\end{lem1}
\begin{proof}
We use the notation in the proof of Lemma \ref{goodconjdual}. Assuming $\rhobar$ is a good conjugate, it is enough to show that if $\Pi$ is compatible with $\rhobar$, then $\Pi^\star$ is compatible with $w_0\tau(\rhobar)^{-1}w_0\otimes \omega^{n-1}$. If $R$ is a (finite-dimensional) algebraic representation of $G^{\gKQ}$ over $\F$, let $R^\star$ be the algebraic representation where $g\in G^{\gKQ}$ acts by $\tau(g)^{-1}$ (inverse transpose on each factor). Then one checks that ${\LLbar}^\star\cong \LLbar\otimes ({\det}^{-(n-1)})^{\otimes [K:\Qp]}$. Let $\Phi$ be a bijection as in Definition \ref{compatible2} and define $\Phi^\star$ from the set of subquotients ${\Pi'}^\star$ of $\Pi^\star$ (where $\Pi'$ is a subquotient of $\Pi$) to the set of good subquotients of $\LLbar\vert_{({}^{w_0}\!\widetilde P_{\rhobar})^{\gKQ}}$ as follows: $\Phi^\star({\Pi'}^\star)$ is the algebraic representation of $({}^{w_0}\!\widetilde P_{\rhobar})^{\gKQ}$ given by $\Phi^\star({\Pi'}^\star)(g)\defeq \Phi(\Pi')(w_0\tau(g)^{-1}w_0){\det(g)}^{n-1}$ for $g\in ({}^{w_0}\!\widetilde P_{\rhobar})^{\gKQ}$ (with obvious notation). We leave to the reader the tedious but formal task to check that $\Phi^\star$ satisfies all conditions of Definitions \ref{compatible1} and \ref{compatible2} with ${}^{w_0}\!\widetilde P_{\rhobar}$ and $w_0\tau(\rhobar)^{-1}w_0\otimes \omega^{n-1}$ instead of $\widetilde P_{\rhobar}$ and $\rhobar$ using (for $Q$ any standard parabolic subgroup of $G$):
\[\Big(\Ind_{Q^-(K)}^{G(K)}(\pi_1\otimes \cdots\otimes\pi_d)\Big)^\star\cong \Ind_{({}^{w_0}Q)^-(K)}^{G(K)}({\pi_d}^{\!\star}\otimes \cdots\otimes{\pi_1}^{\!\star})\]
and Lemma \ref{fgz}.
\end{proof}

\begin{prop1}
Conjecture \ref{theconj} holds for ${\tilde{v}}$ if and only if it holds for ${\tilde{v}}^c$.
\end{prop1}
\begin{proof}
This follows from Lemma \ref{fgz2} together with $\rbar_{{\tilde{v}}^c}\cong \rbar_{\tilde{v}}^\vee\otimes\omega^{1-n}$, Remark \ref{listrem}(iv) and an easy computation.
\end{proof}

There is an obvious analogous statement with Conjecture \ref{theconjbis} instead of Conjecture \ref{theconj}.

\begin{rem1}\label{duality}
Let $\pi$ be an admissible smooth representation of $G(K)$ over $\F$ with a central character. In \cite[Cor.3.15]{Ko}, Kohlhaase associates higher smooth duals $S^i(\pi)$, $i\geq 0$ to $\pi$ which are also admissible (smooth) representations of $G(K)$ over $\F$ with a central character. In view of the results when $n=2$ (see condition (iii) in \S\ref{sec:length-of-pi} below and \cite[Thm.8.2]{HuWang2}), it is natural to expect that, when $K=F_{\tilde{v}}$ and $\Pi_{\tilde{v}}$ is as in Conjecture \ref{theconj}, we have $S^i(\Pi_{\tilde{v}})\ne 0$ if and only if $i=i_0\defeq [K:\Q_p]\frac{n(n-1)}{2}$ and that $S^{i_0}(\Pi_{\tilde{v}})$ is compatible with (a good conjugate of) $\rbar_{\tilde{v}}^\vee\otimes \omega^{n-1}$ (when $n=2$, this is indeed consistent with {\it loc.cit.}\ since $\rbar_{\tilde{v}}^\vee\cong \rbar_{\tilde{v}}\otimes {\det}(\rbar_{\tilde{v}})^{-1}$). It is also natural to ask if we have $S^{i_0}(\Pi_{\tilde{v}})\simeq \Pi_{\tilde{v}}^\star$ (see Lemma \ref{fgz2}).
\end{rem1}

From the results of \cite[\S4.4]{BH} and \cite{Enns2}, we can at least give some very weak evidence for Conjecture \ref{theconj}, more precisely for the stronger Conjecture \ref{theconjbis} in Remark \ref{bandebis}(iii), when $p$ is totally split in $F^+$ and $\rbar_{\tilde{w}}$ is upper-triangular sufficiently generic for all ${{w}}\vert p$ in $F^+$. 

If $\Pi$ is an admissible smooth representation of $G(K)$ over $\F$, we denote by $\Pi^{\rm ord}\subseteq \Pi$ the maximal $G(K)$-subrepresentation such that all its irreducible constituents are isomorphic to irreducible subquotients of principal series of $G(K)$ over $\F$. The following lemma is not difficult using Proposition \ref{minimal}, \cite[Thm.2.2.4]{BH} and the results of \cite[\S3.3]{BH}, \cite[\S3.4]{BH} (the proof is left to the reader).

\begin{lem1}
Assume $K=\Qp$ and let $\rhobar:\gp\rightarrow B(\F)\subseteq G(\F)$ be generic {\upshape(}as at the beginning of \S\ref{compatibilityrhobar}{\upshape)} and a good conjugate {\upshape(}as in Definition \ref{gooddef}{\upshape)}. Let $\Pi$ be compatible with $\rhobar$ {\upshape(}as in Definition \ref{compatible2}{\upshape)}. Then $\Pi^{\rm ord}\simeq \Pi(\rhobar)^{\rm ord}$, where $\Pi(\rhobar)^{\rm ord}$ is the representation of $G(\Qp)$ over $\F$ defined in {\upshape\cite[\S3.4]{BH}}.
\end{lem1}

Note that one can explicitly determine $V_G(\Pi(\rhobar)^{\rm ord})$ inside $\LLbar(\rhobar)$, see \cite[\S9]{breuil-foncteur}.

We let $S_p$ be the set of places of $F^+$ dividing $p$. Recall that an injection between two representations of a group is called {\it essential} if it induces an isomorphism on the respective socles.

\begin{thm1}[\cite{Enns2}]\label{enns}
Assume that $F/F^+$ is unramified at finite places, that $H$ is defined over $\oFF$ with $H\times_{\oFF}F^+$ quasi-split at finite places of $F^+$, and that $p$ is totally split in $F$. Assume that $\rbar:\gF\rightarrow {\GL}_n(\F)$ satisfies assumptions A1 to A6 of {\upshape\cite[\S3.1]{Enns2}}, let $v_1$ be a finite place of $F^+$ as in {\upshape\cite[Lemma 3.1.2]{Enns2}} and $\Sigma\defeq S_p\cup \{v_1\}$. Choose $\widetilde{v_1}\vert v_1$ in $F$ and let $U^p=\prod_{w\nmid p}U_{w}\subseteq \GAp$ such that $U_{w}=H({\mathcal O}_{F_{w}^+})$ if $w$ splits in $F$, $U_{w}$ is maximal hyperspecial in $H(F_{w}^+)$ if $w$ is inert in $F$ and $\iota_{\widetilde{v_1}}(U_{v_1})$ is the Iwahori subgroup of $\GL_n(F_{\widetilde{v_1}})$. Then for any $\tilde w\vert w$ in $F$ and any good conjugates $\rbar_{\tilde{w}}$ {\upshape(}where $w\in S_p${\upshape)}, we have an {\rm essential} injection of admissible smooth representations of $\prod_{w\vert p}H(F^+_w)$ over $\F$:
\[\Big(\bigotimes_{w\vert p}\big(\Pi(\rbar_{\tilde{w}})^{\rm ord}\otimes \omega^{n-1}\circ{\det}\big)\Big)^{\oplus n!}\hookrightarrow S(U^p,\F)[\m^{\Sigma}]^{\rm ord},\]
where $S(U^p,\F)[\m^{\Sigma}]^{\rm ord}\subseteq S(U^p,\F)[\m^{\Sigma}]$ is defined as $\Pi^{\rm ord}\subseteq \Pi$ above replacing $G(K)$ by $\prod_{w\vert p}H(F^+_w)$.
\end{thm1}
\begin{proof}
This follows from \cite[Thm.3.3.3]{Enns2} (which itself improves \cite[Thm.4.4.7]{BH}) and its proof (see just before \cite[Lemma 3.2.1]{Enns2} for the $n!$).
\end{proof}

\begin{rem1}
The cokernel of the injection in Theorem \ref{enns} is an admissible smooth representation of $\prod_{w\vert p}H(F^+_w)$ over $\F$, and its $\prod_{w\vert p}H(F^+_w)$-socle is by construction a direct sum of finitely many irreducible subquotients of principal series. If we could prove that all these irreducible subquotients are irreducible principal series which do not appear in the $\prod_{w\vert p}H(F^+_w)$-socle of $\bigotimes_{w\vert p}(\Pi(\rbar_{\tilde{w}})^{\rm ord}\otimes \omega^{n-1}\circ{\det})$, then it would follow from the mod $p$ version of \cite[Cor.1.4]{Ha3} that the essential injection in Theorem \ref{enns} is an isomorphism.
\end{rem1}

\newpage

\section{The case of \texorpdfstring{$\GL_2(\Q_{p^f})$}{GL\_2(Q\_\{p\^{}f\})}}\label{gl2}

We give evidence for Conjecture \ref{theconjbar} and Conjecture \ref{theconj} when $F^+_v$ is unramified and $G=\GL_2$. We now assume $K=\Q_{p^f}$ and $n=2$ till the end. We fix an embedding $\sigma_0 : \F_{p^f}=\Fq\hookrightarrow\F$ and let $\sigma_i\defeq \sigma_0\circ\varphi^i$ for $\varphi$ the arithmetic Frobenius and $i\geq 0$.

\subsection{\texorpdfstring{$(\varphi,\cO_K^\times)$}{{(phi,O\_K\^{}x)}}-modules and \texorpdfstring{$(\varphi,\Gamma)$}{(phi,Gamma)}-modules}\label{phigamma}

We associate \'etale $(\varphi,\cO_K^\times)$-modules to certain admissible smooth representations of $\GL_2(K)$ over $\F$ and relate them to the \'etale $(\varphi,\Gamma)$-modules of \S\ref{covariant}.

We assume $p>2$. We let $I\defeq \smatr {\cO_K^\times}{\cO_K}{p\cO_K}{\cO_K^\times}$ be the Iwahori subgroup of $\GL_2(\cO_K)$, $K_1\defeq \smatr {1+p\cO_K}{p\cO_K}{p\cO_K}{1+p\cO_K}$ the pro-$p$ radical of $\GL_2(\cO_K)$, $I_1\defeq \smatr {1+p\cO_K}{\cO_K}{p\cO_K}{1+p\cO_K}$ the pro-$p$ radical of $I$, $N_0\defeq \smatr {1}{\mathcal{O}_K}{0}{1}\subseteq I_1$, $N_0^-\defeq \smatr {1}{0}{p\mathcal{O}_K}{1}\subseteq I_1$ and $T_0\defeq \smatr {1+p\cO_K}{0}{0}{1+p\cO_K}\subseteq I_1$. We denote by $Z_1$ the center of $I_1$. If $C$ is a pro-$p$ group then $\F\bbra{C}$ denotes its Iwasawa algebra over $\F$, which is a local ring, and $\m_{C}$ the maximal ideal of $\F\bbra{C}$. If $R$ (resp.\ $M$) is a filtered ring (resp.\ filtered module) in the sense of \cite[\S I.2]{LiOy}, we denote by $F_nR$ (resp.\ $F_nM$) for $n\in \Z$ its ascending filtration and $\gr(R)\defeq \oplus_{n\in \Z}F_nR/F_{n-1}R$ (resp.\ with $M$) the associated graded ring (resp.\ module). When $R=\F\bbra{C}$, we set $F_nR\defeq \m_R^{-n}$ if $n\leq 0$ and $F_nR\defeq R$ if $n\geq 0$. If $M$ is an $R$-module, the filtration defined by $F_nM=\m_R^{-n}M$ if $n\leq 0$ and $F_nM=M$ if $n\geq 0$ is called \emph{the $\m_R$-adic filtration on $M$}.

\subsubsection{The ring \texorpdfstring{$A$}{A}}\label{ringA}

We describe some properties of a complete noetherian ring $A$ which will be a coefficient ring for some multivariable $(\psi,\cO_K^\times)$-modules and $(\varphi,\cO_K^\times)$-modules.

Let $v_{N_0}$ be the $\m_{N_0}$-adic valuation on the ring $\F\bbra{N_0}$ defined by the $\m_{N_0}$-adic filtration (i.e.~$F_n\F\bbra{N_0}=\{x\in \F\bbra{N_0} : v_{N_0}(x)\geq -n\}$ for $n\in \Z$). We use the same notation to denote the unique extension of $v_{N_0}$ to a valuation of the fraction field of $\F\bbra{N_0}$. For $i\in \{0,\dots,f-1\}$ let
\begin{equation}\label{elementsyi}
Y_i\defeq \sum_{a\in\F_q^\times}\sigma_0(a)^{-p^i}
  \begin{pmatrix}
    1 & \tld{a} \\ 0 & 1
  \end{pmatrix}\in \m_{N_0}\backslash \m_{N_0}^2
  \end{equation}
(where $\tilde a \in \cO_K^\times$ denotes the Teichm\"uller lift of $a$) and write $y_i\defeq \gr(Y_i)$ for the image of $Y_i$ in $ \m_{N_0}/ \m_{N_0}^2\subseteq \gr(\F\bbra{N_0})$. Then $\F\bbra{N_0}$ is isomorphic to the power series ring $\F\bbra{Y_0,\dots,Y_{f-1}}$ and $\gr(\F\bbra{N_0})$ to the polynomial algebra $\F[y_0,\dots,y_{f-1}]$. Let $S$ be the multiplicative subset of $\F\bbra{N_0}$ whose elements are the $(Y_0\cdots Y_{f-1})^n$ for $n\geq0$, $\F\bbra{N_0}_S$ the corresponding localization and $F_n\F\bbra{N_0}_S\defeq \{x\in \F\bbra{N_0}_S : v_{N_0}(x)\geq -n\}$. We define the ring $A$ as the completion of the filtered ring $\F\bbra{N_0}_S$ (\cite[\S I.3.4]{LiOy}). Note that $v_{N_0}$ extends to $A$, which is thus a complete filtered ring. As $A$ is complete, an element $x\in A$ is invertible in $A$ if and only if $\gr(x)$ is invertible in $\gr(A)$ (as is easily checked, here $\gr(x)$ is the ``principal part'' of $x$ as in \cite[\S I.4.2]{LiOy}).

Let $M$ be a filtered $\F\bbra{N_0}$-module. The tensor product $A\otimes_{\F\bbra{N_0}}M$ is then a filtered $A$-module for the
tensor product filtration as defined in \cite[p.57]{LiOy}. We let $A\widehat{\otimes}_{\F\bbra{N_0}}M$ be its completion. This filtered $A$-module can also be described as the completion of the localization $M_S$ endowed with the tensor product filtration associated to the isomorphism $M_S\simeq\F\bbra{N_0}_S\otimes_{\F\bbra{N_0}}M$.

\begin{lem}\label{lem:grlocal}
We have an isomorphism
  \begin{equation}\label{eq:tensorgr}
    \gr(A\widehat{\otimes}_{\F\bbra{N_0}}M)\simeq\gr(M_S)\simeq\gr(M)[(y_0\cdots
    y_{f-1})^{-1}].
\end{equation}
\end{lem}
\begin{proof}
As $A\widehat{\otimes}_{\F\bbra{N_0}}M$ is the completion of $M_S$, it
is sufficient to prove that $\gr(M_S)\simeq\gr(M)[(y_0\cdots
    y_{f-1})^{-1}]$. Note that we have an
isomorphism of $\F\bbra{N_0}$-algebras
\!$\F\bbra{N_0}_S\simeq\F\bbra{N_0}[T]/((Y_0\cdots
Y_{f-1})T-1)$. Moreover if we endow the ring $\F\bbra{N_0}[T]$ with the
filtration
\[ F_n(\F\bbra{N_0}[T])=\sum_{k\geq0}\m_{N_0}^{kf-n}T^k\] (with the
convention $\m_{N_0}^i=\F\bbra{N_0}$ for $i\leq0$), the filtration on
$\F\bbra{N_0}_S$ is the quotient filtration via
$\F\bbra{N_0}[T]\twoheadrightarrow\F\bbra{N_0}_S$. Therefore the
filtration on $M_S$ is the quotient filtration of the tensor product
filtration on $M[T]\defeq \F\bbra{N_0}[T]\otimes_{\F\bbra{N_0}}M$.

As the filtered $\F\bbra{N_0}$-module $\F\bbra{N_0}[T]$ is
filtered-free by definition (see \cite[Def.I.6.1]{LiOy}), it follows from
\cite[Lemma I.6.14]{LiOy} that $\gr(M[T])\simeq\gr(M)[T]$ with
$\deg(T)=f$. We claim that the following sequence of filtered
modules is strict exact:
\[M[T]\xrightarrow{(Y_0\cdots Y_{f-1})T-1} M[T]\longrightarrow
M_S\longrightarrow0.\] Namely the exactness of the second arrow
follows from the definition of the quotient filtration. As
$(Y_0\cdots Y_{f-1})T$ and $1$ have degree $0$ in
$\F\bbra{N_0}[T]$, the multiplication by $(Y_0\cdots Y_{f-1})T-1$
induces the multiplication by $(y_0\cdots y_{f-1})T-1$ on
$\gr(M[T])\simeq\gr(M)[T]$ which is injective. It follows from
\cite[Thm.I.4.2.4(2)]{LiOy} (applied with $L=0$, $M=N=M[T]$,
$f=0$ and $g$ being the multiplication by $(Y_0\cdots Y_{f-1})T-1$) that
the multiplication by $(Y_0\cdots Y_{f-1})T-1$ is a strict map.

It then follows from \cite[Thm.I.4.2.4(1)]{LiOy} that the following sequence is exact:
  \begin{equation}\label{eq:gr_of_localization}
    \gr(M[T])\xrightarrow{(y_0\cdots
      y_{f-1})T-1}\gr(M[T])\longrightarrow\gr(M_S)\longrightarrow0.
   \end{equation}
Finally, since $\gr(M[T])\simeq\gr(M)[T]$, we have
$\gr(M_S)\simeq\gr(M)[(y_0\cdots y_{f-1})^{-1}]$.
\end{proof}

\begin{cor}\label{lem:grA}
We have an isomorphism $\gr(A)\simeq\F[y_0,\dots,y_{f-1},(y_0\cdots
y_{f-1})^{-1}]$. As a consequence the ring $A$ is a regular domain, i.e.~a noetherian domain which has a finite global dimension {\upshape(\cite[\S IV.D]{Serre-local-algebra})}.
\end{cor}
\begin{proof}
The first sentence is a direct consequence of Lemma \ref{lem:grlocal}
applied with $M=\F\bbra{N_0}$. This implies that the ring $\gr(A)$
is a noetherian domain. Then the noetherianity of $A$ follows from
\cite[Thm.I.5.7]{LiOy} applied to the ideals of $A$, and the fact that $A$ is a domain
follows easily from $\gr(x)\gr(y)=\gr(xy)$ if $x,y\in A\backslash\{0\}$ (using $\gr(x)\gr(y)\ne 0$). 
As $\gr(A)$ is a regular commutative ring, it follows from \cite[Thm.III.2.2.5]{LiOy} that $A$ is 
an Auslander regular ring (note that $A$ is Zariskian by \cite[Prop.II.2.2.1]{LiOy}) and a
fortiori has finite global dimension (\cite[Def.III.2.1.7]{LiOy}).
\end{proof}

\begin{rem}\label{rem:on_filtered_Amod}
(i) The ring $A$ can also be defined as the microlocalization of $\F\bbra{N_0}$ along the set
$\set{(y_0\cdots y_{f-1})^n, n\geq1}\subset\gr(\F\bbra{N_0})$ (see \cite[Cor.IV.1.20]{LiOy}). This
shows that the ring $A$ does not depend on our choice of elements
$Y_i$ but rather on the elements $y_i$.\\
(ii) If $M$ is a filtered $\F\bbra{N_0}$-module, the filtration on $M_S$ is given explicitly
by the following formula:
\[ F_n(M_S)=\sum_{k\geq0}(Y_0\cdots Y_{f-1})^{-k}F_{n-kf}(M), \quad
n\in\Z.\]
As $(Y_0\cdots Y_{f-1})^mF_n(M)\subset F_{n-mf}(M)$ for all
$n\in\Z$ and $m\in\NN$, we have
\[ (Y_0\cdots Y_{f-1})^{-k}F_{n-kf}(M)\subset (Y_0\cdots
Y_{f-1})^{-k-1}F_{n-(k+1)f}(M)\]
so that $F_n(M_S)$ can also be described as the increasing union
\[ F_n(M_S)=\bigcup_{k\geq0}(Y_0\cdots Y_{f-1})^{-k}F_{n-kf}(M).\]
Note that the filtration on $M_S$ is not necessarily separated even if the filtration on $M$ is separated.\\
(iii) The ring $A$ can also be defined as the set of series
\[A=\set*{\sum_{d\gg-\infty}\frac{P_{d}}{(Y_0\cdots Y_{f-1})^{n_d}}, 
P_d\in(Y_0,\dots,Y_{f-1})^{d+fn_d}, \ n_d\geq0, \ d+fn_d\geq0},\]
equivalently, $A$ is the set of infinite sums of monomials in the $Y_i$ with
$\F$-coefficients such that the total degree of the monomials tends to $+\infty$.
\end{rem}

Let $n\geq0$ be an integer and let $N_0^{p^n}\subset N_0$ be the
subgroup of $p^n$-th powers (which is $p^n\cO_K$ under the
identification $N_0\simeq\cO_K$). Let $S^{p^n}$ be the set of
$p^n$-th powers of $S$ and let $A^{p^n}$ be the completion of
$\F\bbra{N_0^{p^n}}_{S^{p^n}}$ for the filtration coming from the
valuation $v_{N_0}|_{\F\bbra{N_0^{p^n}}}=p^nv_{N_0^{p^n}}$. As the saturation of $S^{p^n}$ (see
 \cite[\S IV.1]{LiOy}) contains $S$, we have by \cite[Cor.IV.1.20]{LiOy}
\begin{equation}\label{Spn}
\F\bbra{N_0}_S=\F\bbra{N_0}_{S^{p^n}}\simeq\F\bbra{N_0^{p^n}}_{S^{p^n}}\otimes_{\F\bbra{N_0^{p^n}}}\F\bbra{N_0}.
\end{equation}
It is easy to check that $\F\bbra{N_0}$ is a
filtered free $\F\bbra{N_0^{p^n}}$-module with respect to the basis
$(Y_0^{i_0}\cdots Y_{f-1}^{i_{f-1}})_{\substack{0\leq i_j\leq p^n-1 \\
    0\leq j\leq f-1}}$. Hence, by \cite[Lemma I.6.15]{LiOy} and
(\ref{Spn}), we conclude that $\F\bbra{N_0}_S$ is a filtered free
$\F\bbra{N_0^{p^n}}_{S^{p^n}}$-module with respect to the same basis,
and thus by \cite[Lemma I.6.13(3)]{LiOy} that $A$ is a filtered free
$A^{p^n}$-module with respect to the same basis again. Moreover, by
\cite[Lemma I.6.14]{LiOy}, we have an isomorphism of graded modules
\begin{equation}\label{grApn}
\gr(A)\simeq\gr(A^{p^n})\otimes_{\gr(\F\bbra{N_0^{p^n}})}\gr(\F\bbra{N_0}).
\end{equation}

Note that the $p^n$-power Frobenius map
$x\mapsto x^{p^n}$ induces an isomorphism of filtered rings
$(\F\bbra{N_0}_S,v_{N_0})\congto(\F\bbra{N_0^{p^n}}_{S^{p^n}},v_{N_0^{p^n}})$
and thus, as $v_{N_0}|_{\F\bbra{N_0^{p^n}}}=p^{n}v_{N_0^{p^n}}$, an
isomorphism of topological rings
$(\F\bbra{N_0}_S,v_{N_0})\congto(\F\bbra{N_0^{p^n}}_{S^{p^n}},v_{N_0}|_{\F\bbra{N_0^{p^n}}})$. It induces an isomorphism
of complete topological rings
$A\congto A^{p^n}$ such that
the composite map
$A\congto A^{p^n}\hookrightarrow A$ is the
$p^n$-power Frobenius. This implies that the image of $A^{p^n}$ in $A$ is
the subring of $p^n$-th powers of $A$. 

The group $\cO_K^\times$ acts on the group $N_0$ via $a\cdot\left(
  \begin{smallmatrix}
    1&b\\0&1
  \end{smallmatrix}\right)=\left(
  \begin{smallmatrix}
    1&ab\\0&1
  \end{smallmatrix}\right)$ 
and thus on $\F\bbra{N_0}$,
preserving the valuation $v_{N_0}$, and hence the filtration. This induces an action of $\cO_K^\times$
on the graded ring $\gr(\F\bbra{N_0})$, where it is immediately checked that $1+p\cO_K$ acts trivially. 
Moreover if $a\in\F_q^\times$ and $0\leq i\leq f-1$, we have $\tld{a}\cdot y_i=\sigma_i(a)y_i$.

\begin{lem}\label{actionoK}
There is a unique continuous action of $\cO_K^\times$ on the ring
$A$ extending the action of $\cO_K^\times$ on $\F\bbra{N_0}$.
\end{lem}
\begin{proof}
As $\cO_K^\times$ acts by ring endomorphisms on $\F\bbra{N_0}$ and as
$\F\bbra{N_0}_S$ is dense in $A$, the uniqueness is clear.

For the existence, let $a\in\cO_K^\times$ and consider the composition
$\F\bbra{N_0}\xrightarrow{a}\F\bbra{N_0}\subset A$ which extends to
a ring homomorphism $\F\bbra{N_0}_S\rightarrow A$ since the elements
of $a(S)$ are invertible in $A$ (because they are invertible in
$\gr(A)$ as $\gr(a(S))=\gr(S)$). The precomposition of the valuation
$v_{N_0}$ on $A$ with this map is a valuation on $\F\bbra{N_0}_S$
which coincides with $v_{N_0}$ on $\F\bbra{N_0}$ since the
multiplication by $a$ preserves the valuation on
$\F\bbra{N_0}$. Therefore the map $\F\bbra{N_0}_S\rightarrow A$ is
isometric and extends to a filtered ring homomorphism $A\rightarrow A$
(\cite[Thm.I.3.4.5]{LiOy}). This defines an action of $\cO_K^\times$
on $A$.
\end{proof}

We recall that $\xi$ is the cocharacter $x\mapsto \left(
  \begin{smallmatrix}
    x&0\\0&1
  \end{smallmatrix}\right)$ of $\GL_2$. The conjugation by the matrix
$\xi(p)$ in $\GL_2(K)$ induces a group endomorphism of $N_0$ and a
continuous endomorphism $\phi$ of $\F\bbra{N_0}$. We have
$\phi(Y_i)=Y_{i-1}^p$ for $1\leq i\leq f-1$ and
$\phi(Y_0)=Y_{f-1}^p$. This implies that $\phi$ is the composite of
the (relative) Frobenius endomorphism with a permutation of the
variables $Y_i$. It follows that $\phi$ extends to a continuous
injective endomorphism of the ring $A$ with image $A^{p}$. More
generally, for $n\geq0$, the subring $A^{p^n}$ is the
image of $\phi^n$.

\begin{prop}\label{prop:control}
  Let $H\subset\cO_K^\times$ be an open subgroup and let
  $\mathfrak{a}\subset A$ be an ideal of $A$ which is $H$-stable.
  Then $\mathfrak{a}$ is controlled by $A^p$, which means
  \[ \mathfrak{a}=A(\mathfrak{a}\cap A^p).\]
\end{prop}
\begin{proof}
  As $H$ is open in $\cO_K^\times$ it contains a subgroup of the form
  $1+p^m\cO_K$ for $m\geq1$ so that we can assume that $H=1+p^m\cO_K$.
  
The proof follows closely the strategy of \cite{ardakov-wadsley09}.

We note that the pair $(A,A^p)$ is a \emph{Frobenius pair} in the
sense of \cite[Def.2.1]{ardakov-wadsley09} (to see this use \cite[Prop.6.6]{AWZ08}
applied to $G=N_0$ together with \cite[Lemma 2.2.(a)]{ardakov-wadsley09} and
Remark \ref{rem:on_filtered_Amod}(i)). We endow $A^p$ with the filtration 
$F_nA^p\defeq A^p\cap F_nA$ induced by the filtration of $A$.

Let $F\defeq \mathfrak{a}/A(\mathfrak{a}\cap A^p)$. Endow $A(\mathfrak{a}\cap A^p)$ and $\mathfrak{a}$ with the filtration 
induced by $A$, and $F$ with the quotient filtration. Then by \cite[Rk.I.5.2(2)]{LiOy}
and \cite[Cor.I.5.5(1)]{LiOy} all these filtrations are good in the sense of \cite[Def.I.5.1]{LiOy}.
Moreover $\mathfrak{a}$ and $A(\mathfrak{a}\cap A^p)$ are complete filtered $A$-modules 
by \cite[Cor.I.6.3(2))]{LiOy} and thus so is $F$ by \cite[Prop.I.3.15]{LiOy}.

We want to prove that $F=0$. Assume for a contradiction that $F\neq0$, or equivalently $\gr(F)\neq0$ by \cite[Prop.I.4.2(1)]{LiOy}.

Let $\Gamma\defeq H=1+p^m\cO_K$ (this not the $\Gamma$ of the
$(\varphi,\Gamma)$-modules!). This is a uniform pro-$p$-group. Note
that the action of $\Gamma$ on $N_0$ is \emph{uniform} in the sense
of \cite[\S4.1]{ardakov-wadsley09}. In the notation of
\cite[\S4.2]{ardakov-wadsley09}, we have $L_{N_0}=\cO_K$,
$\mathfrak{g}=\F_q$ and the action of $\F_q$ on $L_{N_0}/pL_{N_0}$
is given by the multiplication in $\F_q$.

Let $P$ be a (homogeneous) prime ideal in the support of the
$\gr(A)$-module $\gr(F)$ (which exists since $\gr(F)\neq0$).

Let $x\in\F^\times_q$ and $\gamma_x\defeq \exp(p^m[x])\in \Aut(N_0)\hookrightarrow \End(A)$. It follows from
\cite[Prop.4.4]{ardakov-wadsley09} and \cite[Prop.3.2(a)]{ardakov-wadsley09} that the family
\[\mathrm{a}(x)\defeq (\gamma_x,\gamma_x^p,\gamma_x^{p^2},\dots)\]
is a source of derivations of $(A,A^p)$ in the sense of
\cite[Def.3.2]{ardakov-wadsley09}. Let $T_P\subset\gr(A)$ be the set of
homogeneous elements of $\gr(A)$ which are not in $P$ and let
$T_P^{(p)}\defeq T_P\cap\gr(A^p)$. It follows again from
\cite[Prop.3.2(a)]{ardakov-wadsley09} that $\mathrm{a}(x)$ induces on
$(Q_{T_P}(A),Q_{T_P^{(p)}}(A^p))$ a source of derivations
$\mathrm{a}_{T_P}(x)$, where $Q_{T_P}(A)$ (resp.~$Q_{T_P^{(p)}}(A^p)$) is
the microlocalization of $A$ (resp.~$A^p$) with respect to $T_P$
(resp.~$T_P^{(p)}$). Let $\mathcal{S}\defeq \set{\mathrm{a}(x), 
x\in\F_q^\times}$ and $\mathcal{S}_P\defeq \set{\mathrm{a}_{T_P}(x), 
x\in\F_q^\times}$.

As $\mathfrak{a}$ is $\Gamma$-invariant, $\mathfrak{a}$ is also
$\mathcal{S}$-invariant, i.e.~for all $x\in\F_q^\times$ and $r\geq0$, we have
$\gamma_x^{p^r}\mathfrak{a}\subset \mathfrak{a}$. Then $\mathfrak{a}_P\defeq  Q_{T_P}(\mathfrak{a})
\cong Q_{T_P}(A)\otimes_A\mathfrak{a}$ (\cite[Cor.IV.1.18(2)]{LiOy}, though here everything is simpler 
as all rings are commutative) is an ideal of $Q_{T_P}(A)$ 
which is $\mathcal{S}_P$-invariant.

Let $P_0\defeq  P\cap\gr(\F\bbra{N_0})$ (inside $\gr(A)$). We prove that $P_0$ contains $L_{N_0}/pL_{N_0}$,
where the latter is seen in $\gr_{-1}(\F\bbra{N_0})$ (recall $L_{N_0}\cong {N_0}$). Assume this is not true. 
Let $J\defeq \gr(\mathfrak{a}_P)\simeq\gr(\mathfrak{a})_P$ (\cite[Lemma 4.4]{AWZ08}), which is a graded ideal of 
the localization $\gr({A})_P$ of $\gr({A})$ with respect to the set of homogeneous elements which are not 
in $P$, and let $Y\in\gr(A)_P$ such that $Y\in J^{\mathcal{S}_P}$ (see \cite[Def.3.4]{ardakov-wadsley09} 
for the definition of $J^{\mathcal{S}_P}$). Noticing that $\gr(A)_P=\gr(\F\bbra{N_0})_{P_0}$ and that 
$L_{N_0}/pL_{N_0}$ is a $1$-dimensional $\F_q$-vector space, we can apply \cite[Prop.4.3]{ardakov-wadsley09} 
(together with \cite[Prop.4.4(c)]{ardakov-wadsley09}) to the graded prime ideal $P_0$ of $B=\gr(\F\bbra{N_0})$ and
the graded ideal $J$ of $\gr(\F\bbra{N_0})_{P_0}$. We deduce $\mathcal{D}_P(Y)\subset J$ (see 
\cite[\S4.3]{ardakov-wadsley09} for the definition of $\mathcal{D}_P$). It follows from
\cite[Thm.3.5]{ardakov-wadsley09} applied to the Frobenius pair $(Q_{T_P}(A),Q_{T_P^{(p)}}(A^p))$ and the 
ideal $\mathfrak{a}_P$ that $\mathfrak{a}_P$ is controlled by $Q_{T_P^{(p)}}(A^p)$. Then
\cite[Lemma 2.3]{ardakov-wadsley09} shows that $\gr(F)_P=0$. This is
a contradiction.

As $L_{N_0}/pL_{N_0}$ generates the $\F$-vector space $\gr_{-1}(\F\bbra{N_0})=\oplus_{i=0}^{f-1} \F y_i$,
it follows that $y_i\in P$ for all $0\leq i\leq f-1$ and then that $\gr(A)=P$. This is a contradiction
so that $F=0$ i.e.\ $\mathfrak{a}=A(\mathfrak{a}\cap A^p)$.
\end{proof}

\begin{lem}\label{lem:intersection_ideals}
Let $\mathfrak{a}\subsetneq A$ be a proper ideal of $A$. Then
$\cap_{n\geq0}(A(\mathfrak{a}\cap A^{p^n}))=0$. In particular, if
$\phi(\mathfrak{a}) \subset \mathfrak{a}$ we have
$\cap_{n\geq0}A\phi^n(\mathfrak{a})=0$.
\end{lem}
\begin{proof}
Let $\mathfrak{a}_n\defeq A(\mathfrak{a}\cap A^{p^n})$. We endow
$\mathfrak{a}\cap A^{p^n}$ with the induced filtration of
$A^{p^n}$ (or equivalently $A$). As $A$ is a finite free $A^{p^n}$-module, we have
$\mathfrak{a}_n\simeq A\otimes_{A^{p^n}}(\mathfrak{a}\cap
A^{p^n})$. We endow this $A$-module with the tensor product
filtration. Since $A$ is a filtered free $A^{p^n}$-module, it
follows from \cite[Lemma I.6.14]{LiOy} that
$\gr(\mathfrak{a}_n)\simeq\gr(A)\otimes_{\gr(A^{p^n})}\gr(\mathfrak{a}\cap
A^{p^n})$. Since $\gr(A)$ is a finite free $\gr(A^{p^n})$-module,
the natural map $\gr(\mathfrak{a}_n)\rightarrow\gr(A)$ is injective (and the filtration on $\mathfrak{a}_n$ 
is in fact the one induced from $A$). Moreover from (\ref{grApn}) we deduce
\begin{equation}\label{gran}
\gr(\mathfrak{a}_n)\simeq \gr(\F\bbra{N_0})\otimes_{\gr(\F\bbra{N_0^{p^n}})}\gr(\mathfrak{a}\cap A^{p^n}).
\end{equation}

Assume that $\mathfrak{a}\neq A$. Then as both $\mathfrak{a}$ and $A$ are complete and the injection $\mathfrak{a}\hookrightarrow A$ is strict, it follows as for the $A$-module $F$ in the proof of Proposition \ref{prop:control} that $\gr(A/\mathfrak{a})\ne 0$ (with the quotient filtration on $A/\mathfrak{a}$), hence by \cite[Thm.I.4.4(1)]{LiOy} that $\gr(\mathfrak{a})\neq\gr(A)$, and {\it a fortiori} $\gr(\mathfrak{a}_n)\neq\gr(A)$.

Using (\ref{gran}) and the fact $\gr(\F\bbra{N_0})\simeq\F[y_0,\dots,y_{f-1}]$ is free of finite rank over $\gr(\F\bbra{N_0^{p^n}})\simeq\F[y_0^{p^n},\dots,y_{f-1}^{p^n}]$, we have inside $\gr(A)$ that
\begin{equation}
\gr(\mathfrak{a}_n)\cap\gr(\F\bbra{N_0})\simeq \gr(\F\bbra{N_0})\otimes_{\gr(\F\bbra{N_0^{p^n}})}(\gr(\mathfrak{a}\cap A^{p^n})\cap\gr(\F\bbra{N_0^{p^n}})).
\end{equation}
The ideal $\gr(\mathfrak{a}_n)\cap\gr(\F\bbra{N_0})$ is therefore
generated by homogeneous elements of $\gr(\F\bbra{N_0})$ which are of degree $\leq -p^n$
since homogeneous elements of $\F[y_0^{p^n},\dots,y_{f-1}^{p^n}]$
of degree zero are invertible and $\gr(\mathfrak{a}_n)$ does not
contain invertible elements (as $\gr(\mathfrak{a}_n)\neq\gr(A)$). We conclude that
\[\gr(\mathfrak{a}_n)\cap\gr(\F\bbra{N_0})\subset F_{-p^n}(\gr(\F\bbra{N_0})).\]
Consequently (recall $\bigcap_{n\geq0}\mathfrak{a}_n$ has the induced filtration from $A$)
\begin{equation}\label{grinteran}
\gr\big(\bigcap_{n\geq0}\mathfrak{a}_n\big)\cap\gr(\F\bbra{N_0})\subset\bigcap_{n\geq0}(\gr(\mathfrak{a}_n)\cap\gr(\F\bbra{N_0}))=0.
\end{equation}
As $\gr(\bigcap_{n\geq0}\mathfrak{a}_n)$ is an ideal in $\gr(A)\simeq\F[y_0,\dots,y_{f-1},(y_0\cdots
y_{f-1})^{-1}]$, it follows from (\ref{grinteran}) that we must have $\gr(\bigcap_{n\geq0}\mathfrak{a}_n)=0$,
and hence that $\bigcap_{n\geq0}\mathfrak{a}_n=0$ by \cite[Prop.I.4.2(1)]{LiOy}.
\end{proof}

\begin{cor}\label{cor:idealsofA}
Let $H\subset\cO_K^\times$ be an open subgroup. The only ideals of $A$ which are $H$-stable are $0$
and $A$.
\end{cor}
\begin{proof}
Let $\mathfrak{a}$ be such an ideal and assume that $\mathfrak{a}\neq A$. It follows from
Proposition \ref{prop:control} applied recursively with $A$, $A^p$, etc. that
$\mathfrak{a}=A(\mathfrak{a}\cap A^{p^n})$ for all $n\geq0$. Then Lemma \ref{lem:intersection_ideals} 
implies $\mathfrak{a}=0$.
\end{proof}

If $H$ is an open subgroup of $\cO_K^\times$, an \emph{$H$-module over $A$} is a finitely generated
$A$-module with a semilinear action of $H$.

\begin{prop}\label{prop:OK_modules}
  Let $H$ be an open subgroup of $\cO_K^\times$ and let $M$ be an
  $H$-module over $A$. Then $M$ is a finite projective $A$-module.
\end{prop}
\begin{proof}
  (We thank Gabriel Dospinescu for suggesting the following proof
  which is shorter than our original one.) Let $M$ be an
  $H$-module. For $k\geq-1$ let $\mathrm{Fit}_k(M)$ be the
  $k$-th Fitting ideal (see for example
  \cite[\href{https://stacks.math.columbia.edu/tag/07Z9}{Def.07Z9}]{stacks-project}). As $M$ is a finitely generated
  $A$-module, it follows from
  \cite[\href{https://stacks.math.columbia.edu/tag/07ZA}{Lemma
    07ZA}]{stacks-project} that there exists some $r\geq0$ such that
  $\mathrm{Fit}_r(M)\neq0$. Let $r\geq0$ be the smallest integer such
  that $\mathrm{Fit}_r(M)\neq0$. Let $\gamma\in H$. It
  follows easily from the definition of $\mathrm{Fit}_k(M)$ that
  $\mathrm{Fit}_k(M\otimes_{A,\gamma}A)=\gamma(\mathrm{Fit}_k(M))$ as
  ideals of $A$. The action of $\gamma$ on $M$ induces an $A$-linear
  isomorphism $M\otimes_{A,\gamma}A\xrightarrow{\sim} M$, showing that
  $\gamma(\mathrm{Fit}_k(M))=\mathrm{Fit}_k(M)$. It follows then from
  Corollary \ref{cor:idealsofA} that all the ideals
  $\mathrm{Fit}_k(M)$ are zero or $A$. Therefore we have
  $\mathrm{Fit}_{r-1}(M)=0$ and $\mathrm{Fit}_r(M)=A$ and we deduce
  from \cite[\href{https://stacks.math.columbia.edu/tag/07ZD}{Lemma 07ZD}]{stacks-project} that $M$ is projective of rank $r$.
\end{proof}

We record one more useful consequence of Corollary~\ref{cor:idealsofA}.

\begin{cor}\label{cor:OK-times-invariants}
  Let $H$ be an open subgroup of $\cO_K^\times$. We have $A^{H} = \F$,
  i.e.\ the $H$-invariants in $A$ are given by $\F$.
\end{cor}

\begin{proof}
If $x \in A^{H}$, then $xA$ is an $H$-stable ideal of $A$. It follows that $x = 0$ or $x \in A^\times$ by Corollary~\ref{cor:idealsofA},
i.e.\ $A^{H}$ is a field. Therefore, the composition $A^{H} \into A \xonto{\tr} \F\ppar{T}$ is injective. But $\tr$ is also
$\Zp^\times$-equivariant, so $A^{H}$ injects into $\F\ppar{T}^{H \cap \Z_p^\times}$ and it suffices to show that $\F\ppar{T}^{M} = \F$ for any open subgroup $M\subset\Z_p^\times$.
As the $\Zp^\times$-action is $\F$-linear, there is no loss in assuming that $\F = \Fp$.
To see that $\Fp\ppar{T}^{M} = \Fp$, recall that the $\Zp^\times$-action on $\Fp\ppar{T}$ is given by interpreting $\Fp\ppar{T}$ as the field of norms of $\Qp(\mu_{p^\infty})/\Qp$ (with Galois group $\Zp^\times$). Let $L_0\defeq\Q_p(\mu_{p^\infty})^M$, which is a finite totally ramified extension of $\Q_p$. Thus every $x \in \Fp\ppar{X}^{M}$ is given by a norm-compatible system of elements $x_L \in L_0$, $L$ running through
finite subextensions of $\Qp(\mu_{p^\infty})/L_0$. In particular, if $x$ is nonzero, then $x_L$ is $p$-divisible in $L_0^\times$, so $x_L \in [\Fp^\times]$. As
$x$ is then determined by $x_{L_0(\mu_p)}$, we deduce the claim.
\end{proof}

\subsubsection{Multivariable \texorpdfstring{$(\psi,\cO_K^\times)$}{(psi,O\_K\^{}x)}-modules}\label{multivariablepsi}

We define a functor from a certain abelian category of admissible smooth representations of $\GL_2(K)$ over
$\F$ to a category of multivariable $(\psi,\cO_K^\times)$-modules.

Let $R$ be a noetherian commutative ring of characteristic $p$ endowed with
an injective ring endomorphism $F_R$ such that $R$ is a finite free
$F_R(R)$-module. If $M$ is an $R$-module, we define
$F_R^*(M)\defeq R\otimes_{F_R,R}M$. Examples of such pairs $(R,F_R)$
are given by $(\F\bbra{N_0},\phi)$ and $(A,\phi)$ in \S\ref{ringA}.

A \emph{$\psi$-module} over $R$ is a pair $(M,\beta)$, where $M$ is an
$R$-module and $\beta$ is an $R$-linear homomorphism
$M\rightarrow F_R^*(M)$. When $R$ is a regular ring, $F_R$
is the Frobenius endomorphism of $R$ and $\beta$ is an isomorphism, we recover the notion of
$F_R$-module of \cite[Def.1.1]{Lyubeznik}. We say that a $\psi$-module
$(M,\beta)$ is \emph{\'etale} if $\beta$ is injective.

If $(M,\beta)$ is a $\psi$-module, the exact functor $F_R^*$ gives us,
for each $n\geq0$, an $R$-linear map $(F_R^*)^n(\beta) :
(F_R^*)^n(M)\rightarrow (F_R^*)^{n+1}(M)$ and we can define
\[ \beta_n\defeq (F_R^*)^{n-1}(\beta)\circ\cdots\circ(F_R^*)(\beta)\circ\beta : M\longrightarrow
(F_R^*)^n(M).\] The inductive limit of the system
$((F_R^*)^n(M),(F_R^*)^n(\beta))_n$ gives rise to a $\psi$-module
$(\mathcal{M},\underline{\beta})$ with $\underline{\beta}$ an
isomorphism. Then $(M,\beta)$ generates
$(\mathcal{M},\underline{\beta})$ in the sense of
\cite[Def.1.9]{Lyubeznik}. Let $M^{\et}$ be the image of $M$ in
$\mathcal{M}$ and $M^0$ the kernel of $M\rightarrow M^{\et}$. The map
$\beta$ induces a structure of $\psi$-module on $M^0$ and $M^{\et}$
and $M^{\et}$ is an \'etale $\psi$-module. The $\psi$-module $M^{\et}$
is called the \emph{\'etale part} of $M$ and $M^0$ the \emph{nilpotent
part} of $M$. We note that $(M,\beta)$ and $(M^{\et},\beta^{\et})$
generate the same $F_R$-module and $(M^0,\beta^0)$ generates the
trivial $F_R$-module whose underlying module is zero. Note that the
constructions $(M,\beta)\mapsto(M^{\et},\beta^{\et})$ and
$(M,\beta)\mapsto(M^0,\beta^0)$ are functorial in $(M,\beta)$ and
that, if $\beta$ is injective, we have $M^0=0$. This implies that if
$f : (M,\beta)\rightarrow (M',\beta')$ is a morphism of $\psi$-modules
with $(M',\beta')$ \'etale, then $f$ factors through $M^{\et}$.

We are mainly interested in $\psi$-modules with extra structures, which we call $(\psi,\cO_K^\times)$-modules over $A$. If $M$ is a finitely generated $A$-module, we always endow it with the topology defined by any good filtration (note that good filtrations generate the same topologies, cf.~\cite[Lemma I.5.3]{LiOy}). It is also the quotient topology given by any surjection $A^{\oplus d}\twoheadrightarrow M$ (as follows from \cite[Rk.I.5.2(2)]{LiOy}), and we call it the canonical topology on $M$. The group $\cO_K^\times$ acts continuously on $A$ and this action commutes with the endomorphism $\phi$. If $M$ is an $A$-module which is endowed with an action of $\cO_K^\times$, we consider the diagonal action on $\phi^*(M)$, which is well defined since $\phi$ commutes with $\cO_K^\times$.

\begin{definit}\label{psiok}
A \emph{$(\psi,\cO_K^\times)$-module over $A$} is a $\psi$-module
$(M,\beta)$ over $A$ such that $M$ is a finitely generated
$A$-module with a continuous semilinear action of $\cO_K^\times$
such that $\beta$ is $\cO_K^\times$-equivariant (here, continuity
means that the map $\cO_K^\times\times M\rightarrow M$ is
continuous). We say that a $(\psi,\cO_K^\times)$-module over $A$ is
\emph{\'etale} if the underlying $\psi$-module over $A$ is.
\end{definit}

We remark that if $(M,\beta)$ is a $(\psi,\cO_K^\times)$-module, then
$M$ is an $\cO_K^\times$-module and is therefore finite projective as an
$A$-module by Proposition \ref{prop:OK_modules}.

\begin{prop}\label{prop:et_psi_modules}
Let $(M,\beta)$ be an \'etale $(\psi,\cO_K^\times)$-module over
$A$. Then $\beta$ is an isomorphism.
\end{prop}
\begin{proof}
  We note that the two $A$-modules $M$ and
  $\phi^*(M)=A\otimes_{\phi,A}M$ have the same generic rank.  As
  $\beta$ is an injective $A$-linear map between two finitely
  generated modules of the same generic rank over a noetherian domain,
  its cokernel is torsion. This cokernel is then an
  $\cO_K^\times$-module which is moreover torsion as an $A$-module, it
  follows from Proposition \ref{prop:OK_modules} that it is zero and
  $\beta$ is an isomorphism.
\end{proof}

We now define a functor from certain representations of $\GL_2(K)$ over
$\F$ to $(\psi,\cO_K^\times)$-modules over $A$. 

Let $\pi$ be an admissible smooth representation of $\GL_2(K)$ over
$\F$. Its ($\F$-linear) dual $\pi^\vee$ is then a finitely generated
$\F\bbra{I_1}$-module. We fix a good filtration on $\pi^\vee$. As above, 
we endow $A{\otimes}_{\F\bbra{N_0}}\pi^\vee$ with the 
tensor product filtration and define the filtered $A$-module
\begin{equation}\label{dapi}
D_A(\pi)\defeq A\widehat{\otimes}_{\F\bbra{N_0}}\pi^\vee\simeq \widehat{(\pi^\vee)_S}.
\end{equation}
Note that the action of $\F\bbra{N_0}$ on $\pi^\vee$ is given by $\delta_a(f)\defeq f \circ a^{-1}$ for $f\in \pi^\vee$, $a \in N_0$. 
As all the good filtrations on $\pi^\vee$ are equivalent (\cite[Lemma I.5.3]{LiOy}), the underlying topological
$A$-module does not depend on the choice of the good filtration on
$\pi^\vee$. An example of a good filtration on $\pi^\vee$ is given by
the $\m_{I_1}$-adic filtration, as follows directly from the definition. It is very important to note that the topology used 
on $\pi^\vee$ is {\it not} the $\m_{N_0}$-adic topology but the $\m_{I_1}$-adic topology, which 
is actually coarser.

\begin{prop}\label{prop:exactnessDA}
The functor $\pi\longmapsto D_A(\pi)$ is exact.
\end{prop}
\begin{proof}
Let $0\rightarrow\pi'\rightarrow\pi\rightarrow\pi''\rightarrow0$ be an
exact sequence of admissible smooth representations of $\GL_2(K)$ over
$\F$. The sequence $0\rightarrow(\pi'')^\vee\rightarrow\pi^\vee\rightarrow(\pi')^\vee\rightarrow0$
is still exact. We endowed $\pi^\vee$ with a good filtration, $(\pi')^\vee$ with the
quotient filtration and $(\pi'')^\vee$ with the induced
filtration (which are again good by e.g.\ \cite[Prop.II.1.2.3]{LiOy}). With these choices, the exact sequence remains exact
after applying the functor $\gr$ (see for example
\cite[Thm.I.4.2.4(1)]{LiOy}). It follows from Lemma
\ref{lem:grlocal}, from the exactness of localization and from
\cite[Thm.I.4.2.4(2)]{LiOy}) that the sequence
$0\rightarrow(\pi'')_S^\vee\rightarrow(\pi^\vee)_S\rightarrow(\pi')_S^\vee\rightarrow0$
is exact and {\it strict}. The exactness of
$0\rightarrow D_A(\pi'')\rightarrow D_A(\pi)\rightarrow
D_A(\pi')\rightarrow0$ then follows from \cite[Thm.I.3.4.13]{LiOy}.
\end{proof}

We define a continuous action of $\cO_K^\times$ on $\pi^\vee$ as follows, for $f\in\pi^\vee$, $\gamma\in\cO_K^\times$ we have
\[ (\gamma\cdot f)(x)\defeq f\left(
    \begin{pmatrix}
      \gamma^{-1} & 0 \\ 0 & 1
    \end{pmatrix}x\right) \quad \forall\ x\in \pi.\] As $\cO_K^\times$
normalizes $I_1$, the action of $\cO_K^\times$ on $\pi^\vee$ is
continuous for the $\m_{I_1}$-adic topology. We use the continuous
action of $\cO_K^\times$ on $A$ to extend this action diagonally to
$A\otimes_{\F\bbra{N_0}}\pi^\vee$ and, by continuity, to
$D_A(\pi)$. The action of $\cO_K^\times$ is continuous and
$A$-semilinear in the sense that
\[ \gamma\cdot(af)=(\gamma \cdot a)(\gamma\cdot f) \quad \forall\ (\gamma,a,f)\in\cO_K^\times\times A
  \times D_A(\pi).\]
We define an $\F$-linear endomorphism $\psi$ of $\pi^\vee$ by the formula
\begin{equation}\label{psipi}
\psi(f)(x)=f(\xi(p)x) \quad \forall(f,x)\in\pi^\vee\times\pi.
\end{equation}
This endomorphism is continuous, clearly commutes with the action of
$\cO_K^\times$ and satisfies the relation
\[
\psi(\phi(a) f)=a(\psi(f))
\]
for all $a\in \F\bbra{N_0}$, $f\in \pi^\vee$.

\begin{lem}\label{lem:psicontinueN0}
Let $M$ be some $\F\bbra{N_0}$-module and let $\psi$ be an $\F$-linear
endomorphism of $M$ satisfying the relation
\[ \psi(\phi(a)m)=a\psi(m) \quad \forall\ (a,m)\in\F\bbra{N_0}\times
M.\]
Then for all integers $n\geq0$, we have
\[ \psi(\m_{N_0}^{pf-(f-1)+pn}M)\subset \m_{N_0}^{n+1}M.\]
As a consequence, for $n\geq pf-(f-1)$, we have
\[ \psi(\m_{N_0}^nM)\subset
\m_{N_0}^{\lceil\frac{n}{p}\rceil-f}M.\]
\end{lem}
\begin{proof}
For $n=0$, the result follows from the fact that, if
$Y_0^{i_0}\cdots Y_{f-1}^{i_{f-1}}\in\m_{N_0}^{pf-(f-1)}$, there
exists some $0\leq j\leq f-1$ such that $i_j\geq p$. Then, for all
$m\in M$, we have
\[ \psi(Y_0^{i_0}\cdots Y_{f-1}^{i_{f-1}}m)=Y_{j+1}\psi(Y_0^{i_0}\cdots
Y_{j}^{i_j-p}\cdots Y_{f-1}^{i_{f-1}}m)\in\m_{N_0}M.\]
The general statement follows from a simple induction on $n$.

For the last statement, we choose $m$ such that
\[ pm+pf-(f-1)\leq n < p(m+1)+pf-(f-1)\]
and we use the first statement to deduce that
\[
\psi(\m_{N_0}^nM)\subset\psi(\m_{N_0}^{pm+pf-(f-1)}M)\subset\m_{N_0}^{m+1}M\subset
\m_{N_0}^{\lceil\frac{n}{p}\rceil-f}M.\qedhere\]
\end{proof}

We extend $\psi$ to an $\F$-linear map $(\pi^{\vee})_S\rightarrow(\pi^{\vee})_S$ (recall $(\pi^{\vee})_S=\F\bbra{N_0}_S\otimes_{\F\bbra{N_0}}\pi^\vee$) by the formula
\begin{equation}\label{eq:defPsi}
\psi\left(\frac{m}{(Y_0\cdots
      Y_{f-1})^{pn}}\right)=\frac{\psi(m)}{(Y_0\cdots
    Y_{f-1})^{n}}
\end{equation}
for all $m\in\pi^\vee$ and $n\geq0$. Each element of $(\pi^{\vee})_S$
can be written as $(Y_0\cdots Y_{f-1})^{-pn}m$ for some $m\in\pi^\vee$
and $n\geq0$, and it follows from the properties of $\psi$ on
$\pi^\vee$ that the right-hand side of \eqref{eq:defPsi} does not depend on this
choice. For any element $g$ in $I_1$, we denote by $\delta_g$ the corresponding element $[g]$ in $\F\bbra{I_1}$.

\begin{lem}\label{lem:psicontinue}
The map $\psi : (\pi^{\vee})_S\rightarrow(\pi^{\vee})_S$ is continuous.
\end{lem}
\begin{proof}
As all the good filtrations on $\pi^\vee$ are equivalent, we choose
the $\m_{I_1}$-adic filtration on $\pi^\vee$ for this proof, i.e.\ $F_n\pi^\vee=\m_{I_1}^{-n}\pi^\vee$ for $n\leq 0$ and $F_n\pi^\vee=\pi^\vee$ for $n>0$. From the proof of \cite[Prop.5.3.3]{BHHMS1} we have an equality for $n\geq0$:
\begin{equation}\label{mi1}
\m_{I_1}^n=\sum_{\substack{r,s,t\geq0\\r+2s+t=n}}\m_{N_0}^r\m_{T_0}^s\m_{N_0^-}^t.
\end{equation}
As $\xi(p)$ commutes with each element in $T_0$, and
$\xi(p)^{-1}\smat{1&0\\z&1}\xi(p)=\smat{1&0\\z&1}^p$ for any
$\smat{1&0\\z&1}\in N_0^-$, it is easily checked from the definition
of $\psi$ and the $\F\bbra{I_1}$-action on $\pi^\vee$ that
\begin{equation}\label{eq:act:N0-}
  \psi(\delta_{h}\delta_{z}\cdot f)=\delta_{h}\delta_{z^p}\psi(f)
\end{equation}
for all $h\in T_0$, $z\in N_0^{-}$. In particular,
\[
 \psi(\m_{T_0}^s\m_{N_0^-}^t\pi^\vee)\subset\m_{T_0}^s\m_{N_0^-}^{pt}\pi^\vee,
\]
 and it follows from Lemma \ref{lem:psicontinueN0} that if $r\geq pf-(f-1)$ we have
\begin{equation}\label{rpf+}
\psi(\m_{N_0}^r\m_{T_0}^s\m_{N_0^-}^t\pi^\vee)\subset\m_{N_0}^{\lceil\frac{r}{p}\rceil-f}\m_{T_0}^s\m_{N_0^-}^{pt}\pi^\vee\subset\m_{I_1}^{\lceil\frac{r}{p}\rceil+2s+pt-f}\pi^\vee\subset \m_{I_1}^{\lceil\frac{r+2s+t}{p}\rceil-f}\pi^\vee.
\end{equation}
If $r< pf-(f-1)$, we need the following lemma.

\begin{lem}\label{lem:r_small}
Let $M\subset\pi^\vee$ be a closed $\F\bbra{N_0^-}$-submodule. Then
\[ \psi(\F\bbra{N_0}\m_{N_0^-}M)\subset\m_{I_1}\psi(\F\bbra{N_0}M).\]
As a consequence, for all $t\geq0$,
$\psi(\F\bbra{N_0}\m_{N_0^-}^t\pi^\vee)\subset\m_{I_1}^t\pi^\vee$.
\end{lem}
\begin{proof}
Note that $\m_{I_1}\times \F\bbra{N_0} \times M$ is compact, as $M$ is closed, hence so is the image $\m_{I_1}\psi(\F\bbra{N_0}M)$ of the continuous map $\m_{I_1}\times \F\bbra{N_0} \times M\rightarrow \pi^\vee$, $(a,b,m)\mapsto a\psi(bm)$. As $\m_{N_0^-}$ is generated as a right $\F\bbra{N_0^-}$-module by the $\delta_y-1$ for $y\in N_0^-$ and as $\psi$ is continuous on $\pi^\vee$, it is thus sufficient to prove that, for $y\in N_0^-$, $x\in N_0$ and $m\in M$, we have
$\psi(\delta_x(\delta_{y}-1)m)\in\m_{I_1}\psi(\F\bbra{N_0}M)$. As
$N_0^-\subset K_1$, $K_1$ is normalized by $N_0$ and
$K_1=N_0^pT_0 N_0^-$, we can write $xy=x_1^pt_1 y_1 x$ with
$(x_1,t_1,y_1)\in N_0\times T_0\times N_0^-$. Therefore
\begin{align*}
\psi(\delta_x(\delta_y-1)m)&=\psi(\delta_{x_1^p}\delta_{t_1}\delta_{y_1}\delta_xm)-\psi(\delta_xm)
\\
                                 &=\delta_{x_1t_1y_1^p}\psi(\delta_xm)-\psi(\delta_xm)=(\delta_{x_1t_1y_1^p}-1)\psi(\delta_xm)\\
  &\subset\m_{I_1}\psi(\F\bbra{N_0}M).
\end{align*}
For the second statement, inductively apply the first to $M=\m_{N_0^-}^{t-1}\pi^\vee$, $M=\m_{N_0^-}^{t-2}\pi^\vee$, etc.
\end{proof}

When $r<pf-(f-1)=(p-1)f+1$, we have $2s+t\geq r+2s+t-(p-1)f$ so that, using Lemma \ref{lem:r_small} and the fact
that $T_0$ normalizes $N_0$, we obtain
\begin{equation}\label{rpf-}
\psi(\m_{N_0}^r\m_{T_0}^s\m_{N_0^-}^t\pi^\vee)\subset\m_{T_0}^s\psi(\F\bbra{N_0}\m_{N_0^-}^t\pi^\vee)\subset\m_{I_1}^{2s+t}\pi^\vee\subset\m_{I_1}^{r+2s+t-(p-1)f}\pi^\vee.
\end{equation}
We deduce from (\ref{eq:defPsi}), (\ref{rpf+}) and (\ref{rpf-}) that, for all $n\in\Z$, $r\geq 0$, $s\geq0$, $t\geq0$ and
$k\geq0$ such that $r+2s+t\geq pf$, we have
  \[
    \psi\left(\frac{1}{(Y_0\cdots
        Y_{f-1})^{pk}}\m_{N_0}^r\m_{T_0}^s\m_{N_0^-}^t\pi^\vee\right)\subset\frac{1}{(Y_0\cdots
      Y_{f-1})^k}\m_{I_1}^{\lceil\frac{r+2s+t}{p}\rceil-f}\pi^\vee\]
so that, for $n\geq pf$ by (\ref{mi1}) we have
  \[
    \psi\left(\frac{1}{(Y_0\cdots
        Y_{f-1})^{pk}}\m_{I_1}^n\pi^\vee\right)\subset\frac{1}{(Y_0\cdots
      Y_{f-1})^k}\m_{I_1}^{\lceil\frac{n}{p}\rceil-f}\pi^\vee\subset
    F_{kf+f-\lceil\frac{n}{p}\rceil}((\pi^{\vee})_S).\]
From Remark \ref{rem:on_filtered_Amod}(ii), we know
that, for $n\in\Z$, $F_n((\pi^{\vee})_S)$ is the increasing union over
$k\geq \max\{0,\frac n{pf}\}$ of the subspaces
\[ \frac{1}{(Y_0\cdots
Y_{f-1})^{pk}}\m_{I_1}^{-n+pkf}\pi^\vee,\]
hence we deduce for all $n\in\Z$ that
  \[
      \psi(F_n((\pi^{\vee})_S))\subset \bigcup_{k\geq
        \mathrm{max}\{0,\frac{n}{pf}\}}F_{kf+f-\lceil\frac{-n+pkf}{p}\rceil}((\pi^{\vee})_S)\subset
      F_{f+\lfloor\frac{n}{p}\rfloor}((\pi^{\vee})_S).
    \]
This proves the continuity of $\psi$.
\end{proof}

We can therefore extend $\psi$ to a continuous $\F$-linear map $\psi :
D_A(\pi)\rightarrow D_A(\pi)$ such that
\[ \psi(\phi(a)m)=a\psi(m) \quad \forall\ (a,m)\in A\times
\pi^\vee.\]

We fix $\{a_0,\dots,a_{q-1}\}$ a system of representatives of the cosets
of $N_0^p\simeq p\cO_K$ in $N_0\simeq \cO_K$, so that
$\F\bbra{N_0}=\bigoplus_{i=0}^{q-1}\delta_{a_i}\F\bbra{N_0^p}$. As
$\phi(\F\bbra{N_0})=\F\bbra{N_0^p}$ and $A=\bigoplus_{i=0}^{q-1}\delta_{a_i}\phi(A)$, 
we have a canonical isomorphism for any $A$-module $M$:
\[ \phi^*(M)\simeq\bigoplus_{i=0}^{q-1} (\F\delta_{a_i}\otimes_{\F} M).\]
We define an $\F$-linear map $\beta : D_A(\pi)\rightarrow
\phi^*(D_A(\pi))=A\otimes_{\phi,A}D_A(\pi)$ by
\begin{equation}\label{mapbeta}
  \begin{array}{ccc}
    D_A(\pi)&
              \longrightarrow&\bigoplus_{i=0}^{q-1}(\F\delta_{a_i}\otimes_{\F}D_A(\pi)) \\
    m&\longmapsto&\sum_{i=0}^{q-1}\delta_{a_i}\otimes_\phi\psi(\delta_{a_i}^{-1}m)
  \end{array}
\end{equation}
(we write $x\otimes_\phi y$ instead of just $x\otimes y$ in order not to forget the map $\phi$ in the tensor product).

\begin{rem}\label{rem:beta_not_depend_on_choices}
The definition of the map $\beta$ does not depend on the choice of
the system $\{a_i\}$, namely, replacing $a_i$ with $a_ib^p$ for some
$b\in N_0$, we have
\begin{multline*}
 \delta_{a_ib^p}\otimes_\phi\psi(\delta_{a_ib^p}^{-1}m)=\delta_{a_ib^p}\otimes_\phi\psi(\phi(\delta_b)^{-1}\delta_{a_i}^{-1}m)=\delta_{a_ib^p}\otimes_\phi\delta_b^{-1}\psi(\delta_{a_i}^{-1}m)\\
 =\delta_{a_ib^p}\delta_{b^p}^{-1}\otimes_\phi\psi(\delta_{a_i}^{-1}m)=\delta_{a_i}\otimes_\phi\psi(\delta_{a_i}^{-1}m).
\end{multline*}
\end{rem}

Using Remark \ref{rem:beta_not_depend_on_choices}, we easily check that $\beta$ is actually an $A$-linear map (note that it is enough to check it for an element in $\delta_{a_i}\phi(A)$ using $A=\bigoplus_{i=0}^{q-1}\delta_{a_i}\phi(A)$, and thus for $\delta_{a_i}$ and for an element in $\phi(A)$), hence $\beta:D_A(\pi)\rightarrow \phi^*(D_A(\pi))$ can be seen as a ``linearization'' of $\psi:D_A(\pi)\rightarrow D_A(\pi)$. Moreover, letting $\cO_K^\times$ act diagonally on $A\otimes_{\phi,A}D_A(\pi)$, the map $\beta$ is then $\cO_K^\times$-equivariant. Indeed, for $a\in\cO_K^\times$ and $m\in D_A(\pi)$, we have
\begin{align*}
a\cdot\beta(m)&=a\cdot\left(\sum_{i=0}^{q-1}\delta_{a_i}\otimes_\phi\psi(\delta_{a_i}^{-1}m)\right)=\sum_{i=0}^{q-1}\delta_{a\cdot
                  a_i}\otimes_\phi a\cdot\psi(\delta_{a_i}^{-1}m)\\
                &=\sum_{i=0}^{q-1}\delta_{a\cdot
                  a_i}\otimes_\phi\psi(a\cdot\delta_{a_i}^{-1}m)=\sum_{i=0}^{q-1}\delta_{a\cdot
                  a_i}\otimes_\phi\psi(\delta_{a\cdot a_i}^{-1}(a\cdot m))\\
                &=\beta(a\cdot m),
\end{align*}
the last equality coming from Remark \ref{rem:beta_not_depend_on_choices} and the fact that $\{a\cdot a_0,\dots,a\cdot a_{q-1}\}$ is another system of representatives of $N_0^p$ in $N_0$.

It is convenient to assume that the admissible smooth representation $\pi$ has a central character, in which case $Z_1$ acts trivially on $\pi$ and $\pi^\vee$ is a finitely generated $\F\bbra{I_1/Z_1}$-module. We recall from \cite[\S5.3]{BHHMS1} that the graded ring $\gr(\F\bbra{I_1/Z_1})$ of $\F\bbra{I_1/Z_1}$ is isomorphic to a tensor product of (noncommutative) graded rings
\begin{equation}\label{gri1}
\bigotimes_{i=0}^{f-1}\F[y_i,z_i,h_i],\
\end{equation}
where variables with different indices commute, where $[y_i,z_i]=h_i$,
$[h_i,y_i]=[h_i,z_i]=0$, where $y_i,z_i$ are homogeneous of
degree $-1$, and $h_i$ is homogeneous of degree $-2$. Note that the
$\m_{I_1/Z_1}$-adic topology on $\F\bbra{I_1/Z_1}$ induces the $\m_{N_0}$-adic
topology on $\F\bbra{N_0}$ via the inclusion
$\F\bbra{N_0}\subset\F\bbra{I_1/Z_1}$. Therefore the map
$\gr(\F\bbra{N_0})\rightarrow\gr(\F\bbra{I_1/Z_1})$ is injective and
its image is $\F[y_0,\dots,y_{f-1}]$ in $\gr(\F\bbra{I_1/Z_1})$. 

\begin{rem}\label{rmk:alt:def}
The $A$-module $D_A(\pi)$ can also be defined as the microlocalization of
$\pi^\vee$ with respect to the multiplicative subset
$T\defeq \set{(y_0\cdots y_{f-1})^k, k\in\NN}\subset\gr(\F\bbra{I_1/Z_1})$. This shows that $D_A(\pi)$ can
be promoted to a module over the noncommutative ring which is the
microlocalization of $\F\bbra{I_1/Z_1}$ with respect to $T$.
\end{rem}

We now let $\mathcal{C}$ be the category of admissible smooth
representations $\pi$ of $\GL_2(K)$ over $\F$ with a central character
{\it and} such that there exists a good filtration on the
$\F\bbra{I_1/Z_1}$-module $\pi^\vee$ such that $\gr(D_A(\pi))$ is a
finitely generated $\gr(A)$-module, or equivalently by Lemma
\ref{lem:grlocal} and Corollary \ref{lem:grA}
$\gr(\pi^\vee)[(y_0\cdots y_{f-1})^{-1}]$ is finitely generated over
$\gr(\F\bbra{N_0})[(y_0\cdots y_{f-1})^{-1}]$. By
\cite[Thm.I.5.7]{LiOy} this is also equivalent to require that
$D_A(\pi)$ is finitely generated over $A$ and that its natural
filtration in (\ref{dapi}) is good (equivalently gives the canonical
topology). In particular, if this holds for one good filtration on
$\pi^\vee$, then this holds for all good filtrations. It easily
follows from the proof of Proposition \ref{prop:exactnessDA} and the
noetherianity of $\gr(A)$ (Corollary \ref{lem:grA}) that $\mathcal{C}$
is an abelian subcategory stable under subquotients and extensions in
the category of smooth representations of $\GL_2(K)$ over $\F$ with a
central character.
 
For $\pi$ in $\mathcal C$, the pair $(D_A(\pi),\beta)$ is an example of $(\psi,\cO_K^\times)$-module over $A$ as in Definition \ref{psiok}. We can in particular consider its \'etale part $D_A(\pi)^{\et}$. The action of $\cO_K^\times$ on $D_A(\pi)$ preserves its nilpotent part $D_A(\pi)^0$ and thus induces a continuous action of $\cO_K^{\times}$ on $D_A(\pi)^{\et}$. In particular, $D_A(\pi)^{\et}$ is an \'etale $(\psi,\cO_K^\times)$-module over $A$. Note that the canonical topology on the finitely generated $A$-module $D_A(\pi)^{\et}$ is also the quotient topology of $D_A(\pi)\twoheadrightarrow D_A(\pi)^{\et}$.

\begin{cor}\label{cor:psiiso}
Let $\pi$ in $\mathcal{C}$. Then the $A$-modules $D_A(\pi)$ and $D_A(\pi)^{\et}$
are finite projective over $A$. Moreover the map $\beta^{\et} : D_A(\pi)^{\et}\rightarrow \phi^*D_A(\pi)^{\et}$ is an isomorphism.
\end{cor}
\begin{proof}
This is a special case of Propositions \ref{prop:OK_modules} and \ref{prop:et_psi_modules}.
\end{proof}

\begin{rem}
If $\pi$ is $1$-dimensional (a character of $\GL_2(K)$), then $D_A(\pi)=D_A(\pi)^{\et}=0$.
\end{rem}

We give an important condition on an admissible smooth representation $\pi$ (with a central character) which ensures that $\pi$ is in $\mathcal C$. Let $J$ be the following graded ideal of $\gr(\F\bbra{I_1/Z_1})$:
\begin{equation}\label{idealJ}
J\defeq (y_iz_i,h_i, 0\leq i\leq f-1).
\end{equation}
From the definition of equivalent filtrations (see \cite[\S I.3.2]{LiOy}), one easily sees (using \cite[Lemma I.5.3]{LiOy}) that if $\gr(\pi^\vee)$ is annihilated by some power of $J$ for one good filtration on $\pi$, then it is so for all good filtrations (but note that the power of $J$ which annihilates $\gr(\pi^\vee)$ may depend on the fixed good filtration).

\begin{prop}\label{prop:finiteness}
Assume that $\gr(\pi^\vee)$ is annihilated by some power of $J$. Then the $A$-module $D_A(\pi)$ is finite
projective and the $\gr(A)$-module $\gr(D_A(\pi))$ is finitely generated.
\end{prop}
\begin{proof}
As the hypothesis does not depend on the choice of the good
filtration on $\pi^\vee$, we are free to work with the
$\m_{I_1/Z_1}$-adic topology on $\pi^\vee$. Let us first prove that $\gr(D_A(\pi))$ is a finitely generated $\gr(A)$-module. It follows from the admissibility of $\pi$ and from the hypothesis that $\gr(\pi^\vee)$ is a finitely generated $\gr(\F\bbra{I_1/Z_1})/J^N$-module for some $N\geq1$. Lemma \ref{lem:grlocal} then implies that $\gr(D_A(\pi))$ is a finitely generated
$(\gr(\F\bbra{I_1/Z_1})/J^N)[(y_0\cdots y_{f-1})^{-1}]$-module. It is therefore sufficient to prove that
$(\gr(\F\bbra{I_1/Z_1})/J^N)[(y_0\cdots y_{f-1})^{-1}]$ is a finitely generated
$\gr(A)$-module. Since $\gr(\F\bbra{I_1/Z_1})$ is noetherian, we are reduced by d\'evissage to the case
$N=1$, where we have
\begin{eqnarray*} 
\big(\gr(\F\bbra{I_1/Z_1})/J\big)[(y_0\cdots y_{f-1})^{-1}]&\simeq&(\F[y_i,z_i,h_i]/(y_iz_i,h_i))[(y_0\cdots
y_{f-1})^{-1}]\\
&=&\F[y^{\pm1}_i]\ \simeq\ \gr(A).
\end{eqnarray*}
Finally, as $D_A(\pi)$ is a complete filtered $A$-module, it then follows from
\cite[Thm.I.5.7]{LiOy} that $D_A(\pi)$ is finitely generated over $A$
and from Proposition \ref{prop:OK_modules} that it is projective.
\end{proof}

It follows from Proposition \ref{prop:finiteness} that the admissible
smooth representations $\pi$ (with a central character) such that
$\gr(\pi^\vee)$ is annihilated by some power of $J$ for at least one
good filtration is a full subcategory of the category $\mathcal
C$. Moreover this full subcategory is abelian and stable under
subquotients and extensions in $\mathcal C$. Namely, for a short exact
sequence $0\rightarrow\pi'\rightarrow\pi\rightarrow\pi''\rightarrow0$
in $\mathcal{C}$, the filtrations induced on $(\pi'')^\vee$ and
$(\pi')^\vee$ by a good filtration of $\pi^\vee$ are good. For these
filtrations we have a short exact sequence
$0\rightarrow\gr((\pi'')^\vee)\rightarrow\gr(\pi^\vee)\rightarrow\gr((\pi')^\vee)\rightarrow0$
which shows that $\gr(\pi^\vee)$ is annihilated by a power of $J$ if
and only if $\gr((\pi')^\vee)$ and $\gr((\pi'')^\vee)$ are.

\begin{rem}\label{Wu}
It is natural to consider the image $D^\natural_A(\pi)$ of $\pi^\vee$ in $D_A(\pi)=A\widehat\otimes_{\F\bbra{N_0}}\pi^\vee$. Indeed, as the map $\pi^\vee\rightarrow D_A(\pi)$ is continuous and $\pi^\vee$ is compact, it follows that $D^\natural_A(\pi)$ is a compact $\F\bbra{N_0}$-submodule of $D_A(\pi)$. However, the $\F\bbra{N_0}$-module $D_A^\natural(\pi)$ is {\it not} finitely generated when $\pi$ is an irreducible admissible supersingular representation and $[K:\Qp]>1$ (even if $D_A(\pi)$ is finitely generated over $A$). Namely, if this were the case, this would give us the existence of a nontrivial finitely generated
$\F\bbra{N_0}[(
\begin{smallmatrix}
p&0 \\ 0 & 1
\end{smallmatrix}
)]$-submodule of $\pi$ that is admissible as an $\F\bbra{N_0}$-module and this would contradict the results of
\cite{Benj} and \cite{Wu}. Likewise, the image of $\pi^\vee$ in the quotient $D_A(\pi)^{\et}$ of $D_A(\pi)$ won't be finitely generated over $\F\bbra{N_0}$ in general (see Remark \ref{da=daet}(ii)).
Finally, we conjecture in \cite[Conj.1.4]{BHHMS3} that for those $\pi$ coming from cohomology we always have $D_A(\pi)\cong D_A(\pi)^{\et}$.
\end{rem}

\begin{rem}\label{rem:normalization_N0}
  Recall that the action of $\F\bbra{N_0}$ on $\pi^\vee$ is defined by
  $\delta_a(f)=f\circ a^{-1}$ for $f\in\pi^\vee$ and $a\in N_0$. We
  could have defined it by the formula $\delta_a(f)=f\circ a$ for
  $f\in\pi^\vee$ and $a\in N_0$ and would have obtained isomorphic
  $(\psi,\cO_K^\times)$-modules $D_A(\pi)$ and
  $D_A(\pi)^{\et}$ (for instance, this is the convention used in \cite[Lemme 2.6]{breuil-foncteur}). Namely the map $f\mapsto \gamma_{-1}\cdot f$, with
  $\gamma=
  \begin{pmatrix}
    -1 & 0 \\ 0 & 1
  \end{pmatrix}$ induces an intertwining, commuting with $\psi$ and $\cO_K^\times$, between the two $\F\bbra{N_0}$-structures.
\end{rem}

\subsubsection{Multivariable \texorpdfstring{$(\varphi,\cO_K^\times)$}{(phi,O\_K\^{}x)}-modules}\label{multivariable}

Using the results of \S\ref{multivariablepsi}, we promote the functor $\pi\mapsto D_A(\pi)^{\et}$ to an exact functor from $\mathcal C$ to a category of \'etale multivariable $(\varphi,\cO_K^\times)$-modules (Theorem \ref{thm:exactnessDA}) and we compare $D_A(\pi)^{\et}$ with the functor $D_\xi^\vee(\pi)$ of \S\ref{covariant} (Theorem \ref{thm:functors_comparison}).

Let $R$ be a noetherian commutative ring of characteristic $p$ endowed with an injective ring endomorphism $F_R$ such that $R$ is a finite free
$F_R(R)$-module (as at the beginning of \S\ref{multivariablepsi}). A \emph{$\varphi$-module} $(D,\varphi)$ over $R$ is an $R$-module $D$ with an $F_R$-semilinear map $\varphi : D\rightarrow D$. We say that a
$\varphi$-module $(D,\varphi)$ is \emph{\'etale} if the $R$-linear map
$F_R^*(D)\rightarrow D$ defined by $a\otimes d\mapsto a\varphi(d)$ is
an isomorphism.

\begin{definit}\label{phiok}
A \emph{$(\varphi,\cO_K^\times)$-module over $A$} is a $\varphi$-module $(D,\varphi)$ over $A$
such that $D$ is a finitely generated $A$-module, the endomorphism $\varphi$ is continuous
(for the canonical topology on $D$ as at the beginning of \S\ref{multivariablepsi}) and $D$ is endowed with a continuous $A$-semilinear action of $\cO_K^\times$ commuting with $\varphi$. We say that a $(\varphi,\cO_K^\times)$-module over $A$ is
\emph{\'etale} if the underlying $\varphi$-module over $A$ is.
\end{definit}

We note that, by Proposition \ref{prop:OK_modules}, if $(D,\varphi)$
is a $(\varphi,\cO_K^\times)$-module over $A$, then $D$ is a finite
projective $A$-module.

If $(D,\beta)$ is an \'etale $(\psi,\cO_K^\times)$-module over $A$ as
in Definition \ref{psiok}, by Proposition \ref{prop:et_psi_modules} we can define a $\phi$-semilinear endomorphism $\varphi$ of $D$ such that $\Id\otimes\varphi=\beta^{-1}$, so that $(D,\varphi)$ is an \'etale $(\varphi,\cO_K^\times)$-module over $A$. (Note that $\varphi$ is continuous, as the topology of $D$ is defined by any good filtration and $\phi : A \to A$ is continuous.)

We now go back to representations $\pi$ of $\GL_2(K)$, but we first need some more notation. The trace map $\tr : N_0\simeq \cO_K\rightarrow\Zp$ induces a ring homomorphism
$\tr : \F\bbra{N_0}\rightarrow\F\bbra{\Zp}\simeq\F\bbra{X}$, where we recall that $X = (\begin{smallmatrix} 1 & 1 \\ 0 & 1\end{smallmatrix})-1$. Moreover, for $Y_i$ as in (\ref{elementsyi}), we
have $\tr(Y_i)\equiv -X \mod{X^2}$ (see Lemma \ref{lem:XY} and
the last statement in Lemma \ref{lem:ker} below) and the universal property of the ring $A$ shows
that this map extends to a continuous ring homomorphism $\tr : A\rightarrow\F\ppar{X}$. We let
\[\mathfrak{p}\defeq \Ker(\tr : A\rightarrow\F\ppar{X}).\]
Then $\mathfrak{p}$ is a closed maximal ideal of $A$. Note that
\[\mathfrak{p}\cap \F\bbra{N_0}=\Ker(\tr : \F\bbra{N_0}\rightarrow \F\bbra{X})=\m_{N_1}\F\bbra{N_0}=(Y_0-Y_1,\dots,Y_0-Y_{f-1}),\]
where $N_1\subseteq N_0$ is as in (\ref{n1}) (for the second isomorphism write $N_0\cong N_1 \oplus \Zp e$, where $\tr(e)=1$, noting that $\tr : \cO_K\rightarrow\Zp$ is surjective, as $K$ is unramified, and for the third use the first statement of Lemma \ref{lem:ker} below).

\begin{rem}
Let $B$ be the completion of $\F\bbra{N_0}_S$ along the prime ideal generated by $(Y_0-Y_1,\dots,Y_0-Y_{f-1})$ (see the beginning of \S\ref{ringA} for $S$). Expanding $Y_i^n=(Y_0-(Y_0-Y_i))^n$ if $n\geq 0$, and writing $Y_i^{n}=(\sum_{m=0}^{+\infty}\frac{(Y_0-Y_i)^m}{Y_0^{m+1}})^{-n}$ and expanding everything if $n<0$, one can see using Remark \ref{rem:on_filtered_Amod}(iii) that the ring $A$ embeds into $B$. The endomorphism $\phi$ on $A$ extends to $B$ but only the action of $\Zp^\times\subseteq \cO_K^\times$ extends to $B$, as $(Y_0-Y_1,\dots,Y_0-Y_{f-1})$ is not preserved by all of $\cO_K^\times$. Then from Corollary \ref{cor:psiiso} and as $B$ is a local ring, we see that $D_A(\pi)^{\et}\otimes_AB$ is a {\it finite free \'etale $(\varphi,\Zp^\times)$-module over $B$}, which is similar to the generalized $(\varphi,\Gamma)$-modules defined in \cite{schneider-vigneras} (though {\it loc.cit.}\ only considers split algebraic groups over $\Qp$).
\end{rem}

Let $\pi$ be in the category $\mathcal C$. Using Corollary \ref{cor:psiiso}, we can define a $\phi$-semilinear endomorphism $\varphi$ of $D_A(\pi)^{\et}$ such that $\Id\otimes\varphi=(\beta^{\et})^{-1}$, so that $D_A(\pi)^{\et}$ is an \'etale
$(\varphi,\cO_K^\times)$-module over $A$. As $\mathfrak{p}$ is a $\phi$-stable ideal of $A$, we deduce that $D_A(\pi)^{\et}/\mathfrak{p}\simeq D_A(\pi)^{\et}\otimes_A\F\ppar{X}$ is an \'etale $(\varphi,\Zp^\times)$-module over $\F\ppar{X}$.

\begin{thm}\label{thm:exactnessDA}\

  \begin{enumerate}
  \item The functor $\pi\longmapsto D_A(\pi)^{\et}$ is exact from the category $\mathcal{C}$ to the category of \'etale $(\varphi,\cO_K^\times)$-modules over $A$.
  \item The functor $\pi\longmapsto D_A(\pi)^{\et}\otimes_A\F\ppar{X}$ is exact from the category
    $\mathcal{C}$ to the category of \'etale $(\varphi,\Zp^\times)$-modules over $\F\ppar{X}$.
  \end{enumerate}
\end{thm}
\begin{proof}
(i) is a consequence of Proposition \ref{prop:exactnessDA}, of the exactness of $\phi^*$ and of the exactness of direct limits, together
with the description (see the beginning of \S\ref{multivariablepsi})
\[D_A(\pi)^{\et}\cong\varinjlim_{(\phi^*)^n(\beta^{\et})} (\phi^*)^n(D_A(\pi)^{\et})\cong\varinjlim_{(\phi^*)^n(\beta)} (\phi^*)^n(D_A(\pi)).\]
(ii) is a consequence of (i), of Corollary \ref{cor:psiiso} and of the exactness of $(-)\otimes_A\F\ppar{X}$ on short exact sequences of finite projective $A$-modules.
\end{proof}

\begin{rem}
One can prove that if $\pi\in \mathcal{C}$ then the endomorphism $\psi:D_A(\pi)\ra D_A(\pi)$ (defined right after Lemma \ref{lem:r_small}) is always surjective.
(This follows ultimately from the fact that the image of the natural map $A\otimes_{\F\bbra{N_0}}\pi^\vee\ra D_A(\pi)$ is surjective since $A$ is complete and Noetherian, and $A\otimes_{\F\bbra{N_0}}\pi^\vee$ is endowed with a surjective endomorphism that is compatible with $\psi$ on $D_A(\pi)$.)
In particular, this implies that $D_A(\pi)^{\et}\neq 0$ as soon as $D_A(\pi)\neq 0$, since $\psi$ cannot be nilpotent if it is surjective on $D_A(\pi)$ and the latter is nonzero.
Note that for the representations $\pi$ of particular interest for us here, we will actually have $D_A(\pi)=D_A(\pi)^{\et}$; see Remark \ref{da=daet}(ii).
\end{rem}

We now compare the \'etale $(\varphi,\Zp^\times)$-module $D_A(\pi)^{\et}/\mathfrak{p}$ with $D_{\xi}^\vee(\pi)$ (\ref{DxiH}).

Let $\overline{\psi}$ be the $\F$-linear endomorphism of $\pi^\vee/\m_{N_1}\simeq(\pi^{N_1})^\vee$ defined by
\begin{equation}\label{psibar}
\overline{\psi}(x)\defeq \sum_{b\in N_1/N_1^p}\psi(\delta_{\tilde{b}}\tilde{x}) \mod \m_{N_1},
\end{equation}
where $\tilde{b}\in N_1$ is a lift of $b$, $\tilde{x}\in \pi^\vee$ is a lift of $x$ and $\psi$ is as in (\ref{psipi}) (it is easy to check that the definition of $\overline{\psi}$ does not depend on the choice of these lifts). We have $\overline{\psi}(S(X^p)m)=S(X)\overline{\psi}(m)$ for all $S(X)\in\F\bbra{X}$ and $m\in \pi^\vee/\m_{N_1}$, and $\overline{\psi}$ is the dual of the endomorphism $F$ of $\pi^{N_1}$ in \S\ref{covariant}. We define an endomorphism $\overline{\psi}$ of $D_A(\pi)/\mathfrak{p}$ (resp.\ $D_A(\pi)^{\et}/\mathfrak{p}$) by the same formula replacing $\pi^\vee$ by $D_A(\pi)$ (resp.\ $D_A(\pi)^{\et}$) and $\m_{N_1}$ by $\mathfrak{p}$, it is then clear that the following diagram commutes:
\begin{equation}\label{psipida}
\begin{tikzcd}
    \pi^\vee/\m_{N_1}\ar[r,"\overline{\psi}"]\ar[d]
    & \pi^\vee/\m_{N_1} \ar[d] \\
    D_A(\pi)/\mathfrak{p} \ar[r,"\overline{\psi}"] & D_A(\pi)/\mathfrak{p},
\end{tikzcd}
\end{equation}
together with an analogous diagram with $D_A(\pi)/\mathfrak{p}\twoheadrightarrow D_A(\pi)^{\et}/\mathfrak{p}$ that we leave to the reader.

Let $\overline{\beta} :
D_A(\pi)/\mathfrak{p}\rightarrow\phi^*(D_A(\pi)/\mathfrak{p})\simeq\F\bbra{X}\otimes_{\varphi,\F\bbra{X}}(D_A(\pi)/\mathfrak{p})\simeq \phi^*(D_A(\pi))/\mathfrak{p}$
be the $\F\ppar{X}$-linear map defined by
\[\overline{\beta}(m)\defeq \sum_{i=0}^{p-1}(1+X)^{-i}\otimes_{\phi}\overline{\psi}((1+X)^im).\]

\begin{lem}\label{lem:compatibility_of_psis}
The following diagram is commutative {\upshape(}where the horizontal maps are the canonical surjections{\upshape)}:
  \[
    \begin{tikzcd}
      D_A(\pi) \ar[r,twoheadrightarrow] \ar[d,"\beta"] &
      D_A(\pi)/\mathfrak{p}
      \ar[d,"\overline{\beta}"] \\
      \phi^*(D_A(\pi)) \ar[r,twoheadrightarrow] &
      \phi^*(D_A(\pi)/\mathfrak{p}).
    \end{tikzcd}\]
\end{lem}
\begin{proof}
We choose a system of representatives $(g^{-i}b_j)_{\substack{0\leq i\leq p-1 \\ 1\leq j\leq p^{f-1}}}$ of $N_0/N_0^p$ such
that $g\defeq \left(
\begin{smallmatrix}
1 & 1 \\ 0 & 1
\end{smallmatrix}\right)\in N_0$ and $b_1,\dots,b_{p^{f-1}}$ are in $N_1$. We then have for $m\in D_A(\pi)$ that
\begin{align*}
\beta(m)&=\sum_{i=0}^{p-1}\sum_{j=1}^{p^{f-1}}\big(\delta_{g^{i}}^{-1}\delta_{b_j}\otimes_\phi\psi(\delta_{b_j}^{-1}\delta_{g^i}m)\big)\\
&\equiv \sum_{i=0}^{p-1}\big(\delta_{g^{i}}^{-1} \otimes_\phi \sum_{j=1}^{p^{f-1}}\psi(\delta_{b_j}^{-1}(\delta_{g^i}m))\big) \mod\mathfrak{p} \phi^*(D_A(\pi))\\
&\equiv\sum_{i=0}^{p-1}\delta_{g^i}^{-1}\otimes_\phi\overline{\psi}(\delta_{g^{i}}m) \mod\mathfrak{p} \phi^*(D_A(\pi)),
\end{align*}
where the first equality follows from (\ref{mapbeta}), the second from $\delta_{b_j}-1\in \mathfrak{p}\subseteq A$ (and the commutativity of $N_0$), and the third from the analog of (\ref{psibar}) for $D_A(\pi)/\mathfrak{p}$. Noting that the image of $\delta_{g^i}$ in $\F\bbra{X}$ is $(1+X)^i$, we obtain the desired compatibility.
\end{proof}

\begin{lem}\label{lem:continuity_for_comparison_of_D}
Let $M\subset \pi^{N_1}$ be an $\F\bbra{X}$-submodule that is admissible as an $\F\bbra{X}$-module. Then
the surjective map $\pi^\vee\onto M^\vee$ is continuous for
the $\mathfrak{m}_{I_1}$-adic topology on $\pi^\vee$ and the
$X$-adic topology on $M^\vee$.
\end{lem}
\begin{proof}
The map $\pi^\vee \onto M^\vee$ is continuous with respect to the natural profinite topologies arising from Pontryagin duality. As    
$M$ is admissible as an $\F\bbra{X}$-module, the natural topology on $M^\vee$ is the $X$-adic topology. It thus suffices to show that the        
$\mathfrak{m}_{I_1}$-adic topology is at least as fine as the natural topology on $\pi^\vee$. Dually this means that any finite-dimensional       
subspace of $\pi$ is contained in $\pi[\mathfrak{m}_{I_1}^N]$ for some sufficiently large integer $N$, which is true by smoothness.
\end{proof}

Recall that we defined in (\ref{DxiH}) a projective limit $D^\vee_{\xi}(\pi)$ of \'etale $(\varphi,\Zp^\times)$-modules over $\F\ppar{X}$ associated to $\pi$.

\begin{thm}\label{thm:functors_comparison}
We have an isomorphism of \'etale $(\varphi,\Zp^\times)$-modules over $\F\ppar{X}$:
\[D_A(\pi)^{\et}\!/\mathfrak{p}\buildrel\sim\over\longrightarrow D^\vee_{\xi}(\pi).\]
In particular, $D^\vee_{\xi}(\pi)$ is finite-dimensional over $\F\ppar{X}$ and the functor $\pi\longmapsto D^\vee_{\xi}(\pi)$ is exact on $\mathcal C$.
\end{thm}
\begin{proof}
For the purpose of this proof it is convenient to use the action of $\F\bbra{N_0}$ on $\pi^\vee$ given by $\delta_a(f)=f\circ a$ for $f\in\pi^\vee$ and $a\in N_0$. 
This does not change $D_A(\pi)^{\et}$ up to isomorphism by Remark~\ref{rem:normalization_N0}.

As a first step we construct the map. Let $M\subset\pi^{N_1}$ be a
finitely generated $\F\bbra{X}[F]$-submodule that is admissible as an 
$\F\bbra{X}$-module and $\Zp^\times$-stable. By Lemma
\ref{lem:continuity_for_comparison_of_D}, the map
$\pi^\vee\twoheadrightarrow M^\vee$ is continuous. It extends to a
surjection of $\F\bbra{N_0}_S$-modules $(\pi^{\vee})_S\onto
M^\vee[X^{-1}]$. %
By definition of the tensor product filtration on $(\pi^{\vee})_S$,
this surjection is continuous if $M^\vee[X^{-1}]$ is endowed with its
natural topology of finite-dimensional $\F\ppar{X}$-vector space. As
$M^\vee[X^{-1}]$ is complete for this topology, by completion we
obtain a continuous surjection of topological $A$-modules
$\zeta_M : D_A(\pi)\onto M^\vee[X^{-1}]$. Since $N_1$ acts
trivially on $M$, $\zeta_M$ factors through a surjection of
$\F\ppar{X}$-vector spaces $\overline{\zeta_M} : D_A(\pi)/\mathfrak{p}\onto
M^{\vee}[X^{-1}]$. By definition of $\overline{\psi}$, we obtain a
commutative diagram (where $F^\vee$ is the $\F$-linear dual of $F:M\rightarrow M$ that we extend to $M^\vee[X^{-1}]$ using $F^\vee(X^{-i}f)=X^{-i}F(X^{i(p-1)}f)$)
  \[
    \begin{tikzcd}
      D_A(\pi)/\mathfrak{p}\ar[r,twoheadrightarrow,"\overline{\zeta_M}"]\ar[d,"\overline{\psi}"]
      & M^\vee[X^{-1}] \ar[d,"F^\vee"] \\
      D_A(\pi)/\mathfrak{p} \ar[r,twoheadrightarrow,"\overline{\zeta_M}"] &
      M^\vee[X^{-1}].
    \end{tikzcd}\]
It then follows from Lemma
\ref{lem:compatibility_of_psis} that, identifying $\phi^*(M^\vee)\simeq \F\bbra{X}\otimes_{\varphi,\F\bbra{X}}M^\vee$ with
$(\F\bbra{X}\otimes_{\varphi,\F\bbra{X}}M)^\vee$ via \eqref{correctformula}, the following diagram is commutative:
\begin{equation}\label{eq:commutativity_DA_Dxi}
    \begin{tikzcd}
      D_A(\pi)\ar[r,twoheadrightarrow] \ar[d,"\beta"]& D_A(\pi)/\mathfrak{p}\ar[r,twoheadrightarrow,"\overline{\zeta_M}"] \ar[d,"\overline{\beta}"] &
      M^\vee[X^{-1}]\ar[d,"(\Id\otimes F)^\vee"]\\
      \phi^*(D_A(\pi))\ar[r,twoheadrightarrow] & \phi^*(D_A(\pi)/\mathfrak{p})
      \ar[r,twoheadrightarrow,"\Id\otimes\overline{\zeta_M}"]&\phi^*(M^\vee[X^{-1}]),
    \end{tikzcd}\end{equation}
where $(\Id\otimes F)^\vee$ comes from $\F$-linear dual of $\Id\otimes F:\F\bbra{X}\otimes_{\varphi,\F\bbra{X}}M\rightarrow M$.
As $(\Id\otimes F)^\vee$ is an isomorphism (see just after \eqref{correctformula}), the map
$\zeta_M : D_A(\pi)\twoheadrightarrow M^\vee[X^{-1}]$ factors
through $D_A(\pi)^{\et}$ and the map
$\overline{\zeta_M} : D_A(\pi)/\mathfrak{p}\onto
M^\vee[X^{-1}]$ factors through $D_A(\pi)^{\et}/\mathfrak{p}$. The
map $\overline{\zeta_M} : D_A(\pi)^{\et}/\mathfrak{p}\twoheadrightarrow
M^\vee[X^{-1}]$ clearly commutes with the action of $\Zp^\times$ and
the commutative diagram \eqref{eq:commutativity_DA_Dxi} shows that
it is a morphism $\varphi$-modules. These maps are obviously
compatible when $M$ is varying among the finitely
generated $\F\bbra{X}[F]$-submodules of $\pi^{N_1}$ that are
admissible as $\F\bbra{X}$-modules and $\Zp^\times$-stable so that we obtain a
map
\[ \zeta : D_A(\pi)^{\et}/\mathfrak{p} \longrightarrow \varprojlim_M
M^\vee[X^{-1}]=D_\xi^\vee(\pi).\]

We prove that the map $\zeta$ is surjective. Since
$D_A(\pi)^{\et}/\mathfrak{p}$ is a finite-dimensional
$\F\ppar{X}$-vector space, the dimension of the vector spaces
$M^\vee[X^{-1}]$ when $M$ is varying is bounded. This implies that
there exists some $M$ such that $D_\xi^\vee(\pi)=M^{\vee}[X^{-1}]$ and
that the map $\zeta : D_A(\pi)^{\et}/\mathfrak{p}\rightarrow D_\xi^\vee(\pi)$ is
surjective. In particular, $\dim_{\F\ppar{X}}D_\xi^\vee(\pi)<+\infty$.

We prove that the map $\zeta$ is an isomorphism. Let
$D^\natural(\pi)^{\et}$ be the image of $\pi^\vee$ in
$D_A(\pi)^{\et}/\mathfrak{p}$. This is a compact $\F\bbra{X}$-module
in the finite-dimensional $\F\ppar{X}$-vector space
$D_A(\pi)^{\et}/\mathfrak{p}$, hence a finite free
$\F\bbra{X}$-module.  Since the maps
$\pi^\vee\rightarrow D_A(\pi)/\mathfrak{p}\twoheadrightarrow
D_A(\pi)^{\et}/\mathfrak{p}$ commute with the action of $\Zp^\times$,
$D^\natural(\pi)^{\et}$ is preserved by $\Zp^\times$.  The image of
$(\pi^{\vee})_S$ in $D_A(\pi)^{\et}/\mathfrak{p}$ coincides with
$D^\natural(\pi)^{\et}[X^{-1}]$. As $(\pi^{\vee})_S$ has a dense image
in $D_A(\pi)$ by definition, $D^\natural(\pi)^{\et}[X^{-1}]$ is a
dense $\F\ppar{X}$-vector subspace of $D_A(\pi)^{\et}/\mathfrak{p}$
and thus equal to $D_A(\pi)^{\et}/\mathfrak{p}$ by finiteness of the
dimension. The surjective map
$\pi^\vee\twoheadrightarrow D^\natural(\pi)^{\et}$ factors through
$\pi^\vee/\m_{N_1}\simeq(\pi^{N_1})^\vee$ so that the topological
$\F$-linear dual $(D^\natural(\pi)^{\et})^\vee$ of
$D^\natural(\pi)^{\et}$ is identified with an $\F\bbra{X}$-submodule
of $\pi^{N_1}$ (endowed with the discrete topology) preserved by
$\Zp^\times$. As $D^\natural(\pi)^{\et}$ is stable by
$\overline{\psi}$ by (\ref{psipida}), $(D^\natural(\pi)^{\et})^\vee$
is actually an $\F\bbra{X}[F]$-submodule of $\pi^{N_1}$. Since
$\beta^{\et} : D_A(\pi)^{\et}\xrightarrow{\sim}\phi^*(D_A(\pi)^{\et})$
is an isomorphism, it easily follows from Lemma
\ref{lem:compatibility_of_psis} that the map $\overline{\beta}$
induces a surjective map of finite-dimensional $\F\ppar{X}$-vector
spaces
$\overline{\beta}^{\et} :
D_A(\pi)^{\et}/\mathfrak{p}\twoheadrightarrow
\phi^*(D_A(\pi)^{\et}/\mathfrak{p})$. As these spaces have the same
dimension, $\overline{\beta}^{\et}$ is actually an isomorphism, and in
particular
$\overline{\beta}^{\et}|_{D^{\natural}(\pi)^{\et}}:D^{\natural}(\pi)^{\et}\rightarrow
\F\bbra{X}\otimes_{\varphi,\F\bbra{X}}D^{\natural}(\pi)^{\et}$ is an
injection and becomes an isomorphism after inverting $X$. 

We claim that $(D^{\natural}(\pi)^{\et})^\vee$ is finitely generated as an $\F\bbra{X}[F]$-module.
Note that $(D^{\natural}(\pi)^{\et})^\vee$ is
admissible as an $\F\bbra{X}$-module since $D^\natural(\pi)^{\et}$ is
a finitely generated $\F\bbra{X}$-module.  Hence, the claim follows from \cite[Lemma 5.2]{breuil-foncteur} using the last statement of the previous paragraph. 

We now give  another proof of the claim using results of \cite[\S4]{Lyubeznik}. 
In fact, we even prove that $(D^{\natural}(\pi)^{\et})^\vee$ is of finite length as an $\F\bbra{X}[F]$-module.
As $\F$ is a finite extension of $\Fp$, the
$\Fp\bbra{X}$-module $(D^{\natural}(\pi)^{\et})^\vee$ is artinian so
that the $\Fp\bbra{X}[F]$-module $(D^{\natural}(\pi)^{\et})^\vee$ is a
cofinite $\Fp\bbra{X}[F]$-module in the sense of \cite[\S4]{Lyubeznik}
(the ring $\Fp\bbra{X}[F]$ is isomorphic to the ring $A\set{f}$ of
\emph{loc.~cit.}  where $A=\Fp\bbra{X}$). It follows from Theorem 4.7
in \emph{loc.~cit.}~that $(D^{\natural}(\pi)^{\et})^\vee$ has a
filtration
\[ 0=M_0\subset M_1\subset\cdots\subset
  M_r=(D^{\natural}(\pi)^{\et})^\vee\] by $\Fp\bbra{X}[F]$-submodules
such that $M_{i+1}/M_i$ is a simple $\Fp\bbra{X}[F]$-module or a
nilpotent $\Fp\bbra{X}[F]$-module, i.e.~such that some power of $F$ is
zero on $M_{i+1}/M_i$. Let $M_i^\perp$ be the kernel of
$D^\natural(\pi)^{\et}\rightarrow M_i^\vee$ for all $i$. As
$\overline{\beta}^{\et}|_{D^{\natural}(\pi)^{\et}}$ coincides with
$(\Id\otimes F)^\vee$ (this is analogous to
(\ref{eq:commutativity_DA_Dxi}) using \eqref{correctformula} with
$M=(D^{\natural}(\pi)^{\et})^\vee$), the map $\overline{\beta}^{\et}$
induces an isomorphism of $\Fp\ppar{X}\otimes_{\Fp\bbra{X}}M_i^\perp$
onto $\Fp\ppar{X}\otimes_{\varphi,\Fp\bbra{X}}M_i^\perp$. In particular, if
$M_{i+1}/M_i$ is nilpotent then $F^\vee$ induces a nilpotent
endomorphism of $M_i^\perp/M_{i+1}^\perp$ so that $\Fp\ppar{X}\otimes_{\Fp\bbra{X}}M_i^\perp=\Fp\ppar{X}\otimes_{\Fp\bbra{X}}M_{i+1}^\perp$ (as 
$\Fp\ppar{X}\otimes_{\varphi,\Fp\bbra{X}}(M_i^\perp/M_{i+1}^\perp)=0$ in this case)
and hence $M_{i}^\perp/M_{i+1}^\perp$ is a torsion $\Fp\bbra{X}$-module. As
$D^\natural(\pi)^{\et}$ is a finitely generated $\Fp\bbra{X}$-module, we conclude that when $M_{i+1}/M_i$ is nilpotent the $\Fp\bbra{X}$-module $M_{i}^\perp/M_{i+1}^\perp$ is finite-dimensional over $\Fp$, in particular it is an $\Fp\bbra{X}[F]$-module of finite length.
Since $M_{i+1}/M_i$ is obviously of finite length when $M_{i+1}/M_i$ is irreducible, the claim follows.

The claim implies that  
$(D^\natural(\pi)^{\et})^\vee$ is one of the modules
$M\subseteq \pi^{N_1}$ in \S\ref{covariant}, in particular
\[\dim_{\F\ppar{X}}D_\xi^\vee(\pi)\geq \dim_{\F\ppar{X}}(D^\natural(\pi)^{\et}[X^{-1}])=\dim_{\F\ppar{X}}(D_A(\pi)^{\et}/\mathfrak{p}).\]
This implies that the map $\zeta$ is an isomorphism (and that $D_A(\pi)^{\et}/\mathfrak{p}=D^\natural(\pi)^{\et}[X^{-1}]\cong D_\xi^\vee(\pi)$). The very last statement follows from Theorem \ref{thm:exactnessDA}(ii).
\end{proof}

\subsubsection{An upper bound for the ranks of \texorpdfstring{$D_A(\pi)^{\et}$}{D\_A(pi)\^{}\textbackslash et} and \texorpdfstring{$D_\xi^\vee(\pi)$}{D\_\{xi\}\^{}v(pi)}}\label{upperb}

For $\pi$ in $\mathcal C$ we bound the dimension of $D_\xi^\vee(\pi)$ in terms of $\gr(\pi^\vee)$. When $\gr(\pi^\vee)$ is killed by some $J^n$, we give an interpretation of this bound as a certain multiplicity.

We keep all previous notation. We start with the following lemma.

\begin{lem}\label{lem:comparison_ranks}
Let $M$ be a finitely generated $A$-module endowed with a good
filtration. Then the generic rank of the $A$-module $M$ and the
generic rank of the $\gr(A)$-module $\gr(M)$ coincide.
\end{lem}
\begin{proof}
We first note that if $N$ is an $A$-module of generic rank $0$, then
$N\otimes_A\mathrm{Frac}(A)=0$ and $N$ is a torsion module. This
implies that $\gr(N)$ is a torsion module and that its generic rank
is $0$.

Let $d$ be the generic rank of $M$ and
$f : A^{\oplus d}\rightarrow M\otimes_A\mathrm{Frac}(A)$ be a morphism of
$A$-modules sending an $A$-basis of the left-hand side to a
$\mathrm{Frac}(A)$-basis of the right-hand side. The kernel of $f$
is then a torsion $A$-submodule of $A^{\oplus d}$ and is zero since $A$ is a
domain. Moreover there exists $a\in A\backslash\set{0}$ such that the
image of $af$ is contained in $M$. As $\mathrm{Frac}(A)$ is a flat
$A$-module, the generic rank is an additive map on the abelian
category of finitely generated $A$-modules. As $af$ is injective and
$A^{\oplus d}$ and $M$ have identical generic ranks, this implies that the
cokernel $Q$ of $af$ has generic rank $0$. We fix a good filtration
on $M$: it induces good filtrations on $af(A^{\oplus d})$ and on $Q$. For
these filtrations we have a short exact sequence
\[ 0\longrightarrow \gr(af(A^{\oplus d}))\longrightarrow \gr(M)
\longrightarrow \gr(Q)\longrightarrow0.\]
As $Q$ has generic rank
$0$, so does $\gr(Q)$ so that it suffices to prove that
$\gr(af(A^{\oplus d}))$ has generic rank $d$. It follows from the second
paragraph after \cite[Def.4.2]{Bj89} that, for a finitely generated
$A$-module $N$, the generic rank of $\gr(N)$ does not depend on the
choice of good filtration. We can thus choose a good filtration
$af(A^{\oplus d})\simeq A^{\oplus d}$ which is filtered free with respect to the canonical
basis of $A^{\oplus d}$, for which the result is obvious.
\end{proof}

Let $\pi$ be in the category $\mathcal C$ and choose a good filtration on the $\F\bbra{I_1/Z_1}$-module $\pi^\vee$. Since the finitely generated $A$-module $D_A(\pi)$ doesn't depend up to isomorphism on the choice of this good filtration (see \S\ref{multivariablepsi}), it follows from Lemma \ref{lem:comparison_ranks} (applied to $M=D_A(\pi)$) and Lemma \ref{lem:grlocal} (applied to $M=\pi^\vee$) that the generic rank of $\gr(A)\otimes_{\gr(\F\bbra{N_0})}\gr(\pi^\vee)$ also doesn't depend on this choice.

\begin{prop}\label{boundC}
Let $\pi\in \mathcal C$. Then $\rk_A(D_A(\pi)^{\et})=\dim_{\F\ppar{X}}D_\xi^\vee(\pi)$ is bounded by the generic rank of the $\gr(A)$-module $\gr(A)\otimes_{\gr(\F\bbra{N_0})}\gr(\pi^\vee)$. 
\end{prop}
\begin{proof}
As $D_A(\pi)^{\et}$ is a quotient of $D_A(\pi)$, the result follows from Lemma \ref{lem:comparison_ranks}, Lemma \ref{lem:grlocal} and Theorem \ref{thm:functors_comparison}.
\end{proof}

When $\gr(\pi^\vee)$ is moreover killed by the ideal $J^n$ for some $n\geq 1$ (here $J$ is as in (\ref{idealJ}) and recall this doesn't depend on the good filtration), the generic rank of $\gr(A)\otimes_{\gr(\F\bbra{N_0})}\gr(\pi^\vee)$ has a nice and useful interpretation that we give now.

We define $\overline{R}\defeq \gr(\F\bbra{I_1/Z_1})/J$. Recall using (\ref{gri1}) that we have
\begin{equation}\label{rbarJ}
\overline{R}\simeq \F[y_i,z_i, 0\leq i\leq f-1]/(y_iz_i, 0\leq i\leq f-1).
\end{equation}
Therefore $\overline{R}$ has $2^f$ minimal prime ideals which are the ideals $(y_i,z_j, i\in \cJ, j\notin \cJ)$
with $\cJ$ a subset of $\set{0,\dots,f-1}$. Let
\[\mathfrak{p}_0\defeq (z_j, 0\leq j\leq f-1)\]
be the minimal prime ideal corresponding to the choice of $\cJ=\emptyset$. 

If $N$ is a finitely generated module over $\overline{R}$ and $\q$ is a minimal prime ideal of $\overline{R}$, we denote by $m_{\q}(N)$ the length of $N_{\q}$ over $\overline{R}_{\q}$. More generally, if $N$ is a finitely generated $\gr(\F\bbra{I_1/Z_1})$-module annihilated by $J^n$ for some $n\geq 1$, we define the multiplicity of $N$ at $\q$ to be
\begin{equation}\label{def:multiatp}
m_{\q}(N)=\sum_{i=0}^{n-1}m_{\q}(J^iN/J^{i+1}N).
\end{equation}

\begin{lem}\label{lem:m-additive}
If $0\ra N_1\ra N\ra N_2\ra 0$ is a short exact sequence of finitely generated $\gr(\F\bbra{I_1/Z_1})/J^n$-modules, then $m_{\q}(N)=m_{\q}(N_1)+m_{\q}(N_2)$. 
\end{lem} 
\begin{proof}
This is checked by a standard d\'evissage. If $n=1$, the statement is obvious since $\gr(\F\bbra{I_1/Z_1})/J=\overline R$ is commutative (and noetherian). Assume $n\geq 2$ and by induction we assume that the result holds if $N$ is annihilated by $J^{n-1}$. 

Assume first that $N_1$ and $N_2$ are both annihilated by $J^{n-1}$ (but not necessarily $N$). Then $N_2$ is a quotient of $N/J^{n-1}N$. Let $\Ker\defeq \Ker(N/J^{n-1}N\twoheadrightarrow N_2)$ be the corresponding kernel. Then we have two short exact sequences
\[0\ra \Ker \ra N/J^{n-1}N\ra N_2\ra 0\]
\begin{equation}\label{eq:N1}
0\ra J^{n-1}N\ra N_1\ra \Ker\ra0.
\end{equation}
By definition of $m_{\q}(N)$ and the inductive hypothesis, we then obtain
\[m_{\q}(N)=m_{\q}(J^{n-1}N)+m_{\q}(N/J^{n-1}N)= m_{\q}(N_1)+m_{\q}(N_2).\]

Assume now that $N_2$ is annihilated by $J^{n-1}$ (but not necessarily for $N_1$). Then the surjection $N\twoheadrightarrow N_2$ factors through the quotient $N/J^{n-1}N$ of $N$. Again let $\Ker\defeq \Ker(N/J^{n-1}N\twoheadrightarrow N_2)$. Then $m_{\q}(N/J^{n-1}N)=m_{\q}(\Ker)+m_{\q}(N_2)$ by the inductive hypothesis. On the other hand, both $J^{n-1}N$ and $\Ker$ are annihilated by $J^{n-1}$, 
thus $m_{\q}(\cdot)$ is additive for the short exact sequence \eqref{eq:N1} by the discussion in last paragraph. The result also holds in this case. 

To finish the proof it suffices to decompose further $N$ as $0\ra \Ker'\ra N\ra N_2/J^{n-1}N_2\ra0$, with $\Ker'$ sitting in the exact sequence $0\rightarrow N_1\rightarrow \Ker'\rightarrow J^{n-1}N_2\rightarrow 0$, and apply the above discussion. 
\end{proof}

If $N$ is a finitely generated module over $\gr(\F\bbra{I_1/Z_1})/J^n$ for some $n\geq 1$ recall that the $\gr(A)$-module $\gr(A)\otimes_{\gr(\F\bbra{N_0})}N$ is finitely generated by Proposition \ref{prop:finiteness}.

\begin{lem}\label{lem:rank=multi}
Let $N$ be a finitely generated module over $\gr(\F\bbra{I_1/Z_1})/J^n$ for some $n\geq 1$. Then the generic rank of the $\gr(A)$-module $\gr(A)\otimes_{\gr(\F\bbra{N_0})}N$ is equal to $m_{\p_0}(N)$.
\end{lem}
\begin{proof}
By Corollary \ref{lem:grA}, $\gr(A)$ is flat over $\gr(\F\bbra{N_0})$, so $\gr(A)\otimes_{\gr(\F\bbra{N_0})}N$ has a finite filtration with graded pieces given by $\gr(A)\otimes_{\gr(\F\bbra{N_0})}(J^iN/J^{i+1}N)$ for $0\leq i\leq n-1$. Since taking generic rank and taking $m_{\p_0}(\cdot)$ are both additive in short exact sequences (by Lemma \ref{lem:m-additive} for the latter), we are reduced to the case where $N$ is killed by $J$. 

In that case we have
\[\gr(A)\otimes_{\gr(\F\bbra{N_0})}N\cong (\gr(A)\otimes_{\gr(\F\bbra{N_0})}\overline{R})\otimes_{\overline{R}}N.\]
Since the image of $\gr(\F\bbra{N_0})$ in $\overline{R}$ is $\F[y_0,\dots,y_{f-1}]$, we have
\[\gr(A)\otimes_{\gr(\F\bbra{N_0})}\overline{R}\simeq\overline{R}[(y_0\cdots y_{f-1})^{-1}]\simeq \gr(A).\]
Since the fraction field of $\overline{R}[(y_0\cdots y_{f-1})^{-1}]$ is just $\overline{R}_{\p_0}$, we see that the generic rank of the $\overline{R}[(y_0\cdots y_{f-1})^{-1}]$-module $\gr(A)\otimes_{\gr(\F\bbra{N_0})}N$ is equal to $m_{\mathfrak{p}_0}(N)$. 
\end{proof}

We finally deduce from Proposition \ref{boundC} and Lemma \ref{lem:rank=multi}:

\begin{cor}\label{cor:upperbound}
Let $\pi$ be an admissible smooth representation of $\GL_2(K)$ over
$\F$ with a central character having at least one good filtration such that the $\gr(\F\bbra{I_1/Z_1})$-module $\gr(\pi^\vee)$
is killed by some power of $J$. Then we have
\[ \rk_A(D_A(\pi)^{\et})=\dim_{\F\ppar{X}}D_\xi^\vee(\pi)\leq
m_{\mathfrak{p}_0}(\gr(\pi^\vee)).\]
\end{cor}

\subsection{Tensor induction for \texorpdfstring{$\GL_2(\Q_{p^f})$}{GL\_2(Q\_\{p\^{}f\})}}\label{tensorinduction}

We prove that $V_{\GL_2}(\pi)$ (as defined in (\ref{VG})) contains some copies of a tensor induction as in Example \ref{indn=2} for certain admissible smooth representations $\pi$ of $\GL_2(K)$ over $\F$ (Theorem \ref{thm:main-tensor-ind}).

We recall that the definition of the functor $V_{\GL_2}$ depends on the choice of a cocharacter $\xi_{\GL_2}$, which we have fixed to be $\xi_{\GL_2}(x) = {\rm diag}(x,1)$, 
and depends on a normalizing character $\delta_{\GL_2} = \ind_K^{\otimes\Qp}\!(\omega)$ (cf.\ Example \ref{exdelta}).

\subsubsection{Lower bound for \texorpdfstring{$V_{\GL_2}(\pi)$}{V\_\{GL\_2\}(pi)}: statement}\label{lowerstatement}

We state the main theorem of this section on $V_{\GL_2}(\pi)$ for certain admissible smooth representations $\pi$ of $\GL_2(K)$ over $\F$ (Theorem \ref{thm:main-tensor-ind}). After some simple reductions, this theorem will be proved in \S\S\ref{prelgl2} to \ref{lowerproof}.

We keep all the previous notation and denote by $I_K$ the inertia subgroup of $\gK$. We fix an embedding $\sigma'_0:\F_{p^{2f}}\hookrightarrow \F$ such that $\sigma'_0\vert_{\F_{p^{f}}}=\sigma_0$ (see the very beginning of \S\ref{gl2}), and denote by $\omega_f,\omega_{2f}:I_K\rightarrow \F^\times$ Serre's corresponding fundamental characters of level $f$ and $2f$.

We consider $\rhobar:\gK\rightarrow \GL_2(\F)$ of the following form {\it up to twist}:
\begin{equation}\label{eq:0}
\rhobar|_{I_K} \cong
\begin{cases}
\omega_{f}^{\sum_{j=0}^{f-1}(r_j+1)p^j}\oplus 1&\text{ if $\rhobar$ is reducible,}
\\
\omega_{2f}^{\sum_{j=0}^{f-1} (r_{j}+1)p^j}\oplus \omega_{2f}^{\sum_{j=0}^{f-1}(r_j+1)p^{j+f}}
&\text{ if $\rhobar$ is irreducible,}
\end{cases}
\end{equation}
where the integers $r_i$ satisfy the following (strong) genericity condition:
\begin{equation}\label{eq:6}
\begin{aligned}
2f-1 &\le r_j \le p-2-2f &&\quad \text{if $j > 0$ or $\rhobar$ is reducible,}\\
2f &\le r_0 \le p-1-2f &&\quad \text{if $\rhobar$ is irreducible}
\end{aligned}
\end{equation}
(note that this implies in particular $p \ge 4f+1$). Let $\chi:\gK\rightarrow \F^\times$ such that $(\rhobar\otimes \chi)|_{I_K}$ is as in (\ref{eq:0}) and moreover $\det(\rhobar\otimes \chi) = \omega_f^{\sum_j(r_j+1)p^j}$\!.

We refer to \cite{paskunas-coefficients} and \cite[\S\S9,13]{BP} (and the references therein) for the background and definitions about {\it diagrams}. 

We choose {\it one} diagram $D(\rhobar\otimes \chi)=(D_1\hookrightarrow D_0)$ associated to $\rhobar\otimes \chi$ in \cite[\S5]{breuil-IL}, and we set
\begin{equation}\label{diagramchoice}
D(\rhobar)=(D_1(\rhobar)\hookrightarrow D_0(\rhobar))\defeq \big(D_1\otimes (\chi^{-1}\circ{\det})\hookrightarrow D_0\otimes (\chi^{-1}\circ{\det})\big),
\end{equation}
where the actions of $\GL_2(\cO_K)$ and the center $K^\times$ on $D_0(\rhobar)$ (resp.\ of $I$, $\smatr{0}{1}{p}{0}$ and $K^\times$ on $D_1(\rhobar)$) are multiplied by $\chi^{-1}\circ{\det}$ via local class field theory for $K$ (note that $\chi$ is trivial on $K_1$ and $I_1$ and recall that $\smatr{0}{1}{p}{0}$ normalizes $I$ and $I_1$). Recall that the action of $\GL_2(\cO_K)$ on $D_0(\rhobar)$ factors through $\GL_2(\cO_K)\twoheadrightarrow \GL_2(\F_q)$. More precisely, denoting by $W(\rhobar)$ the set of Serre weights of $\rhobar$ defined in \cite[\S3]{BDJ}, $D_0(\brho)$ is the (unique) maximal finite-dimensional representation of $\GL_2(\F_q)$ over $\F$ with socle isomorphic to $\oplus_{\sigma\in W(\rhobar)}\sigma$ such that each $\sigma\in W(\rhobar)$ occurs with multiplicity $1$ in $D_0(\brho)$. 
(For instance, recall that the Serre weight $(r_0, r_1, \dots, r_{f-1})$ with the notation of (\ref{Serreweight}) below is in $W(\brho)$ if $\brho$ is an extension of $1$ by $\omega_{f}^{\sum_{j=0}^{f-1}(r_j+1)p^j}$.)
Finally $K^\times$ acts on $D_0(\rhobar)$ by the character $\det(\rhobar)\omega^{-1}$.

If $\pi$ is an admissible smooth representation of $\GL_2(K)$ over $\F$, recall that $(\pi^{I_1}\hookrightarrow \pi^{K_1})$ is naturally a diagram. We aim to prove the following theorem.

\begin{thm}\label{thm:main-tensor-ind}
Let $\pi$ be an admissible smooth representation of $\GL_2(K)$ over $\F$. Assume that there exists an integer $r\geq 1$ such that one has an isomorphism of diagrams
\[D(\rhobar)^{\oplus r}\buildrel\sim\over\longrightarrow (\pi^{I_1}\hookrightarrow \pi^{K_1}).\]
Then one has an $I_{\Qp}$-equivariant injection $\big(\ind_{K}^{\otimes \Qp}\!(\rhobar)\big)\vert_{I_{\Qp}}^{\oplus r}\hookrightarrow V_{\GL_2}(\pi)\vert_{I_{\Qp}}$. If we assume moreover that the constants $\nu_i$ associated to $D(\rhobar\otimes \chi)$ at the beginning of {\upshape\cite[\S6]{breuil-IL}} are as in {\upshape\cite[Thm.6.4]{breuil-IL}}, then one has a $\gp$-equivariant injection $\big(\ind_{K}^{\otimes \Qp}\!(\rhobar)\big)^{\oplus r}\hookrightarrow V_{\GL_2}(\pi)$.
\end{thm}

Let us first make some straightforward reductions. In order not to repeat arguments, we assume from now on that the constants $\nu_i$ associated to $D(\rhobar\otimes \chi)$ in \cite[\S6]{breuil-IL} are as in \cite[Thm.6.4]{breuil-IL} and we will prove the last statement of Theorem \ref{thm:main-tensor-ind} (the proof for the first one being the same up to some trivial modifications). It is enough to prove Theorem \ref{thm:main-tensor-ind} for the $\GL_2(K)$-subrepresentation of $\pi$ generated by $D_0(\rhobar)^{\oplus r}$. Hence we can assume that $\pi$ has a central character which is $\chi_\pi\defeq \det(\rhobar)\omega^{-1}$. Using Remark \ref{trivial}(ii) (for $n=2$), it is also enough to prove Theorem \ref{thm:main-tensor-ind} for $\rhobar\otimes \chi$ as above and replacing $\pi$ by $\pi\otimes\chi\circ \det$, i.e.\ we can assume $\rhobar|_{I_K}$ is as in (\ref{eq:0}) and $\det(\rhobar) = \omega_f^{\sum_j(r_j+1)p^j}$.

In the sequel, for any $\F\bbra{X}[F]$-submodule $M$ of $\pi^{N_1}$ which is stable under $\Zp^\times$, denote by $M\otimes \chi_\pi^{-1}$ the same $\F\bbra{X}$-module but where the action of $F$ is multiplied by $\chi_\pi(p)^{-1}$ and the action of $x\in \Zp^\times$ is multiplied by $\chi_\pi(x)^{-1}$. 

\begin{lem}\label{twistanddual}
With the notation in \S\ref{covariant}, in order to prove Theorem \ref{thm:main-tensor-ind} it is enough to prove that $(\pi \otimes \chi_\pi^{-1})^{N_1}$ contains a finite type $\F\bbra{X}[F]$-submodule $M$ which is admissible as an $\F\bbra{X}$-module and stable under $\Zp^\times$ such that ${\bf V}(M^\vee[1/X])\cong \big(\ind_{K}^{\otimes \Qp}\!(\rhobar)\big)^{\oplus r}$.
\end{lem}
\begin{proof}
As $(\pi \otimes \chi_\pi^{-1})^{N_1}=\pi^{N_1}$ as $\F$-vector subspaces of $\pi$, it is equivalent to assume that $\pi^{N_1}$ contains a finite type $\F\bbra{X}[F]$-submodule $M$ which is admissible as an $\F\bbra{X}$-module and stable under $\Zp^\times$ such that ${\bf V}((M\otimes \chi_\pi^{-1})^\vee[1/X])\cong \big(\ind_{K}^{\otimes \Qp}\!(\rhobar)\big)^{\oplus r}$. From the definition of $V_{\GL_2}$ in (\ref{VG}), it is enough to prove ${\bf V}^\vee(M^\vee[1/X])\otimes \delta_{\GL_2}\cong \big(\ind_{K}^{\otimes \Qp}\!(\rhobar)\big)^{\oplus r}$. From Example \ref{exdelta} and as in Remark \ref{trivial}(ii) (both for $n=2$), we have
\begin{eqnarray*}
{\bf V}^\vee(M^\vee[1/X])\otimes \delta_{\GL_2}&=&{\bf V}\big((M\otimes \chi_\pi^{-1})^\vee[1/X]\big)^\vee\otimes (\chi_{\pi}\omega)\vert_{\Qp^{\times}}\\
&=&\big(\big(\ind_{K}^{\otimes \Qp}\!(\rhobar)\big)^{\oplus r}\big)^\vee\otimes \ind_K^{\otimes\Qp}\!\!\big({\det(\rhobar)}\big)\\
&\buildrel (\ref{dual})\over =& \big(\ind_{K}^{\otimes \Qp}\!(\rhobar)\big)^{\oplus r}
\end{eqnarray*}
which finishes the proof.
\end{proof}

The sections that follow will be devoted to the proof that there exists a certain finite type $\F\bbra{X}[F]$-submodule $M_\pi$ of $\pi^{N_1}$ which is admissible as an $\F\bbra{X}$-module and stable under $\Zp^\times$ such that ${\bf V}((M_\pi\otimes \chi_\pi^{-1})^\vee[1/X])\cong \big(\ind_{K}^{\otimes \Qp}\!(\rhobar)\big)^{\oplus r}$ (see Proposition \ref{prop:main:functor}). Note that the assumption $\det(\rhobar) = \omega_f^{\sum_j(r_j+1)p^j}$ implies $\chi_\pi(p)=1$, so that the operator $F$ on $M_\pi\otimes \chi_\pi^{-1}$ is the same as on $M_\pi$, but the action of $\gamma\in \Zp^\times$ now comes from the action of $\smatr{1}{0}{0}{\gamma^{-1}}$ on $\pi^{N_1}$.

\subsubsection{Preliminaries}\label{prelgl2}

We give some technical results on $\F\bbra{N_0}$, $\F\bbra{N_0/N_1}$ and on certain modules over these rings coming from Serre weights.

We let ${H}\defeq \smatr{\F_q^{\times}}00{\F_q^{\times}}\simeq I/I_1\subseteq \GL_2(\Fq)$ (this finite group $H$ shouldn't be confused with the algebraic group $H$ in \S\ref{covariant} or in \S\ref{global}). Note that the trace ${\rm Tr}_{K/\Qp}:\oK\rightarrow \Zp$ is surjective (using that $K$ is unramified) hence directly induces an isomorphism $N_0/N_1\buildrel\sim\over\rightarrow \Zp$. Recall we defined the elements $Y_i$ for $i\in \{0,\dots,f-1\}$ in (\ref{elementsyi}). We define analogously
\[Y\defeq  \sum_{a\in \Fp^\times}a^{-1}\begin{pmatrix}
    1 & \tld{a} \\ 0 & 1
  \end{pmatrix}\in \F\bbra{\Zp} =\F\bbra{N_0/N_1}.\]
We write $\un{i}$ for an element $(i_0,\dots,i_{f-1})$ in $\Z^f$, $\un{Y}^{\un{i}}$ for $Y_0^{i_0}\cdots Y_{f-1}^{i_{f-1}}$ and set $\|\un{i}\|\defeq  \sum_{j=0}^{f-1}i_j$. We also write $\un i \le \un{i}'$ to mean $i_j \le i'_j$ for all $0 \le j \le f-1$.

\begin{lem}\label{lem:basic}
We have the following isomorphisms and equalities:
\begin{enumerate}
\item $\F\bbra{N_0}\cong\F\bbra{Y_0,\dots,Y_{f-1}} $ and $\F[N_0/N_0^p]\cong\F\bbra{Y_0,\dots,Y_{f-1}} /(Y_0^p,\dots,Y_{f-1}^p)$;
\item $Y_i^p\smatr{p}{0}{0}{1}=\smatr{p}{0}{0}{1}Y_{i+1}$ and $\smatr{\tld{\lambda}}{0}{0}{\tld{\mu}}Y_i=(\lambda\mu^{-1})^{p^i}Y_i\smatr{\tld{\lambda}}{0}{0}{\tld{\mu}}$ for $\lambda,\mu\in \Fq^\times$;
\item $\F\bbra{N_0/N_1} \cong \F\bbra{Y}$ and $\smatr{\tld{\lambda}}{0}{0}{\tld{\mu}}Y=(\lambda\mu^{-1})Y\smatr{\tld{\lambda}}{0}{0}{\tld{\mu}}$ for $\lambda,\mu\in \Fp^\times$. 
\end{enumerate}

\end{lem}
\begin{proof}
Note that $\F[N_0/N_0^p]\cong \F\big[\smatr{1}{\Fq}{0}{1}\big]$. The first equality in (i) and the explicit action of $\smatr{\tld{\lambda}}{0}{0}{\tld{\mu}}$ on $\un{Y}^{\un{i}}$ in (ii) are immediately obtained from \cite[Lemma 3.2]{morra-iwasawa} (after conjugating by the element $\smatr{0}{1}{p}{0}$). The second equality in (i) follows from the first by dimension reasons, as $Y_i^p = 0$ in $\F[N_0/N_0^p]$. The action of $\smatr{p}{0}{0}{1}$ on $Y_{i+1}$ in (ii) is a direct computation (see also \cite[Lemma 5.1]{morra-iwasawa}). Finally, (iii) is a special case of (i) and (ii).
\end{proof}

Note that $\F\bbra{N_0/N_1}\cong \F\bbra{X}\cong \F\bbra{Y}$ with $X=\smatr{1}{1}{0}{1}-1$ as in \S\ref{covariant}, but it is more convenient in the computations to use the ``$H$-eigenvariable'' $Y$ rather than the variable $X$. To compare them the following lemma will be useful.

\begin{lem}\label{lem:XY}
We have $X\in-Y(1+Y\F\bbra{Y})$ and $Y\in-X(1+X\F\bbra{X})$ in $\F\bbra{N_0/N_1}$.
\end{lem}
\begin{proof}
Equivalently, we have to prove ${Y}=-{X}$ in $\fm/\fm^2$, where $\fm$ is the maximal ideal of $\F\bbra{N_0/N_1}$. We can work modulo $\fm^p$, i.e.\ in $\F[ N_0/N_1 N_0^p]\cong \F\big[\smatr{1}{\Fp}{0}{1}\big]$. In that group ring we have
\[Y = \underset{a\in \Fp^\times}{\sum}a^{-1}\smatr{1}{a}{0}{1} = \sum_{a=1}^{p-1}a^{-1}(1+X)^a \equiv -X\]
(where the last congruence is taken modulo the image of $\fm^2$ in that group ring).
\qedhere
\end{proof}

For $\lambda,\mu\in \Fq^\times$ we set
\[\alpha\big(\smatr{\tld{\lambda}}{0}{0}{\tld{\mu}}\big)\defeq \lambda\mu^{-1}\in \F^\times.\]

\begin{rem}\label{switchchi}
By Lemma \ref{lem:basic}(ii), if $V$ is a representation of $\GL_2(\Fq)$ and $v\in V^{H=\chi}$, then $\un{Y}^{\un{i}}v\in V^{H=\chi\alpha^{\un{i}}}$, where $\alpha^{\un{i}}\defeq  \alpha^{\sum_{j=0}^{f-1}i_jp^j}$.
\end{rem}

\begin{lem}\label{lem:ker}
Assume $p>2$. The kernel of the map $h : \F\bbra{N_0} \onto \F\bbra{N_0/N_1} $ is generated by the elements $Y_i-Y_j$ \emph{(}$i\neq j$\emph{)}. Moreover, there exists $f(Y)\in \F\bbra{N_0/N_1}\cong\F\bbra{Y} $ such that $h(Y_i) = Y+Y^pf(Y)$.
\end{lem}
\begin{proof}
Note that ${\rm Tr}_{K/\Qp}(\tld{\lambda}^{p^i})={\rm Tr}_{K/\Qp}(\tld{\lambda})$ for all $\lambda\in \Fq^\times$ and $i\in\Z$, hence $Y_i-Y_j \in \Ker(h)$. As $\F\bbra{N_0} /(Y_i-Y_j, i\neq j)$ and $\F\bbra{N_0/N_1} $ are both isomorphic to power series rings in one variable, the quotient map $\F\bbra{N_0}/(Y_i-Y_j,i\ne j) \onto \F\bbra{N_0/N_1}$ has to be an isomorphism. To establish the final claim it suffices to prove that the image of $Y_0$ in $\F\bbra{Y} /(Y^p)\cong \F[ N_0/N_1 N_0^p]\cong \F\big[\smatr{1}{\Fp}{0}{1}\big]$ is $Y$. We compute
\begin{equation}\label{eq:1}
\sum_{\lambda\in \Fq^\times}\lambda^{-1}\smatr{1}{{\rm Tr}_{\Fq/\Fp}(\lambda)}{0}{1}=
\sum_{a\in \Fp}\bigg(\sum_{\substack{\lambda\in \Fq^\times\\{\rm Tr}_{\Fq/\Fp}(\lambda)=a}}\lambda^{-1}\bigg)\smatr{1}{a}{0}{1}.
\end{equation}
If $a\neq 0$, the index $\lambda$ runs over the distinct roots of $Y^{p^{f-1}}+Y^{p^{f-2}}+\cdots+Y-a=0$, so the inside sum on the right hand side of~\eqref{eq:1} equals $1/a$ (from the last two coefficients). If $a=0$ the index $\lambda$ runs over the distinct roots of $Y^{p^{f-1}-1}+\cdots+Y^{p-1}+1=0$, so the inside sum in~\eqref{eq:1} equals 0 as $p > 2$. Hence the right-hand side of \eqref{eq:1} is just $Y$.
\end{proof}

By Lemma \ref{lem:ker}, if $V$ is a representation of $\GL_2(\Fq)$, then $Y_i=Y$ on $V^{N_1}$.

For $0\leq i\le q-1$, we set
\[\theta_i\defeq  \sum_{\lambda\in \Fq}\lambda^{i}\smatr{1}{\lambda}{0}{1}\in \F[N_0/N_0^p]\cong \F\big[\smatr{1}{\Fq}{0}{1}\big].\]
So $Y_i = \theta_{q-1-p^i}$ in $\F[N_0/N_0^p]$.
In what follows we write $\un{p-1}$ for the constant $f$-tuple $(p-1,p-1,\dots,p-1)\in \Z^f$.

\begin{lem}\label{lem:theta}
Suppose $\un i \in \{0,\dots,p-1\}^f$ and let $i \defeq  \sum_{j=0}^{f-1} i_j p^j$.

\begin{enumerate}
\item We have
  \[
  \theta_i = (-1)^{f-1} \bigg(\prod_{j=0}^{f-1} i_j!\bigg) \un{Y}^{\un{p-1}-\un{i}}
  \]
  in $\F[N_0/N_0^p]$ for $0 \le i<q-1$.
\item For $f_0,\dots,f_{q-1}$ and $\phi$ as defined in {\upshape\cite[\S2]{BP}} we have
  \[
  f_i=(-1)^{f-1} \bigg(\prod_{j=0}^{f-1} i_j!\bigg) \un{Y}^{\un{p-1}-\un{i}}\smatr{0}{1}{1}{0}\phi
  \]
  for $0\leq i<q-1$.
\end{enumerate}
\end{lem}
\begin{proof}
Part (i) follows from \cite[Lemma 0.2]{morra-iwasawa-CG} after conjugation by $\smatr{0}{1}{p}{0}$.
Indeed, in the notation of \emph{loc.cit.}\ we take $m=n=1$ (so that $A_{1,1}$ is the group algebra of $\smatr{1}{0}{p\cO_K/p^2\cO_K}{1}$): we see that $\theta_i$ corresponds (under conjugation) to $F_{\un{i}}$ if $0 \le i \le q-1$, and the constant $\kappa_{\un{p-1}-\un{i}}$ equals $(-1)^{f-1} \big(\prod_{j=0}^{f-1} i_j!\big)^{-1}$. Part (ii) follows immediately from (i) and the definition of $\theta_i$.
\end{proof}

As in \cite{BP} we write $(s_0,s_1,\dots, s_{f-1})\otimes\eta$ for the Serre weight
\begin{equation}
\label{Serreweight}
\s^{s_0}\F^2\otimes_{\F}(\s^{s_1}\F^2)^{\rm Fr}\otimes\cdots\otimes_{\F}(\s^{s_{f-1}}\F^2)^{{\rm Fr}^{f-1}}\otimes_{\F}\eta\circ\det,\end{equation}
where the $s_i$ are integers between $0$ and $p-1$, $\eta$ is a character $\Fq^{\times}\rightarrow \F^{\times}$ and $\GL_2(\Fq)$ acts on $(\s^{s_i}\F^2)^{{\rm Fr}^i}$ via $\sigma_i:\Fq\hookrightarrow \F$. If $\chi = \chi_1 \otimes \chi_2$ is a character of $H=\smatr {\Fq^\times}00{\Fq^\times}$, we let $\chi^s \defeq  \chi_2 \otimes \chi_1$.

\begin{lem}\label{lem:eigenvct}
Let $\sigma\defeq (s_0,\dots, s_{f-1})\otimes\eta$, $\un{s}\defeq (s_0,s_1,\dots, s_{f-1})\in \{0,\dots,p-1\}^f$, and fix $v\in \sigma^{N_0}$, $v\neq 0$. Let $\chi_\sigma$ denote the $H$-eigencharacter on $\sigma^{N_0}$.
\begin{enumerate}
\item The $\F\bbra{N_0/N_1} =\F\bbra{Y} $-module $\sigma^{N_1}$ is cyclic of dimension $\min\{s_0,\dots,s_{f-1}\}+1$.
\item If $\un{0} \leq \un{i}\leq \un{s}$ and $\un{i} < \un{p-1}$ then $\sigma$ contains a unique $H$-eigenvector, which we call $\un{Y}^{-\un{i}}v$, that is sent by $\un{Y}^{\un{i}}$ to $v$. The corresponding $H$-eigencharacter is $\chi_\sigma \alpha^{-\un i}$. Also, $Y_j \un{Y}^{-\un{i}}v=0$ if $i_j=0$.
\item If $0 \le i\leq \min\{s_0,\dots,s_{f-1}\}$ and $i<p-1$, then $\sigma^{N_1}$ contains a unique
  $\smatr{\Fp^\times}{0}{0}{\Fp^\times}$-eigenvector $Y^{-i}v$ that is sent by $Y^i$ to $v$. The
  corresponding eigencharacter is $\chi_\sigma \alpha^{-i}$. We have
  $Y^{-i}v=\sum_{\un{i},\|\un{i}\|=i}\un{Y}^{-\un{i}}v$.
\end{enumerate}
\end{lem}
\begin{proof}
(i) Note that $\sigma^{N_1}$ is a torsion module over $\F\bbra{N_0/N_1} =\F\bbra{Y}$ as $\sigma^{N_1}$ is finite-dimensional. To show cyclicity it suffices to note that $\sigma^{N_0}=\sigma^{N_1}[X]$ is $1$-dimensional.
Then from \cite[Prop.3.3]{morra-iwasawa} applied with $n = 1$ we have an isomorphism
\begin{equation}\label{eq:2}
\begin{aligned}
\F\bbra{Y_0,\dots,Y_{f-1}} /(Y_j^{s_j+1}, 0\leq j\leq f-1)&\buildrel\sim\over\longrightarrow \sigma\\
g(\un{Y})&\longmapsto g(\un{Y})\smatr{0} {1}{1}{0}v.
\end{aligned}
\end{equation}
(Restrict equation (9) in \cite{morra-iwasawa} to $\smatr{1}{0}{p\cO_K}{1}$ and conjugate by $\smatr{0}{1}{p}{0}$. Note that $\sigma$
is self-dual up to twist.) In particular, $\{\un{Y}^{\un{k}}\smatr{0}{1}{1}{0}v : \un{0}\leq \un{k}\leq \un{s}\}$ is a basis of $\sigma$ consisting of $H$-eigenvectors.

Let $m\defeq  \min\{s_0,\dots,s_{f-1}\}$. We claim that the vectors
\begin{equation}\label{eq:N1bis}
v_i\defeq  \sum_{\substack{\un{0}\leq\un{k}\leq\un{s}\\\|\un{k}\|=\|\un{s}\|-i}}\un{Y}^{\un{k}}\smatr{0}{1}{1}{0}v,\ \ \ \ 0\leq i\leq m
\end{equation}
form a basis of $\sigma^{N_1}$. If $i<m$ and $\|\un{k}\|=\|\un{s}\|-i$, then $k_j>0$ for all $j$. By using also~\eqref{eq:2} we see that $v_i=Y_jv_{i+1}$ for all $j$. Also, $Y_j v_0 = 0$ for all $j$. In particular, $Y_j-Y_{j'}$ annihilates $v_i$ for all $i$, so $v_i\in \sigma^{N_1}$ by Lemma \ref{lem:ker}. Moreover, $X v_{i+1}= v_i$ ($0\leq i<m$) and $Xv_0=0$. It remains to show that $v_m\notin X\sigma^{N_1}$. Choose $j_0$ such that $s_{j_0}=m$. Then $\prod_{j\neq j_0}Y_j^{s_j}\smatr{0}{1}{1}{0}v$ is the only term appearing in the sum (\ref{eq:N1bis}) for $i=m$ that is not divisible by $Y_{j_0}$. Hence $v_m\notin Y_{j_0}\sigma$, and thus $v_m\notin X\sigma^{N_1}$.

(ii) Let $v'\defeq  \un{Y}^{\un{s}}\smatr{0}{1}{1}{0}v$, which is a scalar multiple of $v$. By (\ref{eq:2}), $\left(\un{Y}^{\un{k}}\smatr{0}{1}{1}{0}v\right)_{\un{0}\leq \un{k}\leq\un{s}}$ forms a basis of $\sigma$ consisting of $H$-eigenvectors with eigencharacters $\chi_\sigma^s \alpha^{\un k} = \chi_\sigma \alpha^{\un k - \un r}$. The eigencharacters are pairwise distinct, except if $\un{s}=\un{p-1}$ where $\un{Y}^{\un{p-1}}\smatr{0}{1}{1}{0}v$ and $\smatr{0}{1}{1}{0}v$ have the same eigencharacter. Hence, as $\un{i} < \un{p-1}$, the unique $H$-eigenvector in the preimage $(\un{Y}^{\un{i}})^{-1}(v')$ is $\un{Y}^{\un{s}-\un{i}}\smatr{0}{1}{1}{0}v$. Note also that $Y_j\un{Y}^{\un{s}-\un{i}}\smatr{0}{1}{1}{0}v=0$ if $i_j=0$ by~\eqref{eq:2}.

(iii) Using the notation in (ii), we have $v_i=\sum_{\|\un{i}\|=i}\un Y^{-\un{i}}v'$ for $0\leq i\leq m$ and it is a $\smatr{\Fp^\times}{0}{0}{\Fp^\times}$-eigenvector with eigencharacter $\chi_\sigma\alpha^{-i}$. These characters for $0\leq i\leq m$ are pairwise distinct, except if $\un{s}=\un{p-1}$, in which case $v_0$ and $v_{p-1}$ have the same eigencharacter. As we assume $i<p-1$ the claim follows.
\end{proof}

\begin{lem}\label{lem:gen}
Suppose $V$ is a representation of $\GL_2(\Fq)$ generated by some vector $v\in V^{N_0}$ that is an eigenvector for the action of $H$.
If $\dim_\F V\leq q$, then the map
\begin{align*}
\F\bbra{Y_0,\dots,Y_{f-1}} &\longrightarrow V\\
f(\un{Y})&\mapsto f(\un{Y})\smatr{0}{1}{1}{0}v
\end{align*}
is surjective and its kernel is generated by monomials. In particular, if $\un{Y}^{\un i} \smatr{0}{1}{1}{0}v = \un{Y}^{\un j}\smatr{0}{1}{1}{0}v \ne 0$, then $\un i = \un j$.
\end{lem}
\begin{proof}
Let $\chi$ denote the eigencharacter of $H$ on $v$. Then we have a $\GL_2(\Fq)$-equiva\-riant surjection $S: \Ind_I^{\GL_2(\cO_K)}(\chi)\onto V$ sending $\phi$ to $v$, where $\phi$ is the unique function supported on $I$ which sends $1$ to $1$. Consider $i:\F[Y_0,\dots,Y_{f-1}]/(Y_0^p,\dots,Y_{f-1}^p)\rightarrow \Ind_I^{\GL_2(\cO_K)}(\chi)$ sending $f(\un{Y})$ to $f(\un{Y})\smatr{0}{1}{1}{0}\phi$. By Lemma \ref{lem:theta}, $f_j\in\mathrm{Im}(i)$ for all $j$ (even if $j=q-1$), so by \cite[Lemma 2.5]{BP}, $\Ind_{I}^{\GL_2(\cO_K)}(\chi)=\mathrm{Im}(i)\oplus \F\phi$ (as $\F$-vector spaces) and $i$ is injective.

Suppose first $\chi\not\cong \chi^{s}$. By \cite[Lemma 2.7(i)]{BP} and as $\dim V\leq q$ we have $f_{r}\pm\phi\in \Ker(S)$ for some $r=\sum_{j=0}^{f-1}p^js_j\in \{0,\dots,q-2\}$ and some sign $\pm$ (both depending on $\chi$), so $S\circ i$ is surjective. If $\Ker(S)$ is irreducible (as a $\GL_2(\Fq)$-representation), then by \cite[Lemma 2.7]{BP}, $\Ker(S)=\langle f_{\sum p^jd_j}, 0\leq d_j\leq s_j\text{ (not all equal)},\ f_r\pm\phi\rangle_\F$. Intersecting with $\mathrm{Im}(i)=\langle f_{\sum p^jd_j}, 0\leq d_j\leq p-1\rangle_\F$ we get
\[\Ker(S)\cap \mathrm{Im}(i)=\left\langle f_{\sum p^jd_j}, 0\leq d_j\leq s_j\text{ (not all equal)}\right\rangle_\F.\]
By Lemma \ref{lem:theta}(ii), it follows in particular that $\Ker(S\circ i)$ is generated by monomials. If $\Ker(S)$ is reducible, the argument is analogous using \cite[Lemma 2.7(ii)]{BP}. If $\chi=\chi^s$, it is again almost identical, using \cite[Lemma 2.6]{BP} instead.
\end{proof}

\begin{lem}\label{lem:tr0}
Suppose $f>1$. 
In $\F[N_0/N_0^p]$ we have
\[
\sum_{\lambda\in \Fq,{\rm Tr}_{\Fq/\Fp}(\lambda)=0}\smatr{1}{\lambda}{0}{1}\ =(-1)^{f-1}\bigg(\un{Y}^{\un{p-1}} \ + \!\!\sum_{\substack{\|\un{i}\|=(p-1)(f-1)\\0\leq i_j\leq p-1}}\un{Y}^{\un{i}}\bigg).
\]
\end{lem}
\begin{proof}
First we have (using $x^{p-1}=1$ if $x\in \Fp^\times$):
\begin{align*}
\sum_{\lambda\in \Fq,{\rm Tr}_{\Fq/\Fp}(\lambda)\neq0}\smatr{1}{\lambda}{0}{1}&=\sum_{\lambda\in \Fq}({\rm Tr}_{\Fq/\Fp}(\lambda))^{p-1}\smatr{1}{\lambda}{0}{1}\\
&=\sum_{\lambda\in \Fq}(\lambda+\lambda^p+\cdots+\lambda^{p^{f-1}})^{p-1}\smatr{1}{\lambda}{0}{1}\\
&=\sum_{\lambda\in \Fq}\sum_{\substack{\un{i}\in \Z_{\ge 0}^f\\\|\un{i}\|=p-1}}\frac{(p-1)!}{\prod_j i_j!}\lambda^{i_0+i_1 p+\cdots+i_{f-1}p^{f-1}}\smatr{1}{\lambda}{0}{1}\\
&=\sum_{\substack{\un{i}\in \Z_{\ge 0}^f\\\|\un{i}\|=p-1}}\frac{(p-1)!}{\prod_j i_j!} (-1)^{f-1}\bigg(\prod_j i_j!\bigg) \un{Y}^{\un{p-1}-\un{i}},
\end{align*}
where the last equality follows from Lemma \ref{lem:theta}(i), noting that $\sum_{j=0}^{f-1}i_jp^j<q-1$ since $f>1$. Letting $\un{i}'\defeq  \un{p-1}-\un{i}$ we get (as $(p-1)!=-1$ in $\Fp$):
\[
\sum_{\lambda\in \Fq,{\rm Tr}_{\Fq/\Fp}(\lambda)\neq0}\smatr{1}{\lambda}{0}{1}\ =(-1)^f\!\!\!\!\!\sum_{\substack{\un{i}'\in \Z_{\ge 0}^f\\\|\un{i}'\|=(p-1)(f-1)}}\un{Y}^{\un{i}'}.
\]
On the other hand, Lemma \ref{lem:theta}(i) gives
\[
\sum_{\lambda\in \Fq}\smatr{1}{\lambda}{0}{1}=(-1)^{f-1}\un{Y}^{\un{p-1}}.
\]
The result follows.
\end{proof}

\begin{prop}\label{prop:nsum}
Fix $j_0\in \{0,\dots,f-1\}$. In
\[\F\bbra{N_0/N_1^p}\cong \F\bbra{Y_0,\dots,Y_{f-1}} /\big((Y_i-Y_j)^p, i\neq j\big)\]
we have
\[
\sum_{n\in N_1/N_1^p}n=
(-1)^{f-1}\prod_{j\neq j_0}(Y_j-Y_{j_0})^{p-1}
\]
modulo terms of degree $\geq f(p-1)$.
\end{prop}
\begin{proof}
The statement being trivial if $f=1$, we can assume $f>1$. We prove the first isomorphism. As $Y_i-Y_j\in \Ker\big(\F\bbra{N_0} \ra\F\bbra{N_0/N_1} \big)$ by Lemma \ref{lem:ker}, we deduce that $(Y_i-Y_j)^p\in \Ker\big(\F\bbra{N_0} \ra\F\bbra{N_0/N_1^p} \big)$, and we thus have a surjection $\F\bbra{Y_0,\dots,Y_{f-1}} /\big((Y_i-Y_j)^p, i\neq j\big)\onto \F\bbra{N_0/N_1^p}$. Since both terms are free modules of rank $p(f-1)$ over a power series ring in one variable over $\F$, the surjection has to be an isomorphism.

Let $A\defeq  \F\bbra{N_1/N_1^p}$, $B\defeq  \F\bbra{N_0/N_1^p}$ and $\ovl{B}\defeq  \F\bbra{N_0/N_0^p}$, they are complete local commutative rings of respective maximal ideals denoted by $\fm_A$, $\fm_B$, $\fm_{\ovl{B}}$. Let $Z\defeq  \sum_{n\in N_1/N_1^p}n \in A$. 
Note that $\fm_A$ is the augmentation ideal of $A$, hence the $\fm_A$-torsion $A[\fm_A]$ in $A$ equals $\F Z$.
As $N_1/N_1^p \cong (\Z/p\Z)^{f-1}$, we have an isomorphism $A \cong \F[Z_1,\dots,Z_{f-1}]/(Z_1^p,\dots,Z_{f-1}^p)$, so $\fm_A^{(p-1)(f-1)+1}=0$ and hence $Z\in \fm_A^{(p-1)(f-1)}$.

Let $\imath: A\into B$ denote the inclusion and denote by $\gr^m(\imath)$ the induced map $\fm_A^m/\fm_A^{m+1}\rightarrow \fm_B^m/\fm_B^{m+1}$ for $m\geq 0$. We claim that $\gr^1(\imath)$ is injective with image generated by all $Y_j-Y_{j_0}$ ($j\neq j_0$) in $\fm_B/\fm_B^2$. If so, then $\gr^{(p-1)(f-1)}(\imath)$ has to send the $1$-dimensional $\F$-vector space $\fm_A^{(p-1)(f-1)}$ to a multiple of $\prod_{j\neq j_0}(Y_j-Y_{j_0})^{p-1}$ modulo $\fm_B^{(p-1)(f-1)+1}$. But $\smatr{\tld{\lambda}}{0}{0}{\tld{\mu}} Z=Z\smatr{\tld{\lambda}}{0}{0}{\tld{\mu}}$ for $\lambda,\mu\in \Fp^\times$, and considering the action of $H$, it follows from the sentence following Lemma \ref{lem:basic} that we must have
\[
\imath(Z)=c\prod_{j\neq j_0}(Y_j-Y_{j_0})^{p-1}+({\rm element\ of }\ \fm_B^{f(p-1)})
\]
for some $c\in \F$ (note that every element of $B$ can be written uniquely as $\sum_{\un{i}}c_{\un{i}}\un{Y}^{\un{i}}$ with $i_j<p$ for all $j\neq j_0$ and that $\fm_B$ is generated by the $\un{Y}^{\un{i}}$, $\un{i}\ne \un{0}$). By passing to $\ovl{B}$ and using Lemma \ref{lem:tr0}, we deduce that we must have $c=(-1)^{f-1}$.

It remains to prove the claim. As $\ovl{B}\cong B/(Y_0^p,\dots,Y_{f-1}^p)$, we have $\fm_B/\fm_B^2\buildrel\sim\over\rightarrow \fm_{\ovl{B}}/\fm_{\ovl{B}}^2$ and it is equivalent to prove the claim with $\ovl{\imath}:A\rightarrow \ovl{B}$. 
We first note that $\gr^1(\ovl{\imath})$ is injective with 1-dimensional cokernel, because for any finite abelian $p$-group $U$ the cotangent space of $\Spec \F[U]$ at its closed point is identified with $\F \otimes_{\Z} U$.
Consider the natural map $s:\ovl{B}\onto C\defeq  \F[N_0/N_1 N_0^p]\cong \F[Y]/(Y^p)$. As $\gr^1(s\circ \ovl{\imath})=0$ and $s(Y_i)=Y$ by Lemma \ref{lem:ker}, we deduce from {\it loc.cit.}\ that the image of $\gr^1(\ovl{\imath})$ is indeed spanned by all $Y_j-Y_{j_0}$ ($j \ne j_0$).
\end{proof}

\subsubsection{A computation for the operator \texorpdfstring{$F$}{F}}\label{operatorF}

We give a crucial computation for the operator $F$ on $\pi^{N_1}$ for $\pi$ as at the end of \S\ref{lowerstatement}. The main result of this section is Proposition \ref{prop:F-action}(ii).

We keep the notation of \S\ref{prelgl2}. For $\sigma= (t_0,\dots,t_{f-1})\otimes \eta\in W(\rhobar)$, recall we have $t_j \in \{r_j,r_j+1,p-2-r_j,p-3-r_j\}$ if $j > 0$ or $\rhobar$ is reducible and $t_0 \in \{r_0-1,r_0,p-1-r_0,p-2-r_0\}$ if $\rhobar$ is irreducible (see e.g.\ \cite[\S2]{breuil-IL}). We deduce from~\eqref{eq:6} that
\begin{equation}\label{eq:9}
t_j \in \{2f-1,\dots,p-1-2f\} \quad \text{for all $j$.}
\end{equation}
We identify $W(\rhobar)$ with the subsets of $\{0,1,\dots,f-1\}$ as in \cite[\S2]{breuil-IL} and let $J_\sigma\subseteq \{0,\dots,f-1\}$ be the subset associated to $\sigma$. We have $t_j\in \{p-2-r_j,p-3-r_j\}$ for $j\in J_\sigma$ if $j> 0$ or $\rhobar$ is reducible, $t_0\in \{p-2-r_0,p-1-r_0\}$ if $0\in J_\sigma$ and $\rhobar$ is irreducible.

Let $\sigma=(t_0,\dots,t_{f-1})\otimes \eta\in W(\rhobar)$. Denote $\delta(\sigma)\defeq  \delta_{\mathrm{red}}(\sigma)$ if $\rhobar$ is reducible and $\delta(\sigma)\defeq  \delta_{\mathrm{irr}}(\sigma)$ if $\rhobar$ is irreducible the Serre weights $\delta_{\mathrm{red}}(\sigma)$, $\delta_{\mathrm{irr}}(\sigma)$ defined in \cite[\S 5]{breuil-IL}. We write $\delta(\sigma)= (s_0,\dots,s_{f-1})\otimes \eta'$. Let $x_\sigma\in \sigma^{N_0}\setminus \{0\}$ and let $\chi_\sigma: H\rightarrow \F^\times$ denote the $H$-eigencharacter of $x_\sigma$. We also identify the irreducible constituents of $\Ind_I^{\GL_2(\cO_K)}(\chi_\sigma^s)$ with the subsets of $\{0,\dots,f-1\}$ as in \cite[\S2]{BP} (for instance $\emptyset$ corresponds to the socle $\sigma$ of $\Ind_I^{\GL_2(\cO_K)}(\chi_\sigma^s)$). For any $J\subseteq \{0,\dots,f-1\}$ let $Q(\chi_{\sigma}^s,J)$ denote the unique quotient of $\Ind_I^{\GL_2(\cO_K)}(\chi_\sigma^s)$ with irreducible $\GL_2(\cO_K)$-socle parametrized by $J$ (see \cite[Thm.2.4(iv)]{BP}). We know that the Serre weight $\delta(\sigma)$ occurs in $\Ind_I^{\GL_2(\cO_K)}(\chi_\sigma^s)$ (see the proof of \cite[Prop.5.1]{breuil-IL}) and we denote by $J^{\max}(\sigma)\subseteq \{0,\dots,f-1\}$ the associated subset. We thus have
\[\soc_{\GL_2(\cO_K)} Q(\chi_{\sigma}^s,J^{\max}(\sigma))\cong \delta(\sigma)\]
(by definition of $\delta(\sigma)$, it is the only constituent of $Q(\chi_{\sigma}^s,J^{\max}(\sigma))$ that is in $W(\rhobar)$). We also have from \cite[\S 2]{BP} (with $-1=f-1$):
\begin{equation}\label{ttos}
\begin{aligned}
s_j&=p-2-t_j+\mathbf{1}_{J^{\max}(\sigma)}(j-1)&\ \ \ \text{if $j\in J^{\max}(\sigma)$},\\
s_j&=t_j-\mathbf{1}_{J^{\max}(\sigma)}(j-1)&\ \ \ \text{if $j\notin J^{\max}(\sigma)$}.
\end{aligned}
\end{equation}
(Above, we write $\mathbf{1}_{J^{\max}(\sigma)}$ for the indicator function of $J^{\max}(\sigma)$.)
Moreover, using \cite[Lemma 2.7]{BP} it is a combinatorial exercise (left to the reader) to prove
\begin{equation}\label{diffsym}
J^{\max}(\sigma)=(J_\sigma \cup J_{\delta(\sigma)})\setminus (J_\sigma\cap J_{\delta(\sigma)}).
\end{equation}
We define
\[m\defeq  |J^{\max}(\sigma)| \in \{0,\dots,f\}.\]
We have $m = 0$ if and only if $\delta(\sigma) \cong \sigma$, and this occurs precisely if $\rhobar$ is reducible and $\sigma$ is an ``ordinary'' Serre weight of $\rhobar$, i.e.\ such that $J_\sigma = \emptyset$ or $J_\sigma = \{0,\dots,f-1\}$ (this follows, for example, from the proof of Lemma~\ref{lem:Jmax} below). 

We consider a $\GL_2(K)$-representation $\pi$ as at the end of \S\ref{lowerstatement}, and fix an embedding $\sigma\hookrightarrow \soc_{\GL_2(\cO_K)}(\pi)$ (recall there are $r$ copies of $\sigma$ inside $\soc_{\GL_2(\cO_K)}(\pi)$). From the assumption on $\pi$, we know that $\smatr{0}{1}{p}{0}x_\sigma$ generates $Q(\chi_{\sigma}^s,J^{\max}(\sigma))$ as a $\GL_2(\cO_K)$-subrepresentation of $\pi\vert_{{\GL_2(\cO_K)}}$, in particular $\delta(\sigma)$ can also be seen in $\soc_{\GL_2(\cO_K)}(\pi)$ (its embedding being determined by that of $\sigma$ up to a scalar).

\begin{prop}\label{prop:F-action}\

  \begin{enumerate}
  \item The vector
    \begin{equation}\label{xdelta}
      x_{\delta(\sigma)}\defeq  \prod_{j\in J^{\max}(\sigma)}Y_j^{s_j}\prod_{j\notin J^{\max}(\sigma)}Y_j^{p-1}\smatr{p}{0}{0}{1}x_\sigma
    \end{equation}
    spans $\delta(\sigma)^{N_0}$ as an $\F$-vector space.
  \item We have in $\pi^{N_1}$ that
    \begin{alignat*}{2}
      Y^{\sum_{j\in J^{\max}(\sigma)}s_j}F(Y^{1-m}x_{\sigma})&=(-1)^{f-1}Y^{1-m}x_{\delta(\sigma)} && \quad\text{if $m > 0$,}\\
      Y^{p-1}F(x_{\sigma})&=(-1)^{f-1}x_{\delta(\sigma)} && \quad\text{if $m = 0$.}
    \end{alignat*}
  \end{enumerate}
\end{prop}

\begin{proof}[Proof of Proposition~\ref{prop:F-action}(i)]
Suppose first $m > 0$. From \cite[Lemma 2.7(ii)]{BP} and Lemma \ref{lem:theta}(ii) we see that $\delta(\sigma)$
has basis $\un{Y}^{\un{i}}\smatr{p}{0}{0}{1}x_\sigma$, where $0\leq i_j\leq s_j$ if $j\in J^{\max}(\sigma)$ and $p-1-s_j\leq i_j\leq p-1$
if $j\notin J^{\max}(\sigma)$. Hence the only vectors in $\delta(\sigma)$ that are killed by all $Y_j$ are the multiples of 
$x_{\delta(\sigma)}$. The statement follows by an inspection of the $H$-action on this basis (which is formed by $H$-eigenvectors), see Remark \ref{switchchi}.

If $m = 0$, then $\delta(\sigma)$ is the socle of $\Ind_I^{\GL_2(\cO_K)}(\chi_\sigma^s)$. By \cite[Lemma 2.7(i)]{BP}, $f_0$ is the
unique $I$-invariant element of $\delta(\sigma) \subseteq \Ind_I^{\GL_2(\cO_K)}(\chi_\sigma^s)$. The statement follows from Lemma \ref{lem:theta}(ii).
\end{proof}

In order to prove Proposition~\ref{prop:F-action}(ii), we first need several lemmas.

\begin{lem}\label{lem:Jmax}
We have $|J^{\max}(\sigma)|=|J^{\max}(\delta(\sigma))|$.
\end{lem}
\begin{proof}
If $\rhobar$ is reducible, identifying $\{0,\dots,f-1\}$ with $\Z/f$ we have $J_{\delta(\sigma)}=J_\sigma-1$ as subsets of $\Z/f$ by \cite[\S 5]{breuil-IL}, and the statement follows in that case by (\ref{diffsym}). If $\rhobar$ is irreducible, let $J'_\sigma\defeq  J_\sigma\coprod (\ovl{J_\sigma}+f)\subseteq \{0,\dots,2f-1\}$ as in \cite[\S 5]{breuil-IL}, where $\ovl{J_\sigma}$ is the complement of $J_\sigma$ in $\{0,\dots,f-1\}$. It follows from (\ref{diffsym}) that $|J^{\max}(\sigma)|=\frac{1}{2}|(J'_\sigma \cup J'_{\delta(\sigma)})\setminus (J'_\sigma\cap J'_{\delta(\sigma)})|$. Identifying $\{0,\dots,2f-1\}$ with $\Z/2f$, we again have $J'_{\delta(\sigma)}=J_\sigma'-1$ as subsets of $\Z/(2f)$ by \cite[\S 5]{breuil-IL}, and the statement follows. 
\end{proof}

The three lemmas that follow only apply to $m>0$ and require the strong genericity assumption. 
In these three lemmas, we identify without comment $\{0,\dots,f-1\}$ with $\Z/f\Z$ (so $-1=f-1$, $f=0$, etc.).

\begin{lem}\label{lem:expl:pres}
Assume $m>0$ and let $\un{i}\in\Z^f_{\geq 0}$ with $\|\un{i}\|\leq m-1$. Then we have
\begin{multline}\label{eq:10}
\left\langle \GL_2(\cO_K)\smatr{p}{0}{0}{1}\un{Y}^{-\un{i}}x_\sigma\right\rangle\Big\slash
\!\sum_{\un{0}\leq\un{j}<\un{i}}\left\langle \GL_2(\cO_K)\smatr{p}{0}{0}{1}\un{Y}^{-\un{j}}x_\sigma\right\rangle\quad\\
\cong \quad Q\big(\chi_\sigma^s\alpha^{\un{i}},\{j \in J^{\max}(\sigma) : i_{j+1} = 0\}\big).
\end{multline}
\end{lem}
\begin{proof}
Note first that $t_j\in \{2i_j+1,\dots,p-2\}$ for all $j$ by \eqref{eq:9} and the assumption on $\un{i}$, so that the vectors $\un{Y}^{-\un{i}}x_\sigma$ and $\un{Y}^{-\un{j}}x_\sigma$ are well-defined elements of $\sigma$ by Lemma \ref{lem:eigenvct}(ii). We rewrite $\langle\GL_2(\cO_K)\smatr{p}{0}{0}{1}\un{Y}^{-\un{j}}x_\sigma\rangle = \langle\GL_2(\cO_K)\smatr{0}{1}{p}{0}\un{Y}^{-\un{j}}x_\sigma\rangle$ and, using notation from \cite[\S\S2.1,2.2]{BHHMS1}, $\sigma \cong F(\lambda)$ where $\lambda=(\lambda_0,\dots,\lambda_{f-1})$ with $\lambda_j=(\lambda_{j,1},\lambda_{j,2})\in \{0,\dots,p-1\}^2$. We have $\lambda_{j,1}-\lambda_{j,2}=t_j$ for all $j$.

Let $W'$ (resp.\ $W$) be the $I$-subrepresentation of $\pi$ generated by $\un{Y}^{-\un{i}}x_\sigma$ (resp.\ $\smatr{0}{1}{p}{0}\un{Y}^{-\un{i}}x_\sigma$). We deduce from Lemma~\ref{lem:eigenvct}(ii) that $W' = \langle N_0 \un{Y}^{-\un{i}}x_\sigma\rangle$ has $\F$-basis $\un{Y}^{-\un{j}}x_\sigma$ for all $\un 0 \le \un j \le \un i$, and $\soc_I (W') = \F x_\sigma$. We moreover have $W=\smatr{0}{1}{p}{0}W'$ since $I$ is normalized by $\smatr{0}{1}{p}{0}$. In particular we see that $W$ injects into the $I$-representation $\cJ_{\chi_\sigma}$ of \cite[Cor.6.1.4]{BHHMS1} and that $W$ has Jordan--H\"older factors $\chi_\sigma^s \alpha^{\un j}$ for $\un 0 \le \un j \le \un i$, each occurring with multiplicity $1$. Let $V \defeq  \Ind_I^{\GL_2(\cO_K)}(W)$. Then $V$ is the representation appearing in the first paragraph of the proof of \cite[Prop.6.2.2]{BHHMS1}, with $B_j$ taken to be $2i_j+1$ for all $j$ (and note the bounds on $\lambda_{j,1}-\lambda_{j,2}$ which let us invoke {\it loc.cit.}). Hence, by \cite[Prop.6.2.2]{BHHMS1} and its proof in the case $\varepsilon_j = -1$ and $B_j = 2i_j+1$ for all $j$, we get that $V$ is multiplicity-free, has Jordan--H\"older factors $\sigma_{\un a} \defeq  F(\mathfrak t_\lambda(-\sum a_j \ovl\eta_j))$ for $\un 0 \le \un a \le 2\un i+1$ with the notation of \cite[\S2.4]{BHHMS1}, and $\GL_2(\cO_K)$-socle $\sigma$. Moreover, the unique subrepresentation of $V$ with cosocle $\sigma_{\un a}$ has constituents $\sigma_{\un b}$ for $\un 0 \le \un b \le \un a$. On the other hand, $\Ind_I^{\GL_2(\cO_K)}(W)$ has a filtration with subquotients $\Ind_I^{\GL_2(\cO_K)} (\chi_\sigma^s \alpha^{\un j})$ for $\un 0 \le \un j \le \un i$, and by \cite[Lemma 6.2.1(i)]{BHHMS1} the constituents of $\Ind_I^{\GL_2(\cO_K)} (\chi_\sigma^s \alpha^{\un j})$ are the Serre weights $\sigma_{\un a}$ with $2\un j \le \un a \le 2\un j +1$. By the proof of \cite[Lemma 6.2.1(i)]{BHHMS1}, one easily checks that the constituent $\sigma_{\un a}$ of $\Ind_I^{\GL_2(\cO_K)} (\chi_\sigma^s \alpha^{\un j})$ corresponds to the subset $\{\ell : \text{$a_{\ell+1}$ is odd}\}\subseteq \{0,\dots,f-1\}$ in the parametrization of \cite[\S2]{BP} (note that twisting $\chi_\sigma^s$ by $\alpha^{\un j}$ corresponds to shifting by $-2 \sum j_\ell \ovl\eta_\ell$ in the extension graph).

By Frobenius reciprocity $\ovl V \defeq  \langle \GL_2(\cO_K)\smatr{0}{1}{p}{0}\un{Y}^{-\un{i}}x_\sigma\rangle$ is the image of a nonzero map $\Ind_I^{\GL_2(\cO_K)}(W) \to \pi$ and any Serre weight in its $\GL_2(\cO_K)$-socle has to be in $W(\rhobar)$. By \cite[Prop.2.4.2]{BHHMS1} if $\sigma_{\un a} \in W(\rhobar)$, then $\un 0 \le \un a \le \un 1$, so $\sigma_{\un a}$ is a constituent of $\Ind_I^{\GL_2(\cO_K)} (\chi_\sigma^s) \subseteq V$. Thus by the definition of $\delta(\sigma)$ and as $\pi^{K_1}/\soc_{\GL_2(\cO_K)} \pi$ does not contain any Serre weight of $W(\rhobar)$ it follows that $\ovl V$ is the unique quotient of $V$ with $\GL_2(\cO_K)$-socle $\delta(\sigma)$. By the previous paragraph and the definition of $J^{\max}(\sigma)$, we have $\delta(\sigma) \cong \sigma_{\un b}$, where $b_j = \mathbf{1}_{J^{\max}(\sigma)+1}(j)$ for all $j$, and $\ovl V$ has constituents $\sigma_{\un a}$ with $\mathbf{1}_{J^{\max}(\sigma)+1}(j) \le a_j \le 2 i_j +1$ for all $j$. By construction, the left-hand side of~\eqref{eq:10} is a quotient of $\Ind_I^{\GL_2(\cO_K)} (\chi_\sigma^s \alpha^{\un i})$. Moreover, by what is before, it must have constituents $\sigma_{\un a}$ with $\max(\mathbf{1}_{J^{\max}(\sigma)+1}(j),2 i_j) \le a_j \le 2 i_j + 1$ for all $j$. It follows that its $\GL_2(\cO_K)$-socle is irreducible and isomorphic to $\sigma_{\un c}$, where $c_j\defeq \max(\mathbf{1}_{J^{\max}(\sigma)+1}(j),2 i_j)$ for all $j$. Since $2i_{j+1}$ is even and $>1$ as soon as $i_{j+1}\ne 0$, we see that $c_{j+1}$ is odd if and only if $i_{j+1}=0$ and $j\in J^{\max}(\sigma)$. Hence the $\GL_2(\cO_K)$-socle of this quotient of $\Ind_I^{\GL_2(\cO_K)} (\chi_\sigma^s \alpha^{\un i})$ corresponds to the subset $\{j \in J^{\max}(\sigma) : i_{j+1} = 0\}$, as required.
\end{proof}

\begin{lem}\label{lem:iell=zero}
Assume $m>0$ and let $\un{i}\in\Z^f_{\geq 0}$, $\ell\in J^{\max}(\sigma)$ such that $\|\un{i}\|\leq m-1$ and $i_{\ell+1}=0$. Then
\[
Y_\ell^{p-t_{\ell}+2i_{\ell}}\smatr{p}{0}{0}{1}\un{Y}^{-\un{i}}x_{\sigma}=0.
\]
\end{lem}
\begin{proof}
Recall $p-t_{\ell}+2i_{\ell}\geq 0$ by \eqref{eq:9}, so that $Y_\ell^{p-t_{\ell}+2i_{\ell}}\smatr{p}{0}{0}{1}\un{Y}^{-\un{i}}x_{\sigma}$ is well-defined. Suppose on the contrary that $Y_\ell^{p-t_{\ell}+2i_{\ell}}\smatr{p}{0}{0}{1}\un{Y}^{-\un{i}}x_{\sigma}\neq 0$ for some $\ell\in J^{\max}(\sigma)$ such that $i_{\ell+1}=0$ and $\|\un{i}\|\leq m-1$. By Lemma \ref{lem:basic}(ii) and Lemma \ref{lem:eigenvct}(ii) this is an eigenvector for $\{\smatr{\tld{\lambda}}{0}{0}{\tld{\mu}} :  \lambda,\mu\in \Fq^\times\}$ with eigencharacter $\chi_\sigma\alpha^{-\un{i}}\alpha^{(p-t_\ell+2i_\ell)p^\ell}$.
By Lemma \ref{lem:expl:pres} it suffices to show that the $H$-eigencharacter $\chi_\sigma\alpha^{-\un{i}}\alpha^{(p-t_\ell+2i_\ell)p^\ell}$ does not occur in
\[
V_{\un{i}'}\defeq  Q(\chi_\sigma^s\alpha^{\un{i}'}, J_{\un{i}'})
\]
for any $\un{i}'$ such that $\un 0\leq \un i'\leq \un i$, where $J_{\un{i}'}\defeq  \{j \in J^{\max}(\sigma) : i'_{j+1}=0\}$.

Using the notation $\lambda=(\lambda_0(x_0),\dots,\lambda_{f-1}(x_{f-1}))$ and ${\mathcal P}(x_0,\dots,x_{f-1})$ of \cite[Thm.2.4]{BP}, the irreducible constituents of $V_{\un{i}'}$ are given by the Serre weights $(\lambda_0(t_0-2i_0'),\dots, \lambda_{f-1}(t_{f-1}-2i_{f-1}'))$ (up to twist) for those $\lambda\in \mathcal{P}(x_0,\dots,x_{f-1})$ such that $J(\lambda)\supseteq J_{\un{i}'}$. Recall that $\lambda_j(x)=p-2-x+\mathbf{1}_{J(\lambda)}(j-1)$ if $j\in J(\lambda)$ and $\lambda_j(x)=x-\mathbf{1}_{J(\lambda)}(j-1)$ if $j\notin J(\lambda)$.
By \cite[Lemma 2.5(i)]{BP} and \cite[Lemma 2.7]{BP}, the $H$-eigencharacters that occur in $V_{\un{i}'}$ are $\chi_\sigma\alpha^{-\un{i}'}\alpha^{\un{k}}$, where $\un{k}$ is such that there exists $\lambda\in \mathcal{P}(x_0,\dots,x_{f-1})$ with $J(\lambda)\supseteq J_{\un{i}'}$ and
\begin{equation}\label{eq:iell=zero:1}
\begin{aligned}
0\leq k_j\leq p-2-(t_j-2i'_j)+\mathbf{1}_{J(\lambda)}(j-1)&\ \ \ \text{if $j\in J(\lambda)$},\\
p-1-(t_j-2i'_j-\mathbf{1}_{J(\lambda)}(j-1))\leq k_j\leq p-1&\ \ \ \text{if $j\notin J(\lambda)$}.
\end{aligned}
\end{equation}
(Note that $J_{\un{i}'}\neq\emptyset$ as $\ell\in J_{\un{i}'}$, noting that $\ell\in J^{\max}(\sigma)$ and $0 \le i_{\ell+1}' \le i_{\ell+1}=0$.)

Assume $\chi_\sigma\alpha^{-\un{i}}\alpha^{(p-t_\ell+2i_\ell)p^{\ell}}=\chi_\sigma\alpha^{-\un{i}'}\alpha^{\un{k}}$ for some $\lambda$ and $\un{k}$ as above. Then
\[-\sum_{j=0}^{f-1} i_jp^j +(p-t_\ell+2i_\ell)p^{\ell}\equiv -\sum_{j=0}^{f-1} i'_jp^j +\sum_{j=0}^{f-1} k_jp^j \pmod{q-1}\]
or equivalently
\begin{equation}\label{eq:iell=zero:2}
(p-t_\ell+2i_\ell)p^{\ell}-\sum_{j=0}^{f-1} (i_j-i'_j)p^j\equiv \sum_{j=0}^{f-1} k_jp^j\pmod{q-1}.
\end{equation}
Note that, since $\ell\in J_{\un{i}'}$, we have in particular $\ell\in J(\lambda)$.

If $i_{j}'=i_{j}$ for all $j\neq \ell$ (for example if $\un{i}'=\un{i}$ or if $f=1$), then \eqref{eq:iell=zero:2} gives $(p-t_\ell+i_\ell+i_\ell')p^\ell\equiv \sum_j k_jp^j$, so $k_\ell =p-t_\ell+i_\ell+i_\ell'$ as (using \eqref{eq:9} for $t_\ell$ and $0 \le i'_\ell\leq i_\ell\leq m-1\leq f-1$):
\begin{equation}\label{bornep-tj}
p-t_\ell+i_\ell+i_\ell'\in \{2f+1,\dots,p-1\}.
\end{equation}
This contradicts \eqref{eq:iell=zero:1} as $\ell\in J(\lambda)$ and $i'_\ell \le i_\ell$. Therefore $f>1$ and $i'_j<i_j$ for some $j\neq \ell$. For $m\in \Z_{\geq 0}$, let $[m]$ the unique element of $\{0,\dots,f-1\}$ which is congruent to $m$ modulo $f$. In particular $p^m\equiv p^{[m]}\pmod{q-1}$. Let $h\in \{\ell+1,\dots,\ell+f-1\}$ be minimal such that $i'_{[h]}<i_{[h]}$. Then modulo $q-1$:
\[\sum_{j=0}^{f-1} (i_j-i'_j)p^j\equiv \sum_{j=\ell+1}^{\ell+f} (i_{[j]}-i'_{[j]})p^{[j]}=\sum_{j=h}^{\ell+f} (i_{[j]}-i'_{[j]})p^{[j]}\]
and we deduce the following congruences modulo $q-1$:
\begin{align}
(p-&t_\ell+2i_\ell)p^{\ell}- \sum_{j=0}^{f-1} (i_j-i'_j)p^j \notag\\
&\equiv (p-1-t_\ell+2i_\ell)p^{\ell}+p^\ell -\sum_{j=h}^{\ell+f} (i_{[j]}-i'_{[j]})p^{[j]} \notag\\
&\equiv (p-1-t_\ell+2i_\ell)p^{\ell}+\!\!\sum_{j=h+1}^{\ell+f-1}(p-1)p^{j}+p^{h+1}- \sum_{j=h}^{\ell+f} (i_{[j]}-i'_{[j]})p^{[j]} \notag\\
&\equiv (p-1-t_\ell+2i_\ell)p^{\ell}+\!\!\sum_{j=h+1}^{\ell+f-1}(p-1)p^{[j]}+p^{[h]+1}- \sum_{j=h}^{\ell+f} (i_{[j]}-i'_{[j]})p^{[j]} \notag\\
\label{eq:3}&\equiv
(p\!-\!1\!-\!t_\ell\!+\!i_\ell\!+\!i'_\ell)p^{\ell}+\!\!\sum_{j=h+1}^{\ell+f-1}\!\!(p\!-\!1\!-\!(i_{[j]}\!-\!i'_{[j]}))p^{[j]}+(p\!-\!(i_{[h]}\!-\!i'_{[h]}))p^{[h]}.
\end{align}
Note that all powers of $p$ in \eqref{eq:3} are distinct in $\{0,\dots,f-1\}$ and all coefficients are in $\{0,\dots,p-1\}$. Moreover these coefficients cannot all equal $0$ as $p-(i_{[h]}-i'_{[h]})\ne 0$, nor $p-1$ by (\ref{bornep-tj}). Hence by \eqref{eq:iell=zero:2} we get $k_\ell=p-1-t_\ell+i_\ell+i_\ell'$. As $\ell\in J(\lambda)$ and $i_\ell'\leq i_\ell$, we get from \eqref{eq:iell=zero:1} that $i_\ell=i'_\ell$ and $\ell-1\in J(\lambda)$. By \eqref{eq:iell=zero:1} for $j=\ell-1$ and by (\ref{eq:3}), (\ref{eq:iell=zero:2}) we get
\[
p-1-(i_{\ell-1}-i_{\ell-1}')\leq k_{\ell-1}\leq p-1-t_{\ell-1}+2i'_{\ell-1}
\]
(note that by (\ref{eq:3}) the left-hand side is an equality as soon as $\ell-1\not\equiv h$ mod $f$ which can only occur if $f>2$). This implies $t_{\ell-1}\leq i_{\ell-1}+i_{\ell-1}'\leq 2(m-1)\leq 2f-2$, which contradicts genericity \eqref{eq:9}. This finishes the proof.
\end{proof}

\begin{lem}\label{lem:key:structure}
Assume $m > 0$ and let $\un{k}\in\Z^f_{\geq 0}$.
\begin{enumerate}
\item If $\un{Y}^{\un{k}}\smatr{p}{0}{0}{1}Y^{1-m}x_\sigma\neq 0$, then
  \[\|\un k\|\leq (f-1)(p-1)+(m-1)+\sum _{j\in J^{\max}(\sigma)}{s_j}.\]
  If moreover equality holds, then $\un{Y}^{\un{k}}\smatr{p}{0}{0}{1}Y^{1-m}x_\sigma=x_{\delta(\sigma)}$
  {\upshape(}see {\upshape(\ref{xdelta}))} and
  \begin{alignat*}{3}
    k_j&\equiv s_j &\pmod{p}&& \quad\text{if $j\in J^{\max}(\sigma)$,}\\
    k_j&\equiv -1&\pmod{p}&& \quad\text{if $j\notin J^{\max}(\sigma)$.}
  \end{alignat*}
\item If $\|\un k\|=(f-1)(p-1)+\sum_{J^{\max}(\sigma)}s_j$ then
  $\un{Y}^{\un{k}}\smatr{p}{0}{0}{1}Y^{1-m}x_\sigma\in \delta(\sigma)$, more precisely:
  \begin{equation*}
    \un{Y}^{\un{k}}\smatr{p}{0}{0}{1}Y^{1-m}x_\sigma \in \left\langle \un{Y}^{-\un{\ell}}x_{\delta(\sigma)}, \|\un\ell\|=m-1\right\rangle_\F.
  \end{equation*}
\end{enumerate}
\end{lem}
\begin{proof}
We prove the following statements inductively on $\|\un{i}\|\leq m-1$ for $\un{i}\in \Z_{\ge 0}^f$:
\begin{enumerate}[(a)]
\item\label{it:reduction:1}
If $\un{Y}^{\un{k}}\smatr{p}{0}{0}{1}\un{Y}^{-\un{i}}x_\sigma\neq 0$ then
\[\|\un k\|\leq (f-1)(p-1)+(m-1)+\sum_{j\in J^{\max}(\sigma)} s_j-(m-1-\|\un{i}\|)p.\]
If moreover equality holds, then $\un{Y}^{\un{k}}\smatr{p}{0}{0}{1}\un{Y}^{-\un{i}}x_\sigma=x_{\delta(\sigma)}$ and
\begin{alignat*}{3}
k_j&= i_{j+1}p + s_j & \quad\text{if $j\in J^{\max}(\sigma)$,}\\
k_j&= i_{j+1}p + (p-1)& \quad\text{if $j\notin J^{\max}(\sigma)$.}
\end{alignat*}
\item\label{it:reduction:2}
If $\|\un k\|= (f-1)(p-1)+\sum_{j\in J^{\max}(\sigma)} s_j-(m-1-\|\un{i}\|)p$ then
\[
\un{Y}^{\un{k}}\smatr{p}{0}{0}{1}\un{Y}^{-\un{i}}x_\sigma=\un{Y}^{-\un{\ell}}x_{\delta(\sigma)}
\]
for some $\|\un\ell\|=m-1$, or it is zero.
\end{enumerate}
By Lemma \ref{lem:eigenvct}(iii) we have
\[Y^{1-m}x_\sigma=\sum_{\substack{\un{i}\in \Z_{\ge 0}^f\\\|\un{i}\|=m-1}}\un{Y}^{-\un{i}}x_\sigma\]
and we see that \ref{it:reduction:1} and \ref{it:reduction:2} for $\|\un{i}\|=m-1$ imply (i) and (ii) (note that in \ref{it:reduction:1} if $\un{Y}^{\un{k}}\smatr{p}{0}{0}{1}\un{Y}^{-\un{i}}x_\sigma\neq 0$ and equality holds, then $\un{i}$ is uniquely determined by $\un{k}$ and $J^{\max}(\sigma)$).

We first prove by induction on $\|\un{i}\|\leq m-1$ for $\un{i}\in \Z_{\ge 0}^f$ that if $\|\un k\|\geq (f-1)(p-1)+\sum_{J^{\max}(\sigma)} s_j-(m-1-\|\un{i}\|)p$ and $\un{Y}^{\un{k}}\smatr{p}{0}{0}{1}\un{Y}^{-\un{i}}x_\sigma\neq 0$, then $\un{Y}^{\un{k}}\smatr{p}{0}{0}{1}\un{Y}^{-\un{i}}x_\sigma=\un{Y}^{\un{k}'}\smatr{p}{0}{0}{1}x_\sigma$ for $\un k'\in \Z^f_{\geq 0}$ such that $k'_j=k_j-i_{j+1}p$ for all $j$. A examination of \ref{it:reduction:1} and \ref{it:reduction:2} shows it will then be enough to prove them for $\un{i}=\un 0$ (replacing $\un k$ by $\un k'$).

There is nothing to prove for $\un{i}= \un{0}$, so we can assume $\un{i}\neq \un{0}$. If $k_{j_0}\geq p$ for some $j_0$, then using Lemma \ref{lem:basic}(ii):
\[\un{Y}^{\un{k}}\smatr{p}{0}{0}{1}\un{Y}^{-\un{i}}x_\sigma=\un{Y}^{\un{k}-p\un\eps_{j_0}}Y_{j_0}^p\smatr{p}{0}{0}{1}\un{Y}^{-\un{i}}x_\sigma=\un{Y}^{\un{k}-p\un\eps_{j_0}}\smatr{p}{0}{0}{1}\un{Y}^{-(\un{i}-\un\eps_{j_0+1})}x_\sigma,\]
where $\un\eps_j\defeq (0,\dots,0,1,0,\dots,0)$ with $1$ in position $j$ and $0$ elsewhere (note that $Y_{j_0+1}\un{Y}^{-\un{i}}x_\sigma=\un{Y}^{-(\un{i}-\un\eps_{j_0+1})}x_\sigma$ is nonzero by assumption, and hence $\un{i}-\un\eps_{j_0+1}\in \Z_{\ge 0}^f$ by the last statement in Lemma \ref{lem:eigenvct}(ii)). As $\|\un{i}-\un\eps_{j_0+1}\|=\|\un{i}\|-1$ and $\|\un{k}-p\un\eps_{j_0}\|=\|\un{k}\|-p\geq (f-1)(p-1)+\sum_{J^{\max}(\sigma)} s_j-(m-1-\|\un{i}-\un\eps_{j_0+1}\|)p$, we can apply the induction hypothesis and a small computation shows that $\un k'$ is the right one, so we are done in that case.

We assume $k_{j}<p$ for all $j$ and derive below a contradiction (so this case can't happen). Define
\[J\defeq  \{j\in J^{\max}(\sigma) : i_{j+1}=0\},\]
then by Lemma \ref{lem:iell=zero} (applied to $\ell=j$ and using ${Y_j}^{k_j}\smatr{p}{0}{0}{1}\un{Y}^{-\un{i}}x_\sigma\neq 0$):
\begin{alignat*}{3}
k_j&\leq p-1-t_j+2i_j & \quad\text{if $j\in J$,}\\
k_j&\leq p-1& \quad\text{if $j\notin J$},
\end{alignat*}
which implies $\|\un k\|\leq (f-|J|)(p-1)+\sum_{j\in J}(p-1-t_j+2i_j)$. From (\ref{ttos}) we deduce
\[\|\un k\|\leq (f-|J|)(p-1)+\sum_{j\in J}(s_j+2i_j)+|J\setminus (J^{\max}(\sigma)+1)|.\]
So to get a contradiction it is enough to show that
\begin{multline*}
(f-|J|)(p-1)+\sum_{j\in J}(s_j+2i_j)+|J\setminus (J^{\max}(\sigma)+1)|<(f-1)(p-1)\\
+\sum_{j\in J^{\max}(\sigma)}s_j-(m-1-\|\un{i}\|)p,
\end{multline*}
or equivalently
\begin{eqnarray}
\notag pm+|J\!\setminus \!(J^{\max}(\sigma)+1)| &\!\!\!\le \!\!\!&(p-1)|J| + p\sum_{j \not\in J}i_j+(p-2)\sum_{j\in J}i_j+\!\!\!\sum_{j\in J^{\max}(\sigma)\setminus J}\!\!\!s_j\\
&\!\!\!=\!\!\!&(p-2)\|\un{i}\|+(p-1)|J|+\Big(2\sum_{j \not\in J}i_j + \!\!\!\sum_{j\in J^{\max}(\sigma)\setminus J}\!\!\!s_j\Big).\label{eq:7}
\end{eqnarray}
Case $1$: assume $|J^{\max}(\sigma)\setminus J|>0$.\\
If $j \in J^{\max}(\sigma)\setminus J$, then $i_{j+1}>0$, so $|J^{\max}(\sigma)\setminus J|\leq \|\un{i}\|$. As $|J^{\max}(\sigma)\setminus J|=m-|J|$, this means $m\leq \|\un{i}\|+|J|$, hence \eqref{eq:7} is implied by
\begin{equation}\label{eq:7bis}
2m+|J\setminus (J^{\max}(\sigma)+1)|\leq |J|+\Big(2\sum_{j \not\in J}i_j + \sum_{j\in J^{\max}(\sigma)\setminus J}s_j\Big).
\end{equation}
Using $|J\setminus (J^{\max}(\sigma)+1)|\leq |J|$, (\ref{eq:7bis}) is implied by
\begin{equation}\label{eq:8}
2m\leq \sum_{j\in J^{\max}(\sigma)\setminus J}s_j.
\end{equation}
Genericity~\eqref{eq:9} with (\ref{ttos}) give $s_{j} \ge 2f-1 \ge 2m-1$ for $j\in J^{\max}(\sigma)$, hence \eqref{eq:8} holds if either $s_{j} \ge 2m$ for at least one $j\in J^{\max}(\sigma)\setminus J$ or if $|J^{\max}(\sigma)\setminus J|\geq 2$ (using $2m-2\geq 0$ for the latter). Therefore, the only way inequality~\eqref{eq:7bis} may fail is when $J^{\max}(\sigma)\setminus J = \{j_0\}$ (for some $j_0$) and moreover $J \setminus (J^{\max}(\sigma)+1) = J$ and $i_j = 0$ for all $j \not\in J$. But then $i_{j_0+1} > 0$ so we have $j_0+1 \in J \cap (J^{\max}(\sigma)+1)$, which contradicts $J \cap (J^{\max}(\sigma)+1)=\emptyset$. Hence inequality~\eqref{eq:7bis} holds.\\
\noindent
Case $2$: assume $J^{\max}(\sigma)=J$.\\
Then using
\[|J\setminus (J^{\max}(\sigma)+1)|\leq |\{0,\dots,f-1\}\setminus (J^{\max}(\sigma)+1)|=|\{0,\dots,f-1\}\setminus J^{\max}(\sigma)|=f-m\]
and $|J|=m$, we see that \eqref{eq:7} is implied by $(p-1)m+f\leq (p-2)\|\un{i}\|+(p-1)m$ which is true as $\|\un{i}\|>0$ and $f\leq p-2$ by \eqref{eq:6}.

To prove \ref{it:reduction:1} and \ref{it:reduction:2}, it therefore suffices to consider the case $\un{i}=\un 0$, which we prove now.

Recall $\left\langle \GL_2(\cO_K)\smatr{0}{1}{p}{0}x_{\sigma}\right\rangle\cong Q(\chi_\sigma^s,J^{\max}(\sigma))$. By \cite[Thm.2.4(iv)]{BP} the constituents of this $\GL_2(\cO_K)$-representation are the Serre weights $(\lambda_0(t_{0}),\dots,\lambda_{f-1}(t_{f-1}))$ up to twist, where $\lambda\in \mathcal{P}(x_0,\dots,x_{f-1})$, $J(\lambda)\supseteq J^{\max}(\sigma)$ and $\lambda_j(t_j)=p-2-t_j+\mathbf{1}_{J(\lambda)}(j-1)$ if $j\in J(\lambda)$ (we use the notation of \cite[\S2]{BP} as in the proof of Lemma \ref{lem:iell=zero}). By \cite[Lemma 2.7,~Lemma 2.6]{BP} and Lemma \ref{lem:theta}(ii), $Q(\chi_\sigma^s,J^{\max}(\sigma))$ has $\F$-basis $\un{Y}^{\un{k}}\smatr{p}{0}{0}{1}x_\sigma$, where
\begin{equation}\label{eq:key:1}
\begin{aligned}
0&\leq k_j \leq \lambda_j(t_j)&\quad\text{if\ }&j\in J(\lambda),\\
p-1-\lambda_j(t_j)&\leq k_j\leq p-1&\quad\text{if\ }& j\not\in J(\lambda)
\end{aligned}
\end{equation}
for some $\lambda\in\mathcal{P}(x_0,\dots,x_{f-1})$ with $J(\lambda)\supseteq J^{\max}(\sigma)$. We see that (\ref{eq:key:1}) implies
\begin{eqnarray}\label{eq:bound}
\|\un{k}\|&\leq &(f-|J(\lambda)|)(p-1)+\sum_{j\in J(\lambda)}(p-2-t_j+\mathbf{1}_{J(\lambda)}(j-1))
\end{eqnarray}
with equality if and only if $k_j = \lambda_j(t_j)$ if $j\in J(\lambda)$ and $k_j=p-1$ otherwise. Moreover, $\un{Y}^{\un{k}}\smatr{p}{0}{0}{1}x_\sigma\in\delta(\sigma)\!\setminus\{0\}$ if and only if \eqref{eq:key:1} holds with $J(\lambda)=J^{\max}(\sigma)$. Hence if $\un{Y}^{\un{k}}\smatr{p}{0}{0}{1}x_\sigma\not\in\delta(\sigma)$ we deduce that \eqref{eq:key:1} holds for some $\lambda\in \mathcal{P}(x_0,\dots,x_{f-1})$ with $J(\lambda)\supsetneq J^{\max}(\sigma)$. 

We claim that the right-hand side of (\ref{eq:bound}) is smaller or equal than $(f-1)(p-1)+m-1+\sum_{J^{\max}(\sigma)}s_j-p(m-1)$ if $J(\lambda)= J^{\max}(\sigma)$ and strictly smaller than $(f-1)(p-1)+\sum_{J^{\max}(\sigma)}s_j-p(m-1)$ if $J(\lambda)\supsetneq J^{\max}(\sigma)$. Recalling that $s_j=p-2-t_j+\mathbf{1}_{J^{\max}(\sigma)}(j-1)$ for $j\in J^{\max}(\sigma)$, the first case follows from $(f-|J^{\max}(\sigma)|)(p-1)=(f-1)(p-1)+m-1-p(m-1)$. For the second case, as $(f-1)(p-1)-p(m-1)=(f-|J^{\max}(\sigma)|)(p-1)-(m-1)$, it is enough to prove
\begin{multline*}
(f-|J(\lambda)|)(p-1)+\underset{j\in J(\lambda)}{\sum}(p-2-t_j)+|J(\lambda)\cap (J(\lambda)+1)|\\
\quad < (f\!-\!|J^{\max}(\sigma)|)(p-1)+\!\!\!\!\underset{j\in J^{\max}(\sigma)}{\sum}\!\!(p-2-t_j)+|J^{\max}(\sigma)\cap (J^{\max}(\sigma)+1)|\\
-(m-1),
\end{multline*}
or equivalently (by an easy calculation):
\begin{equation*}
(m-1)+|J(\lambda)\cap (J(\lambda)+1)|-|J^{\max}(\sigma)\cap (J^{\max}(\sigma)+1)| < \underset{j\in J(\lambda)\setminus J^{\max}(\sigma)}{\sum}(t_j+1).
\end{equation*}
This is true, as $m-1\leq f-1$ (so the left-hand side is at most $(f-1)+f$), $J(\lambda)\setminus J^{\max}(\sigma)\ne\emptyset$ and $t_j+1\geq 2f$ for any $j$ by genericity~\eqref{eq:9}.

Therefore $\|\un k\|\leq (f-1)(p-1)+(m-1)+\sum_{J^{\max}(\sigma)} s_j-p(m-1)$ if $\un{Y}^{\un{k}}\smatr{p}{0}{0}{1}x_\sigma \ne 0$ and $\un{Y}^{\un{k}}\smatr{p}{0}{0}{1}x_\sigma\in \delta(\sigma)$ if $\|\un{k}\|\geq (f-1)(p-1)+\underset{J^{\max}(\sigma)}{\sum}s_j-p(m-1)$.

We prove the remaining statements in \ref{it:reduction:1} and \ref{it:reduction:2} (for $\un{i}=\un 0$). If $\|\un{k}\|\geq (f-1)(p-1)+\underset{J^{\max}(\sigma)}{\sum}s_j-p(m-1)$ and $\un{Y}^{\un{k}}\smatr{p}{0}{0}{1}x_\sigma\neq 0$, we know by above that $J(\lambda) = J^{\max}(\sigma)$. By \eqref{eq:key:1} we then have $k_j\leq s_j$ if $j\in J^{\max}(\sigma)$ and $k_j\leq p-1$ if $j\notin J^{\max}(\sigma)$. By the definition of $x_{\delta(\sigma)}$ in (\ref{xdelta}) and by Lemma \ref{lem:eigenvct}(ii) (and Remark \ref{switchchi}) we deduce $\un{Y}^{\un{k}}\smatr{p}{0}{0}{1}x_\sigma=\un{Y}^{-\un{\ell}}x_{\delta(\sigma)}$, where $\ell_j=s_j-k_j$ if $j\in J^{\max}(\sigma)$ and $\ell_j= p-1-k_j$ if $j\notin J^{\max}(\sigma)$. This implies $\|\un\ell\|=(f-m)(p-1)+\underset{J^{\max}(\sigma)}{\sum}s_j-\|\un{k}\|$, and in particular $\|\un\ell\|=0$ if $\|\un{k}\|= (f-1)(p-1)+(m-1)+\sum_{J^{\max}(\sigma)} s_j-(m-1)p$ and $\|\un\ell\|=m-1$ if $\|\un{k}\|= (f-1)(p-1)+\underset{J^{\max}(\sigma)}{\sum}s_j-p(m-1)$. This finishes the proof of \ref{it:reduction:1} and \ref{it:reduction:2}.
\end{proof}

Now we can finally complete the proof of Proposition~\ref{prop:F-action}.

\begin{proof}[Proof of Proposition~\ref{prop:F-action}(ii)]
Suppose first that $m > 0$ and fix $j_0\in J^{\max}(\sigma)$. By Lemma \ref{lem:ker}, Proposition \ref{prop:nsum} and the definition of $F$ (see \ref{def:F} in \S\ref{covariant}), we have
\begin{multline*}
Y^{\sum_{j\in J^{\max}(\sigma)}s_j}F(Y^{1-m}x_\sigma)\\
=\left[
(-1)^{f-1}\!\!\!\!\!\!\prod_{j\in J^{\max}(\sigma)}\!\!\!\!\!Y_j^{s_j}\prod_{j\neq j_0}(Y_j-Y_{j_0})^{p-1}+f(\un Y)
\right]\smatr{p}{0}{0}{1}Y^{1-m}x_\sigma
\end{multline*}
for \ some \ $f(\un Y) \in \F\bbra{Y_0,\dots,Y_{f-1}} $ \ of \ $\m_{N_0}$-adic \ valuation \ (i.e.~total \ degree) $\geq \sum_{J^{\max}(\sigma)} s_j+(p-1)f$. As $p>f\geq m$ we have $(p-1)f>(p-1)(f-1)+m-1$ and by Lemma \ref{lem:key:structure}(i) we get $f(\un Y)\smatr{p}{0}{0}{1}Y^{1-m}x_\sigma=0$, hence
\[
Y^{\sum_{J^{\max}(\sigma)}s_j}F(Y^{1-m}x_\sigma)=
(-1)^{f-1}\prod_{j\in J^{\max}(\sigma)}Y_j^{s_j}\prod_{j\neq j_0}(Y_j-Y_{j_0})^{p-1}\smatr{p}{0}{0}{1}Y^{1-m}x_\sigma.
\]
Moreover, the right-hand side is contained in $\langle \un{Y}^{-\un{\ell}}x_{\delta(\sigma)}, \|\un\ell\|=m-1\rangle_\F\subset \delta(\sigma)$ by Lemma \ref{lem:key:structure}(ii). As it is also $N_1$-invariant, it is contained in $\F Y^{1-m}x_{\delta(\sigma)}$ by Lemma \ref{lem:eigenvct}(iii). It is therefore enough to show that $Y^{m-1 + \sum_{J^{\max}(\sigma)}s_j}F(Y^{1-m}x_\sigma)=(-1)^{f-1}x_{\delta(\sigma)}$, or again by Lemma \ref{lem:ker}, Proposition \ref{prop:nsum} and Lemma \ref{lem:key:structure}(i) that
\[
Y_{j_0}^{m-1}\prod_{j\in J^{\max}(\sigma)}Y_j^{s_j}\prod_{j\neq j_0}(Y_j-Y_{j_0})^{p-1}\smatr{p}{0}{0}{1}Y^{1-m}x_\sigma=x_{\delta(\sigma)}.
\]
As $\binom{p-1}{i} = (-1)^i$ for $0 \le i \le p-1$, the left-hand side equals
\begin{equation}
Y_{j_0}^{m-1}\prod_{J^{\max}(\sigma)}Y_j^{s_j}
\underset{\substack{\|\un{k}'\|=(p-1)(f-1)\\k'_j\leq p-1\ \text{if $j\neq j_0$}}}{\sum}\un{Y}^{\un{k}'}\smatr{p}{0}{0}{1}Y^{1-m}x_\sigma.\label{eq:4}
\end{equation}
By Lemma \ref{lem:key:structure}(i), as $k'_j+s_j$ can never be congruent to $s_j$ modulo $p$ when $k'_j\in \{1,\dots,p-1\}$, only the terms with $k'_{j}=0$ for $j\in J^{\max}(\sigma)\setminus\{j_0\}$ and $k'_{j}=p-1$ for $j\notin J^{\max}(\sigma)$ survive. As $\|\un{k}'\|=(p-1)(f-1)$, we must have $k'_{j_0}=(p-1)(m-1)$, and by Lemma \ref{lem:key:structure}(i) again it follows that \eqref{eq:4} equals $x_{\delta(\sigma)}$, as required.

Finally suppose $m = 0$. As $Y_j^p \smatr{p}{0}{0}{1} x_\sigma = 0$ for all $j$, we get again by Lemma \ref{lem:ker}, Proposition \ref{prop:nsum} and (\ref{xdelta}):
\begin{align*}
Y^{p-1} F(x_\sigma) &= (-1)^{f-1} Y_0^{p-1} \prod_{j\neq 0}(Y_j-Y_{0})^{p-1} \smatr{p}{0}{0}{1}x_\sigma \\
&= (-1)^{f-1} \prod_{j=0}^{f-1} Y_j^{p-1} \smatr{p}{0}{0}{1}x_\sigma = (-1)^{f-1} x_{\delta(\sigma)}.\qedhere
\end{align*}
\end{proof}

\subsubsection{Lower bound for \texorpdfstring{$V_{\GL_2}(\pi)$}{V\_\{GL\_2\}(pi)}: proof}\label{lowerproof}

We prove Theorem \ref{thm:main-tensor-ind}.

We keep the notation of \S\S\ref{lowerstatement}, \ref{prelgl2}, \ref{operatorF}. Fix $\sigma \in W(\rhobar)$ and define $\sigma_i \in W(\rhobar)$ inductively by $\sigma_1 \defeq  \sigma$ and $\sigma_i \defeq  \delta(\sigma_{i-1})$ for $i > 1$ ($\sigma_i$ here shouldn't be confused with the embedding $\sigma_i=\sigma_0\circ\varphi^i$). Let $n \ge 1$ be the smallest integer such that $\sigma_{n+1} \cong \sigma_1$ and write $\sigma_i=(s_0^{(i)},\ldots,s_{f-1}^{(i)})\otimes\eta_i$. Recall that $n = 1$ if and only if $J^{\max}(\sigma) = \emptyset$ if and only if $\rhobar$ is reducible and $\sigma$ corresponds to $J_\sigma = \emptyset$ or $J_\sigma = S$ (see the beginning of \S\ref{operatorF}). We set $m \defeq  |J^{\max}(\sigma_i)|$ if $n>1$ (this doesn't depend on $i\in \{1,\dots,n\}$ by Lemma \ref{lem:Jmax}) and $m\defeq 1$ if $n=1$, so that $m\in \{1,\dots,f\}$. For $i\in\{1,\dots,n\}$ we let $\chi_i$ denote the $H$-eigencharacter on $\sigma_i^{N_0}=\sigma_i^{I_1}$. We also define for $i\in \{1,\dots,n\}$:
\begin{alignat*}{2}
s_i&\defeq  \!\!\!\sum_{j\in J^{\max}(\sigma_i)}s_j^{(i+1)} & \quad \text{if $n > 1$,} \\
s_1&\defeq  p-1 & \quad \text{if $n = 1$.}
\end{alignat*}

The following lemma will be useful later.

\begin{lem}\label{lem:sum:si}
We have $\sum_{i=1}^n s_i\equiv 0\pmod{p-1}$.
\end{lem}
\begin{proof}
Let $s(\chi_i)\in \{0,\dots,q-1\}$ such that $\chi_{i+1}=\chi_i\alpha^{-s(\chi_i)}$ and denote by $|s(\chi_i)|\in \{0,\dots,(p-1)f\}$ the sum of the digits of $s(\chi_i)$ in its $p$-expansion. Then it follows from (\ref{eq:xi+1}) below that we have
\[\alpha^{\sum_{j\in J^{\max}(\sigma_i)}s_j^{(i+1)} p^j + \sum_{j\notin J^{\max}(\sigma_i)}(p-1)p^j}\chi_i = \chi_{i+1}\]
and so
\begin{equation}\label{s(chi)}
s(\chi_i)=\sum_{j\in J^{\max}(\sigma_i)}(p-1-s_j^{(i+1)}) p^j
\end{equation}
which implies $|s(\chi_i)| = (p-1)m - s_i$. As $\chi_{n+1}=\chi_1=\chi_1\alpha^{-\sum_{i=1}^ns(\chi_i)}$, we have $\sum_{i=1}^n s(\chi_i)\equiv 0\pmod{q-1}$, hence $\sum_{i=1}^n |s(\chi_i)|\equiv 0\pmod{p-1}$ and the result follows.
\end{proof}

Recall $\pi$ is as at the end of \S\ref{lowerstatement}. In \cite[\S 4]{breuil-IL} there is defined an $\F$-linear isomorphism
\begin{equation}\label{isoS}
S:(\soc_{\GL_2(\cO_K)}\pi)^{I_1}\buildrel\sim\over\lra (\soc_{\GL_2(\cO_K)}\pi)^{I_1}.
\end{equation}
Fixing an embedding $\sigma\hookrightarrow \soc_{\GL_2(\cO_K)}\pi$, for $i\in \{2,\dots,n\}$ there are unique embeddings $\sigma_i\hookrightarrow \soc_{\GL_2(\cO_K)}\pi$ such that the morphism $S$ cyclically permutes the lines $\sigma_i^{I_1}$. In particular there exists $\nu\in \F^{\times}$ (which depends on $\sigma$ but not on the fixed embedding $\sigma\hookrightarrow \soc_{\GL_2(\cO_K)}\pi$) such that $S^{n}\vert_{\sigma_i^{I_1}}$ is the multiplication by $\nu$ for all $i\in \{1,\dots,n\}$. We define $\mu_i \in \F^\times$ for $1 \le i \le n$ by $\mu_1\defeq \nu$ if $n=1$ and if $n>1$:
\begin{equation*}
\mu_i \defeq 
\begin{cases}
\Big(\prod_{1\leq i'\leq n} \prod_{j\in J^{\max}(\sigma_{i'})}(p-1-s_j^{(i'+1)})!\Big)^{-1}\nu & \text{if $i = n$,} \\
1 & \text{otherwise.}
\end{cases}
\end{equation*}

We let $M_\sigma$ be the $\F\bbra{X}[F]$-submodule of $\pi^{N_1}$, or equivalently the $\F\bbra{Y}[F]$-submo\-dule, generated by $Y^{1-m}\sigma_i^{N_0}=Y^{1-m}\sigma_i^{I_1}$ for $1\leq i\leq n$. Recall $\gamma\in\Zp^\times$ acts on $M_\sigma\otimes \chi_\pi^{-1}$ by the action of $\smatr{1}{0}{0}{\gamma^{-1}}$ (see the end of \S\ref{lowerstatement}).

\begin{prop}\label{prop:phi-gamma-piece} 
The module $M_\sigma\otimes \chi_\pi^{-1}$ is admissible as an $\F\bbra{X}$-module {\upshape(}see \S\ref{covariant}{\upshape)}, $\Zp^\times$-stable, and such that $(M_\sigma\otimes \chi_\pi^{-1})^\vee$ is free of rank $n$ as an $\F\bbra{X}$-module. Moreover the \'etale $(\phz,\Gamma)$-module $(M_\sigma\otimes \chi_\pi^{-1})^\vee[1/X]$ admits a basis $(e_1,\ldots,e_n)$ over $\F\bbra{X}[1/X]$ such that for $i\in \{1,\dots,n\}$ {\upshape(}with $e_{n+1}\defeq e_1${\upshape)}:
\begin{align}
\phz(e_i)&=\mu_i^{-1}X^{s_i}e_{i+1},\label{eq:describeM:1}\\
\gamma(e_i)&\in\chi_i\big(\smatr{1}{0}{0}{\gamma}\big)\overline\gamma^m(1+X\F\bbra{X})e_i \text{\ for all $\gamma \in \Zp^\times$},\label{eq:describeM:2}
\end{align}
where $\overline\gamma$ is the image of $\gamma\in \Zp^\times$ in $\F$. Moreover $\gamma(e_i)$ is uniquely determined by {\upshape(\ref{eq:describeM:1})} and {\upshape(\ref{eq:describeM:2})}.
\end{prop}

To prepare for the proof, fix $x_1\in\sigma_1^{N_0} \setminus \{0\}$ and define for $1\leq i\leq n-1$:
\begin{equation*}
x_{i+1}\defeq  (-1)^{f-1}\prod_{j\in J^{\max}(\sigma_i)}Y_j^{s_j^{(i+1)}}\!\!\!\prod_{j\notin J^{\max}(\sigma_i)}Y_j^{p-1}\smatr{p}{0}{0}{1}x_i\in \sigma_{i+1}^{N_0}\setminus \{0\}
\end{equation*}
and $x_{n+1} \defeq  x_1$ (note that this formula is (\ref{xdelta}) multiplied by $(-1)^{f-1}$).

\begin{lem}\label{lem:basis-xi}
For $i\in \{1,\dots,n\}$ we have 
\begin{equation}\label{eq:summary}
S(x_i)=\Big(\underset{j\in J^{\max}(\sigma_i)}{\prod}(p-1-s_j^{(i+1)})!\Big)\mu_i x_{i+1}
\end{equation}
and
\begin{equation}\label{eq:F-action}
Y^{s_i}F(Y^{1-m}x_i)=\mu_i Y^{1-m}x_{i+1}.
\end{equation}
\end{lem}
\begin{proof}
If $i\in \{1,\dots,n\}$ we have
\begin{multline}\label{eq:xi+1}
(-1)^{f-1}\prod_{j\in J^{\max}(\sigma_i)}Y_j^{s_j^{(i+1)}}\!\!\!\prod_{j\notin J^{\max}(\sigma_i)}Y^{p-1}_j\smatr{p}{0}{0}{1}x_i\\
=\Big(\prod_{j\in J^{\max}(\sigma_i)}{(p-1-s_j^{(i+1)})!}\Big)^{-1}\theta_{\!\!\!\!\!\!\underset{J^{\max}(\sigma_i)}{\sum}(p-1-s_j^{(i+1)})p^j}\smatr{p}{0}{0}{1}x_i\\
=\Big(\prod_{j\in J^{\max}(\sigma_i)}{(p-1-s_j^{(i+1)})!}\Big)^{-1}S(x_i),
\end{multline}
where the first equality follows from Lemma \ref{lem:theta}(i) and the second from the definition of the function $S$ in \cite[\S 4]{breuil-IL}. From the definition of $x_{i+1}$, we obtain \eqref{eq:summary} for $i < n$. For $i = n$, using inductively
\[x_{i+1}=\Big(\prod_{j\in J^{\max}(\sigma_i)}{(p-1-s_j^{(i+1)})!}\Big)^{-1}S(x_i)\]
for $i=n-1$, $i=n-2$ till $i=1$ we obtain (as $S$ is $\F$-linear):
\begin{align*}
S(x_n)&=\Big(\prod_{j\in J^{\max}(\sigma_{n-1})} (p-1-s_j^{(n)})!\Big)^{-1} S^2(x_{n-1})\\
&=\cdots\\
&=\Big(\prod_{1\leq i\leq n-1} \prod_{j\in J^{\max}(\sigma_i)}(p-1-s_j^{(i+1)})!\Big)^{-1} S^n(x_1).
\end{align*}
Since $S^n(x_1) = \nu x_1$ and from the definition of $\mu_n$, we get \eqref{eq:summary} for $i = n$. The last part follows from Proposition \ref{prop:F-action} combined with (\ref{eq:xi+1}) and (\ref{eq:summary}).
\end{proof}

The following lemma is stated with the variable $Y$, but remains the same with the variable $X$.

\begin{lem}\label{lem:gen:Sigma}
Suppose $M$ is a torsion $\F\bbra{Y}$-module. Let $\Sigma\subseteq M$ be a subset spanning $M$ as an $\F$-vector space and set $\wtld{\Sigma}\defeq  \bigcup_{v\in\Sigma}\F^\times v$. If
\begin{enumerate}
\item\label{it:gen:Sigma:1}
$Y\Sigma\subseteq \wtld{\Sigma}\cup\{0\}$;
\item\label{it:gen:Sigma:2}
$\F Yv_1=\F Yv_2\neq 0\Longrightarrow v_1=v_2$ for $v_1,v_2\in \Sigma$;
\item\label{it:gen:Sigma:3}
$\Sigma\cap M[Y]$ is a finite set of $\F$-linearly independent vectors,
\end{enumerate}
then $\Sigma$ is an $\F$-basis of $M$ and $M$ is an admissible $\F\bbra{Y}$-module. If moreover $Y\wtld{\Sigma}= \wtld{\Sigma}\cup\{0\}$, then $M^\vee$ is a finite free $\F\bbra{Y}$-module of rank $\dim_\F M[Y]$.
\end{lem}
\begin{proof}
Write $\Sigma\cap M[Y]=\{v_1,\dots,v_d\}$ (assuming $\Sigma\cap M[Y]\ne \emptyset$ otherwise $M=0$ and there is nothing to prove). For $\ell\in \{1,\dots,d\}$ let $\Sigma_\ell\defeq \{v\in \Sigma : Y^jv\in \F^\times v_\ell{\rm\ for\ some\ }j\geq 0\}$. Then $M_\ell\defeq \oplus_{v\in \Sigma_\ell}\F v$ is an $\F\bbra{Y}$-module using \ref{it:gen:Sigma:1}. If $v,v'\in \Sigma_\ell$, then using \ref{it:gen:Sigma:2} there is $j\geq 0$ such that either $\F v=\F Y^j v'$, or $\F v'=\F Y^j v$, from which one easily deduces $M_\ell[Y]=\F v_\ell$, in particular $M_\ell$ is admissible. Since $\Sigma$ spans $M$ over $\F$ and $\Sigma=\coprod_{\ell=1}^n\Sigma_\ell$, the natural map $f:\bigoplus_{\ell=1}^dM_\ell\rightarrow M$ is surjective, and thus $M$ is also admissible. Since $\bigoplus_\ell M_\ell[Y]=\bigoplus_\ell \F v_\ell \hookrightarrow M[Y]$ (the last injection following from \ref{it:gen:Sigma:3}), we deduce that $\Ker(f)[Y]=0$, hence $\Ker(f)=0$ and $f$ is an isomorphism. This proves the first part of the statement. It follows from $Y\wtld{\Sigma}= \wtld{\Sigma}\cup\{0\}$ that the multiplication by $Y$ is surjective on each $M_\ell$, i.e.\ we have exact sequences $0\rightarrow \F v_\ell\rightarrow M_\ell\buildrel Y \over\rightarrow M_\ell\rightarrow 0$. Dualizing, this gives $0\rightarrow M_\ell^\vee \buildrel Y \over\rightarrow M_\ell^\vee \rightarrow (\F v_\ell)^\vee\rightarrow 0$, which shows $M_\ell^\vee$ is free of rank $1$ over $\F\bbra{Y}$. The last statement follows.
\end{proof}

Recall that $M_\sigma$ is the $\F\bbra{Y}[F]$-submodule of $\pi^{N_1}$ generated by $Y^{1-m}x_i$ for $1\leq i\leq n$. Let
\begin{equation*}
\Sigma\defeq \left\{
\begin{array}{llr}
Y^jF^k(Y^{1-m}x_i) : 1\leq i\leq n,\ k\geq 0,&0\leq j < p^{k-1}s_i&\text{if $k\geq 1$}\\ &0\leq j<m&\text{if $k=0$}\end{array}
\right\}.
\end{equation*}
We now check that $M_\sigma$ and $\Sigma$ satisfy all the assumptions in Lemma \ref{lem:gen:Sigma}. Define for $\ell\in\Z_{\geq 1}$:
\begin{equation*}
\Sigma_\ell\defeq \left\{
Y^jF^k(Y^{1-m}x_i)\in \Sigma : k+i\equiv \ell\pmod{n}
\right\}
\end{equation*}
and $M_{\ell,\sigma}\defeq  \bigoplus_{v\in \Sigma_\ell}\F v$. We have $\Sigma=\coprod_{\ell=1}^n\Sigma_\ell$. Applying $F^{k-1}$ to (\ref{eq:F-action}) for $k\geq 1$ we get (recall that $F\circ Y=Y^p \circ F$ on $\pi^{N_1}$): 
\begin{equation}\label{cycling}
Y^{p^{k-1}s_i}F^{k}(Y^{1-m}x_i)\in \F^\times F^{k-1}(Y^{1-m}x_{i+1}),
\end{equation}
hence $\Sigma$ spans $M_\sigma$ and condition \ref{it:gen:Sigma:1} of Lemma \ref{lem:gen:Sigma} holds for $\Sigma$. Using (\ref{cycling}) we also see that the multiplication by $Y$ induces an injection ${\Sigma}_\ell\hookrightarrow \wtld{\Sigma}_\ell\cup\{0\}$ and that $Y\wtld{\Sigma}_\ell=\wtld{\Sigma}_\ell\cup\{0\}$, hence $M_{\ell,\sigma}$ is an $F\bbra{Y}$-submodule of $M_\sigma$ and condition \ref{it:gen:Sigma:2} of Lemma \ref{lem:gen:Sigma} holds for $\Sigma_\ell$ and $\Sigma$. Moreover, $Y\wtld{\Sigma}= \wtld{\Sigma}\cup\{0\}$. Finally, $\Sigma\cap M_\sigma[Y]=\{x_1,\ldots,x_n\}$ (and $\Sigma\cap M_{\ell,\sigma}[Y]=x_\ell$). By Lemma \ref{lem:gen:Sigma} and its proof, we deduce that $\Sigma$ is an $\F$-basis of $M_\sigma$, that $M_\sigma=\bigoplus_{\ell=1}^nM_{\ell,\sigma}$ and that each $M_{\ell,\sigma}^\vee$ is free of rank $1$ over $F\bbra{Y}$. In fact one can visualize the ``$Y$-divisible line'' $M_{i+1,\sigma}$ as follows using (\ref{eq:F-action}):
\begin{multline*}
\F x_{i+1} \buildrel Y^{m-1} \over \longleftarrow \F Y^{1-m}x_{i+1} \buildrel Y^{s_i} \over \longleftarrow \F F(Y^{1-m}x_{i}) \buildrel Y^{ps_{i-1}} \over \longleftarrow \F F^2(Y^{1-m}x_{i-1}) \\
\buildrel Y^{p^2s_{i-2}} \over \longleftarrow \F F^3(Y^{1-m}x_{i-2})\longleftarrow \cdots,
\end{multline*}
where $\F x_{i+1}=M_{i+1,\sigma}[Y]$ and the arrows mean ``multiplication by the power of $Y$ just above''. In particular we see that if $d(v)\defeq  \min\{j\geq 1 : Y^jv=0\}$ for $v\in \Sigma$, then $v\in \Sigma_{i+1}$ is contained in $F(\wtld{\Sigma})$ if and only if $d(v)\equiv s_{i}+m \pmod{p}$. 

Define a basis $f_1,\dots,f_n$ of the free $F\bbra{Y}$-module $M_\sigma^\vee$ by
\[f_i(x_i)\defeq 1\ {\rm and}\ f_i(\Sigma\setminus\{x_i\})\defeq 0,\ i\in \{1,\dots,n\}.\]
From what is above we then easily deduce the following formula, where $F(f)(v)\defeq f(F(v))$ for $f\in M_\sigma^\vee$ and $v\in M_\sigma$ (and using conventions as in \S\ref{covariant}):
\begin{equation}\label{eq:F-act:dual}
F (Y^{\ell + (s_i+m-1)}f_{i+1})=\begin{cases}
\mu_i Y^{m-1}f_i&\text{if $\ell=0$,}\\
0&\text{if $1\leq \ell\leq p-1$}.
\end{cases}
\end{equation}

\begin{lem}
\label{lem:Zp-act}
The module $M_\sigma\otimes \chi_\pi^{-1}$ is $\Zp^\times$-stable, hence $\Zp^\times$ acts on $(M_\sigma\otimes \chi_\pi^{-1})^\vee$. Moreover we have for $\gamma\in \Zp^\times$ {\upshape(}recall $\gamma(f)(v)=f\big(\smatr{1}{0}{0}{\gamma}v\big)$ for $f\in (M_\sigma\otimes \chi_\pi^{-1})^\vee$, $v\in M_\sigma${\upshape)}:
\[
\gamma(f_i)\in \chi_i\big(\smatr{1}{0}{0}{\gamma}\big)(1+Y\F\bbra{Y})f_i
\]
for $1\leq i\leq n$.
\end{lem} 
\begin{proof}
As $M_\sigma=\bigoplus_{i=1}^{n}\F\bbra{Y} [F]Y^{1-m}x_i$ and $Y^{1-m}x_i$ is a $\Zp^\times$-eigenvector by Lemma \ref{lem:eigenvct}(iii) we deduce that $M_\sigma$, and hence $M_\sigma\otimes \chi_\pi^{-1}$, are $\Zp^\times$-stable.

From $\gamma\circ X= ((1+X)^\gamma-1)\circ \gamma$ and Lemma~\ref{lem:XY} it is easy to deduce that $\gamma \circ Y=f_\gamma(Y)\circ \gamma$ for some $f_\gamma(Y)\in \gamma Y+Y^2\F\bbra{Y}$, hence $\Zp^\times$ preserves the decomposition of $\F\bbra{Y}$-modules $M_\sigma\otimes \chi_\pi^{-1}=\bigoplus_{i=1}^nM_{i,\sigma}\otimes \chi_\pi^{-1}$. In particular, $\gamma(f_i)$ annihilates $M_{i',\sigma}\otimes \chi_\pi^{-1}$ for all $i'\ne i$. Let $Y^{-j}x_i$ for $j\geq 0$ denote the unique element of $\wtld{\Sigma}_i$ such that $Y^j(Y^{-j}x_i)=x_i$ (this is compatible with our previous notation in Lemma \ref{lem:eigenvct}(iii)). Then
\[
\gamma(f_i)=\sum_{j\geq 0}(\gamma(f_i)(Y^{-j}x_i)) Y^jf_i\in \chi_i\big(\smatr{1}{0}{0}{\gamma}\big)(1+Y\F\bbra{Y})f_i.\qedhere
\]
\end{proof}

\begin{proof}[Proof of Proposition~\ref{prop:phi-gamma-piece}]
We have already seen above that $M_\sigma\otimes \chi_\pi^{-1}$ is admissible, $\Zp^\times$-stable, and that $(M_\sigma\otimes \chi_\pi^{-1})^\vee$ is free of rank $n$ as an $\F\bbra{N_0/N_1}$-module. To find the basis $(e_i)_i$, first note from Lemma \ref{lem:XY} and (\ref{eq:F-act:dual}) that (using $F \circ Y^p=Y\circ F$ on $(M_\sigma\otimes \chi_\pi^{-1})^\vee$):
\begin{align}
F( X^{s_i+m-1}f_{i+1})&=
F\big(\sum_{j\geq 0}c_jY^{s_i+m-1+j}f_{i+1}\big)\nonumber\\
&=\mu_i \sum_{j\geq 0}c_{jp}Y^{m-1+j}f_{i}\nonumber\\
&\in(-1)^{s_i}\mu_i (1+X\F\bbra{X})X^{m-1}f_i\label{eq:F-act:alpha}
\end{align}
for some $c_j\in \F$ with $c_0=(-1)^{s_i+m-1}$. Similarly for $\ell\in \{1,\dots,p-1\}$:
\begin{equation}\label{eq:F-act:alpha:2}
F\big(X^{s_i+m-1+\ell}f_{i+1}\big)\in \F\bbra{X} X^mf_i.
\end{equation}
It easily follows from (\ref{correctformula}) that
\begin{equation}\label{eq:F-act:alpha:3}
\sum_{\ell=0}^{p-1}(1+X)^{-\ell}\varphi\big(F((1+X)^\ell f)\big)=f
\end{equation}
for all $f\in (M_\sigma\otimes \chi_\pi^{-1})^\vee[1/X]$. Let $f\defeq  X^{s_i+m-1}f_{i+1}$, by (\ref{eq:F-act:alpha}) and (\ref{eq:F-act:alpha:2}) we have for $\ell\in \{0,\dots,p-1\}$:
\[
F((1+X)^\ell f)\in (-1)^{s_i}\mu_i (1+X\F\bbra{X})X^{m-1}f_i,
\]
and so
\[
\phz\big(F((1+X)^\ell f)\big)\in (-1)^{s_i}\mu_i (1+X^p\F\bbra{X})\phz(X^{m-1}f_i).
\]
Using
\[\sum_{\ell=0}^{p-1}(1+X)^{-\ell}= \Big(\frac{X}{1+X}\Big)^{p-1}\equiv X^{p-1}\pmod{X^p},\]
we see that (\ref{eq:F-act:alpha:3}) applied to $f= X^{s_i+m-1}f_{i+1}$ becomes
\[(-1)^{s_i}\mu_i X^{p-1}\phz(X^{m-1}f_i)\in (1+X\F\bbra{X})X^{s_i+m-1}f_{i+1}\]
or equivalently in $(M_\sigma\otimes \chi_\pi^{-1})^\vee[1/X]$:
\begin{equation}\label{eq:5}
\phz(X^mf_i)= (-1)^{s_i}\mu_i^{-1}g_i(X)X^{s_i+m}f_{i+1}
\end{equation}
for some $g_i(X)\in 1+X\F\bbra{X}$.

Let $e_i\defeq  (-1)^{\sum_{j=1}^{i-1}s_j}h_i(X)X^{m}f_i$ for some $h_i(X)\in 1+X\F\bbra{X}$ and note that the sign doesn't change if $i$ is replaced by $i+n$ by Lemma \ref{lem:sum:si}. Then (\ref{eq:describeM:1}) is equivalent to
\[
h_i(X^p)\phz(X^mf_i)=(-1)^{s_i}\mu_i^{-1}h_{i+1}(X)X^{s_i}X^m f_{i+1},
\]
or equivalently $h_i(X^p)g_i(X)=h_{i+1}(X)$ by~\eqref{eq:5}. 
This system has the unique solution
\[
h_i(X)=\prod_{j=1}^\infty g_{i-j}(X^{p^{j-1}})
\]
in $1+X\F\bbra{X}$, where the indices are considered modulo $n$. Then (\ref{eq:describeM:2}) follows from Lemma \ref{lem:Zp-act}.
The final uniqueness assertion follows from $\gamma\circ \varphi=\varphi\circ \gamma$ and is left as an exercise (similar to \cite[Lemma 4.5]{breuil-IL}).
\end{proof}

Let $\mathcal O(\pi)$ (resp.\ $\mathcal O(\rhobar)$) be a set of representatives for the orbits of $\delta$ on the set of Serre weights in $\soc_{\GL_2(\cO_K)} \pi$ counted with their multiplicity $r$ (resp.\ on the set $W(\rhobar)$). We define $M_\pi \defeq  \bigoplus_{\sigma \in \mathcal O(\pi)} M_\sigma$ (with $M_\sigma$ as above). It follows from the assumptions on $\pi$ that we have
\[M_\pi \cong \bigoplus_{\sigma \in \mathcal O(\rhobar)} M_\sigma^{\oplus r}.\]
In particular $(M_\pi\otimes \chi_\pi^{-1})^\vee[1/X]$ is an \'etale $(\phz,\Gamma)$-module over $\F\ppar X$ of rank $r|W(\rhobar)| = r2^f$. From the description of $M_\sigma[X]$, we also see that the natural map $M_\pi\rightarrow \pi^{N_1}$ of torsion $F\bbra{X}$-modules is injective as the following composition is injective:
\[M_\pi[X]\cong \oplus\sigma^{I_1}\hookrightarrow \pi^{I_1}\subseteq \pi^{N_1}[X],\]
where the direct sum is over all Serre weights $\sigma$ in $\soc_{\GL_2(\cO_K)} \pi$ (counting their multiplicity $r$).

\begin{prop}\label{prop:main:functor}
We have an isomorphism of representations of $\gp$ over $\F$:
\[{\bf V}((M_\pi\otimes \chi_\pi^{-1})^\vee[1/X])\cong \big(\ind_{K}^{\otimes \Qp}\!(\rhobar)\big)^{\oplus r}.\]
\end{prop}
\begin{proof}
We are going to use a computation of \cite[\S 4]{breuil-IL}. Associated to the diagram $D\defeq D(\rhobar)^{\oplus r}$ of \S\ref{lowerstatement}, there is defined in {\it loc.cit.}\ an \'etale $(\phz,\Gamma)$-module over $\F\ppar X$ denoted there $M(D)$ and which is of the form $M(D)=\oplus_{\sigma \in \mathcal O(\pi)}M(D)_\sigma\footnote{A more consistent notation with the ones of this article would have been $M(D)^\vee$ and $M(D)^\vee_\sigma$\dots}$, where $M(D)_\sigma$ is a rank $n$ \'etale $(\phz,\Gamma)$-module over $\F\ppar X$ associated to the orbit of $\sigma$, i.e.\ to the cycle $\sigma=\sigma_1,\dots,\sigma_n$ as above (so in fact one has $M(D)=\oplus_{\sigma \in \mathcal O(\rhobar)}M(D)_\sigma^{\oplus r}$).
 
Let $N\defeq  \F\ppar X e$ be the rank $1$ \'etale $(\phz,\Gamma)$-module over $\F\ppar X$ defined by
\begin{align*}
\phz(e)&= X^{-(p-1)\sum_j (r_j+1)}e,\\
\gamma(e)&=\left(\frac{\overline\gamma X}{(1+X)^\gamma-1}\right)^{\sum_j(r_j+1)}\!\!\!e.
\end{align*}
We have ${\bf V}(N)\cong\omega^{\sum_j(r_j+1)}=\ind_{K}^{\otimes \Qp}\!(\det\rhobar)$ (using $\ind_{K}^{\otimes \Qp}\!(\omega_f)\cong\omega$) by \cite[Prop.3.5]{breuil-IL} and
\[{\bf V}(M(D))\cong \big(\ind_{K}^{\otimes \Qp}\!(\rhobar\otimes(\det\rhobar)^{-1})\big)^{\oplus r} \cong \big(\ind_{K}^{\otimes \Qp}\!(\rhobar) \otimes \ind_{K}^{\otimes \Qp}\!(\det\rhobar^{-1})\big)^{\oplus r}\]
by \cite[Thm.6.4]{breuil-IL}. We therefore deduce
\[{\bf V}(M(D)\otimes_{\F\ppar X}N)\cong\big(\ind_{K}^{\otimes \Qp}\!(\rhobar)\big)^{\oplus r}.\] 
Therefore it suffices to show that $M(D)_\sigma\otimes_{\F\ppar X}N\cong M_\sigma^\vee[1/X]$ for each $\sigma\in \mathcal O(\pi)$, or equivalently each $\sigma\in \mathcal O(\rhobar)$.

Let $x_1^\vee,\dots,x_n^\vee\in (\bigoplus_{i=1}^n\sigma_i^{I_1})^\vee$ be the dual basis of the $\F$-basis $(x_i)_i$ of $\bigoplus_{i=1}^n\sigma_i^{I_1}$, it follows from its definition in \cite[\S 4]{breuil-IL} and from (\ref{s(chi)}) that $M(D)_\sigma$ has basis $x_1^\vee,\dots,x_n^\vee$ as $\F\ppar X$-module with
\[
\phz(x_i^\vee)= X^{s_i+(p-1)(f-m)}\bigg(\prod_{j\in J^{\max}(\sigma_i)}(p-1-s_j^{(i+1)})!\bigg)(x_i^\vee\circ S\vert_{\oplus \sigma_i^{I_1}}^{-1}),
\]
where $S^{-1}$ is the inverse of the bijection $S$ of (\ref{isoS}) (which preserves $\bigoplus_{i=1}^n\sigma_i^{I_1}$). By (\ref{eq:summary}) we have
\[
x_i^\vee\circ S\vert_{\oplus \sigma_i^{I_1}}^{-1}=\bigg(\prod_{J^{\max}(\sigma_i)}(p-1-s_j^{(i+1)})!\bigg)^{-1}\mu_i^{-1}x_{i+1}^\vee,
\]
so we obtain
\[
\phz(x_i^\vee)=\mu_i^{-1}X^{s_i+(p-1)(f-m)}x_{i+1}^\vee.
\]
Also we have for $\gamma\in \Zp^\times$ (using the hypothesis on the central character of $\pi$):
\begin{align*}
x_i^\vee\circ\smatr{\gamma^{-1}}{0}{0}{1}&=\overline\gamma^{-\sum_j r_j}\left(x_i^\vee\circ \smatr{1}{0}{0}{\gamma}\right)\\
&=\overline\gamma^{-\sum_j r_j}\chi_i\big(\smatr{1}{0}{0}{\gamma}\big)x_i^\vee,
\end{align*}
hence with the definition of $\gamma(x_i^\vee)$ given in \cite[Lemma 4.5]{breuil-IL}:
\[\gamma(x_i^\vee)\in \chi_i\big(\smatr{1}{0}{0}{\gamma}\big)\overline\gamma^{-\sum_j r_j}(1+X\F\bbra{X})x_i^\vee.\]
We deduce that $M(D)_\sigma\otimes_{\F\ppar X}N\cong \bigoplus_{i=1}^n\F\bbra{X} (x_i^\vee\otimes e)$ with
\begin{align*}
\phz(x_i^\vee\otimes e)&=\mu_i^{-1}X^{s_i-(p-1)(m+\sum_j r_j)}(x_{i+1}^\vee\otimes e),\\
\gamma(x_i^\vee\otimes e)&\in \chi_i\big(\smatr{1}{0}{0}{\gamma}\big)\overline\gamma^{-\sum_j r_j}(1+X\F\bbra{X})(x_i^\vee\otimes e).
\end{align*}
Now, let $e_i'\defeq  X^{m+\sum_j r_j}(x_i^\vee\otimes e)$ for all $i$. Then $e'_1,\ldots,e'_n$ is a basis of $M(D)_\sigma\otimes_{\F\ppar X}N$ and we have for $i\in \{1,\dots,n\}$ (with $e'_{n+1}\defeq e'_1$): 
\begin{align*}
\phz(e'_i)&=\mu_i^{-1}X^{s_i}e'_{i+1},\\
\gamma(e'_i)&\in\chi_i\big(\smatr{1}{0}{0}{\gamma}\big)\overline{\gamma}^m(1+X\F\bbra{X})e'_i.
\end{align*}
From Proposition~\ref{prop:phi-gamma-piece} we see that $M(D)_\sigma\otimes_{\F\ppar X} N\cong M_\sigma^\vee[1/X]$.
\end{proof}

By Lemma \ref{twistanddual} this completes the proof of Theorem~\ref{thm:main-tensor-ind} when the constants $\nu_i$ are as in \cite[Thm.6.4]{breuil-IL}. When they are arbitrary, the proof of Proposition \ref{prop:main:functor} gives ${\bf V}((M_\pi\otimes \chi_\pi^{-1})^\vee[1/X])\vert_{I_{\Qp}}\cong \big(\ind_{K}^{\otimes \Qp}\!(\rhobar)\big)\vert_{I_{\Qp}}^{\oplus r}$ using \cite[Cor.5.4]{breuil-IL}, which finishes the proof of Theorem~\ref{thm:main-tensor-ind}.

\subsection{On the structure of some representations of \texorpdfstring{$\GL_2(K)$}{GL\_2(K)}}\label{subsection:gr-pi}

We prove results on the structure of an admissible smooth representation $\pi$ of $\GL_2(K)$ over $\F$ associated to a semisimple sufficiently generic representation $\rhobar$ of $\gK$ as in \cite{BP} when $\pi$ satisfies a further multiplicity one assumption as in \cite{BHHMS1} and a self-duality property. In particular we prove that such a $\pi$ is irreducible if and only if $\rhobar$ is, and is semisimple when $f=2$ (Corollary \ref{cor:split2} and Corollary \ref{cor:pi-irred}).

We keep the notation at the beginning of \S\S\ref{gl2},~\ref{phigamma}, and set $\Lambda\defeq \F\bbra{I_1/Z_1}$. We recall that the graded ring $\gr(\Lambda)$ is isomorphic to $\otimes_{i=0}^{f-1}\F[y_i,z_i,h_i]$ with $h_i$ lying in the center (see (\ref{gri1})). We set
\[R\defeq \gr(\Lambda)/(h_0,\dots,h_{f-1}),\]
which is commutative and isomorphic to $\F[y_i,z_i, 0\leq i\leq f-1]$, and recall that $\overline{R}=R/(y_iz_i, 0\leq i\leq f-1)=\gr(\Lambda)/J$ (see (\ref{rbarJ})). Moreover the finite torus $H$ naturally acts on $\Lambda$ by the conjugation on $I_1$ (via its Teichm\"uller lift) and we see (using (\ref{elementsyi})) that the induced action on $\gr(\Lambda)$ is trivial on $h_i$ and is the multiplication by the character $\alpha_i$ (resp.\ $\alpha_i^{-1}$) on $y_i$ (resp.\ $z_i$), where $\alpha_i\big(\smatr{\lambda}00{\mu}\big)\defeq \sigma_i(\lambda\mu^{-1})$ for $\smatr{\lambda}00{\mu}\in H$.

Notice that $\gr(\Lambda)$ is an Auslander regular ring (see \cite[Def.III.2.1.7]{LiOy}, \cite[Def.III.2.1.3]{LiOy}) by the first statement in \cite[Thm.5.3.4]{BHHMS1} and so is $\Lambda$ itself by \cite[Thm.III.2.2.5]{LiOy}. This allows us to apply (many) results of \cite[\S III.2]{LiOy}. 

For any ring $S$ and any $S$-module $M$, we set ${\rm E}^i_S(M)\defeq \Ext_S^i(M,S)$ for $i\geq 0$.

\subsubsection{Combinatorial results}\label{combi}

We define some explicit ideals $\mathfrak{a}(\lambda)$ of $R$ and study some of their properties.

We fix a continuous representation $\brho:\Gal(\overline{\Q}_p/K)\ra \GL_2(\F)$ which is generic in the sense of \cite[\S11]{BP} and let $D_0(\rhobar)$ be the representation of $\GL_2(\F_q)$ over $\F$ defined in \cite[\S13]{BP} (see also \S\ref{lowerstatement} when $\rhobar$ is semisimple). Recall from \cite[Cor.13.6]{BP} that $D_0(\brho)^{I_1}$ is multiplicity-free as a representation of $H\simeq I/I_1$. By \cite[\S4]{breuil-buzzati}, there is a bijection between the characters of $H$ appearing in $D_0(\brho)^{I_1}$ and a certain set of $f$-tuples, denoted by
\[\mathscr{PID}(x_0,\ldots,x_{f-1}),\  \mathrm{resp}.\ \mathscr{PRD}(x_0,\ldots,x_{f-1}),\  \mathrm{resp}.\ \mathscr{PD}(x_0,\ldots,x_{f-1}),\]
if $\brho$ is irreducible, resp.\ reducible split, resp.\ reducible nonsplit. We refer to \cite[\S4]{breuil-buzzati} for the precise definition of these sets and we simply write $\mathscr{P}$ for the set associated to $\brho$. We write $\chi_{\lambda}$ for the character of $H$ associated to $\lambda\in \mathscr P$ (more precisely, in {\it loc.cit.}\ one rather associates a Serre weight $\sigma_\lambda$ to $\lambda$, and $\chi_\lambda$ is the action of $H=I/I_1$ on the $1$-dimensional subspace $\sigma_\lambda^{I_1}$, different $\sigma_\lambda$ giving different $\chi_\lambda$).

On the other hand, the set $W(\brho)$ is in bijection with another set of $f$-tuples, denoted by (see \cite[\S11]{BP})
\[\mathscr{ID}(x_0,\ldots,x_{f-1}),\  \mathrm{resp}.\ \mathscr{RD}(x_0,\ldots,x_{f-1}),\  \mathrm{resp}.\ \mathscr{D}(x_0,\ldots,x_{f-1}),\] depending on $\brho$ as above. We simply write $\mathscr{D}$ for the set associated to $\brho$. Since the socle of $D_0(\rhobar)$ is $\oplus_{\sigma\in W(\rhobar)}\sigma$, we may view $\mathscr{D}$ as a subset of $\mathscr{P}$. For example, if $\brho$ is reducible split, then $\mathscr{D}$ is the subset of $\mathscr{P}$ consisting of $\lambda$ such that 
\[\lambda_j(x_j)\in\{x_j,x_j+1,p-2-x_j,p-3-x_j\},\]
while if $\brho$ is nonsplit, then we require moreover that $\lambda_j(x_j)\in\{x_j+1,p-3-x_j\}$ implies $j\in J_{\brho}$, where $J_{\brho}$ is a certain subset of $\{0,\dots,f-1\}$ uniquely determined by the Fontaine--Laffaille module of $\brho$ (cf.\ \cite[(17)]{breuil-buzzati}). 

\begin{definit} \label{def:a(lambda)}
We associate to $\lambda\in\mathscr{P}$ an ideal $\mathfrak{a}(\lambda)$ of $R$ as follows. 
\begin{itemize}
\item If $\brho$ is irreducible, then $\mathfrak{a}(\lambda)=(t_0,\dots,t_{f-1})$, where 
\[t_0\defeq \left\{\begin{array}{llll}z_0& {\rm if}& \lambda_0(x_0)\in\{x_0-1,p-2-x_0\}\\
y_0& {\rm if}&\lambda_0(x_0)\in\{x_0+1,p-x_0\}\\
y_0z_0& {\rm if}&\lambda_0(x_0)\in \{x_0,p-1-x_0\}, \end{array}\right.\]
and if $j\neq 0$
\[t_j\defeq \left\{\begin{array}{llll}z_j& {\rm if} &\lambda_j(x_j)\in\{x_j,p-3-x_j\}\\
y_j& {\rm if} &\lambda_j(x_j)\in\{x_j+2,p-1-x_j\}\\
y_jz_j& {\rm if} &\lambda_j(x_j)\in \{x_j+1,p-2-x_j\}.\end{array}\right.\]

\item If $\brho$ is reducible nonsplit, then $\mathfrak{a}(\lambda)=(t_0,\dots,t_{f-1})$, where
\[t_j\defeq \left\{\begin{array}{llll}z_j& {\rm if} &\lambda_j(x_j)\in\{x_j,p-3-x_j\} \ \mathrm{and}\ j\in J_{\brho}\\
y_j& {\rm if} &\lambda_j(x_j)\in\{x_j+2,p-1-x_j\} \ \mathrm{and}\ j\in J_{\brho}\\
y_jz_j & {\rm if} &\lambda_j(x_j)\in \{x_j,p-1-x_j\} \ \mathrm{and}\ j\notin J_{\brho}\\
y_jz_j& {\rm if} &\lambda_j(x_j)\in \{x_j+1,p-2-x_j\}.\end{array}\right.\]

\item If $\brho$ is reducible split, then $\mathfrak{a}(\lambda)=(t_0,\dots,t_{f-1})$ is defined as in the nonsplit case by letting $J_{\brho}=\{0,\dots,f-1\}$, namely 
\[t_j\defeq \left\{\begin{array}{llll}z_j& {\rm if} &\lambda_j(x_j)\in\{x_j,p-3-x_j\}\\
y_j& {\rm if} &\lambda_j(x_j)\in\{x_j+2,p-1-x_j\}\\
y_jz_j& {\rm if} &\lambda_j(x_j)\in \{x_j+1,p-2-x_j\}.\end{array}\right.\]
\end{itemize}
\end{definit}
In particular, if $\brho$ is reducible nonsplit and $J_{\brho}=\emptyset$, then $\mathfrak{a}(\lambda)\!=\!(y_0z_0,\dots,y_{f-1}z_{f-1})$ for any $\lambda\in \mathscr{P}$. Note that $R/\mathfrak{a}(\lambda)$ is always a quotient of $\overline{R}$.

\begin{rem}
An equivalent form of Definition \ref{def:a(lambda)} is as follows (compare the proof of Theorem \ref{thm:cycle-pi}). Given $\lambda\in\mathscr{P}$, $t_j=y_j$ (resp.\ $t_j=z_j$) if and only if the character $\chi_{\lambda}\alpha_j^{-1}$ (resp.\ $\chi_{\lambda}\alpha_j$) occurs in $D_0(\brho)^{I_1}$ (i.e.\ has the form $\chi_{\lambda'}$ for some $\lambda'\in\mathscr{P}$), and $t_j=y_jz_j$ if and only if neither of $\chi_{\lambda}\alpha_j^{\pm1}$ occurs in $D_0(\brho)^{I_1}$. 
\end{rem}

\begin{lem}\label{lem:Yj-fa}
Let $\lambda\in\mathscr{P}$.
\begin{enumerate}
\item Assume $\brho$ is semisimple. Then $\lambda\in\mathscr{D}$ if and only if $y_j\notin\mathfrak{a}(\lambda)$ for any $j\in\{0,\dots,f-1\}$.
\item Assume $\brho$ is reducible nonsplit and let $\brho^{\rm ss}$ be the semisimplification of
  $\brho$. Then there is a bijection between $\mathscr{D}(\brho^{\rm ss})$ {\upshape(}defined as the set
  $\mathscr{D}$ associated to $\brho^{\rm ss}${\upshape)} and the set of $\lambda\in\mathscr{P}$ such
  that $y_j\notin\mathfrak{a}(\lambda)$ for any $j\in\{0,\dots,f-1\}$.
\end{enumerate}
\end{lem}
\begin{proof}
(i) It is clear by definition of $\mathscr{D}$ and $\mathfrak{a}(\lambda)$.\\
(ii) Let $\lambda\in\mathscr{P}$ such that $y_j\notin \mathfrak{a}(\lambda)$ for any $j\in\{0,\dots,f-1\}$. By definition, we have (for $\brho$ reducible nonsplit)
\[\lambda_j(x_j)\in\{x_j,x_j+1,p-1-x_j,p-2-x_j,p-3-x_j\}\]
and from the definition of $\mathfrak{a}(\lambda)$ if $\lambda_j(x_j)=p-1-x_j$ then $j\notin J_{\brho}$ (note that if $\lambda_j(x_j)=p-3-x_j$ then it is automatic that $j\in J_{\brho}$). We define an $f$-tuple $\mu$ by 
\[\mu_j(x_j)\defeq \left\{\begin{array}{cll} p-3-x_j & \mathrm{if}\ \lambda_j(x_j)=p-1-x_j\\
\lambda_j(x_j) & \mathrm{otherwise}.
\end{array}\right.\] 
It is then easy to see that $\mu$ is an element of $\mathscr{D}(\brho^{\rm ss})$ and that any element of $\mathscr{D}(\brho^{\rm ss})$ arises (uniquely) in this way. 
\end{proof} 

\begin{cor}\label{cor:dim-D(pi)}
The set $\{\lambda\in\mathscr{P} : y_j\notin \mathfrak{a}(\lambda)\ \forall \ j\in\{0,\dots,f-1\}\}$ has cardinality $2^f$.
\end{cor}
\begin{proof}
This is a direct consequence of Lemma \ref{lem:Yj-fa} and of $\vert W(\rhobar^{\rm ss})\vert=2^f$.
\end{proof}

Given $\lambda\in \mathscr{P}$, write $\mathfrak{a}(\lambda)=(t_0,\dots,t_{f-1})$ as in Definition \ref{def:a(lambda)} and define
\begin{equation}\label{eq:A(lambda)}
\mathcal{A}(\lambda)\defeq \{j\in\{0,\dots,f-1\} : t_j=y_jz_j\}\subseteq \{0,\dots,f-1\}.
\end{equation}

The following proposition will only be used in Corollary \ref{cor:mul=4^f} below.

\begin{prop}\label{prop:sum=4^f}
We have $\sum_{\lambda\in \mathscr{P}}2^{|\mathcal{A}(\lambda)|}=4^f$.
\end{prop}
\begin{proof}
We will only give the proof in the case $\brho$ is reducible (split or not), the irreducible case can be treated similarly. 

First assume that $\brho$ is split. Given $\lambda\in\mathscr{P}$, we define an element $\overline{\lambda}\in\mathscr{D}$ as follows: 
\[\overline{\lambda}_j(x_j)\defeq \left\{\begin{array}{llll} x_j & {\rm if}& \lambda_j(x_j)\in\{x_j,x_j+2\}\\
p-3-x_j & {\rm if}&\lambda_j(x_j)\in\{p-1-x_j,p-3-x_j\}\\
\lambda_j(x_j) & &\mathrm{otherwise}.\end{array}\right.\] 
It is easy to see that $\overline{\lambda}\in \mathscr{D}$. 
By definition of $\mathscr{P}$ (see \cite[\S4]{breuil-buzzati} and recall $\mathscr{P}=\mathscr{PRD}(x_0,\ldots,x_{f-1})$), for each $\overline{\lambda}\in \mathscr{D}$, there are exactly $2^{|\{0,\dots,f-1\}\backslash\mathcal{A}(\overline{\lambda})|}$ elements $\lambda$ in $\mathscr{P}$ giving rise to $\overline{\lambda}$ under the above rule. 
Moreover, it is direct from the definitions that $\mathcal{A}(\lambda)=\mathcal{A}(\overline{\lambda})$. 
Hence
\[\sum_{ \lambda\in\mathscr{P}}2^{|\mathcal{A}(\lambda)|}=\sum_{\overline{\lambda}\in\mathscr{D}}(2^{f-|\mathcal{A}(\overline{\lambda})|} 2^{|\mathcal{A}(\overline{\lambda})|})= 2^f |\mathscr{D}|=2^f 2^f=4^f.\]

Now assume that $\brho$ is nonsplit. Let $\overline{\mathscr{P}}$ be the subset of $\mathscr{P}$ considered in the proof of Lemma \ref{lem:Yj-fa}(ii), namely $\lambda\in\overline{\mathscr{P}}$ if and only if
\[\lambda_j(x_j)\in\{x_j,x_j+1,p-1-x_j,p-2-x_j,p-3-x_j\} \]
and $\lambda_j(x_j)=p-1-x_j$ implies $j\notin J_{\brho}$. By the proof of \emph{loc.cit.}, we have $|\overline{\mathscr{P}}|=|\mathscr{D}(\brho^{\rm ss})|=2^f$. Given $\lambda\in\mathscr{P}$, we define an element $\overline{\lambda}\in \overline{\mathscr{P}}$ as follows: 
\[\overline{\lambda}_j(x_j)\defeq \left\{\begin{array}{llll} x_j & {\rm if}&\lambda_j(x_j)\in\{x_j,x_j+2\}\\
p-3-x_j & {\rm if}&\!\lambda_j(x_j)=p-3-x_j \ \mathrm{or}\ (\lambda_j(x_j)=p-1-x_j \ \mathrm{and}\ j\in J_{\brho})\\
\lambda_j(x_j) && \mathrm{otherwise}.\end{array}\right.\]
As in the split case it is easy to see that $\mathcal{A}(\lambda)=\mathcal{A}(\overline{\lambda})$ and that given $\overline{\lambda}\in\overline{\mathscr{P}}$, there exist exactly $2^{|\{0,\dots,f-1\}\backslash\mathcal{A}(\overline{\lambda})|}$ elements $\lambda$ in $\mathscr{P}$ giving rise to $\overline{\lambda}$. The result follows as in the split case.
\end{proof}

\medskip 

\begin{definit}\label{def:lambda-dual}
Given $\lambda\in\mathscr{P}$, we define another $f$-tuple $\lambda^{*}$ as follows:
\[\lambda^*_j(x_j)\defeq \left\{\begin{array}{llll}
p-3-\lambda_j(x_j)&{\rm if}&t_j=z_j\\
p+1-\lambda_j(x_j)&{\rm if}&t_j=y_j\\
p-1-\lambda_j(x_j)&{\rm if}&t_j=y_jz_j.
\end{array}\right.\]
\end{definit}

If $\lambda\in\mathscr{D}$, we define its ``length'' $\ell(\lambda)$ to be (see \cite[\S4]{BP}):
\begin{equation}\label{eq:ell}
\ell({\lambda})\defeq |\{j\in\{0,\dots,f-1\} : \lambda_j(x_j)\in\{p-2-x_j\pm1, x_j\pm1\}\}|.
\end{equation}

\begin{lem}
Let $\lambda\in\mathscr{P}$.
\begin{enumerate}
\item We have $\lambda^*\in\mathscr{P}$ and $\mathfrak{a}(\lambda)=\mathfrak{a}(\lambda^*)$.
\item Assume that $\brho$ is semisimple. Then $\lambda\in\mathscr{D}$ if and only if
  $\lambda^*\in\mathscr{D}$, and in this case $\ell(\lambda^*)=f-\ell(\lambda)$.
\end{enumerate}
\end{lem}
\begin{proof}
(i) The first statement can be checked directly using the definition of $\mathscr{P}$ and the second one is obvious from the definitions. 

(ii) The first statement follows from (i) and Lemma \ref{lem:Yj-fa}(i). By definition of $\mathscr{D}$ (see \cite[\S11]{BP}), $\ell(\lambda)$ can be computed as the cardinality of the following set:
\[\big\{j\in\{0,\dots,f-1\} : \lambda_j(x_j)\in\{p-1-x_j,p-2-x_j,p-3-x_j\}\big\}.\]
For example, when $\brho$ is reducible split, we have (cf.\ the beginning of \cite[\S11]{BP}) 
\[\lambda_j(x_j)\in\{p-2-x_j,p-3-x_j\} \Longleftrightarrow \lambda_{j+1}(x_{j+1})\in\{p-3-x_{j+1},x_{j+1}+1\}.\] 
The second statement of (ii) follows from this and Definition \ref{def:lambda-dual}.
\end{proof}

\begin{lem}\label{lem:lambda-dual}
Let $\lambda\in \mathscr{P}$, $\chi_\lambda$ the character of $H$ associated to $\lambda$, $(t_0,\dots,t_{f-1})$ the ideal $\mathfrak{a}(\lambda)$ in Definition \ref{def:a(lambda)} and $\eta_{\lambda}$ be the character of $H$ acting on $\prod_{j=0}^{f-1}t_j$. Then we have
\[\chi_{\lambda}\chi_{\lambda^{*}}=\eta_{\lambda}(\eta\circ\det),\]
where $\lambda^{*}$ is as in Definition \ref{def:lambda-dual} and $\eta(a)\defeq \chi_{\lambda}\big(\smatr{a}00a\big)$ for $a\in\F_q^{\times}$ {\upshape(}$\eta$ does not depend on $\lambda\in \mathscr{P}${\upshape)}.
\end{lem}
\begin{proof}
This is an easy computation, but we give some details. Note that $\lambda_j(x_j)+\lambda^*_j(x_j)=(p-1)+2\varepsilon_j$, where $\varepsilon_j$ equals $1$, $0$ or $-1$ if $t_j$ equals $y_j$, $y_jz_j$ or $z_j$ respectively. Moreover, in the notation of \cite[\S 4]{breuil-buzzati}, we have
\begin{align*}
e(\lambda)+e(\lambda^*)&=\frac{1}{2}\Big(p^f-1+\sum_{j=0}^{f-1}p^j(x_j-\lambda_j(x_j)+x_j-\lambda^*_j(x_j))\Big)\\
&=\sum_{j=0}^{f-1}p^j(x_j-\varepsilon_j).
\end{align*}
The conclusion follows now from a simple computation, noting that for $\smatr{a}00{b}\in H$
\[\chi_\lambda\big(\smatr{a}00{b}\big)=\sigma_0(a)^{\big(\sum_{j=0}^{f-1}p^j\lambda_j(r_j)\big)+e(\lambda)(r_0,\dots, r_{f-1})}\sigma_0(b)^{e(\lambda)(r_0,\dots, r_{f-1})}\]
(see \cite[\S 4]{breuil-buzzati}) and that $H$ acts on $y_i$ (resp.~$z_i$) via $\alpha_i$ (resp.~$\alpha_i^{-1}$).
\end{proof}

Note that $H$ acts on $I_1/Z_1$ by conjugation and hence on $\Lambda$ and $\gr(\Lambda)$. This induces
$H$-actions also on $R$, $\overline{R}$, and $R/\mathfrak{a}(\lambda)$ for any $\lambda \in \mathscr{P}$.
We say that $M$ is a $\gr(\Lambda)$-module with compatible $H$-action if $H$ acts on $M$ such that
$h(rm) = h(r)h(m)$ for $h \in H$, $r \in R$, and $m \in M$. In this case ${\rm E}^i_{\gr(\Lambda)}(M)$
is again a $\gr(\Lambda)$-module with compatible $H$-action for any $i \ge 0$.

\begin{lem}\label{lem:devissage}
If $M$ is a $\gr(\Lambda)$-module with compatible $H$-action that is annihilated by $(h_0,\dots,h_{f-1})$, then we have isomorphisms of $\gr(\Lambda)$-modules with compatible $H$-action for $i\geq 0$:
\begin{equation}\label{eigr}
{\rm E}^{i+f}_{\gr(\Lambda)}(M)\cong {\rm E}^{i}_{R}(M).
\end{equation}
If moreover $M$ is annihilated by $J$, then we have isomorphisms of $\gr(\Lambda)$-modules with compatible $H$-action for $i\geq 0$:
\begin{equation}\label{eigr-bis}
{\rm E}^{i+2f}_{\gr(\Lambda)}(M)\cong {\rm E}^{i+f}_{R}(M)\cong {\rm E}^{i}_{\overline{R}}(M).
\end{equation}
\end{lem}
\begin{proof}
Since $(h_0,\dots,h_{f-1})$ is a regular sequence of central elements in $\gr(\Lambda)$ and $(y_0z_0,\dots,y_{f-1}z_{f-1})$ is a regular sequence in $R$ (which is commutative), the isomorphisms \eqref{eigr} and \eqref{eigr-bis} as $\gr(\Lambda)$-modules are proved as in the proof of \cite[Lemma 5.1.3]{BHHMS1}. Moreover, $H$ acts trivially on $h_j$ and $y_jz_j$ (for $0\leq j\leq f-1$), the isomorphisms are also $H$-equivariant, from which the results follow.
\end{proof}

We don't use the following proposition in the sequel, but it is consistent with Remark \ref{rem:yongquan}(i) and the essential self-duality assumption (iii) in \S\ref{sec:length-of-pi} below (see Proposition \ref{prop:LvO-filtonExt}).

\begin{prop}\label{prop:Ext-R/I}
For $\lambda\in \mathscr{P}$ there is an isomorphism of $\gr(\Lambda)$-modules with compatible $H$-action:
\[{\rm E}^{2f}_{\gr(\Lambda)}\big(\chi_{\lambda}^{-1}\otimes R/\mathfrak{a}(\lambda)\big)\cong \big(\chi_{\lambda^*}^{-1}\otimes R/\mathfrak{a}(\lambda)\big)\otimes\eta\circ\det. \]
\end{prop}
\begin{proof}
Applying \eqref{eigr-bis} with $i=0$ and $M=\chi_{\lambda}^{-1}\otimes R/\mathfrak{a}(\lambda)$, we are left to prove 
\[\Hom_{\overline{R}}(\chi_{\lambda}^{-1}\otimes R/\mathfrak{a}(\lambda),\overline{R})\cong \big(\chi_{\lambda^*}^{-1}\otimes R/\mathfrak{a}(\lambda)\big)\otimes\eta\circ\det.\]
Using Lemma \ref{lem:lambda-dual}, it suffices to construct an isomorphism of $\gr(\Lambda)$-modules
with compatible $H$-action
\begin{equation}\label{eq:isom-barR}
\Hom_{\overline{R}}(R/\mathfrak{a}(\lambda),\overline{R})\cong \eta_{\lambda}^{-1}\otimes R/\mathfrak{a}(\lambda),
\end{equation}
where $\eta_{\lambda}$ is the character of $H$ acting on $\prod_{j=0}^{f-1}t_j$ if we write $\mathfrak{a}(\lambda)=(t_0,\dots,t_{f-1})$ with $t_j\in\{y_j,z_j,y_jz_j\}$. 
Put $t'\defeq \prod_{j=0}^{f-1}(y_jz_j/t_j)$. One easily checks that $t'\overline{R} =\overline{R}[\mathfrak{a}(\lambda)]$ and there is an isomorphism of $\overline{R}$-modules 
\[\theta:\ \eta_{\lambda}^{-1}\otimes \overline{R}/\mathfrak{a}(\lambda) \simto t'\overline{R}, \]
where the first map sends $1$ to $t'$. As $H$ acts on $t'$ via $\eta_{\lambda}^{-1}$, $\theta$ is also $H$-equivariant. The isomorphism \eqref{eq:isom-barR} is then obtained by sending $r\in \eta_{\lambda}^{-1}\otimes R/\mathfrak{a}(\lambda)$ to $\phi\in \Hom_{\overline{R}}(R/\mathfrak{a}(\lambda),\overline{R})$ such that $\phi(1)\defeq \theta(r)$.
\end{proof}

\subsubsection{On the structure of \texorpdfstring{$\gr(\pi^{\vee})$}{gr(pi\^{}v)}}\label{grstr}

We give a partial result on the structure of $\gr(\pi^\vee)$ for certain admissible smooth representations $\pi$ of $\GL_2(K)$ over $\F$ associated to $\rhobar$ when $\gr(\pi^\vee)$ comes from the $\m_{I_1/Z_1}$-adic filtration on $\pi^\vee$.

We let $\rhobar$ be as in \S\ref{combi} (in particular $\rhobar$ is not necessarily semisimple) and keep the notation of {\it loc.cit.}\ As in \S\ref{lowerstatement} when $\rhobar$ is semisimple, we consider $D_0(\rhobar)$ as a representation of $\GL_2(\oK)K^\times$, where $\GL_2(\oK)$ acts via its quotient $\GL_2(\Fq)$ and the center $K^\times$ acts by the character $\det(\rhobar)\omega^{-1}$. We now write $\m$ for $\m_{I_1/Z_1}$.

We consider an admissible smooth representation $\pi$ of $\GL_2(K)$ over $\F$ satisfying the following two conditions:
\begin{enumerate}
\item there is $r\geq 1$ such that $\pi^{K_1}\simeq D_0(\rhobar)^{\oplus r}$ as a representation of $\GL_2(\oK)K^\times$ (in particular $\pi$ has a central character);
\item for any $\lambda\in\mathscr{P}$, we have an equality of multiplicities
\[[\pi[\m^3]:\chi_{\lambda}]=[\pi[\m]:\chi_{\lambda}].\]
\end{enumerate}
Note that (ii) implies that the $\gr(\Lambda)$-module $\gr(\pi^\vee)$ (defined with the $\m$-adic filtration on $\pi^\vee$) is annihilated by the ideal $J$ in (\ref{idealJ}) by the proof of \cite[Cor.5.3.5]{BHHMS1}, and in particular is an $\overline{R}$-module.

\begin{thm}\label{thm:cycle-pi}
For $\pi$ as above, there is a surjection of $\gr(\Lambda)$-modules with compatible $H$-action
\begin{equation}\label{eq:gr(pi)}
\big(\bigoplus_{\lambda\in\mathscr{P}}\chi_{\lambda}^{-1}\otimes R/\mathfrak{a}(\lambda)\big)^{\oplus r}\twoheadrightarrow \gr(\pi^{\vee}),
\end{equation}
where $\mathfrak{a}(\lambda)$ is as in Definition \ref{def:a(lambda)}.
\end{thm}
\begin{proof}
Consider the $\gr(\Lambda)$-module with compatible $H$-action:
\[M\defeq \big(\bigoplus_{\lambda\in\mathscr{P}}\chi_{\lambda}^{-1}\otimes R/\mathfrak{a}(\lambda)\big)^{\oplus r}.\]
Since there is a bijection $\lambda\mapsto \chi_\lambda$ between $\mathscr{P}$ and the characters of $H$ on $D_0(\brho)^{I_1}$ (see \S\ref{combi}), we can choose a basis of $\pi^{I_1}$ over $\F$, say $\{v_{\lambda,k} :  \lambda\in \mathscr{P}, 1\leq k\leq r\}$, such that each $v_{\lambda,k}$ is an eigenvector for $I$ of character $\chi_{\lambda}$. We denote by $\{e_{\lambda,k} : \lambda\in \mathscr{P}, 1\leq k\leq r\}$ the basis of $\gr^0(\pi^{\vee})$ over $\F$ which is the dual basis of $\{v_{\lambda,k}\}$, and note that $\{e_{\lambda,k}\}$ generates the $\gr(\Lambda)$-module $\gr(\pi^{\vee})$. To prove that there exists a surjective morphism $M\twoheadrightarrow \gr(\pi^{\vee})$ it suffices to prove that, for any $\lambda\in \mathscr P$ and any $k\in \{1,\dots,r\}$, the vector $e_{\lambda,k}$ is annihilated by the ideal $\mathfrak{a}(\lambda)$ of $R=\gr(\Lambda)/(h_0,\dots,h_{f-1})$. Writing $\mathfrak{a}(\lambda)=(t_0,\dots,t_{f-1})$ as in Definition \ref{def:a(lambda)}, we already see that if $t_j=y_jz_j$, then $t_j$ kills all the $e_{\lambda,k}$ since $\gr(\pi^{\vee})$ is annihilated by $J$.

Let $j\in \{0,\dots,f-1\}$ such that $t_j\in \{y_j,z_j\}$ and define $\chi'\defeq \chi_{\lambda}\alpha_j^{-1}$ if $t_j=y_j$, $\chi'\defeq \chi_{\lambda}\alpha_j$ if $t_j=z_j$. By Definition \ref{def:a(lambda)} one checks that $\chi'=\chi_{\lambda'}$, where $\lambda'\in\mathscr{P}$ is defined by $\lambda'_i(x_i)\defeq \lambda_i(x_i)$ if $i\neq j$, and $\lambda'_j(x_j)\defeq \lambda_j(x_j)+\varepsilon_j$, where $\varepsilon_j$ equals either $-2$ or $2$ when $t_j$ equals either $y_j$ or $z_j$ respectively. Note that $\chi'^{-1}$ is equal to the character of $I$ acting on $t_je_{\lambda,k}\in \gr^1(\pi^\vee)$. Thus, if $t_je_{\lambda,k}\neq0$, then dually the $\chi'$-isotypic subspace of $\pi[\fm^2]/\pi[\m]$ would be nonzero. But this contradicts condition (ii) above. Hence $e_{\lambda,k}$ is annihilated by the whole ideal $\mathfrak{a}(\lambda)$ and we are done.
\end{proof}

\begin{cor}\label{cor:cycle-pi'}
Let $\pi'$ be a subrepresentation of $\pi$ and $\mathscr{P}'\subset \mathscr{P}$ be the subset corresponding to the characters {\upshape(}without multiplicities{\upshape)} of $H$ appearing in $\pi'^{I_1}$. Then $\gr(\pi'^{\vee})$ {\upshape(}with the $\fm$-adic filtration on $\pi'^{\vee}${\upshape)} is a quotient of $\big(\bigoplus_{\lambda\in\mathscr{P}'}\chi_{\lambda}^{-1}\otimes R/\mathfrak{a}(\lambda)\big)^{\oplus r}$.
\end{cor}
\begin{proof}
We have a natural quotient map $\pi^{\vee}\twoheadrightarrow \pi'^{\vee}$ which induces a quotient map $\gr(\pi^{\vee})\twoheadrightarrow \gr(\pi'^{\vee})$. It is enough to prove that the composition
\[\big(\bigoplus_{\lambda\in\mathscr{P}'}\chi_{\lambda}^{-1}\otimes R/\mathfrak{a}(\lambda)\big)^{\oplus r}\hookrightarrow \big(\bigoplus_{\lambda\in\mathscr{P}}\chi_{\lambda}^{-1}\otimes R/\mathfrak{a}(\lambda)\big)^{\oplus r}\twoheadrightarrow \gr(\pi^{\vee})\twoheadrightarrow \gr(\pi'^{\vee})\]
is surjective (where the second map is the surjection of Theorem \ref{thm:cycle-pi}). The assumption implies that it is surjective on $\gr^0(-)$, and we conclude using that $\gr(\pi'^{\vee})$ is generated by $\gr^0(\pi'^{\vee})$ as a $\gr(\Lambda)$-module. 
\end{proof}

If $N$ is a finitely generated $\overline{R}$-module and $\q$ a minimal prime ideal of $\overline{R}$, recall that $m_{\q}(N)\in \Z_{\geq 0}$ denotes the multiplicity of $N$ at $\q$, see \eqref{def:multiatp}.

\begin{thm}\label{thm:upperbound}
We have $\dim_{\F}V_{\GL_2}(\pi)=\dim_{\F\ppar{X}}D_\xi^\vee(\pi)\leq m_{\mathfrak{p}_0}(\gr(\pi^\vee))\leq 2^fr$, where the minimal ideal $\mathfrak{p}_0$ is as in \S\ref{upperb}.
\end{thm}
\begin{proof}
This is a direct consequence of (\ref{VG}), of Corollary \ref{cor:upperbound}, of Theorem \ref{thm:cycle-pi} and of Corollary \ref{cor:dim-D(pi)}, noting that, if $y_j\in {\mathfrak a}(\lambda)$ for some $j\in \{0,\dots,f-1\}$, then $m_{\mathfrak{p}_0}(R/{\mathfrak a}(\lambda))=0$ (as $y_j\notin \mathfrak{p}_0$), and if $y_j\notin {\mathfrak a}(\lambda)\ \forall\ j\in \{0,\dots,f-1\}$, then $m_{\mathfrak{p}_0}(R/{\mathfrak a}(\lambda))=1$ (as $(R/{\mathfrak a}(\lambda))[(y_0\cdots y_{f-1})^{-1}]\simeq \F[y_0,\dots,y_{f-1}][(y_0\cdots y_{f-1})^{-1}]\simeq \gr(A)$).
\end{proof}

Combined with the results of \S\ref{tensorinduction}, we can deduce the following important corollary.

\begin{cor}\label{maintensorind}
Assume moreover that $\rhobar$ is semisimple, satisfies the genericity condition {\upshape(\ref{eq:6})} and that condition (i) above can be enhanced into an isomorphism of diagrams $(\pi^{I_1}\hookrightarrow \pi^{K_1}) \simeq D(\rhobar)^{\oplus r}$, where $D(\rhobar)$ is as in {\upshape(\ref{diagramchoice})}. Then we have an isomorphism of representations of $I_{\Qp}$:
\[V_{\GL_2}(\pi)\vert_{I_{\Qp}}\cong \big(\ind_{K}^{\otimes \Qp}\!(\rhobar)\big)\vert_{I_{\Qp}}^{\oplus r}.\]
In particular we have $\dim_{\F}V_{\GL_2}(\pi)=\dim_{\F\ppar{X}}D_\xi^\vee(\pi)= m_{\mathfrak{p}_0}(\pi^\vee) = 2^fr$. If moreover the constants $\nu_i$ associated to $D(\rhobar\otimes \chi)$ {\upshape(}$\chi$ as in \S\ref{lowerstatement}{\upshape)} at the beginning of {\upshape\cite[\S6]{breuil-IL}} are as in {\upshape\cite[Thm.6.4]{breuil-IL}}, then we have an isomorphism of representations of $\gp$:
\[V_{\GL_2}(\pi)\cong \big(\ind_{K}^{\otimes \Qp}\!(\rhobar)\big)^{\oplus r}.\]
\end{cor}
\begin{proof}
It follows from Theorem \ref{thm:main-tensor-ind} and Theorem \ref{thm:upperbound} as $\dim_{\F}\!\big(\!\ind_{K}^{\otimes \Qp}\!(\rhobar)\big)^{\!\oplus r}\!\!\!\!=2^fr$.
\end{proof}

It is also worth mentioning the following corollary of Theorem \ref{thm:cycle-pi}.

\begin{cor}\label{cor:mul=4^f}
We have $\sum_{\q} m_{\mathfrak{q}}(\gr(\pi^\vee))\leq 4^fr$, where the sum is taken over all minimal prime ideals $\mathfrak{q}$ of $\overline{R}$. 
\end{cor} 
\begin{proof}
By an easy computation, we have $\sum_{\q}m_{\mathfrak{q}}(R/\mathfrak{a}(\lambda))=2^{|\mathcal{A}(\lambda)|}$ (see \eqref{eq:A(lambda)} for $\mathcal{A}(\lambda)$). Thus the result follows from Proposition \ref{prop:sum=4^f} and Theorem \ref{thm:cycle-pi}.
\end{proof}

\begin{rem}\label{rem:yongquan}
(i) It seems possible to us that the surjection in Theorem \ref{thm:cycle-pi} could actually be an isomorphism, at least for $\pi$ coming from the global theory as in \S\ref{globalp} below. Note that such an isomorphism implies in particular ${\rm E}^i_{\gr(\Lambda)}(\gr(\pi^\vee))\ne 0$ if and only if $i=2f$ (i.e.\ the $\gr(\Lambda)$-module $\gr(\pi^\vee)$ is Cohen--Macaulay of grade $2f$), which in turns implies ${\rm E}^i_{\Lambda}(\pi^\vee)\ne 0$ if and only if $i=2f$ (use \cite[Cor.6.3]{Venjakob} and the similar result with $\gr(\Lambda)$ instead of $\Lambda$, the first statement in \cite[Thm.3.21(ii)]{Venjakob} and \cite[Thm.I.7.2.11(1)]{LiOy}).
Note moreover that by \cite[Prop.A.8]{HuWang2} we know that $\pi^\vee$ is Cohen--Macaulay for $\pi$ coming from the global theory in the so-called {\it minimal case} (see \S\ref{lcresults}), but we don't know this for $\gr(\pi^\vee)$ without extra assumptions (e.g.~that the surjection of Theorem \ref{thm:cycle-pi} is an isomorphism).
\\
(ii) It is worth recalling here the following implications that we have seen. Consider the following conditions on an admissible smooth representation $\pi$ of $\GL_2(K)$ over $\F$ with a central character:
\begin{enumerate}[(a)]
\item $[\pi[\m^3]:\chi]=[\pi[\m]:\chi]$ for every character $\chi:I\rightarrow \F^\times$ appearing in $\pi[\m]$;
\item $\gr(\pi^\vee)$ is killed by $J$, where $\gr(\pi^\vee)$ is computed with the $\m$-adic filtration on $\pi^\vee$;
\item $\gr(\pi^\vee)$ is killed by some power of $J$, where $\gr(\pi^\vee)$ is computed with any good filtration on the $\Lambda$-module $\pi^\vee$;
\item $\pi$ is in the category $\mathcal C$ of \S\ref{multivariablepsi}.
\end{enumerate}
Then we have (a) $\Rightarrow$ (b) $\Rightarrow$ (c) $\Rightarrow$ (d). We suspect that every implication is strict.
\end{rem}

\subsubsection{Examples}

We completely compute the $\gr(\F\bbra{I/Z_1})$-module $\gr(V^{\vee})$ for certain irreducible admissible smooth representations $V$ of $\GL_2(K)$ over $\F$ (with $V^\vee$ endowed with the $\m$-adic filtration). We assume $p\geq 5$ in this section.

We keep the previous notation. If $V$ is a smooth representation of $I_1/Z_1$ over $\F$, we write $\gr(V^\vee)$ for the graded module associated to the $\m$-adic filtration on $V^\vee$.

\begin{lem}\label{grgr}
Let $V$ be a smooth representation of $I_1/Z_1$ over $\F$ such that $V\vert_{N_0}$ is admissible as a representation of $N_0$ and such that the natural map $\gr_{\m_{N_0}}(V^\vee)\rightarrow \gr(V^\vee)$ {\upshape(}induced by the inclusions $\m_{N_0}^nV^\vee\subseteq \m^nV^\vee$ for $n\geq 0${\upshape)} is surjective. Then this map is an isomorphism.
\end{lem}
\begin{proof}
Since $V\vert_{N_0}^\vee$ is a finite type $\F\bbra{N_0}$-module by assumption, it is a complete filtered $\F\bbra{N_0}$-module for the $\m_{N_0}$-adic filtration. As all the maps $\m_{N_0}^nV^\vee/\m_{N_0}^{n+1}V^\vee \rightarrow \m^nV^\vee/\m^{n+1}V^\vee$ are surjective, any element in $v\in \m^nV^\vee$ can be written $v=v_0+w$, where $v_0\in \sum_{m\geq n}\m_{N_0}^mV^\vee= \m_{N_0}^nV^\vee$ (as $V\vert_{N_0}^\vee$ is complete) and $w\in \cap_{m\geq n}\m^mV^\vee = 0$ (as the $\m$-adic filtration is separated since $V$ is smooth). Thus the inclusion $\m_{N_0}^nV^\vee\subseteq \m^nV^\vee$ is an equality for $n\geq 0$, and we are done.
\end{proof}

The following two lemmas are motivated by \cite[Prop.7.1, Prop.7.2]{paskunas-extensions}. We consider the finite group $H$ as subgroup of $I$ via the Teichm\"uller lift.
\begin{lem}\label{lem:Pas-bis}
Let $V$ be an admissible smooth representation of $I/Z_1$ over $\F$. Assume that $V|_{HN_0}$ is isomorphic to an injective envelope of some character $\chi$ in the category of smooth representations of $HN_0$ over $\F$ (so in particular $\dim_{\F}V^{N_0}=1$). Then $\Ext^1_{I/Z_1}(\chi\alpha_j^{-1},V)=0$ for any $0\leq j\leq f-1$.
\end{lem}
\begin{proof}
Consider an extension class in $\Ext^1_{I/Z_1}(\chi\alpha_j^{-1},V)$ represented by $0\ra V\ra V'\ra \chi\alpha_j^{-1}\ra0$. By assumption on $V$, this extension splits when restricted to $HN_0$, hence we may find $v'\in V'\backslash V$ on which $HN_0$ acts via $\chi\alpha_j^{-1}$ (in particular $v'\in V'^{N_0}$). Notice that $(g-1)v'\in V$ for any $g\in I_1$. Let $v\in V^{N_0}$ be a nonzero vector so that $V^{N_0}=\F v$ by assumption. 

First take $g\in \smatr{1+p\cO_K}00{1+p\cO_K}$. It is easy to see that $(g-1)v'$ is again fixed by $N_0$ and $H$ acts on it via $\chi\alpha_j^{-1}$. But, by assumption $V^{N_0}$ is $1$-dimensional on which $H$ acts via $\chi$, thus we must have $(g-1)v'=0$. We deduce that $v'$ is fixed by $I_1\cap B(\cO_K)$.

We claim that $v'$ is fixed by $N_1^-\defeq \smatr{1}0{p\cO_K}1$. This will imply that $v'$ is fixed by $I_1$ by the Iwahori decomposition, and consequently $V'$ splits as an $I$-representation. Let $k\geq 1$ be the smallest integer such that $v'$ is fixed by $N_{k}^-\defeq \smatr{1}0{p^k\cO_K}1$; such an integer always exists, as $V$ is a smooth representation of $I$. Suppose $k\geq 2$ and take $g\in N_{k-1}^-$. Using the matrix identity (see \cite[Eq.(14)]{paskunas-extensions})
\[\smatr{1}{b}01\smatr{1}0c1=\smatr{1}0{c(1+bc)^{-1}}1\smatr{1+bc}{b}0{(1+bc)^{-1}}\]
and the fact that $v'$ is fixed by $\smatr{1+p\cO_K}{\cO_K}{p^{k}\cO_K}{1+p\cO_K}$, 
 one checks that $(g-1)v'\in V^{N_0}$. Consequently, $\F v\oplus \F v'$ gives rise to an extension in $\Ext^{1}_{HN_{k-1}^-}(\chi\alpha_j^{-1},\chi)$ which is nonsplit by the choice of $k$. But, as in \cite[Lemma 5.6]{paskunas-extensions}, one shows that $\Ext^1_{HN_{k-1}^-}(\chi',\chi)\neq 0$ if and only if $\chi'=\chi\alpha_i$ for some $0\leq i\leq f-1$. Indeed, after conjugating by $\smatr{p^{k-2}}001$, we are reduced to the case $k=2$, in which case the result is proved by determining the $H$-action on $\Hom(N_{1}^-,\F)$ as in \cite[Lemma 5.3]{paskunas-extensions} (see the proof of \cite[Prop.5.1]{BP} for the computation). This finishes the proof as $\chi\alpha_j^{-1}\neq \chi\alpha_{i}$ for any $0\leq i,j\leq f-1$ (as $p\geq 5$).
\end{proof}

\begin{lem}\label{lem:Pas}
Let $V$ be an admissible smooth representation of $I/Z_1$ over $\F$. Assume that $V|_{HN_0}$ is isomorphic to an injective envelope of some character $\chi$ in the category of smooth representations of $HN_0$ over $\F$ (so in particular $\dim_{\F}V^{N_0}=1$). Then we have an isomorphism of $\gr(\F\bbra{I/Z_1})$-modules:
\[\gr(V^{\vee})\cong \chi^{-1}\otimes R/(z_0,\dots,z_{f-1}).\]
\end{lem}
\begin{proof}
By assumption, $V[\fm] = V[\fm_{N_0}]$ is one-dimensional and isomorphic to $\chi$, hence we may view $\gr(V^{\vee})$ as a cyclic module over $\gr(\Lambda)$ generated by $e_\chi\in \gr^0(V^\vee)=V[\fm]^\vee$, where $H$ acts on $e_\chi$ by $\chi^{-1}$. Let $\mathfrak{a}\subseteq \gr(\Lambda)$ be the annihilator of $e_{\chi}$.

We first prove that $z_j\in \mathfrak{a}$ for $0\leq j\leq f-1$. Since $H$ acts on $z_j$ via $\alpha_j^{-1}$ (see just above \S\ref{combi}), to prove $z_je_{\chi}=0$ in $\gr^1(V^{\vee})$ it is equivalent to prove that
\[\Hom_{H}(\chi\alpha_j,V[\fm^2]/V[\fm])=0 \ \ \forall\ j\in \{0,\dots,f-1\}.\]
If not, then $V$ would admit a subrepresentation isomorphic to $E_{\chi,\chi\alpha_j}$ (for some $j$), where $E_{\chi,\chi\alpha_j}$ denotes the unique $I/Z_1$-representation which is a nonsplit extension of $\chi\alpha_j$ by $\chi$. But by \cite[Lemma 6.1.1(ii)]{BHHMS1} (after conjugating by the element $\smatr01p0$), $N_0$ acts trivially on $E_{\chi,\chi\alpha_j}$, which implies $\dim_{\F}V[\m_{N_0}]\geq 2$, a contradiction to the assumption on $V$. 

Using \cite[Lemma 6.1.1(ii)]{BHHMS1}, we then deduce an embedding 
\begin{equation}\label{eq:V[m2]-bis}
V[\m^2]/V[\m]\hookrightarrow \oplus_{j=0}^{f-1}\chi\alpha_j^{-1}.\end{equation}
On the other hand, since $\Hom_I(\chi\alpha_j^{-1},V)=0$, we deduce from Lemma \ref{lem:Pas-bis} that 
\[\Hom_{I}(\chi\alpha_j^{-1},V[\m^2]/V[\m])=\Hom_{I}(\chi\alpha_j^{-1},V/V[\m])\simto\Ext^1_{I/Z_1}(\chi\alpha_j^{-1},\chi)\] 
which have dimension $1$ over $\F$ by \cite[Lemma 6.1.1(ii)]{BHHMS1} again. Combining this with \eqref{eq:V[m2]-bis}, we obtain 
\begin{equation}\label{eq:V[m2]}
0\ra \chi\ra V[\fm^2]\ra \oplus_{j=0}^{f-1}\chi\alpha_j^{-1}\ra0.
\end{equation}
and that $V[{\m^2}]=V[\m_{N_0}^2]$. 

Next, we prove that $\Ext^1_{I/Z_1}(\chi,E_{\chi,\chi\alpha_j^{-1}})$ has dimension $1$ over $\F$ for any $0\leq j\leq f-1$. 
A straightforward d\'evissage using $\Ext^1_{I/Z_1}(\chi,\chi)=0$ and $\dim_{\F}\Ext^1_{I/Z_1}(\chi,\chi\alpha_j^{-1})=1$ (see \cite[Lemma 6.1.1(ii)]{BHHMS1}) yields $\dim_{\F}\Ext^1_{I/Z_1}(\chi,E_{\chi,\chi\alpha_j^{-1}})\leq 1$. So it suffices to explicitly construct a nonzero element in this space, as follows. Let $\mathcal{E}_j\defeq \F v_0\oplus \F v_1\oplus \F v_2 $ equipped with the action of $I/Z_1$ determined by:
\begin{itemize}
\item $H$ acts on $v_0$, $v_1$, $v_2$ by $\chi,$ $\chi\alpha_j^{-1}$, $\chi$ respectively;
\item if $g=\smatr{1+pa}{b}{pc}{1+pd}\in I_1$, then 
\[gv_0=v_0,\ \ gv_1=v_1+\sigma_j(\overline{b})v_0,\]
\[gv_2=v_2+\sigma_j(\overline{c})v_1+\frac{1}{2}\big(\sigma_j(\overline{a})-\sigma_j(\overline{d})+\sigma_j(\overline{b}\overline{c})\big)v_0.\]
\end{itemize} 
One easily checks that $\mathcal{E}_j$ is well defined and yields the desired nonsplit extension class in $\Ext^1_{I/Z_1}(\chi,E_{\chi,\chi\alpha_j^{-1}})$. Moreover one also checks that $\mathcal{E}_j^{N_0}=\F v_0\oplus \F v_2$.
 
We prove that $h_j\in \mathfrak{a}$ for $0\leq j\leq f-1$. Since $\Ext^1_{I/Z_1}(\chi,\chi)=0$,
 the sequence \eqref{eq:V[m2]} induces an embedding
\[\Ext^1_{I/Z_1}(\chi,V[\fm^2])\hookrightarrow \Ext^1_{I/Z_1}(\chi,\oplus_{j=0}^{f-1}\chi\alpha_j^{-1}).\]
Note that the right-hand side has dimension $f$ over $\F$. Since $\mathcal{E}_j/\chi$ is nonzero in $\Ext^1_{I/Z_1}(\chi,\chi\alpha_j^{-1})$ for $0\leq j\leq f-1$, we easily see that the above embedding is actually an isomorphism and that $\Ext^1_{I/Z_1}(\chi,V[\fm^2])$ is spanned by the $\mathcal{E}_j$'s. By the last statement of the previous paragraph, if an extension $\mathcal{E}\in \Ext^1_{I/Z_1}(\chi,V[\fm^2])$ is nonzero then $\dim_{\F}\mathcal{E}^{N_0}\geq 2$. Since $\dim_{\F}V^{N_0}=1$ by assumption, we see that there exists no embedding $\mathcal{E}\hookrightarrow V$. From (\ref{eq:V[m2]}) we then easily deduce
\[\Hom_{H}(\chi,V[\fm^3]/V[\fm^2])=0.\]
Since $H$ acts trivially on $h_j$ and $h_je_\chi\in \gr^2(V^\vee)\simeq (V[\fm^3]/V[\fm^2])^\vee$, we thus must have $h_je_\chi=0$, i.e.\ $h_j\in \mathfrak{a}$ for $0\leq j\leq f-1$. This proves the claim.

We deduce a surjection $\gr(\Lambda)/(z_j,h_j,\ 0\leq j\leq f-1)\twoheadrightarrow \gr(V^\vee)$. 
As the left-hand side is $\F[y_0,\dots,y_{f-1}]\simeq \gr(\F\bbra{N_0})$ and $(V\vert_{N_0})^\vee\simeq \F\bbra{N_0}$ from the assumption, we obtain a surjection $\gr_{\m_{N_0}}(V^\vee)\twoheadrightarrow \gr(V^\vee)$. By Lemma \ref{grgr} this surjection is an isomorphism (and hence $\mathfrak{a}=(z_j,h_j, 0\leq j\leq f-1)$). This finishes the proof.
\end{proof}

If $\chi = \chi_1 \otimes \chi_2$ is a character of $H$ or of $T(K)$, recall $\chi^s = \chi_2 \otimes \chi_1$.

\begin{prop}\label{exemplesgr}
Let $V$ be an irreducible smooth $\F$-representation of $\GL_2(K)$ with a central character.
\begin{enumerate}
\item If $V\cong \psi\circ\det$ for some smooth character $\psi:K^{\times}\ra\F^{\times}$, then $\gr(V^{\vee})\cong (\psi\otimes\psi)^{-1}\otimes\F$, where $\psi\otimes\psi$ is viewed as a character of $H$.
\item If $V\cong \Ind_{B(K)}^{\GL_2(K)}\chi$ for some smooth character $\chi:T(K)\ra \F^{\times}$, then
  \[\gr(V^{\vee})\cong \big((\chi^s|_{H})^{-1}\otimes R/(z_0,\dots,z_{f-1})\big) \oplus
  \big((\chi|_{H})^{-1}\otimes R/(y_0,\dots,y_{f-1})\big).\]
\item If $V\cong (\Ind_{B(K)}^{\GL_2(K)}1)/1$ is the special series, then $\gr(V^{\vee})\!\cong \! R/(y_iz_j,0\leq i,j\leq f-1)$.
\item Assume $K=\Q_p$. If $V$ is supersingular, i.e.\ isomorphic to
  $(\cInd_{\GL_2(\Z_p)\Q_p^{\times}}^{\GL_2(\Q_p)}\sigma)/T$ for some Serre weight $\sigma$
  {\upshape(}recall that $\cInd$ here means compact induction and that
  $\End_{\GL_2(\Q_p)}(\cInd_{\GL_2(\Z_p)\Q_p^{\times}}^{\GL_2(\Q_p)}\sigma)\simeq \F[T]${\upshape)}, then
  \[\gr(V^{\vee})\cong \big(\chi_{\sigma}^{-1}\otimes R/(y_0z_0)\big)\oplus
  \big((\chi_{\sigma}^{s})^{-1}\otimes R/(y_0z_0)\big),\]
  where $\chi_\sigma$ is the action of $H$ on $\sigma^{I_1}$.
\end{enumerate}
\end{prop}
\begin{proof}
(i) It is trivial.\\
(ii) The restriction of $V$ to $I$ admits a decomposition
\begin{equation}\label{eq:V=PS}V|_{I}\cong \Ind_{I\cap B(K)}^{I}\chi\oplus \Ind_{I\cap B^-(K)}^{I}\chi^s,\end{equation}
(cf.\ the proof of \cite[Prop.11.1]{paskunas-extensions}). By \emph{loc.cit.}, when restricted to $HN_0$, $\Ind_{I\cap B^-(K)}^I\!\chi^s$ is an injective envelope of $\chi^s$ in the category of smooth representations of $HN_0$ over $\F$, hence \[\gr((\Ind_{I\cap B^-(K)}^I\chi^s)^{\vee})\cong (\chi^s|_{H})^{-1}\otimes R/(z_0,\dots,z_{f-1})\]
 by Lemma \ref{lem:Pas}. 
One handles the other direct summand by taking conjugation by the element $\smatr01p0$.\\
(iii) By assumption we have a short exact sequence $0\ra 1\ra \Ind_{B(K)}^{\GL_2(K)}1\ra V\ra0$. Write $W=(\Ind_{B(K)}^{\GL_2(K)}1)|_I$ and decompose $W=W_1\oplus W_2$ as in \eqref{eq:V=PS}. The image of $1\hookrightarrow W$ is equal to the subspace of constant functions, hence the composition $1\hookrightarrow W\twoheadrightarrow W_i$ is nonzero for $i\in\{1,2\}$. Consequently, the dual morphism $\gr(W_i^{\vee})\ra \gr(1^{\vee})$ is also nonzero, and using (ii) (applied to $W$) we obtain an exact sequence of $\gr(\F\bbra{I/Z_1})$-modules
\begin{equation}\label{eq:gr-Sp} 0\ra R/(y_iz_j,0\leq i,j\leq f-1)\ra \gr(W_1^{\vee})\oplus\gr(W_2^{\vee})\ra \gr(1^{\vee})\ra0.\end{equation}
Denote by $F$ the induced filtration on $V^{\vee}$ from the $\m$-adic filtration on $W^{\vee}$. By \eqref{eq:gr-Sp} we have an isomorphism $\gr_F(V^{\vee})\cong R/(y_iz_j,0\leq i,j\leq f-1)$. To finish the proof, it suffices to prove that $F$ coincides with the $\m$-adic filtration on $V^{\vee}$, or equivalently the inclusion $\m^{n}V^{\vee}\subset \m^{n}W^{\vee}\cap V^{\vee}$ (for $n\geq0$) is an equality. As in the proof of Lemma \ref{grgr} it suffices to prove that the induced graded morphism $\gr_{\m}(V^{\vee})\ra \gr_F(V^{\vee})$ is surjective. But, $\gr_F(V^{\vee})$ is generated by $\gr_F^0(V^{\vee})$, so it suffices to show that $\gr_{\m}^0(V^{\vee})\ra \gr_F^0(V^{\vee})$ is surjective, which follows from \eqref{eq:gr-Sp} and the exact sequence 
 \[\gr_{\m}^0(V^{\vee})\ra \gr^0_{\m}(W^{\vee})\ra \gr^0_{\m}(1^{\vee})\ra0\]
induced by $0\ra 1^{I_1}\ra W^{I_1}\ra V^{I_1}$ (this sequence is actually right exact but we don't need this fact).\\
 (iv) The proof is analogous to (iii), using \cite[Thm.1.2]{paskunas-extensions} together with \cite[Prop.4.7]{paskunas-extensions}.
\end{proof}

By the classification of irreducible admissible smooth representations of $\GL_2(\Q_p)$ over $\F$, we deduce from Proposition \ref{exemplesgr} and the results of \S\ref{multivariablepsi}:

\begin{cor}\label{qp}
Let $V$ be an admissible smooth representation of $\GL_2(\Q_p)$ over $\F$ which has a central character and is of finite length. Then there is an integer $n\geq 0$ such that $\gr(V^\vee)$ is annihilated by $J^n$. In particular $V$ is in the category $\mathcal C$ of \S\ref{multivariablepsi}.
\end{cor}

Finally, we give the first example of an explicit $D_A(\pi)$ for arbitrary $f \ge 1$.

\begin{prop}\label{prop:D_A-principal-series}
  Suppose $\pi=\Ind_{B(K)}^{\GL_2(K)}\chi$ is a principal series for some smooth character $\chi = \chi_1 \otimes \chi_2$.
  Then $\pi$ lies in category $\cC$ and $D_A(\pi) = D_A(\pi)^{\et}$ is \'etale and free of rank 1.
  More precisely, let $\kappa \in \pi^\vee$ be the element sending $f \in \Ind_{B(K)}^{\GL_2(K)}\chi$ to $f(\smatr{}{1}{1}{}) \in \F$.
  Then the image of $\kappa$ in $D_A(\pi)$ is a basis of $D_A(\pi)$, and we have
  \begin{align}
    \varphi(\kappa) &= \chi_2(p)^{-1} \kappa, \label{eq:ps1} \\
    a(\kappa) &= \chi_2(a)^{-1} \kappa \text{\ \ $\forall\ a \in \cO_K^\times$.}\label{eq:ps2}
  \end{align}
\end{prop}

\begin{proof}
  Note that $\pi \in \cC$ by Proposition~\ref{exemplesgr}(ii).
  The torus $T(K)$ (hence also $H$) acts on $\kappa$ by the character $(\chi^s)^{-1}$; in particular, we get~\eqref{eq:ps2}.
  On graded pieces the map $\pi^\vee \to D_A(\pi)$ becomes the map $\gr(\pi^\vee) \to \gr(\pi^\vee)[(y_0\cdots y_{f-1})^{-1}]$ (Lemma \ref{lem:grlocal}).
  As $\kappa$ does not annihilate $\pi^{I_1}$, it induces a nonzero element of $\gr^0(\pi^\vee) = (\pi^{I_1})^\vee$, which is in fact a $\gr(A)$-basis of $\gr(\pi^\vee)[(y_0\cdots y_{f-1})^{-1}]$ by Proposition~\ref{exemplesgr}(ii).
  (If $\chi = \chi^s$ the argument still works because $\kappa$ annihilates the first direct summand in the Mackey decomposition \eqref{eq:V=PS}.)
  By the proof of \cite[Thm.I.5.7]{LiOy} it follows that $D_A(\pi) = A\kappa$.
  As $D_A(\pi)$ is a projective $A$-module and $\gr(D_A(\pi)) \ne 0$, it follows that $D_A(\pi)$ is free of rank 1 with basis $\kappa$.

  It remains to show~\eqref{eq:ps1}.
  First, from the definitions we see that $\psi(\kappa) = \chi_2(p)\kappa$.
  This implies that the $(\psi,\cO_K^\times)$-module $D_A(\pi)$ is \'etale and so by the previous sentence it becomes an \'etale $(\varphi,\cO_K^\times)$-module, cf.\ \eqref{mapbeta}.
  Say $\varphi(\kappa) = a\kappa$ for some $a \in A^\times$.
  As the actions of $\varphi$ and $\cO_K^\times$ commute, equation~\eqref{eq:ps2} and Corollary~\ref{cor:OK-times-invariants} imply that $a \in \F^\times$.
  We deduce \eqref{eq:ps1}.
\end{proof}

\subsubsection{Characteristic cycles}\label{charcycl}

We define the characteristic cycle of a finitely generated filtered $\Lambda$-module $M$ such that $\gr(M)$ is annihilated by a power of $J$ and prove an important property (Theorem \ref{prop:cycle-Ext}).
 
Recall from \S\ref{upperb} that the minimal prime ideals of $\overline{R}=R/(y_jz_j,0\leq j\leq f-1)$ are the $(y_i,z_j, i\in\cJ, j\notin\cJ)$ with $\cJ$ a subset of $\{0,\dots,f-1\}$. 

\begin{definit}\label{Zdef}
Let $N$ be a finitely generated module over $\gr(\Lambda)$ which is annihilated by some power of $J$. We define the \emph{characteristic cycle} of $N$, denoted by $\mathcal{Z}(N)$\footnote{A more standard notation is $\mathcal{Z}_f(N)$, where $f$ indicates the dimension of the cycles.} as follows:
\[\mathcal{Z}(N)\defeq \sum_{\q}m_{\q}(N)\q\ \in \oplus_{\q}\Z_{\geq 0}\q,\]
where $\q$ runs over all minimal prime ideals of $\overline{R}$.
\end{definit} 
 
\begin{lem}\label{lem:Z-additive-gr}
Let $n\geq 0$. If $0\ra N_1\ra N\ra N_2\ra 0$ is a short exact sequence of finitely generated $\gr(\Lambda)/J^n$-modules, then $\mathcal{Z}(N)=\mathcal{Z}(N_1)+\mathcal{Z}(N_2)$ in $\oplus_{\q}\Z_{\geq 0}\q$. 
\end{lem} 
\begin{proof}
It is a direct consequence of Lemma \ref{lem:m-additive}.
\end{proof}
 
Let $M$ be a finitely generated $\Lambda$-module which is equipped with a good filtration $F\defeq \{F_nM : n\in\Z\}$ (in the sense of \cite[\S I.5]{LiOy}) such that $\gr_{F}(M)$ is annihilated by some power of $J$. Recall that this condition doesn't depend on the choice of the good filtration $F$ (see just before Proposition \ref{prop:finiteness}) and that $\gr_{F}(M)$ is also finitely generated over $\gr(\Lambda)$ (\cite[Lemma I.5.4]{LiOy}).

\begin{lem}\label{indgood}
If $F, F'$ are two such good filtrations on $M$, then 
\[\mathcal{Z}(\gr_{F}(M))=\mathcal{Z}(\gr_{F'}(M)).\]
\end{lem}
\begin{proof}
The proof is (almost) the same as in \cite[\S4]{Bj89}. We recall it for the convenience of the reader. Since $F$ and $F'$ are equivalent by \cite[Lemma I.5.3]{LiOy}, we may find $c\in\Z_{\geq 0}$ such that 
\[F_{n-c}M\subset F'_nM\subset F_{n+c}M,\ \ \forall\ n\in\Z.\]
For $i\in \{-c,-c+1,\dots,c\}$ define a sequence of filtrations $F^{(i)}=\{F^{(i)}_nM : n\in\Z\}$ on $M$ by
\[F^{(i)}_nM\defeq F_{n+i}M\cap F_n'M.\]
It is clear that $F^{(-c)}=F[-c]$ and $F^{(c)}=F'$, where $F[-c]$ denotes the shifted filtration $F[-c]_n\defeq F_{n-c}$, $n\in \Z$. Hence it suffices to show that each $F^{(i)}$ is a good filtration on $M$ such that
\begin{equation}\label{eq:Bjork-Z}
\mathcal{Z}(\gr_{F^{(i)}}(M))=\mathcal{Z}(\gr_{F^{(i+1)}}(M)).
\end{equation}
Put for $-c\leq i\leq c$:
\[T_i\defeq \bigoplus_{n\in \Z}(F_{n+i}M\cap F'_{n}M)/(F_{n+i}M\cap F'_{n-1}M),\]
\[S_i\defeq \bigoplus_{n\in\Z}(F_{n+i+1}M\cap F_{n}'M)/(F_{n+i}M\cap F'_{n}M).\]
Since $T_i$ is a $\gr(\Lambda)$-submodule of $\gr_{F'}(M)$ and $S_i$ is a $\gr(\Lambda)$-submodule of $\gr_{F}(M)[i+1]$, both $T_i$ and $S_i$ are finitely generated $\gr(\Lambda)$-modules and are annihilated by some power of $J$. Moreover, one checks that there are short exact sequences of $\gr(\Lambda)$-modules (annihilated by some power of $J$):
\[0\ra T_i\ra \gr_{F^{(i+1)}}(M)\ra S_i\ra0,\]
\[0\ra S_i[-1]\ra \gr_{F^{(i)}}(M)\ra T_i\ra0.\]
Hence, $\gr_{F^{(i)}}(M)$ is also finitely generated over $\gr(\Lambda)$ and annihilated by a power of $J$. Consequently, $F^{(i)}$ is a good filtration on $M$ by \cite[Thm.I.5.7]{LiOy} and \eqref{eq:Bjork-Z} follows from Lemma \ref{lem:Z-additive-gr}.
\end{proof}

Thanks to Lemma \ref{indgood}, we can define $m_{\q}(M)$ to be $m_{\q}(\gr_{F}(M))$ and $\cZ(M)$ to be $\mathcal{Z}(\gr_{F}(M))$ for any minimal prime ideal $\q$ of $\overline R$ and any good filtration $F$ on $M$.

\begin{lem}\label{lem:Z-additive}
Let $M$ be as above and let $0\ra M_1\ra M\ra M_2\ra 0$ be an exact sequence of $\Lambda$-modules. Then we have in $\oplus_{\q}\Z_{\geq 0}\q$:
\[\mathcal{Z}(M)=\mathcal{Z}(M_1)+\mathcal{Z}(M_2).\]
\end{lem}
\begin{proof}
We may equip $M_1$ (resp.\ $M_2$) with the induced filtration (resp.\ quotient filtration) from the one of $M$, which are automatically good by \cite[Cor.I.5.5(1)]{LiOy} and \cite[Rem.I.5.2(2)]{LiOy}. Moreover the sequence $0\ra \gr(M_1)\ra \gr(M)\ra\gr(M_2)\ra0$ is again exact. In particular, both $\gr(M_1)$ and $\gr(M_2)$ are finitely generated $\gr(\Lambda)$-modules annihilated by some power of $J$, and the result follows from Lemma \ref{lem:Z-additive-gr}.
\end{proof}

If $M$ is a finitely generated $\Lambda$-module, recall from \cite[Def.III.2.1.1]{LiOy} that the grade of $M$ is by definition the smallest integer $j_\Lambda(M)\ge 0$ such that ${\rm E}_\Lambda^{j_{\Lambda}(M)}(M)\ne 0$ (with $j_{\Lambda}(M)\defeq +\infty$ if ${\rm E}_\Lambda^{j}(M)=0$ for all $j\geq 0$). For a good filtration $F$ on $M$, we define similarly the grade $j_{\gr(\Lambda)}(\gr_F(M))$ of the $\gr(\Lambda)$-module $\gr_F(M)$. By \cite[Thm.III.2.5.2]{LiOy} we have $j_{\gr(\Lambda)}(\gr_F(M))=j_\Lambda(M)$ (note that $\Lambda$ is a left and right Zariski ring by \cite[Prop.II.2.2.1]{LiOy}), in particular $j_{\gr(\Lambda)}(\gr_F(M))$ doesn't depend on the good filtration $F$.

Recall that the Krull dimension $\dim_{R}(N)$ of a finitely generated module $N$ over $R$ (which is commutative) is the Krull dimension of $R/{\rm Ann}_R(N)$. For such a module $N$, by the argument in the proof of \cite[Lemma 5.1.3]{BHHMS1} applied to $A=\gr(\Lambda)$, $I=(h_0,\dots,h_{f-1})$ and with $N$ instead of $\gr_{\m}M$ there, we have
\begin{equation}\label{dimannihi}
j_{\gr(\Lambda)}(N)=\dim(I_1/Z_1)-\dim_{R}(N).
\end{equation}
Now, for $M$ as above, assume that $\gr_F(M)$ is annihilated by a power of $J$. Then applying (\ref{dimannihi}) to the $\overline R$-modules $N=J^i\gr_F(M)/J^{i+1}\gr_F(M)$ for $i\geq 0$ and by an obvious d\'evissage using \cite[Lemma III.2.1.2(1)]{LiOy}, we deduce
\begin{equation}\label{>=2f}
j_\Lambda(M)\geq \dim(I_1/Z_1)-\dim(\overline R)=3f-f=2f.
\end{equation}
Moreover, by the same d\'evissage using \cite[Cor.III.2.1.6]{LiOy} (note that all assumptions are satisfied since $\gr(\Lambda)$ is Auslander regular) and (\ref{dimannihi}), we deduce that if $j_{\Lambda}(M)=j_{\gr(\Lambda)}(\gr_F(M))>2f$, then we have $\dim_{R}(J^i\gr_F(M)/J^{i+1}\gr_F(M))<f$ for all $i$, hence ${\mathcal Z}(J^i\gr_F(M)/J^{i+1}\gr_F(M))=0$ for all $i\geq 0$ and $\mathcal{Z}(M)=0$ (see \eqref{def:multiatp}). 

\begin{thm}\label{prop:cycle-Ext}
Let $M$ be a finitely generated $\Lambda$-module such that $\gr(M)$ is annihilated by a power of $J$ for one {\upshape(}equivalently every{\upshape)} good filtration on $M$. Then $\mathcal{Z}(\mathrm{E}^{2f}_{\Lambda}(M))$ is well-defined and we have
\[\mathcal{Z}(M)=\mathcal{Z}(\mathrm{E}^{2f}_{\Lambda}(M)).\]
\end{thm}
\begin{proof}
If $j_\Lambda(M)> 2f$, then the result is trivial since both terms are $0$ by the sentence just before the proposition. So from (\ref{>=2f}) we may assume $j_\Lambda(M)= 2f$ in the rest of the proof.
 
 Choose a good filtration $F$ of $M$ so that $\cZ(M) = \cZ(\gr_{F}(M))$. We first show that the $\gr(\Lambda)$-module ${\rm E}^{2f}_{\gr(\Lambda)}(\gr_F(M))$ is also annihilated by some power of $J$. Indeed, $\gr_F(M)$ has a finite filtration whose graded pieces are annihilated by $J$, hence by d\'evissage it suffices to show that $\mathrm{E}^{2f}_{\gr(\Lambda)}(N)$ is annihilated by $J$ if $N$ is a finitely generated $\overline{R}$-module. As in the proof of Proposition \ref{prop:Ext-R/I} it is equivalent to prove the same property for $\mathrm{E}^{f}_{R}(N)$, which is obvious as $R$ is commutative.
 
As a consequence, by the first statement in Proposition \ref{prop:LvO-filtonExt} below the graded module associated to the filtration on $\mathrm{E}^{2f}_{\Lambda}(M)$ in {\it loc.cit.}\ is again finitely generated over $\gr(\Lambda)$ and annihilated by some power of $J$. Hence $\cZ(\mathrm{E}^{2f}_{\Lambda}(M))$ can be defined. By Proposition \ref{prop:LvO-filtonExt} the cokernel of the injection $\gr(\mathrm{E}^{2f}_{\Lambda}(M))\hookrightarrow \mathrm{E}^{2f}_{\gr(\Lambda)}(\gr_F(M))$ has grade $>2f$, hence its associated characteristic cycle is $0$, as explained above. From Lemma \ref{lem:Z-additive-gr} we deduce an equality of cycles
\[\mathcal{Z}\big(\gr(\mathrm{E}^{2f}_{\Lambda}(M))\big)=\mathcal{Z}\big(\mathrm{E}^{2f}_{\gr(\Lambda)}(\gr_F(M))\big).\]
Hence, we are left to show that
\[\mathcal{Z}(\gr_F(M))= \mathcal{Z}\big(\mathrm{E}^{2f}_{\gr(\Lambda)}(\gr_F(M))\big).\]
As $\gr(\Lambda)$ is an Auslander regular ring, any subquotient $N$ of $\gr_F(M)$ has grade $\geq 2f$ (by \cite[Prop.III.2.1.6]{LiOy}) and is such that $\mathrm{E}^{j}_{\gr(\Lambda)}(N)$ has grade $\geq j$ for any $j\geq 0$, so that $\EE^j_{\gr(\Lambda)}(N)$ and all its subquotients have zero cycle if $j<2f$ or if $j> 2f$ (by Lemma \ref{lem:Z-additive-gr} and the discussion before the proposition for the latter). Hence, for $n$ large enough so that $J^n$ annihilates $\gr_F(M)$, we deduce using again Lemma \ref{lem:Z-additive-gr}:
\[\mathcal{Z}\big(\mathrm{E}^{2f}_{\gr(\Lambda)}(\gr_F(M))\big)=\sum_{i=0}^{n-1}\mathcal{Z}\big(\mathrm{E}^{2f}_{\gr(\Lambda)}(J^i\gr_F(M)/J^{i+1}\gr_F(M))\big).\]
By the definition of $\mathcal Z$ and of $m_\q(N)$, see \eqref{def:multiatp}, it thus suffices to show 
\[\mathcal{Z}(N)=\mathcal{Z}(\mathrm{E}^{2f}_{\gr(\Lambda)}(N))\]
for any finitely generated $\overline{R}$-module $N$.
Using Lemma \ref{lem:devissage} it suffices to show 
\[\mathcal{Z}(N)= \mathcal{Z}(\Hom_{\overline{R}}(N,\overline{R})),\]
which is equivalent to show that for any minimal prime ideal $\q$ of $\overline{R}$, \[\mathrm{lg}_{\overline{R}_{\q}}(N_\q)=\mathrm{lg}_{\overline{R}_{\q}}(\Hom_{\overline{R}}(N,\overline{R})_{\q}).\] 
Using the isomorphism $\Hom_{\overline{R}}(N,\overline{R})_{\q}\cong \Hom_{\overline{R}_{\q}}(N_{\q},\overline{R}_{\q})$ and noting that $\overline{R}_{\q}$ is a field (being artinian, and reduced as $\overline{R}$ is), the result is clear.
\end{proof}

The first part of the following general result was used in the proof of Theorem \ref{prop:cycle-Ext}. Recall that a finitely generated $\gr(\Lambda)$-module of grade $j$ is Cohen--Macaulay if all its ${\rm E}^i_{\gr(\Lambda)}$ are $0$ when $i\ne j$.

\begin{prop}\label{prop:LvO-filtonExt}
Let $M$ be a finitely generated $\Lambda$-module of grade $j_0$ with a good filtration. Then there exists a good filtration on $\mathrm{E}^{j_0}_\Lambda(M)$ such that $\gr(\mathrm{E}^{j_0}_\Lambda(M))$ is a submodule of $\mathrm{E}^{j_0}_{\gr(\Lambda)}(\gr(M))$ and the corresponding cokernel has grade {\upshape(}over $\gr(\Lambda)${\upshape)} $\geq j_0+1$. If $\gr(M)$ is moreover Cohen--Macaulay, then
\[\gr(\mathrm{E}^{j_0}_{\Lambda}(M))\buildrel\sim\over\rightarrow \mathrm{E}^{j_0}_{\gr(\Lambda)}(\gr(M)).\]
\end{prop}
\begin{proof}
See \cite[Prop.3.1]{Bj89} and the remark following it. We explain the proof following the presentation of \cite[\S III.2.2]{LiOy}.

As in \cite[\S III.2.2]{LiOy}, we may construct a filtered free resolution of $M$
\[\cdots\ra L_j\ra L_{j-1}\ra \cdots \ra L_0\ra M\ra0\]
and taking $\EE^0_{\Lambda}(-)=\Hom_\Lambda(-,\Lambda)$ obtain a filtered complex of finitely generated $\Lambda$-modules
\begin{equation}\label{filtcomplex}
0\ra \EE^0_{\Lambda}(L_0)\ra \EE^0_{\Lambda}(L_1)\ra \cdots,
\end{equation}
where each $\EE^0_{\Lambda}(L_j)$ is endowed with a good filtration as in \emph{loc.cit.}. Taking the associated graded complex of (\ref{filtcomplex}), we obtain a complex of $\gr(\Lambda)$-modules (denoted $G(*)$ in \emph{loc.cit.}):
\[0\ra \gr(\EE^0_{\Lambda}(L_0))\ra \gr(\EE^0_{\Lambda}(L_1))\ra \cdots\] 
and \ by \ \cite[Lemma III.2.2.2(2)]{LiOy} \ we \ have \ isomorphisms \ $\EE^0_{\gr(\Lambda)}(\gr(L_j))\cong \gr(\EE^0_{\Lambda}(L_j))$ for $j\geq 0$. Next, as in \cite[\S III.1]{LiOy} we may associate a spectral sequence $\{E_{j}^{r} : r\geq0, j\geq 0\}$ to the filtered complex (\ref{filtcomplex}) and define a good filtration on $\EE^j_{\Lambda}(M)$ for $j\geq 0$ with the following properties (for convenience we have shifted the numbering):
\begin{enumerate}[(a)]
\item $E_j^0=\gr(\EE^0_{\Lambda}(L_j))$ and $E_j^1=\EE^{j}_{\gr(\Lambda)}(\gr(M))$ for any $j$;
\item for any fixed $r\geq 1$, there is a complex 
\[0\ra E_{0}^{r}\ra \cdots\ra E_{j}^{r}\ra E_{j+1}^r\ra \cdots \]
whose homology gives $E_j^{r+1}$;
\item for $r$ large enough (depending on $j$), $E_j^{\infty}=E_j^{r}\cong \gr(\EE_{\Lambda}^j(M))$.
\end{enumerate}
Since $j_{\Lambda}(M)=j_0$ by assumption, we also have $j_{\gr(\Lambda)}(\gr(M))=j_0$ by \cite[Thm.III.2.5.2]{LiOy} and so $E^{1}_{j}=0$ for $j<j_0$. By (b), this implies short exact sequences
\[0\ra E^{r+1}_{j_0}\ra E^{r}_{j_0}\ra E^{r}_{j_0+1}, \ \ \forall \ r\geq 1.\] 
In particular, by taking $r$ large enough, $\gr(\EE_{\Lambda}^{j_0}(M))=E^{\infty}_{j_0}$ is a submodule of $E^{1}_{j_0}$. Moreover, since $E^{r}_{j_0+1}$ has grade $\geq j_0+1$ for all $r$ and so do its subquotients, the cokernel of $E^{\infty}_{j_0}\hookrightarrow E^1_{j_0}$ also has grade $\geq j_0+1$. 

If moreover $\gr(M)$ is Cohen--Macaulay, then $E^{1}_{j}=0$ except for $j=j_0$, hence $E^{\infty}_{j_0}=E^1_{j_0}$ which implies the last claim.
\end{proof}
 
\subsubsection{On the length of \texorpdfstring{$\pi$}{pi} in the semisimple case}\label{sec:length-of-pi}

For $\rhobar$ as in \S\ref{combi} assumed moreover semisimple and strongly generic, and $\pi$ as in \S\ref{grstr} with moreover $r=1$ and satisfying one more hypothesis, we prove that $\pi$ is generated over $\GL_2(K)$ by its $\GL_2(\oK)$-socle, is irreducible if $\rhobar$ is, and is semisimple of length $3$ if $\brho$ is reducible split and $f=2$.
 
We keep the notation in \S\ref{grstr} and we assume moreover that $\rhobar$ is \emph{semisimple} and satisfies the strong genericity condition (\ref{eq:6}) (we will use the results of \S\ref{tensorinduction}). We fix an admissible smooth representation $\pi$ of $\GL_2(K)$ over $\F$ satisfying the conditions (i), (ii) in \emph{loc.cit.}\ with $r=1$ in (i), i.e.\ $\pi^{K_1}\cong D_0(\brho)$. Recall this implies that $\gr(\pi^\vee)$ is annihilated by $J$, where $\gr(\pi^\vee)$ is computed with the $\m$-adic filtration. We assume moreover:
\begin{enumerate}
\item[(iii)] 
$\pi^{\vee}$ is \emph{essentially self-dual} of grade $2f$, i.e.\ there is a $\GL_2(K)$-equivariant isomorphism of $\Lambda$-modules
\begin{equation}\label{eq:selfdual}
{\rm E}^{2f}_{\Lambda}(\pi^{\vee})\cong \pi^{\vee}\otimes (\det(\rhobar)\omega^{-1})
\end{equation}
(recall $\det(\rhobar)\omega^{-1}$ is the central character of $\pi$). Here $\EE_{\Lambda}^j(\pi^{\vee})$ is endowed with the action of $\GL_2(K)$ (compatible with the $\Lambda$-module structure) defined in \cite[Prop.3.2]{Ko}. 
\end{enumerate}
(Note that, compared with \cite[Def.A.7]{HuWang2}, in the definition of essentially self-dual we do not assume that $\pi^\vee$ is Cohen--Macaulay. However, by \cite[Prop.III.4.2.8(1)]{LiOy} $\pi^{\vee}$ is \emph{pure} in the sense of \cite[Def.III.4.2.7]{LiOy}.)

\begin{rem}\label{formnonminimal}
Conditions (i) to (iii), with $r=1$ in (i), will be satisfied for $\pi$ coming from the global theory in the minimal case (see \S\ref{lcresults}). The reason to impose the extra assumption $r=1$ in (i) is that although for general $r$ we have an equality of diagrams
\[(\pi^{I_1}\hookrightarrow \pi^{K_1})=(D_0(\brho)^{I_1}\hookrightarrow D_0(\brho))^{\oplus r}\] 
for the representations $\pi$ coming from cohomology (see Theorem \ref{nonminimal} below), we do not know if this implies that $\pi$ has the form ${\pi'}^{\oplus r}$ for some representation $\pi'$ of $\GL_2(K)$.
\end{rem}

Given $\sigma\in W(\brho)$, we define the length of $\sigma$ as follows: if ${\lambda}\in\mathscr{D}$ corresponds to $\sigma$ (see \S\ref{combi}), then $\ell(\sigma)\defeq \ell({\lambda})$, see \eqref{eq:ell}. For $0\leq \ell\leq f$, let
\[W_{\ell}(\brho)\defeq \{\sigma\in W(\brho), \ell(\sigma)=\ell\}\]
and define $\tau_{\ell}(\brho)\defeq \oplus_{\sigma\in W_{\ell}(\brho)}\sigma$. We call $W_{\ell}(\brho)$, or by abuse of notation $\tau_{\ell}(\brho)$, an \emph{orbit} in $W(\brho)$. Note that this is different from an orbit of $\delta$ in $W(\rhobar)$ as defined in \S\ref{lowerproof} (see \S\ref{operatorF} for $\delta$), i.e.\ in general $\tau_{\ell}(\brho)$ contains several orbits of $\delta$. 

\begin{lem}\label{lem:orbit}
If $\pi'$ is a nonzero subrepresentation of $\pi$, then $\soc_{\GL_2(\cO_K)}(\pi')$ is a direct sum of orbits in $W(\brho)$.
\end{lem}
\begin{proof}
It is clear that $(\pi'^{I_1}\hookrightarrow \pi'^{K_1})$ is a subdiagram of $(\pi^{I_1}\hookrightarrow \pi^{K_1})$. The result follows from this using \cite[Thm.15.4]{BP} together with the proof of \cite[Thm.19.10]{BP}. Actually, when $\brho$ is irreducible, we even have $\soc_{\GL_2(\cO_K)}(\pi')=\soc_{\GL_2(\cO_K)}(\pi)$ by (the proof of) \cite[Thm.19.10]{BP}. 
\end{proof}

We use without comment the notation and definitions in \S\ref{upperb} and denote by lg$(\tau)$ the length of a finite-dimensional representation $\tau$ of $\GL_2(\oK)$ over $\F$.

\begin{prop}\label{prop:socle=orbit} 
Let $\pi'$ be a subquotient of $\pi$.
\begin{enumerate}
\item We have $\dim_{\F\ppar{X}} D_{\xi}^{\vee}(\pi')=m_{\p_0}(\pi'^{\vee})$.
\item Assume that $\pi'$ is a subrepresentation of $\pi$. Then
  \[\dim_{\F\ppar{X}}
  D_{\xi}^{\vee}(\pi')=m_{\p_0}(\pi'^{\vee})=\mathrm{lg}(\soc_{\GL_2(\cO_K)}(\pi')).\]
  In particular, if $\pi'\neq0$, then $D_{\xi}^{\vee}(\pi')\neq0$.
\item Assume that $\pi'$ is a nonzero quotient of $\pi$. Then $D_{\xi}^{\vee}(\pi')\neq0$.
\end{enumerate}
\end{prop}
\begin{proof}
(i) First, for any subquotient $\pi'$ of $\pi$, we equip the $\Lambda$-module $\pi'^{\vee}$ with a good filtration $F$ by choosing two submodules $\pi_1^{\vee}\subset\pi_2^{\vee}$ of $\pi^{\vee}$ (with filtrations induced from the $\m$-adic one on $\pi^{\vee}$) such that $\pi'^{\vee}\cong\pi_2^{\vee}/\pi_1^{\vee}$ and taking the induced filtration.\footnote{The filtrations on $\pi_2^{\vee}$ and $\pi_1^{\vee}$ might not be the $\m$-adic ones, and the resulting filtration on $\pi'^{\vee}$ might depend on the choice of $\pi_1^{\vee}$ and $\pi_2^{\vee}$.} Then $\gr_{F}(\pi'^{\vee})$ is again an $\overline{R}$-module, and $\dim_{\F\ppar{X}} D_{\xi}^{\vee}(\pi')\leq m_{\p_0}(\pi'^{\vee})$ by Corollary \ref{cor:upperbound}. Since $\dim_{\F\ppar{X}}D_{\xi}^{\vee}(\pi)=m_{\p_0}(\pi^{\vee})$ by Corollary \ref{maintensorind}, since $D_{\xi}^{\vee}(-)$ is an exact functor by Theorem \ref{thm:exactnessDA} and since $\mathcal{Z}(-)$, and in particular $m_{\p_0}(-)$, are additive by Lemma \ref{lem:Z-additive}, the result follows.

(ii) By assumption $\pi'$ is a subrepresentation of $\pi$. Using that $\soc_{\GL_2(\cO_K)}(\pi')$ is a union of orbits of $\delta$, or equivalently of $S$ as in \eqref{isoS}, by Lemma \ref{lem:orbit}, it follows from Proposition \ref{prop:phi-gamma-piece} that
\[\dim_{\F\ppar{X}}D_{\xi}^{\vee}(\pi')\geq \mathrm{lg}(\soc_{\GL_2(\cO_K)}(\pi')).\]
On the other hand, by Lemma \ref{lem:Yj-fa}(i) and Corollary \ref{cor:cycle-pi'}, we have $m_{\p_0}(\pi'^{\vee})\leq \mathrm{lg}(\soc_{\GL_2(\cO_K)}(\pi'))$ (see the proof of Theorem \ref{thm:upperbound}). Hence all the three quantities are equal by (i).

(iii) Let $\pi''$ be the kernel of the quotient map $\pi\twoheadrightarrow \pi'$ so that we have an exact sequence of $\Lambda$-modules:
\[0\ra \pi'^{\vee}\ra \pi^{\vee}\ra \pi''^{\vee}\ra0.\]
Since $\pi^{\vee}$ is essentially self-dual of grade $2f$ by assumption, $\pi'^{\vee}$ also has grade $2f$ by \cite[Prop.III.4.2.8(1)]{LiOy} and \cite[Prop.III.4.2.9]{LiOy}. Taking ${\rm E}^i_{\Lambda}(-)$, we obtain a long exact sequence of $\Lambda$-modules
\begin{equation}\label{exactgr}
0\ra {\rm E}^{2f}_{\Lambda}(\pi''^{\vee})\ra {\rm E}^{2f}_{\Lambda}(\pi^{\vee})\ra {\rm E}^{2f}_{\Lambda}(\pi'^{\vee})\ra {\rm E}^{2f+1}_{\Lambda}(\pi''^{\vee})
\end{equation}
which gives rise by Pontryagin duality to an exact sequence of admissible smooth representations of $\GL_2(K)$ with central character (see \cite[Cor.1.8]{Ko}). Define $\widetilde \pi$ to be the admissible smooth representation of $\GL_2(K)$ such that
\begin{equation}\label{eq:define:pi0}
\widetilde \pi^{\vee}\otimes(\det(\rhobar)\omega^{-1})\cong \mathrm{Im}\big({\rm E}^{2f}_{\Lambda}(\pi^{\vee})\ra {\rm E}^{2f}_{\Lambda}(\pi'^{\vee})\big).
\end{equation}
Since $\pi^{\vee}$ is essentially self-dual by assumption (see \eqref{eq:selfdual}), $\widetilde \pi^{\vee}$ is a quotient of $\pi^{\vee}$ and dually $\widetilde \pi$ is a subrepresentation of $\pi$. Since ${\rm E}^{2f+1}_{\Lambda}(\pi''^{\vee})$ has grade $\geq 2f+1$ as $\Lambda$ is Auslander regular, we have by (\ref{exactgr}) and the discussion before Theorem \ref{prop:cycle-Ext}:
\[\mathcal{Z}({\rm E}^{2f}_{\Lambda}(\pi'^{\vee}))=\mathcal{Z}(\widetilde \pi^{\vee}\otimes (\det(\rhobar)\omega^{-1})),\]
hence 
$\mathcal{Z}(\pi'^{\vee})=\mathcal{Z}(\widetilde \pi^{\vee})$ by Theorem \ref{prop:cycle-Ext} which implies in particular by (i):
\begin{equation}\label{eq:pi1=pi3}
\dim_{\F\ppar{X}}D_{\xi}^{\vee}(\pi')=\dim_{\F\ppar{X}}D_{\xi}^{\vee}(\widetilde \pi).
\end{equation} 
Since $j_{\Lambda}(\pi'^{\vee})=2f$, $\mathcal{Z}(\pi'^{\vee})$ is nonzero (using e.g.\ (\ref{dimannihi})), hence $\widetilde \pi$ is nonzero, thus $D_{\xi}^{\vee}(\widetilde \pi)\neq 0$ by (ii), and finally $D_{\xi}^{\vee}(\pi')\ne 0$ by (\ref{eq:pi1=pi3}). 
\end{proof} 

\begin{rem}\label{da=daet}
(i) The construction of $\widetilde \pi$ in the proof of Proposition \ref{prop:socle=orbit}(iii) does not use the assumption that $\brho$ is semisimple.
Moreover, items (i) and (ii) of Proposition \ref{prop:socle=orbit} do not require the essential self-duality of $\pi^\vee$ (equation \eqref{eq:selfdual} above).
\\
(ii) It follows from Proposition \ref{prop:socle=orbit}(ii), from Corollary \ref{cor:upperbound}, from Lemma \ref{lem:rank=multi}, from Lemma \ref{lem:comparison_ranks} and from (\ref{dapi}) that for $\pi'\subseteq \pi$ as in Proposition \ref{prop:socle=orbit}(ii) we have
\begin{equation}\label{=}
\rk_A(D_A(\pi')^{\et})=\dim_{\F\ppar{X}}D_\xi^\vee(\pi')=m_{\mathfrak{p}_0}(\gr(\pi'^\vee))=\rk_A(D_A(\pi')).
\end{equation}
By Corollary \ref{cor:psiiso}, both $D_A(\pi')$ and $D_A(\pi')^{\et}$
are finite projective $A$-modules and it follows from (\ref{=}) that the surjection of $A$-modules $D_A(\pi')\twoheadrightarrow D_A(\pi')^{\et}$ is here an isomorphism.
\end{rem}

\begin{thm}\label{thm:gen-socle}
As a $\GL_2(K)$-representation, $\pi$ is generated by its $\GL_2(\cO_K)$-socle.
\end{thm}
\begin{proof}
Let $\tau\defeq \soc_{\GL_2(\cO_K)}(\pi)$, let $\pi'\defeq \langle \GL_2(K).\tau\rangle$ be the subrepresentation of $\pi$ generated by $\tau$ and let $\pi''\defeq \pi/\pi'$. Since $D_{\xi}^{\vee}(-)$ is exact by Theorem \ref{thm:functors_comparison}, we have 
\[\dim_{\F\ppar{X}}D_{\xi}^{\vee}(\pi)=\dim_{\F\ppar{X}}D_{\xi}^{\vee}(\pi')+\dim_{\F\ppar{X}}D_{\xi}^{\vee}(\pi''). \]
However, since $\pi$ and $\pi'$ have the same $\GL_2(\cO_K)$-socle, we have \[\dim_{\F\ppar{X}}D_{\xi}^{\vee}(\pi)=\dim_{\F\ppar{X}}D_{\xi}^{\vee}(\pi')\] by Proposition \ref{prop:socle=orbit}(ii), thus $D_{\xi}^{\vee}(\pi'')=0$. If $\pi''$ is nonzero this contradicts Proposition \ref{prop:socle=orbit}(iii). 
\end{proof}

\begin{cor}\label{cor:pi-irred}
Assume that $\brho$ is irreducible. Then $\pi$ is irreducible and is a supersingular representation.
\end{cor}
\begin{proof}
This follows from Theorem \ref{thm:gen-socle} and \cite[Thm.19.10(i)]{BP}.
\end{proof}

\begin{rem}
(i) A result analogous to Theorem \ref{thm:gen-socle} when $\brho$ is not semisimple is proved in \cite[Thm.1.6]{HuWang2}.\\
(ii) While we believe that Proposition \ref{prop:socle=orbit} and Theorem \ref{thm:gen-socle} should be true without assuming $r=1$, we don't know how to prove a generalization of Corollary \ref{cor:pi-irred} (i.e.\ $\pi$ is semisimple and has length $r$ in general), as mentioned in Remark \ref{formnonminimal}.
\end{rem} 

\begin{cor}\label{cor:split2}
Assume that $\brho$ is reducible split. Then $\pi$ has the form 
\begin{equation}\label{eq:decomp:ss}
\pi=\pi_0\oplus \pi_f \oplus \pi',
\end{equation}
where
\begin{itemize}
\item $\pi_0$ and $\pi_f$ are irreducible principal series such that $\EE^{2f}_{\Lambda}(\pi_i^{\vee})\cong \pi_{f-i}^{\vee}\otimes(\det(\brho)\omega^{-1})$, $i\in \{0,f\}$;
\item $\pi'$ is generated by its $\GL_2(\cO_K)$-socle and $\pi'^{\vee}$ is essentially self-dual {\upshape(}as in {\upshape(\ref{eq:selfdual}))}. Moreover, $\pi'$ is irreducible and supersingular when $f=2$. 
\end{itemize}
\end{cor}
\begin{proof}
By the definition of $W(\brho)$ (see \cite[\S11]{BP}), there exists a unique Serre weight $\sigma_0\in W(\brho)$ such that $\ell(\sigma_0)=0$. Let $\chi_{\sigma_0}$ be the character of $I$ acting on $\sigma_0^{I_1}$. It is easy to check that
\[\mathrm{JH}\big(\Ind_I^{\GL_2(\cO_K)}\chi_{\sigma_0}\big)\cap W(\brho)=\{\sigma_0\}.\]
Let $\pi_0\defeq \langle \GL_2(K).\sigma_0\rangle$, a subrepresentation of $\pi$. We claim that $\pi_0$ is an irreducible principal series. 
Indeed, by \cite[Lemma 5.14]{HuWang2} and its proof, the morphism (induced from $\sigma_0\hookrightarrow \pi$ by Frobenius reciprocity)
\[\cInd_{\GL_2(\cO_K)K^{\times}}^{\GL_2(K)}\sigma_0\ra \pi\]
(where $\cInd$ means compact induction) factors through $\cInd_{\GL_2(\cO_K)K^{\times}}^{\GL_2(K)}\sigma_0/(T-\mu_0)$ for some $\mu_0\in\F^{\times}$ (as $\soc_{\GL_2(\cO_K)}(\pi)$ is multiplicity-free). Note that the genericity of $\brho$ implies that $\dim_{\F}\sigma_0\geq 2$, hence the representation $\cInd_{\GL_2(\cO_K)K^{\times}}^{\GL_2(K)}\sigma_0/(T-\mu_0)$ is irreducible and isomorphic to some principal series by \cite[Thm.30]{BL1}. 
This proves the claim. Moreover, the $\GL_2(\cO_K)$-socle of $\pi_0$ is exactly $\sigma_0$, and if $\pi_0\cong\Ind_{B(K)}^{\GL_2(K)}\chi_0$ for some smooth character $\chi_0:T(K)\ra \F^{\times}$ then $\chi_0^s|_H=\chi_{\sigma_0}$. Similarly, there exists a unique Serre weight $\sigma_f\in W(\brho)$ such that $\ell(\sigma_f)=f$. It satisfies again
\[\mathrm{JH}\big(\Ind_I^{\GL_2(\cO_K)}\chi_{\sigma_f}\big)\cap W(\brho)=\{\sigma_f\}\]
and by the same argument as above the subrepresentation $\pi_f\defeq \langle\GL_2(K).\sigma_f\rangle$ of $\pi$ is an irreducible principal series with $\GL_2(\cO_K)$-socle equal to $\sigma_f$, and if $\pi_f\cong\Ind_{B(K)}^{\GL_2(K)}\chi_f$ then $\chi_f^s|_H=\chi_{\sigma_f}$. The map $\pi_0\oplus \pi_f\rightarrow \pi$ is injective since it is injective on the $\GL_2(\oK)$-socles.

Letting $\pi'\defeq \pi/(\pi_0\oplus \pi_f)$, we have an exact sequence of $\Lambda$-modules:
\[0\rightarrow {\pi'}^\vee\rightarrow \pi^\vee\rightarrow \pi_0^\vee\oplus \pi_f^\vee\rightarrow 0.\]
As $\Lambda$ is Auslander regular and $\pi^\vee$ is of grade $2f$, it follows from \cite[Cor.III.2.1.6]{LiOy} that ${\pi'}^\vee$ is of grade $\geq 2f$, hence ${\rm E}^{2f-1}_{\Lambda}({\pi'}^{\vee})=0$ and there is an exact sequence of (finitely generated) $\Lambda$-modules
\[0\rightarrow {\rm E}^{2f}_{\Lambda}(\pi_0^{\vee})\oplus {\rm E}^{2f}_{\Lambda}(\pi_f^{\vee})\rightarrow {\rm E}^{2f}_{\Lambda}({\pi}^{\vee})\rightarrow {\rm E}^{2f}_{\Lambda}({\pi'}^{\vee}).\]
Since $\pi^{\vee}$ is essentially self-dual by assumption (see \eqref{eq:selfdual}) and since ${\rm E}^{2f}_{\Lambda}(\pi_0^{\vee})^\vee$ and $\EE^{2f}_{\Lambda}(\pi_f^{\vee})^\vee$ are also irreducible principal series by \cite[Prop.5.4]{Ko}, we see that $\pi$ admits a quotient isomorphic to $\pi_0'\oplus \pi_f'$, where $\pi_{i}'$ (for $i\in \{0,f\}$) is the (irreducible) principal series such that
\begin{equation}\label{eq:pif'}
\pi_{i}'^{\vee}\otimes(\det(\brho)\omega^{-1})=\EE^{2f}_{\Lambda}(\pi_{f-i}^{\vee}).
\end{equation}
Explicitly, if $\pi_i'\cong\Ind_{B(K)}^{\GL_2(K)}\chi_i'$ for some smooth characters $\chi_i':T(K)\ra\F^{\times}$, and if we let $\alpha_{B}\defeq \omega\otimes\omega^{-1}:T(K)\ra \F^{\times}$ and $\eta\defeq \det(\brho)\omega^{-1}$ (for short), then by \cite[Lemma 10.7]{HuWang2} (which is based on \cite[Prop.5.4]{Ko}):
\begin{equation}\label{eq:Koh-char}
\chi_f'=\chi_0^{-1}\alpha_B(\eta\otimes\eta),\ \ \chi_0'=\chi_f^{-1}\alpha_B(\eta\otimes\eta).
\end{equation}
Let us compute the $\GL_2(\cO_K)$-socle of $\pi_f'$ (the case of $\pi_0'$ is similar). 
Since $\eta$ is equal to the central character of $\pi_0$, we have $\chi_0^{-1}(\eta\otimes\eta)=\chi_0^s$, so that \eqref{eq:Koh-char} becomes 
$\chi_f'=\chi_0^s\alpha_B$. 
Since $\chi_0^s|_H=\chi_{\sigma_0}$ as seen in the first paragraph, we deduce
\begin{equation}\label{eq:socle-dual}(\chi_f')^s|_H=\chi_{\sigma_0}^s\alpha_B^{-1}=\chi_{\sigma_f},\end{equation}
where the last equality holds by an easy check using the definition of $\sigma_0$ and $\sigma_f$ (see \cite[\S11]{BP}). 
In particular, our genericity assumption on $\brho$ implies that $\chi_f'\neq \chi_f'^s$ when restricted to $T(\cO_K)$. Using \cite[Thm.34(2)]{BL1}, this implies that the $\GL_2(\cO_K)$-socle of $\pi_f'$ is irreducible and actually isomorphic to $\sigma_f$ by \eqref{eq:socle-dual}. Similarly, 
the $\GL_2(\cO_K)$-socle of $\pi_0'$ is isomorphic to $\sigma_0$.

We claim that the composite morphism
\[\pi_0\oplus \pi_f\hookrightarrow \pi\twoheadrightarrow \pi_0'\oplus \pi_f'\]
is an isomorphism. Since $\pi$ is generated by its $\GL_2(\cO_K)$-socle, namely $\bigoplus_{\sigma\in W(\brho)}\sigma$, the composite morphism 
\[\iota_0:\bigoplus_{\sigma\in W(\brho)}\sigma\hookrightarrow \pi\twoheadrightarrow \pi_0'\]
is nonzero. Since the image is contained in $\soc_{\GL_2(\cO_K)}(\pi_0')$, which is equal to $\sigma_0$ as seen in the last paragraph, $\iota_0$ is nonzero 
when restricted to $\sigma_0$. But, by construction we have $\langle \GL_2(K).\sigma_0\rangle=\pi_0$ inside $\pi$, hence the composite morphism $\pi_0\hookrightarrow\pi\twoheadrightarrow \pi_0'$ is nonzero, hence an isomorphism as both $\pi_0$ and $\pi_0'$ are irreducible. In the same way the composite morphism $\pi_f\hookrightarrow\pi\twoheadrightarrow \pi_f'$ is also an isomorphism. This proves the claim, from which the decomposition \eqref{eq:decomp:ss} immediately follows. From \eqref{eq:pif'} we also deduce the isomorphism $\EE^{2f}_{\Lambda}(\pi_i^{\vee})\cong \pi_{f-i}^{\vee}\otimes\eta$ for $i\in\{0,f\}$. 

We now finish the proof. First, $\pi'$ is generated by its $\GL_2(\cO_K)$-socle by Theorem \ref{thm:gen-socle}. Explicitly, we have
\[\mathrm{soc}_{\GL_2(\cO_K)}(\pi')=\bigoplus_{\stackrel{\sigma\in W(\rhobar)}{0<\ell(\sigma)<f}}\sigma.\]
In particular, if $f=2$, then $\pi'$ is irreducible and is a supersingular representation by \cite[Thm.19.10(ii)]{BP}. 
Finally we prove that $\pi'^{\vee}$ is essentially self-dual (as in (\ref{eq:selfdual})). In fact, using \eqref{eq:decomp:ss} and noting that \[(\EE^{2f}_{\Lambda}(\pi^{\vee}))^{\vee}\otimes \eta\cong \pi_0\oplus \pi_f \oplus (\EE^{2f}_{\Lambda}(\pi'^{\vee}))^{\vee}\otimes \eta,\] it suffices to prove that the composite morphism
\[\pi'\hookrightarrow \pi \simto (\EE^{2f}_{\Lambda}(\pi^{\vee}))^{\vee}\otimes \eta\twoheadrightarrow (\EE^{2f}_{\Lambda}(\pi'^{\vee}))^{\vee}\otimes \eta\]
is an isomorphism. Since both the source and the target have the same $\GL_2(\cO_K)$-socle, the morphism is injective because it is when restricted to the $\GL_2(\cO_K)$-socle of $\pi'$ and is surjective because $(\EE^{2f}_{\Lambda}(\pi'^{\vee}))^{\vee}\otimes\eta$ is generated by its $\GL_2(\cO_K)$-socle.
\end{proof}

\subsection{Local-global compatibility results for \texorpdfstring{$\GL_2(\Q_{p^f})$}{GL\_2(Q\_\{p\^{}f\})}}\label{gl2results}

We prove special cases of Conjecture \ref{theconjbar} and Conjecture \ref{theconj} when $F^+_v=\Qpf$ and $n=2$. We assume $E=W(\F)[1/p]$ (thus $\oE=W(\F)$ and $\pE=p$).

\subsubsection{Global setting and results}\label{globalp}

We refine the global setting of \S\S\ref{global},~\ref{slgc} when $n=2$ in order to apply the results of \cite{BHHMS1} and we state the first main global result.

We come back to the setting of \S\ref{global} when $n=2$ and we assume $p>7$. We make the following extra assumptions on the field $F$ and the unitary group $H$:
\begin{enumerate}
\item $F/F^+$ is unramified at all finite places;
\item $p$ is unramified in $F^+$;
\item $H$ is defined over $\oFF$ and $H\times_{\oFF}F^+$ is quasi-split at all finite places of $F^+$.
\end{enumerate}
Condition (i) (together with the fact that any $p$-adic place of $F^+$ splits in $F$) implies $[F^+:\Q]$ is even (see \cite[\S3.1]{gee-kisin}). By \cite[\S3.1.1]{gee-kisin} such groups $H$ always exist. We denote by $R_{\rbar_{\tilde{w}}}^\square$ the universal framed deformation ring of $\rbar_{\tilde{w}}$ over $W(\F)$ (${\tilde{w}}$ is any finite place of $F$). We set $K\defeq F_v^+$ and $f\defeq [K:\Qp]$.

We let $\rbar:\gF\rightarrow {\GL}_2(\F)$ as in \S\ref{wlgc} and make the following extra assumptions on $\rbar$ (recall that $S_p$ is the set of places of $F^+$ dividing $p$):
\begin{enumerate}
\setcounter{enumi}{3}
\item $\rbar\vert_{{\Gal}(\overline F/F(\sqrt[p]{1}))}$ is adequate (\cite[Def.2.20]{Th2});
\item $\rbar_{\tilde{w}}$ is unramified if ${\tilde{w}}\vert_{F^+}$ is inert in $F$;
\item $R_{\rbar_{\tilde{w}}}^\square$ is formally smooth over $W(\F)$ if $\rbar_{\tilde{w}}$ is ramified and ${\tilde{w}}\vert_{F^+}\notin S_p$;
\item $\rbar_{\tilde{w}}$ is generic in the sense of \cite[Def.11.7]{BP} if ${\tilde{w}}\vert_{F^+}\in S_p\backslash\{v\}$;
\item $\rbar_{\tilde{v}}$ is, up to twist, of one of the following forms for ${\tilde{v}}\vert_{F^+}=v$:
\begin{itemize}
\item $\rbar_{\tilde{v}}\vert_{I_K}\cong \begin{pmatrix}\omega_f^{(r_0+1)+\cdots+p^{f-1}(r_{f-1}+1)}&0\\0&1\end{pmatrix}$\ \ $3\leq r_i\leq p-6$,
\item $\rbar_{\tilde{v}}\vert_{I_K}\cong\begin{pmatrix}\omega_{2f}^{(r_0+1)+\cdots+p^{f-1}(r_{f-1}+1)}&0\\0&\omega_{2f}^{p^f(\mathrm{same})}\end{pmatrix}$\ $4\leq r_0\leq p-5$, $3\leq r_i\leq p-6$ for $i>0$.
\end{itemize}
\end{enumerate}
Note that conditions (iv) to (viii) only depend on ${\tilde{w}}\vert_{F^+}$ and ${\tilde{v}}\vert_{F^+}$ using condition (i) in \S\ref{wlgc} (the genericity conditions in (viii) are satisfied in \cite[\S3.3]{DoLe} {\it and} don't depend on the choices of $\sigma_0$, $\sigma'_0$). We denote by $S_{\rbar}$ the finite set of finite places of $F^+$ such that ${\tilde{w}}\vert_{F^+}\in S_{\rbar}$ if and only if $\rbar_{\tilde{w}}$ is ramified. Thus $S_p\subseteq S_{\rbar}$ and by (ii) any place in $S_{\rbar}$ splits in $F^+$. We fix a finite place $v_1$ of $F^+$ which is not in $S_{\rbar}$ and satisfies the assumptions in \cite[\S6.6]{EGS}, and we choose $\widetilde{v_1}\vert v_1$ in $F$.

We choose $S$ a finite set of finite places of $F^+$ that split in $F$ containing $S_{\rbar}$ but not $v_1$, and a compact open subgroup $U=\prod_{w}U_{w}\subseteq \GA$ such that
\begin{enumerate}
\setcounter{enumi}{8}
\item $U_{w}\subseteq H({\mathcal O}_{F_{w}^+})$ if $w$ splits in $F$;
\item $U_{w}$ is maximal hyperspecial in $H(F_{w}^+)$ if $w$ is inert in $F$;
\item $U_{w}=H({\mathcal O}_{F_{w}^+})$ if $w\notin S\cup \{v_1\}$ and $w$ splits in $F$ or if $w\in S_p$;
\item $\iota_{\widetilde{v_1}}(U_{v_1})$ is contained in the upper-triangular unipotent matrices mod $\widetilde{v_1}$.
\end{enumerate}
We also define $V\defeq U^p\prod_{w\in S_p}\!V_w$, where $U^p\defeq \prod_{w\notin S_p}U_w$ and $V_{w}$ is a pro-$p$ normal subgroup of $U_{w}$ if $w\in S_p$ (hence $V$ is normal in $U$). We set $\Sigma\defeq S\cup \{v_1\}$ and assume $S(V,\F)[\m^{\Sigma}]\ne 0$ (see \S\ref{somprel}). Note that $S(V,\F)[\m^{\Sigma}]$ doesn't depend on $S$ as above by the proof of \cite[Lemma 4.6(a)]{BDJ}. For each place $w\in S_p$ we choose a place ${\tilde{w}}\vert w$ in $F$. For $w\in S_p$ recall from \S\ref{lowerstatement} that $W(\rbar_{\tilde{w}}(1))$ is the set of Serre weights associated to $\rbar_{\tilde{w}}(1)\defeq \rbar_{\tilde{w}}\otimes \omega$ defined as in \cite[\S3]{BDJ}. Then it follows from \cite[Thm.A]{GLS} and \cite[Def.2.9]{BLGG13} that we have
\begin{equation}\label{gls}
\Hom_{U}\big(\!\otimes_{w\in S_p}\sigma_{\tilde w},S(V,\F)[\m^{\Sigma}]\big)\ne 0 \Longleftrightarrow \sigma_{\tilde w}\in W(\rbar_{\tilde w}(1)) \ \forall \ w\in S_p,
\end{equation}
where we consider $\otimes_{w\in S_p}\sigma_{\tilde w}$ as a representation of $U$ via $U\twoheadrightarrow U/V\buildrel\sim\over\rightarrow \prod_{w\in S_p}U_{w}/V_{w}$ and the isomorphisms $\iota_{\tilde w}$. Note that the left-hand side of (\ref{gls}) is also isomorphic to $\Hom_{U}(\otimes_{w\in S_p}\sigma_{\tilde w},S(U^p,\F)[\m^{\Sigma}])$, where $S(U^p,\F)[\m^{\Sigma}]$ is defined as in \S\ref{somprel}, replacing $U^v$ by $U^p$.

We freely use the previous local notation ($I_1$ is the pro-$p$ Iwahori subgroup in $\GL_2(\cO_K)=\GL_2(\cO_{F_{\tilde v}})$ etc.) and set $\rhobar\defeq \rbar_{\tilde v}(1)$.

\begin{thm}\label{nonminimal}
Choose Serre weights $\sigma_{\tilde w}\in W(\rbar_{\tilde w}(1))$ for $w\in S_p\backslash \{v\}$ and set
\[\pi\defeq  \Hom_{U^v}(\otimes_{w\in S_p\backslash\{v\}}\sigma_{\tilde w},S(V^v,\F)[\m^{\Sigma}]).\]
Then there exist an integer $r\geq 1$ only depending on $v$, $U^v$, $V^v$, $\otimes_{w\in S_p\backslash\{v\}}\sigma_{\tilde w}$ and $\overline r$ and a diagram $D(\rhobar)=(D_1(\rhobar)\hookrightarrow D_0(\rhobar))$ as in \S \ref{lowerstatement} only depending on $\rhobar=\rbar_{\tilde v}(1)$ {\upshape(}and satisfying the assumptions in {\it loc.cit.}\ on the constants $\nu_i${\upshape)} such that there is an isomorphism of diagrams
\[D(\rhobar)^{\oplus r}\cong (\pi^{I_1}\hookrightarrow \pi^{K_1}).\]
\end{thm}

The case $r=1$ of Theorem \ref{nonminimal} is known and due to Dotto and Le (\cite[Thm.1.3]{DoLe}). We generalize below their proof to the case $r\geq 1$ using the results in \cite[\S8.2]{BHHMS1}. Moreover the diagram $D(\rhobar)$ in Theorem \ref{nonminimal} is in fact the same as the diagram ${\mathcal D}(\pi_{\rm glob}(\rhobar))$ of \cite[Thm.1.3]{DoLe}.

\subsubsection{Review of patching functors}\label{patching}

We recall the patching functors of \cite[\S6.6]{EGS} and some results of \cite[\S8.2]{BHHMS1}.

We keep the notation of \S\ref{globalp}. We choose Serre weights $\sigma_{\tilde w}\in W(\rbar_{\tilde w}(1))$ for $w\in S_p\backslash \{v\}$ and set
\[\sigma^v\defeq \bigotimes_{w\in S_p\backslash \{v\}}\sigma_{\tilde w}.\]
For each $w\in S_p\backslash \{v\}$ we fix a tame inertial type $\tau_{\tilde w}$ at the place $\tilde w$ such that, denoting by $\sigma(\tau_{\tilde w})$ the irreducible smooth representation of $\GL_2({\mathcal O}_{F_{\tilde w}})$ over $E$ associated by Henniart to $\tau_{\tilde w}$ in the appendix to \cite{BM}, $\JH(\ovl{\sigma(\tau_{\tilde w})})$ contains exactly one Serre weight in $W(\rbar_{\tilde w}(1))$ (where $\ovl{(-)}$ means the mod $p$ semisimplification). The existence of such $\tau_{\tilde w}$ follows from \cite[Prop.3.5.1]{EGS}, and the fact $\sigma(\tau_{\tilde w})$ can be realized over $E=W(\F)[1/p]$ follows from \cite[Lemma 3.1.1]{EGS}. For each $w\in S_p\backslash \{v\}$ we also fix a $\GL_2({\mathcal O}_{F_{{\tilde w}}})$-invariant $W(\F)$-lattice $\sigma^0(\tau_{\tilde w})$ in $\sigma(\tau_{\tilde w})$.

We define
\begin{equation*}
\sigma^{0,v}\defeq  \bigotimes_{w\in S_p\backslash \{v\}}\sigma^0(\tau_{\tilde w}),
\end{equation*}
and for any continuous representation $\sigma_{\tilde v}$ of $\GL_2({\mathcal O}_{F_{\tilde v}})$ on a finite type $W(\F)$-module, we consider $\sigma^{0,v}\otimes_{W(\F)}\sigma_{\tilde v}$ as a representation of $U$ via $U\twoheadrightarrow \prod_{w\in S_p}U_w$ and the isomorphisms $\iota_{\tilde w}$. We define $S(U^p,W(\F))_{\m^{\Sigma}}$ exactly as in \S\ref{somprel} replacing $\F$ by $W(\F)$ and $U^v$ by $U^p$. Then, as in \cite[\S\S 6.2,6.6]{EGS}, by ``patching'' $\Hom_U(\sigma^{0,v}\otimes_{W(\F)}\sigma_{\tilde v},S(U^p,W(\F))_{\m^{\Sigma}})^*$ for various $U$ (where $(-)^*\defeq \Hom_{W(\F)}((-),E/W(\F))$ as in {\it loc.cit.}), we obtain a patching functor
\[M_\infty:\ \sigma_{\tilde v}\longmapsto M_\infty(\sigma^{0,v}\otimes_{W(\F)}\sigma_{\tilde v})\]
which is an exact functor from the category of continuous representations $\sigma_{\tilde v}$ of $\GL_2({\mathcal O}_{F_{\tilde v}})$ on finite type $W(\F)$-modules to the category of finite type $R_\infty$-modules (though this patching functor depends on $\sigma^{0,v}$, we just write $M_\infty(\sigma_{\tilde v})$ in the sequel). The local ring $R_\infty$ is (see \cite[\S4.3]{gee-kisin} or \cite[\S6.2]{DoLe}):
\begin{equation*}
R_\infty\defeq  R^{\rm loc}\bbra{X_1,\ldots,X_{q-[F^+:{\mathbb Q}]}},
\end{equation*}
where $q$ is an integer $\geq [F^+:{\mathbb Q}]$ and
\[R^{\rm loc}\defeq \Big(\widehat\otimes_{w\in S\backslash S_p}R_{\rbar_{\tilde w}(1)}^{\square}\Big)\widehat\otimes_{W(\F)}\Big(\widehat\otimes_{w\in S_p\backslash \{v\}}R_{\rbar_{\tilde w}(1)}^{\square,(1,0),\tau_{\tilde w}}\Big)\widehat\otimes_{W(\F)}R_{\rbar_v(1)}^{\square}.\]
Recall $R_{\rbar_{\tilde w}(1)}^{\square,(1,0),\tau_{\tilde w}}$ is the reduced $p$-torsion free quotient of $R_{\rbar_{\tilde w}(1)}^{\square}$ parametrizing framed potentially Barsotti--Tate deformations with inertial type $\tau_{\tilde w}$ (by local-global compatibility and the inertial Langlands correspondence, for $w\in S_p \backslash \{v\}$ the action of $R_{\rbar_{\tilde w}(1)}^{\square}$ on $M_\infty(\sigma^{0,v}\otimes_{W(\F)}\sigma_v)$ factors through this quotient, see \cite[\S6.6]{EGS}). As in \cite[\S8.1]{BHHMS1} (see the discussion before \cite[Rem.8.1.3]{BHHMS1} but note that we do not need to fix the determinant here) we have isomorphisms $R_{\rbar_{\tilde w}(1)}^{\square}\cong W(\F)\bbra{X_1,X_2,X_3,X_4}$ for $w\in S\backslash S_p$, and, by genericity of $\rbar_v$,
\begin{equation*}
R_{\rbar_v(1)}^{\square}\cong W(\F)\bbra{X_1,\dots,X_{4+4[F_{\tilde v}:\Qp]}}.
\end{equation*}
By \cite[Thm.7.2.1(2)]{EGS} (and \cite[Rk.5.2.2]{gee-kisin}) we have
\begin{equation*}
R_{\rbar_{\tilde w}(1)}^{\square,(1,0),\tau_{\tilde w}}\cong W(\F)\bbra{X_1,\dots,X_{4+[F_{\tilde w}:\Qp]}},
\end{equation*}
so that we finally get
\begin{equation}\label{rinfini}
R_\infty\cong R_{\rbar_v(1)}^{\square}\bbra{X_1,\dots,X_{4(|S|-1)+q-[F^+_v:\Qp]}}\cong W(\F)\bbra{X_1,\dots,X_{4|S|+q+3[F^+_v:\Qp]}}.
\end{equation}
Moreover, if $\sigma_{\tilde v}$ is free of finite type over $W(\F)$, then $M_\infty(\sigma_{\tilde v})$ is free of finite type over a subring $S_\infty$ of $R_\infty$, where $S_\infty\cong W(\F)\bbra{x_1,\dots,x_{4|S|+q}}$. Finally, denoting by $\m_\infty$ the maximal ideal of $R_\infty$, we have
\begin{equation}\label{modmax}
M_\infty(\sigma_{\tilde v})/\m_\infty\cong \Hom_{U}\!\big(\mkern-3mu(\otimes_{w\in S_p\backslash{v}}\sigma_{\tilde w})\otimes_{\F}\ovl{\sigma_{\tilde v}},S(U^p,\F)[\m^{\Sigma}]\big)^\vee\cong \Hom_{U_v}(\ovl{\sigma_{\tilde v}},\pi)^\vee,
\end{equation}
where $\pi$ is as in Theorem \ref{nonminimal}.

Since everything is now at the place $\tilde v$, we drop the index $\tilde v$. If $\tau$ is a tame inertial type, we set $R_\infty^{(1,0),\tau}\defeq R_\infty\otimes_{R_{\rhobar}^{\square}}R_{\rhobar}^{\square,(1,0),\tau}$. If $\sigma\in W(\rhobar)$, we denote by $P_{\sigma}$ the projective $\F[\GL_2(\Fq)]$-envelope of $\sigma$ and by $\widetilde P_\sigma$ the projective $W(\F)[\GL_2(\Fq)]$-module lifting $P_\sigma$. We recall that the scheme theoretic support of an $R_\infty$-module $M$ is $R_\infty/{\rm Ann}_{R_\infty}(M)$. The following theorem then follows by exactly the same proof as for \cite[Prop.8.2.3]{BHHMS1} and \cite[Prop.8.2.6]{BHHMS1}.

\begin{thm}\label{free1}
There exists an integer $r\geq 1$ such that
\begin{enumerate}
\item for any $\sigma\in W(\rhobar)$ the module $M_\infty(\sigma)$ is free of rank $r$ over its scheme-theoretic support which is a domain;
\item for any $\sigma\in W(\rhobar)$ the modules $M_\infty(\widetilde P_{\sigma})$ and $M_\infty(P_{\sigma})$ are free of rank $r$ over their respective scheme-theoretic support;
\item for any tame inertial type $\tau$ such that $\JH(\ovl{\sigma(\tau)})\cap W(\rhobar)\ne \emptyset$ and any $\GL_2({\mathcal O}_{K})$-invariant \ $W(\F)$-lattice \ $\sigma^0(\tau)$ \ in \ $\sigma(\tau)$ \ with \ irreducible \ cosocle, \ the \ module $M_\infty(\sigma^0(\tau))$ is free of rank $r$ over its scheme-theoretic support, which is the domain $R_\infty^{(1,0),\tau}$.
\end{enumerate}
\end{thm}

\begin{cor}\label{D0r}
Let $\pi$ as in Theorem \ref{nonminimal} and $r$ as in Theorem \ref{free1}. We have an isomorphism of $\GL_2({\mathcal O}_K)K^\times$-representations $D_0(\rhobar)^{\oplus r}\cong \pi^{K_1}$.
\end{cor}
\begin{proof}
The action of the center $K^\times$ being by definition the same on both sides, we can focus on the action of $\GL_2({\mathcal O}_K)$. It follows from Theorem \ref{free1}(i) and (ii) and from (\ref{modmax}) that the surjection $P_{\sigma}\twoheadrightarrow \sigma$ induces an isomorphism of $r$-dimensional $\F$-vector spaces $\Hom_{\GL_2({\mathcal O}_K)}(\sigma,\pi^{K_1})\buildrel\sim\over\rightarrow \Hom_{\GL_2({\mathcal O}_K)}(P_{\sigma},\pi^{K_1})$. In particular the multiplicity of each $\sigma\in W(\rhobar)$ in $\pi^{K_1}$ is $r$. It follows from $M_\infty(D_{0,\sigma}(\rhobar)/\sigma)=0$ (recall $D_{0}(\rhobar)=\oplus_{\sigma\in W(\rhobar)}D_{0,\sigma}(\rhobar)$) and from (\ref{modmax}) that the injection $\sigma\hookrightarrow D_{0,\sigma}(\rhobar)$ induces an isomorphism $\Hom_{\GL_2({\mathcal O}_K)}(D_{0,\sigma}(\rhobar),\pi^{K_1})\buildrel\sim\over\rightarrow \Hom_{\GL_2({\mathcal O}_K)}(\sigma,\pi^{K_1})$. This gives an inclusion $D_0(\rhobar)^{\oplus r}\hookrightarrow \pi^{K_1}$. If this inclusion is strict, then 
by maximality of $D_0(\rhobar)^{\oplus r}$ (an obvious generalization of \cite[Prop.13.1]{BP}) this implies there exists $\sigma\in W(\rhobar)$ which appears in $\pi^{K_1}/D_0(\rhobar)^{\oplus r}$, and hence has multiplicity $>r$ in $\pi^{K_1}$, which is a contradiction.
\end{proof}

\begin{rem}
In the proof of Theorem \ref{free1}, and hence also in Corollary \ref{D0r}, one only needs the slightly weaker bounds $1\leq r_i\leq p-4$ (and $2\leq r_0\leq p-3$ if $\rbar_{\tilde v}$ is irreducible) in the genericity conditions (viii) on $\rbar_{\tilde v}$ (or equivalently $\rhobar$) in \S\ref{globalp} (these bounds are used in \cite[\S4]{LMS} which is used in the proof of \cite[Prop.8.2.6]{BHHMS1}).
\end{rem}

\subsubsection{Direct sums of diagrams}\label{directsum}

We prove Theorem \ref{nonminimal} using the method of \cite[\S4]{DoLe}.

We keep the notation in \S\S\ref{globalp}, \ref{patching}. Everything in this section being at the place $\tilde v$, we drop it from the notation. Recall we identify the set of embeddings $\Fq\hookrightarrow \F$ with $\{0,\dots,f-1\}$ via $\sigma_0\circ \varphi^i \mapsto i$. We denote by $\mathcal P$ the set of subsets of $\{0,\dots,f-1\}$ and by $J^c\in \mathcal P$ the complement of a subset $J\in \mathcal P$.

We start by fixing a tame inertial type $\tau$ such that $\JH(\ovl{\sigma(\tau)})\cap W(\rhobar)\ne \emptyset$ and a $\GL_2({\mathcal O}_{K})$-invariant $W(\F)$-lattice $\theta_0$ in $\sigma(\tau)$ with irreducible cosocle. With the notation of \cite[\S5.1]{EGS} there is $I\in \mathcal P$ such that this cosocle is $\overline\sigma_{I}(\tau)$ and $\theta_0=\sigma^{\rm o}_{I}(\tau)$. As in \cite[p.77]{EGS} we can reindex the irreducible constituents of $\theta_0/p$ by elements $J'$ in $\mathcal P$ as follows:
\[\sigma_{J'}\defeq \overline \sigma_{(J' \cup I^c)\backslash (J'\cap I^c)}(\tau),\]
so that (by \cite[Thm.5.1.1]{EGS}) the $j$-th layer of the cosocle filtration of $\theta_0/p$ consists of the $\sigma_{J'}$ for $|J'|=f-j$, $0\leq j\leq f$. By the beginning of the proof of \cite[Thm.10.1.1]{EGS} (see {\it loc.cit.}\ p.77), there is $J'_{\min}\subseteq J'_{\max}$ in ${\mathcal P}$ such that $\JH(\ovl{\theta_0/p})\cap W(\rhobar)=\{\sigma_{J'} : J'_{\min}\subseteq J'\subseteq J'_{\max}\}$. By \cite[Thm.7.2.1]{EGS} we have
\[R_\infty^{(1,0),\tau}\cong \big(W(\F)\bbra{(X'_j,Y'_j)_{j\in J'_{\max}\backslash J'_{\min}}}/(X'_jY'_j-p)_{j\in J'_{\max}\backslash J'_{\min}}\big)\bbra{U_1,\dots,U_d}\]
for some integer $d\geq 0$. Up to renumbering the variables we can assume that the irreducible component of $R_\infty^{(1,0),\tau}/p$ corresponding to $\sigma_{J'}$, $J'_{\min}\subseteq J'\subseteq J'_{\max}$, in \cite[p.77]{EGS} (which is the support of $M_\infty(\sigma_{J'})$ by Theorem \ref{free1}(i)) is given by the ideal $((X'_j)_{j\in J'\backslash J'_{\min}},(Y'_j)_{j\in J'_{\max}\backslash J'})$.

We first fix $J\in {\mathcal P}$ such that $|J|=f-1$, so that $J^c=\{j\}$ for some $j\in \{0,\dots,f-1\}$. We let $\theta$ be the unique (up to homothety) $\GL_2({\mathcal O}_{K})$-invariant $W(\F)$-lattice in $\sigma(\tau)$ with irreducible cosocle $\sigma_{J}$ (\cite[Lemma 4.1.1]{EGS}). Up to multiplication by an element in $W(\F)^\times$, there is a unique $\GL_2({\mathcal O}_{K})$-equivariant saturated inclusion $\iota:\theta\hookrightarrow \theta_0$, i.e.\ such that the induced morphism $\overline\iota:\theta/p\rightarrow \theta_0/p$ is nonzero. Recall that by Theorem \ref{free1}(iii) both $M_\infty(\theta)$ and $M_\infty(\theta_0)$ are free of rank $r$ over $R_\infty^{(1,0),\tau}$.

\begin{lem}\label{casf-1}
The image of $M_\infty(\iota):M_\infty(\theta)\hookrightarrow M_\infty(\theta_0)$ is $xM_\infty(\theta_0)$, where $x=p$ if $j\in J'_{\min}$, $x=X'_{j}$ if $j\in J'_{\max}\backslash J'_{\min}$ and $x=1$ if $j\notin J'_{\max}$.
\end{lem}
\begin{proof}
It follows from \cite[Thm.5.2.4(4)]{EGS} (up to a reindexation as above) that $p(\theta_0/\iota(\theta))=0$ and that the irreducible constituents of $\theta_0/\iota(\theta)$ are the $\sigma_{J'}$ for $J'$ containing $j$. In particular $\theta_0/\iota(\theta)$ is of the form $\overline \sigma^{\mathcal J}$ for a capped interval $\mathcal J$ as in \cite[p.81]{EGS} (namely ${\mathcal J}=\{J' : j\in J'\}$). By the proof of \cite[Prop.8.2.3]{BHHMS1} the module $M_\infty(\theta_0/\iota(\theta))=M_\infty(\overline \sigma^{\mathcal J})$ is free of rank $r$ over its schematic support, which is the unique reduced quotient of $R_\infty^{(1,0),\tau}/p$ with irreducible components corresponding to the $\sigma_{J'}$ such that $j\in J'$ and $J'_{\min}\subseteq J'\subseteq J'_{\max}$. If $j\notin J'_{\max}$, there are no such $J'$, so this quotient is $0$ (i.e.\ $M_\infty(\theta_0/\iota(\theta))=0$). If $j\in J'_{\max}\backslash J'_{\min}$, then this quotient is clearly $(R_\infty^{(1,0),\tau}/p)/(X'_{j})=R_\infty^{(1,0),\tau}/(X'_{j})$. Finally, if $j\in J'_{\min}$, all irreducible components remain, i.e.\ this quotient is $R_\infty^{(1,0),\tau}/p$. The lemma follows by exactness of $M_\infty$.
\end{proof}

We now consider an arbitrary $J\in {\mathcal P}$ and let $\theta$ be the unique invariant $W(\F)$-lattice in $\sigma(\tau)$ with irreducible cosocle $\sigma_{J}$. If $J^c\ne \emptyset$ we set ${J}^c=\{j_1,\dots,j_h\}$ and $J_i\defeq J\amalg \{j_1,\dots, j_{h-i}\}$ for $i\in \{0,\dots,h\}$ (so $J_0=\{0,\dots,f-1\}$ and $J_h=J$). As above we then denote by $\theta_i$ for $i\in \{0,\dots,h\}$ the unique (up to homothety) invariant $W(\F)$-lattice in $\sigma(\tau)$ with irreducible cosocle $\sigma_{J_i}$ and $\iota_i:\theta_i\hookrightarrow \theta_{i-1}$ the corresponding saturated inclusion for $i\in \{1,\dots,h\}$ (so $\theta_0$ is the same as before and $\theta_h=\theta$). The composition
\[\iota_1\circ\cdots\circ\iota_i:\theta_i\buildrel {\iota_i}\over \hookrightarrow \theta_{i-1}\buildrel {\iota_{i-1}}\over\hookrightarrow \cdots \theta_1\buildrel {\iota_1}\over\hookrightarrow \theta_0\]
is still saturated since one can check using \cite[Thm.5.1.1]{EGS} that the cosocle $\sigma_{J_h}$ of $\theta_h/p$ remains in the image of $\theta_i/p\rightarrow \theta_{i-1}/p$ for all $i\in \{h, h-1,\dots,1\}$ (indeed, by {\it loc.\ cit.}\ the Serre weights $\sigma_{J_{i}}-\sigma_{J_{i-1}}$ in $\theta_0/p$ form a nonsplit extension as $J_{i}\subseteq J_{i-1}$ and $|J_{i-1}\backslash J_{i}|=1$). In particular $\iota\defeq \iota_1\circ\cdots\circ\iota_h$ is the unique (up to scalar) saturated inclusion $\theta\hookrightarrow \theta_0$.

\begin{prop}\label{freenessprop}
There is $x\in R_\infty^{(1,0),\tau}$ such that the image of $M_\infty(\iota):M_\infty(\theta)\hookrightarrow M_\infty(\theta_0)$ is $xM_\infty(\theta_0)$. Moreover the principal ideal $xR_\infty^{(1,0),\tau}$ only depends on {\upshape(}the semisimplification of{\upshape)} $\theta_0/\iota(\theta)$.
\end{prop}
\begin{proof}
The statement being trivial if $J^c=\emptyset$ (equivalently if $\theta=\theta_0$) we can assume $J^c\ne \emptyset$. For $i\in \{1,\dots ,h\}$ we can apply Lemma \ref{casf-1} to $\iota_i:\theta_i\hookrightarrow \theta_{i-1}$ instead of $\iota:\theta\hookrightarrow \theta_0$. Hence there is $x_i\in R_\infty^{(1,0),\tau}$ such that the image of $M_\infty(\iota_i)$ is $x_iM_\infty(\theta_{i-1})$. The image of $M_\infty(\iota)$ is thus $(\prod_{i=1}^hx_i)M_\infty(\theta_0)$, i.e.\ we can take $x=\prod_{i=1}^hx_i$. It follows that $M_\infty(\theta_0/\iota(\theta))\cong (R_\infty^{(1,0),\tau}/(x))^{\oplus r}$. Hence the irreducible components of $R_\infty^{(1,0),\tau}/(x)$ are the ones corresponding to the $\sigma_{J'}$ such that $J'_{\min}\subseteq J'\subseteq J'_{\max}$ and $\sigma_{J'}$ appears in $\theta_0/\iota(\theta)$, and their multiplicities are the multiplicities of the $\sigma_{J'}$ in $\theta_0/\iota(\theta)$. The second assertion then follows by the same argument as at the end of the proof of \cite[Prop.4.17]{DoLe} (it also follows from an explicit computation of $x$ via Lemma \ref{casf-1}).
\end{proof}

Till the end of this section, we now extensively use notation and results from \cite[\S4]{DoLe} to which we refer the reader for more details.

Recall that $D_{0}(\rhobar)=\oplus_{\sigma\in W(\rhobar)}D_{0,\sigma}(\rhobar)$. If $\chi:I\rightarrow \F^\times$ is a character appearing in $D_0(\rhobar)^{I_1}$ and $\F v_\chi\subseteq D_0(\rhobar)$ is the corresponding eigenspace (which is $1$-dimensional), we define as in \cite[Def.4.1]{DoLe} $R\chi$ as the character of $I$ on $(\soc_{\GL_2({\mathcal O}_K)}\langle \F\GL_2({\mathcal O}_K) v_\chi\rangle)^{I_1}$, which is also $1$-dimensional as it is $\sigma^{I_1}$ for the unique $\sigma\in W(\rhobar)$ such that $\chi$ appears in $D_{0,\sigma}(\rhobar)^{I_1}$. As in \cite[p.8]{BP} we denote by $\chi^s$ the character of $I$ on $\smatr{0}{1}{p}{0}v_\chi\in D_0(\rhobar)^{I_1}$ and by $\sigma(\chi)$ the Serre weight which is the cosocle of $\Ind_I^{\GL_2({\mathcal O}_K)}\chi$.

We define as in \cite[Prop.4.14]{DoLe} an isomorphism
\[\overline h_{\chi}:M_\infty(\sigma(R\chi^s))/\m_\infty \buildrel\sim\over\longrightarrow M_\infty(\sigma(R\chi))/\m_\infty\]
(the ``one-dimensional by Theorem 4.6'' in the proof of {\it loc.cit.}\ can just be replaced by ``of the same dimension by Theorem \ref{free1}''; also note that $\overline h_{\chi}$ is an isomorphism, as it is dual to the isomorphism $g_\chi$ in {\it loc.cit.}).

\begin{prop}\label{dole}
Let $k\geq 1$ and $\chi_0,\dots,\chi_{k-1}$ arbitrary characters of $I$ which occur on $\pi^{I_1}$ {\upshape(}equivalently on $D_0(\rhobar)^{I_1}${\upshape)} such that $R\chi_i^s=R\chi_{i+1}$ for $i\in \{0,\dots,k-2\}$ and $R\chi_{k-1}^s=R\chi_{0}$. Then the isomorphism
\[\overline h_{\chi_1}\circ \overline h_{\chi_2}\circ \cdots \circ \overline h_{\chi_{k-1}}\circ \overline h_{\chi_{0}}:M_\infty(\sigma(R\chi_0^s))/\m_\infty\buildrel\sim\over\longrightarrow M_\infty(\sigma(R\chi_0^s))/\m_\infty\]
is the multiplication by a scalar in $\F^\times$ which depends neither on $r$ nor on $M_\infty$. In particular this scalar is the same as in {\upshape\cite[(34)]{DoLe}}.
\end{prop}
\begin{proof}
We just indicate the steps in the proofs of \cite[\S\S4.4, 4.5]{DoLe}, where the assumption $r=1$ is used, and how one can extend the argument there to $r\geq 1$. We use without comment the notation of {\it loc.cit.}\\
$\bullet$ The definition of the isomorphism $\widetilde h_\chi:M_\infty(\theta^{R\chi^s})\buildrel\sim\over\rightarrow M_\infty(\theta^{R\chi})$ in \cite[(28)]{DoLe} holds because one only needs to know that $M_\infty(\theta^{R\chi^s})$ and $M_\infty(\theta^{R\chi})$ are free of the same finite rank over $R_\infty(\tau)$.\\
$\bullet$ By Proposition \ref{freenessprop} there exists $\widetilde U_p(\chi)\in R_\infty(\tau)$ such that $M_\infty(\iota)(M_\infty(\theta^{R\chi}))=\widetilde U_p(\chi)M_\infty(\theta^{R\chi^s})$, where $\iota:\theta^{R\chi}\hookrightarrow \theta^{R\chi^s}$ is as in the unlabelled commutative diagram below \cite[(27)]{DoLe}. Since $R_\infty(\tau)$ is a domain by \cite[Thm.7.2.1(2)]{EGS} and $M_\infty(\theta^{R\chi})$, $M_\infty(\theta^{R\chi^s})$ are free of rank $r$ over $R_\infty(\tau)$ by Theorem \ref{free1}(iii), there is a unique $R_\infty(\tau)$-linear isomorphism $\widetilde \iota_\chi:M_\infty(\theta^{R\chi})\buildrel\sim\over\rightarrow M_\infty(\theta^{R\chi^s})$ such that $M_\infty(\iota)=\widetilde \iota_\chi \circ \widetilde U_p(\chi)$, where $\widetilde U_p(\chi)$ here means multiplication by $\widetilde U_p(\chi)$ on $M_\infty(\theta^{R\chi})$. Then we have a commutative diagram analogous to \cite[(29)]{DoLe} replacing the multiplication by $\widetilde U_p(\chi)$ in the diagonal map by the map $\widetilde h_\chi \circ \widetilde \iota_\chi\circ \widetilde U_p(\chi)=\widetilde U_p(\chi)(\widetilde h_\chi \circ \widetilde \iota_\chi)$.\\
$\bullet$ By the commutativity of the right-hand side of (the analog of) \cite[(28)]{DoLe} and by the isomorphism $M_\infty(Q(\chi^s)^{R\chi})\cong M_\infty(\theta(\chi^s)^{R\chi})/p$, we deduce that the map
\[h_\chi \circ \iota_Q:M_\infty(Q(\chi^s)^{R\chi})\longrightarrow M_\infty(Q(\chi^s)^{R\chi})\]
is the multiplication by the image of $p^{-e(\chi)}U_p(\chi)$ in $R_\infty(\tau(\chi^s))/p$. As the image of $h_\chi \circ \iota_Q$ is $\widetilde U_p(\chi)M_\infty(Q(\chi^s)^{R\chi})$ by the commutativity of the left-hand side of (the analog of) \cite[(28)]{DoLe} and the definition of $\widetilde U_p(\chi)$, we deduce that
\[\widetilde U_p(\chi)(R_\infty(\tau(\chi^s))/p)=(p^{-e(\chi)}U_p(\chi))(R_\infty(\tau(\chi^s))/p).\]
In particular, multiplying $\widetilde U_p(\chi)$ by a unit in $R_\infty(\tau)$ we can assume that $\widetilde U_p(\chi)$ and $p^{-e(\chi)}U_p(\chi)$ have the same image in the quotient $R_\infty(\tau(\chi^s))/p$ of $R_\infty(\tau)$. As a consequence the analogue of \cite[Prop.4.17]{DoLe} holds.\\
$\bullet$ Since by definition $p^{-e(\chi)}U_p(\chi)\in R_\infty(\tau(\chi^s))\backslash pR_\infty(\tau(\chi^s))$, we have
\begin{equation}\label{upchibar}
{\rm Ann}_{R_\infty(\tau(\chi^s))/p}\big(\overline{p^{-e(\chi)}U_p(\chi)}\big)\subseteq \m_\infty (R_\infty(\tau(\chi^s))/p).
\end{equation}
As $\widetilde U_p(\chi)\mapsto \overline{p^{-e(\chi)}U_p(\chi)}\in R_\infty(\tau(\chi^s))/p$ (previous point), we deduce $\widetilde U_p(\chi)(\widetilde h_\chi \circ \widetilde \iota_\chi - \Id)\mapsto 0$ in $\End_{R_\infty(\tau(\chi^s))/p}(M_\infty(Q(\chi^s)^{R\chi}))$ by the analog of \cite[(28)]{DoLe}. As $M_\infty(Q(\chi^s)^{R\chi})\cong M_\infty(\theta(\chi^s)^{R\chi})/p$ is free of rank $r$ over $R_\infty(\tau(\chi^s))/p$ (by Theorem \ref{free1}(iii)), (\ref{upchibar}) implies the image of $\widetilde h_\chi \circ \widetilde \iota_\chi - \Id$ in $\End_{R_\infty(\tau(\chi^s))/p}(M_\infty(Q(\chi^s)^{R\chi}))$ lands in $\m_\infty \End_{R_\infty(\tau(\chi^s))/p}(M_\infty(Q(\chi^s)^{R\chi}))$. Since $\Ker(R_\infty(\tau)\twoheadrightarrow R_\infty(\tau(\chi^s))/p)\subseteq \m_\infty R_\infty(\tau)$, we also have
\begin{equation}\label{minftyr}
\widetilde h_\chi \circ \widetilde \iota_\chi - \Id\in \m_\infty\End_{R_\infty(\tau)}(M_\infty(\theta^{R\chi})).
\end{equation}
$\bullet$ The big unlabelled diagram before \cite[(33)]{DoLe} still holds but the diagonal maps are not simply multiplication by some $\widetilde U_p(\chi_i)$. For instance in the case $k=3$ (the general case being similar) one has to replace the left diagonal maps in {\it loc.cit.} by successively (from top to bottom) $\widetilde U_p(\chi_0)((\widetilde\iota_{\chi_2}\circ \widetilde \iota_{\chi_1})^{-1}\circ (\widetilde h_{\chi_0} \circ \widetilde \iota_{\chi_0})\circ \widetilde\iota_{\chi_2}\circ \widetilde \iota_{\chi_1})$, $\widetilde U_p(\chi_2)(\widetilde\iota_{\chi_1}^{-1}\circ (\widetilde h_{\chi_2} \circ \widetilde \iota_{\chi_2})\circ \widetilde\iota_{\chi_1})$, and $\widetilde U_p(\chi_1)(\widetilde h_{\chi_1} \circ \widetilde \iota_{\chi_1})$. By (\ref{minftyr}) and the $R_\infty(\tau)$-linearity of the isomorphisms $\widetilde \iota_{\chi_i}$, all these diagonal maps are in $\widetilde U_p(\chi_i)(\Id + \m_\infty \End_{R_\infty(\tau)}(M_\infty(\theta^{R\chi_0^s})))$, and their composition is thus in
\begin{equation}\label{minftyp}
\big(\prod_{i=0}^{k-1}\widetilde U_p(\chi_i)\big)(\Id + \m_\infty \End_{R_\infty(\tau)}(M_\infty(\theta^{R\chi_0^s}))).
\end{equation}
$\bullet$ For $\nu\geq 1$ defined as above \cite[(33)]{DoLe}, we have from the definition of the $\widetilde \iota_{\chi_i}$:
\begin{equation}\label{nuunit}
\big(\prod_{i=0}^{k-1}\widetilde U_p(\chi_i)\big)(\widetilde\iota_{\chi_0}\circ \widetilde \iota_{\chi_{k-1}}\circ \cdots \circ \widetilde \iota_{\chi_{1}})=p^{\nu}\Id
\end{equation}
which implies $p^{-\nu}(\prod_{i=0}^{k-1}\widetilde U_p(\chi_i))\in R_\infty(\tau)^\times$ since the $\widetilde\iota_{\chi_i}$ are isomorphisms. By the commutativity in the (analog of) the big unlabelled diagram before \cite[(33)]{DoLe} (see the previous point) together with (\ref{minftyp}) and (\ref{nuunit}) we finally obtain
\[\widetilde h_{\chi_1}\circ \cdots \circ \widetilde h_{\chi_{k-1}}\circ \widetilde h_{\chi_{0}}\in \big(p^{-\nu}\prod_{i=0}^{k-1}\widetilde U_p(\chi_i)\big)(\Id + \m_\infty \End_{R_\infty(\tau)}(M_\infty(\theta^{R\chi_0^s})))\]
which is our analog of \cite[(33)]{DoLe}. Then \cite[(34)]{DoLe} follows by the same argument. The rest of the proof in \cite[\S5]{DoLe} is unchanged.
\end{proof}

We can now prove Theorem~\ref{nonminimal}.

\begin{proof}[Proof of Theorem~\ref{nonminimal}]

We let $D(\rhobar)=(D_1(\rhobar)\hookrightarrow D_0(\rhobar))$ be the diagram denoted by ${\mathcal D}(\pi_{\rm glob}(\rhobar))$ in \cite{DoLe}, which only depends on $\rhobar$. Let $D(\pi)=(D_1(\pi)\hookrightarrow D_0(\pi))\defeq (\pi^{I_1}\hookrightarrow \pi^{K_1})$ be the diagram defined by $\pi$. We will show that $D(\rhobar)^{\oplus r} \cong D(\pi)$ as diagrams.

Define first $R:\pi^{I_1}\rightarrow (\soc_{\GL_2(\cO_K)}\pi)^{I_1}$ as in \cite[Def.4.1]{DoLe}, i.e.\ $Rv=S_{i(\chi)}v$ with $S_{i(\chi)}$ as in \cite[Rem.4.2]{DoLe} if $v\in \pi^{I_1}$ is an $I$-eigenvector with eigencharacter $\chi$. Note that the eigencharacter of $Rv$ is $R\chi$. 

Starting from $D(\rhobar)$ we define a groupoid $\GG$ with objects ${\mathbf x}_\xi$, where $\xi$ is any character of $I$ such that $(\soc_{\GL_2(\cO_K)} D_0(\rhobar))^{I_1}[\xi] \ne 0$, and morphisms freely generated by $g_\chi : {\mathbf x}_{R\chi} \congto {\mathbf x}_{R\chi^s}$, where $\chi$ is any character of $I$ such that $D_1(\rhobar)[\chi] \ne 0$, as in \cite[Def.4.3]{DoLe}.

The diagram $D(\pi)$ defines an $r$-dimensional representation of $\GG$, sending ${\mathbf x}_\xi$ to the vector space $(\soc_{\GL_2(\cO_K)} D_0(\pi))^{I_1}[\xi]$ and $g_\chi$ to the linear map
\[g_\chi^\pi : (\soc_{\GL_2(\cO_K)} D_0(\pi))^{I_1}[R\chi] \congto (\soc_{\GL_2(\cO_K)} D_0(\pi))^{I_1}[R\chi^s]\]
as in \cite[\S4]{DoLe}. Similarly, we have an $r$-dimensional representation of $\GG$ defined by the diagram $D(\rhobar)^{\oplus r}$; we denote the linear maps by $g_\chi^{\rhobar}$.

To check that the two $r$-dimensional representations of $\GG$ are isomorphic it suffices to check that
for each object $\mathbf x$ the restrictions of the two representations to the automorphism group $\GG_{\mathbf x}$ are
isomorphic (see \cite[Prop.4.5]{DoLe}), which is the case by Proposition~\ref{dole}, remembering that
$g_\chi^\pi$ is the dual of $\overline h_\chi$ by (the analog of) \cite[Prop.4.14]{DoLe}.

Therefore there exists an isomorphism
\[\lambda : (\soc_{\GL_2(\cO_K)} D_0(\pi))^{I_1} \congto (\soc_{\GL_2(\cO_K)} D_0(\rhobar)^{\oplus r})^{I_1}\]
of $I$-representations such that $\lambda \circ g_\chi^\pi = g_\chi^{\rhobar} \circ \lambda$ on $(\soc_{\GL_2(\cO_K)} D_0(\pi))^{I_1}[R\chi]$ for all $\chi$. As $\pi^{K_1} \cong D_0(\rhobar)^{\oplus r}$ as $K$-representations we can extend $\lambda$ uniquely to an isomorphism $\lambda : D_0(\pi) \congto D_0(\rhobar)^{\oplus r}$ of $K$-representations (extending to the ${\GL_2(\cO_K)}$-socle first). As in the proof of \cite[Prop.4.4]{DoLe} we deduce that $\lambda$ restricts to an isomorphism 
$\lambda : D_1(\pi) \congto D_1(\rhobar)^{\oplus r}$ commuting with $\smatr{0}{1}{p}{0}$ and $I$, which completes the proof.
\end{proof}

\subsubsection{Local-global compatibility results}\label{lcresults}

We collect our previous results to deduce (together with the results of \cite{HuWang2}) special cases of Conjecture \ref{theconjbar} and Conjecture \ref{theconj} when $n=2$ and $K$ is unramified.

We keep all the previous notation. We also keep the assumptions (i) to (xii) of \S\ref{globalp} (in particular $\rbar_{\tilde{v}}$ is semisimple), except that we replace the bounds on the $r_i$ in (viii) by the stronger bounds (which are those of \cite[\S1]{BHHMS1}):
\begin{equation*}
\begin{aligned}
12 &\le r_j \le p-15 &&\quad \text{if $j > 0$ or $\rhobar$ is reducible;}\\
13 &\le r_0 \le p-14 &&\quad \text{if $\rhobar$ is irreducible.}
\end{aligned}
\end{equation*}

Recall that we choose Serre weights $\sigma_{\tilde w}\in W(\rbar_{\tilde w}(1))$ for $w\in S_p\backslash \{v\}$ and consider $\pi= \Hom_{U^v}(\otimes_{w\in S_p\backslash\{v\}}\sigma_{\tilde w},S(V^v,\F)[\m^{\Sigma}])$ (see Theorem \ref{nonminimal}).

\begin{thm}\label{mi13}
We have $[\pi[\m_{I_1/Z_1}^3]:\chi]=[\pi[\m_{I_1/Z_1}]:\chi]$ for all smooth characters $\chi:I\rightarrow \F^\times$ appearing in $\pi[\m_{I_1/Z_1}]$.
\end{thm}
\begin{proof}
The statement of \cite[Thm.8.3.11]{BHHMS1} applies {\it verbatim} with the same proof to $\pi$ as above using Theorem \ref{free1} and (\ref{modmax}). Combining this with Corollary \ref{D0r}, we see that $\pi$ satisfies all the assumptions of \cite[Thm.1.4]{BHHMS1}, whence the result by \cite[Thm.1.5]{BHHMS1}.
\end{proof} 

\begin{rem}\label{dim=f}
A similar argument as in (ii) of the proof of \cite[Thm.8.4.1]{BHHMS1} (which uses \cite[App.A]{GN}) shows that we also have $\dim_{\GL_2(K)}(\pi)=f$, where $\dim_{\GL_2(K)}(\pi)$ is the Gelfand--Kirillov dimension of $\pi$ as defined in \cite[\S5.1]{BHHMS1}.
\end{rem}

The following theorem is one of the main results of this paper.

\begin{thm}\label{specialcase1}
Keep all the previous assumptions and assume that the $r_i$ in $\rbar_{\tilde{v}}$ satisfy the following stronger bounds:
\begin{equation}\label{sstrong}
\begin{array}{llllll}
\max\{12,2f-1\} \!&\!\le \!&\! r_j \!&\!\le \!&\! p-\max\{15,2f+2\} \!& \text{if $j > 0$ or $\rhobar$ is reducible;}\\
\max\{13,2f\} \!&\!\le \!&\! r_0 \!&\!\le \!& \!p-\max\{14,2f+1\} \!& \text{if $\rhobar$ is irreducible.}
\end{array}
\end{equation}
Let $\sigma^v\defeq \otimes_{w\in S_p\backslash\{v\}}\sigma_{\tilde w}$, where the $\sigma_{\tilde w}$ are Serre weights in $W(\rbar_{\tilde w}(1))$ for $w\in S_p\backslash \{v\}$. Then Conjecture \ref{theconjbar} holds for $\Hom_{U^v}(\sigma^v,S(V^v,\F)[\m^{\Sigma}])$.
\end{thm}
\begin{proof}
This follows from Corollary \ref{maintensorind} applied to $\pi=\Hom_{U^v}(\sigma^v,S(V^v,\F)[\m^{\Sigma}])$, which satisfies all the assumptions there by Theorem \ref{nonminimal} and Theorem \ref{mi13}, and by Remark \ref{trivial}(ii).
\end{proof}

We now give some evidence for Conjecture \ref{theconj}, still assuming (\ref{sstrong}). As we also need $r=1$, and to make things as simple as possible, we replace assumptions (v) and (vii) in \S\ref{globalp} by
\begin{enumerate}
\item[]$\overline r$ is unramified at all finite places outside $S_p$
\end{enumerate}
and we then take $S\defeq S_p$ (hence $\Sigma=S_p\cup \{v_1\}$). We also replace assumption (xii) in \S\ref{globalp} by
\begin{enumerate}
\item[]$\iota_{\widetilde{v_1}}(U_{v_1})$ is equal to the upper-triangular unipotent matrices mod $\widetilde{v_1}$.
\end{enumerate}
We take $V^v=U^p\prod_{w\in S_p\backslash\{v\}}V_w$ with $\iota_{\tilde{w}}(V_w)=1+pM_2(\cO_{F_{\tilde w}})\subseteq \GL_2(\cO_{F_{\tilde w}})=\iota_{\tilde{w}}(U_w)$. We let $T_{\widetilde{v_1}}$ be the Hecke operator acting on $S(V^v,\F)$ by the double coset
\[\iota_{\widetilde{v_1}}^{-1}\left[\iota_{\widetilde{v_1}}(U_{v_1})\begin{pmatrix}\varpi_{\widetilde{v_1}}&\\& 1\end{pmatrix}\iota_{\widetilde{v_1}}(U_{v_1})\right],\]
where $\varpi_{\widetilde{v_1}}$ is a uniformizer in ${\mathcal O}_{F_{\widetilde{v_1}}}$. Increasing $\F$ if necessary, we fix a choice of eigenvalues $\overline \alpha_{\widetilde{v_1}}\in \F$ of $\rhobar(\Frob_{\widetilde{v_1}})$ (the image of a geometric Frobenius at $\widetilde{v_1}$) and consider the ideal
\[\m^S\defeq (\m^\Sigma,T_{\widetilde{v_1}}-\alpha_{\widetilde{v_1}})\subseteq {\mathcal T}^\Sigma[T_{\widetilde{v_1}}],\]
where $\alpha_{\widetilde{v_1}}$ is any element in $W(\F)$ lifting $\overline \alpha_{\widetilde{v_1}}$ (see \S\ref{somprel} for ${\mathcal T}^\Sigma$). Then, replacing $\m^\Sigma$ by $\m^S$ everywhere in \S\S\ref{globalp}, \ref{patching}, \ref{directsum}, by a multiplicity $1$ result analogous to the one in \cite[Prop.3.5.1]{BD} (see for instance the argument in the proof of \cite[Lemma 3.1.4]{Enns2}) all the previous global results hold with $r$ being $1$.

\begin{prop}\label{allsatisfied}
Choose Serre weights $\sigma_{\tilde w}\in W(\rbar_{\tilde w}(1))$ for $w\in S_p\backslash \{v\}$ and let
\[\pi\defeq  \Hom_{U^v}(\otimes_{w\in S_p\backslash\{v\}}\sigma_{\tilde w},S(V^v,\F)[\m^{S}]).\]
The representation $\pi$ satisfies all the assumptions of \S\ref{sec:length-of-pi} {\upshape(}with $\rhobar=\rbar_{\tilde{v}}(1)${\upshape)}.
\end{prop}
\begin{proof}
The only missing assumption is the essential self-duality (\ref{eq:selfdual}). But it holds by the same proof as for the definite case of \cite[Thm.8.2]{HuWang2} using Remark \ref{dim=f}.
\end{proof}

From the results of \S\ref{sec:length-of-pi}, we thus deduce the following theorems.

\begin{thm}\label{thm:gen-socleglob}
The $\GL_2(F_{\tilde v})$-representation $\pi$ is generated by its $\GL_2(\cO_{F_{\tilde v}})$-socle, in particular is of finite type.
\end{thm}

\begin{thm}\label{cor:pi-irredglob}\

  \begin{enumerate}
  \item Assume that $\rbar_{\tilde{v}}$ is irreducible. Then $\pi$ is irreducible and is a supersingular representation.
  \item Assume that $\rbar_{\tilde{v}}$ is reducible {\upshape(}split{\upshape)} and write
    $\rhobar=\rbar_{\tilde{v}}(1)=\begin{pmatrix}\chi_{1} &0\\0 &\chi_2\end{pmatrix}$. Then one
    has\begin{equation*} \pi=\Ind_{B^-(F_{\tilde v})}^{\GL_2(F_{\tilde v})}(\chi_1\omega^{-1}\otimes
      \chi_2)\oplus \pi' \oplus \Ind_{B^-(F_{\tilde v})}^{\GL_2(F_{\tilde v})}(\chi_2\omega^{-1}\otimes
      \chi_1),
    \end{equation*}
    where $\pi'$ is generated by its $\GL_2(\cO_{F_{\tilde v}})$-socle and $\pi'^{\vee}$ is essentially
    self-dual, i.e.\ satisfies {\upshape(\ref{eq:selfdual})}. Moreover, when $f=2$, $\pi'$ is irreducible
    and supersingular {\upshape(}and hence $\pi$ is semisimple{\upshape)}.
  \end{enumerate}
\end{thm}
\begin{proof}
Everything is in Corollary \ref{cor:pi-irred} and Corollary \ref{cor:split2}, except the precise form of the irreducible principal series $\pi_0$, $\pi_f$ in {\it loc.cit.}, but this easily follows from (\ref{eq:Koh-char}) and Theorem \ref{nonminimal} (which is \cite[\S5]{DoLe} since $r=1$).
\end{proof}

Combining Theorem \ref{cor:pi-irredglob} with Theorem \ref{specialcase1}, we obtain:

\begin{cor}\label{specialcase2}
Keep the same assumptions as just before Proposition \ref{allsatisfied}. If $\rbar_{\tilde{v}}$ is irreducible or if $f=2$, then $\pi$ is compatible with $\rhobar$ {\upshape(}Definition \ref{compatible2}{\upshape)}. In particular in these cases Conjecture \ref{theconj} holds for $\Hom_{U^v}(\sigma^v,S(V^v,\F)[\m^{S}])$.
\end{cor}

\begin{rem}
When \ $\rbar_{\tilde{v}}$ \ is \ reducible \ nonsplit, \ a \ similar \ proof \ as \ for \cite[Thm.1.6]{HuWang2} (with the hypothesis of {\it loc.cit.}\ on $\rbar_{\tilde{v}}$) implies that $\pi$ is generated over $\GL_2(F_{\tilde v})$ by $\pi^{K_1}$. When moreover $f=2$, a similar proof as for \cite[Thm.1.7]{HuWang2} implies that $\pi$ is at least compatible with $\widetilde P_{\rhobar}=P_{\rhobar}=B$ (Definition \ref{compatible1}).
\end{rem}

\newpage

\bibliography{Biblio}

\bibliographystyle{amsalpha}

\end{document}